\documentclass[10pt,french,leqno]{smfart}

\makeindex

\usepackage[T1]{fontenc}
\usepackage{enumitem}
\usepackage{bbm}
\usepackage{bm}
\usepackage{amsmath}
\usepackage{amssymb,url,xspace}
\usepackage{mathtools}
\usepackage{mathrsfs,euscript,color}
\usepackage[all]{xy}
\usepackage{xcolor}
\usepackage{stmaryrd}

\newtheorem{theorem}{\textsc{Th\'eor\`eme}}[subsection]
\newtheorem{proposition}[theorem]{\textsc{Proposition}}
\newtheorem{lemme}[theorem]{\textsc{Lemme}}
\newtheorem{corollary}[theorem]{\textsc{Corollaire}}
\newtheorem{remark}[theorem]{\textsc{Remarque}}
\newtheorem{convention}[theorem]{\textsc{Convention}}
\newtheorem{definition}[theorem]{\textsc{D\'efinition}}

\newtheorem{hypothese}[theorem]{\textsc{Hypoth\`ese}}

\newtheorem{exemple}[theorem]{\textsc{Exemple}}

\newtheorem{notation}[theorem]{\textsc{Notation}}

\catcode`\@=11
\catcode`\Ž=13
\catcode`\=13
\catcode`\ˆ=13
\catcode`\=13
\catcode`\=13
\catcode`\‰=13
\catcode`\™=13
\catcode`\"=13
\catcode`\•=13
\catcode`\ž=13
\catcode`\=13
\catcode`\'=13
\catcode`\Ë=13
\def Ž{\'e}
\def {\`e}
\def ˆ{\`a}
\def {\`u}
\def {\^e}
\def ‰{\^a}
\def ™{\^o}
\def "{\^\i}
\def •{\"\i}
\def ž{\^u}
\def {\c c}
\def '{\"e}
\def Ë{\`A}
\catcode`\:=13
\def :{~\string:}
\catcode`\;=13
\def ;{~\string;}
\catcode`\!=13
\def !{~\string!}
\def\cad{c'est-\`a-dire\ }

\def\bydef{\buildrel \mathrm{d\acute{e}f}\over{=}}

\def\wt#1{\widetilde{#1}}
\def\bs#1{\boldsymbol{#1}}
\def\wh{\widehat}

\def\mbb#1{\mathbb{#1}}
\def\bsbbc{\mathbbm{c}}

\def\NN{\mathfrak N}
\def\UU{\mathfrak U}
\def\Siegel{\bs{\mathfrak{S}}}

\def\ag{\mathfrak{a}}

\def\nn{\mathfrak n}

\def\bs#1{\boldsymbol{#1}}
\def\bsX{\bs{X}}
\def\bsY{\bs{Y}}

\def\ES#1{\EuScript{#1}}
\def\ESA{\ES{A}}
\def\ESC{\ES{C}}
\def\ESB{\ES{B}}
\def\ESF{\ES{F}}
\def\ESN{\ES{N}}
\def\ESO{\ES{O}}
\def\ESP{\ES{P}}

\def\brT_#1{[T]_{#1}}
\def\brTo_#1{[T_0]_{#1}}
\def\brTX_#1{[T-X]_{#1}}

\def\ptf{\,.}
\def\vg{\,,}
\def\vgq{\,,\quad}

\def\mathpvg{\!\!;}

\def\dd{\,{\mathrm d}}

\def\Jres{\mathfrak J}

\def\u{{\mathrm{unip}}}

\def\f{{\ES{F}}}

\def\TopF{{$\mathrm{Top}_F$}}
\def\TopFv{{$\mathrm{Top}_{F_v}$}}
\def\TopFS{{$\mathrm{Top}_{F_S}$}}
\def\TopwhF{{$\mathrm{Top}_{\smash{\wh{F}}}$}}

\def\cF{\operatorname{c-\textit{F}}}
\def\cFG{\operatorname{c-\textit{F}-\textit{G}}}

\def\bsfrX{\bs{\mathfrak{X}}}
\def\bsfrY{\bs{\mathfrak{Y}}}

\def\FN{{_F\ES{N}}}
\def\NF{\ES{N}_F}
\def\FU{{_F\UU}}
\def\UF{\UU_F}

\def\scr#1{\mathscr{#1}}
\def\scrX{\scr{X}}
\def\scrY{\mathscr{Y}}

\def\FXv{{_F\scrX_v}}
\def\FYv{{_F\scrY_v}}

\def\guill#1{{\guillemotleft\,#1\,\guillemotright}}

\title[Stratification de Hesselink et formule des traces]
{Stratification de Hesselink rationnelle et contribution unipotente ˆ la formule des traces en caractŽristique $p>0$}

\author{Bertrand Lemaire \\ (\today)}
\email{Bertrand.Lemaire@univ-amu.fr}
\address{Institut de Math\'ematique de Marseille (I2M), Aix-Marseille Universit\'e (AMU), CNRS (UMR 7373), France}

\thanks{Je remercie vivement Jean-Pierre Labesse. C'est lui qui m'a patiemment 
expliquŽ la formule des traces et remis dans la voie quand je m'Žgarais. Sans son aide, ce texte n'aurait probablement jamais vu le jour. }

\begin{document}

\setcounter{tocdepth}{3}

\maketitle

\begin{abstract} 
Pour un corps commutatif quelconque $F$ et un groupe rŽductif connexe $G$ dŽfini sur $F$, on dŽveloppe 
une thŽorie de Kempf-Rousseau-Hesselink des $F$-strates unipotentes dans $G(F)$ qui devrait permettre d'attaquer des problmes ouverts 
en caractŽristique non nulle. 
En guise d'application, on utilise cette thŽorie pour Žcrire un dŽveloppement (prŽliminaire au dŽveloppement fin) de la contribution unipotente ˆ la formule des traces 
sur un corps global $F$ de caractŽristique $p>0$, sans restriction sur $p$ (\cad que l'on traite aussi les mauvais $p$). 
Les $F$-strates unipotentes dans $G(F)$ jouent le r™le des orbites gŽomŽtriques unipotentes 
dans le travail d'Arthur sur un corps de nombres. 
La dŽcomposition en termes de produits de distributions locales (le dŽveloppement fin) n'est pas abordŽe ici; elle fera l'objet d'un prochain article.
\end{abstract}

\begin{altabstract}
For any field $F$ and any connected reductive group $G$ defined over $F$, we develop a theory of Kempf-Rousseau-Hesselink unipotent $F$-strata in 
$G(F)$ that should allow to attack open problems in nonzero characteristic. As an application, we use this theory to write an expansion (preliminary to the fine expansion) of 
the unipotent contribution to the trace formula over a global field of characteristic $p>0$, whitout any restriction on $p$ (that is we also treat the bad $p$). 
The unipotent $F$-strata play the role of the 
unipotent geometric orbits in Arthur's work over a number field. The expansion in terms of products of local distributions (the fine expansion) is not discussed here;
it will be the subject of further work. 
\end{altabstract}

\subjclass{20G15, 14L24, 20G35, 11F72}

\keywords{instabilitŽ, co-caractre optimal, strate unipotente, formule des traces, contribution unipotente}

\tableofcontents

\section{Introduction}
\subsection{Les Žtapes du dŽveloppement}\label{les Žtapes du dŽveloppement}
Soit $G$ un groupe rŽductif connexe dŽfini sur un corps global $F$ de caractŽristique $p>0$. On a Žcrit 
en \cite[9]{LL} le dŽveloppement spectral fin ainsi que le dŽveloppement gŽomŽtrique grossier de la formule des traces tordue pour $G(\mbb{A})$ o $\mbb{A}$ est l'anneau 
des adles de $F$; \cad la formule des traces pour $(\wt{G}(\mbb{A}),\omega)$ o 
$\wt{G}$ est un $G$-espace tordu dŽfini sur $F$ (avec $\wt{G}(F)\neq \emptyset$) 
et $\omega$ est un caractre automorphe unitaire de $G(\mbb{A})$. 
Ce dŽveloppement gŽomŽtrique grossier s'exprime en termes des classes de $G(F)$-conjugaison des paires primitives $(\wt{M}, \delta)$ 
dans $\wt{G}(F)$, lesquelles jouent le r™le des 
classes de ss-conjugaison du cas des corps des nombres. Une telle classe $\mathfrak{o}=[\wt{M},\delta]$ Žtant fixŽe, 
pour une fonction $f\in C^\infty_{\rm c}(\wt{G}(\mbb{A}))$, le dŽveloppement fin de la contribution\footnote{$\bs{Y}_G = A_{\wt{G}}(\mbb{A})G(F)\backslash \wt{G}(\mbb{A})$ o $A_{\wt{G}}$ est le tore $F$-dŽployŽ maximal du centre de $\wt{G}$.} 
$$\mathfrak{J}^T_\mathfrak{o}(f)=\int_{\bsY_{G}}k^T_{\mathfrak{o}}(f\mathpvg x)\dd x$$ associŽe ˆ $\mathfrak{o}$ devrait comme dans le cas des corps de nombres 
comporter deux Žtapes principales: 
\begin{enumerate}
\item[(I)] une rŽduction (descente centrale) au centralisateur connexe $G_\delta(\mbb{A})$ --- ou 
au centralisateur stable, notion qu'il faudra de toutes faons dŽfinir si $\delta$ est insŽparable --- de $\delta$ dans $G(\mbb{A})$; 
\item[(II)] une description de la contribution unipotente ˆ formule des traces (non tordue) pour $G_\delta(\mbb{A})$.
\end{enumerate} On s'intŽresse ici ˆ l'Žtape (II). Plus prŽcisŽment, on s'intŽresse au dŽveloppement fin de la contribution 
$$\mathfrak{J}^T_\mathrm{unip}(f)= \mathfrak{J}^T_{\mathfrak{o}}(f)\quad\hbox{avec}\quad\wt{G}=G,\;\omega=1,\;\mathfrak{o}= [M_0,1]\ptf$$
Il s'agit d'adapter le travail d'Arthur \cite{A2} 
ˆ la caractŽristique positive. Rappelons les deux principales 
Žtapes de \textit{loc.~cit.} (pour les corps de nombres): 
\begin{enumerate}
\item[(II.1)] une dŽcomposition suivant les orbites gŽomŽtriques unipotentes $\ESO$ qui possdent un point $F$-rationnel; 
\item[(II.2)] pour tout ensemble fini $S$ de places de $F$, une dŽcomposition de la contribution associŽe ˆ $(\ESO,S)$ 
en termes des intŽgrales orbitales pondŽrŽes dŽfinies par les orbites gŽomŽtriques unipotentes $\ES{O}^M$ 
qui s'induisent ˆ $\ES{O}$ (pour l'application d'induction parabolique de Lusztig-Spaltenstein \cite{LS}); o $M$ est 
un $F$-facteur de Levi de $G$ . 
\end{enumerate} 
La dŽcomposition (II.2) exprime le dŽveloppement fin de la contribution associŽe ˆ $(\ESO,S)$ comme une combinaison linŽaire (finie) d'intŽgrales orbitales pondŽrŽes locales. 
Les coefficients de cette combinaison linŽaire sont dŽfinis par rŽcurrence gr‰ce ˆ la finitude du nombre de classes de $G(F_S)$-conjugaison dans $\ES{O}(F_S)$.
Ce sont des objets de nature locale-globale du type fonction Zta de Riemann partielle hors de $S$, en gŽnŽral trs difficiles 
ˆ calculer. 

Dans les annŽes 2015, Hoffmann \cite{Ho} a conjecturŽ une variante de (II.1) basŽe sur la forme finale de la dŽcomposition (II.2). 
La version \guill{algbre de Lie} de la variante de Hoffmann pour les corps de fonctions, \textit{i.e.} les corps globaux de caractŽristique $p>0$, a d'abord ŽtŽ dŽmontrŽe par 
Chaudouard et Laumon \cite{CL} pour le groupe $G= \mathrm{GL}_n$ et pour une fonction test trs simple. Chaudouard a ensuite dŽmontrŽ dans 
\cite{C}\footnote{La preuve de Chaudouard \cite{C} est 
essentiellement calquŽe sur celle de Chaudouard-Laumon \cite{CL} et donc naturellement gŽnŽralisable aux corps de fonctions, 
ce que nous faisons ici pour la version \guill{groupe} de la variante de Hoffmann.} la version 
\guill{algbre de Lie} de la variante de Hoffmann pour les corps de nombres et pour toutes les fonctions de Schwartz-Bruhat 
sur $\mathfrak{g}(\mbb{A})$; o $\mathfrak{g}=\mathrm{Lie}(G)$. Au mme moment, Finis et Lapid \cite{FL}\footnote{Il est probablement possible d'adapter les techniques de Finis-Lapid \cite{FL} 
au cas des corps de fonctions mais ce n'est pas l'approche que nous avons choisie.} ont dŽmontrŽ la version \guill{groupe} 
de la variante de Hoffmann pour les corps de nombres et pour des fonctions test bien plus gŽnŽrales 
que les fonctions lisses ˆ support compact. Observons que les dŽmonstrations de \cite{FL} et \cite{C} reprennent l'argument clŽ de \cite{CL}. 
La formule de Hoffmann permet dans certains cas particuliers (\textit{e.g.} \cite{HW,CL,C})  
de calculer explicitement les coefficients de l'Žtape (II.2); d'o son intŽrt. 

Nous traitons ici l'analogue de (II.1), prŽcisŽment la variante de Hoffmann, pour les corps de fonctions. 
L'analogue de (II.2) sera traitŽ ultŽrieurement. 

Pour $(\wt{M},\delta)$ comme dans l'Žtape (I), le centralisateur schŽmatique de $\delta$ peut ne pas tre lisse ---  
auquel cas on sort de la thŽorie des groupes algŽbriques linŽaires de Borel \cite{B} --- et mme s'il est lisse, il peut ne pas tre rŽductif. 
Il faudra donc ou bien raffiner la descente centrale (I) de manire ˆ se ramener ˆ la contribution unipotente dans un \guill{vrai} groupe algŽbrique linŽaire rŽductif 
connexe (quitte ˆ changer le corps de base); ou bien modifier le prŽsent travail ainsi que l'analogue de (II.2) 
de manire ˆ traiter une classe plus large de schŽmas en groupes sur un corps (\textit{e.g.} les groupes pseudo-rŽductifs de Conrad-Gabber-Prasad). L'Žtape 
(II.1) traitŽe ici est de toutes faons un passage obligŽ. Pour la distinguer du dŽveloppement fin (Žtapes (II.1) et (II.2)), nous appelerons dŽsormais 
cette Žtape (II.1) le \guill{prŽ-dŽveloppement fin}.

\subsection{\guill{Vrais} unipotents}\label{vrais unipotents}
Soit $\overline{F}$ une cl™ture algŽbrique de $F$. On note $F^\mathrm{s\acute{e}p}$, resp. $F^\mathrm{rad}$, la cl™ture sŽparable, 
resp. radicielle, de $F$ dans $\overline{F}$. Un ŽlŽment de $G(F)$ est appelŽ \textit{vrai unipotent}\footnote{Tits les appelle \guill{bons} \cite{Ti}; pour Žviter la confusion avec les \guill{bons $p$}, 
nous prŽfŽrons ici les appeler \guill{vrais}.}, ou \textit{$F$-unipotent}, s'il est contenu 
dans le radical unipotent d'un $F$-sous-groupe parabolique de $G$. On ne considre ici que les vrais unipotents de $G(F)$. Observons que si $G$ est $F$-anisotrope, tout ŽlŽment de $G(F)$ 
est \textit{($F$-)primitif} au sens de \cite{LL}\footnote{Les ŽlŽments primitifs de $G(F)$ sont ceux qui ne sont contenus dans aucun $F$-sous-groupe para\-bolique propre de $G$. 
Tits les appelle \guill{anisotropes} \cite{Ti}.} et l'ŽlŽment neutre est le seul (vrai) unipotent de $G(F)$. 
Insistons sur cette notion, au c\oe ur de la diffŽrence entre la caractŽristique $p>0$ et la caractŽristique nulle: 

\begin{exemple}\label{exemple insŽparabilitŽ PGL(p)}
\textup{Soit $\gamma$ un ŽlŽment primitif de $\mathrm{GL}_p(F)$ d'image $\overline{\gamma}$ dans $\mathrm{PGL}_p(F)$. 
Le polyn™me caractŽristique $\phi_\gamma$ de $\gamma$ est irrŽductible sur $F$ et on distingue deux cas: ou bien $\phi_\gamma$ est sŽparable, 
auquel cas il a $p$ racines distinctes (dans $F^\mathrm{s\acute{e}p}$) et $\overline{\gamma}$ est (absolument) semi-simple; 
ou bien il est insŽparable, auquel cas il a une racine de multiplicitŽ $p$ (dans $F^\mathrm{rad}$) et $\overline{\gamma}$ 
est un (vrai) unipotent de $\mathrm{PGL}_p(F^\mathrm{rad})$ qui est aussi primitif dans $\mathrm{PGL}_p(F^\mathrm{s\acute{e}p})$. Dans les deux cas, $\overline{\gamma}$ est primitif dans $\mathrm{PGL}_p(F)$ 
et il est traitŽ dans la formule des traces pour $\mathrm{PGL}_p(\mbb{A})$ comme 
un \guill{vrai} ŽlŽment semi-simple rŽgulier elliptique. }
\end{exemple}

On voit appara"tre dans l'exemple \ref{exemple insŽparabilitŽ PGL(p)} une autre notion clŽ en caractŽris\-tique $p>0$, gŽomŽtrique celle-ci, la 
sŽparabilitŽ: un ŽlŽment $x$ de $G$ est dit \textit{sŽparable} si le morphisme $G\rightarrow \mathrm{Int}_G(x),\, g \mapsto gxg^{-1}$ 
est sŽparable, ou ce qui revient au mme, 
si l'algbre de Lie du centralisateur rŽduit\footnote{C'est-ˆ-dire le centralisateur au sens de Borel \cite{B}. Observons que si $G= \mathrm{GL}_n$, le centralisateur schŽmatique 
de n'importe quel ŽlŽment de $G(F)$ est lisse, \textit{i.e.} gŽomŽtriquement rŽduit.} 
$\{g\in G\,\vert \, gxg^{-1}=x\}$ de $x$ dans $G$ co\"{\i}ncide avec le centrali\-sateur 
$\ker (\mathrm{Ad}_x-\mathrm{Id})$ de $x$ dans $\mathfrak{g}$. 
On sait que si $p \gg 1$\footnote{PrŽcisŽment si $p> 1$ est \textit{trs bon} pour $G$. Pour la notion de $p$ \guill{bon}, resp. \guill{trs bon}, pour $G$, on renvoie 
ˆ \ref{p bon et p trs bon}.}, tout ŽlŽment de $G$ est sŽparable. Observons 
que l'ŽlŽment $\gamma$ de $\mathrm{GL}_p(F)$ de l'exemple \ref{exemple insŽparabilitŽ PGL(p)} 
est sŽparable si et seulement si son polyn™me caractŽristique est sŽparable.

Soit $\UU$ l'ensemble des ŽlŽments unipotents de $G=G(\overline{F})$. 
C'est une sous-variŽtŽ algŽbrique fermŽe de $G$, dŽfinie sur $F$ et $G$-invariante pour la conjugaison. L'exemple \ref{exemple insŽparabilitŽ PGL(p)} montre 
que l'ensemble $\UU(F)$ des points $F$-rationnels de $\UU$ est en gŽnŽral plus gros que l'ensemble des vrais ŽlŽments unipotents de $G(F)$. 
D'autre part on sait d'aprs Lusztig \cite{L1} que $\UU$ est rŽunion \textit{finie} d'orbites gŽomŽtriques, \textit{i.e.} de $G$-orbites. 
Si $\ES{O}$ est une $G$-orbite dans $\UU$ qui possde un point $F$-rationnel 
et si $S$ est un ensemble fini de places de $F$, l'ensemble $\ES{O}(F_S)$ de ses points $F_S$-rationnels 
est rŽunion de $G(F_S)$-orbites; o $F_S = \prod_{v\in S}F_v$. En caractŽristique nulle ces $G(F_S)$-orbites sont toutes formŽes de vrais unipotents de $G(F_S)$  
et leur rŽunion est toujours finie ($\UU(F_S)$ est rŽunion finie de $G(F_S)$-orbites $F_S$-unipotentes). 
En revanche ici, $\ES{O}(F_S)$ peut contenir des $G(F_S)$-orbites qui ne sont pas $F_S$-unipotentes (\textit{cf.} \ref{exemple insŽparabilitŽ PGL(p)}) ou un 
nombre infini de $G(F_S)$-orbites $F_S$-unipotentes (\textit{e.g.} si $G=\mathrm{SL}_2$, $p=2$ et $S= \{v\}$). Observons que l'infinitŽ du nombre d'orbites rationnelles est Žtroitement reliŽ ˆ 
l'insŽparabilitŽ des orbites gŽomŽtriques (voir Morris \cite{M}).

Soit $\UF$ l'ensemble des (vrais) ŽlŽments unipotents de $G(F)$. C'est une rŽunion \textit{finie} (disjointe) de sous-ensembles $G(F)$-invariants appelŽs 
$F$-strates (voir \ref{la thŽorie de K-R-HG}). 
Une $F$-strate de $\UF$ est rŽunion (Žventuellement infinie) de $G(F)$-orbites. Tous les ŽlŽments ˆ l'intŽrieur d'une $F$-strate de $\UF$ partagent le mme invariant 
(classe de $G(F)$-conjugaison d'un co-caractre $F$-optimal normalisŽ, \textit{cf.} \ref{la thŽorie de K-R-HG}): c'est cet invariant qui dŽfinit la $F$-strate. 
Par produit, on dŽcoupe aussi l'ensemble $\UU_{F_S}= \prod_{v\in S} \UU_{F_v}$ en $F_S$-strates. Comme sur $F$, il n'y qu'un nombre \textit{fini} de $F_S$-strates mais 
une $F_S$-strate peut contenir un nombre infini de $G(F_S)$-orbites.

Toute $F$-strate de $\UF$ est l'intersection avec $\UF$ d'une $\overline{F}$-strate de $\UU$ (\textit{cf.} \ref{propriŽtŽs algŽbriques des F-lames}). Soit  
$[\UU]_F$ l'ensemble des $\overline{F}$-strates de $\UU$ qui intersectent non trivialement $\UF$. Soit aussi $(\UU/G)_F$ l'ensemble des $G$-orbites dans $\UU$ 
qui intersectent non trivialement $\UF$; et pour $\bsfrY \in [\UU]_F$, soit  
$(\bsfrY/G)_F$ l'ensemble des $G$-orbites dans $\bsfrY$ qui appartiennent ˆ $(\UU/G)_F$. On a la dŽcomposition  
$$\UF = \coprod_{\bsfrY\in [\UU]_F}\coprod_{\ES{O}\in (\bsfrY/G)_F} \ES{O}_F \quad \hbox{avec} \quad \ES{O}_F \bydef \UF \cap \ES{O}\ptf $$ 
Pour toute $\overline{F}$-strate $\bsfrY$ de $\UU$, il existe une unique $G$-orbite dans $\bsfrY$ qui soit dense dans $\bsfrY$; on la note $\ES{O}(\bsfrY)$. 
Observons que $\bsfrY \in [\UU]_F$ si et seulement si $\ES{O}(\bsfrY)\in (\UU/G)_F$. 

 Le dŽcoupage de $\mathfrak{U}$ en $G$-orbites utilisŽ par Arthur \cite{A2} est dans un premier temps 
 remplacŽ ici par le dŽcoupage de $\UF$ en $\coprod_{\ES{O} \in (\UU/G)_F} \ES{O}_F$. Pour chaque strate $\bsfrY\in [\UU]_F$, 
 on peut ensuite regrouper les contributions des $G$-orbites $\ES{O}\in (\bsfrY/G)_F$, l'idŽe sous-jacente Žtant que seule la $F$-strate 
 $\bsfrY_F\bydef \UF \cap \bsfrY \;(= \coprod_{\ES{O}\in (\bsfrY/G)_F}\ES{O}_F)$ 
pourra donner lieu ˆ un dŽveloppement en termes de produits d'objets locaux (Žtape II.2 de \ref{les Žtapes du dŽveloppement}): l'invariant associŽ 
ˆ $\bsfrY_F$ est ˆ voir comme un substitut de la thŽorie des $\mathfrak{sl}_2$-triplets (\textit{cf.} \ref{F-strates et orbites gŽomŽtriques}\,(ii)) utilisŽe en caractŽristique nulle (\textit{e.g.} \cite[7.1]{A2}).

\subsection{La thŽorie de Kempf-Rousseau-Hesselink}\label{la thŽorie de K-R-HG} 
Les rŽsultats contenus dans cette sous-section sont valables pour n'importe quel corps commutatif $F$. Soit $p\geq 1$ l'exposant caractŽristique de $F$.
\`A tout ŽlŽment unipotent $u$ de $ G= G(\overline{F})$, la thŽorie de Kempf-Rousseau \cite{K1,R}  
associe un co-caractre indivisible $\lambda$ de $G$ dit \textit{$u$-optimal}: la limite $\lim_{t\rightarrow 0} \mathrm{Int}_{t^\lambda} (u)$ 
existe\footnote{Au sens o le morphisme de variŽtŽs algŽbriques 
$\mbb{G}_\mathrm{m} \rightarrow V,\, t \mapsto \mathrm{Int}_{t^\lambda}(u)$ 
se prolonge (de manire unique) en un morphisme de variŽtŽs algŽbriques $\mbb{G}_\mathrm{a}\rightarrow V$; 
la limite est par dŽfinition la valeur en $0$ de ce prolongement.}  
et vaut $1=e_G$ (l'ŽlŽment neutre de $G$); et $\mathrm{Int}_{t^\lambda}(u)$ tend vers $1$ \guill{plus vite} 
que pour tout autre co-caractre $\lambda'\in \check{X}(G)$ tel que la limite $\lim_{t\rightarrow 0} \mathrm{Int}_{t^{\lambda'}} (u)$ existe. On renvoie ˆ \ref{le thm de KR} pour une dŽfinition prŽcise de 
la notion d'optimalitŽ. Ce co-caractre $\lambda$ n'est pas unique mais l'ensemble $\Lambda_u^\mathrm{opt}$ des co-caractres $u$-optimaux 
forme une seule orbite sous l'action du sous-groupe parabolique $P_\lambda$ de $G$ associŽ ˆ $\lambda$; 
en particulier $P_\lambda$ ne dŽpend que de $u$. \`A tout $\lambda \in \check{X}(G)$ tel que $\lim_{t\rightarrow 0} \mathrm{Int}_{t^\lambda}(u)=1$ est associŽ un 
\textit{niveau} $m_u(\lambda)$ qui est un entier $>0$. Il vŽrifie $$m_{pup^{-1}}( \lambda)= m_u(\lambda) \quad \hbox{pour tout}\quad p\in P_\lambda\,;$$ 
en particulier $m_u(\lambda)$ ne dŽpend pas de $\lambda\in \Lambda_u^{\mathrm{opt}}$. 

Pour $u\in \UF$, 
on a une version rationnelle de cette thŽorie, qui consiste ˆ ne tester que les co-caractres $\lambda$ qui sont $F$-rationnels. 
On dŽfinit de la mme manire la notion de \textit{$(F,u)$-optimalitŽ} et le sous-ensemble 
(non vide) $\Lambda_{F,u}^\mathrm{opt}$ de $\check{X}_F(G)$. Comme dans le cas gŽomŽtrique, l'ensemble $\Lambda_{F,u}^\mathrm{opt}$ 
forme une seule orbite sous l'action de $P_\lambda(F)$ pour un (\textit{i.e.} pour tout) $\lambda \in \Lambda_{F,u}^\mathrm{opt}$; 
en particulier $P_\lambda$ est un $F$-sous-groupe parabolique de $G$ qui ne dŽpend que de $u$. De mme le niveau $m_u(\lambda)$ ne dŽpend pas de $\lambda \in \Lambda_{F,u}^\mathrm{opt}$. 
On note $\bs{\Lambda}_{F,u}$ l'ensemble des co-caractres virtuels $(F,u)$-optimaux \textit{normalisŽs}\footnote{On renvoie ˆ \ref{la variante de Hesselink} (\S3) pour la dŽfinition de $\check{X}_F(G)_{\mbb{Q}}$.} : 
$$\bs{\Lambda}_{F,u}=\left\{m_u(\lambda)^{-1}\lambda\,\vert\, \lambda \in \Lambda_{F,u}^\mathrm{opt}\right\}\subset \check{X}_F(G)_{\mbb{Q}}\ptf$$ 
\`A la suite de Hesselink \cite{H2}, on pose 
$$\scrY_{F,u} = \{u' \in \UF\,\vert \, \bs{\Lambda}_{F,u'} = 
\bs{\Lambda}_{F,u}\}\quad \hbox{et} \quad \bsfrY_{F,u} = \mathrm{Int}_{G(F)}(\scrY_{F,u})\ptf\leqno{(1)}$$ 
Les ensembles $\scrY_{F,u}$, resp. $\bsfrY_{F,u}$, sont appelŽs $F$-lames, resp. $F$-strates (de $\UU_F$). Les $F$-strates 
sont donc les classes de $G(F)$-conjugaison de $F$-lames. Observons que mme pour $F=\overline{F}$, une $\overline{F}$-strate (de $\UU= \UU_{\smash{\overline{F}}}$) peut contenir plusieurs 
orbites gŽomŽtriques unipotentes (\textit{cf.} l'exemple de $\mathrm{Sp}_4$ en \ref{exemples en basse dimension}). Le point clŽ est que $\UF$ est rŽunion \textit{finie} de $F$-strates. Pour $u\in \UF$ et $\lambda\in \Lambda_{F,u}^{\mathrm{opt}}$, la $F$-lame $\scrY_{F,u}$ est $P_\lambda(F)$-invariante 
(c'est une consŽquence de la 
$P_\lambda(F)$-invariance de $\bs{\Lambda}_{F,u}$). De plus l'Žtude des classes de $G(F)$-conjugaison dans $\bsfrY_{F,u}$ se ramne ˆ celle des classes 
de $P_\lambda(F)$-conjugaison dans $\scrY_{F,u}$: l'application naturelle $G(F)\times^{P_\lambda(F)}\scrY_{F,u}\rightarrow \bsfrY_{F,u}$ est bijective.

La thŽorie de Kempf-Rousseau-Hesselink se comporte trs bien par extension sŽparable  
(algŽbrique ou non) du corps de base. Soit $E/F$ une telle extension, avec $(E^{\mathrm{s\acute{e}p}})^{\mathrm{Aut}_F(E^{\mathrm{s\acute{e}p}})}=F$ (cette condition est toujours rŽalisŽe si 
$E/F$ est algŽbrique ou de degrŽ de transcendance infini). 
Pour tout $u\in \UF$, on a l'ŽgalitŽ
$$\bs{\Lambda}_{F,u} = \check{X}_F(G)_{\mbb{Q}} \cap \bs{\Lambda}_{E,u}\ptf$$ On en dŽduit que
$$\scrY_{F,u}= G(F) \cap \scrY_{E,u} \quad \hbox{et} \quad \bsfrY_{F,u}= G(F) \cap \bsfrY_{E,u}\ptf$$ 
Si $F$ est parfait (\textit{e.g.} si $p=1$ ou si $F$ est un corps fini), par descente galoisienne cela ramne la thŽorie des $F$-lames et des $F$-strates ˆ celle sur $\overline{F}$. 
Observons que si $F$ est un corps global, le complŽtŽ $F_v$ de $F$ en une place $v$ est une extension 
sŽparable de $F$ (de degrŽ de transcendance infini). On peut dans les ŽgalitŽs ci-dessus remplacer $E$ par $F_v$ ou par $F_S$ 
pour un sous-ensemble fini $S$ de places de $F$. 

\begin{remark}\label{F-strates et orbites gŽomŽtriques}
\textup{\begin{enumerate}
\item[(i)]Si $p\geq 1$ est \textit{bon} pour $G$ (\textit{cf.} \ref{p bon et p trs bon}), les $\overline{F}$-strates de $\UU=\UU_{\smash{\overline{F}}}$ sont exactement les orbites gŽomŽtriques unipotentes. 
En particulier si $p=1$, une $F$-strate de $\UF$ n'est autre que l'ensemble $\ESO(F)$ des points $F$-rationnels d'une $G$-orbite $\ESO$ de $\UU$ dŽfinie sur $F$ et telle que $\ESO(F)\neq \emptyset$. Cette propriŽtŽ reste vraie si $p>1$ est \textit{trs bon} pour $G$ ou si $G= \mathrm{GL}_{n/F}$. Si $p>1$ est seulement bon pour $G$ (\textit{e.g.} si $p>5$), alors pour tout 
$u\in \UF$, on a $\bsfrY_{F,u}= \UF \cap \ES{O}_u$. Par exemple l'ŽlŽment $\overline{\gamma}$ de l'exemple \ref{exemple insŽparabilitŽ PGL(p)} est primitif dans $\mathrm{PGL}_p(F)$ mais il est conjuguŽ 
dans $\mathrm{PGL}_p(F^{\rm rad})$ ˆ un ŽlŽment unipotent rŽgulier de $\mathrm{PGL}_p(F)$; dans ce cas la $\mathrm{PGL}_p$-orbite unipotente rŽgulire de 
$\mathrm{PGL}_p$ est une $\overline{F}$-strate de $\UU$ qui contient la $F$-strate rŽguliere de $\UF$ ainsi que des ŽlŽments primitifs de $\mathrm{PGL}_p(F)$.
\item[(ii)] Si $p=1$ ou si $p\gg 1$, la thŽorie des $\mathfrak{sl}_2$-triplets de Jacobson-Morosov associe ˆ tout ŽlŽment $u\in \UU_F$ un $F$-morphisme de groupes $\epsilon: \mathrm{SL}_{2/F}\rightarrow G_{\mathrm{der}}$ avec  
$\epsilon\left(\begin{array}{cc}1& 1\\ 0 & 1\end{array}\right)= u$, tel que le $\mathrm{SL}_{2/F}$-module $(\mathfrak{g}, \mathrm{Ad}_{\epsilon(\cdot)})$ soit 
infinitŽsimalement complŽtement rŽductible sur $F$; de plus le co-caractre $\lambda_\epsilon$ de $G$ dŽfini par $t^{\lambda_\epsilon} = \epsilon \left( \begin{array}{cc}t& 0\\ 0 & t^{-1}\end{array}\right)$ vŽrifie 
$ \frac{1}{2} \lambda_\epsilon \in \bs{\Lambda}_{F,u}$ (\textit{cf.} \ref{du groupe ˆ l'algbre de Lie}, \S3).  
\item[(iii)]Si $p$ n'est pas bon pour $G$, ce qui implique que $p\in \{2,3,5\}$, une ($\overline{F}$-)strate de $\UU$ est en gŽnŽral composŽe de plusieurs 
$G$-orbites; et pour $u\in \UF$, la $F$-strate $\bsfrY_{F,u}$ n'est pas en gŽnŽral Žgale ˆ $G(F)\cap \bsfrY_u$, ni mme ˆ $\UF \cap \bsfrY_u$. 
\end{enumerate}
}\end{remark}

\subsection{Sur les propriŽtŽs algŽbriques des $F$-lames}\label{propriŽtŽs algŽbriques des F-lames}La thŽorie des $F$-strates de $\UF$ s'insre dans le cadre plus gŽnŽral d'une $G$-variŽtŽ (algŽbrique) affine pointŽe 
$(V,e_V)$ dŽfinie sur $F$, avec $e_V\in V(F)$; \cad que $V$ est une $G$-variŽtŽ affine dŽfinie sur $F$ et $e_V$ est un point ($F$-rationnel) $G$-invariant de $V$. 
Supposons de plus que $e_V$ soit rŽgulier dans $V$ (hypothse \ref{hyp reg}); c'est toujours le cas si $V$ est un groupe algŽbrique linŽaire (\textit{e.g.} $V=G$ ou $V=\mathfrak{g}$). L'ensemble 
$\UF$ correspond au sous-ensemble $\NF\subset V(F)$ formŽ des ŽlŽments \textit{$F$-instables}; \cad les $v\in V(F)$ tels que $\lim_{t\rightarrow 0}t^\lambda\cdot v=e_V$ pour un $\lambda\in \check{X}_F(G)$. 
Pour $v\in \NF$, on dŽfinit comme en \ref{la thŽorie de K-R-HG} le sous-ensemble $\bs{\Lambda}_{F,v}\subset \check{X}_F(G)_{\mbb{Q}}$, la $F$-lame $\scrY_{F,v}= \{v'\in \NF\,\vert\, \bs{\Lambda}_{F,v'}=\bs{\Lambda}_{F,v}\}$ 
et la $F$-strate $\bsfrY_{F,v}=G(F)\cdot \scrY_{F,v}$ de $\NF$. 

Pour $F=\overline{F}$, 
les propriŽtŽs des $\overline{F}$-lames et des $\overline{F}$-strates de $\ES{N}_{\smash{\overline{F}}}$ ont ŽtŽ dŽcrites par Hesselink \cite{H2}; en particulier ce sont des sous-variŽtŽs localement fermŽes dans $V$. 
Pour allŽger l'Žcriture, on supprime l'indice $\overline{F}$ dans les notations: pour $v\in \ES{N}= \ES{N}_{\smash{\overline{F}}}$, on pose $\bs{\Lambda}_v= \bs{\Lambda}_{\smash{\overline{F}},v}$, $\scrY_v= \scrY_{\,\smash{\overline{F}},v}$, etc.

Revenons ˆ $F$ quelconque. Pour $v\in \NF$ tel que 
l'intersection $\check{X}_F(G)_{\mbb{Q}}\cap \bs{\Lambda}_{v}$ soit non vide, on a l'ŽgalitŽ $\scrY_{F,v}= \scrY_{v}(F)$. 
En revanche si $\check{X}_F(G)_{\mbb{Q}}\cap \bs{\Lambda}_{v}= \emptyset$, la $F$-lame $\scrY_{F,v}$ de $\NF$ n'a 
\textit{a priori} aucune structure algŽbrique raisonnable. 
Cela nous amne ˆ introduire l'hypothse suivante:
\begin{enumerate}[leftmargin=1.15cm]
\item[(\ref{hypo bonnes F-strates faible})]
pour tout $v\in \NF$, on a $\check{X}_{F}(G)_{\mbb{Q}}\cap \bs{\Lambda}_v \neq \emptyset$. 
\end{enumerate} 
Si $F$ est parfait (\textit{e.g.} si $p=1$), cette hypothse est toujours vŽrifiŽe par descente sŽparable. 
Si l'hypothse \ref{hypo bonnes F-strates faible} est vŽrifiŽe, alors pour tout $v\in \NF$, on a les ŽgalitŽs $$\scrY_{F,v}= \scrY_v \quad \hbox{et} \quad \bsfrY_{F,v}= \NF \cap \bsfrY_{v}\ptf$$

Dans le cas o $V$ est un $G$-module dŽfini sur un corps global $F$ (avec $e_V=0$), on donne une condition suffisante pour que l'hypothse \ref{hypo bonnes F-strates faible} soit vŽrifiŽe 
(\textit{cf.} \ref{cor 3 bonnes F-strates}). Par passage au complŽtŽ 
$E_w$ d'une extension sŽparable finie $E$ de $F$ en une place finie $w$ de $E$ (qui est une extension sŽparable de $F$), on se ramne au cas o $F$ est un corps local non archimŽdien et $G$ est 
$F$-dŽployŽ. Alors $G\simeq_F \scr{G}\times_{\mbb{Z}}F$ pour un $\mbb{Z}$-groupe rŽductif de Chevalley-Demazure $\scr{G}$. Supposons que l'action de $G$ sur $V$ 
provienne par le changement de base $\mbb{Z} \rightarrow F$ d'une action $\mbb{Z}$-linŽaire de $\scr{G}$ sur un $\mbb{Z}$-espace affine $\scr{V}\;(\simeq_{\mbb{Z}}\mbb{A}_{\mbb{Z}}^n)$. On prouve alors l'implication 
$$\ES{N}_F \neq V(F) \Rightarrow \ES{N} \neq V\ptf$$ Autrement dit s'il existe un ŽlŽment de $V(F)$ qui soit $F$-semi-stable (\cad qui ne soit pas $F$-instable), alors 
il existe un ŽlŽment de $V=V(\overline{F})$ qui soit $\overline{F}$-semi-stable. C'est Žvidemment toujours vrai si $F$ est une extension finie du corps $p$-adique $\mbb{Q}_p$ (par descente sŽparable). 
On en dŽduit le rŽsultat pour $F \simeq \mbb{F}_q((\varpi))$ par la mŽthode des corps proches (qui permet de passer de $F$ ˆ une extension finie de $\mbb{Q}_p$), 
gr‰ce au thŽorme de Seshadri \cite{Se} (qui permet de passer de $\overline{\mbb{Q}}_p$ ˆ $\overline{F}$). 
Si maintenant $v\in \ES{N}_F \smallsetminus \{0\}$, $\lambda\in \Lambda_{F,v}^{\mathrm{opt}}$ 
et $k=m_v(\lambda)$, le critre de Kirwan-Ness rationnel (\ref{thm de KN rationnel}) assure que l'image de $v$ dans le $M_\lambda$-module $V_\lambda(k)$ est $(F,M_\lambda^\perp)$-semi-stable; 
on renvoie ˆ \ref{le critre de KN rat} pour les dŽfinitions. Comme l'action de $M_\lambda$ sur $V_\lambda(k)$ provient elle aussi par le changement de base $\mbb{Z}\rightarrow F$ d'une 
action $\mbb{Z}$-linŽaire d'un $\mbb{Z}$-groupe rŽductif de Chevalley-Demazure $\scr{M}_\lambda$ sur un $\mbb{Z}$-espace affine $\scr{V}_\lambda(k)\;(\simeq_{\mbb{Z}}\mbb{A}_{\mbb{Z}}^{n'})$, on en dŽduit que le co-caractre virtuel 
$\frac{1}{k}\lambda$ appartient 
ˆ $\check{X}_F(G)_{\mbb{Q}} \cap \bs{\Lambda}_{v}$. 

Ce rŽsultat s'applique au cas de l'action de $G$ sur lui-mme par conjugaison (par passage au $G$-module $\mathfrak{g}= \mathrm{Lie}(G)$). 
On obtient que pour un groupe rŽductif connexe $G$ dŽfini sur un corps local ou un corps global $F$, on a 
$$\check{X}_F(G)_{\mbb{Q}} \cap \bs{\Lambda}_{u}\neq \emptyset\quad \hbox{pour tout} \quad u\in \UF\ptf$$ 
Cela entra"ne que l'application $\bsfrY \mapsto \bsfrY_F =\UF \cap \bsfrY$ est une bijection entre $[\UU]_F$ et les $F$-strates de $\UF$. 

\begin{remark}\label{remarque sur la structure algŽbrique des F-lames}
\textup{
Insistons sur l'importance des $F$-lames dans cette approche. Prenons le cas de $G=\mathrm{PGL}_{2/F}$ en caractŽristique $2$. Notons 
$u$ l'image de $\left(\begin{array}{cc}1 & 1\\ 0 & 1\end{array}\right)$ dans $G(F)$. 
L'ensemble $\bsfrY_{F,u}=\UF \smallsetminus \{1\}$ est l'unique $F$-strate non triviale de $\UF$. Elle est formŽe des images dans $G(F)$ des ŽlŽments 
$\gamma \in \mathrm{GL}_2(F)$ tels que $\mathrm{trace}(\gamma)=0$ et $\det(\gamma)\in (F^\times)^2$; elle ne peut donc pas tre l'ensemble des points 
$F$-rationnels d'une variŽtŽ algŽbrique dŽfinie sur $F$. En revanche $\scrY_{F,u}= \scrY_u(F)$ est l'ensemble des images dans $G(F)$ des matrices $\left(\begin{array}{cc}1 & x\\ 0 & 1\end{array}\right)$ 
avec $x\in F^\times$. 
 }
\end{remark}

\subsection{Comparaison avec la thŽorie sur $\mbb{C}$} Rappelons que $p\geq 1$ est l'exposant caractŽristique de $F$. 
Si $p=1$ ou $p>1$ est \textit{bon} pour $G$ (\textit{cf.} \ref{p bon et p trs bon}), la thŽorie sur $\overline{F}$ est pratiquement la mme que celle sur 
$\mbb{C}$: les ($\overline{F}$-)strates de $\UU$ co\"{\i}ncident avec les $G$-orbites de $\UU$, et elles sont en bijection naturelle avec les 
$G'$-orbite de $\UU'=\UU_{\,\mbb{C}}^{G'}$ o $G'$ est \guill{le} $\mbb{C}$-groupe algŽbrique rŽductif connexe correspondant ˆ $G$, \textit{i.e.} $G'= \scr{G}\times_{\mbb{Z}}\mbb{C}$ pour un $\mbb{Z}$-groupe de Chevalley-Demazure $\scr{G}$ tel que $\scr{G}\times_{\mbb{Z}} \overline{F}\simeq  G\times_F\overline{F}$. Cette constatation est ˆ l'origine de l'introduction des \guill{morceaux unipotents} de Lusztig: un morceau unipotent est une rŽunion finie de $G$-orbites de $\UU$ censŽ correspondre ˆ une $G'$-orbite de $\UU'$. La dŽfinition des morceaux unipotents a ŽvoluŽ au cours du temps pour aboutir ˆ sa version dŽfinitive, les morceaux unipotents de Lusztig-Xue. On renvoie ˆ \ref{strates et orbites} (\S4) pour des prŽcisions. Les travaux de Lusztig et Xue puis, plus rŽcemment, de Voggesberger et Premet, permettent d'affirmer que ces morceaux unipotents de Lusztig-Xue sont exactement les strates de $\UU$. Pour les groupes de type exceptionnel (${\bf E}_6$, ${\bf E}_7$, ${\bf E}_8$, ${\bf F}_4$ ou ${\bf G}_2$) il nous a cependant semblŽ nŽcessaire de vŽrifier cette affirmation. PrŽcisŽment, nous dŽmontrons que si $\bsfrY$ est une strate 
de $\UU$, le bord $\overline{\bsfrY}\smallsetminus  \bsfrY$ est bien une rŽunion de strates. Les spŽcialistes des orbites unipotentes en mauvaise caractŽristique trouverons probablement cette dŽmonstration inutile; nous nous en excusons ˆ l'avance. Au bout du compte, on a une bijection canonique entre les strates de $\UU$ et les $G'$-orbites de $\UU'$, qui peut s'exprimer 
comme suit. 

Choisissons un $\mbb{Z}$-tore dŽployŽ maximal $\scr{T}$ de $\scr{G}$. Notons $T$ 
le tore maximal $\scr{T}\times_{\mbb{Z}}\overline{F}$ de $G\times_F \overline{F}$ (identifiŽ ˆ $\scr{G}\times_F \overline{F}$) et $T'$ le tore maximal $\scr{T}\times_{\mbb{Z}}\mbb{C}$ de $G'$. 
On a les identifications canoniques $\check{X}(T)=\check{X}(T')$ et $X(T)=X(T')$. De mme, en notant $W=W^{G}(T)$ le groupe de Weyl de $G$ et $W=W^{G'}(T')$ celui de $G'$, on a une identification canonique $W=W'$. Soit $D_{G'}\subset \check{X}(G)$ le sous-ensemble formŽ des co-caractres $\lambda'$ pour lesquels il existe 
un $\mbb{C}$-morphisme de groupes algŽbriques $\epsilon': \mathrm{SL}_{2/\mbb{C}} \rightarrow G$ vŽrifiant $t^{\lambda'} = \epsilon' \left(\begin{array}{cc} t & 0 \\ 0 & t^{-1}\end{array}\right)$. 
L'application qui ˆ $\lambda' \in D_{G'}$ associe l'ŽlŽment $\epsilon' \left(\begin{array}{cc} 1 & 1 \\ 0 & 1\end{array}\right)\in \UU'$ induit une bijection entre l'ensemble $D_{G'}/G'$ des classes de $G'$-conjugaison dans $D_{G'}$ et l'ensemble $\UU'/G'$ des $G'$-orbites dans $\UU'$. Puisque $\check{X}(G')/G' = \check{X}(T')/W'$ s'identifie ˆ $\check{X}(G)/G= \check{X}(T)/W$, l'ensemble $D_{G'}/G'$ s'identifie ˆ un sous-ensemble de $\check{X}(G)/G$, notŽ $D_G/G$. Pour $\blacktriangle \in D_G/G$, 
en notant $\wt{\blacktriangle}$ la $G$-orbite $ \{\frac{1}{2}\lambda \,\vert \, \lambda \in \blacktriangle \}$ dans $\check{X}(G)_{\mbb{Q}}$, l'ensemble 
$$\bsfrY(\wt{\blacktriangle})= \{u\in \UU \,\vert \, \bs{\Lambda}_u \subset \wt{\blacktriangle}\}$$ est une strate de $\UU$. L'application qui ˆ $\blacktriangle \in D_G/G$ associe 
$\bsfrY(\wt{\blacktriangle})$ est une bijection entre $D_G/G$ et les strates de $\UU$. Via l'identification $D_G/G=D_{G'}/G'$, on obtient bien une  
bijection $\bsfrY \mapsto \ES{O}'(\bsfrY)$ entre les strates de $\UU$ et les $G'$-orbites de $\UU'$. Pour une strate $\bsfrY$ de $\UU$, notons $\ES{O}(\bsfrY)$ l'unique $G$-orbite de $\UU$ qui soit (contenue et) dense dans $\bsfrY$. On peut voir $\ES{O}(\bsfrY)$ comme 
la $G$-orbite de $\UU$ correspondant ˆ la $G'$-orbite $\ES{O}'(\bsfrY)$. 

\begin{remark}
\textup{
Si $G$ est quasi-simple et $p$ n'est pas \textit{bon} pour $G$ (\textit{cf.} \ref{p bon et p trs bon}), ce qui entra"ne que $p\in \{2,3,5\}$, ou bien 
$G$ est de type ${\bf B}_n$, ${\bf C}_n$ ou ${\bf D}_n$ et $p=2$, auquel cas toute strate de $\UU$ est composŽe d'un nombre fini (Žgal ˆ $2^m$ pour un $m\in \mbb{N}$) 
de $G$-orbites (\textit{cf.} \cite[I, 2.6]{Spa}); ou bien $G$ est de type ${\bf E}_6$, ${\bf E}_7$, ${\bf E}_8$, ${\bf F}_4$ ou ${\bf G}_2$, 
auquel cas toute strate de $\UU$ est constituŽe d'une ou 
deux $G$-orbites (\textit{cf.} \cite[I, 2.13, page 29]{Spa}). }
\end{remark}

\subsection{Induction parabolique} 
Sur un corps commutatif algŽbriquement clos $F=\overline{F}$, Lusztig et Spaltenstein \cite{LS} ont dŽfini une application d'induction 
parabolique pour les orbites gŽomŽtriques unipotentes d'un facteur Levi $M$ de $G$. Pour un sous-groupe parabolique $P$ de $G$ de composante 
de Levi $M$ et une $M$-orbite unipotente $\ES{O}^M$ de $M$, l'induite parabolique $\ES{J}_P^{G}(\ES{O}^M)$ de $\ES{O}^M$ est par dŽfinition l'unique $G$-orbite 
unipotente $\ES{O}$ de $G$ telle que l'intersection $\ES{O} \cap \ES{O}^M U_P$ soit ouverte dense dans $\ES{O}^MU_P$. Cette application $\ES{J}_P^{G}$ ne dŽpend pas de $P$; on la note aussi $\ES{J}_M^G$. 

Pour $w\in \UU^M$, l'ensemble $\bsfrY_w^MU_P$ est une sous-variŽtŽ fermŽe irrŽductible de $\UU$. Comme il n'y a qu'un nombre fini de ($\overline{F}$-)strates de $\UU$ et qu'elles sont localement fermŽes dans $\UU$, il en existe une et une seule qui intersecte $\bsfrY_w^M U_P$ de manire dense; 
on la note $I_P^G(w)$. L'intersection $I_P^G(w)\cap \bsfrY_w^MU_P$ est alors ouverte dense dans $\bsfrY_w^M U_P$. 
Puisque la $M$-orbite 
$\ES{O}(\bsfrY^M_w)$ est ouverte dense dans $\bsfrY^M_w$, on a forcŽment l'inclusion $\ES{J}_M^{G}(w) \subset I_P^G(w)$. Cela entra"ne que 
l'application $I_P^G$, tout comme l'application $\ES{J}_P^{G}$, ne dŽpend pas du sous-groupe parabolique $P$ de composante de Levi $M$; on peut donc la noter $I_M^G$. 
On prouve que $$\ES{J}_{M}^{G}(\ES{O}(\bsfrY^M_w)) = \ES{O}(I_M^G(w))\ptf$$ 
Cette ŽgalitŽ traduit la compatibilitŽ de l'application $\ES{J}_M^{G}$ avec la correspondance ce Springer. 
En revanche pour $w' \in  \bsfrY^M_w \smallsetminus \ES{O}(\bsfrY^M_w)$, l'orbite induite $\ES{J}_M^{G}(w')$ n'est pas forcŽment contenue dans 
la strate induite $I_M^G(w)=I_M^G(w')$.

Pour $F$ quelconque avec $M$ et $P$ dŽfinis sur $F$, on peut dŽfinir une application similaire entre les $F$-strates de $\UF^M$ et celles de $\UF$, en supposant l'hypothse 
\ref{hypo bonnes F-strates faible} vŽrifiŽe pour $G$. Supposons qu'elle le soit aussi pour $M$, \textit{e.g.} si $F$ est un corps fini, un corps local ou un corps global. Cela assure que pour tout $w\in \UF^M$, l'intersection $\UF \cap I_M^G(w)$ est non vide et ne dŽpend que de la $F$-strate $\bsfrY_{F,w}^M\;(=\UF^M\cap \bsfrY^M_w)$ de $\UF^M$; c'est donc une $F$-strate de $\UF$, que l'on note $I_{F,M}^G(w)\;(= \UF \cap I_M^G(w))$. 

\begin{remark}
\textup{Supposons que $p\geq 1$ soit \textit{bon} pour pour $G$. Alors $p$ est bon pour $M$. Dans ce cas les strates de $\UU^M$, resp. de $\UU$, co\"{\i}ncide avec 
les $M$-orbites, resp. les $G$-orbites, unipotentes et l'application 
$I_{M}^G$ co\"{\i}ncide avec $\ES{J}_M^{G}$. On en dŽduit que 
pour tout $w\in \UF^M$, on a $\bsfrY^M_w = \UF^M \cap \ES{O}_w^M$ et $I_{F,M}^G(w) = \UF \cap \ES{J}_M^{G}(w)$. }
\end{remark}

\subsection{Application ˆ la formule des traces}\label{application ˆ la FdT} On reprend sans les rŽintroduire les notations de \cite{LL}. 
En particulier $F$ est un corps global de caractŽristique $p>1$. 

Pour un paramtre $T\in \ag_0$ assez rŽgulier, soit $$\mathfrak{J}^T_{\u}(f)= \mathfrak{J}^T_{\mathfrak{o}}(f)\quad \hbox{avec} \quad \mathfrak{o} =[M_0,1]$$ 
la contribution unipotente ˆ la formule des traces (non tordue). 
Elle est est donnŽe par 
la formule intŽgrale \index{XbsG@$\overline{\bs{X}}_G$}
$$\mathfrak{J}^T_\u(f)= \int_{\overline{\bs{X}}_G}k^T_{\u}(x)\dd x \vgq \overline{\bs{X}}_G = A_G(\mbb{A})G(F)\backslash G(\mbb{A})\vg$$ 
o $k^T_\u(x)= k^T_\u(f\mathpvg x)$ est la fonction dŽfinie par
$$k^T_\u(x)= \sum_{P\in \ESP_\mathrm{st}}(-1)^{a_P-a_G}\sum_{\xi \in P(F)\backslash G(F)} \wh{\tau}_P({\bf H}_0(\xi x)-T)k_{P,\u}(\xi x)$$ avec 
$$k_{P,\u}(x)= \sum_{\eta\in \UU^{M_P}(F)}\int_{U_P(\mbb{A})} f(x^{-1}\eta u x) \dd u\ptf$$ De plus, la fonction $T\mapsto \mathfrak{J}^T_{\u}(f)$ dŽfinit un 
ŽlŽment de PolExp. Comme le fait Arthur dans \cite{A2}, 
on commence par rŽcrire $\mathfrak{J}^T_{\u}(f)$ 
pour $T\in \ag_0$ suffisamment rŽgulier: il existe une constante $c(f)>0$ (ne dŽpendant que du support de $f$) 
telle que pour $\bs{d}_0(T)\ge c(f)$, on ait\footnote{Observons que la formule est \guill{exacte}. 
L'analogue de cette formule sur un corps de nombre est une intŽgrale asymptotique (en $T$). } 
$$  \mathfrak{J}^T_\mathrm{unip}(f) =  \int_{\overline{\bsX}_G} F_{P_0}^G(x,T)k_\u(x) \dd x \quad \hbox{avec} 
\quad k_\u(x) = \sum_{\eta\in \UF}f(x^{-1}\eta x);$$ 
o $F^G_{P_0}(\cdot,T)$ est la fonction caractŽristique d'un sous-ensemble compact de $\overline{\bs{X}}_G$.  
Cela prouve que l'intŽgrale est absolument convergente.

Rappelons que pour $\ES{O}\in (\UU/G)_F$, on a posŽ $\ES{O}_F= \UF\cap \ES{O}$. Pour $\ES{O}\in (\UU/G)_F$, on souhaite approximer l'intŽgrale (absolument convergente) 
$$\int_{\overline{\bsX}_{G}} F^G_{P_0}(x,T)k_{\ES{O}}(x) \dd x\quad \hbox{avec} \quad 
k_{\ES{O}}(x)= \sum_{\eta \in \ES{O}_F} f(x^{-1}\eta x)\ptf$$ Pour cela, on commence par dŽfinir pour chaque $P\in \ES{P}_\mathrm{st}$ une fonction  
$$k_{P,\ES{O}}(x)= \sum_{\eta \in (\ES{O}_F \cap P(F))U_P(F)/U_P(F)} \int_{U_P(\mbb{A})} f(x^{-1}\eta u x) \dd u \ptf$$ 
En notant $\xi_{P,\ES{O}}$ la fonction sur $\UF\cap P(F)=\UF^{M_P}U_P(F)$ dŽfinie par\index{xiPYeta@$\xi_{P,\ES{O}}(\eta)$} $$\xi_{P,\ES{O}}(\eta)= \left\{\begin{array}{ll}
1 & \hbox{si $\ES{J}_P^{G}(\eta)\subset \ES{O}$}\\
0 & \hbox{sinon}
\end{array}
\right.,$$ on a donc 
$$k_{P,\ES{O}}(x) = \sum_{\eta \in \UF^{M_P}}\xi_{P,\ES{O}}(\eta)\int_{U_P(\mbb{A})} f(x^{-1}\eta u x) \dd u \ptf$$ 
Observons que si la $G$-orbite $\ES{O}$ est \textit{$F$-rigide} (\ref{orbites et strates F-rigides}), alors $k_{P,\ES{O}}=0$ pour tout $P\neq G$. 
Pour $T\in \ag_0$, on pose 
$$k_{\ES{O}}^T(x)=\sum_{P\in \ESP_\mathrm{st}}(-1)^{a_P-a_G} \sum_{\xi \in P(F)\backslash G(F)} 
\wh{\tau}_P({\bf H}_0(\xi x)-T)k_{P,\ES{O}}(\xi x)\ptf$$ On a la dŽcomposition $$k^T_\u(x) = \sum_{\ES{O} \in (\UU/G)_F} k^T_{\ES{O}}(x)\ptf$$ 

Le raffinement de la contribution unipotente ˆ la formule des traces prouvŽ dans cet article est dŽcrit dans le thŽorme suivant (\textit{cf.} \ref{les rŽsultats} pour des ŽnoncŽs prŽcis): 
\begin{theorem}\label{thŽorme principal}
\begin{enumerate}
\item[(i)] Si le paramtre $T\in \ag_0$ est assez rŽgulier, on a 
$$\sum_{\ES{O} \in (\UU/G)_F} \int_{\overline{\bs{X}}_{\!G}}\vert k^T_{\ES{O}}(x) \vert \dd x < + \infty\ptf$$
\item[(ii)] Pour chaque $\ES{O}\in (\UU/G)_F$, la fonction $$T\mapsto \mathfrak{J}^T_{\ES{O}}(f)=
 \int_{\overline{\bs{X}}_{\!G}} k^T_{\ES{O}}(x) \dd x$$ dŽfinit un ŽlŽment de \textup{PolExp}.
\item[(iii)] Pour chaque $\ES{O}\in (\UU/G)_F$, l'expression $\mathfrak{J}^T_{\ES{O}}(f)$ est asymptotique ˆ l'intŽgrale 
$$\int_{\overline{\bs{X}}_{\!G}} F_{P_0}^G(x,T) k_{\ES{O}}(x) \dd x \ptf$$
\end{enumerate}
\end{theorem}

Si une $G$-orbite $\ES{O}\in (\UU/G)_F$ est $F$-rigide (\ref{orbites et strates F-rigides}), la fonction $T\mapsto \mathfrak{J}^T_{\ES{O}}(f)$ 
ne dŽpend pas de $T$ et on peut noter $J_{\bsfrY}(f)$ sa valeur constante. Dans ce cas on a
$$J_{\ES{O}}(f) = \int_{\overline{\bs{X}}_{\!G}}\left(\sum_{\eta \in \ES{O}_F} f(x^{-1}\eta x) \right)\!\dd x ;$$  
l'intŽgrale est absolument convergente. En particulier pour $\ES{O}=\{1\}$, on a $$J_{\{1\}}(f)= \mathrm{vol}(\overline{\bs{X}}_{\!G})f(1)\ptf$$

\begin{corollary}Soit $\bsfrY\in [\UU]_F$. 
\begin{enumerate}
\item[(i)] La fonction $$T\mapsto \mathfrak{J}^T_{\bsfrY}(f)=
\sum_{\ES{O} \in (\bsfrY/G)_F} \mathfrak{J}^T_{\ES{O}}(f) $$ dŽfinit un ŽlŽment de \textup{PolExp}.
\item[(ii)] L'expression 
$\mathfrak{J}^T_{\bsfrY}(f)$ est asymptotique ˆ l'intŽgrale 
$$\int_{\overline{\bs{X}}_{\!G}} F_{P_0}^G(x,T) k_{\bsfrY}(x) \dd x \quad \hbox{avec} \quad k_{\bsfrY}(x) = \sum_{\eta \in \bsfrY_F}f(x^{-1}\eta x)\ptf$$
\end{enumerate}
\end{corollary}

On a donc l'ŽgalitŽ dans \textrm{PolExp} :
$$\mathfrak{J}_\mathrm{unip}^T(f) = \sum_{\bsfrY \in [\UU]_F}\mathfrak{J}^T_{\bsfrY}(f)\ptf$$

\subsection{Sur l'Žtape suivante (II.2)}\label{sur l'Žtape (II.2)} Pour un ensemble fini $S$ de places de $F$, les $F_S$-strates serviront ˆ dŽcomposer les distributions 
$\mathfrak{J}^T_{\bsfrY}(f)$ sur $G(\mbb{A})$ en produits de distributions locales. On l'a dŽjˆ dit, l'une des difficultŽs nouvelles par rapport au cas des corps de nombres 
est que pour $u\in \UF$, le nombre de $G(F_S)$-orbites dans la $F_S$-strates $\bsfrY_{F_S,u}$ peut tre infini. 
Dans l'Žtape (II.2) de \ref{les Žtapes du dŽveloppement}, il ne suffit donc pas de produire une distribution $G(F_S)$-invariante ˆ support dans $\bsfrY_{F_S,u}$ pour obtenir 
une combinaison linŽaire de $G(F_S)$-intŽgrales orbitales. D'ailleurs ces dernires ne sont pas dŽfinies en caractŽristique $p>1$\footnote{On peut cependant 
dŽfinir une mesure de Radon $G(F_S)$-invariante sur $\bsfrY_{F_S,u}$ analogue ˆ celle dŽfinie par Ranga Rao et Deligne sur la $G(F_S)$-orbite de $u$ (\textit{cf.} \cite{Le}).}. 
DŽjˆ pour le groupe $G=\mathrm{SL}_2$ en caractŽristique $2$, la mŽthode d'Arthur ne fonctionne pas et le rŽsultat final s'exprime diffŽremment: 
comme nous l'a expliquŽ Jean-Pierre Labesse, il faut aprs la rŽgularisation classique appliquer une formule sommatoire de Poisson permettant de rŽcupŽrer la finitude 
par un argument global. Nous expliquerons cela dans un prochain travail consacrŽ ˆ l'Žtape (II.2). 

\subsection{Organisation des rŽsultats}\label{organisation} L'article est divisŽ en trois parties. La partie I (sections \ref{la theorie de KRH rat} et \ref{le cas de la variŽtŽ U}) 
contient les principaux rŽsultats de la thŽorie de Kempf-Rousseau-Hesselink 
sur un corps commutatif quelconque. Dans la partie II (sections \ref{la distribution J} et \ref{dŽcomposition suivant les strates}), on dŽcrit le prŽ-dŽveloppement fin de la contribution 
unipotente ˆ la formule des traces sur un corps global de caractŽristique $p>1$. La partie III contient des annexes qu'il nous a semblŽ prŽfŽrable de sŽparer du reste du texte.

Dans les sections \ref{la theorie de KRH rat} et \ref{le cas de la variŽtŽ U}, l'objectif est de dŽfinir les objets et d'Žtablir les rŽsultats qui nous serviront dans les sections  
\ref{la distribution J} et \ref{dŽcomposition suivant les strates} mais aussi de dŽvelopper la thŽorie en vue des applications futures (\textit{e.g.} \cite{Le}). 
Dans la section \ref{la theorie de KRH rat}, on traite d'abord le cas gŽnŽral d'une $G$-variŽtŽ affine $V$ munie d'une sous-variŽtŽ fermŽe $G$-invariante non vide $\ES{Q}$, puis on se restreint rapidement au cas d'une $G$-variŽtŽ pointŽe (\textit{i.e.} $\ES{Q}=\{e_V\}$), 
le cas d'un $G$-module $V$ (avec $e_V=0$) Žtant l'archŽtype de la thŽorie. La thŽorie gŽomŽtrique (\textit{i.e.} $F=\overline{F}$) est due ˆ Kempf-Rousseau \cite{K1,R} et Hesselink \cite{H2}. 
La thŽorie de l'optimalitŽ dans le cadre rationnel (pour $F$ quelconque) a d'abord ŽtŽ ŽtudiŽe par Hesselink \cite{H1}, puis par de nombreux auteurs (\textit{cf.} \cite{BHMR,BMRT}). En ce qui concerne 
les $F$-lames et les $F$-strates de $\ES{N}_F$ pour une $G$-variŽtŽ affine pointŽe $(V,e_V)$, ˆ notre connaissance rien n'avait ŽtŽ Žcrit jusqu'ˆ prŽsent dans le cadre rationnel. 
Notre Žtude est bien sžr largement inspirŽe du cas gŽomŽtrique. 
Dans la section \ref{le cas de la variŽtŽ U}, on traite plus profondŽment le cas de la variŽtŽ pointŽe $(V=G, e_V=1)$ donnŽe par l'action par conjugaison de $G$ sur lui-mme. 
Le lien avec les morceaux unipotents de Lusztig-Xue est traitŽ en \ref{strates et orbites}. 
L'induction parabolique des $F$-strates unipotentes est traitŽe en \ref{IP des ensembles FS} et \ref{induction parabolique des F-strates}. 
En \ref{exemples en basse dimension}, on dŽcrit les $F$-strates unipotentes des groupes $\mathrm{SL}_2$, $\mathrm{SU}(2,1)$ et 
$\mathrm{Sp}_4$. 

Dans la section \ref{la distribution J}, on reprend en la prŽcisant la dŽmonstration de la convergence de la contribution unipotente. 
Le prŽ-dŽveloppement fin de cette dernire est Žtabli dans la section \ref{dŽcomposition suivant les strates}. 
On commence par rappeler la dŽcomposition d'Arthur puis la variante de Hoffmann (sur un corps de nombres). 
Nos principaux rŽsultats sont ŽnoncŽs en \ref{les rŽsultats} et prouvŽs dans les sous-sections qui suivent.  

L'annexe \ref{annexe A} contient des rappels sur les $F$-algbres sŽparables non algŽbriques et la descente des variŽtŽs (algŽbriques) 
relativement ˆ une extension de corps sŽparable non algŽbrique. Dans l'annexe \ref{annexe B}, on introduit la c-$F$-topologie de \cite{BHMR} et on prouve que les ensembles 
$\ES{O}_{\mathfrak{o}}$ de \cite{LL} ont les propriŽtŽs voulues pour la c-$F$-topologie. 
Dans l'annexe \ref{annexe C}, pour un corps global $F$ d'exposant caractŽristique $p\geq 1$ et 
un ensemble fini $S$ de places de $F$, on dŽfinit les $F_S$-lames et les $F_S$-strates de $\UU_{F_S}$; puis on dŽcrit le lien entre 
une $F$-lame $\scrY$, resp. $F$-strate $\bsfrY$, de $\UF$ et la $F_S$-lame $\scrY_{F_S}$, 
resp. $F_S$-strate $\bsfrY_{F_S}$, de $\UU_{F_S}$ qui lui est naturellement associŽe 
pour le plongement diagonal de $\UF$ dans $\UU_{F_S}$\footnote{Si $F$ est un corps de fonctions, l'inclusion ${\rm Int}_{G(F_S)}(\bsfrY) \subset \bsfrY_{F_S}$ peut tre stricte. En revanche c'est toujours une 
ŽgalitŽ si $F$ est un corps de nombres, ce qui n'est autre que le lemme~7.1 de \cite{A2}. 
C'est en essayant de prouver ce lemme d'Arthur dans le cas des corps de fonctions 
que nous avons ŽtŽ amenŽs ˆ remplacer l'argument utilisant les $\mathfrak{sl}_2$-triplet de Jacobson-Morosov 
par un argument analogue utilisant les co-caractres 
optimaux de la thŽorie de Kempf-Rousseau. Cela nous a conduits ˆ abandonner le point de vue des orbites gŽomŽtriques unipotentes 
et ˆ les remplacer par les strates unipotentes de Hesselink.}.  

L'index figurant aprs les rŽfŽrences bibliographiques renvoie principalement aux notations introduites dans les sections \ref{la theorie de KRH rat} et \ref{le cas de la variŽtŽ U}. Pour celles utilisŽes dans 
les sections \ref{la distribution J} et \ref{dŽcomposition suivant les strates}, on renvoie ˆ (l'index de) \cite{LL}.

\part*{Partie I: strates unipotentes rationnelles}

Dans toute cette partie, $F$ est un 
corps commutatif quelconque --- sauf mention expresse du contraire (\textit{e.g.} en \ref{le cas d'un corps top}) --- et $G$ est un groupe algŽbrique linŽaire rŽductif connexe dŽfini sur $F$. 
On note $p$ l'exposant caractŽristique de $F$: $p=1$ si $F$ est de caractŽristique nulle et $p>1$ est la caractŽristique de 
$F$ sinon. On note $e_G$, ou simplement $1$, l'ŽlŽment neutre de $G$. 

\section{La thŽorie de Kempf-Rousseau-Hesselink rationnelle}\label{la theorie de KRH rat}

\subsection{Notations et rappels}\label{notations et rappels}
Soient $\overline{F}$ une cl™ture algŽbrique de $F$ et 
$F^\mathrm{s\acute{e}p}$, resp. $F^\mathrm{rad}$, la cl™ture sŽparable, resp. radicielle, de $F$ dans $\overline{F}$.\index{Fbar,Fsep,Frad@$\overline{F}$, $F^{\mathrm{s\acute{e}p}}$, $F^{\mathrm{rad}}$} On note 
$\Gamma_F$\index{GammaF@$\Gamma_F$} le groupe de Galois $\mathrm{Aut}_F(\overline{F})$. Le corps des points fixes de $\Gamma_F$ est $F^\mathrm{rad}$ 
et le morphisme de restriction $\gamma \mapsto \gamma\vert_{F^\mathrm{s\acute{e}p}}$ est un 
isomorphisme de $\Gamma_F$ sur $\mathrm{Aut}_F(F^\mathrm{s\acute{e}p})$. 

\vskip2mm
\textit{\S1.\hskip1mm VariŽtŽs. ---} On appelle \textit{variŽtŽ algŽbrique}, ou plus simplement \textit{variŽtŽ}, une $\overline{F}$-variŽtŽ algŽbrique, 
\cad un $\overline{F}$-schŽma de type fini, rŽduit et sŽparŽ. On ne demande pas qu'elle soit irrŽductible. Comme il est d'usage, on identifie une variŽtŽ 
ˆ l'ensemble de ses points $\overline{F}$-rationnels\footnote{Ainsi on Žcrira \guill{ŽlŽment $v\in V$}, resp. \guill{sous-ensemble $Z\subset V$}, pour \guill{ŽlŽment $v\in V(\overline{F})$},  
resp. \guill{sous-ensemble $Z\subset V(\overline{F})$}.}. On appelle \textit{$F$-variŽtŽ} une \textit{variŽtŽ dŽfinie sur $F$} au sens de Borel \cite{B}, \cad 
un $F$-schŽma de type fini, gŽomŽtriquement rŽduit et sŽparŽ. Un \textit{$F$-morphisme} ou \textit{morphisme dŽfini sur $F$} (entre $F$-variŽtŽs) est simplement un morphisme de $F$-schŽmas. 
Si $X$ est une sous-variŽtŽ fermŽe 
--- pas forcŽment dŽfinie sur $F$, ni mme sur $F^\mathrm{rad}$ --- 
d'une $F$-variŽtŽ $V$, on pose $$X(F)= X \cap V(F)\ptf$$ 
Observons que l'on peut avoir $X(F)=\emptyset$ mme si $F=F^\mathrm{s\acute{e}p}$. 
En revanche si la variŽtŽ $X$ est dŽfinie sur $F$, puisque $X(F^\mathrm{s\acute{e}p})$ est dense dans $X=X(\overline{F})$ 
\cite[ch.~AG, 13.3]{B}, on a toujours $X(F^\mathrm{s\acute{e}p})\neq \emptyset$. 

Si $V$ est une $F$-variŽtŽ, une sous-variŽtŽ \textit{$F$-fermŽe} de $V$ au sens de Borel \cite[ch.~AG, 12.2]{B}     
\guill{est} un $F$-sous-schŽma fermŽ rŽduit (mais pas forcŽment gŽomŽtriquement rŽduit) du $F$-schŽma 
$V$. Si $V$ est une $F$-variŽtŽ affine, une sous-variŽtŽ $F$-fermŽe, resp. fermŽe et dŽfinie sur $F$, de 
$V$ correspond ˆ un idŽal $I$ de l'algbre affine $F[V]$ de $V$ tel que le quotient $F[V]/I$ soit \textit{rŽduit} (\textit{i.e.} sans ŽlŽments nilpotents $\neq 0$), 
resp. tel que $\overline{F}\otimes_F F[V]/I$ soit rŽduit (\textit{cf.} \cite[ch.~AG, 12.1]{B})\footnote{
Les notions de $F$-variŽtŽ et de sous-variŽtŽ $F$-fermŽe ont ŽtŽ supplantŽes depuis les annŽes 60 par celle, plus souple, de $F$-schŽma. Nul doute que pour 
traiter les questions fines liŽes ˆ l'insŽparabilitŽ, le langage des schŽmas et la topologie plate sont indispensables. 
Le langage \guill{datŽ} du livre de Borel \cite{B} permet nŽanmoins de dŽvelopper une thŽorie rapide et efficace des groupes alg\'ebriques linŽaires sur un corps commutatif, 
largement suffisante pour cet article. C'est pourquoi, concernant cette thŽorie, 
nous avons adoptŽ \cite{B} comme r\'efŽrence principale.}. 

\`A un groupe algŽbrique linŽaire $H$ dŽfini sur $F$ correspond 
un $F$-schŽma en groupes affine lisse dont l'algbre affine est l'algbre affine de $H$, notŽe $F[H]$. Le produit, resp. l'inverse, dans le groupe 
algŽbrique $H$ ou dans le $F$-schŽma en groupes $H$ est donnŽ par le mme homomorphisme de $F[H]$ dans $F[H]\otimes_F F[H]$, resp. $F[H]$.  
Un $F$-sous-groupe fermŽ de $H$, \textit{i.e.} un sous-groupe fermŽ de $H$ dŽfini sur $F$, correspond ˆ un $F$-sous-schŽma en groupes fermŽ lisse.

\vskip2mm
\textit{\S2.\hskip1mm Descente sŽparable (algŽbrique). ---} 
Si $V$ est une $F$-variŽtŽ, le groupe de Galois $\Gamma_F$ opre sur l'ensemble $V(F^\mathrm{s\acute{e}p})$ de ses points $F^\mathrm{s\acute{e}p}$-rationnels. 
Rappelons le \guill{critre galoisien} (\textit{cf.} \cite[ch.~AG, 14.4)]{B}): 

\begin{proposition}\label{critre galoisien}
Soit $V$ une $F$-variŽtŽ et soit $X$ une sous-variŽtŽ fermŽe de $V$. Les conditions suivantes sont Žquivalentes:
\begin{enumerate}
\item[(i)] $X$ est dŽfinie sur $F$;
\item[(ii)] $X$ est dŽfinie sur $F^\mathrm{s\acute{e}p}$ et $X(F^\mathrm{s\acute{e}p})$ est $\Gamma_F$-stable;
\item[(iii)] il existe un sous-ensemble 
$\Gamma_F$-stable de $X(F^\mathrm{s\acute{e}p})= X\cap V(F^\mathrm{s\acute{e}p})$ qui soit Zariski-dense dans $X= X(\overline{F})$. 
\end{enumerate}
\end{proposition}

En particulier pour qu'une sous-variŽtŽ fermŽe $X$ d'une $F$-variŽtŽ $V$ soit dŽfinie sur $F^{\mathrm{rad}}$ --- \textit{i.e.} \guill{soit} une sous-variŽtŽ $F$-fermŽe de $V$ (\textit{cf.} \cite[ch.~AG, 12.2]{B}) ---, il faut et il suffit 
qu'elle soit $\Gamma_F$-stable (au sens o l'ensemble $X=X(\overline{F})$ est $\Gamma_F$-stable).

D'autre part si $V$ et $W$ sont deux $F$-variŽtŽs, pour qu'un 
morphisme de variŽtŽs $f: V \rightarrow W$ soit dŽfini sur $F$, il faut et il suffit qu'il soit dŽfini sur $F^\mathrm{s\acute{e}p}$ 
et qu'il induise une application $\Gamma_F$-Žquivariante $V(F^\mathrm{s\acute{e}p})\rightarrow W(F^\mathrm{s\acute{e}p})$ 
(\textit{cf.} \cite[ch.~AG, 14.3]{B}). En d'autres termes, une $F$-variŽtŽ $V$ est dŽterminŽe (ˆ $F$-isomorphisme unique prs) par la $F^\mathrm{s\acute{e}p}$-variŽtŽ 
$V_{F^\mathrm{s\acute{e}p}} =V \times_F F^\mathrm{s\acute{e}p}$ et l'action de $\Gamma_F$ sur $V(F^\mathrm{s\acute{e}p})$: 
le foncteur $$V \mapsto (\hbox{$V_{F^\mathrm{s\acute{e}p}}$, $\Gamma_F$-ensemble $V(F^\mathrm{s\acute{e}p})$})\leqno{(1)}$$ est pleinement fidle\footnote{La description de l'image du foncteur (1) est plus difficile. Il s'agit de dŽterminer si une 
$F^\mathrm{s\acute{e}p}$-variŽtŽ $\wt{V}$ est dŽfinie sur $F$, autrement dit s'il 
existe une $F$-variŽtŽ $V$ et un $F^\mathrm{s\acute{e}p}$-isomorphisme $\varphi: \wt{V} \rightarrow V_{F^\mathrm{s\acute{e}p}}$; on dit alors que le couple $(V,\varphi)$ 
est un \textit{$F$-modle} de $\wt{V}$. Une condition nŽcessaire pour l'existence d'un tel couple $(V,\varphi)$ 
est l'existence d'une donnŽe de descente ˆ la Weil pour $\wt{V}$. Cette condition est 
souvent suffisante (\textit{e.g.} si $\wt{V}$ est quasi-projective).}. 

\vskip2mm
\textit{\S3.\hskip1mm Orbites sŽparables. ---} Soit $H$ un groupe algŽbrique affine et soit $V$ une variŽtŽ non vide, \textit{a priori} ni affine, ni lisse, munie d'une action algŽbrique 
de $H$ (ˆ gauche) $$H\times V \rightarrow V\vgq (h,v)\mapsto h\cdot v\ptf$$ On suppose que cette action est dŽfinie sur $F$, \cad que $H$, $V$ et le morphisme ci-dessus sont dŽfinis sur $F$. 
Pour $v\in V$, d'aprs \cite[ch.~I, 1.8]{B} l'orbite 
$$H\cdot v = \{h\cdot v\,\vert \, h\in H\}\subset V$$ est une variŽtŽ lisse, localement fermŽe dans $V$; on la notera parfois $\ES{O}_v^H$ ou 
simplement $\ES{O}_v$ si aucune confusion n'est possible. 
Le stabilisateur schŽmatique de $v$ dans $H$ est notŽ $Z_H(v)$; c'est un $\kappa(v)$-sous-schŽma en groupes fermŽ de $H\times_F \kappa(v)$ o $\kappa(v)$ est le corps rŽsiduel de $F$. 
Observons que $Z_H(v)$ peut ne pas tre rŽduit ou, ce qui revient au mme, lisse (\textit{i.e.} gŽomŽtriquement rŽduit). Le groupe $Z_H(v)(\overline{F})$ de ses point $\overline{F}$-rationnels est le stabilisateur $v$ dans $H=H(\overline{F})$ au sens de Borel \cite[ch.~I, 1.7]{B}, \cad 
le sous-groupe fermŽ $$H^v \bydef \{h\in H\,\vert \, h\cdot v =v \}\subset H\ptf$$ 
En d'autres termes, $H^v$ est le stabilisateur schŽmatique \textit{rŽduit} $(Z_H(v))^{\mathrm{r\acute{e}d}}$  de $v$ dans $H$; c'est un sous-groupe 
$\kappa(x)$-fermŽ de $H\times_F \kappa(x)$  
(correspondant au quotient de l'algbre affine de $H^v$ par son nilradical).   

\begin{remark}\label{stabilisateur schŽmatique rŽduit}
\textup{
Pour $v\in V(F)$, la sous-variŽtŽ fermŽe $\ES{O}_v$ de $V$ est dŽfinie sur $F$ et $H^v\;(=(Z_H(v))^{\mathrm{r\acute{e}d}})$ est un sous-groupe 
$F$-fermŽ de $H$ (correspondant au quotient de l'algbre affine 
$F[H^v]$ par son nilradical). Observons que si $Z_H(v)$ est lisse, alors $H^v\;(=Z_H(v))$ est dŽfini sur $F$. }
\end{remark}

Pour $v\in V$, le morphisme de variŽtŽs 
$$\pi_v: H \rightarrow \ES{O}_v,\, h \mapsto h\cdot v$$ se factorise en un morphisme bijectif de variŽtŽs $$\overline{\pi}_v: 
H/H^v \rightarrow \ES{O}_v$$ 
qui n'est en gŽnŽral pas un isomorphisme. Notons $T_v(\ES{O}_v)$ l'espace tangent de 
$\ES{O}_v$ au point $v$ et $\dd (\pi_v)_1: \mathrm{Lie}(H) \rightarrow T_v(\ES{O}_v)$ 
la diffŽrentielle de $\pi_v$ au point $1$. Alors on a les inclusions 
$$\mathrm{Lie}(H^v) \subset \ker (\dd(\pi_v)_1) \quad \hbox{et}\quad \dd(\pi_v)_1(\mathrm{Lie}(H))\subset T_v(\ES{O}_v)\ptf$$ 
D'aprs \cite[ch.~AG, 10.1]{B}, on a toujours l'ŽgalitŽ $$\dim(H)= \dim(H^v)+ \dim(\ES{O}_v)\ptf$$ On en dŽduit le lemme suivant \cite[ch.~II, 6.7]{B}: 

\begin{lemme}\label{orbites sŽparables}
Pour $v\in V$, les conditions suivantes sont Žquivalentes:
\begin{enumerate}
\item[(i)] $\overline{\pi}_v$ est un isomorphisme de variŽtŽs, \textit{i.e.} la $H$-orbite $\ES{O}_v$ est \guill{le} quotient gŽomŽtrique de $H$ par $Z_H(v)$;
\item[(ii)] $\pi_v$ est un morphisme sŽparable; 
\item[(iii)] $\mathrm{Lie}(H^v)= \ker (\dd(\pi_v)_1)$;
\item[(iv)] $\mathrm{Im}(\dd(\pi_v)_1) = T_v(\ES{O}_v)$;
\item[(v)] le stabilisateur schŽmatique $Z_H(v)$ est lisse, \textit{i.e.} il co\"{\i}ncide avec $H^v$. 
\end{enumerate}
\end{lemme}

Pour $v\in V$, on dira que la $H$-orbite $\ES{O}_v$ est \textit{sŽparable} si les conditions Žquivalentes du lemme \ref{orbites sŽparables} 
sont vŽrifiŽes (ces conditions ne dŽpendent que 
de l'orbite $\ES{O}_v$ et pas du point-base $v$ dans cette orbite). 
Si $p=1$, toutes les orbites sont sŽparables. 

Par abus de langage, lorsque $V$ est le groupe 
$H$ lui-mme muni de l'action par conjugaison, on dira aussi qu'un ŽlŽment $x\in H$ est 
sŽparable\footnote{\`A ne pas confondre avec un ŽlŽment de $H(F^\mathrm{s\acute{e}p})$!} si sa $H$-orbite 
$\ES{O}_x$ est sŽparable. Par exemple si $x\in H$ est semi-simple, alors il est sŽparable; et si de plus $x\in H(F)$, alors $H^x$ est dŽfini 
sur $F$ (\textit{cf.} \cite[ch.~III, 9.1]{B}), autrement dit le centralisateur schŽmatique $Z_H(x)$ de $x$ dans $H$ est lisse. 

\vskip2mm
\textit{\S4.\hskip1mm Conventions topologiques. ---} Tout sous-ensemble $X$ d'une variŽtŽ $V$ est muni de la topologie de Zariski induite par celle de $V$. Sauf mention expresse 
du contraire, les notions d'ouvert, de fermŽ, de densitŽ (etc.) se rŽfrent ˆ la topologie de Zariski. Pour un sous-ensemble $X$ de $V$, la fermeture (de Zariski) de $X$ dans $V$ est notŽe $\overline{X}$.\index{Xbar@$\overline{X}$} 
Si $X$ est une partie d'une $F$-variŽtŽ $V$ (\textit{i.e.} $X\subset V(\overline{F})$), la \textit{$F$-fermeture} de $X$ 
dans $V$ au sens \cite[ch.~AG, \S12]{B} est notŽe $\smash{\overline{X}}^{(F)}$; c'est le plus petit $F$-sous-schŽma fermŽ rŽduit de $V$ contenant $X$\footnote{Rappelons que nous identifions une sous-variŽtŽ $F$-fermŽe de $V$ au sens de 
\textit{loc.~cit.} ˆ un $F$-sous-schŽma fermŽ rŽduit de $V$ 
(\textit{cf.} \ref{notations et rappels}).}.\index{XbarF@$\smash{\overline{X}}^{(F)}$}

Si $V$ est une $G$-variŽtŽ affine dŽfinie sur $F$ (\textit{cf.} \ref{le thm de KR}), on peut la munir de la \textit{$\cFG$-topologie}, ou simplement \textit{$\cF$-topologie} (sous-entendu pour la $F$-action de $G$ sur $V$), 
dŽfinie par l'action de l'ensemble $\check{X}_F(G)$ des co-caractres algŽbriques de $G$ qui sont 
dŽfinis sur $F$ (\textit{cf.} l'annexe \ref{annexe B}). La $\cF$-fermeture d'une partie $X$ de $V$ est notŽe $\smash{\overline{X}}^{(\cF)}$.\index{XbarcF@$\smash{\overline{X}}^{(\cF)}$} La $\cF$-topologie sur $V$ ne nŽcessite 
aucune topologie sur $F$ (mais bien sžr 
elle dŽpend de la $F$-action de $G$ sur $V$); elle a un sens pour tout corps commutatif $F$.

Si $F$ est un corps commutatif topologique (sŽparŽ, non discret) et $V$ est une $F$-variŽtŽ, on peut munir l'ensemble $V(F)$ des points $F$-rationnels de $V$ 
de la topologie (forte) dŽfinie par $F$, \textit{cf.} \ref{le cas d'un corps top}. Nous la noterons ${\rm Top}_F$. Par exemple si $F$ est un corps global et $v$ 
est une place de $F$, le complŽtŽ $\wh{F}=F_v$ de $F$ en $v$ est un corps local (localement compact, non discret) qui est non archimŽdien si $v$ est finie; 
pour toute $F$-variŽtŽ $V$, l'ensemble $V(F)$ est ${\rm Top}_{\wh{F}}$-dense dans $V(\wh{F})$.  

Ces topologies --- Zariski et $\cF$-topologie pour un corps commutatif quelconque; ${\rm Top}_{F}$ pour un corps commutatif topologique  --- sont les seules que nous utiliserons dans cet article. 

\subsection{Le thŽorme de Kempf-Rousseau}\label{le thm de KR} 
Dans cette sous-section, on introduit le critre d'instabilitŽ de Hilbert-Mumford puis on rappelle le thŽorme de Kempf-Rousseau gŽomŽtrique (\ref{Kempf1}), 
\cad sur $F=\overline{F}$\footnote{Nous avons choisi de rappeler brivement la thŽorie gŽomŽtrique plut™t que dŽmarrer d'emblŽe 
par la thŽorie rationnelle, ce qui est discutable puisque la seconde ne se dŽduit en gŽnŽral pas de la premire. D'un autre c™tŽ la thŽorie 
rationnelle peut tre vue comme une variante de 
la thŽorie gŽomŽtrique, qui consiste ˆ ne ne tester que les co-caractres qui sont dŽfinis sur $F$ ou, ce qui revient au mme, sur $F^{\rm s\acute{e}p}$. 
Le lecteur familier avec la thŽorie gŽomŽtrique peut 
directement commencer la lecture par la sous-section \ref{le thm de KR rat}.}, ainsi qu'une version rationnelle de ce dernier dans le cas o $F$ est parfait (\ref{Kempf2}). 

\vskip2mm
\textit{\S1.\hskip1mm Le critre d'instabilitŽ de Hilbert-Mumford. ---} Soit $V$ une $G$-variŽtŽ affine, \cad une variŽtŽ algŽbrique affine munie d'une action algŽbrique ˆ gauche 
$G\times V \rightarrow V ,\, (g,v)\mapsto g\cdot v$. 
Rappelons qu'un \textit{co-caractre} (algŽbrique) de $G$ est un morphisme de groupes algŽbriques $\mbb{G}_\mathrm{m} \rightarrow G$. On note 
$\check{X}(G)$ l'ensemble des co-caractres de $G$.\index{XcheckG@$\check{X}(G)$}

\begin{definition}
\textup{Un ŽlŽment $\lambda \in \check{X}(G)$ est dit \textit{primitif} ou \textit{indivisible} s'il n'existe aucun 
$\lambda'\in \check{X}(G)$ tel que $\lambda = k\lambda'$ avec $k\in \mbb{Z} _{\geq 2}$. }
\end{definition}

Observons que $\check{X}(G)$ est un groupe si $G$ est commutatif et que c'est un $\mbb{Z} $-module libre de rang fini si $G$ est un tore. 
L'ensemble $\check{X}(G)$ est muni d'une action de $G$ donnŽe par $$g\bullet \lambda= \mathrm{Int}_g\circ \lambda\quad \hbox{pour tout}\quad (g,\lambda)\in G\times \check{X}(G)\ptf$$
Pour $(v,\lambda)\in V\times \check{X}(G)$, on note $\phi_{v,\lambda}: \mbb{G}_\mathrm{m}\rightarrow V$ le morphisme de variŽtŽs 
$t\mapsto t^\lambda \cdot v$\footnote{L'ensemble $\check{X}(G)$ n'est en gŽnŽral pas un groupe 
mais on le note nŽanmoins additivement: pour $\lambda,\, \lambda'\in \check{X}(G)$, l'application 
$t\mapsto t^\lambda t^{\lambda'}$ n'est \textit{a priori} pas un ŽlŽment de $\check{X}(G)$;  
c'en est un si les co-caractres $\lambda$ et $\lambda'$ commutent entre eux, 
auquel cas on le note $\lambda + \lambda'$. Ainsi pour $n\in \mbb{Z}$ et $\lambda \in \check{X}(G)$, 
on note $n\lambda$ l'ŽlŽment $t\mapsto (t^\lambda)^n$ de $\check{X}(G)$. Enfin on note 
\guill{$0$} le co-caractre trivial $t \mapsto 1\;(=e_G)$.}.\index{Phivl@$\phi_{v,\lambda}$} 
On note $\Lambda_v$\index{Lambdav@$\Lambda_v$} l'ensemble des $\lambda\in \check{X}(G)$ tels que la limite $\lim_{t\rightarrow 0}t^\lambda\cdot v$ existe, au sens o  
$\phi_{v,\lambda}$ se prolonge en un morphisme de variŽtŽs $\phi_{v,\lambda}^+: \mbb{G}_\mathrm{a} \rightarrow V$.\index{Phivl+@$\phi_{v,\lambda}^+$} Ce 
prolongement est alors unique et l'on a
$$\lim_{t\rightarrow 0}t^\lambda\cdot v = \phi_{v,\lambda}^+(0)\ptf$$ 
Observons que $$\Lambda_{g\cdot v}=g\bullet \Lambda_v \quad \hbox{pour tout} \quad g\in G\ptf $$

Pour $v\in V$, la fermeture de Zariski $\overline{\ES{O}_v}$ de 
la $G$-orbite $\ES{O}_v=G\cdot v$\index{Orondv@$\ES{O}_v$} contient une unique $G$-orbite fermŽe que l'on note $\ES{F}_v$\index{Frondv@$\ES{F}_v$}. 
Le critre d'instabilitŽ de Hilbert-Mumford dit qu'il existe un co-caractre $\lambda\in \Lambda_v$ tel que la limite $\phi_{\lambda,v}^+(0)$ appartienne 
ˆ $\ES{F}_v$. Puisque pour tout 
$\lambda \in \Lambda_v$, la limite $\phi_{v,\lambda}^+(0)$ appartient ˆ $\overline{\ES{O}_v}$, on a en particulier:  
l'orbite $\ES{O}_v$ est fermŽe (dans $V$) si et seulement si 
$\phi_{v,\lambda}^+(0)\in \ES{O}_v$ pour tout $\lambda \in \Lambda_v$. 

\vskip2mm
\textit{\S2.\hskip1mm DŽfinition \guill{dynamique} des sous-groupes paraboliques. ---} 
\`A tout co-caractre $\lambda \in \check{X}(G)$ est 
associŽ comme suit un sous-groupe parabolique $P_\lambda$ de $G$ muni d'une dŽcomposition de Levi 
$P_\lambda =M_\lambda \ltimes U_\lambda$:\index{Plambda@$P_\lambda =M_\lambda \ltimes U_\lambda$}
$$P_\lambda=\{g\in G\,\vert\, \hbox{$\lim_{t\rightarrow 0} \mathrm{Int}_{t^\lambda}(g)$ existe}\}\vg$$ 
$$U_\lambda = \{g\in G\,\vert\, \lim_{t\rightarrow 0} \mathrm{Int}_{t^\lambda}(g)=1\}\vg$$ 
$$M_\lambda= \{g\in G\,\vert \, \mathrm{Int}_g\circ \lambda = \lambda \}\ptf$$
On a un morphisme surjectif naturel $P_\lambda \rightarrow M_\lambda$ donnŽ par $g \mapsto \phi_{v,\lambda}^+(0)$. Enfin on sait 
que pour tout sous-groupe parabolique $P$ de $G$ et toute composante de Levi $M$ de $P$, il existe un $\lambda \in \check{X}(G)$ tel que $P=P_\lambda$ et 
$M=M_\lambda$ (voir \ref{existence Plambda} pour la preuve d'une variante plus gŽnŽrale de ce rŽsultat). 

\vskip2mm
\textit{\S3.\hskip1mm Le thŽorme de Kempf-Rousseau}. Le thŽorme de Kempf-Rousseau (voir \ref{Kempf1}) est une version renforcŽe du critre de Hilbert-Mumford. 

Soit $\ES{Q}$ une sous-variŽtŽ fermŽe $G$-invariante\footnote{Au sens o $G\cdot \ES{Q}=\ES{Q}$. 
Plus gŽnŽralement, sauf mention expresse 
du contraire, nous utiliserons le terme \guill{invariant} pour \guill{globalement invariant}.} non vide de $V$. 
Pour $v\in V$ et $\lambda\in \Lambda_v$, on dŽfinit comme suit un invariant 
$m_{\ES{Q},v}(\lambda)\in \mbb{N} \cup \{+\infty\}$. Si $v\in \ES{Q}$, \cad si la $G$-orbite $\ES{O}_v$ est contenue dans $\ES{Q}$, on pose $m_{\ES{Q},v}(\lambda)=+\infty$. Si $v\notin \ES{Q}$, 
la fibre schŽmatique $(\phi_{v,\lambda}^+)^{-1}(\ES{Q})$ est un sous-$\mbb{G}_{\rm m}$-schŽma fermŽ de $\mbb{G}_{\rm a}$, o $\mbb{G}_{\rm m}$ opre sur $\mbb{G}_{\rm a}$ par multiplication. 
C'est donc un diviseur (effectif) de support $t=0$ et on note $m_{\ES{Q},v}(\lambda)\in \mbb{N}$ son degrŽ.\index{mQvl@$m_{\ES{Q},v}(\lambda)$}  
En rŽsumŽ on a:
\begin{itemize}
\item $m_{\ES{Q},v}(\lambda)=0$ si et seulement si $\phi_{v,\lambda}^+(0) \notin \ES{Q}$;
\item $m_{\ES{Q},v}(\lambda) \in \mbb{N}^*$ si et seulement si $v\notin \ES{Q}$ et $\phi_{v,\lambda}^+(0) \in \ES{Q}$;
\item $m_{\ES{Q},v}(\lambda) = +\infty $ si et seulement si $ v\in \ES{Q}$.
\end{itemize}
Observons que $$m_{\ES{Q},g\cdot v}(g\bullet \lambda)= m_{\ES{Q},v}(\lambda)\quad \hbox{pour tout} \quad g\in G\ptf$$

\begin{lemme}[\cite{K1}]\label{invariance des m sous P}
Pour $v\in V$, $\lambda\in \Lambda_v$ et $p\in P_\lambda$, on a $$\lambda\in \Lambda_{p\cdot v} \quad \hbox{et}\quad m_{\ES{Q},p\cdot v}(\lambda)= m_{\ES{Q},v}(\lambda)\ptf$$
\end{lemme}

Fixons une \textit{norme $G$-invariante sur $\check{X}(G)$}, \textit{i.e.} une application $\| \, \|:\check{X}(G)\rightarrow \mbb{R} _+$ telle que:\index{Norme@$\parallel\,\parallel$}
\begin{itemize}
\item $\| g\cdot \lambda  \| =\| \lambda \|$ pour tout $g\in G$ et tout $\lambda \in \check{X}(G)$;
\item pour tout tore maximal $T$ de $G$, il existe une forme bilinŽaire symŽtrique dŽfinie positive 
$(\cdot , \cdot) : \check{X}(T)\times \check{X}(T) \rightarrow \mbb{Z} $ telle que $(\lambda,\lambda)^{\frac{1}{2}} 
=\| \lambda \|$ pour tout $\lambda\in  \check{X}(T)$.
\end{itemize}
L'existence d'une telle norme est une consŽquence de la propriŽtŽ de conjugaison dans $G=G(\overline{F})$ des tores maximaux de 
$G$. Il suffit en effet de fixer un tore maxiaml $T_0$ de $G$ et une forme bilinŽaire symŽtrique dŽfinie positive $(\cdot , \cdot) : 
\check{X}(T_0)\times \check{X}(T_0) \rightarrow \mbb{Z} $ qui soit 
invariante sous l'action du groupe de Weyl $W^G(T_0)$. La norme $W^G(T_0)$-invariante sur $\check{X}(T_0)$ dŽfinie par 
$\| \lambda \|=(\lambda,\lambda)^{\frac{1}{2}} $ 
se prolonge (de manire unique) en une norme $G$-invariante sur $\check{X}(G)$. 

\begin{remark}
\textup{ Soit $X(T_0)$ le groupe des caractres algŽbriques de $T_0$, naturellement identifiŽ ˆ  $\mathrm{Hom}(\check{X}(T_0),\mbb{Z})$. 
Observons que la forme bilinŽaire symŽtrique dŽfinie positive 
$W^G(T_0)$-invariante sur $\check{X}(T_0)$ ne permet en gŽnŽral pas d'identifier $X(T_0)$ ˆ $\check{X}(T_0)$. Elle permet en revanche toujours 
d'identifier $\check{X}(T_0)_{\mbb{Q}}= \check{X}(T_0)\otimes_{\mbb{Z}}\mbb{Q}$ ˆ $X(T_0)_{\mbb{Q}}= \mathrm{Hom}(\check{X}(T_0),\mbb{Q})$. }
\end{remark}

Pour $v\in V$ et $\lambda\in \Lambda_v$, on pose\index{rhoQvl@$\rho_{\ES{Q},v}(\lambda)$}
$$\rho_{\ES{Q},v}(\lambda)= \frac{m_{\ES{Q},v}(\lambda)}{ \| \lambda \|}\in \mbb{R} _+ \cup \{+\infty\}\ptf$$
Observons que pour $k\in \mbb{N}^*$, on a $\rho_{\ES{Q},v}(k\lambda)= \rho_{\ES{Q},v}(\lambda)$. 
On dŽfinit comme suit un sous-ensemble de co-caractres \textit{$(\ES{Q},v)$-optimaux} (relativement ˆ l'action de $G$ sur $V$) $$\Lambda_{\ES{Q},v}^\mathrm{opt} \subset \Lambda_{v}\ptf$$ 
Si $v\in \ES{Q}$, on pose $\Lambda_{\ES{Q},v}^\mathrm{opt}=\{0\}$; 
sinon, on note $\Lambda_{\ES{Q},v}^\mathrm{opt}$\index{LQvopt@$\Lambda_{\ES{Q},v}^\mathrm{opt}$} l'ensemble des $\lambda \in \Lambda_v \smallsetminus \{0\}$ qui sont primitifs et tels que $\phi_{\lambda,v}^+(0)\in \ES{Q}$ avec 
$\rho_{\ES{Q},v}(\lambda) \geq \rho_{ \ES{Q},v}(\lambda')$ pour tout $\lambda'\in \Lambda_{v}$. Observons que si $\lambda\in \Lambda_{\ES{Q},v}^{\rm opt}$ alors $\rho_{\ES{Q},v}(\lambda)>0$. 
On a  $$\Lambda_{\ES{Q},g\cdot v}^\mathrm{opt}= g\bullet \Lambda_{\ES{Q},v}^\mathrm{opt}\quad \hbox{pour tout}\quad g\in G\ptf$$ 

\begin{remark}
\textup{\textit{A priori} la notion d'optimalitŽ dŽpend du choix de la norme $G$-invariante sur $\check{X}(G)$. 
}\end{remark}
\begin{theorem}[\cite{K1}, voir aussi \cite{R,K2}]\label{Kempf1}Soit $v\in V$ tel que $\overline{\ES{O}_v}\cap \ES{Q} \neq \emptyset$.
\begin{enumerate}
\item[(i)]L'ensemble $\Lambda_{ \ES{Q},v}^\mathrm{opt}$ est non vide.
\item[(ii)]Le sous-groupe parabolique $P_\lambda$ ne dŽpend pas de $\lambda \in \Lambda_{ \ES{Q},v}^{\mathrm{opt}}$; on le note $P_{ \ES{Q},v}$.\index{PQv@$P_{ \ES{Q},v}$}
\item[(iii)]Le radical unipotent $U_{ \ES{Q},v}$\index{UQv@$U_{ \ES{Q},v}$} de $P_{ \ES{Q},v}$ opre simplement transitivement sur $\Lambda_{ \ES{Q},v}^\mathrm{opt}$. En d'autres termes, pour toute 
composante de Levi $M$ de $P_{ \ES{Q},v}$, il existe un unique $\lambda\in \Lambda_{ \ES{Q},v}^\mathrm{opt}$ tel que $M_\lambda =M$. 
\item[(iv)]L'invariant $m_{ \ES{Q},v}(\lambda)$ ne dŽpend pas de $\lambda\in \Lambda_{ \ES{Q},v}^{\rm opt}$; on le note $m_{ \ES{Q},v}$.\index{mQv@$m_{ \ES{Q},v}$}
\end{enumerate}
\end{theorem}

\begin{corollary}\label{cor Kempf1}
\begin{enumerate}
\item[(i)]Pour $g\in G$, on a $P_{ \ES{Q},g\cdot v}= \mathrm{Int}_g(P_{ \ES{Q},v})$.
\item[(ii)] Pour $g\in G$, on a $m_{ \ES{Q},g \cdot v}= m_{ \ES{Q},v}$.
\item[(iii)] Le stabilisateur $Z_G(v)= \{g\in G\,\vert \, g\cdot v = v\}$ est contenu dans $P_{ \ES{Q},v}$. 
\end{enumerate}
\end{corollary}

\begin{definition}
\textup{Les ŽlŽments $v\in V$ tels que $\overline{\ES{O}_v}\cap  \ES{Q} \neq \emptyset$ sont dits \textit{$(G, \ES{Q})$-instables} ou simplement 
\textit{$ \ES{Q}$-instables} (sous-entendu pour l'action de $G$ sur $V$).}
\end{definition}

On note\index{NGVQ@$\ES{N}^G(V,\ES{Q})$} $$\ES{N}^G(V, \ES{Q})\subset V$$ le sous-ensemble formŽ des ŽlŽments $ \ES{Q}$-instables. 
Tout ŽlŽment $v\in V$ est par dŽfinition $\ES{F}_v$-instable et l'on a 
$$\ES{N}^G(V,\ES{F}_v)= \{v'\in V\, \vert \, \ES{F}_{v'}= \ES{F}_v\}= \{v'\in V\,\vert \, \ES{F}_v \subset \overline{\ES{O}_{v'}}\}\ptf$$

\begin{remark}\label{rŽduction au cas d'un G-module}
\textup{
\begin{enumerate}
\item[(i)]D'aprs \cite[lemma~1.1.b]{K1}, il existe un $G$-module $W$, \cad un $\overline{F}$-espace vectoriel (de dimension finie) muni 
d'une action algŽbrique linŽaire de $G$ (\textit{i.e.} un morphisme de variŽtŽs $G\rightarrow \mathrm{GL}(W)$), et un morphisme de variŽtŽs $G$-Žquivariant $\pi: V \rightarrow W$ tels que $ \ES{Q}$ soit 
la fibre schŽmatique $\pi^{-1}(0)$ au-dessus de la sous-variŽtŽ fermŽe $\{0\}$ de $W$. Ce morphisme permet de ramener 
l'Žtude du sous-ensemble $\ES{N}^G(V, \ES{Q})\subset V$ ˆ celle de $\ES{N}^G(W,0)\subset W$. 
\item[(ii)]Supposons que $V$ soit un $G$-module et considrons la sous-variŽtŽ $G$-invariante fermŽe $\ES{Q}=\{0\}$ de $V$. Un ŽlŽment 
$v\in V$ est dit \textit{instable} s'il appartient ˆ $\ES{N}^G(V,0)$ et \textit{semi-stable} sinon. Notons $\overline{F}[V]$ l'algbre des fonctions polynomiales sur $V$ et $\overline{F}[V]^{G}\subset \overline{F}[V]$ la sous-$F$-algbre 
formŽe des ŽlŽments $G$-invariants. On sait que $\overline{F}[V]^G$ est une $\overline{F}$-algbre unifre de type fini (Nagata). 
La \guill{conjecture de Mumford} dit qu'un ŽlŽment $v\in V\smallsetminus \{0\}$ est semi-stable si et seulement s'il existe une fonction polynomiale 
$f\in \overline{F}[V]^G$ homogne de degrŽ $\geq 1$ telle que $f(v)\neq 0$. Si $p=1$, c'est une consŽquence du thŽorme de Weyl sur la complte 
rŽductibilitŽ des reprŽsentations de $G$; d'ailleurs on peut prendre $f$ de degrŽ $1$. Si $p>1$, la conjecture de Mumford a ŽtŽ prouvŽe par Haboush \cite{Ha}; et l'on peut prendre $f$ de degrŽ une puissance de $p$. 
\end{enumerate}}
\end{remark}

\textit{\S4.\hskip1mm Une version rationnelle du thŽorme de Kempf-Rousseau. ---} 
Supposons de plus que la $G$-variŽtŽ (affine) $V$ soit dŽfinie sur $F$, \cad que le groupe $G$, la variŽtŽ $V$ et l'action de $G$ sur $V$ soient dŽfinis sur $F$. On peut supposer que la norme 
$G$-invariante $\| \, \|$ sur $\check{X}(G)$ est une \textit{$F$-norme}, 
\cad qu'elle vŽrifie $\| {^\gamma\lambda} \|= \| \lambda \|$ pour tout $\lambda \in \check{X}(G)$ et tout $\gamma \in \Gamma_F$. 
Il suffit pour cela de fixer un tore maximal $T_0$ de $G$ qui soit dŽfini sur $F$. 
Ce tore se dŽploie sur une sous-extension finie $F'/F$ de $F^\mathrm{sep}/F$, par consŽquent l'action de $\Gamma_F$ sur $\check{X}(T_0)$ 
se factorise par un quotient fini et l'on peut s'arranger pour que la forme bilinŽaire symŽtrique dŽfinie positive $W^G(T_0)$-invariante 
$(\cdot,\cdot):\check{X}(T_0)\times \check{X}(T_0) \rightarrow \mbb{Z} $ soit aussi $\Gamma_F$-invariante. 

On suppose aussi que la sous-variŽtŽ fermŽe $G$-invariante $\ES{Q}$ de $V$ est dŽfinie sur $F$. Sous ces hypothses, on peut toujours choisir un morphisme 
$G$-Žquivariant $\pi:V\rightarrow W$ comme dans la remarque \ref{rŽduction au cas d'un G-module}\,(i) qui soit dŽfini sur $F$. 
De plus Kempf \cite{K1} donne une version rationnelle de son rŽsultat, valable seulement si $F$ est \textit{parfait} (\textit{i.e.} $F=F^\mathrm{rad}$): 

\begin{proposition}[\cite{K1}]\label{Kempf2}On suppose que $F$ est parfait. 
Soit $v\in V(F)$ tel que $\overline{\ES{O}_v} \cap \ES{Q}\neq \emptyset$. 
\begin{enumerate}
\item[(i)] L'ensemble $\Lambda_{\ES{Q},v}^\mathrm{opt}$ est $\Gamma_F$-invariant.
\item[(ii)] Le sous-groupe parabolique $P_{\ES{Q},v}$ de $G$ est dŽfini sur $F$.
\item[(iii)] Il existe un co-caractre $\lambda\in \Lambda_{\ES{Q},v}^\mathrm{opt}$ qui soit dŽfini sur $F$.
\end{enumerate}
\end{proposition}

\begin{remark}
\textup{
Pour $v\in V(F)$, l'ensemble $\ES{O}_v(F)$ est rŽunion de $G(F)$-orbites. Si $F$ est parfait, la structure de ces $G(F)$-orbites est contr™lŽe 
par le groupe de cohomologie galoisienne $\mathrm{H}^1(\Gamma_F,G^v(F))$. La structure des $G(F)$-orbites contenues dans $\overline{\ES{O}_v}(F)$ 
est nettement plus compliquŽe, encore plus si $F$ n'est pas parfait.
}
\end{remark}

\begin{exemple}\label{exemple ŽlŽment primitif sur un corps non parfait}
\textup{
Supposons que le corps $F$ soit non parfait, de caractŽristique $p$. ConsidŽrons l'action par conjugaison de $G=\textrm{GL}_p$ sur lui-mme. 
Soit $\gamma\in G(F)$ un ŽlŽment primitif tel que l'extension $F[\gamma]/F$ 
soit radicielle (de degrŽ $p$). Le polyn™me caractŽristique de $\gamma$ est de la forme $(T-\delta)^p$ pour un ŽlŽment $\delta\in F^\mathrm{rad}$ 
tel que $\delta^p = \det(\gamma)$. La $G$-orbite fermŽe $\ES{F}_\gamma$ dans $\overline{\ES{O}_\gamma}$ est donnŽe par $\ES{F}_\gamma=\{\gamma_1\}$ avec $\gamma_1=\mathrm{diag}(\delta,\ldots,\delta)\in G(F^{\textrm{rad}})$. 
Il existe un co-caractre $\lambda_1\in \check{X}(G)$ dŽfini sur $F[\delta]$ tel que $\lim_{t\rightarrow 0} \mathrm{Int}_{t^{\lambda_1}}(\gamma)=\gamma_1$. 
En revanche il n'existe aucun co-caractre $\lambda\in \check{X}(G)$ qui soit dŽfini sur $F$ et tel que 
$\lim_{t\rightarrow 0} \mathrm{Int}_{t^\lambda}(\gamma)= \gamma_1$. En effet si un tel $\lambda$ existait, on aurait $\gamma \in P_\lambda(F)$ 
ce qui contredit l'hypothse de primitivitŽ sur $\gamma$ (puisque $\gamma\notin \ES{O}_{\gamma_1}= \{ \gamma_1\}$, on a $P_\lambda \subsetneq G$).  
}
\end{exemple}

\subsection{Le thŽorme de Kempf-Rousseau rationnel}\label{le thm de KR rat} Les hypothses sont celles du paragraphe \S4 de \ref{le thm de KR}: 
$G$, $V$ et l'action de $G$ sur $V$ sont dŽfinis sur $F$; $\|\,\|$ est une $F$-norme $G$-invariante sur $\check{X}(G)$; $\ES{Q}$ est une $F$-sous-variŽtŽ 
fermŽe $G$-invariante de $V$. 

\vskip1mm
\textit{\S1.\hskip1mm Suppression de l'hypothse \guill{$F$ est parfait}. ---} 
La suppression de l'hypothse \guill{$F$ est parfait} dans la version rationnelle du rŽsultat de Kempf (\ref{Kempf2}) est due ˆ Hesselink \cite{H1} dans le cas o 
$\ES{Q}= \{e_V\}$ pour un point $F$-rationnel $G$-invariant $e_V$ de $V$. Le cas d'une $F$-sous-variŽtŽ fermŽe $G$-invariante $\ES{Q}$ quelconque est dž 
ˆ Bart-Herpel-Martin-R\"ohrle-Tange  \cite{BMRT,BHMR}. Dans les deux cas, les auteurs prouvent une version \guill{uniforme} du critre de Hilbert-Mumford rationnel. 
C'est cette version uniforme que nous reprenons ici. 

Soit\index{XcheckFG@$\check{X}_F(G)$} $$\check{X}_F(G)\subset \check{X}(G)$$ le sous-ensemble formŽ des co-caractres de $G$ qui sont dŽfinis sur $F$. 
Observons que pour $v\in V(F)$ et $\lambda\in \check{X}_F(G)$, le morphisme 
$\phi_{v,\lambda}$ est dŽfini sur $F$; si de plus $\lambda\in \Lambda_v$ alors le morphisme 
$\phi_{v,\lambda}^+$ est lui aussi dŽfini sur $F$ et $\phi_{v,\lambda}^+(0)$ appartient ˆ $V(F)$.

Pour tout sous-ensemble non vide $Z\subset V$, on pose\index{LambdaZ@$\Lambda_Z$}\index{LambdaFZ@$\Lambda_{F,Z}$} $$\Lambda_Z= \bigcap_{v\in Z}\Lambda_v\quad \hbox{et}\quad \Lambda_{F,Z}= \check{X}_F(G)\cap \Lambda_Z\ptf$$ 
On a $$\Lambda_{F,g\cdot Z}= g\bullet \Lambda_{F,Z}\quad \hbox{pour tout} \quad g\in G(F)\ptf$$ 

\begin{definition}\label{ŽlŽments (F,G,Q)-instables}
\textup{Un sous-ensemble $Z\subset V$ est dit \textit{uniformŽment $(F,G,\ES{Q})$-instable}, ou simplement \textit{uniformŽment $(F,\ES{Q})$-instable} (sous-entendu pour l'action de $G$ sur 
$V$), s'il est non vide et s'il existe un co-caractre $\lambda\in \Lambda_{F,Z}$ tel que $\phi_{\lambda,v}^+(0)\in  \ES{Q}$ pour tout $v\in Z$.}
\end{definition}

Pour $\lambda\in \Lambda_Z$, on pose\index{mQZl@$m_{\ES{Q},Z}(\lambda)$} 
$$m_{ \ES{Q},Z}(\lambda) = \inf \{m_{ \ES{Q},v}(\lambda)\,\vert \, v\in Z\} \in \mbb{N}\cup \{+\infty\}\ptf$$ On a donc:
\begin{itemize}
\item $m_{ \ES{Q},Z}(\lambda)=0$ si et seulement s'il existe un $v\in Z$ tel que $\phi_{\lambda,v}^+(0)\notin  \ES{Q}$;
\item $m_{ \ES{Q},Z}(\lambda)\in \mbb{N}^*$ si et seulement si $Z\not\subset  \ES{Q}$ et $\phi_{\lambda,v}^+(0)\in  \ES{Q}$ pour tout $v\in Z$;
\item $m_{ \ES{Q},Z}(\lambda)= + \infty $ si et seulement si $Z\subset  \ES{Q}$. 
\end{itemize}
On pose aussi\index{rhoQZl@$\rho_{ \ES{Q},Z}(\lambda)$} $$\rho_{ \ES{Q},Z}(\lambda)= \frac{m_{ \ES{Q},Z}(\lambda)}{\|\lambda\|}\in \mbb{R}_+\cup \{+\infty\}\ptf$$ 
On dŽfinit comme plus haut un sous-ensemble de co-caractres \textit{$(F,\ES{Q},Z)$-optimaux} (relativement ˆ l'action de $G$ sur $V$)\index{LambdaFQZopt@$\Lambda_{F,\ES{Q},Z}^{\mathrm{opt}}$}
$$\Lambda_{F,\ES{Q},Z}^{\mathrm{opt}}\subset \Lambda_{F,Z}\ptf$$ 
Si $Z\subset  \ES{Q}$, on pose $\Lambda_{F,\ES{Q},Z}^\mathrm{opt}=\{0\}$; 
sinon, on note $\Lambda_{F,\ES{Q},Z}^\mathrm{opt}$ l'ensemble des $\lambda \in \Lambda_{F,Z} \smallsetminus \{0\}$ qui sont primitifs et tels que:
\begin{itemize}
\item $\phi_{\lambda,v}^+(0)\in  \ES{Q}$ pour tout $v\in Z$; 
\item $\rho_{\ES{Q},Z}(\lambda) \geq \rho_{ \ES{Q},Z}(\lambda')$ pour tout $\lambda'\in \Lambda_{F,Z}$.
\end{itemize} 
Observons que si $\lambda\in \Lambda_{ \ES{Q},Z}^{\rm opt}$ alors $\rho_{ \ES{Q},Z}(\lambda)>0$. 
On a  $$\Lambda_{F,\ES{Q},g\cdot Z}^\mathrm{opt}= g\bullet \Lambda_{F,\ES{Q},Z}^\mathrm{opt}\quad \hbox{pour tout}\quad g\in G(F)\ptf$$ 

\begin{remark}
\textup{Comme dans le cas gŽomŽtrique ($\overline{F}=F$), la notion de $F$-optimalitŽ dŽpend \textit{a priori} du choix de la $F$-norme $G$-invariante 
$\| \, \|$ sur $\check{X}(G)$. Plus prŽcisŽment, elle dŽpend de la norme $G(F)$-invariante sur $\check{X}_F(G)$ induite par $\|\,\|$.}
\end{remark} 

\begin{convention}
\textup{Lorsque $Z=\{v\}$ pour un ŽlŽment $v\in V$, on supprimera l'adjectif \guill{uniformŽment} et on remplacera l'indice $\{v\}$ par un indice $v$ dans les notations. 
Lorsque $F=\overline{F}$, on supprimera l'indice $F$ dans les notations. Ainsi pour $Z=\{v\}$ et $F=\overline{F}$, on retrouve les notations 
de \ref{le thm de KR}.}
\end{convention}

\begin{remark}\label{Kempf3}
\textup{
Soit $v\in V(F^{\rm rad})$ tel que $\overline{\ES{O}_v}\cap \ES{Q} \neq \emptyset$. 
D'aprs \ref{Kempf2}, l'ensemble $\Lambda_{F^\mathrm{rad},\ES{Q},v}^\mathrm{opt}$ est non vide et il est Žgal ˆ $\check{X}_{F^\mathrm{rad}}(G)\cap \Lambda_{\ES{Q},v}^\mathrm{opt}$.
}
\end{remark}

On note $$\FN^G(V,\ES{Q}) \subset \ES{N}^G(V,\ES{Q})$$ le sous-ensemble formŽ des ŽlŽments qui sont $(F,\ES{Q})$-instables et on pose\index{NFGVQ@$\FN^G(V,\ES{Q})$, $\NF^G(V,\ES{Q})$} 
$$\NF^G(V,\ES{Q})= V(F) \cap \FN^G(V,\ES{Q})\ptf$$

\begin{remark}
\textup{Soit $v\in V$. 
\begin{enumerate}
\item[(i)] Si $v\in V(F)$, l'orbite $\ES{O}_v$ est dŽfinie sur $F$, par consŽquent la $G$-orbite fermŽe $\ES{F}_v\subset \overline{\ES{O}_v}$ l'est aussi. 
Mais cela n'implique pas que l'ensemble $\ES{F}_v(F)$ soit non vide; il l'est si de plus il existe un co-caractre $\lambda \in \Lambda_{F,v}$ tel que $\phi_{v,\lambda}^+ (0)\in \ES{F}_v$ 
(car $\phi_{v,\lambda}^+ (0)\in V(F)$).
\item[(ii)] Si $v\in V(F)$ et si la $G$-orbite fermŽe $\ES{F}_v$ n'est pas dŽfinie sur $F$, alors il n'existe aucun co-caractre qui soit $(F,\ES{F}_v,v)$-optimal (\textit{cf.} l'exemple \ref{exemple ŽlŽment primitif sur un corps non parfait}). 
\item[(iii)] Si $\lambda \in \Lambda_{F,\ES{Q},v}^{\rm opt}$, alors $\overline{\ES{O}_v}\cap \ES{Q} \neq \emptyset$ et pour 
tout $\mu\in \Lambda_{\ES{Q},v}^{\rm opt}$, on a $\rho_{\ES{Q},v}(\lambda) \geq \rho_{\ES{Q},v}(\mu)$. Si $F$ n'est pas parfait, cette inŽgalitŽ est en gŽnŽral 
stricte: un co-caractre $\lambda\in \Lambda_{F,v}$ peut tre $(F,\ES{Q},v)$-optimal sans tre $(\ES{Q},v)$-optimal (voir plus loin l'exemple \ref{l'exemple de [5.6]{H1}}\,(ii)). 
\end{enumerate}}
\end{remark}

Le rŽsultat suivant \cite[theorem~4.5]{BMRT}, valable pour $F$ quelconque, est la version uniforme du critre de Hilbert-Mumford rationnel. 

\begin{theorem}\label{HMrat1}
Soit $Z\subset V$ un sous-ensemble non vide uniformŽment $(F,\ES{Q})$-instable. 
\begin{enumerate}
\item[(i)] L'ensemble $\Lambda_{F,\ES{Q},Z}^\mathrm{opt}$ est non vide.
\item[(ii)] Le $F$-sous-groupe parabolique $P_\lambda$ de $G$ ne dŽpend pas de $\lambda \in \Lambda_{F,\ES{Q},Z}^{\rm opt}$; 
on le note ${_FP_{\ES{Q},Z}}$ et on pose $P_{F,\ES{Q},Z}={_FP_{\ES{Q},Z}}(F)$.\index{PFQZ@${_FP_{\ES{Q},Z}}$, $P_{F,\ES{Q},Z}$}
\item[(iii)] On note ${_FU_{\ES{Q},Z}}$ le radical unipotent de ${_FP_{\ES{Q},Z}}$ et on pose $U_{F,\ES{Q},Z}= {_FU_{\ES{Q},Z}}(F)$.\index{UFQZ@${_FU_{\ES{Q},Z}}$, $U_{F,\ES{Q},Z}$} Le groupe $U_{F,\ES{Q},Z}$ 
opre simplement transitivement sur $\Lambda_{F,\ES{Q},Z}^\mathrm{opt}$. En d'autres termes, pour toute 
$F$-composante de Levi $M$ de ${_FP_{\ES{Q},Z}}$, il existe un unique co-caractre $\lambda\in \Lambda_{F,\ES{Q},Z}^\mathrm{opt}$ tel que $M_\lambda =M$. 
\item[(iv)] L'invariant $m_{\ES{Q},Z}(\lambda)$ ne dŽpend pas de $\lambda\in \Lambda_{F,\ES{Q},Z}^{\rm opt}$; on le note $m_{F,\ES{Q},Z}$.\index{mFQZ@$m_{F,\ES{Q},Z}$} 
\end{enumerate}
\end{theorem}

\begin{corollary}\label{cor de HMrat1}
\begin{enumerate}
\item[(i)]Pour $g\in G(F)$, on a ${_FP_{\ES{Q},g\cdot Z}}= \mathrm{Int}_g({_FP_{\ES{Q},Z}})$.
\item[(ii)]Pour $g\in G(F)$, on a $m_{F,\ES{Q},g\cdot Z}= m_{F,\ES{Q},Z}$.
\item[(iii)]Si $Z=\{v\}$, le stabilisateur $G^v(F)$ de $v$ dans $G(F)$ est contenu dans $P_{F,\ES{Q},v}$. 
\end{enumerate}
\end{corollary}

\begin{definition}
\textup{Si $Z\subset V$ est un sous-ensemble non vide uniformŽment $(F,\ES{Q})$-instable, un tore $F$-dŽployŽ maximal $S$ de $G$ est dit \textit{$(F,\ES{Q},Z)$-optimal} (relativement ˆ l'action de $G$ sur $V$) 
s'il est contenu dans ${_FP_{\ES{Q},Z}}$. Pour un tel $S$, on a (d'aprs \ref{HMrat1}\,(iii)) $$\check{X}(S)\cap \Lambda_{F,\ES{Q},Z}^\mathrm{opt} = \{\lambda\}\ptf$$}
\end{definition}

\textit{\S2.\hskip1mm Descente sŽparable. ---} Pour tout sous-ensemble non vide $Z\subset V$, on note $\smash{\overline{Z}}^{(F)}$ la \textit{$F$-fermeture} de $Z$ dans 
$V$ au sens de Borel \cite[ch.~AG, 11.3]{B}, \cad le plus petit $F$-sous-schŽma fermŽ rŽduit de $V$ contenant $Z$. C'est une sous-variŽtŽ fermŽe de $V$ dŽfinie sur $F^{\rm rad}$ ou, ce qui revient au mme, $\Gamma_F$-invariante pour l'action de 
$\Gamma_F = \mathrm{Gal}(\overline{F}/F^{\mathrm{rad}})$ sur $V$ (\textit{cf.} \cite[ch.~AG, 14.3]{B}). On a toujours l'inclusion 
$$\overline{Z}\subset \smash{\overline{Z}}^{(F)} $$ o $\overline{Z}$ est la fermeture de Zariski de $Z$ dans 
$V$. Observons que pour tout 
sous-ensemble $\Gamma_F$-stable non vide $Z\subset V(F^{\mathrm{s\acute{e}p}})$, en particulier pour tout sous-ensemble non vide 
$Z\subset V(F)$, la fermeture de Zariski $\overline{Z}$ de 
$Z$ dans $V$ est dŽfinie sur $F$ \cite[ch.~AG, 14.4]{B}; en particulier elle co\"{\i}ncide avec $\smash{\overline{Z}}^{(F)}$.

\begin{remark}
\textup{Si $v\in V(F^{\mathrm{s\acute{e}p}})$, alors $v\in V(E)$ pour une sous-extension galoisienne finie $E/F$ de $F^{\mathrm{s\acute{e}p}}/F$ et la $F$-fermeture $\smash{\overline{Z}}^{(F)}$ de $Z=\{v\}$ est la $\Gamma_F$-orbite
\index{GammaFv@$\Gamma_F(v)$} 
$$\Gamma_F(v) = \{\gamma(v)\,\vert \, \gamma \in \Gamma_F\}= \{\gamma(v)\,\vert \, \gamma \in \textrm{Gal}(E/F)\}\ptf$$  }
\end{remark}

On a la propriŽtŽ de descente sŽparable \cite[theorem~4.7]{BMRT}:

\begin{theorem}\label{BMRT(4.7)}
Soit $Z\subset V$ un sous-ensemble non vide.
\begin{enumerate}
\item[(i)] $Z$ est uniformŽment $(F,\ES{Q})$-instable si et seulement si $\smash{\overline{Z}}^{(F)}$ est uniformŽment $(F^{\rm s\acute{e}p},\ES{Q})$-instable.
\item[(ii)] Si $Z$ est uniformŽment $(F,\ES{Q})$-instable, alors $\Lambda_{F,\ES{Q},Z}^\mathrm{opt} = 
\check{X}_F(G)\cap  \Lambda_{F^\mathrm{s\acute{e}p},\ES{Q},\smash{\overline{Z}}^{(F)}}^\mathrm{opt}$ et $m_{F,Z,\ES{Q}}= m_{\smash{F^{\mathrm{s\acute{e}p}}},\smash{\overline{Z}}^{(F)},\ES{Q}}$. 
\end{enumerate}
\end{theorem}

\begin{lemme}\label{variante de BMRT(4.7)}
Le thŽorme \ref{BMRT(4.7)} reste vrai si l'on remplace $F^\mathrm{s\acute{e}p}$ par n'importe quelle extension sŽparable (algŽbrique ou non)
$E$ de $F$ telle que $(E^\mathrm{s\acute{e}p})^{\mathrm{Aut}_F(E^\mathrm{s\acute{e}p})}= F$. 
(Rappelons que cette condition est automatiquement vŽrifiŽe si l'extension $E/F$ est algŽbrique ou si son degrŽ de transcendance est infini.)  
\end{lemme}

\begin{proof} 
D'aprs la preuve de \cite[theorem~4.7]{BMRT}, $Z$ est uniformŽment $(F,\ES{Q})$-instable si et seulement si $\smash{\overline{Z}}^{(F)}$ est uniformŽment $(F,\ES{Q})$-instable, 
auquel cas $$\Lambda_{F,Z,\ES{Q}}^{\mathrm{opt}}= \Lambda_{F,\smash{\overline{Z}}^{(F)},\ES{Q}}^{\mathrm{opt}}\quad\hbox{et} \quad m_{F,Z,\ES{Q}}=
m_{F,\smash{\overline{Z}}^{(F)},\ES{Q}}\ptf$$ 
On peut donc supposer que $Z=\smash{\overline{Z}}^{(F)}$, \cad que $Z$ est une sous-variŽtŽ $F$-fermŽe de $V$. 
Alors $Z_E=Z\times_FE$ est une sous-variŽtŽ $E$-fermŽe de $V_E$. 
Si $E/F$ est algŽbrique, on peut prendre 
$E^\mathrm{s\acute{e}p}=F^\mathrm{s\acute{e}p}$ et le rŽsultat est directement impliquŽ par 
\ref{BMRT(4.7)} (appliquŽ ˆ $F^\mathrm{s\acute{e}p}/F$ et ˆ $E^\mathrm{s\acute{e}p}/E$). 

Supposons que l'extension $E/F$ ne soit pas algŽbrique. Si $Z\;(=\smash{\overline{Z}}^{(F)})$ est uniformŽment $(F,\ES{Q})$-instable, alors $Z$ est \textit{a fortiori} uniformŽment $(E,\ES{Q})$-instable. 
RŽciproquement, supposons que $Z$ soit uniformŽment $(E,\ES{Q})$-instable. Alors (d'aprs \ref{HMrat1}\,(i)) on peut choisir un 
$\lambda \in \Lambda_{E,\ES{Q},Z}^{\rm opt}$. Puisque $Z$ est $\Gamma_F$-invariant, 
le sous-groupe parabolique $P_{E^{\rm s\acute{e}p},\ES{Q},Z}=P_\lambda(E^{\rm s\acute{e}p})$ de $G(E^{\rm s\acute{e}p})$ vŽrifie 
$$\sigma(P_{E^\mathrm{s\acute{e}p},\ES{Q},Z})= P_{E^\mathrm{s\acute{e}p},\ES{Q},Z}\quad \hbox{pour tout} \quad \sigma \in \mathrm{Aut}_F(E^\mathrm{s\acute{e}p})\ptf$$ 
On en dŽduit d'aprs \ref{critre galoisien bis} que le sous-groupe parabolique ${_{E^{\rm s\acute{e}p}}P_{\ES{Q},Z}}=P_\lambda$ de $G$ est dŽfini sur $F$. 
On peut donc choisir un tore $E^\mathrm{s\acute{e}p}$-dŽployŽ maximal $T$ de ${_{E^\mathrm{s\acute{e}p}}P_{\ES{Q},Z}}$ qui soit 
dŽfini sur $F$. On peut aussi supposer que $\lambda$ appartient ˆ $\check{X}(T)$, \cad que $\lambda$ soit l'unique ŽlŽment de $\check{X}(T) \cap \Lambda_{E^\mathrm{s\acute{e}p},\ES{Q},Z}^\mathrm{opt}$.  
Alors $\sigma(\lambda) = \lambda$ pour tout $\sigma\in \mathrm{Aut}_F(E^\mathrm{s\acute{e}p})$, ce qui entra"ne (ˆ nouveau d'aprs \ref{critre galoisien bis}) que $\lambda$ est dŽfini sur $F$.  Cela prouve que 
$Z$ est uniformŽment $(F,\ES{Q})$-instable si et seulement si $Z$ est uniformŽment $(E^{\rm s\acute{e}p},\ES{Q})$-instable, et que dans ce cas on a l'ŽgalitŽ 
$$\Lambda_{F,\ES{Q},Z}^{\rm opt} = \check{X}_F(G) \cap \Lambda_{E^{\rm s\acute{e}p},\ES{Q},Z}^{\rm opt}\quad \hbox{et}\quad 
m_{F,Z,\ES{Q}}= m_{E^{\smash{\mathrm{s\acute{e}p}}},Z,\ES{Q}}\ptf$$ 
D'o le lemme, d'aprs \ref{BMRT(4.7)} appliquŽ ˆ l'extension $E^\mathrm{s\acute{e}p}/E$. 
\end{proof}

\subsection{La variante de Hesselink (cas d'une $G$-vari\'et\'e pointŽe)}\label{la variante de Hesselink} 
On suppose dans cette sous-section que $V$ est une $G$-vari\'et\'e affine \textit{point\'ee} au sens des $G$-variŽtŽs, \cad qu'elle est munie d'un point 
(fermŽ) $G$-invariant $e_V$. On note\index{N@$\ES{N}$} $$\ES{N}=\ES{N}^G(V,e_V)$$ l'ensemble des $v\in V$ tels que 
$e_V$ appartienne ˆ $\overline{\ES{O}_v}$; ou, ce qui revient au mme, tels que $\ES{F}_v = \{e_V\}$. C'est une sous-variŽtŽ fermŽe 
$G$-invariante de $V$: d'aprs le critre de Hilbert-Mumford gŽomŽtrique, 
$\ES{N}$ est l'ensemble des $v\in V$ tels que $f(v)=0$ pour toute fonction polynomiale $G$-invariante $f$ sur $V$ telle que 
$f(e_V)=0$. 

On suppose de plus comme en \ref{le thm de KR rat} que $G$, $V$ et l'action de $G$ sur $V$ sont dŽfinis sur $F$ et que 
$\check{X}(G)$ est muni d'une $F$-norme $G$-invariante $\|\,\|$. On suppose aussi que $e_V$ appartient ˆ $V(F)$. On conserva ces hypothses jusqu'ˆ la fin de la section \ref{la theorie de KRH rat}. 
Si de plus $V$ est un groupe, sauf mention expresse du contraire, on supposera toujours que $e_V$ est l'ŽlŽment neutre (\textit{e.g.} $e_V=0$ si $V$ est un $G$-module). 
D'aprs le critre galoisien, la variŽtŽ $\ES{N}$ est dŽfinie sur $F^{\mathrm{rad}}$.

\vskip2mm
\textit{\S1.\hskip1mm ElŽments $F$-instables, resp. $F$-semi-stables. ---} 
Le point $e_V$ va jouer le r™le de la $F$-sous-variŽtŽ fermŽe $G$-invariante $\ES{Q}$ de \ref{le thm de KR rat}. Pour allŽger l'Žcriture et quand aucun risque 
de confusion ne sera possible, on supprimera l'indice $e_V$ dans les notations: on Žcrira $\Lambda_{F,Z}= \Lambda_{F,e_V,Z}$ et 
pour $\lambda\in \Lambda_{F,Z}$, on Žcrira $m_Z(\lambda)= m_{e_V,Z}(\lambda)$. Au lieu de $(F,e_V)$-instable, resp. uniformŽment $(F,e_V)$-instable, on dira simplement \textit{$F$-instable}, resp. 
\textit{uniformŽment $F$-instable}. Si $Z\subset V$ est un sous-ensemble (non vide) uniformŽment $F$-instable, on Žcrira\index{LambdaFZopt@$\Lambda_{F,Z}^{\mathrm{opt}}$}\index{mFZ@$m_{F,Z}$}
\index{PFZ@${_FP_Z}$} 
$$\Lambda_{F,Z}^{\mathrm{opt}}= \Lambda_{F,e_V,Z}^{\mathrm{opt}}\vgq m_{F,Z}= m_{F,e_V,Z}\vgq {_FP_Z}= {_FP_{e_V,Z}}\vgq \hbox{etc.}$$ Lorsqu'on voudra 
spŽcifier le groupe $G$ relativement auquel sont dŽfinis ces objets, on les affublera d'un exposant $G$. 

On pose\index{NF@$\FN$, $\NF$} $$\FN = \FN^G(V,e_V) \quad \hbox{et} \quad \NF = \NF^G(V,e_V) \;(= V(F) \cap \FN)\ptf$$ 
Ainsi $\FN$, resp. $\NF$, est l'ensemble des ŽlŽments $F$-instables de $V$, resp. $V(F)$. Tout sous-ensemble uniformŽment $F$-instable de $V$ est contenu dans $\FN$. 
Les ŽlŽments de $V \smallsetminus \FN$ sont dits 
\textit{$F$-semi-stables}\footnote{La terminologie, empruntŽe ˆ la thŽorie gŽomŽtrique ($F=\overline{F}$) est un peu perturbante mais standard: 
le contraire de \guill{$F$-instable} n'est pas \guill{$F$-stable}.}. Plus gŽnŽralement, on introduit la 

\begin{definition}\label{F-stable, F-semi-stable}
\textup{
Soit $H$ est un sous-groupe fermŽ de $G$ dŽfini sur $F$ (pas forcŽment rŽductif ni mme connexe). Un ŽlŽment $v\in V$ est dit \textit{$(F,H)$-instable} 
s'il existe un co-caractre\index{LambdaFvH@$\Lambda_{F,v}^H$} $\lambda \in \Lambda_{F,v}^H = \check{X}_F(H) \cap \Lambda_{F,v}$ tel que $\phi_{v,\lambda}^+(0) = e_V$ 
et il est dit \textit{$(F,H)$-semi-stable} sinon. On note\index{NFH@$\FN^H$, $\NF^H$} $\FN^H=\FN^H(V,e_V)$ l'ensemble des ŽlŽments $(F,H)$-instables de $V$ et $\NF^H=\NF^H(V,e_V)$ 
l'ensemble $V(F) \cap \FN^H$. 
}
\end{definition}

\begin{remark}
\textup{\begin{enumerate}
\item[(i)]Un ŽlŽment de $V$ qui est $F$-instable est \textit{a fortiori} $\overline{F}$-instable. Mais en gŽnŽral il existe des ŽlŽments dans 
$\ES{N}\;(= {_{\smash{\overline{F}}}\ES{N}}=\ES{N}_{\smash{\overline{F}}})$, et mme dans $\ES{N}\cap V(F)$, qui ne sont pas $F$-instables (\textit{cf.} l'exemple \ref{exemple insŽparabilitŽ PGL(p)}). 
\item[(ii)]Si $H\subset G$ est un tore $F$-dŽployŽ, puisque $\check{X}_F(H)=\check{X}(H)$, un ŽlŽment $v\in V$ est $(F,H)$-instable si et seulement s'il est $(\overline{F},H)$-instable. 
\end{enumerate}}
\end{remark}

\textit{\S2.\hskip1mm PropriŽtŽs des ensembles $\FN$ et $\NF$. ---} Les ensembles $\FN$ et $\NF$ sont $G(F)$-invariants. 
On a les inclusions $$\FN\subset \ES{N}\quad \hbox{et}\quad 
\NF \subset  V(F) \cap \ES{N}\ptf$$ 
Insistons sur le fait que si $F \neq \overline{F}$, les deux inclusions ci-dessus sont en gŽnŽral strictes. D'aprs \ref{Kempf2}, on a cependant le

\begin{lemme}
Si $F$ est parfait, alors $\ES{N}_F = V(F)\cap  \ES{N}$.
\end{lemme} 

Mme si $\FN\subsetneq \ES{N} $, on peut toujours munir $\FN$, resp. $\NF$, de la topologie de Zariski induite par celle de $\ES{N}$; 
mais il faut faire attention car $\FN$, resp. $\NF$, n'est en gŽnŽral pas fermŽ 
dans $\ES{N}$, resp. $V(F)\cap \ES{N}$, comme le montre l'exemple suivant. 

\begin{exemple}\label{le cas de PGL(2)}
\textup{
ConsidŽrons le cas o $V=G={\rm PGL}_2$ est muni de l'action par conjugaison. 
Notons $\pi:\wt{G}={\rm GL}_2 \rightarrow G$ l'application naturelle. Supposons $p=2$. Alors: 
\begin{itemize}
\item $\ES{N}$ est l'image par $\pi$ de l'ensemble des ŽlŽments de $\wt{G}$ 
de trace nulle; 
\item $\FN$ est l'image par $\pi$ de l'ensemble des ŽlŽments de $\wt{G}$ de trace nulle et de dŽterminant un carrŽ dans $F^\times$;
\item $\NF = G(F)\cap \FN$ est l'image par $\pi$ de l'ensemble des ŽlŽments de $\wt{G}(F)$ de trace nulle et de dŽterminant un carrŽ dans $F^\times$;
\item $G(F)\cap \ES{N}$ est l'image par $\pi$ de l'ensemble des $\gamma\in \wt{G}(F)$ tels que $\gamma^2\in F^\times$ (o l'on identifie $F^\times$ au centre de $\wt{G}(F)$). 
\end{itemize} 
Si $F$ est infini, $\ES{N}_F$ est Zariski-dense dans $\ES{N}$; et si de plus $F$ n'est pas parfait, 
il est distinct de $G(F)\cap \ES{N}$. Si $F$ n'est pas parfait, la propriŽtŽ \guill{tre un carrŽ dans $F^\times$} dŽfinit un sous-ensemble de $F^\times$ qui n'est 
pas un \textit{ensemble algŽbrique}.}
\end{exemple}

\begin{remark}\label{V-FN dense dans V}
\textup{Si $V$ est distinct de $\ES{N}$, \cad s'il existe un ŽlŽment de $V$ qui soit $\overline{F}$-semi-stable, alors 
$V\smallsetminus \FN$ contient l'ouvert non vide $V\smallsetminus \ES{N}$ de $V$. Si de plus $V$ est irrŽductible, alors 
$V\smallsetminus \FN$ est dense dans $V$.}
\end{remark}

Compte-tenu de la dŽfinition de $\FN$, il est naturel de considŽrer 
la c-$F$-topologie donnŽe par l'action de $\check{X}_F(G)$ sur $V$ (voir l'annexe \ref{annexe B}):

\begin{lemme}
$\FN$ est c-$F$-fermŽ (dans $V$).
\end{lemme}

\begin{proof}
D'aprs \ref{HMrat2}, $\ES{N}_F$ est l'ensemble des ŽlŽments $v\in V(F)$ tels que $\ES{F}_{F,v}= \{e_V\}$ o (rappel) $\ES{F}_{F,v}$ est l'unique $G(F)$-orbite c-$F$-fermŽe 
contenue dans la c-$F$-fermeture $\smash{\overline{X}}^{\cF}$ de la $G(F)$-orbite $X=\ES{O}_{F,v}$ de $v$. D'o le lemme puisque (d'aprs \textit{loc.~cit.}) $\smash{\overline{X}}^{\cF}\subset \FN$.
\end{proof}

\textit{\S3.\hskip1mm Co-caractres \guill{virtuels} et optimalitŽ. ---} La variante de Hesselink consiste ˆ normaliser les co-caractres optimaux de telle sorte  
qu'au lieu de maximaliser la valeur $\rho_{Z}(\lambda)= \frac{m_Z(\lambda)}{\|\lambda\|}$, on se ramne ˆ minimiser la norme d'un co-caractre virtuel.
 
Soit\index{XcheckGQ@$\check{X}(G)_\mbb{Q}$} $\check{X}(G)_\mbb{Q}$ le quotient de $\mbb{N} \times \check{X}(G)$ 
par la relation d'Žquivalence 
$$(n,\lambda)\sim (m,\lambda') \quad \hbox{si et seulement si} \quad m\lambda = n\lambda'\ptf$$
L'action de $G$ sur $\check{X}(G)$ se prolonge naturellement en une action sur $\check{X}(G)_\mbb{Q} $. Si $T$ est 
un tore, alors $\check{X}(T)_\mbb{Q}  = \check{X}(T)\otimes_\mbb{Z}  \mbb{Q} $ est 
un $\mbb{Q} $-espace vectoriel (de dimension finie). Pour $\mu\in \check{X}(G)_{\mbb{Q}}$, on note $P_\mu$ le sous-groupe parabolique de $G$ dŽfini 
par $P_\mu = P_\lambda$ pour un (\textit{i.e.} pour tout) $\lambda \in  \check{X}(G)\cap \mbb{N}^*\mu$; on dŽfinit de la mme manire $U_\mu$ et $M_\mu$. 

Pour un sous-ensemble non vide $Z\subset V$, on note\index{LambdatildeZ@$\wt{\Lambda}_{Z}$} $\wt{\Lambda}_{Z}$ l'ensemble des $\mu \in \check{X}(G)_\mbb{Q} $ tels que 
$n \mu \in \Lambda_Z\;(=\Lambda_{\smash{\overline{F}},e_V,Z})$ pour un $n\in \mbb{N}^*$. On Žtend la mesure de l'instabilitŽ $m_Z(\mu)=m_{e_V,Z}(\mu)$ ˆ tout $\mu\in \wt{\Lambda}_Z$ de la manire suivante: 
on choisit un $n\in \mbb{N}^*$ tel que $n\mu \in \Lambda_Z$ et on pose\index{mZmu@$m_Z(\mu)$} $$m_{Z}(\mu)= n^{-1} m_{Z}(n\mu)\in \mbb{Q} _+ \cup \{+\infty\}\ptf$$ 
On a donc $m_{Z}(\mu)>0$ si et seulement s'il existe un $n\in \mbb{N}^*$ tel que $n\mu \in \check{X}(G)$ et $\lim_{t\rightarrow 0} t^{n\mu} 
\cdot v = e_V$ pour tout $v\in Z$; auquel cas $Z$ est uniformŽment instable. 
Observons que pour 
$Z=\{e_V\}$, on a $$m_{Z}(\mu)= +\infty\quad\hbox{pour tout}\quad \mu \in \wt{\Lambda}_{Z}= \check{X}(G)_{\mbb{Q}}\ptf$$

On dŽfinit de la mme manire l'ensemble $\check{X}_F(G)_{\mbb{Q}}$; c'est un sous-ensemble de $\check{X}(G)_{\mbb{Q}}$. 
L'action de $G(F)$ sur $\check{X}_F(G)$ se prolonge naturellement en une action sur $\check{X}_F(G)_{\mbb{Q}}$. Pour $Z\subset V$, $Z\neq\emptyset$, on pose\index{LambdatildeFZ@$\wt{\Lambda}_{F,Z}$} 
$$\wt{\Lambda}_{F,Z}= \check{X}_F(G)_{\mbb{Q}} \cap \wt{\Lambda}_Z\ptf$$

La $F$-norme $G$-invariante $\| \, \|$ sur $\check{X}(G)$ se prolonge naturellement ˆ $\check{X}(G)_\mbb{Q} $. On rappelle qu'on l'a dŽfinie ˆ partir d'une norme 
$\Gamma_F$-invariante et $W^G(T_0)$-invariante sur $\check{X}(T_0)$ pour un tore maximal $T_0$ de $G$ dŽfini sur $F$. On commence 
par Žtendre linŽairement cette dernire au $\mbb{Q} $-espace vectoriel $\check{X}(T_0)_\mbb{Q}  
=\check{X}(T_0)\otimes_\mbb{Z}  \mbb{Q} $, puis on 
pose $\| g\bullet \mu  \| =\| \mu \|$ pour tout $g\in G$ et tout $\mu \in \check{X}(T_0)_\mbb{Q} $. Pour un sous-ensemble (non vide) $Z\subset \FN$ uniformŽment $F$-instable, $Z\neq \{e_V\}$, on pose\index{qFZ@$\bs{q}_F(Z)$}
$$\bs{q}_F(Z) = \inf \{ \| \mu \|\,\vert \, \mu \in \wt{\Lambda}_{F,Z} ,\,m_Z(\mu)\geq 1\}>0$$ et\index{LambdaboldFZ@$\bs{\Lambda}_{F,Z}$} 
$$\bs{\Lambda}_{F,Z} = \{\mu \in \wt{\Lambda}_{F,Z} \,\vert \, m_Z(\mu)\geq 1,\,\| \mu \| = \bs{q}_F(v)\}\ptf$$ 
On a clairement 
\begin{align*}\bs{q}_F(Z) &=  \inf \{ \| \mu \|\,\vert \, \mu \in \wt{\Lambda}_{F,Z},\,m_Z(\mu)= 1\}\\
&= \left(\sup \{\rho_Z(\mu)\,\vert \, \mu \in \Lambda_{F,Z}\smallsetminus \{0\}\}\right)^{-1}
\end{align*} et 
$$\mu \in \bs{\Lambda}_{F,Z} \Rightarrow m_Z(\mu)=1\ptf$$ 
On pose aussi $$\bs{q}_F(e_V)=0 \quad \hbox{et}\quad \bs{\Lambda}_{F,e_V} = \{0\}\ptf$$ 
Pour $Z\subset \FN$ uniformŽment $F$-instable et $\lambda \in \Lambda_{F,Z}$ tels que $m_Z(\lambda)>0$, on pose 
$$\wt{\lambda}_Z= \frac{1}{m_Z(\lambda)}\lambda \in \wt{\Lambda}_{F,Z}\ptf$$ 
Puisque $$\| \wt{\lambda}_Z \| = \rho_Z(\lambda)^{-1}\quad \hbox{et}\quad m_Z(\wt{\lambda}_Z)= 1\vg$$
on en dŽduit le

\begin{lemme}\label{variante Hesselink rat}
Pour $Z\subset \FN$ uniformŽment $F$-instable, l'application $\lambda \mapsto \wt{\lambda}_Z$ induit une bijection de 
$\Lambda_{F,Z}^\mathrm{opt}$ sur $ \bs{\Lambda}_{F,Z}$. 
\end{lemme}

\begin{remark}\label{indep norme}
\textup{Si la $G$-variŽtŽ affine pointŽe $V$ est de \textit{type adjoint} au sens de \cite[7.1]{H1}, alors d'aprs \cite[7.2]{H1} 
la notion de $F$-optimalitŽ pour les sous-ensembles uniformŽment $F$-instables 
$Z\subset V$ ne dŽpend pas du choix de la $F$-norme $G$-invariante $ \| \, \|$ sur $\check{X}(G)$: l'ensemble $\Lambda_{F,Z}^\mathrm{opt}$, l'invariant $m_{F,Z}$ et le $F$-sous-groupe parabolique ${_FP_Z}$ ne dŽpendent 
pas de $ \| \, \|$. Observons que le groupe $V=G$ lui-mme muni de l'action par conjugaison, est de type adjoint \cite[7.1\,(c)]{H1}.
}\end{remark}

\begin{remark}\label{comparaison m et q}
\textup{
Soit $Z\subset \FN$ uniformŽment $F$-instable. Pour tout $v\in Z$, on a $\bs{q}_F(v)\leq \bs{q}_F(Z)$. On en dŽduit que pour $v'\in \FN$ et $\mu \in \bs{\Lambda}_{F,Z} \cap \wt{\Lambda}_{v'}$, on a 
$$m_{v'}(\mu)\geq 1 \Rightarrow \bs{q}_F(v') \leq \bs{q}_F(Z) = \| \mu \|  \ptf$$ 
}
\end{remark}

\textit{\S4.\hskip1mm Des filtrations. ---}
Pour $\mu \in \check{X}(G)_\mbb{Q} $ et $r\in \mbb{Q} _+\cup \{+\infty\}$, on pose\index{Vmur@$V_{\mu,r}$} 
$$V_{\mu,r}=\{v\in V\,\vert\, \mu\in \wt{\Lambda}_v,\, m_{v}(\mu)\geq r\}\ptf$$
C'est une sous-variŽtŽ fermŽe $P_\mu$-invariante de $V$. On a: 
\begin{itemize}
\item $V_{\mu , r} \subset V_{\mu,s}$ si $r\geq s$;
\item $V_{\mu,0}= \{v\in V\,\vert \, \mu \in \wt{\Lambda}_v \}$;
\item $\bigcap_{r\in \mbb{Q} _{\geq 0}} V_{\mu,r} = V_{\mu , +\infty} = \{e_V\}$; 
\item $m_{v}(\mu)= \sup \{r \in \mbb{Q} _+\cup\{+\infty\}\,\vert \, v \in V_{\mu ,r}\}$ pour tout $v\in \ES{N} $;
\item $V_{g\bullet \mu,r}= g\cdot V_{\mu,r}$ pour tout $g\in G$.
\end{itemize} 
La famille $(V_{\mu ,r})_{r\in \mbb{Q} _+}$ est appelŽe la \textit{$\mu$-filtration} de $V$. On pose aussi\index{Vmu(0)@$V_\mu(0)$} 
$$V_\mu(0)= \{v\in V\,\vert \, t^\lambda\cdot v = v ,\, \forall t\in \smash{\overline{F}}^\times\}$$ pour un (\textit{i.e.} pour tout) $\lambda \in \check{X}(G)\cap \mbb{N}^*\mu$. 
C'est une sous-variŽtŽ fermŽe $M_\mu$-invariante de $V_{\mu,0}$. 

\begin{remark}
\textup{Pour $\mu\in \check{X}_F(G)_{\mbb{Q}}$, les sous-variŽtŽs fermŽes $V_{\mu,r}$ ($r\in \mbb{Q}_+$) et $V_\mu(0)$ de $V$ sont \textit{a priori} seulement 
$F$-fermŽes dans $V$. On verra plus loin qu'elles sont toujours dŽfinies sur $F$ (\ref{consŽquence de V'=V}).}
\end{remark}

Si $\lambda\in \check{X}(G)\cap \mbb{N}^*\mu$, on a:
\begin{itemize}
\item $\phi_{\lambda,v}^+(0)\in V_\mu(0)=V_\lambda(0)$ pour tout $v\in V_{\mu,0}=V_{\lambda,0}$;
\item $\phi_{\lambda,v}^+(0)=e_V$ si et seulement si $v\in V_{\mu,s}$ pour un $s>0$. 
\end{itemize}
Puisque pour tout $s\in \mbb{Q}_+^*$, on a $V_{\mu,s}= V_{\frac{1}{\mu}s,1}$, on obtient que
$$\ES{N} = \bigcup_{\mu \in \check{X}(G)_{\mbb{Q}}}V_{\mu,1}\quad \hbox{et}\quad \FN = \bigcup_{\mu \in \check{X}_F(G)_{\mbb{Q}}} V_{\mu,1}\ptf$$

Pour $V=G$ muni de l'action par conjugaison (avec $e_V=e_G=1$), $\mu \in \check{X}(G)_{\mbb{Q}}$ et $r\in \mbb{Q}_+$, 
on pose\index{Gmur@$G_{\mu ,r}$} $$G_{\mu ,r}= V_{\mu,r}\ptf$$  
D'aprs \cite[2.5, 5.1]{H1}, on a les propriŽtŽs:
\begin{itemize}
\item $G_{\mu ,r}$ est un sous-groupe algŽbrique fermŽ de $G$;
\item $G_{\mu ,0}=P_\mu $, $G_\mu(0)= M_\mu$ et $\bigcup_{r>0} G_{\mu,r} = U_\mu$; 
\item pour $r>0$, $G_{\mu,r}$ est un sous-groupe unipotent connexe distinguŽ de $P_\mu$; 
\item $G_{\mu,r}$ est dŽfini sur $F$ si $\mu\in \check{X}_F(G)_{\mbb{Q}}$ (c'est Žvident si $r=0$ car $G_{\mu,0}=P_\mu$; si $r>0$, cela rŽsulte de \cite[3.13]{BT} car 
$G_{\mu,r}$ est un sous-groupe $F$-fermŽ connexe de $G$ normalisŽ par $M_\mu$).
\end{itemize} 

\vskip2mm
\textit{\S5.\hskip1mm Interlude sur le cas des $G$-modules. ---} Le cas des $G$-modules est l'archŽtype de la thŽorie, c'est pourquoi nous le reprenons brivement ici 
(\textit{cf.} \cite{K2,H1}). 

On appelle \textit{$G$-module} la donnŽe d'un $\overline{F}$-espace vectoriel $W$ (de dimension finie) muni d'un morphisme de variŽtŽs 
$\rho: G \rightarrow \mathrm{Aut}_{\overline{F}}(W)$. Un tel $G$-module $(W,\rho)$ est \textit{dŽfini sur $F$} s'il existe un sous-$F$-espace vectoriel $W'\subset W$ tel que $W= W'\otimes_F\overline{F}$ et $\rho$ soit un $F$-morphisme. On a alors une opŽration de $G$ sur le $F$-foncteur $W'_{\mathrm{a}}$ qui ˆ toute $F$-algbre $A$ associe le 
$G(A)$-module $W'\otimes_F A$ (libre et de rang fini sur $A$). Suivant les conventions de \ref{notations et rappels} (\S1), on identifie $W= W'_{\mathrm{a}}(\overline{F})$ au $F$-foncteur $W'_{\mathrm{a}}$. 

On peut toujours se ramener au cas des $G$-modules gr‰ce au rŽsultat bien connu suivant (\textit{cf.} \cite[1.1]{K1}, \cite[2.3]{H1}):

\begin{lemme}\label{rŽduction au cas d'un G-module rat}
Soit $(V,e_V)$ une $G$-variŽtŽ pointŽe dŽfinie sur $F$. Il existe un $G$-module $W$ dŽfini sur $F$ et une $F$-immersion fermŽe $G$-Žquivariante 
$\iota:V \rightarrow W$ telle que $\iota(e_V)=0$.
\end{lemme}

On suppose jusqu'ˆ la fin de ce paragraphe (\S5) que $V$ est un $G$-module dŽfini sur $F$ (avec $e_V=0$). 
On a toujours $\ES{N}=\ES{N}^G(V,0)$, $\FN= \FN^G(V,0)$ et $\NF= V(F)\cap \FN$. 

Pour $\lambda \in \check{X}(G)$ et $i\in \mbb{Z} $, on pose 
$$V_\lambda(i) = \{v\in V\,\vert \,t^\lambda \cdot v = t^i v,\, \forall t \in \smash{\overline{F}}^\times \}\ptf$$ 
C'est un sous-espace $M_\lambda$-invariant de $V$. 
Pour $k\in \mbb{Z} $, on pose $$V'_{\lambda ,k} =\bigoplus_{i\geq k} V_\lambda(i)\ptf$$ 
C'est un sous-espace $P_\lambda$-invariant de $V$. Observons que pour $v\in V'_{\lambda ,k}$, la composante $v_\lambda(k)$ de $v$ sur $V_\lambda(k)$ est donnŽe par 
$$v_\lambda(k) = \lim_{t\rightarrow 0} t^{-k} t^\lambda\cdot v \ptf$$ On verra plus loin (\ref{V'=V}) que pour $k\geq 0$, le sous-espace $V'_{\lambda,k}$ co\"{\i}ncide avec la 
variŽtŽ $V_{\lambda,k}$ dŽfinie plus haut. 

\begin{remark}\label{[B,5.2]}
\textup{D'aprs \cite[ch.~II, 5.2]{B}, pour $\lambda\in \check{X}_F(G)$ et $k\in \mbb{Z}$, les sous-espaces $V_\lambda(k)$ et $V'_{\lambda,k}$ de $V$ sont dŽfinis sur $F$.}
\end{remark}

Pour $\mu\in \check{X}(G)_{\mbb{Q}}$ et $r\in \mbb{Q}$, on Žcrit $\mu= \frac{1}{n}\lambda$ et $r= \frac{k}{n}$ avec $n\in \mbb{N}^*$, $\lambda\in \check{X}(G)$ et $k\in \mbb{Z}$; et 
l'on pose\index{Vmu(r)@$V_\mu(r)$}\index{V'mur@$V'_{\mu,r}$} $$V_\mu(r)= V_\lambda(k)\quad\hbox{et}\quad V'_{\mu,r}= V'_{\lambda,k}\ptf$$ 

Si $T$ est un tore de $G$ (pas forcŽment maximal, ni mme dŽfini sur $F$), on a la dŽcomposition\index{Vchi@$V_\chi$} 
$$V= \bigoplus_{\chi \in X(T)}V_\chi\quad \hbox{avec} \quad V_\chi= \{v\in V\,\vert \, t\cdot v = t^\chi
v,\, \forall t\in T\}\ptf$$ 
Pour $\lambda \in \check{X}(T)$ et $k\in \mathbb{Z}$, on a la dŽcomposition 
$$V_\lambda(k)= \bigoplus_{\{\chi\in X(T) \,\vert\, \langle \chi,\lambda \rangle = k\}} V_\chi \ptf$$
On note\footnote{Le \guill{prime} dans la notation est dž au fait que pour $V={\rm Lie}(G)$ muni de l'action adjointe, on notera $\ES{R}_T$ l'ensemble des racines de $T$ dans ${\rm Lie}(G)$ et 
on posera $\ES{R}'_T= \ES{R}_T\cup\{0\}$.} $\ES{R}'_T=\ES{R}'^V_T\subset X(T)$\index{R'TV@$\ES{R}'_T(V)$} le sous-ensemble (fini) formŽ des $\chi$ tels que $V_\chi\neq \{0\}$, \cad l'ensemble des \guill{poids} de $T$ dans $V$. 
Pour $v\in V$, on Žcrit $v= \sum_{\chi} v_\chi$ avec $v_\chi \in V_\chi$ et on pose 
$$\ES{R}'_T(v)= \{\chi \in X(T)\,\vert \, v_\chi \neq 0\}\subset \ES{R}'_T \ptf$$ 
On a donc $\ES{R}'_T(0)=\emptyset$. Pour $Z\subset V$, $Z\neq \emptyset$, on pose\index{R'TZ@$\ES{R}'_T(Z)$} 
$$\ES{R}'_T(Z)= \bigcup_{v\in Z} \ES{R}'_T(v);$$ et pour 
$\mu \in \check{X}(T)_\mbb{Q} $, on pose\index{m'muZ@$m'_Z(\mu)$} $$m'_Z(\mu)\bydef \inf \{\langle \chi , \mu \rangle \,\vert \, \chi \in \ES{R}'_T(Z)\}$$ avec par convention $m'_Z(\mu)=+\infty$ si $Z=\{0\}$.  
Ainsi $\mu$ appartient ˆ $\wt{\Lambda}_Z$ si et seulement si $m'_Z(\mu)\geq 0$; et pour $n\in \mbb{N}^*$ tel que $\lambda= n\mu \in \check{X}(T)$, la limite $\lim_{t\rightarrow 0}t^\lambda\cdot v$ 
existe et vaut $0$ pour tout $v\in Z$ si et seulement si $m'_Z(\mu)>0$ (observons que $m'_Z(\lambda)= n m'_Z(\mu)$). 
Cette dŽfinition ne dŽpend pas du tore $T$ tel que $\mu \in X(T)_\mbb{Q} $: on a $T\subset M_\mu$ et quitte ˆ remplacer $T$ 
par un tore plus gros, on peut supposer que 
$T$ est un tore maximal de $M_\mu$; ensuite on utilise la propriŽtŽ de $M_\mu$-conjugaison des tores maximaux de $M_\mu$. 
Si de plus $\mu\in \wt{\Lambda}_Z$, cette dŽfinition co\"{i}ncide avec la prŽcŽdente:

\begin{lemme}\label{m'=m}
Soit $Z\subset V$, $Z\neq \emptyset$. Pour $\mu \in \check{X}(T)_{\mbb{Q}}\cap \wt{\Lambda}_Z$, on a 
$$m'_Z(\mu)= m_Z(\mu)\;(=m_{0,Z}(\mu))\geq 0 \ptf$$ 
\end{lemme}

\begin{proof}
Puisque $$m'_Z(\mu)= \inf\{m'_v(\mu)\,\vert\, v\in Z\}\quad \hbox{et} \quad m_{Z}(\lambda) = \inf \{m_{v}(\lambda)\,\vert \, v\in Z\}\vg$$ 
on peut supposer que $Z=\{v\}$ avec $v\neq 0$. 
On peut aussi supposer que $\mu\in X(T)$. \'Ecrivons $v= \sum_{\chi \in \ES{R}'_T(v)}v_\chi$ avec $v_\chi\in V_\chi$. 
Pour $\chi \in \ES{R}'_T(v)$, posons $m_\chi = \langle \chi , \mu \rangle \in \mbb{N}$. 
S'il existe un $\chi \in \ES{R}'_T(v)$ tel que $m_\chi =0$, 
alors $\lim_{t \rightarrow 0}t^\mu\cdot v\neq 0$ et dans ce cas on a $m'_v(\mu)=0 = m_{0,v}(\mu)$.  
On peut donc supposer $m'_v(\mu)\geq 1$. 
La fibre schŽmatique $\phi^{-1}(0)$ du morphisme $\phi=\phi_{v,\mu}^+: \mbb{G}_\mathrm{a} \rightarrow V_{\mu,1}$ donnŽ par 
$$\phi(t)= \sum_{\chi \in \ES{R}'_T(v)} t^{m_\chi} v_\chi$$ est d'algbre affine $\overline{F}[T]/ I$ o $I$ est l'idŽal 
engendrŽ par les $T^{m_\chi}$ pour $\chi \in \ES{R}'_T(v)$. Donc $I= (T^k)$ avec $k = \min \{m_\chi \,\vert \, \chi \in \ES{R}'_T(v)\}$. Cela prouve le lemme.
\end{proof}

\begin{remark}\label{V'=V}
\textup{Pour $\mu\in \check{X}(G)_{\mbb{Q}}$ et $r\in \mbb{Q}$, le sous-espace $V'_{\mu,r}$ est donnŽ par  
$$V'_{\mu,r}= \{v\in V\,\vert \, m'_v(\mu)\geq r\}\ptf$$ 
Si $r\geq 0$, d'aprs \ref{m'=m}, on a l'ŽgalitŽ $V'_{\mu,r}=V_{\mu,r}$.}
\end{remark}

Soit\index{iotaT@$\iota_T$} $\iota_T:X(T)_\mbb{Q}  \rightarrow \check{X}(T)_\mbb{Q} $ 
l'isomorphisme $\mbb{Q} $-linŽaire donnŽ par $$(\iota_T(\chi),\lambda)= \langle \chi, \lambda\rangle\quad \hbox{pour tout}\quad \lambda \in \check{X}(T)\ptf$$ 
Pour $Z\subset V$, $Z\neq \emptyset$, on note\index{KTZ@$\ES{K}_T(Z)$} $\ES{K}_T(Z)$ l'enveloppe convexe (\textit{i.e.} le polytope de Newton) de 
$\iota_T(\ES{R}'_T(Z))$ dans $\check{X}(T)_\mbb{Q} $ avec par convention $\ES{K}_T(0)=\{0\} $. Pour $\mu\in \check{X}(T)_{\mbb{Q}}$ et $Z\neq \{0\}$, on a donc 
$$m'_Z(\mu)= \inf \{(\eta,\mu)\,\vert \, \eta\in \ES{K}_T(Z)\}\ptf$$ 
Par convexitŽ, il existe un unique ŽlŽment $\mu^{T\!}(Z)\in \ES{K}_T(Z)$ qui minimise la norme $ \|\; \|$ sur $\check{X}(T)_\mbb{Q} $ (\textit{cf.} \cite[3.2]{H1}). 
On a $\mu^{T\!}(Z)\neq 0$ si et seulement si $0\notin \ES{K}_T(Z)$, auquel cas, d'aprs le thŽorme de la projection sur un convexe fermŽ, 
$(\eta -\mu^{T\!}(Z) ,\mu^{T\!}(Z)) \geq 0$ pour tout $\eta \in \ES{K}_T(Z)$ et donc $$\|\mu^{T\!}(Z)\|^2 = m'_Z(\mu^{T\!}(Z))=m_{Z}(\mu^{\,T\!}(Z))>0\ptf $$ 
Si $0\notin \ES{K}_T(Z)$, l'ŽlŽment $\wt{\mu}^{\,T\!}(Z) = \frac{\mu^{T\!}(Z)}{m_Z(\mu^{\,T\!}(Z))}$ vŽrifie $m_Z(\wt{\mu}^{\,T\!}(Z))=1$ et (\textit{cf.} \textit{loc.~cit.})\index{qTZ@$\bs{q}^{T}(Z)$} 
$$  \|\wt{\mu}^{\,T\!}(Z) \| = \inf \{\|\mu \|\,\vert \, \mu\in \check{X}(T)_\mbb{Q}  ,\, m_Z(\mu)\geq 1 \}\bydef \bs{q}^{T\!}(Z)\ptf$$
Enfin, toujours si $0\notin \ES{K}_T(Z)$, on note $\lambda^{T\!}(Z)$ l'unique ŽlŽment primitif de $\check{X}(T)$ qui soit de la forme $k\wt{\mu}^{\,T\!}(v)$ 
pour un $k\in \mbb{N}^*$. 

Si maintenant $Z\subset \FN$ est uniformŽment $F$-instable et si $S$ est un tore $F$-dŽployŽ maxi\-mal de $G$, on a toujours $$\bs{q}_F(Z) = \bs{q}_F^G(Z)\leq \bs{q}_F^S(Z)= \bs{q}^S(Z)\ptf$$
La propriŽtŽ de $G(F)$-conjugaison des tores $F$-dŽployŽs maximaux de $G$ entra"ne 
que $$ \bs{q}_F(Z) = \inf \{\bs{q}^{g^{-1}Sg}(Z)\,\vert \, g\in G(F)\}= \inf \{\bs{q}^{S\!}(g\cdot v)\,\vert \, g\in G(F)\}\ptf$$ 
Si $S$ est choisi $(F,Z)$-optimal, \cad contenu dans ${_FP_Z}$, alors 
$$ \bs{q}_F(Z) =\bs{q}^{S\!}(Z)= \| \wt{\mu}^{\,S\!}(Z) \| \quad \hbox{et}\quad \bs{\Lambda}_{F,Z}\cap \check{X}(S)_\mbb{Q}  = \{ \wt{\mu}^{S\!}(Z)\} \ptf$$
En Žcrivant $\lambda^{S}(v) = k \wt{\mu}^{S\!}(Z)$ avec $k\in \mbb{N}^*$, on a donc 
$$m_Z(\lambda^{S}(Z))= k\quad \hbox{et}\quad \Lambda_{F,Z}^\mathrm{opt}\cap \check{X}(S) =\{\lambda^{S}(Z)\}\ptf$$ 

\vskip2mm
\textit{\S6.\hskip1mm Les sous-ensembles $F$-saturŽs. ---} Reprenons les hypothses du dŽbut de \ref{la variante de Hesselink}: $(V,e_V)$ est une $G$-variŽtŽ pointŽe dŽfinie sur $F$, avec $e_V\in V(F)$. 

\begin{lemme}\label{consŽquence de V'=V}
Pour $\mu\in \check{X}_F(G)_{\mbb{Q}}$ et $r\in \mbb{Q}_+$, la sous-variŽtŽ fermŽe $V_{\mu,r}$ de $V$ est dŽfinie sur $F$; et la sous-variŽtŽ 
fermŽe $V_\mu(0)$ de $V$ est elle aussi dŽfinie sur $F$. 
\end{lemme}

\begin{proof}
D'aprs \ref{rŽduction au cas d'un G-module rat}, il existe un $G$-module $W$ dŽfini sur $F$ et une $F$-immersion fermŽe 
$G$-Žquivariante $\iota: V \rightarrow W$ telle que $\iota(e_V)= 0$. Pour $\mu\in \check{X}(G)_{\mbb{Q}}$ et $r\in \mbb{Q}_+$, on a 
$\iota(V_{\mu,r})= \iota(V) \cap W_{\mu,r}$. Or d'aprs \ref{V'=V} et \ref{[B,5.2]}, le sous-espace $W_{\mu,r}$ de $W$ est dŽfini sur $F$. D'autre part on a 
$\iota(V_\mu(0))= \iota(V)\cap W_\mu(0)$ et le sous-espace $W_\mu(0)$ de $W$ est dŽfini sur $F$ (\ref{[B,5.2]}). D'o le lemme. 
\end{proof}

Si $Z\subset \FN$ est uniformŽment $F$-instable, 
l'ensemble $\bs{\Lambda}_{F,Z}$ forme une seule orbite sous le $F$-sous-groupe parabolique ${_FP_Z}$ de $G$. Par consŽquent 
le sous-ensemble (lui aussi uniformŽment $F$-instable) $V_{\mu,1}\subset \FN$ ne dŽpend pas de $\mu\in \bs{\Lambda}_{F,Z}$. On pose\index{XFZ@${_F\scrX_Z}$} 
$${_F\scrX_Z}= V_{\mu,1} \quad \hbox{pour un (\textit{i.e.} pour tout)} \quad \mu\in \bs{\Lambda}_{F,Z}\ptf$$ 
De manire Žquivalente, on a 
$${_F\scrX_Z}= V_{\lambda,m_Z(\lambda)}\quad \hbox{pour un (\textit{i.e.} pour tout)}\quad \lambda \in \Lambda_{F,Z}^{\mathrm{opt}}\ptf$$ 
D'aprs \ref{consŽquence de V'=V}, ${_F\scrX_Z}$ est une sous-variŽtŽ fermŽe de $V$ dŽfinie sur $F$. 

\begin{definition}
\textup{Pour $Z\subset \FN$ uniformŽment $F$-instable, le sous-ensemble ${_F\scrX_Z}\subset \FN$ est appelŽ le \textit{$F$-saturŽ de $Z$}. 
Un sous-ensemble uniformŽment $F$-instable de $\FN$ est dit \textit{$F$-saturŽ} s'il co\"{\i}ncide avec son $F$-saturŽ. 
}\end{definition}

On a la version rationnelle de \cite[lemma~2.8]{H2}:

\begin{lemme}\label{lemma 2.8 de [H2]}
Soient $Z$ et $Z'$ deux sous-ensembles (non vides) de $\FN$ uniformŽment $F$-instables. Posons $X={_F\scrX_Z}$ et $X'={_F\scrX_{Z'}}$. 
\begin{enumerate}
\item[(i)] $X$ est uniformŽment $F$-instable et $\bs{\Lambda}_{F,X}= \bs{\Lambda}_{F,Z}$.
\item[(ii)] $Z\subset X = {_F\scrX_X}$. 
\item[(iii)] $ \bs{\Lambda}_{F,Z}= \bs{\Lambda}_{F,Z'}$ si et seulement si $\bs{q}_F(Z)= \bs{q}_F(Z')$ et $Z\subset X'$.
\item[(iv)] ${_FP_Z}= \{g\in G\,\vert \, g\cdot Z \subset X\}$. 
\end{enumerate}
\end{lemme}

\begin{proof}
Les points (i) et (ii) sont clairs. On en dŽduit que $ \bs{\Lambda}_{F,Z}= \bs{\Lambda}_{F,Z'}$ si et seulement si $X=X'$, d'o le point (iii). 
Quant au point (iv), si $g\cdot Z \subset X$, d'aprs (iii) on a $g\bullet {_FP_Z}= {_FP_{g\cdot Z}}= {_FP_Z}$. Puisque ${_FP_Z}$ est son propre normalisateur, 
cela entra"ne que $g\in {_FP_Z}$. RŽciproquement si $g\in {_FP_Z}$, puisque $X$ est par dŽfinition ${_FP_Z}$-invariant, \textit{a fortiori} on a $g\cdot Z\subset X$. 
\end{proof}

Pour $Z\subset \FN$ uniformŽment $F$-instable, on pose\index{XbsFZ@${_F\bsfrX_Z}$} 
$${_F\bsfrX_Z}= G(F)\cdot {_F\scrX_Z} = \bigcup_{g\in G(F)} {_F\scrX_{g\cdot Z}}\ptf$$
Pour $s\in \mbb{Q}_+$, on pose\index{NFs@$\FN_{<s}$} $$\FN_{<s}= \{v\in \FN\,\vert \, \bs{q}_F(v)<s\}\ptf$$

On a vu que pour $Z\subset \FN$ uniformŽment $F$-instable, ${_F\scrX_Z}$ est une $F$-sous-variŽtŽ fermŽe de $V$. En revanche ${_F\bsfrX_Z}$, tout comme $\FN$, 
n'est en gŽnŽral pas fermŽ dans $V$. On a cependant la variante rationnelle suivante de \cite[2.9]{H2}:

\begin{proposition}\label{prop 2.9 de [H2] rat}
\begin{enumerate}
\item[(i)]Il n'y a qu'un nombre fini d'ensembles ${_F\bsfrX_Z}$ avec $Z\subset V$ uniformŽnent $F$-instable, \cad de $G(F)$-orbites de sous-ensembles  $F$-saturŽs de $\FN$. 
\item[(ii)]Pour $Z\subset \FN$ uniformŽment $F$-instable, ${_F\bsfrX_Z}$ est c-$F$-fermŽ (dans $V)$.
\item[(iii)]Pour $s\in \mbb{Q} _+$, $\FN_{<s}$ est c-$F$-fermŽ (dans $V$); c'est la rŽunion (finie) des ensembles ${_F\bsfrX_v}$ avec $v\in \FN$ tel que 
$\bs{q}_F(v) < s$. 
\end{enumerate}
\end{proposition}

\begin{proof}
On peut gr‰ce ˆ \ref{rŽduction au cas d'un G-module rat} supposer que $V$ est un 
$G$-module dŽfini sur $F$ (avec $e_V=0$). 

Prouvons (i). Soit $S$ un tore $F$-dŽployŽ maximal de $G$. Rappelons qu'on a notŽ $\ES{R}'_S=\ES{R}'^V_S$ 
l'ensemble (fini) des poids de $S$ dans $V$ et $\ES{R}'_S(Z)\subset \ES{R}'_S$ le sous-ensemble formŽ des poids de $S$ dans $Z$. 
La propriŽtŽ de $G(F)$-conjugaison des tores $F$-dŽployŽs maximaux de $G$ entra"ne 
qu'il existe un $g\in G(F)$ tel que $S$ soit $(F,g\cdot Z)$-optimal, \cad 
que $\bs{\Lambda}_{F,g\cdot Z} \cap  \check{X}(S)_{\mbb{Q} }= \{\wt{\mu}_S(g\cdot Z)\}$ avec $\wt{\mu}_S(g\cdot Z)= \frac{\mu_S(g\cdot Z)}{m_Z(\mu_S(g\cdot Z))}>0$. On a donc
$$ {_F\scrX_Z}= g^{-1} \cdot {_F\scrX_{g\cdot Z}}= g^{-1} \cdot V_{\wt{\mu}_S(g\cdot Z),1}\ptf$$ 
Le co-caractre virtuel $\wt{\mu}_S(g\cdot Z)\in \check{X}(S)_{\mbb{Q}}$ ne dŽpend que de l'ensemble $\ES{R}'_S(g\cdot Z)$. 
Puisqu'il n'y a qu'un nombre fini de sous-ensembles de $\ES{R}'_S$, cela prouve (i).

Prouvons (ii). Soient $v_1\in {_F\bsfrX_Z}$ et $\mu \in \Lambda_{F,v_1}$. 
On veut prouver que $v'=\phi_{\mu,v_1}^+(0)$ appartient ˆ ${_F\bsfrX_Z}$. Soit $S$ un tore $F$-dŽployŽ maximal de $M_\mu$ et soit $P_0$ un $F$-sous-groupe parabolique minimal 
de $G$ tel que $S\subset P_0\subset P_\mu$. Quitte ˆ remplacer $Z$ par $g\cdot Z$ pour un $g\in G(F)$, on peut supposer 
que $P_0\subset {_FP_Z}$. \'Ecrivons ${_F\scrX_Z}= V_{\lambda,k}$ avec 
$\check{X}(S)\cap \Lambda_{F,Z}^{\rm opt}=\{\lambda\}$ et $k=m_Z(\lambda)\geq 1$. Soit $x\in G(F)$ tel que $x^{-1}\cdot v_1\in {_F\scrX_Z}$; l'ŽlŽment $v_1$ appartient donc ˆ $x\cdot V_{\lambda,k} \cap V_{\mu,0}$. 
\'Ecrivons (dŽcomposition de Bruhat) $x= u n_w p$ avec $u\in U_{P_0}(F)$, $w \in W^G(S)=N^G(S)/Z^G(S)$ et $p\in P_0(F)$, o $n_w$ dŽsigne un reprŽsentant 
de $w$ dans $N^G(S)(F)$. Puisque $V_{\lambda,k}$ et $V_{\mu,0}$ sont 
$P_0$-invariant, $u^{-1}\cdot v_1$ appartient ˆ $n_w\cdot V_{\lambda ,k}\cap V_{\mu,0}= V_{w\bullet \lambda , k}\cap V_{\mu,0}$. Comme $S\subset P_{w\bullet \lambda}= n_w P_\lambda n_w^{-1}$, 
pour tout $v_2\in V_{w\bullet \lambda,k}\cap V_{\mu,0}$, la limite $\lim_{t\rightarrow 0} t^\mu \cdot v_2$ existe et elle appartient ˆ $V_{w\bullet \lambda,k}$. D'autre part puisque 
$u \in U_0(F)\subset P_\mu(F)$, la limite $\lim_{t\rightarrow 0}t^\mu u t^{-\mu}$ existe et elle appartient ˆ $M_\mu(F)$; on la note $m$. On a donc 
$$v'= \lim_{t\rightarrow 0} t^\mu \cdot v_1 = \left(\lim_{t\rightarrow 0} t^\mu u t^{-1}\right) \cdot \left(\lim_{t\rightarrow 0} t^\mu\cdot (u^{-1}\cdot v_1)\right)\in m \cdot V_{w\bullet\lambda , k}\ptf$$
Par consŽquent 
$v'\in m n_w\cdot V_{\lambda,k} \subset G(F)\cdot V_{\lambda,k}={_F\bsfrX_Z}$. 

Le point (iii) est clair.
\end{proof}

\begin{remark}\label{ensembles F-saturŽs et Zariski-top}
\textup{
\begin{enumerate}
\item[(i)]Pour $s\in \mbb{Q}_+$, posons\index{NFs=@$\FN_{\leq s}$} $$\FN_{\leq s}= \{v\in \FN\,\vert \, \bs{q}_F(v)\leq s\}\ptf$$ 
D'aprs \ref{prop 2.9 de [H2] rat}\,(i), c'est la rŽunion (finie) des ensembles ${_F\bsfrX_v}$ avec $v\in \FN$ tel que $\bs{q}_F(v)\leq s$. 
Observons que $\FN_{\leq s}= \bigcap_{s<r} \FN_{<r}$ et qu'il existe un $\epsilon>0$ tel que $\FN_{\leq s}= \FN_{<s+\epsilon}$. 
\item[(ii)] Pour tout sous-ensemble $Z\subset \FN$ uniformŽment $F$-instable, puisque le quotient $G/{_FP_Z}$ 
est une variŽtŽ projective (donc complte), l'ensemble $$G\cdot {_F\scrX_Z}=\{g\cdot v\,\vert \, g\in G,\, v\in {_F\scrX_Z}\}$$ est une sous-variŽtŽ fermŽe 
de $V$ (\textit{cf.} la preuve de \cite[ch.~IV, 11.9\,(1)]{B}). Si le corps $F$ est infini, puisque $G(F)$ est dense dans $G$ \cite[ch.~V, 18.3]{B}, on en 
dŽduit que ${_F\bsfrX_Z}=G(F)\cdot {_F\scrX_Z}$ est dense dans $G\cdot {_F\scrX_Z}$. Par consŽquent (toujours si $F$ est infini), 
l'ensemble ${_F\bsfrX_Z}$ est fermŽ \textit{dans $\FN$} si et seulement s'il co\"{\i}ncide avec l'intersection 
$\FN \cap (G\cdot {_F\scrX_Z}) $ (c'est bien sžr toujours le cas si $F=\overline{F}$). 
\item[(iii)] Si pour tout $v\in \FN$, ${_F\bsfrX_v}$ est fermŽ \textit{dans $\FN$}, alors pour tout $s\in \mbb{Q}_+$, $\FN_{<s}$ est fermŽ \textit{dans $\FN$}. 
\end{enumerate}
}
\end{remark}

Pour $Z\in \FN$ uniformŽment $F$-instable, on pose\index{XbsFZ+@${_F\bsfrX_Z^+}$}  
$${_F\bsfrX_Z^+} = \FN \cap (G\cdot {_F\scrX_Z})\ptf$$ 
On a les inclusions $$ {_F\bsfrX_Z}\subset {_F\bsfrX_Z^+}\subset \FN\,;$$ et 
si $F$ est infini, ${_F\bsfrX_Z^+}$ est la fermeture de ${_F\bsfrX_Z}$ dans $\FN$ (\ref{ensembles F-saturŽs et Zariski-top}\,(ii)). 

\vskip2mm
\textit{\S7.\hskip1mm Une hypothse de rŽgularitŽ. ---} Jusqu'ˆ prŽsent on n'a fait aucune hypothse de rŽgularitŽ sur le point-base $e_V$. On la fait maintenant. 

\begin{hypothese}\label{hyp reg}
Le point-base $e_V$ est rŽgulier (\textit{i.e.} non singulier) dans la variŽtŽ $V$. (Attention: cela n'entra"ne pas que 
$e_V$ soit rŽgulier dans $\ES{N}=\ES{N}^G(V,e_V)$.) 
\end{hypothese}

Jusqu'ˆ la fin de ce paragraphe (\S7), on suppose que l'hypothse \ref{hyp reg} est vŽrifiŽe.

Hesselink a prouvŽ \cite[3.8]{H2} que pour tout sous-ensemble $Z\subset \ES{N}$ uniformŽment instable, la sous-variŽtŽ 
fermŽe $\scrX_Z = {_{\smash{\overline{F}}}\scrX_Z}$ de $\ES{N}= {_{\smash{\overline{F}}}\ES{N}}$ est isomorphe ˆ son espace tangent $T_{e_V}(\scrX_Z)$, qui est un espace affine; 
en particulier elle est irrŽductible et non-singulire. 
Puisque $G$ est connexe donc irrŽductible, la sous-variŽtŽ fermŽe (d'aprs \ref{ensembles F-saturŽs et Zariski-top}\,(ii)) $G$-invariante 
$\bsfrX_Z=G\cdot \scrX_Z$ de $\ES{N}$ 
est elle aussi irrŽductible. 

On a un rŽsultat analogue ˆ celui de \cite[3.8]{H2} pour les sous-ensembles $Z\subset \FN$ qui sont uniformŽment $F$-instables: 

\begin{proposition}\label{saturŽ espace affine}
(On suppose que l'hypothse \ref{hyp reg} est vŽrifiŽe.) Soit $Z\subset \FN$ un sous-ensemble uniformŽment $F$-instable. L'espace tangent $T_{e_V}(X)$ du $F$-saturŽ $X={_F\scrX_Z}$ de $Z$ 
est un sous-ensemble $F$-saturŽ de $T_{e_V}(V)$ qui est $F$-isomorphe ˆ $X$ et vŽrifie $$\bs{\Lambda}_{F,T_{e_V}(X)}= \bs{\Lambda}_{F,X}=\bs{\Lambda}_{F,Z}\ptf$$ 
\end{proposition}

\begin{proof}
Elle est identique \`a celle de \cite[3.8]{H2}, compte-tenu de la propriŽtŽ suivante: 
si $S$ est un tore $F$-dŽployŽ maximal 
de $G$, puisque $S$ est linŽairement rŽductif, il existe un $S$-logarithme $\phi: V \rightarrow T_{e_V}(V)$ 
au sens de \cite[3.1]{H2} qui soit dŽfini sur $F$ (\textit{cf.} \cite[17.6]{EGA}).
\end{proof}

Si $Z\subset \FN$ est uniformŽment $F$-instable, d'aprs \ref{saturŽ espace affine}, son $F$-saturŽ ${_F\scrX_Z}$ 
est un $F$-espace affine (\textit{i.e.} ${_F\scrX_Z}\simeq_F \mbb{A}_F^n$); en particulier c'est une variŽtŽ irrŽductible non singulire. 

\begin{remark}\label{densitŽ F-saturŽ 1}
\textup{
(On suppose que l'hypothse \ref{hyp reg} est vŽrifiŽe.) Si le corps $F$ est infini, puisque $F^n$ est (Zariski-)dense dans $\mbb{A}_F^n$, pour tout sous-ensemble $Z\subset \FN$ uniformŽment $F$-instable, l'ensemble 
$\scrX_{F,Z}= V(F) \cap {_F\scrX_Z}$ est dense dans ${_F\scrX_Z}$.
}
\end{remark}

\subsection{La stratification de Hesselink}\label{la stratification de Hesselink}
Continuons avec les hypo\-thses de \ref{la variante de Hesselink}. Sauf mention expresse du contraire, on 
ne suppose pas que l'hypothse \ref{hyp reg} ($e_V$ est rŽgulier dans $V$) est vŽrifiŽe.

\vskip2mm
\textit{\S1.\hskip1mm Les $F$-lames et les $F$-strates. ---} Pour $v\in \FN$, on pose\index{YFv@$\FYv$}
$$\FYv =\{v'\in \FN\,\vert \, \bs{\Lambda}_{F,v'}= \bs{\Lambda}_{F,v} \}\ptf$$ 

\begin{lemme}\label{caractŽrisation F-lames avec Lambda-opt et m}
Pour $v\in \FN$, on a $$\FYv= \{v'\in \FN\,\vert \,\hbox{$\Lambda_{F,v'}^{\rm opt}= \Lambda_{F,v}^{\rm opt}$ et $m_{F,v'}=m_{F,v}$}\}\ptf$$
\end{lemme}

\begin{proof}
Si $v=e_V$, il n'y a rien ˆ dŽmontrer. Si $v\in \FN\smallsetminus \{e_V\}$, l'inclusion 
$$\{v'\in \ES{N}\,\vert \, \hbox{$\Lambda_{F,v'}^{\rm opt}= \Lambda_{F,v}^{\rm opt}$ et $m_{F,v'}=m_{F,v}\}\subset {_F\scrY_v}$}$$ est claire (d'aprs \ref{variante Hesselink rat}). 
Quant ˆ l'inclusion inverse, soient $\lambda\in \Lambda_{F,v}^{\rm opt}$ et $k= m_v(\lambda)$. Puisque $\wt{\lambda}= \frac{1}{k}\lambda$ appartient ˆ $\bs{\Lambda}_{F,v}$ (\ref{variante Hesselink rat}), 
pour $v'\in {_F\scrY_v}$ on a $m_{v'}(\tilde{\lambda})\geq 1$ 
\cad $m_{v'}(\lambda)\geq k$. Par consŽquent ${_F\scrY_v} \subset V_{\lambda,k}\;(= {_F\scrX_v})$. 
Si $v'\in V_{\lambda, k +1}$, alors $\lambda' = \frac{1}{k+1}\lambda$ vŽrifie $m_{v'}(\lambda')\geq 1$ et $\| \lambda' \| = \frac{k}{k+1} \| \tilde{\lambda}_v\|< \| \tilde{\lambda}_v \|$; par consŽquent 
$v'\notin {_F\scrY_v}$. D'o l'inclusion inverse.
\end{proof}

Pour $v\in \FN$, on pose\index{YFbsv@${_F\bsfrY_v}$} 
$${_F\bsfrY_v}= G(F)\cdot {_F\scrY_v}= \bigcup_{g\in G(F)} {_F\scrY_{g\cdot v}}\ptf$$ 
Puisque $\bs{\Lambda}_{F,g\cdot v}= g\bullet \bs{\Lambda}_{F,v}$ pour tout $g\in G(F)$, on a aussi 
$${_F\bsfrY_v}= \{v'\in \FN\, \vert \, \hbox{il existe un $g\in G(F)$ tel que $\bs{\Lambda}_{F,v'} = g\bullet \bs{\Lambda}_{F,v}$}\}\ptf$$

\begin{definition}\label{def F-lames et F-strates}
\textup{
Les ensembles ${_F\scrY_v}$, resp. ${_F\bsfrY_v}$, avec $v\in \FN$ sont appelŽs 
\textit{$F$-lames}, resp. \textit{$F$-strates} (de $\FN$). Les $F$-strates sont donc les classes de $G(F)$-conjugaison de 
$F$-lames. }
\end{definition}

\begin{remark}
\textup{D'aprs \ref{rŽduction au cas d'un G-module rat}, il existe un $G$-module $W$ dŽfini sur $F$ et une $F$-immersion 
fermŽe $G$-Žquivariante $\iota: V\rightarrow W$ tels que $\iota(e_V)=0$. Soit $v\in \FN$. Posons $w= \iota(v)$; c'est un ŽlŽment de $\FN^{\,G}(W,0)$. Il dŽfinit comme plus haut 
des sous-ensembles ${_F\scrX_w}$, ${_F\scrY_w}$, ${_F\bsfrX_w}$ et ${_F\bsfrY_w}$ de 
$\FN^{\,G}(W,0)$. Puisque $\bs{\Lambda}_{F,w}= \bs{\Lambda}_{F,v}$, on a $$\iota({_F\scrX_v})= \iota(V)\cap {_F\scrX_w}\quad\hbox{et}\quad 
\iota({_F\scrY_v})= \iota(V) \cap {_F\scrY_w}\ptf$$ On a aussi 
$$\iota({_F\bsfrX_v})= \iota(V) \cap {_F\bsfrX_w}\quad\hbox{et}\quad \iota({_F\bsfrY_v})= \iota(V) \cap {_F\bsfrX_v}\ptf$$On peut donc en principe ramener la plupart des questions concernant les 
$F$-lames et les $F$-strates de $\FN$ au cas o $V$ est un $G$-module dŽfini sur $F$ (avec $e_V=0$). 
}
\end{remark}

\begin{lemme}\label{lemme 1 F-lames et F-strates}
Soit $v\in \FN$.
\begin{enumerate}
\item[(i)]$\FYv= \{v'\in \FXv\,\vert \, \bs{q}_F(v')=\bs{q}_F(v)\}$.
\item[(ii)]$v'\in {_F\scrX_v}\Rightarrow \bs{q}_F(v')\leq \bs{q}_F(v)$.
\item[(iii)] ${_F\bsfrY_v}=\{v'\in {_F\bsfrX_v}\,\vert \, \bs{q}_F(v')= \bs{q}_F(v)\}$.
\item[(iv)] $v'\in {_F\bsfrX_v} \Rightarrow \bs{q}_F(v')\leq  \bs{q}_F(v)$.
\item[(v)]${_F\scrX_v}$ est ${_FP_v}$-invariant et ${_F\scrY_v}$ est $P_{F,v}$-invariant. 
\end{enumerate}
\end{lemme}

\begin{proof}
\'Ecrivons ${_F\scrX_v}= V_{\mu,1}$ avec $\mu\in \bs{\Lambda}_{F,v}$. 

Prouvons (i). Si $v' \in \FYv$ alors $\mu\in \bs{\Lambda}_{F,v'}$; par consŽquent 
$v'\in V_{\mu,1}=\FXv$ et $\bs{q}_F(v') = \| \mu \| = \bs{q}_F(v)$. Inversement si $v'\in\FXv$ vŽrifie $\bs{q}_F(v')= \bs{q}_F(v)= \l\mu\l$, alors $\mu$ appartient 
ˆ $\bs{\Lambda}_{F,v'}$ et $\bs{\Lambda}_{F,v'}= P_\mu \bullet \mu= \bs{\Lambda}_{F,v}$. 

Prouvons (ii). Si $v'\in {_F\scrX_v}$, puisque $m_{v'}(\mu)\geq 1$, on a $\bs{q}_F(v') \leq \| \mu \| = \bs{q}_F(v)$. 

Puisque $\bs{q}_F(g\cdot v')= \bs{q}_F(v')$ pour tout $g\in G(F)$, les points (iii) et (iv) rŽsultent de (i) et (ii).

Prouvons (v). L'ensemble $\FXv=V_{\mu,0}$ est clairement ${_FP_v}\;(=P_\mu$)-invariant. Quant 
ˆ $\FYv$, si $p\in P_{F,v}$ et $v'\in {_F\scrY_v}$, alors $p\cdot v'$ est dans $\FXv$ et comme $\bs{q}_F(p\cdot v')= \bs{q}_F(v')$, il est dans $\FYv$ (d'aprs (i)). 
\end{proof}
 
Observons que pour $v,\,v'\in \FN$, on a $${_F\scrY_v}={_F\scrY_{v'}}\quad \hbox{si et seulement si}\quad{_F\scrX_v}= {_F\scrX_{v'}}\vg$$
$${_F\bsfrY_v}= {_F\bsfrY_{v'}}\quad\hbox{si et seulement si}\quad {_F\bsfrX_v}= {_F\bsfrX_{v'}}\ptf$$
On a aussi
$${_F\scrX_v}\smallsetminus {_F\scrY_v} = {_F\scrX_v}\cap \FN_{<\bs{q}_F(v)}\vg$$ 
$$ {_F\bsfrX_v}\smallsetminus {_F\bsfrY_v} = {_F\bsfrX_v}\cap \FN_{<\bs{q}_F(v)}\ptf$$ 
La gŽomŽtrie du bord ${_F\bsfrX_v} \smallsetminus {_F\bsfrY_v}$ est \textit{a priori} assez compliquŽe. Pour $v'\in {_F\bsfrX_v}\smallsetminus {_F\bsfrY_v}$, 
la $F$-lame ${_F\scrY_{v'}}$ n'est pas forcŽment contenu dans ${_F\bsfrX_v}$; ou, ce qui revient au mme, 
la $F$-strate ${_F\bsfrY_{v'}}$ n'est pas forcŽment contenu dans ${_F\bsfrX_v}$ (et cela, mme si $F=\overline{F}$ \cite[4.3, remark]{H2}). 
En gŽnŽral, on a $${_F\bsfrX_v} \smallsetminus {_F\bsfrY_v} = \bigcup_{v'\in \FN \vert \bs{q}_F(v')<\bs{q}_F(v)} \hspace{-1em}{_F\bsfrX_{v'}}\cap {_F\bsfrX_v}= \coprod_{\bsfrY'}\hspace{0.5em}\bsfrY'\cap {_F\bsfrX_v}$$ o 
$\bsfrY'$ parcourt l'ensemble des $F$-strates de $\FN$ telles que $\bsfrY'\cap ({_F\bsfrX_v}\smallsetminus {_F\bsfrY_v}) \neq \emptyset$; ou, ce qui revient au mme, telles que 
$\bsfrY'\cap ({_F\scrX_v}\smallsetminus {_F\scrY_v}) \neq \emptyset$. On parle nŽanmoins de \guill{stratification} car la fonction $\bs{q}_F$ dŽfinit un ordre total strict sur $\FN$.

Les $F$-lames sont en bijection avec les sous-ensembles ${_F\scrX_v}\subset \FN$ ($v\in \FN$) qui sont des cas particuliers d'ensembles $F$-saturŽs ${_F\scrX_Z}$ ($Z\subset \FN$ uniformŽment $F$-instable). 
DŽterminer parmi les sous-ensembles $F$-saturŽs de $\FN$ ceux qui sont de la forme ${_F\scrX_v}$ avec $v\in \FN$ est une question difficile, reliŽe 
ˆ l'autre question difficile suivante:  pour deux $F$-strates ${_F\bsfrY_v},\,{_F\bsfrY_{v'}}$ de $\FN$, quand a-t-on l'inclusion ${_F\bsfrY_{v'}} \subset {_F\bsfrX_v}$? Pour ces questions 
dans le cas gŽomŽtrique (\textit{i.e.} $F=\overline{F}$), on renvoie ˆ \cite[5]{H2}. 
Cela nous amne naturellement ˆ considŽrer l'hypothse suivante: 

\begin{hypothese}\label{hypo sur le bord}
Pour tout $v\in \FN$, le bord ${_F\bsfrX_v}\smallsetminus {_F\bsfrY_v}$ est rŽunion (finie) de $F$-strates de $\FN$; en d'autres termes pour tout $v'\in {_F\bsfrX_v}\smallsetminus {_F\bsfrY_v}$, 
on a ${_F\bsfrY_{v'}}\subset {_F\bsfrX_v}$. 
\end{hypothese}

\begin{remark}\label{hypo sur le bord en dŽfaut}
\textup{L'hypothse sur le bord \ref{hypo sur le bord} n'est en gŽnŽral pas vŽrifiŽe, mme pour $F=\overline{F}$ \cite[4.3, remark]{H2} (\textit{cf.} l'exemple des \guill{ternary sextic forms} dans \cite[6.4]{H2}).
}
\end{remark}

Sauf mention expresse du contraire, on ne suppose pas que l'hypothse \ref{hypo sur le bord} est vŽrifiŽe. D'aprs \ref{prop 2.9 de [H2] rat}, on a la

\begin{proposition}\label{F-strates et c-F-top}
\begin{enumerate}
\item[(i)] Il n'y a qu'un nombre fini de $F$-strates ${_F\bsfrY_v}$ avec $v\in \FN$.
\item[(ii)] $\FN$ est la rŽunion (finie) disjointe des $F$-strates ${_F\bsfrY_v}$ avec $v\in \FN$.
\item[(iii)] Pour $v\in \FN$, la $F$-strate ${_F\bsfrY_v}$ est la rŽunion disjointe (en gŽnŽral infinie) des $F$-lames qu'elle contient: on a ${_F\bsfrY_v}= \coprod_{g\in G(F)/P_{F,v}}{_F\scrY_{g\cdot v}}$. 
\item[(iv)] Pour $v\in \FN$, ${_F\bsfrX_v}\smallsetminus {_F\bsfrY_v}$ est c-$F$-fermŽ (dans $V$).
\item[(v)] Pour $s\in \mbb{Q}_+$, l'ensemble $\FN_{\leq s} \smallsetminus \FN_{<s}$ est la rŽunion (Žventuellement vide) des 
$F$-strates ${_F\bsfrY_v}$ avec $v\in \ES{N}_F$ tel que $\bs{q}_F(v)=s$. 
\end{enumerate}
\end{proposition}

\begin{corollary}
Les $F$-strates sont localement $\cF$-fermŽes (dans $V$).
\end{corollary}

\begin{proof}
Pour tout $v\in \FN$, la $F$-strate ${_F\bsfrY_v}$ est $\cF$-ouverte dans ${_F\bsfrX_v}$ (\ref{F-strates et c-F-top}\,(iv)) et d'aprs \ref{prop 2.9 de [H2] rat}\,(ii), ${_F\bsfrX_v}$ est $\cF$-fermŽ (dans $V$), d'o le corollaire.
\end{proof}

\begin{remark}
\textup{
Pour $v\in \FN$, le fait que la $F$-strate ${_F\bsfrY_v}$ soit $\cF$-ouverte dans ${_F\bsfrX_v}$ n'implique \textit{a priori} pas 
qu'elle soit $\cF$-dense dans ${_F\bsfrX_v}$: on a toujours $\smash{\overline{_F\bsfrY_v}}^{(\cF)}\subset {_F\bsfrX_v}$ 
mais l'inclusion peut tre stricte. }
\end{remark}

\begin{remark}
\textup{
\begin{enumerate}
\item[(i)] Si $F=\overline{F}$, les lames et les strates de $\ES{N}$ sont des sous-variŽtŽs localement fermŽes dans $V$. En effet les ensembles $\bsfrX_v$ ($v\in \ES{N}$) sont 
des sous-variŽtŽs fermŽes de $V$, par consŽquent les ensembles $\ES{N}_{<s}$ ($s\in \mbb{Q}_+$) sont eux aussi des sous-variŽtŽs fermŽes de $V$; et pour tout $v\in \ESN$, on a: 
\begin{itemize}
\item $\scrY_v= \scrX_v \smallsetminus (\scrX_v\cap \ES{N}_{< \bs{q}(v)})$ est une sous-variŽtŽ ouverte de $\scrX_v$;
\item $\bsfrY_v = \bsfrX_v \smallsetminus (\bsfrX_v\cap \ES{N}_{< \bs{q}(v)})$ est une sous-variŽtŽ ouverte de $\bsfrX_v$. 
\end{itemize}
Observons que si l'hypothse \ref{hyp reg} est vŽrifiŽe, les variŽtŽs $\scrX_v$ et $\bsfrX_v$ sont irrŽductibles; dans ce cas $\scrY_v$, resp. $\bsfrY_v$, 
est dense dans $\scrX_v$, resp. $\bsfrX_v$. 
\item[(ii)] Si pour tout $v\in \FN$, ${_F\bsfrX_v}$ est (Zariski-)fermŽ dans $\FN$, alors pour tout $s\in \mbb{Q}_+$, $\FN_{<s}$ est fermŽ dans 
$\FN$. Dans ce cas pour tout $v\in \FN$, on a (comme dans le cas o $F=\overline{F}$): 
\begin{itemize}
\item la $F$-lame ${_F\scrY_v}$ est (Zariski-)ouverte dans son $F$-saturŽ ${_F\scrX_v}$; 
\item la $F$-strate ${_F\bsfrY_v}$ est (Zariski-)ouverte dans ${_F\bsfrX_v}$. 
\end{itemize}
\item[(iii)]Supposons $F$ infini. Pour $v\in \FN$, l'ensemble ${_F\bsfrX_v}$ est fermŽ dans $\FN$ si et seulement s'il co\"{\i}ncide avec 
${_F\bsfrX_v^+}= \FN \cap (G\cdot {_F\scrX_v})$ (\textit{cf.} \ref{ensembles F-saturŽs et Zariski-top}\,(ii)). Il est en gŽnŽral difficile de calculer cette intersection car 
$\FN$ \textit{n'est pas} (en gŽnŽral) un sous-ensemble algŽbrique de $V$ (\textit{cf.} \ref{le cas de PGL(2)}). De plus si $s= \bs{q}_F(v)$,  \textit{a priori} 
${_F\bsfrX_v^+}$ n'est pas contenu dans $\FN_{\leq s}$\ptf
\end{enumerate}
}
\end{remark}

\textit{\S2.\hskip1mm $F$-lames standard. ---} Soit $P_0$ un $F$-sous-groupe parabolique minimal de $G$. Un $F$-sous-groupe parabolique $P$ de $G$ 
est dit \textit{standard} s'il contient $P_0$. 

\begin{definition}\label{F-lame standard et ŽlŽment en ps}
\textup{Un ŽlŽment $v\in \FN$ est dit \textit{en position standard} si le $F$-sous-groupe parabolique ${_FP_v}$ de $G$ 
est standard. Une $F$-lame ${_F\scrY_v}$ de $\FN$ est dite \textit{standard} si $v$ (\textit{i.e.} si tout ŽlŽment de ${_F\scrY_v}$) 
est en position standard.
}
\end{definition}

Si ${_F\scrY_v}$ et ${_F\scrY_{v'}}$ sont deux $F$-lames standard de $\FN$ avec ${_F\scrY_{v'}}= g \cdot {_F\scrY_v}$ pour un 
ŽlŽment $g\in G(F)$, alors ${_FP_{v'}}= g({_FP_v})g^{-1}$; par suite on a ${_FP_{v'}}= {_FP_v}$ (puisque les $F$-sous-groupes paraboliques ${_FP_v}$ et ${_FP_{v'}}$ sont standard) et 
$g\in  P_{F,v}= {_FP_v}(F)$ (car ${_FP_v}$ est son propre normalisateur), donc ${_F\scrY_v'}={_F\scrY_v}$. Ainsi toute $F$-strate de $\FN$ 
contient une unique $F$-lame standard et l'application ${_F\scrY_v} \mapsto {_F\bsfrY_v}= G(F)\cdot {_F\scrY_v}$ est 
une bijection de l'ensemble des $F$-lames standard de $\FN$ sur l'ensemble des $F$-strates de $\FN$. Les ensembles \textit{finis} suivants sont naturellement en bijection:
\begin{itemize}
\item les $F$-lames standard de $\FN$;
\item les ensembles $F$-saturŽs ${_F\scrX_v}$ avec $v\in \FN$ en position standard;
\item les $F$-strates de $\FN$;
\item les ensembles ${_F\bsfrX_v}$ avec $v\in \FN$;  
\end{itemize}

Soit $A_0$ un tore $F$-dŽployŽ maximal de $P_0$. Un co-caractre virtuel $\mu \in \check{X}(G)_{\mathbb{Q}}$ est dit \textit{en position standard} (par rapport ˆ $(P_0,A_0)$) 
si $P_0\subset P_\mu$ et $A_\mu \subset A_0$; o $A_\mu$ dŽsigne le tore central $F$-dŽployŽ maximal de $M_\mu$. Observons qu'on a toujours l'inclusion 
$\mathrm{Im}(\mu) \subset A_\mu$ et que $A_\mu \subset A_0$ si et seulement si $A_0 \subset M_\mu$. 

Si $v\in \FN$, il existe un $g\in G(F)$ tel que la $F$-lame ${_F\scrY_{g\cdot v}}$ soit standard. Alors le co-caractre virtuel 
$\mu \in \check{X}(A_0)_{\mbb{Q}}$ dŽfini par $\bs{\Lambda}_{F,g\cdot v} \cap \check{X}(A_0)_{\mbb{Q}}=\{\mu\}$ est en position standard. 
On dit alors que $\mu$ est un co-caractre virtuel \textit{$F$-optimal} (au sens o il est $(F,v)$-optimal pour un $v\in \FN$) \textit{en position standard}. Observons que $\mu$ est 
l'unique ŽlŽment de l'ensemble $$G(F)\bullet \bs{\Lambda}_{F,v}= \{{\rm Int}_g\circ \mu\,\vert\, g\in G(F),\,\mu\in \bs{\Lambda}_{F,v}\}$$ qui soit en position standard.
On obtient de cette manire une bijection naturelle entre:
\begin{itemize}
\item l'ensemble des $F$-strates de $\FN$;
\item le sous-ensemble (fini)\index{LambdaFst@$\bs{\Lambda}_{F,\mathrm{st}}$} $$\bs{\Lambda}_{F,\mathrm{st}}= \bs{\Lambda}_{F,\mathrm{st}}^G(V,e_V)\subset \check{X}(A_0)_{\mbb{Q}}$$ formŽ des co-caractres virtuels qui sont $F$-optimaux et en position standard.
\end{itemize}

\begin{remark}\label{co-car st et u-st}
\textup{
\begin{enumerate}
\item[(i)]
Notons $\FN'_{\textrm{st}}$ l'ensemble des ŽlŽments de $\FN$ en position standard, \cad la rŽunion (finie) des $F$-lames standard, et posons  
$$\FN_{\rm st}\bydef \bigcup_{v\in \FN'_{\mathrm{st}}}{_F\scrX_v}\subset \FN \ptf$$ Observons que 
$\FN'_{\textrm{st}}$ est $P_0(F)$-invariant et $\FN_{\textrm{st}}$ est $P_0$-invariant, et que l'on a 
$$G(F)\cdot \FN'_{\textrm{st}} = G(F) \cdot \FN_{\textrm{st}} = \FN\ptf$$ Puisque $\FN_{\textrm{st}}$ 
est la rŽunion \textit{finie} des ensembles ${_F\scrX_v}$ avec $v\in \FN$ en position standard, c'est une sous-variŽtŽ fermŽe  
de $V$ dŽfinie sur $F$. Par exemple si $V$ est le groupe $G$ lui-mme muni de l'action par conjugaison, alors $\FN_{\textrm{st}}$ est le radical unipotent 
$U_0=U_{P_0}$ de $P_0$.
\item[(ii)]On dŽfinit comme en \ref{F-lame standard et ŽlŽment en ps} la notion de sous-ensemble uniformŽment $F$-instable $Z\subset \FN$ \textit{en position standard}. Notons 
$\bs{\Lambda}_{F,\mathrm{u-st}}$ le sous-ensemble de $\check{X}(A_0)_{\mbb{Q}}$ formŽ des $\mu$ tels que 
$\bs{\Lambda}_{F,Z}\cap \check{X}(A_0)_{\mbb{Q}}=\{\mu\}$ pour un sous-ensemble uniformŽment $F$-instable $Z\subset\FN$ en position standard. Puisque 
$\bs{\Lambda}_{F,Z}= \bs{\Lambda}_{F,{_F\scrX_Z}}$ (\ref{lemma 2.8 de [H2]}\,(i)) et que l'ensemble des classes de $G(F)$-conjugaison de sous-ensembles $F$-saturŽs de $\FN$ est fini 
(\ref{prop 2.9 de [H2] rat}\,(i)), l'ensemble $\bs{\Lambda}_{F,\mathrm{u-st}}$ est lui aussi fini. 
On a clairement l'inclusion $\bs{\Lambda}_{F,\mathrm{st}}\subset \bs{\Lambda}_{F,\mathrm{u-st}}$. Supposons de plus que $V$ soit un $G$-module. Alors 
on peut comme en \cite[5.4]{H2} dŽfinir la notion de sous-ensemble (\textit{$F$-instable}) \textit{$F$-saturŽ} de l'ensemble des poids $\ES{R}'_{A_0}=\ES{R}'^V_{A_0}$ de $A_0$ dans $V$, et prouver l'analogue 
de la proposition \cite[5.5]{H2}: $\bs{\Lambda}_{F,\mathrm{u-st}}$ est en bijection naturelle avec l'ensemble des classes de $W^G(A_0)$-conjugaison 
de sous-ensembles $F$-saturŽs de $\ES{R}'_{A_0}$. 
\item[(iii)]On peut introduire les notions de $F$-sous-groupe parabolique \textit{semi-standard} de $G$, de $F$-lame \textit{semi-standard} de $\NF$, etc., en remplaant 
la condition \guill{contient $P_0$} par la condition \guill{contient $A_0$}. Notons $\FN'_{\mathrm{s-st}}$ l'ensemble des ŽlŽments de $\FN$ en position semi-standard, \cad 
la rŽunion (finie) des $F$-lames semi-standard de $\FN$. Alors $\FN'_{\mathrm{s-st}}=W^G(A_0)\cdot \FN'_{\mathrm{st}}$ et d'aprs la dŽcomposition 
de Bruhat, on a $$\FN= U_0(F)\cdot \FN'_{\mathrm{s-st}}\ptf$$
\end{enumerate}}
\end{remark}

\vskip2mm
\textit{\S3.\hskip1mm Les applications ${_F\bs{\pi}_{v}}$ et ${_F\bs{\pi}'_{v}}$. ---} Pour $v\in \FN$, notons $G(F)\times^{P_{F,v}}\!{_F\scrX_v}$ le quotient du produit 
$G(F)\times {_F\scrX_v}$ par $P_{F,v}$ pour l'action de $P_{F,v}$ \textit{ˆ droite} donnŽe par 
$$(g,x)\cdot p = (gp, p^{-1}\cdot x)\ptf$$ Pour un ŽlŽment $(g,x)\in G(F)\times {_F\scrX_v}$, on note $[g,x]$ son image dans $G(F)\times^{P_{F,v}}{_F\scrX_v}$ et 
on munit $G(F)\times^{P_{F,v}}{_F\scrX_v}$ de l'action de $G(F)$ (ˆ gauche) donnŽe par 
$$g'\cdot[g,x]= [g'g, x]\quad \hbox{pour tout} \quad g'\in G(F)\ptf$$ 
Considrons l'application\index{piFv@${_F\bs{\pi}_{v}}$} 
$${_F\bs{\pi}_{v}}: G(F)\times^{P_{F,v}}{_F\scrX_v} \rightarrow {_F\bsfrX_v}\vgq [g,x] \mapsto g \cdot x\ptf$$ 
Elle est $G(F)$-Žquivariante et surjective. 

\begin{lemme}\label{l'application pi}
Soit $v\in \FN$. 
\begin{enumerate}
\item[(i)] $({_F\bs{\pi}_{v}})^{-1} ({_F\bsfrY_v})= G(F)\times^{P_{F,v}}{_F\scrY_v}$.
\item[(ii)] ${_F\bs{\pi}_{v}}$ induit une application bijective\index{pi'Fv@${_F\bs{\pi}'_{v}}$} ${_F\bs{\pi}'_{v}}:G(F)\times^{P_{F,v}}{_F\scrY_v} \rightarrow {_F\bsfrY_v}$.
\end{enumerate} 
\end{lemme}

\begin{proof}
Pour $(g,v')\in G(F)\times {_F\scrX_v}$, si $g\cdot v'\in {_F\bsfrY_v}$, alors $\bs{q}_F(v')= \bs{q}_F(v)$ et $v'\in {_F\scrY_v}$. Cela prouve (i). 
Quant au point (ii), soient $(g_1,v_1)$ et $(g_2,v_2)$ deux ŽlŽments 
de $G(F)\times {_F\scrY_v}$ tels que $g_1\cdot v_1= g_2\cdot v_2$. Posons $g= g_1^{-1} g_2$. On a donc 
$g\cdot v_2 = v_1 \in {_F\scrY_v}= {_F\scrY_{v_1}} \subset {_F\scrX_{v_1}}$. 
Cela entra"ne que $g$ appartient ˆ $P_{F,v_1}= P_{F,v}$ et donc que $[g_1,v_1]=[g_1g,g^{-1}\cdot v_1]=[g_2,v_2]$. 
\end{proof}

\textit{\S4.\hskip1mm Points $F$-rationnels. ---} Rappelons que l'on a posŽ $$\NF=V(F) \cap \FN\ptf$$ Pour $s\in \mbb{Q}_+$, on pose\index{NFsrat@$\ES{N}_{F,<s}$}\index{NFsrat=@$\ES{N}_{F,\leq s}$} 
$$\ES{N}_{F,<s} = V(F) \cap \FN_{<s} \quad \hbox{et} \quad \ES{N}_{F,\leq s} = V(F) \cap \FN_{\leq s}\ptf$$ 
Pour $v\in \NF$, on pose\index{XFratv@$\scrX_{F,v}$}\index{XbsFratv@$\bsfrX_{F,v}$}\index{YFratv@$\scrY_{F,v}$}\index{YbsFratv@$\bsfrY_{F,v}$} 
$$\scrX_{F,v}=V(F)\cap {_F\scrX_v} \quad\hbox{et}\quad \bsfrX_{F,v}= V(F)\cap {_F\bsfrX_v}\vg$$
$$\scrY_{F,v}= V(F)\cap {_F\scrY_v} \quad \hbox{et} \quad \bsfrY_{F,v} = V(F) \cap {_F\bsfrY_v}\ptf$$ 
On a donc $$\bsfrX_{F,v}=G(F)\cdot \scrX_{F,v}\quad \hbox{et}\quad \bsfrY_{F,v}= G(F)\cdot \scrY_{F,v}\ptf$$ 
On pose aussi\index{XbsFratv+@$\bsfrX_{F,v}^+$} $$\bsfrX_{F,v}^+ = V(F)\cap {_F\bsfrX_v^+}= \NF \cap (G \cdot {_F\scrX_v})\ptf$$

\begin{definition}
\textup{Par abus de langage, on utilise la mme dŽfinition que \ref{def F-lames et F-strates} pour les ensembles de points $F$-rationnels: 
les ensembles $\scrY_{F,v}$, resp. $\bsfrY_{F,v}$, avec $v\in \NF$ sont appelŽs 
\textit{$F$-lames}, resp. \textit{$F$-strates}, \textit{de $\ES{N}_F$}. Les $F$-strates de $\NF$ sont donc les classes de $G(F)$-conjugaison de 
$F$-lames de $\NF$. }
\end{definition}

Les principales propriŽtŽs des $F$-lames et les $F$-strates de $\FN$ restent vraies pour les $F$-lames et les $F$-strates de $\NF$. En particulier: 
\begin{itemize}
\item Il n'y a qu'un nombre fini de $F$-strates de $\NF$.
\item On a les dŽcompositions en union disjointe: $\NF = \coprod_{v} \bsfrY_{F,v}$ o $v\in \NF$ parcourt un ensemble (fini) de reprŽsentants des $F$-strates de $\NF$; 
et pour chaque $v\in \NF$, $\bsfrY_{F,v} = \coprod_{g\in G(F)/P_{F,v}} g\cdot \scrY_{F,v}$.
\item Pour $v\in \NF$, $$\scrX_{F,v}\smallsetminus \scrX_{F,v} = \scrX_{F,v} \cap \ES{N}_{F,<\bs{q}_F(v)}\vg$$ 
$$\bsfrX_{F,v}\smallsetminus \bsfrX_{F,v} = \bsfrX_{F,v} \cap \ES{N}_{F,<\bs{q}_F(v)}\ptf$$
\item Pour $v\in \NF$, l'application ${_F\bs{\pi}_v}$ donne par restriction une application $G(F)$-Žquivariante surjective\index{piFratv@$\bs{\pi}_{F,v}$} $$\bs{\pi}_{F,v}: G(F)\times^{P_{F,v}}\scrX_{F,v} \rightarrow \bsfrX_{F,v}$$ qui 
(d'aprs \ref{l'application pi}) vŽrifie $(\bs{\pi}_{F,v})^{-1}(\bsfrY_{F,v})= G(F)\times^{P_{F,v}}\scrY_{F,v}$ et induit une application 
bijective\index{pi'Fratv@$\bs{\pi}'_{F,v}$} $$\bs{\pi}'_{F,v}: G(F)\times^{P_{F,v}}\scrY_{F,v} \rightarrow \bsfrY_{F,v}\ptf$$
\end{itemize} 

Pour $v,\, v'\in \NF$, on a $$\scrY_{F,v} = \scrY_{F,v'} \quad\hbox{si et seulement si} \quad {_F\scrY_v}= {_F\scrY_{v'}}\vg$$ 
$$\bsfrY_{F,v} = \bsfrY_{F,v'} \quad\hbox{si et seulement si} \quad {_F\bsfrY_v}= {_F\bsfrY_{v'}}\ptf$$

D'aprs ce qui prŽcde, l'application $Y\mapsto V(F)\cap Y$ induit une bijection entre:
\begin{itemize}
\item les $F$-lames, resp. $F$-strates, de $\FN$ qui possdent un point $F$-rationnel; 
\item les $F$-lames, resp. $F$-strates, de $\NF$.
\end{itemize}
Pour $v\in \FN$ tel que l'intersection $V(F) \cap {_F\scrY_v}\neq \emptyset$, resp. $V(F)\cap {_F\bsfrY_v}$, soit non vide, on note 
encore $\scrY_{F,v}$, resp. $\bsfrY_{F,v}$, cette intersection. 

On note\index{Lambda*Fst@$\bs{\Lambda}_{F,\mathrm{st}}^*$} $$\bs{\Lambda}_{F,\mathrm{st}}^*=\bs{\Lambda}_{F,\mathrm{st}}^{G,*}(V,e_V)\subset \bs{\Lambda}_{F,\mathrm{st}}$$ le sous-ensemble 
associŽ aux $F$-strates de $\FN$ qui possdent un point $F$-rationnel. Observons qu'un ŽlŽment $\mu\in \bs{\Lambda}_{F,\mathrm{st}}$ est dans $\bs{\Lambda}_{F,\mathrm{st}}^*$ si et seulement s'il 
existe un $v\in \NF$ tel que $\mu \in \bs{\Lambda}_{F,v}$; auquel cas $v$ appartient ˆ $V_{\mu,1}(F)$. On aimerait que l'inclusion ci-dessus soit une ŽgalitŽ, \cad que toute $F$-lame, ou ce qui revient au mme toute $F$-strate, de $\FN$ 
possde un point $F$-rationnel.

\begin{lemme}\label{il existe un point F-rationnel}
Supposons $F$ infini et l'hypothse \ref{hyp reg} vŽrifiŽe. Pour $v\in \FN$ tel que la $F$-lame ${_F\scrY_v}$ soit ouverte dans son $F$-saturŽ ${_F\scrX_v}$, on a 
$V(F)\cap {_F\scrY_v}\neq \emptyset$; et $\scrY_{F,v}=V(F)\cap {_F\scrY_v}$ est dense dans ${_F\scrX_v}$.
\end{lemme}

\begin{proof}
Puisque ${_F\scrX_v}\simeq_F \mbb{A}_F^n$ (d'aprs \ref{saturŽ espace affine}), l'ensemble $V(F) \cap \scrX_{F,v}$ est dense dans ${_F\scrX_v}$ et tout ouvert 
de ${_F\scrX_v}$ intersecte $V(F) \cap {_F\scrX_v}$ de manire dense. 
\end{proof}

ConsidŽrons la variante (moins forte) suivante\footnote{On verra que dans le cas o $F$ est un corps local (\textit{i.e.} un corps commutatif topologique localement compact, non discret) ou un corps global et $V$ est le 
groupe $G$ lui-mme muni de l'action par conjugaison (avec $e_V=1$), pour tout $v\in \NF$, l'ensemble $(\NF \cap \bsfrX_v)\smallsetminus \bsfrY_{F,v}\;(\supset \bsfrX_{F,v} \smallsetminus \bsfrY_{F,v})$ est rŽunion (finie) de $F$-strates de $\NF$.} de l'hypothse sur le bord \ref{hypo sur le bord}: 

\begin{hypothese}\label{hypo sur le bord faible}
Pour tout $v\in \NF$, le bord $\bsfrX_{F,v}\smallsetminus \bsfrY_{F,v}$ est rŽunion (finie) de $F$-strates de $\NF$; \textit{i.e.} pour tout $v'\in \bsfrX_{F,v}\smallsetminus \bsfrY_{F,v}$, 
on a $\bsfrY_{F,v'}\subset \bsfrX_{F,v}$. 
\end{hypothese}

Si toute $F$-strate de $\FN$ possde un point $F$-rationnel, les hypothses \ref{hypo sur le bord} et \ref{hypo sur le bord faible} sont Žquivalentes. 

\vskip2mm
\textit{\S5.\hskip1mm Descente sŽparable. ---} Soit $E/F$ une extension sŽparable (algŽbrique ou non) telle que $(E^{\mathrm{s\acute{e}p}})^{\mathrm{Aut}_F(E^{\mathrm{s\acute{e}p}})}=F$. 
D'aprs \ref{BMRT(4.7)} et \ref{variante de BMRT(4.7)}, on a le 

\begin{lemme}\label{descente sŽparable c™ne F-nilpotent}
$\NF= V(F)\cap \ES{N}_E$.
\end{lemme}

D'aprs \ref{BMRT(4.7)}, \ref{variante de BMRT(4.7)} et \ref{variante Hesselink rat}, on a aussi:

\begin{lemme}\label{descente sŽparable lames}
Soit $v\in \ES{N}_F$\footnote{L'hypothse \guill{$v\in \NF$} est indispensable ici: pour $v\in \FN$, la $\Gamma_F$-orbite $\Gamma_F(v)= \{\gamma(v)\,\vert \, \gamma \in \Gamma_F\}$ est uniformŽment 
$E$-instable et on a $\bs{\Lambda}_{F,v} = \check{X}_F(G)_{\mbb{Q}}\cap \bs{\Lambda}_{E,\Gamma_F(v)}$.}. 
\begin{enumerate}
\item[(i)]$\bs{\Lambda}_{F,v}= 
\check{X}_F(G)_\mbb{Q}  \cap \bs{\Lambda}_{E,v}$ et $\bs{q}_F(v)= \bs{q}_{E}(v)$.
\item[(ii)] ${_F\scrX_v}={_E\scrX_v}$ et ${_F\scrY_v}= {_E\scrY_v}$. 
\item[(iii)]$\scrX_{F,v}= V(F)\cap \scrX_{E,v}$ et $\scrY_{F,v}= V(F)\cap \scrY_{E,v}$.
\end{enumerate}
\end{lemme}

On en dŽduit la 
\begin{proposition}\label{descente sŽparable strates}
Pour $v\in \ES{N}_F$, on a  $\bsfrY_{F,v} = V(F) \cap \bsfrY_{E,v}$. 
\end{proposition}

\begin{proof}
L'inclusion $\bsfrY_{F,v}\;(=G(F)\cdot \scrY_{F,v}) \subset V(F) \cap \bsfrY_{E,v}$ est claire.  
Pour l'inclusion inverse, soit $v'\in  V(F) \cap \bsfrY_{E,v}\;(\subset \NF)$. 
Fixons des ŽlŽments $\mu \in \bs{\Lambda}_{F,v}$ et $\mu' \in \bs{\Lambda}_{F,v'}$. Puisque l'extension 
$E/F$ est sŽparable, $\mu$ appartient ˆ $\bs{\Lambda}_{E,v}$ et $\mu'$ appartient ˆ 
$\bs{\Lambda}_{E,v'}$ (\ref{BMRT(4.7)} et \ref{variante de BMRT(4.7)}). Puisque $\bsfrY_{E,v'}= G(E)\cdot \scrY_{E,v'}$, 
d'aprs \ref{HMrat1}, il existe un $y\in G(E)$ tel que $y\cdot v\in \scrY_{E,v'}$ et $y\bullet \mu=\mu'$. 
On a donc $$P_{\mu'} = y P_{\mu} y^{-1}\quad \hbox{et} \quad M_{\mu'} = y M_{\mu} y^{-1}$$ 
Comme les paires paraboliques $(P_\mu,M_\mu)$ et $(P_{\mu'},M_{\mu'})$ sont conjuguŽes dans $G(E)$ et dŽfinies sur $F$, elles sont conjuguŽes 
dans $G(F)$: il existe un $x\in G(F)$ tel que $$P_{\mu'}= xP_\mu x^{-1} \quad \hbox{et}\quad M_{\mu'}= xM_\mu x^{-1}\ptf$$
Posons $g=x^{-1}y\in G(E)$. On a donc 
$$gP_\mu g^{-1}= P_\mu \quad \hbox{et}\quad gM_\mu g^{-1} = M_\mu \ptf$$ 
Cela entra"ne que $g$ appartient ˆ $M_\mu(E)$. Par consŽquent $g\bullet \mu = \mu$ et $\mu' = x\bullet \mu$. 
Donc $\bs{\Lambda}_{F,v'}= x \bullet \bs{\Lambda}_{F,v}$ et $\bsfrY_{F,v'}= \bsfrY_{F,v}$, ce qui prouve l'inclusion $V(F)\cap \bsfrY_{E,v} \subset \bsfrY_{F,v}$. 
\end{proof}

\begin{remark}
\textup{
\begin{enumerate}
\item[(i)]Pour $v\in \ES{N}_F$, on a l'inclusion $\bsfrX_{F,v} \subset V(F) \cap \bsfrX_{E,v}$ mais \textit{a priori} cette inclusion est stricte. 
PrŽcisŽment, on a l'ŽgalitŽ $$V(F) \cap \bsfrX_{E,v} = \bsfrY_{F,v} \cup \left((V(F) \cap \bsfrX_{E,v}) \smallsetminus  \bsfrY_{F,v}\right)$$ 
avec $$(V(F) \cap \bsfrX_{E,v}) \smallsetminus  \bsfrY_{F,v} = \bsfrX_{E,v} \cap \ES{N}_{F,< \bs{q}_F(v)}$$ 
et l'inclusion $\bsfrX_{F,v} \smallsetminus \bsfrY_{F,v} \subset \bsfrX_{E,v} \cap \ES{N}_{F,< \bs{q}_F(v)}$ est \textit{a priori} stricte. 
\item[(ii)]Supposons de plus que $E\subset F^{\mathrm{s\acute{e}p}}$. 
Pour $v\in \ES{N}_F$, tout ŽlŽment $v'\in V(F)\cap \bsfrY_{E,v}$ dŽfinit de la manire suivante un ŽlŽment $$\bs{z}_v(v')\in 
\textrm{H}^1(\Gamma_F, P_{E,v})\ptf$$ On commence par Žcrire $v'= g\cdot v''$ avec $g\in G(E)$ et $v''\in \scrY_{E,v}$. Pour $\gamma\in \Gamma_F$, puisque 
$v' = \gamma(v') = \gamma(g)\cdot \gamma(v'')$, d'aprs \ref{l'application pi}\,(ii) il existe un (unique) ŽlŽment $p_\gamma\in P_{E,v}$ 
tel que $\gamma(g) = gp_\gamma$ et 
$\gamma(v'')= p_\gamma^{-1} \cdot v''$. L'application $\gamma \mapsto p_\gamma$ est un cocycle de $\Gamma_F$ ˆ valeurs dans $P_{E,v}$. 
Ce cocycle dŽfinit 
un ŽlŽment de $\textrm{H}^1(\Gamma_F, P_{E,v})$ qui ne dŽpend pas de la dŽcomposition $v'= g\cdot v''$ choisie; on le note $\bs{z}_v(v')$. 
On Žcrit simplement \guill{$\bs{z}_v(v')=0$} au lieu de \guill{le cocycle $\gamma \mapsto p_\gamma$ est un cobord}. 
La proposition \ref{descente sŽparable strates} dit prŽcisŽment que l'on a toujours $\bs{z}_v(v')=0$. 
\end{enumerate}
}\end{remark}

D'aprs ce qui prŽcde, l'application qui ˆ $v\in \NF$ associe $\scrY_{E,v}$, resp. $\bsfrY_{E,v}$, induit une bijection 
entre:
\begin{itemize}
\item les $F$-lames, resp. $F$-strates, de $\NF$;
\item les $E$-lames, resp. $E$-strates, de $\ES{N}_E$ qui possdent 
un point $F$-rationnel. 
\end{itemize} 
La bijection rŽciproque est donnŽe par l'application $Y\mapsto V(F)\cap Y$. 

Soit $(P_1,A_1)$ une $E$-paire parabolique minimale de $G$ telle que $A_0\subset A_1$ et $P_1\subset P_0$. On peut supposer que 
le tore $E$-dŽployŽ maximal $A_1$ de $G$ est dŽfini sur $F$. 
\`A l'aide de cette paire $(P_1,A_1)$, on dŽfinit comme plus haut le sous-ensemble 
$\bs{\Lambda}_{E,\mathrm{st}}=\bs{\Lambda}_{E,\mathrm{st}}^G(V,e_V)$ de $\check{X}_F(A_1)_{\mbb{Q}}$. 

\begin{lemme}\label{descente des co-caractres st}
On a les inclusions 
$$\bs{\Lambda}_{F,\mathrm{st}}^* \subset \check{X}(A_0)_{\mbb{Q}}\cap \bs{\Lambda}_{E,\mathrm{st}} \subset \bs{\Lambda}_{F,\mathrm{st}}\ptf$$ 
En particulier si $\bs{\Lambda}_{F,\mathrm{st}}= \bs{\Lambda}_{F,\mathrm{st}}^*$, \textit{i.e.} si toute $F$-strate de $\FN$ possde un point $F$-rationnel, 
ces inclusions sont des ŽgalitŽs. 
\end{lemme}

\begin{proof}
Puisque $\check{X}_F(A_1)_{\mbb{Q}}= \check{X}(A_0)_{\mbb{Q}}$, 
l'inclusion $\bs{\Lambda}_{F,\mathrm{st}}^* \subset \check{X}(A_0)_{\mbb{Q}}\cap \bs{\Lambda}_{E,\mathrm{st}}$ est impliquŽe par \ref{descente sŽparable lames}\,(i). 
Quant ˆ l'autre inclusion, si $\mu\in  \check{X}(A_0)_{\mbb{Q}}\cap \bs{\Lambda}_{E,\mathrm{st}}$, alors $P_\mu\supset P_0$. 
La sous-variŽtŽ fermŽe $V_{\mu,1}\subset V$ est contenue dans $\FN$ et 
pour tout $v'$ dans la $E$-lame standard $\{v\in V_{\mu,1}\,\vert\, \bs{q}_E(v) = \| \mu \|\} \subset V_{\mu,1}$ de ${_E\ES{N}}$ associŽe ˆ $\mu$, 
puisque $\bs{q}_F(v')\geq \bs{q}_E(v')= \|\mu\|$, on a $\bs{q}_F(v')=\|\mu\|$. Par consŽquent $\mu$ appartient $\bs{\Lambda}_{F,\mathrm{st}}$.
\end{proof}

\begin{remark}\label{le cas de rang relatif constant}
\textup{Si le rang $F$-dŽployŽ de $G$ co\"{\i}ncide avec son rang $E$-dŽployŽ (\textit{i.e.} si $A_1=A_0$), alors $\check{X}(A_0)_{\mbb{Q}}\cap \bs{\Lambda}_{E,\mathrm{st}}= \bs{\Lambda}_{E,\mathrm{st}}$. 
En ce cas l'inclusion $\bs{\Lambda}_{E,\mathrm{st}}\subset \bs{\Lambda}_{F,\mathrm{st}}$ fournit une injection naturelle entre les $E$-strates de ${_E\ES{N}}$ et les $F$-strates de $\FN$, qui est une bijection si $\bs{\Lambda}_{F,\mathrm{st}}=\bs{\Lambda}_{F,\mathrm{st}}^*$.
}
\end{remark}

\begin{notation}\label{notation allŽgŽe}
\textup{Pour allŽger l'Žcriture, on introduit les notations suivantes.  
\begin{enumerate}
\item[(i)]
Si $\scrY$ est une $F$-lame de $\NF$, on note\index{X(Y)@$\scrX(\scrY)$} $\scrX(\scrY)$ le sous-ensemble de $\NF$ dŽfini par $\scrX(\scrY)=\scrX_{F,v}$ pour un 
(\textit{i.e.} pour tout) $v\in \scrY_{F,v}$; et si $\bsfrY$ est une $F$-strate de $\NF$, on note\index{Xbs(Ybs)@$\bsfrX(\bsfrY)$} 
$\bsfrX(\bsfrY)$ le sous-ensemble de $\NF$ dŽfini par $\bsfrX(\bsfrY)= G(F)\cdot \scrX(\scrY)$ pour une (\textit{i.e.} pour toute) $F$-lame $\scrY$ de $\NF$ contenue dans $\bsfrY$.
\item[(ii)]Soit $E/F$ une extension sŽparable telle que $(E^{\mathrm{s\acute{e}p}})^{\mathrm{Aut}_F(E^{\mathrm{s\acute{e}p}})}=F$. Si $\scrY$ est une $F$-lame de $\NF$, on note\index{YE@$\scrY_E$} $\scrY_E$ la $E$-lame 
de $\ES{N}_E$ dŽfinie par $\scrY_E= \scrY_{E,v}$ pour un (\textit{i.e.} pour tout) $v\in \scrY$; et si $\bsfrY$ est une $F$-strate de $\NF$, on note\index{YbsE@$\bsfrY_E$} 
$\bsfrY_E$ la $E$-strate de $\ES{N}_E$ dŽfinie par $\bsfrY_E=\bsfrY_{E,v}$ pour un (\textit{i.e.} pour tout) $v\in \bsfrY$. 
On dŽfinit $\scrX(\scrY_E)$ et $\bsfrX(\bsfrY_E)$ comme plus haut (en remplaant $F$ par $E$). On a donc les ŽgalitŽs
$$\scrY = V(F) \cap \scrY_E\vgq \scrX(\scrY)= V(F)\cap \scrX(\scrY_E) \quad \hbox{et}\quad \bsfrY= V(F) \cap \bsfrY_E\ptf$$
\end{enumerate}}
\end{notation}

\subsection{Comparaison avec la stratification gŽomŽtrique}\label{comparaison avec les strates gŽo} 
Continuons avec les hypo\-thses de \ref{la variante de Hesselink} et \ref{la stratification de Hesselink}. 
\`A l'exception du lemme \ref{le cas o F-lame = lame gŽo}, tous les rŽsultats de cette sous-section supposent que l'hypothse \ref{hyp reg} ($e_V$ est rŽgulier dans $V$) est vŽrifiŽe.

\vskip2mm
\textit{\S1.\hskip1mm Le cas o $\check{X}_F(G)_{\mbb{Q}} \cap \bs{\Lambda}_v \neq \emptyset$. ---} 
Pour $v\in \FN$, on a l'inŽgalitŽ $\bs{q}(v) \leq \bs{q}_F(v)$ avec ŽgalitŽ si et seulement s'il existe un $\mu \in \bs{\Lambda}_{F,v}$ tel que $\| \mu \| = \bs{q}(v)$, \textit{i.e.} si et seulement si 
l'intersection $\check{X}_F(G)_{\mbb{Q}}\cap \bs{\Lambda}_v$ est non vide. 
On en dŽduit un critre trs simple assurant que 
la $F$-lame ${_F\scrY_v}$ est ouverte dans son $F$-saturŽ ${_F\scrX_v}$:

\begin{lemme}\label{le cas o F-lame = lame gŽo}
Soit $v\in \FN$. Si $\check{X}_F(G)_{\mbb{Q}}\cap \bs{\Lambda}_v\neq \emptyset$ alors 
$$\bs{\Lambda}_{F,v}= \check{X}_F(G)_{\mbb{Q}}\cap \bs{\Lambda}_v\vgq {_F\scrX_v}= \scrX_v \quad \hbox{et}\quad {_F\scrY_v}=\scrY_v\ptf$$ 
Dans ce cas, la $F$-lame ${_F\scrY_v}$ est (Zariski-)ouverte dans son $F$-saturŽ ${_F\scrX_v}$.
\end{lemme}

\begin{proof}
Soit $\mu\in \check{X}_F(G)_{\mbb{Q}}\cap \bs{\Lambda}_v$. Puisque $\bs{\Lambda}_{F,v}$, resp. $\bs{\Lambda}_v$, est un espace principal homogne 
sous $U_\mu(F)$, resp. $U_\mu$, pour $\mu' = u \bullet \mu \in \bs{\Lambda}_v$ avec $u\in U_\mu$, on a 
$$\mu' \in \check{X}_F(G)_{\mbb{Q}} \Leftrightarrow u \in U_\mu(F)\Leftrightarrow  \mu'\in \bs{\Lambda}_{F,v}\,;$$ d'o l'ŽgalitŽ $\bs{\Lambda}_{F,v}= \check{X}_F(G)\cap \bs{\Lambda}_v$. 
D'autre part on a (par dŽfinition) 
$${_F\scrY_v}= \{v'\in \FN\,\vert \, \mu \in \bs{\Lambda}_{F,v'}\}\quad \hbox{et}\quad \scrY_v= \{v'\in \ESN\,\vert \, \mu \in \bs{\Lambda}_{v'}\}\ptf$$ 
Par consŽquent ${_F\scrY_v} = \FN \cap \scrY_v$. Comme d'autre part $$\scrY_v \subset \scrX_v = V_{\mu,1}={_F\scrX_v} \subset \FN\vg$$ on obtient que la $F$-lame ${_F\scrY_v}$ co\"{\i}ncide avec la $\overline{F}$-lame 
$\scrY_v$, laquelle est ouverte dans ${_F\scrX_v}= \scrX_v$. 
\end{proof}

\begin{remark}\label{densitŽ bonnes F-strates}
\textup{Soit $v\in \FN$ tel que $\check{X}_F(G)_{\mbb{Q}}\cap \bs{\Lambda}_v\neq \emptyset$. 
\begin{enumerate}
\item[(i)]Supposons $F$ infini. D'aprs \cite[ch.~V, 18.3]{B}, le groupe $G(F)$ est (Zariski-)dense dans $G$. Puisque ${_F\scrY_v}=\scrY_v$ est ouvert dense dans ${_F\scrX_v}= \scrX_v$, 
on en dŽduit que la $F$-strate ${_F\bsfrY_v}=G(F)\cdot {_F\scrY_v}$ de $\FN$ est dense dans 
$\bsfrX_v= G\cdot \scrX_v$.
\item[(ii)]Supposons $F$ infini et l'hypothse \ref{hyp reg} vŽrifiŽe. D'aprs \ref{le cas o F-lame = lame gŽo} et \ref{il existe un point F-rationnel}, l'ensemble 
$\scrY_{F,v}=V(F)\cap {_F\scrY_v}$ est non vide et dense dans ${_F\scrX_v}$ (\ref{il existe un point F-rationnel}). 
Comme $G(F)$ est dense dans $G$ (\textit{cf.} le point (i)), on en dŽduit que la $F$-strate 
$\bsfrY_{F,v}= G(F)\cdot \scrY_{F,v}$ de $\NF$ est dense dans $\bsfrX_v = G\cdot \scrX_v$.
\end{enumerate}}
\end{remark}

\textit{\S2.\hskip1mm La $\overline{F}$-strate $\bsfrY(F,v)$ associŽe ˆ $v\in \FN$. ---} 
Jusqu'ˆ la fin de \ref{comparaison avec les strates gŽo}, on suppose de plus de plus que l'hypothse \ref{hyp reg} est vŽrifiŽe: {$e_V$ est rŽgulier dans $V$.  
Cela assure que les sous-ensembles $F$-saturŽs de $\FN$ --- et en particulier 
les variŽtŽs ${_F\scrX_v}$ avec $v\in \FN$ --- sont des $F$-espaces affines.

Pour $v\in \FN$, puisque les $\overline{F}$-strates de $\ES{N}={_{\smash{\overline{F}}}\ES{N}}$ sont localement fermŽes dans $V$, il en existe une 
et une seule $ \bsfrY_{v_1}$ (avec $v_1\in \ES{N}$ que l'on peut choisir dans ${_F\scrX_v}$) telle que l'intersection $\bsfrY_{v_1} \cap {_F\scrX_v}$ soit dense dans  la variŽtŽ irrŽductible 
${_F\scrX_v}$; on la note\index{YbsgeoFv@$\bsfrY(F,v)$} $$\bsfrY(F,v)=\bsfrY_{v_1}\ptf$$ 
Observons que l'intersection $\bsfrY(F,v) \cap {_F\scrX_v}$ est ouverte dans ${_F\scrX_v}$. 
On pose aussi\index{qgeoFv@$\bs{q}(F,v)$}\index{XbsgeoFv@$\bsfrX(F,v)$} 
$$\bs{q}(F,v)=\bs{q}(v_1)\quad \hbox{et}\quad \bsfrX(F,v) = \bsfrX_{v_1}=G\cdot \scrX_{v_1}\ptf$$ On a donc 
$\bs{q}(F,v)=\bs{q}(v')$ et $\bsfrX(F,v)= \bsfrX_{v'}$ pour tout $v'\in \bsfrY(F,v)$. 
Puisque $\bsfrY(F,v)$ est ouvert dense dans la sous-variŽtŽ fermŽe irrŽductible $\bsfrX(F,v)$ de $V$, on a l'inclusion ${_F\scrX_v}\subset \bsfrX(F,v)$. 
On en dŽduit que $$\bsfrY(F,v)\cap {_F\scrX_v}= \{v'\in {_F\scrX_v}\,\vert\, \bs{q}(v')= \bs{q}(F,v)\}\ptf$$ 
Puisque $v\in {_F\scrX_v}\subset \bsfrX(F,v)$, on a toujours $$\bs{q}(v)\leq \bs{q}(F,v)$$ 
avec ŽgalitŽ si et seulement si $v\in \bsfrY(F,v)$.

\begin{lemme}\label{strate gŽo associŽe ˆ F-strate}
Soit $v\in \FN$. 
\begin{enumerate}
\item[(i)] Si $\check{X}_F(G)_{\mbb{Q}}\cap \bs{\Lambda}_v\neq \emptyset$, alors $\bsfrY(F,v)=\bsfrY_v$ et $\bsfrY(F,v) \cap {_F\scrX_v}= {_F\scrY_v}$. 
\item[(ii)] Si $\check{X}_F(G)_{\mbb{Q}}\cap \bs{\Lambda}_v=\emptyset$, alors $\bs{q}(F,v)< \bs{q}_F(v)$ et on distingue deux cas: ou bien $\bsfrY(F,v)\cap {_F\scrX_v}\subset {_F\scrX_v}\smallsetminus {_F\scrY_v}$; ou bien $\bsfrY(F,v)\cap {_F\scrY_v}\neq \emptyset$ auquel cas 
pour tout $v'\in\bsfrY(F,v)\cap {_F\scrY_v}$, on a 
$\bsfrY(F,v')= \bsfrY(F,v)$ et $\check{X}_F(G)_{\mbb{Q}}\cap \bs{\Lambda}_{v'}= \emptyset$. 
\end{enumerate}
\end{lemme}

\begin{proof}
\'Ecrivons $\bsfrY(F,v)=\bsfrY_{v_1}$ avec $v_1\in {_F\scrX_v}$. 

Supposons que l'intersection $\check{X}_F(G)_{\mbb{Q}}\cap \bs{\Lambda}_v$ soit non vide. Puisque ${_F\scrY_v}=\scrY_v$ est ouvert (dense) dans ${_F\scrX_v}$ (\ref{le cas o F-lame = lame gŽo}), 
on peut prendre $v_1=v$. On a donc $\bsfrY(F,v)= \bsfrY_v$. L'inclusion 
$\scrY_v={_F\scrY_v}\subset  \bsfrY_v \cap {_F\scrX_v}$ est claire. D'autre part si $v'\in \bsfrY_{v} \cap {_F\scrX_v}$, on a $$\bs{q}(v')\leq \bs{q}_F(v')\leq \bs{q}_F(v) =\bs{q}(v)\ptf$$ Or $\bs{q}(v')=\bs{q}(v)$ par consŽquent 
toutes les inŽgalitŽs ci-dessus sont des ŽgalitŽs et $v'\in {_F\scrY_v}$. D'o l'inclusion $\bsfrY_{v}\cap {_F\scrX_v}\subset {_F\scrY_v}$. Cela prouve (i).

Supposons maintenant que $\check{X}_F(G)_{\mbb{Q}}\cap \bs{\Lambda}_v=\emptyset$, \textit{i.e.} $\bs{q}(v)<\bs{q}_F(v)$. On a donc \textit{a priori} les inŽgalitŽs 
$$\bs{q}(v) \leq \bs{q}(v_1) \leq \bs{q}_F(v_1) \leq \bs{q}_F(v)\ptf$$
Si l'intersection $\bsfrY_{v_1}\cap {_F\scrY_v}$ est vide, alors $v_1\in {_F\scrX_v}\smallsetminus {_F\scrY_v}$ par suite $\bs{q}_F(v_1)<\bs{q}_F(v)$ et donc $\bs{q}(v_1)<\bs{q}_F(v)$. Supposons 
qu'il existe un $v'\in \bsfrY_{v_1}\cap {_F\scrY_v}$. On a ${_F\scrY_{v'}}={_F\scrY_v}$ et ${_F\scrX_{v'}}={_F\scrX_v}$, par consŽquent 
$$\bsfrY(F,v')= \bsfrY_{v_1}\ptf$$  
Si $\bs{q}(v')=\bs{q}_F(v')$, alors $\bsfrY_{v'}= \bsfrY(F,v')$ et d'aprs (1) on a l'ŽgalitŽ $\bsfrY_{v'}\cap {_F\scrX_v}= {_F\scrY_v}$; en particulier 
$v$ appartient ˆ $\bsfrY_{v'}$, par consŽquent $\bs{q}(v)=\bs{q}(v')=\bs{q}_F(v')=\bs{q}_F(v)$ ce qui contredit l'hypothse $\bs{q}(v)<\bs{q}_F(v)$. On a donc 
$$\bs{q}(v_1)=\bs{q}(v')< \bs{q}_F(v')=\bs{q}_F(v)\ptf$$ Cela prouve (ii). 
\end{proof}

L'application qui ˆ $v\in \FN$ associe la $\overline{F}$-strate $\bsfrY(F,v)$ de $\ES{N}$ est constante sur les $F$-strates de $\FN$: on a 
$$\bsfrY(F,v')= \bsfrY(F,v)\quad \hbox{pour tout} \quad v'\in {_F\bsfrY_v}\ptf$$ 
L'application ${_F\bsfrY_v} \mapsto \bsfrY(F,v)$ n'est en gŽnŽral ni injective ni surjective. 

\begin{lemme}\label{bonnes F-strates}
Soit $v\in \FN$ tel que $\check{X}_F(G)_{\mbb{Q}}\cap G\bullet \bs{\Lambda}_v\neq \emptyset$. 
\begin{enumerate}
\item[(i)] Il existe une unique $F$-strate ${_F\bsfrY_{v_1}}$ de $\FN$ (avec $v_1\in \FN$) 
telle que $$\check{X}_F(G)_{\mbb{Q}}\cap \bs{\Lambda}_{v_1}\neq \emptyset\quad\hbox{et}\quad\bsfrY(F,v_1)=\bsfrY_{v_1}=\bsfrY_v\ptf$$
\item[(ii)] On a ${_F\bsfrY_{v_1}}= \{v'\in \FN \cap \bsfrY_v \,\vert \, \check{X}_F(G)_{\mbb{Q}} \cap \bs{\Lambda}_{v'} \neq \emptyset \}$.
\end{enumerate}
\end{lemme}

\begin{proof}
Soit $g\in G$ tel que l'intersection $\check{X}_F(G)_{\mbb{Q}}\cap \bs{\Lambda}_{g\cdot v}$ 
soit non vide, et soit $\mu\in \check{X}_F(G)_{\mbb{Q}} \cap \bs{\Lambda}_{g\cdot v}$. Quitte ˆ remplacer $g$ par un ŽlŽment de $G(F)g$, on peut supposer que $P_\mu\supset P_0$. On peut aussi 
supposer que ${\rm Im}(\mu)\subset A_0$. Alors $$\check{X}(A_0)_{\mbb{Q}}\cap \bs{\Lambda}_{g\cdot v}= \{\mu\}= \bs{\Lambda}_{\mathrm{st}} \cap G\bullet \bs{\Lambda}_v\ptf$$ 
L'ŽlŽment $v_1=g\cdot v$ 
appartient ˆ $\FN$ et $ \bsfrY(F,v_1)=\bsfrY_{v_1}=\bsfrY_v$. Le co-caractre virtuel 
$\mu$ appartient ˆ $\bs{\Lambda}_{F,\mathrm{st}}$ et ${_F\bsfrY_{v_1}}$ est l'unique $F$-strate de $\FN$ associŽe ˆ $\mu$. Cela prouve (i). 

Quant au point (ii), on a l'inclusion ${_F\bsfrY_{v_1}}\subset \{v'\in \FN \cap \bsfrY_v \,\vert \, \check{X}_F(G)_{\mbb{Q}} \cap \bs{\Lambda}_{v'} \neq \emptyset \}$. Inversement 
si $v'\in \FN \cap \bsfrY_v $ avec $\check{X}_F(G)_{\mbb{Q}} \cap \bs{\Lambda}_{v'}\neq \emptyset$, alors d'aprs \ref{strate gŽo associŽe ˆ F-strate}\,(i), on a $\bsfrY(F,v')= \bsfrY_{v'}=\bsfrY_v$. Par consŽquent (d'aprs le point (i)), $v'$ appartient ˆ ${_F\bsfrY_{v_1}}$. Cela prouve (ii). 
\end{proof}

Si le groupe $G$ est $F$-dŽployŽ, puisque $\bs{\Lambda}_{\mathrm{st}} \subset \check{X}(A_0)\subset \check{X}_F(G)$, on a 
$$\check{X}_F(G)_{\mbb{Q}} \cap G\bullet \bs{\Lambda}_v \neq \emptyset \quad \hbox{pour tout} \quad v\in \ES{N}\ptf$$

\begin{proposition}\label{mauvaises F-strates}
Supposons $F$ infini et $G$ dŽployŽ sur $F$ (\textit{e.g.} si $F=F^{\mathrm{s\acute{e}p}}$). Pour $v\in \FN$ tel que $\check{X}_F(G)_{\mbb{Q}} \cap \bs{\Lambda}_v = \emptyset$, le sous-ensemble 
$$\{v'\in {_F\scrX_v}\,\vert \, \hbox{$\check{X}_F(G)_{\mbb{Q}}\cap \bs{\Lambda}_{v'}\neq \emptyset$ et $\bsfrY_{v'}=\bsfrY(F,v)$}\} \subset  {_F\scrX_v} \cap \bsfrY(F,v)$$
est dense dans ${_F\scrX_v}$ (en particulier il est non vide) et contenu dans ${_F\scrX_v}\smallsetminus {_F\scrY_v}$. 
\end{proposition}

\begin{proof}
Soit $v\in \FN$ tel que $\check{X}_F(G)_{\mbb{Q}}\cap \bs{\Lambda}_v = \emptyset$. Puisque $G$ est $F$-dŽployŽ, on a $\check{X}_F(G)_{\mbb{Q}}\cap G\bullet \bs{\Lambda}_v \neq  \emptyset$; et d'aprs \ref{bonnes F-strates}, il existe une unique $F$-strate 
${_F\bsfrY_{v_1}}$ de $\FN$ telle que $\check{X}_F(G)_{\mbb{Q}} \cap \bs{\Lambda}_{v_1}\neq \emptyset$ et $\bsfrY(F,v_1)=\bsfrY(F,v)$. 
On a $\bsfrY(F,v_1)= \bsfrY_{v_1}$ et $\bs{q}(F,v)= \bs{q}(v_1)=\bs{q}_F(v_1)$. Puisque $F$ est infini, 
la $F$-strate ${_F\bsfrY_{v_1}}= G(F)\cdot {_F\scrY_{v_1}}$ 
est dense dans la $\overline{F}$-strate $\bsfrY_{v_1}=G\cdot \scrY_{v_1}$ (\ref{densitŽ bonnes F-strates}\,(i)). Comme l'intersection $\bsfrY_{v_1}\cap {_F\scrX_v}$ est un ouvert dense de ${_F\scrX_v}$, on obtient 
que l'intersection ${_F\bsfrY_{v_1}}\cap {_F\scrX_v}$ est dense dans ${_F\scrX_v}$. Cette intersection co\"{\i}ncide avec le sous-ensemble de l'ŽnoncŽ (d'aprs la propriŽtŽ d'unicitŽ 
dans \ref{bonnes F-strates}). Observons que puisque ${_F\bsfrY_{v_1}}\cap {_F\scrX_v}\neq \emptyset$, on peut prendre $v_1$ dans ${_F\scrX_v}$. D'aprs \ref{strate gŽo associŽe ˆ F-strate}\,(ii), on a 
$\bs{q}(F,v)\;(=\bs{q}_F(v_1)) < \bs{q}_F(v)$. D'autre part pour $v'\in {_F\bsfrY_{v_1}}\cap {_F\scrX_v}$, on a $\bs{q}_F(v')= \bs{q}_F(v_1)$; donc $\bs{q}_F(v')< \bs{q}_F(v)$ et $v'\in {_F\scrX_v} \smallsetminus {_F\scrY_v}$. 
\end{proof}

\begin{corollary}\label{deux critres bonnes F-lames}
(Sous les hypothses de \ref{mauvaises F-strates}.) Pour $v\in \FN$, les conditions suivantes sont Žquivalentes:
\begin{enumerate}
\item[(i)]$\check{X}_F(G)_{\mbb{Q}} \cap \bs{\Lambda}_v\neq \emptyset$;
\item[(ii)]la $F$-lame ${_F\scrY_v}$ est ouverte dans ${_F\scrX_v}$; 
\item[(iii)]la $F$-lame ${_F\scrY_v}$ est ${_FP_v}$-invariante\footnote{Rappelons que pour tout $v\in \FN$, on sait seulement que la $F$-lame ${_F\scrY_v}$ est $P_{F,v}$-invariante.}.
\end{enumerate}
\end{corollary}

\begin{proof}
Si $\check{X}_F(G)_{\mbb{Q}} \cap \bs{\Lambda}_v\neq \emptyset$ alors ${_F\scrY_v}=\scrY_v$; par consŽquent $(i)\Rightarrow (ii)$ et $(i)\Rightarrow (iii)$.
Si la $F$-lame ${_F\scrY_v}$ est ouverte dans ${_F\scrX_v}$, le bord ${_F\scrX_v}\smallsetminus {_F\scrY_v}$ est fermŽ dans ${_F\scrX_v}$ et 
il ne peut pas tre dense dans ${_F\scrX_v}$; par consŽquent $(ii)\Rightarrow (i)$ (d'aprs \ref{mauvaises F-strates}). 

Prouvons $(iii)\Rightarrow (ii)$. Choisissons un tore maximal $T'$ de $G$ (pas forcŽment dŽfini sur $F$) qui soit contenu dans ${_FP_v}\cap P_v$. 
Il existe un $p\in {_FP_v}$ tel que  $T=pT'p^{-1}$ soit un tore $F$-dŽployŽ maximal de ${_FP_v}$. Soit $v'=p'\cdot v$. On a $$P_{v'}= p'P_vp'^{-1}\supset p'T'p'^{-1}=T\quad \hbox{et}\quad \check{X}(T)_{\mbb{Q}}\cap \bs{\Lambda}_{v'}=\{\mu'\}\ptf$$ 
En particulier $v'$ appartient ˆ $\FN$ et $\check{X}_F(G)_{\mbb{Q}}\cap \bs{\Lambda}_{v'}\neq \emptyset$ ce qui assure que la $F$-lame ${_F\scrY_{v'}}$ est ouverte dans son 
$F$-saturŽ ${_F\scrX_{v'}}$. Si la $F$-lame ${_F\scrY_v}$ est ${_FP_v}$-invariante, alors $v'$ appartient ˆ ${_F\scrY_v}$, ${_F\scrY_{v'}}={_F\scrY_v}$ et ${_F\scrX_{v'}}={_F\scrX_v}$. 
Par consŽquent $(iii)\Rightarrow (ii)$. 
\end{proof}

\begin{corollary}\label{deux critres bonnes F-lames bis}
Pour $v\in \NF\;(=V(F)\cap \FN)$, la proposition \ref{mauvaises F-strates} et le corollaire \ref{deux critres bonnes F-lames} restent vrais sans hypothse sur $F$ ou $G$ (hormis l'hypothse \ref{hyp reg} que l'on suppose toujours vŽrifiŽe).
\end{corollary}

\begin{proof}
Pour $v\in \NF$,  on a (\ref{descente sŽparable lames})
$$\bs{\Lambda}_{F,v} = \check{X}_F(G)_{\mbb{Q}}\cap \bs{\Lambda}_{F^{\mathrm{s\acute{e}p}},v}\vgq {_F\scrX_v}= {_{F^{\mathrm{s\acute{e}p}}}\scrX_v}\vgq {_F\scrY_v}= {_{F^{\mathrm{s\acute{e}p}}}\scrY_v}\ptf$$ 
D'o le corollaire puisque $F^{\mathrm{s\acute{e}p}}$ est infini et que $G$ est $F^{\mathrm{s\acute{e}p}}$-dŽployŽ. 
\end{proof}

\begin{exemple}\label{l'exemple de [5.6]{H1}}
\textup{Reprenons l'exemple de \cite[5.6]{H1}. Supposons $p>1$.
\begin{enumerate}
\item[(i)] ConsidŽ\-rons l'action du groupe $G=\mathrm{SL}_{2}$ sur 
$V=\mbb{A}_F^2$ dŽfinie par le morphisme $G\rightarrow G,\, g \mapsto \rho(g)$ donnŽ par $\rho(g)_{i,j}= g_{i,j}^p$ (pour $g=(g_{i,j})_{1\leq i,j\leq 2}$); \cad que  
pour $v=(x,y)\in V$, on note $g\cdot v = (x'\!,y')\in V$ l'ŽlŽment dŽfini par $\rho(v)\left(\begin{array}{c}x\\y\end{array}\right)=\left(\begin{array}{c}x'\\y'\end{array}\right)$. 
Supposons $xy\neq 0$ et $xy^{-1}\notin F^p$. L'ŽlŽment $v$ n'est pas $F$-instable: il appartient ˆ $V(F)\smallsetminus \NF$. En revanche il appartient ˆ 
$\ES{N}_{F^{\mathrm{rad}}}$: pour $\alpha\in F^{\mathrm{rad}}$ tel que $\alpha^p = xy^{-1}$, le co-caractre $\lambda_\alpha\in \check{X}_{F[\alpha]}(G)$ dŽfini par 
$t^{\lambda_\alpha} = \left(\begin{array}{cc} t^{-1} & \alpha(t-t^{-1}) \\ 0 & t\end{array}\right)$ vŽrifie $t^{\lambda_\alpha} \cdot v = t^pv$.
\item[(ii)] ConsidŽrons l'action du groupe $G=\mathrm{SL}_3$ sur $V=\mbb{A}_{F}^3$ dŽfinie comme en (i) par le morphisme 
$G \rightarrow G,\, (g_{i,j})\mapsto (g_{i,j}^p)$. L'ŽlŽment $v=(x,y,0)$ est dans $\NF$. Supposons $xy\neq 0$ et $xy^{-1}\notin F^p$. 
Le co-caractre $t\mapsto \mathrm{diag}(t,t,t^{-2})$ appartient ˆ $\Lambda_{F,v}^{\mathrm{opt}}$; d'autre part le co-caractre $t \mapsto \mathrm{diag}(t^{\lambda_\alpha},1)$ appartient 
ˆ $\Lambda_v^{\mathrm{opt}}$, o l'on a identifiŽ diagonalement $\mathrm{SL}_2$ ˆ $\mathrm{SL}_2\times \{1\}\subset \mathrm{SL}_3$. Dans ce cas on a $\check{X}_F(G)_{\mbb{Q}} \cap \bs{\Lambda}_v = \emptyset$. 
La $F$-lame ${_F\scrY_v}$ de $\FN$ est l'ensemble des $v'=(x'\!,y'\!,z')$ tels que $z'=0$, $x'y'\neq 0$ et $x'y'^{-1}\notin F^p$. Son $F$-saturŽ ${_F\scrX_v}$ est le sous-espace $\overline{F}\times \overline{F} \times\{0\}$ de $V$. 
La $\overline{F}$-lame $\scrY_v$ de $\ES{N}$ est l'ensemble des $v'=(x'\!,y'\!,z')$ tels que $z'=0$, $x'y'\neq 0$ et $yx'=xy'$. Son $\overline{F}$-saturŽ $\scrX_v$ est le sous-espace 
$\{(x',yx^{-1}x',0)\,\vert \, x'\in \overline{F}\}$ de $V$. D'autre part pour tout $v'\in {_F\scrX_v}\smallsetminus {_F\scrY_v}$ tel que $x'y'\neq 0$, on a $\check{X}_F(G)_{\mbb{Q}} \cap \bs{\Lambda}_{v'}\neq \emptyset$ 
et $\bsfrY(F,v)= \bsfrY_{v'}$.
\end{enumerate}}
\end{exemple}

\textit{\S3.\hskip1mm Une hypothse simplificatrice. ---} On a vu en 
\ref{la stratification de Hesselink} que les $F$-strates de $\FN$ qui possdent un point $F$-rationnel --- ou, ce qui revient au mme, les $F$-strates de $\NF$ --- 
se comportent bien par extension sŽparable (algŽbrique ou non) du corps de base. On ramne donc facilement la thŽorie 
au cas o $F=F^{\mathrm{s\acute{e}p}}$, ce qui entra"ne que $G$ est $F$-dŽployŽ. Le passage de $F=F^{\mathrm{s\acute{e}p}}$ ˆ $\overline{F}=F^{\mathrm{rad}}$ 
est, comme on vient de le voir, nettement plus compliquŽ. 

Les difficultŽs proviennent des ŽlŽments $v\in\NF$ tels que 
$\check{X}_F(G)_{\mbb{Q}}\cap \bs{\Lambda}_v =\emptyset$. Pour $F$ quelconque, ces difficultŽs disparaissent si l'on fait l'hypothse 
suivante\footnote{On verra dans la section \ref{le cas de la variŽtŽ U} que dans le cas o $F$ est un corps local ou un corps global et $V$ 
est le groupe $G$ lui-mme muni de l'action par conjugaison (avec $e_V=1$), l'hypothse \ref{hypo bonnes F-strates faible} est toujours vŽrifiŽe.}:

\begin{hypothese}\label{hypo bonnes F-strates faible}
Pour tout $v\in \NF$, on a $\check{X}_{F}(G)_{\mbb{Q}}\cap \bs{\Lambda}_v \neq \emptyset$. 
\end{hypothese}

ConsidŽrons aussi l'hypothse (plus forte) suivante:

\begin{hypothese}\label{hypo bonnes F-strates}
Pour tout $v\in \FN$, on a $\check{X}_{F}(G)_{\mbb{Q}}\cap \bs{\Lambda}_v \neq \emptyset$. 
\end{hypothese} 

Soit $\wt{F}\subset F^\mathrm{s\acute{e}p}$ une extension galoisienne finie de $F$ dŽployant $G$. Fixons une $\wt{F}$-paire parabolique minimale 
$(\wt{P}_0,\wt{A}_0)$ de $G$ telle que $A_0\subset \wt{A}_0$ et $\wt{P}_0\subset P_0$. On peut supposer que le tore $\wt{F}$-dŽployŽ maximal 
$\wt{A}_0$ de $G$ est dŽfini sur $F$. La paire $(\wt{P}_0,\wt{A}_0)$ dŽfinit un sous-ensemble 
$$\bs{\Lambda}_{\smash{\wt{F},\mathrm{st}}}=\bs{\Lambda}_{\wt{F},\mathrm{st}}^G(V,e_V)\subset \check{X}(\wt{A}_0)_{\mbb{Q}}$$ de co-caractres virtuels 
$\wt{F}$-optimaux qui sont en position standard (relativement ˆ $(\wt{P}_0,\wt{A}_0)$). Elle dŽfinit aussi un sous-ensemble 
$$\bs{\Lambda}_{\mathrm{st}}=\bs{\Lambda}_{\mathrm{st}}^G(V,e_V)\subset \check{X}(\wt{A}_0)_{\mbb{Q}}$$ de co-caractres virtuels 
$\overline{F}$-optimaux qui sont en position standard (relativement ˆ $(\wt{P}_0,\wt{A}_0)$). D'aprs \ref{bonnes F-strates}, on a l'inclusion 
$$\bs{\Lambda}_{\mathrm{st}}\subset \bs{\Lambda}_{\smash{\wt{F},\mathrm{st}}}$$ avec ŽgalitŽ si l'hypothse 
\ref{hypo bonnes F-strates} est vŽrifiŽe \textit{pour $\wt{F}$}, \cad si $\check{X}_{\smash{\wt{F}}}(G)_{\mbb{Q}}\cap \bs{\Lambda}_v\neq \emptyset$ pour tout $v\in {_{\smash{\wt{F}}}\ES{N}}$. 

\begin{lemme}\label{descente 1 sous HBFS (forte)} 
Si l'hypothse \ref{hypo bonnes F-strates} est vŽrifiŽe, on a l'ŽgalitŽ 
$$\check{X}(A_0)_{\mbb{Q}} \cap \bs{\Lambda}_{\mathrm{st}} = \check{X}(A_0)_{\mbb{Q}}\cap \bs{\Lambda}_{\smash{\wt{F},\mathrm{st}}}\ptf$$
\end{lemme} 

\begin{proof}
On a clairement l'inclusion $\check{X}(A_0)_{\mbb{Q}} \cap \bs{\Lambda}_{\mathrm{st}} \subset \check{X}(A_0)_{\mbb{Q}}\cap \bs{\Lambda}_{\smash{\wt{F},\mathrm{st}}}$. 
Soit $\mu\in \check{X}(A_0)_{\mbb{Q}}\cap \bs{\Lambda}_{\smash{\wt{F},\mathrm{st}}}$ et soit $v\in V_{\mu,1}$ tel que $\mu \in \bs{\Lambda}_{\smash{\wt{F},v}}$. On a donc 
$\check{X}(\wt{A}_0)_{\mbb{Q}}\cap \bs{\Lambda}_{\smash{\wt{F},v}}=\{\mu\}$ et $V_{\mu,1}= {_{\smash{\wt{F}}}\scrX_v}$. Puisque par hypothse le co-caractre virtuel $\mu$ est dŽfini sur $F$, 
la variŽtŽ $V_{\mu,1}$ est contenue dans $\FN$ et $\mu$ appartient ˆ $\bs{\Lambda}_{F,v}$ (\textit{i.e.} $\check{X}(A_0)_{\mbb{Q}}\cap \bs{\Lambda}_{F,v}=\{\mu\}$). 
Si l'hypothse \ref{hypo bonnes F-strates} est vŽrifiŽe, alors $\bs{q}(v)= \bs{q}_F(v)$ et $\mu$ appartient ˆ $\bs{\Lambda}_v$. 
\end{proof}

\begin{lemme}\label{descente 2 sous HBFS (forte)} 
\begin{enumerate}
\item[(i)] On a l'inclusion $\check{X}(A_0)_{\mbb{Q}}\cap \bs{\Lambda}_{\mathrm{st}} \subset \bs{\Lambda}_{F,\mathrm{st}}$ avec ŽgalitŽ si l'hypothse 
\ref{hypo bonnes F-strates} est vŽrifiŽe et si toute $F$-strate de $\FN$ possde un point $F$-rationnel.
\item[(ii)] Supposons que $F$ soit infini et que l'hypothse \ref{hypo bonnes F-strates} soit  
vŽrifiŽe. Alors toute $F$-strate de $\FN$ possde un point $F$-rationnel (en particulier l'inclusion de (i) est une ŽgalitŽ) et 
l'application $Y \mapsto \FN \cap Y$ induit une bijection entre les $\overline{F}$-strates de $\ES{N}$ qui intersectent non trivialement $\FN$ et les $F$-strates de $\FN$. 
\end{enumerate}
\end{lemme}

\begin{proof}
Puisque $\bs{\Lambda}_{\mathrm{st}}\subset \bs{\Lambda}_{\smash{\wt{F}},\mathrm{st}}$, on a $\check{X}(A_0)_{\mbb{Q}}\cap \bs{\Lambda}_{\mathrm{st}} \subset \check{X}(A_0)_{\mbb{Q}}\cap \bs{\Lambda}_{\smash{\wt{F},\mathrm{st}}}$. 
D'autre part puisque 
l'extension $\wt{F}/F$ est sŽparable, on a $\check{X}(A_0)_{\mbb{Q}}\cap \bs{\Lambda}_{\smash{\wt{F},\mathrm{st}}}\subset \bs{\Lambda}_{F,\mathrm{st}}$ avec ŽgalitŽ si toute $F$-strate de 
$\FN$ possde un point $F$-rationnel (\ref{descente des co-caractres st}). D'o l'inclusion 
$\check{X}(A_0)_{\mbb{Q}}\cap \bs{\Lambda}_{\mathrm{st}} \subset \bs{\Lambda}_{F,\mathrm{st}}$. D'autre part si l'hypothse \ref{hypo bonnes F-strates} est vŽrifiŽe, d'aprs \ref{descente 1 sous HBFS (forte)}, on a l'ŽgalitŽ 
$\check{X}(A_0)_{\mbb{Q}} \cap \bs{\Lambda}_{\mathrm{st}} = \check{X}(A_0)_{\mbb{Q}}\cap \bs{\Lambda}_{\smash{\wt{F},\mathrm{st}}}$. Cela prouve (i).

Supposons que $F$ soit infini et que l'hypothse \ref{hypo bonnes F-strates} soit vŽrifiŽe. 
Jointes ˆ \ref{hyp reg}, ces hypothses assurent que toute $F$-strate de $\FN$ possde un point $F$-rationnel (\textit{cf.}  \ref{il existe un point F-rationnel}). 
Quant ˆ la dernire assertion du point (ii), soit $v'\in \FN \cap \bsfrY_v $ pour un ŽlŽment $v\in \FN$. Soient 
$\mu$ et $\mu'$ les ŽlŽments de $\bs{\Lambda}_{F,\mathrm{st}}$ associŽs respectivement ˆ $v$ et $v'$. Quitte ˆ remplacer $v$ par $g\cdot v$ et $v'$ par $g'\cdot v'$ 
pour des ŽlŽments $g,\,g'\in G(F)$, on peut supposer que ${_FP_v}\supset P_0$ et ${_FP_{v'}}\supset P_0$. 
Posons $\check{X}(A_0)_{\mbb{Q}}\cap \bs{\Lambda}_{F,v}= \{\mu\}$ et $\check{X}(A_0)_{\mbb{Q}}\cap \bs{\Lambda}_{F,v'}= \{\mu'\}$. Alors $\mu\in \bs{\Lambda}_v$ et $\mu'\in \bs{\Lambda}_{v'}$. 
Puisque $\mu$ et $\mu'$ appartiennent ˆ $\bs{\Lambda}_{\mathrm{st}}$ et que $v$ et $v'$ sont dans la mme $\overline{F}$-strate de $\ES{N}$, 
on a forcŽment $\mu=\mu'$. Par consŽquent $v$ et $v'$ sont dans la mme $F$-strate de $\FN$. 
\end{proof}

\begin{lemme}\label{descente 3 sous HBFS}
On suppose que l'hypothse \ref{hypo bonnes F-strates faible} est vŽrifiŽe. 
\begin{enumerate}
\item[(i)] On a l'inclusion $\bs{\Lambda}_{F,\mathrm{st}}^*\subset \check{X}(A_0)_{\mbb{Q}}\cap \bs{\Lambda}_{\mathrm{st}}$.
\item[(ii)] Toute $\overline{F}$-lame de $\ES{N}$ qui intersecte non trivialement $\NF$ est une $F$-lame de $\FN$.
\item[(iii)] L'application $Y\mapsto \NF \cap Y$ induit une bijection entre les $\overline{F}$-strates de $\ES{N}$ qui intersectent non trivialement $\NF$ et les $F$-strates de $\NF$.
\end{enumerate}
\end{lemme}

\begin{proof}
Prouvons (i). Pour $\mu\in \bs{\Lambda}_{F,\mathrm{st}}^*$, il existe un $v\in V_{\mu,1}(F)$ tel que $\mu \in \bs{\Lambda}_{F,v}$. Puisque 
$v\in \NF$, si l'hypothse \ref{hypo bonnes F-strates faible} est vŽrifiŽe, on a $\bs{\Lambda}_{F,v}= \check{X}_F(G)_{\mbb{Q}}\cap \bs{\Lambda}_v$. 
Par consŽquent $\mu$ appartient ˆ $\check{X}(A_0)_{\mbb{Q}}\cap \bs{\Lambda}_{\mathrm{st}}$. 

Prouvons (ii). Pour $v \in \ES{N}$ et $v'\in \NF \cap \scrY_v$, on a ${_F\scrY_{v'}}= \scrY_{v'}= \scrY_v$. 

On prouve (iii) comme on l'a fait pour la dernire assertion de 
\ref{descente 2 sous HBFS (forte)}\,(ii).
\end{proof}

\begin{notation}\label{bracket et bracketF} \textup{On note:\index{Nbracket@$[\ES{N}]$, $[\ES{N}]_F$, $[\ES{N}_F]$}
\begin{itemize}
\item $\llbracket \ES{N} \rrbracket$, resp. $[\ES{N}]$, l'ensemble des ($\overline{F}$-)lames, resp. ($\overline{F}$-)strates, de $\ES{N}$; 
\item $\llbracket \ES{N} \rrbracket_F \subset \llbracket \ES{N} \rrbracket$, resp. $[\ES{N}]_F \subset [\ES{N}]$, le sous-ensemble formŽ des ŽlŽments 
qui intersectent non trivialement $\ES{N}_F$; 
\item $\llbracket \ES{N}_F \rrbracket$, resp. $[\ES{N}_F]$, l'ensemble des $F$-lames, resp. $F$-strates, de $\ES{N}_F$. 
\end{itemize}}
\end{notation} 

D'aprs \ref{descente 3 sous HBFS}, si l'hypothse \ref{hypo bonnes F-strates faible} est vŽrifiŽe, les applications 
$$\llbracket \ES{N} \rrbracket_F \rightarrow \llbracket \NF \rrbracket,\, \scrY \mapsto V(F) \cap \scrY \quad \hbox{et}\quad 
[\ES{N}]_F \rightarrow [\NF],\, \bsfrY \mapsto \NF \cap \bsfrY$$ sont bijectives. 

\vskip1mm
Rappelons que pour 
$v\in \NF$, on a posŽ $\bsfrX_{F,v}^+ = V(F) \cap {_F\bsfrX_v^+}= \NF \cap (G\cdot {_F\scrX_v})$. Pour $v\in \NF$ tel que 
$\check{X}_F(G)_{\mbb{Q}} \cap \bs{\Lambda}_v \neq \emptyset$, puisque ${_F\scrX_v}= \scrX_v$ (\ref{le cas o F-lame = lame gŽo}), on a 
donc $$\bsfrX^{\scriptscriptstyle +}_{F,v}=\NF \cap \bsfrX_v\ptf$$ 

ConsidŽrons la variante suivante de l'hypothse \ref{hypo sur le bord faible}: 

\begin{hypothese}\label{hypo sur le bord faible +} 
Pour tout $v\in \NF$, l'ensemble $\bsfrX_{F,v}^{\scriptscriptstyle +}\smallsetminus \bsfrY_{F,v}$ est rŽunion (finie) de $F$-strates de $\NF$; \textit{i.e.} pour tout $v'\in \bsfrX_{F,v}^{\scriptscriptstyle +}\smallsetminus \bsfrY_{F,v}$, 
on a $\bsfrY_{F,v'}\subset \bsfrX_{F,v}^{\scriptscriptstyle +}$. 
\end{hypothese}

\begin{lemme}\label{lemme sur X^+_{F,v}}
On suppose que l'hypothse \ref{hypo bonnes F-strates faible} est vŽrifiŽe. Soit $v\in \NF$. 
\begin{enumerate}
\item[(i)] $v'\in \bsfrX^{\scriptscriptstyle +}_{F,v} \Rightarrow \bs{q}_F(v') \leq \bs{q}_F(v)$.
\item[(ii)] $\bsfrY_{F,v} = \{v'\in \bsfrX^{\scriptscriptstyle +}_{F,v}\,\vert \, \bs{q}_F(v') = \bs{q}_F(v) \}$. 
\item[(iii)] Supposons que l'hypothse \ref{hypo sur le bord} soit vŽrifiŽe pour $\overline{F}$: 
pour toute $\overline{F}$-strate $\bsfrY$ de $\ES{N}$, le bord $\bsfrX(\bsfrY)\smallsetminus \bsfrY$ est rŽunion de $\overline{F}$-strates de $\ES{N}$. 
Alors l'hypothse \ref{hypo sur le bord faible +} est vŽrifiŽe. 
\end{enumerate}
\end{lemme}

\begin{proof}
Pour $v'\in \bsfrX^{\scriptscriptstyle +}_{F,v}$, on a $q_F(v')= q(v')\leq q(v) = q_F(v)$ avec ŽgalitŽ si et seulement si $v'\in \bsfrY_v$; d'o les points (i) et (ii). Quant au point (iii), supposons 
l'hypothse \ref{hypo sur le bord} vŽrifiŽe pour $\overline{F}$. Alors l'ensemble $\bsfrX_v$ s'Žcrit $\bsfrX_v = \coprod_{i=0}^m \bsfrY_{v_i}$ 
pour des ŽlŽments $v_i \in \ESN$. On peut supposer que si l'intersection $\NF \cap \bsfrY_{v_i}$ est non vide, l'ŽlŽment $v_i$ appartient ˆ 
$\NF$. D'aprs \ref{descente 3 sous HBFS}\,(iii), 
l'intersection $\NF \cap \bsfrY_{v_i}$, si elle est non vide, est une $F$-strate de $\NF$. Cela prouve (iii). 
\end{proof}

\begin{remark}
\textup{On suppose que l'hypothse \ref{hypo bonnes F-strates faible} est vŽrifiŽe. Pour $v\in \NF$, d'aprs \ref{lemme sur X^+_{F,v}}\,(i) et (ii), on a $\bsfrX_{F,v}^+ \smallsetminus \bsfrY_{F,u} \subset \ES{N}_{F,<\bs{q}_F(v)}$, ce qui donne un sens ˆ l'hypothse \ref{hypo sur le bord faible +}: on peut voir $\bsfrX_{F,v}^+ \smallsetminus \bsfrY_{F,u}$ comme une variante Žlargie du bord 
$\bsfrX_{F,v}\smallsetminus \bsfrY_{F,v}$, compatible ˆ la stratification de $\NF$ donnŽe par la fonction $\bs{q}_F$.  }
\end{remark}

\subsection{Le critre de Kirwan-Ness (cas d'un $G$-module)}\label{le critre de KN rat} Dans cette sous-section, on 
suppose que $V$ est un $G$-module dŽfini sur $F$ (avec $e_V=0$). Puisque $V$ est lisse, l'hypothse 
de rŽgularitŽ \ref{hyp reg} est automatiquement vŽrifiŽe. On reprend les notations introduites en 
\ref{la variante de Hesselink} (\S5).

\vskip2mm
\textit{\S1.\hskip1mm Le sous-groupe $M_\lambda^\perp$ de $M_\lambda$. ---} 
Pour un co-caractre $\lambda \in \check{X}(G)\smallsetminus \{0\}$ et un tore $T$ de $M_\lambda$ contenant $\mathrm{Im}(\lambda)$, notons 
$T^\lambda$ le sous-tore de $T$ dŽfini par\index{Tlambda@$T^\lambda$} $$T^\lambda = \langle \mathrm{Im}(\mu) \,\vert \, \mu \in \check{X}(T),\, (\mu,\lambda)=0 \rangle\ptf$$ 
On a la dŽcomposition $$T= T^\lambda \mathrm{Im}(\lambda) \quad \hbox{et}\quad 
\check{X}(T^\lambda)_\mbb{Q} = \left\{ \left. \mu - \frac{(\mu,\lambda)}{(\lambda,\lambda)}\lambda \, \right | \, \mu \in \check{X}(T)_\mbb{Q}  \right\}\ptf$$ 
PrŽcisŽment, le morphisme produit $T^\lambda \times \mathrm{Im}(\lambda) \rightarrow T$ est surjectif et 
c'est une isogŽnie (pas forcŽment sŽparable), d'o la dŽcomposition 
$$\check{X}(T)_{\mbb{Q}}=\check{X}(T^\lambda)_{\mbb{Q}}\oplus  \mbb{Q}\lambda \ptf$$ 
On a dŽfini en \ref{la variante de Hesselink} des isomorphismes $\mbb{Q} $-linŽaires 
$$\iota_T: X(T)_{\mbb{Q}} \rightarrow \check{X}(T)_{\mbb{Q}}\quad \hbox{et}\quad \iota_{T^\lambda}: X(T^\lambda)_\mbb{Q} 
 \rightarrow \check{X}(T^\lambda)_\mbb{Q}\ptf$$ 
Pour $\chi \in X(T)$, on a $$\iota_{T^\lambda}(\chi\vert_{T^\lambda})= \iota_T(\chi) - \frac{(\iota_T(\chi),\lambda)}{(\lambda,\lambda)}\lambda\ptf$$ 

\begin{remark}\label{relation entre les mu}
\textup{
Si $v\in V_\lambda(k)\smallsetminus \{0\}$ pour un entier $k>0$, alors 
$$\ES{K}_{T^\lambda}(v) =\left\{ \left. \mu - \frac{k}{(\lambda,\lambda)}\lambda \, \right | \, \mu \in \ES{K}_T(v) \right\}$$ 
et (puisque $\lambda$ est orthogonal ˆ $\check{X}(T^\lambda)_\mbb{Q} $) 
$$\mu^{T^\lambda\!}(v) = \mu^{T\!}(v) - \frac{k}{(\lambda ,\lambda)} \lambda \ptf$$
}
\end{remark}

Supposons de plus que $T$ soit un tore maximal de $M_\lambda$. On pose alors\index{Mlambdaperp@$M_\lambda^\perp$} 
$$ M_\lambda^\perp = \langle T^\lambda , (M_\lambda)_\mathrm{der}\rangle \ptf$$ 
Puisque $\mathrm{Int}_m(T^\lambda)= (\mathrm{Int}_m(T))^\lambda$ pour tout $m\in M_\lambda$ et que $(M_\lambda)_\mathrm{der}$ 
est distinguŽ dans $M_\lambda$, le groupe $M_\lambda^\perp$ ne dŽpend pas de $T$. 
On a la dŽcomposition $$M_\lambda^\perp = T^\lambda  (M_\lambda)_\mathrm{der} $$ et le groupe $M_\lambda^\perp$ 
est rŽductif connexe (\cite[cor.~2.2.7]{Sp}, \cite[IV.14.2]{B}). 
Si $\alpha$ est une racine de $T$ dans $M_\lambda$, alors $\langle \alpha , \lambda \rangle =0$ d'o $(\check{\alpha}, \lambda)=0$ et 
$\mathrm{Im}(\check{\alpha})\subset T^\lambda$. Puisque les $\mathrm{Im}(\check{\alpha})$ engendrent un tore maximal du groupe dŽrivŽ $(M_\lambda)_\mathrm{der}$, 
cela prouve que $T^\lambda$ est un tore maximal de $M_\lambda^\perp$.

\vskip2mm
\textit{\S2.\hskip1mm Le critre de Kirwan-Ness gŽomŽtrique. ---} Pour $\lambda\in \check{X}(G)\smallsetminus \{0\}$ et $k\in \mbb{N}^*$, 
le $\overline{F}$-espace vectoriel $V_\lambda(k)$ est muni d'une structure de $M_\lambda$-module. Si de plus $\lambda$ est primitif, le critre de Kirwan-Ness gŽomŽtrique dit 
que la $v$-optimalitŽ de $\lambda$ pour un ŽlŽment $v\in V_{\lambda,k}$ Žquivaut ˆ la $M_\lambda^\perp$-semi-simplicitŽ de la composante $v_\lambda(k)$ de $v$ sur $V_\lambda(k)$. PrŽcisŽment 
(\textit{cf.} \cite[2.8]{Ts} pour une dŽmonstration valable en toute caractŽristique): 

\begin{proposition}\label{KNG Tsuji}
$\lambda \in \Lambda_v^{\mathrm{opt}}$ si et seulement si $v_\lambda(k)$ est $M_\lambda^\perp$-semi-stable, \textit{i.e.} appartient ˆ 
$V_\lambda(k)\smallsetminus \ES{N}^{M_\lambda^\perp}(V_\lambda(k),0)$; auquel 
cas $\Lambda_v^{\mathrm{opt}}=\Lambda_{v_\lambda(k)}^{\mathrm{opt}}$ et $P_\lambda=P_v=P_{v_\lambda(k)}$.  
\end{proposition}

En particulier on a 
$$\lambda \in \Lambda_v^{\rm opt} \Rightarrow \scrY_v + V_{\lambda,k+1}= \scrY_v\ptf$$

\begin{remark}\label{cor GIT}
\textup{
D'aprs \ref{rŽduction au cas d'un G-module}\,(ii), on a aussi que 
 $\lambda\in \Lambda_{v}^\mathrm{opt}$ si et seulement s'il existe une fonction polynomiale $f\in \overline{F}[V_\lambda(k)]^{M_\lambda^\perp}$ 
 homogne de degrŽ $\geq 1$ telle que $f(v_\lambda(k))\neq 0$.   
}\end{remark}

L'espace projectif $$\mbb{P}(V) =(V\smallsetminus \{0\})/ \smash{\overline{F}}^\times$$
est encore un $G$-module. On note $q=q_V: V\smallsetminus \{0\} \rightarrow \mbb{P}(V)$ la projection naturelle. 
Pour $v\in \ES{N}\smallsetminus \{0\}$, $\lambda \in \Lambda_v^{\rm opt}$ et $k = m_v(\lambda)$, d'aprs le critre de Kirwan-Ness, la composante 
$\overline{v}= v_\lambda(k)$ de $v$ sur $V_\lambda(k)$ est dans la lame $\scrY_v$ et elle vŽrifie 
$$\smash{\overline{F}}^\times \cdot \overline{v}\subset M_\lambda \cdot \overline{v} \subset V_\lambda(k) \cap \scrY_v\ptf$$ 
On en dŽduit (gr‰ce ˆ \textit{loc.~cit.}) que $$\smash{\overline{F}}^\times \cdot \scrY_v = \scrY_v\ptf$$ 
Soit $G^{q(\overline{v})}\;(=Z_G(q(\overline{v}))(\overline{F}))$ le stabilisateur de $q(\overline{v})$ dans $G$ au sens de Borel \cite[ch.~II, 1.7]{B} (\textit{cf.} \ref{notations et rappels}, \S3):  
$$G^{q(\overline{v})}= \{g\in G\,\vert \, g\cdot \overline{v} \in \smash{\overline{F}}^\times \overline{v}\}\ptf$$ 
C'est un sous-groupe fermŽ de $G$ qui contient ${\rm Im}(\lambda)$ et le stabilisateur $G^{\overline{v}}$ de $\overline{v}$ dans $G$. 

\begin{lemme}\label{espace projectif}
Soient $v\in \ES{N}\smallsetminus \{0\}$, $\lambda\in \Lambda_v^{\rm opt}$, $k= m_v(\lambda)$ et $\overline{v}= v_\lambda(k)$.
\begin{enumerate}
\item[(i)] On a $G^{q(\overline{v})}= {\rm Im}(\lambda)\ltimes G^{\overline{v}}$.
\item[(ii)] Si $T^\natural$ est un tore maximal de $G^{q(\overline{v})}$, on a $\check{X}(T^\natural)\cap \Lambda_v^{\rm opt} = \{\lambda^\natural\}$.
\end{enumerate}
\end{lemme}

\begin{proof} 
Posons $G^\natural= G^{q(\overline{v})}$. 
Pour $g\in G^\natural$, notons $a_g\in \smash{\overline{F}}^\times$ l'ŽlŽment dŽfini par $g\cdot \overline{v}= a_g \overline{v}$. 
L'application $g\mapsto a_g$ est un ŽlŽment de $X(G^\natural)$ et pour chaque $g\in G^\natural$, il existe un $t\in \smash{\overline{F}}^\times$ tel que $t^\lambda a_g =1$. 
Cela implique (i). 

Comme $G^{\overline{v}}\subset P_{\overline{v}}=P_\lambda$ et ${\rm Im}(\lambda)\subset P_\lambda$, d'aprs (i) on a l'inclusion 
$G^\natural\subset P_\lambda$. Soit $T^\natural$ un tore maximal de $G^\natural$. Puisque tous les tores maximaux de $G^\natural$ sont conjuguŽs dans $G^\natural$, il existe un 
$g^\natural\in G^\natural$ tel que $${\rm Im}(g^\natural\bullet \lambda)= g^\natural\bullet  {\rm Im}(\lambda) \subset T^\natural\ptf$$ Comme $g^\natural\in P_\lambda$, le co-caractre $\lambda^\natural = 
g^\natural \bullet \lambda$ est 
$v$-optimal et c'est l'unique ŽlŽment de $\check{X}(T^\natural ) \cap \Lambda_v^{\rm opt}$. 
\end{proof}

\vskip2mm
\textit{\S3.\hskip1mm Le critre de Kirwan-Ness rationnel. ---} On prouve dans ce paragraphe la version  $F$-rationnelle du critre de Kirwan-Ness. Il suffit pour cela de reprendre la 
preuve de Tsuji \cite[2.8]{Ts} en remplaant les tores maximaux par les tores $F$-dŽployŽs maximaux. 

Pour $\lambda\in \check{X}_F(V)$ et $i\in \mbb{Z}$, le sous-espace $V_\lambda(i)$ 
est muni d'une structure de $M_\lambda$-module dŽfini sur $F$. 
La projection\index{vlambda(i)@$v_\lambda(i)$} $V_{\lambda, i} \rightarrow V_\lambda(i),\, v\mapsto v_\lambda(i)$ est elle aussi dŽfinie sur $F$. On a dŽfini plus haut le sous-groupe 
$M_\lambda^\perp= \langle T^\lambda ,(M_{\lambda})_\mathrm{der}\rangle$ de $M_\lambda$, o $T$ est un tore maximal 
de $M_\lambda$. C'est un groupe rŽductif connexe, dŽfini sur $F$ puisque $M_\lambda$ l'est (il suffit de prendre $T$ dŽfini sur $F$). 
Soit $S$ un tore de $M_\lambda$ contenant $\mathrm{Im}(\lambda)$. Si $S$ est dŽfini sur $F$, resp. $F$-dŽployŽ, alors $S^\lambda$ l'est aussi; 
et si $S$ est un tore $F$-dŽployŽ maximal de $M_\lambda$, 
alors $S^\lambda$ est un tore $F$-dŽployŽ maximal de $M_\lambda^\perp$. En effet  
dans ce dernier cas, en prenant pour $T$ un tore maximal de $M_\lambda$ dŽfini sur $F$ et contenant $S$, on a 
$$\check{X}(S)= \check{X}_F(T)\quad \hbox{et} \quad S^\lambda = S \cap T^\lambda\ptf$$

\begin{lemme}\label{lemme 2.7 de [Ts]}
Soit $S$ un tore $F$-dŽployŽ maximal de $G$ et soit $v\in V \smallsetminus \{0\}$ un ŽlŽment $S$-instable (donc $(F,S)$-instable puisque $\check{X}_F(S)=\check{X}(S)$). 
Posons $\lambda = \lambda^S(v)$ et $k= m_v(\lambda)$. Pour tout $v'\in v + V_{\lambda,k+1}$, on a $\mu^S(v')=\mu^S(v)$ et $\lambda^S(v')= \lambda$ (\textit{cf.} \ref{la variante de Hesselink} 
pour les notations). 
\end{lemme}

\begin{proof}La preuve est identique ˆ celle de \cite[lemma~2.7]{Ts}; le fait que $S$ ne soit pas un tore maximal de $G$ ne change rien ˆ l'affaire. 
\end{proof}

\begin{proposition}\label{thm de KN rationnel}
Soient $\lambda\in \check{X}_F(G)$, $k\in \mbb{N}^*$ et $v\in V$ tels que $m_v(\lambda)=k$. On suppose que $\lambda$ est primitif. Alors 
$\lambda \in \Lambda_{F,v}^\mathrm{opt}$ si et seulement si $v_\lambda(k)$ est 
$(F,M_\lambda^\perp)$-semi-stable, \textit{i.e.} appartient ˆ $V_\lambda(k) \smallsetminus \FN^{M_\lambda^\perp}(V_\lambda(k),0)$, 
auquel cas $$\Lambda_{F,v}^\mathrm{opt}= \Lambda_{F,v_\lambda(k)}^\mathrm{opt}\quad\hbox{et}\quad P_\lambda = {_FP_v} ={_FP_{v_\lambda(k)}}\ptf$$
\end{proposition}

\begin{proof} On reprend celle de \cite[theorem~2.8]{Ts}. 
Observons que l'ŽlŽment $\overline{v}=v_\lambda(k)$ est $(F,M_\lambda^\perp)$-semi-stable si et seulement si pour tout tore $F$-dŽployŽ maximal $S$ de $M_\lambda$,  
il est $S^\lambda$-semi-stable; ce qui Žquivaut ˆ $\mu^{S^\lambda\!}(\overline{v})=0$, ou encore (d'aprs \ref{relation entre les mu}) ˆ $\mu^{S\!}(\overline{v})= \frac{k}{(\lambda,\lambda)}\lambda$. 
Soit $S$ un tore $F$-dŽployŽ maximal de $M_\lambda$. Puisque ${\rm Im}(\lambda)\subset S$, $v$ est $S$-instable. Posons $\lambda_1=\lambda^S(v)$ et $k_1=m_v(\lambda_1)$. D'aprs \ref{lemme 2.7 de [Ts]}, on a donc 
$$\mu^{S\!}(v)= \mu^{S\!}(v_{\lambda_1}(k_1))\quad\hbox{et}\quad \lambda_1 =\lambda^{S\!}(v_{\lambda_1}(k_1))\ptf$$

Supposons que $\lambda$ soit $(F,v)$-optimal. Alors $S$ est $(F,v)$-optimal (relativement ˆ l'action de $G$ sur $V$) et $\lambda_1=\lambda$. 
Par consŽquent $\mu^{S\!}(\overline{v})$ et $\lambda$ sont proportionnels, ce qui Žquivaut (d'aprs \ref{relation entre les mu}) ˆ $\mu^{S^\lambda\!}(\overline{v})=0$. 
Puisque cela vaut pour tout tore $F$-dŽployŽ maximal $S$ de $M_\lambda$, 
on a montrŽ que $\overline{v}$ est $(F,M_\lambda^\perp)$-semi-stable.

RŽciproquement, supposons que $\overline{v}$ soit $(F,M_\lambda^\perp)$-semi-stable. Soit $S'$ un tore $F$-dŽployŽ maximal de $G$ contenu 
dans $P_\lambda \cap {_FP_v}$. Choisissons un $u\in U_\lambda(F)$ tel que $S=uS'u^{-1}$ soit contenu dans $M_\lambda$. Alors $\overline{v}$ 
est $S^\lambda$-semi-stable par hypothse. Par consŽquent $\mu^{S^\lambda\!}(\overline{v})=0$ et $\lambda = \frac{k}{(\lambda,\lambda)}\mu^{S\!}(\overline{v})$ (\ref{relation entre les mu}). 
On a donc $\lambda = \lambda^{S\!}(\overline{v})= \lambda^{S\!}(v)$ (\ref{lemme 2.7 de [Ts]}). Puisque $u\cdot v - v \in V_{\lambda,k+1}$, on a $(u\cdot v)_\lambda(k)= \overline{v}$. On a donc 
aussi $\lambda= \lambda^{S\!}(u\cdot v)$. Or ${_FP_{u\cdot v}}= u({_FP_v})u^{-1} \supset S$ et $S$ est un tore $F$-dŽployŽ maximal de $G$ qui est $(F,u\cdot v)$-optimal. 
Donc $\check{X}(S)\cap \Lambda_{F,u \cdot v}^{\rm opt} = \{\lambda\}$ et $\Lambda_{F,u\cdot v}^{\rm opt}= \Lambda_{F,v}^{\rm opt}$. 
\end{proof}

\begin{corollary}\label{Y et translations}
(Sous les hypothses de \ref{thm de KN rationnel}.) Si $\lambda\in \Lambda_{F,v}^{\mathrm{opt}}$, on a $${_F\scrY_v}+ V_{\lambda,k+1} = {_F\scrY_v}\ptf$$
\end{corollary}

On aimerait pouvoir tester la $F$-semi-stabilitŽ sur un tore $F$-dŽployŽ maximal de $M_\lambda^\perp$ plut™t que sur le groupe $M_\lambda^\perp$ tout entier. 
L'intersection $P_\lambda\cap {_FP_v}$ contient un tore $F$-dŽployŽ maximal $S'$ de $G$ et l'on peut toujours 
choisir un ŽlŽment $u\in U_\lambda(F)$ 
tel que le tore $S=uS'u^{-1}$ soit contenu dans $M_\lambda$ (cet argument est utilisŽ dans la preuve de \ref{thm de KN rationnel}). On en dŽduit le

\begin{corollary}\label{variante 1 du thm de KN rationnel}(Sous les hypothses de \ref{thm de KN rationnel}.)
$\lambda\in \Lambda_{F,v}^{\rm opt}$ si et seulement s'il existe un tore $F$-dŽployŽ maxi\-mal $S$ de $M_\lambda$ et un ŽlŽment 
$u\in U_\lambda(F)$ tel que $u^{-1}Su\subset {_FP_{v_\lambda(k)}}$ et $v_\lambda(k)$ soit $S^\lambda$-semi-stable.
\end{corollary}

\begin{proof}
Posons $\overline{v}=v_\lambda(k)$. Si $\lambda\in \Lambda_{F,v}^\mathrm{opt}$, alors $\overline{v}$ est $(F,M_\lambda^\perp)$-semi-stable et 
${_FP_{\overline{v}}}={_FP_v}= P_\lambda$. Dans ce cas, pour tout tore $F$-dŽployŽ maximal $S$ de $M_\lambda$ 
et tout ŽlŽment $u\in U_\lambda(F)$, on a $u^{-1}Su\subset {_FP_{\overline{v}}}$ et $\overline{v}$ est $S^\lambda$-semi-stable. 
RŽciproquement, supposons 
qu'il existe une telle paire $(S,u)$. Alors $S$ est contenu dans ${_FP_{v'}}$ o l'on a posŽ $v'= u\cdot \overline{v}$. Par 
consŽquent $\check{X}(S)\cap \Lambda_{F,v'}^\mathrm{opt}= \{\lambda'\}$ et $\lambda'$ est l'unique co-caractre primitif de $S$ 
qui soit proportionnel ˆ $\mu^{S}(v')$. Comme $v' \in \overline{v} + V_{\lambda,k+1}$, on a (d'aprs \ref{relation entre les mu}) $\mu^S(v')= \mu^S(\overline{v})$; et comme 
$\overline{v}$ est $S^\lambda$-semi-stable, on obtient $\mu^S(v')= \frac{k}{(\lambda,\lambda)}\lambda$. Donc $\lambda'= \lambda$ appartient ˆ $\Lambda_{F,v'}^{\mathrm{opt}}$. 
Comme $v'- v$ appartient ˆ $V_{\lambda,k+1}$ et que (d'aprs \ref{Y et translations}) 
${_F\scrY_{v'}}+ V_{\lambda ,k+1}= {_F\scrY_{v'}}$, 
cela entra"ne que $v$ appartient ˆ ${_F\scrY_{v'}}$ et donc que $\lambda$ appartient ˆ $\Lambda_{F,v}^\mathrm{opt}=\Lambda_{F,v'}^\mathrm{opt}$. 
\end{proof}

Notons $A_\lambda$ le tore central $F$-dŽployŽ maximal 
de $M_\lambda$ et posons 
$$A_\lambda^\perp = (A_\lambda)^\lambda \;(= \langle \mathrm{Im}(\mu)\,\vert \, \mu \in \check{X}(A_\lambda),\,(\lambda,\mu)=0\rangle)\ptf$$ 
Observons que le $F$-sous-groupe parabolique $P_\lambda$ de $G$ est minimal (parmi les $F$-sous-groupes paraboliques de $G$) 
si et seulement si $A_\lambda$ est l'unique tore $F$-dŽployŽ maximal de $M_\lambda$; auquel cas un ŽlŽment est $(F,M_\lambda^\perp)$-semi-stable 
si et seulement s'il est $A_\lambda^\perp$-semi-stable. En gŽnŽral (\textit{i.e.} sans supposer $P_\lambda$ minimal), on se ramne au tore $A_\lambda$ gr‰ce au

\begin{lemme}\label{lemme de S ˆ A} 
(Sous les hypothses de \ref{thm de KN rationnel}.) Soit $S$ un tore $F$-dŽployŽ maximal de $M_\lambda$. L'ŽlŽment 
$v_\lambda(k)$ est $S^\lambda$-semi-stable si et seulement s'il est $A_\lambda^\perp$-semi-stable et 
$\|\mu^{S\!}(v_\lambda(k))\|=\| \mu^{A_\lambda\!}(v_\lambda(k))\|$.
\end{lemme}

\begin{proof}
Posons $\overline{v}=v_\lambda(k)$. Rappelons qu'on a notŽ $\ES{R}'_{S}(\overline{v})$ l'ensemble des $\chi\in X(S)$ tels que 
$\overline{v}_\chi \neq 0$ et $\ES{K}_{S}(\overline{v})$ l'enveloppe convexe de $\iota_{S}(\ES{R}'_S(\overline{v}))$ dans $\check{X}(S)_{\mbb{Q}}$. 
La surjection naturelle $X(S)\rightarrow X(A_\lambda),\, \chi \mapsto \chi\vert_{A_\lambda}$ se restreint en une 
application surjective $\ES{R}'_S(\overline{v})\rightarrow \ES{R}'_{A_\lambda}(\overline{v})$. 
Posons $$\overline{X}= \ker \left(X(S) \rightarrow X(A_\lambda)\right)= \ker \left(X(S^\lambda)\rightarrow 
X(A_\lambda^\perp)\right)\ptf$$ On a la dŽcomposition 
$$\check{X}(S)_{\mbb{Q}}= \check{X}(A_\lambda)_{\mbb{Q}}\oplus \mathrm{Hom}(\overline{X},\mbb{Q})$$ et $\ES{K}_{A_\lambda}(\overline{v})$ est la 
projection orthogonale de $\ES{K}_S(\overline{v})$ sur $\check{X}(A_\lambda)_{\mbb{Q}}$. 
D'o l'inŽgalitŽ $$\|\mu^{A_\lambda\!}(\overline{v}) \|\leq \|\mu^{S\!}(\overline{v})\|\ptf$$
D'autre part on a 
$$\mu^{A_\lambda\!}(\overline{v})= \mu^{A_\lambda^\perp\! }(\overline{v}) + \frac{k}{(\lambda,\lambda)}\lambda \quad \hbox{et}\quad 
\mu^{S\!}(\overline{v})= \mu^{S^\lambda\!}(\overline{v}) + \frac{k}{(\lambda,\lambda)}\lambda \ptf$$ 
On en dŽduit que $\mu^{S^\lambda\!}(\overline{v})=0$ si et seulement si 
$\mu^{A_\lambda^\perp\!}(\overline{v})=0$ et $\|\mu^{S\!}(\overline{v})\|=\| \mu^{A_\lambda\!}(\overline{v})\|$. 
D'o le lemme. 
\end{proof}

On peut comparer le critre de Kirwan-Ness rationnel avec le critre de Kirwan-Ness gŽomŽtrique: 

\begin{proposition}\label{comparaison KN rat avec KN gŽo}
Soient $v\in \FN$, $\lambda\in \Lambda_{F,v}^{\mathrm{opt}}$ et $k= m_v(\lambda)$. Supposons que 
$F$ soit infini et $G$ soit $F$-dŽployŽ; ou bien que $v$ appartienne ˆ $\NF$. Les conditions suivantes sont Žquivalentes:
\begin{enumerate}
\item[(i)] $\lambda\in \Lambda_v^{\mathrm{opt}}$ (\textit{i.e.} $\frac{1}{k}\lambda\in \bs{\Lambda}_v$);
\item[(ii)] $v_\lambda(k)\notin \ES{N}^{M_\lambda^\perp}(V_\lambda(k),0)$;
\item[(iii)] $\ES{N}^{M_\lambda^\perp}(V_\lambda(k),0) \neq V_\lambda(k)$.;
\item[(iv)] l'ensemble $V_\lambda(k) \smallsetminus{_F\ES{N}^{M_\lambda^\perp}(V_\lambda(k),0)}$ est ouvert dans $V_\lambda(k)$;
\item[(v)] l'ensemble $V_\lambda(k) \smallsetminus{_F\ES{N}^{M_\lambda^\perp}(V_\lambda(k),0)}$ est $M_\lambda$-invariant. 
\end{enumerate}
\end{proposition}

\begin{proof}L'Žquivalence $(i)\Leftrightarrow (ii)$ est le critre de Kirwan-Ness gŽomŽtrique. D'aprs le critre de Kirwan-Ness rationnel (\ref{thm de KN rationnel}) on a 
$$V_\lambda(k)\smallsetminus {_F\ES{N}^{M_\lambda^\perp}(V_\lambda(k),0)}= \{v'_\lambda(k)\,\vert \, v'\in {_F\scrY_v}\}\ptf$$ 
D'aprs \ref{Y et translations}, on en dŽduit que $V_\lambda(k)\smallsetminus {_F\ES{N}^{M_\lambda^\perp}(V_\lambda(k),0)}$ est ouvert dans $V_\lambda(k)$ si et seulement si 
${_F\scrY_v}$ est ouvert dans $V_{\lambda,k}$; et que $V_\lambda(k)\smallsetminus {_F\ES{N}^{M_\lambda^\perp}(V_\lambda(k),0)}$ est $M_\lambda$-invariant si et seulement si 
${_F\scrY_v}$ est $P_\lambda$-invariant. Gr‰ce ˆ \ref{deux critres bonnes F-lames} et \ref{deux critres bonnes F-lames bis}, 
on obtient que les conditions $(i)$, $(iv)$ et $(v)$ sont Žquivalentes. 

L'implication $(ii)\Rightarrow (iii)$ est claire. Supposons que $\ES{N}^{M_\lambda^\perp}(V_\lambda(k),0) \neq V_\lambda(k)$. Puisque 
$\ES{N}^{M_\lambda^\perp}(V_\lambda(k),0)$ est fermŽ dans $V_\lambda(k)$ et contient ${_F\ES{N}^{M_\lambda^\perp}(V_\lambda(k),0)}$, 
cela entra"ne que $V_\lambda(k) \smallsetminus {_F\ES{N}^{M_\lambda^\perp}(V_\lambda(k),0)}$ contient un ouvert non vide de $V_\lambda(k)$. Par consŽquent 
la $F$-lame ${_F\scrY_v}$ contient un ouvert non vide de ${_F\scrY_v}=V_{\lambda,k}$, ce qui n'est possible que si $\lambda \in \Lambda_v^{\mathrm{opt}}$ (d'aprs \ref{mauvaises F-strates} et 
\ref{deux critres bonnes F-lames bis}). Donc $(iii)\Rightarrow (i)$.
\end{proof}

\textit{\S4.\hskip1mm SŽparabilitŽ et rationalitŽ. ---} Reprenons l'espace projectif $$\mbb{P}(V)= (V\smallsetminus \{0\})/ \smash{\overline{F}}^\times$$ et la projection naturelle 
$q=q_V: V\smallsetminus \{0\} \rightarrow \mbb{P}(V)$. Soient $v\in \ES{N} \smallsetminus \{0\}$, $\lambda \in \Lambda_v^{\rm opt}$ et 
$k= m_v(\lambda)\in \mbb{N}^*$. Notons $\overline{v}=v_\lambda(k)$ la composante de $v$ sur $V_\lambda(k)$. 

Si $\lambda'$ est un autre ŽlŽment de 
$\Lambda_{v}^{\rm opt}$, alors $\lambda'=u\bullet \lambda$ pour un unique $u\in U_v= U_\lambda$ et $$v_{\lambda'}(k) = (u\cdot v)_{\lambda'}(k)= u \cdot \overline{v}\in \overline{v} + V_{\lambda,k+1};$$ 
en particulier $v_{\lambda'}(k)$ est sŽparable si et seulement si $\overline{v}$ est sŽparable.

Plus gŽnŽralement, pour $p\in P_v= P_\lambda$, on a 
$$(p\cdot v)_{p\bullet \lambda}(k)= p \cdot \overline{v}\in m\cdot \overline{v} + V_{\lambda, k+1}\quad \hbox{si $p\in mU_v$ avec $m\in M_\lambda$}\ptf$$ 
On en dŽduit que les deux conditions suivantes sont Žquivalentes:
\begin{itemize}
\item $\overline{v}$ est sŽparable;
\item il existe un $w\in P_v\cdot v$ et un $\lambda'\in \Lambda_v^{\rm opt}= \Lambda_w^{\rm opt}$ tels que $w_{\lambda'}(k)$ soit sŽparable. 
\end{itemize}
\begin{lemme}\label{projseprat}Supposons que l'ŽlŽment $\overline{v}=v_\lambda(k)$ soit sŽparable, \textit{i.e.} que le morphisme 
$G\rightarrow \ES{O}_{\overline{v}},\, g \mapsto g\cdot \overline{v}$ soit sŽparable. 
\begin{enumerate}
\item[(i)]L'ŽlŽment $q(\overline{v})$ est sŽparable, \textit{i.e.} le morphisme $G \rightarrow \ES{O}_{q(\overline{v})},\, g \mapsto g \cdot q(\overline{v})= q(g\cdot \overline{v})$ est sŽparable. 
\item[(ii)]Si $\overline{v}\in V(F)$, alors $\overline{v}$ appartient ˆ $\ES{N}_F$, $\bs{\Lambda}_{F,\overline{v}}= \check{X}_F(G)_{\mbb{Q}}\cap \bs{\Lambda}_v$, ${_FP_{\overline{v}}}= P_v$ et 
$$\scrY_{F,\overline{v}}= V(F)\cap \scrY_{v}\ptf$$
\item[(iii)]Si $\overline{v}\in V(F)$ et si pour tout $v'\in V(F) \cap (\bsfrY_v \smallsetminus \scrY_v)$, il existe un ŽlŽment $w \in \scrY_{v'}$ et un co-caractre $\lambda' \in \Lambda_{w}^{\rm opt}= \Lambda_{v'}^{\rm opt}$ tels que l'ŽlŽment $w_{\lambda'}(k)$ soit sŽparable 
et appartienne ˆ $V(F)$, alors $$\bsfrY_{F,\overline{v}}= V(F)\cap \bsfrY_v\ptf$$
\end{enumerate}
\end{lemme} 

\begin{proof}
Prouvons (i). Posons $\ES{O}=\ES{O}_{\overline{v}}$ et $\bs{\ES{O}}= \ES{O}_{q(\overline{v})}$. Notons $\pi_{\overline{v}}: G\rightarrow \ES{O}$ et $\pi_{q(\overline{v})}: G \rightarrow \bs{\ES{O}}$ les morphismes orbites. 
Le morphisme ${\rm d}(\pi_{\overline{v}})_1: \mathfrak{g}\rightarrow T_{\overline{v}}(\ES{O})$ est par hypothse surjectif. Puisque 
$\pi_{q(\overline{v})}= (\ES{O} \buildrel q\over{\longrightarrow} \bs{\ES{O}})\circ \pi_{\overline{v}}$, pour prouver que le morphisme $\pi_{q(\overline{v})}$ est sŽparable, il suffit de prouver que le morphisme 
$q\vert_{\ES{O}}: \ES{O}\rightarrow \bs{\ES{O}}$ est sŽparable. Comme $t^\lambda\cdot \overline{v} = t^k \overline{v}$ pour tout $t\in \smash{\overline{F}}^\times$, la 
$G$-orbite $\ES{O}$ contient $\smash{\overline{F}}^\times \overline{v}$ et $q\vert_{\ES{O}}$ est un $\mbb{G}_{\rm m}$-fibrŽ principal localement trivial. Cela assure que pour 
chaque $v'\in \ES{O}$, le morphisme  
$${\rm d}(q\vert_{\ES{O}})_{v'}: T_{v'}(\ES{O}) \rightarrow T_{q(v')}(\bs{\ES{O}})$$ est surjectif.

Prouvons (ii). Supposons $\overline{v}\in V(F)$. Alors $q(\overline{v})\in \mbb{P}(V)(F)$ et puisque $q(\overline{v})$ est sŽparable (d'aprs (i)), 
le stabilisateur schŽmatique $Z_G(q(\overline{v}))$ est lisse, \textit{i.e.} il co\"{\i}ncide avec le stabilisateur $G^\natural= G^{q(\overline{v})}$ au sens de Borel (\textit{cf.} \cite[ch.~II, 6.7]{B}) et ce dernier est dŽfini sur $F$. On peut donc 
choisir un tore maximal $T^\natural$ de $G^\natural$ qui soit dŽfini sur $F$. D'aprs \ref{espace projectif}\,(ii), on a $\check{X}(T^\natural)\cap \Lambda_v^{\rm opt} = \{\lambda^\natural\}$. 
Puisque $T$ se dŽploie sur $F^{\rm s\acute{e}p}$, le co-caractre $\lambda^\natural$ est dŽfini sur $F^{\rm s\acute{e}p}$. En particulier 
$\overline{v}$ appartient ˆ $V(F)\cap \ES{N}_{F^{\rm s\acute{e}p}}= \ES{N}_F$ et l'on a 
$$\bs{\Lambda}_{F^{\rm s\acute{e}p},\overline{v}} = \check{X}_{F^{\rm s\acute{e}p}}(G)_{\mbb{Q}}\cap \bs{\Lambda}_v\quad \hbox{et} \quad \scrY_{F^{\rm s\acute{e}p},\overline{v}} = V(F^{\rm s\acute{e}p})\cap \scrY_v\ptf$$ 
Par descente sŽparable (\ref{descente sŽparable lames}), on a aussi 
$$\bs{\Lambda}_{F,\overline{v}}= \check{X}_F(G)_{\mbb{Q}} \cap \bs{\Lambda}_{F^{\rm s\acute{e}p},\overline{v}}\quad \hbox{et} \quad \scrY_{F,\overline{v}}= V(F) \cap \scrY_{F^{\rm s\acute{e}p},\overline{v}}\ptf$$ 
Cela prouve (ii).

Prouvons (iii). Puisque $\bsfrY_{F,\overline{v}}= G(F)\cdot \scrY_{F,\overline{v}}$ et (d'aprs (ii)) $ \scrY_{F,\overline{v}}= V(F)\cap \scrY_v$, on a l'inclusion 
$\bsfrY_{F,\overline{v}}\subset V(F)\cap \bsfrY_v$. Quant ˆ l'inclusion inverse, d'aprs (ii) il suffit de prouver l'inclusion $$V(F) \cap (\bsfrY_v \smallsetminus \scrY_v)\subset \bsfrY_{F,\overline{v}}\ptf$$ 
Soient $v'$, $w$ et $\lambda'$ comme dans l'ŽnoncŽ. Posons $\overline{w}= w_{\lambda'}(k)$; 
par hypothse $\overline{w}$ est sŽparable et appartient ˆ $V(F)$. Il s'agit de prouver que $\overline{v}$ et $\overline{w}$ sont dans la mme $F$-strate. D'aprs (ii),  
il existe des co-caractres $\mu\in \check{X}_F(G)_{\mbb{Q}} \cap \bs{\Lambda}_v$ et $\mu'\in \check{X}_F(G)_{\mbb{Q}}\cap \bs{\Lambda}_{w}$. Puisque $v$ et $w$ sont dans la mme $\overline{F}$-strate, 
il existe un $g\in G=G(\overline{F})$ tel que $\mu'= {\rm Int}_g \circ \mu$. On a donc $(P_{\mu'}, M_{\mu'})= {\rm Int}_g(P_{\mu},M_{\mu})$.  
Comme les paires paraboliques $(P_{\mu},M_{\mu})$ et $(P_{\mu'},M_{\mu'})$ sont dŽfinies sur $F$ et conjuguŽes dans $G$, elles 
le sont dans $G(F)$: il existe un $x\in G(F)$ tel que $(P_{\mu'}, M_{\mu'})= {\rm Int}_x(P_{\mu},M_{\mu})$. L'ŽlŽment $y=x^{-1}g$ stabilise $(P_{\mu},M_{\mu})$; il appartient donc 
ˆ $M_{\mu}$ et $y\bullet  \mu= \mu$. On a donc $\mu'= x\bullet \mu$. Puisque (d'aprs (ii)) $\mu\in \bs{\Lambda}_{F,\overline{v}}$ et $\mu'\in \bs{\Lambda}_{F,\overline{w}}$, 
cela entra"ne que $\overline{v}$ et $\overline{w}$ sont dans la mme $F$-strate et donc que $V(F)\cap \scrY_{v'}= \scrY_{F,\overline{w}}$ est 
une $F$-lame contenue dans $\bsfrY_{F,\overline{v}}$. D'o le point (iii). 
\end{proof}

\subsection{Le cas d'un corps topologique}\label{le cas d'un corps top} Soit $F$ un corps commutatif topologique sŽparŽ et non discret\footnote{Observons qu'un tel corps est forcŽment infini.}. 
Si $X$ est une variŽtŽ dŽfinie sur $F$, on peut 
munir $X(F)$ de la topologie (forte) dŽfinie par $F$. On a, pour des $F$-variŽtŽs 
$X$ et $Y$: 
\begin{itemize}
\item si $X= \mbb{G}_{\mathrm{a}}$, la topologie dŽfinie par $F$ sur $X(F)=F$ est celle de $F$;
\item tout morphisme de $F$-variŽtŽs $\alpha: X \rightarrow Y$ induit une application continue $\alpha_F: X(F) \rightarrow Y(F)$ 
qui est un morphisme topologiquement fermŽ, resp. ouvert, si $\alpha$ est une immersion fermŽe, resp. ouverte. 
\item la bijection canonique $(X\times Y)(F) \rightarrow X(F)\times Y(F)$ est un homŽomorphisme.
\end{itemize}
Pour allŽger l'Žcriture, nous dirons\index{TopF@\TopF} \guill{\TopF} au lieu de \guill{la topologie dŽfinie par $F$}. Nous dirons aussi 
\TopF-ouvert, resp. \TopF-fermŽ (etc.), pour ouvert pour \TopF, fermŽ pour {\TopF} (etc.).

Continuons avec les hypothses de \ref{la variante de Hesselink}. Sauf mention expresse du contraire (\textit{e.g.} \ref{TopF-densitŽ des F-lames}), 
on ne suppose pas que l'hypothse \ref{hyp reg} ($e_V$ est rŽgulier dans $V$) est vŽrifiŽe. 

On suppose jusqu'ˆ la fin de \ref{le cas d'un corps top} que $F$ est un corps topologique et on munit 
$V(F)$ et $G(F)$ de la topologie \TopF. Cela fait de $V(F)$ un $G(F)$-espace \textit{topologique}, \cad que l'application $G(F)\times V(F)\rightarrow V(F),\,(g,v)\mapsto g\cdot v$ 
est continue pour \TopF. 

Le rŽsultat suivant est bien connu\footnote{C'est une consquence du fait que pour tout polyn™me non nul $f\in F[x_1,\ldots ,x_n]$, puisque $F$ est sŽparŽ et non discret, l'ensemble 
$\{x\in F^n\,\vert \, f(x)\neq 0\}$ est $\mathrm{Top}_F$-dense dans $F^n$.}.

\begin{lemme}\label{dense implique TopF-dense}
On suppose que $F$ est un corps commutatif topologique sŽparŽ et non discret. Soit $Y$ une $F$-sous-variŽtŽ ouverte de $X= \mbb{A}_F^n$. 
L'ensemble $Y(F)$ est $\mathrm{Top}_F$-dense (et $\mathrm{Top}_F$-ouvert) dans $X(F)= F^n$.  
\end{lemme}

\begin{remark}\label{TopF-densitŽ des F-lames}
\textup{(Sous les hypothses de \ref{dense implique TopF-dense}.) 
Supposons que l'hypothse \ref{hyp reg} soit vŽrifiŽe. Pour $v\in \FN$ tel que la $F$-lame ${_F\scrY_v}$ soit ouverte dans ${_F\scrX_v}$, on sait que 
l'ensemble $\scrY_{F,v}=V(F)\cap {_F\scrY_v}$ est non vide et dense dans ${_F\scrX_v}\simeq_F \mbb{A}_F^n$ (\ref{il existe un point F-rationnel});  
d'aprs \ref{dense implique TopF-dense}, $\scrY_{F,v}$ est $\mathrm{Top}_F$-ouvert et $\mathrm{Top}_F$-dense dans $\scrX_{F,v}$.
}\end{remark}

\textit{\S1.\hskip1mm Le cas d'un corps local. ---} 
Supposons que $F$ soit un \textit{corps local}, \cad:
\begin{itemize}
\item un corps local archimŽdien (\textit{i.e.} isomorphe ˆ $\mbb{R}$ ou $\mbb{C}$);
\item un corps local non archimŽdien (\textit{i.e.} une extension finie de $\mbb{Q}_p$ ou une extension fini du corps des sŽries formelles $\mbb{F}_p((t))$).
\end{itemize}

D'aprs \cite[2.1.1]{MB}, on a le 

\begin{lemme}\label{[2.1.1,MB]}
($F$ est un corps local.) Soit $X$ une $F$-variŽtŽ irrŽductible et non singulire (\textit{i.e.} lisse). Si $\Omega$ est un sous-ensemble non vide 
\TopF-ouvert de $X(F)$, alors $\Omega$ est Zariski-dense dans $X$. 
\end{lemme}

\begin{remark}
\textup{D'aprs \textit{loc.~cit.}, si de plus $X= \mbb{A}^n_F$, 
le lemme \ref{[2.1.1,MB]} reste vrai pour n'importe quel corps commutatif topologique $F$ sŽparŽ et non discret.}
\end{remark}

Si $P$ est un $F$-sous-groupe parabolique de $G$, puisque 
la variŽtŽ quotient $G/P$ est projective, l'ensemble $(G/P)(F)=G(F)/P(F)$ \cite[ch.~V, 20.5]{B} de ses points $F$-rationnels est \TopF-\textit{compact}. 
C'est cette propriŽtŽ qui implique toutes les assertions qui vont suivrent.

Pour $v\in \NF$, l'ensemble $\scrX_{F,v}$ est \TopF-fermŽ dans $V(F)$. 
Comme l'espace quotient $G(F)/P_{F,v}$ est \TopF-compact, l'ensemble $\bsfrX_{F,v}= G(F)\cdot \scrX_{F,v}$ est lui aussi \TopF-fermŽ 
dans $V(F)$. Puisque d'aprs \ref{prop 2.9 de [H2] rat}\,(i), il n'y a qu'un nombre fini d'ensembles $\bsfrX_{F,v}$ avec $v \in \NF$, 
l'ensemble $\NF$ est \TopF-fermŽ dans $V(F)$; et pour $s\in \mbb{Q}_+$, l'ensemble $\ES{N}_{F,<s}$ 
est lui aussi \TopF-fermŽ dans $V(F)$. Par consŽquent pour $v\in \NF$, les ensembles $$ \scrX_{F,v}\smallsetminus \scrY_{F,v}= \scrX_{F,v}\cap \ES{N}_{F,<\bs{q}_F(v)}
\quad \hbox{et}\quad\bsfrX_{F,v}\smallsetminus \bsfrY_{F,v}= \bsfrX_{F,v}\cap \ES{N}_{F,<\bs{q}_F(v)}$$ sont \TopF-fermŽs dans $V(F)$; en d'autres termes, 
la $F$-lame $\scrY_{F,v} $ est \TopF-ouverte dans $\scrX_{F,v}$ et la $F$-strate $\bsfrY_{F,v}$ est \TopF-ouverte dans $\bsfrX_{F,v}$. 
En particulier les $F$-lames et les $F$-strates, de $\NF$ sont des sous-ensembles localement fermŽs de $V(F)$ pour \TopF.

\begin{lemme}\label{homeo pour TopF} ($F$ est un corps local.) 
Pour tout $v\in \NF$, l'application bijective $\bs{\pi}'_{F,v}: G(F)\times^{P_{F,v}} \scrY_{F,v} \rightarrow \bsfrY_{F,v}$ est un 
$\mathrm{Top}_F$-isomorphisme (\textit{i.e.} un homŽomorphisme pour $\mathrm{Top}_F$).
\end{lemme}

\begin{proof}
Soient $v\in \NF$ et $P={_FP_v}$; on a donc $P(F)= P_{F,v}$. L'application naturelle $\alpha:G(F)\times^{P(F)} \scrY_{F,v} \rightarrow G(F)/P(F)$ est un $\mathrm{Top}_F$-fibrŽ de fibre $\scrY_{F,v}$. Comme la projection canonique $G(F)\rightarrow G(F)/P(F)$ admet toujours des sections 
$\mathrm{Top}_F$-continues locales, le $\mathrm{Top}_F$-fibrŽ $\alpha$ est localement trivial, ce qui assure que la bijection ($\mathrm{Top}_F$-continue)  
$\bs{\pi}'_{F,v}$ est $\mathrm{Top}_F$-ouverte; 
c'est donc un $\mathrm{Top}_F$-isomorphisme.
\end{proof}

\vskip2mm
\textit{\S2.\hskip1mm Le cas d'un corps valuŽ admissible. ---} 
Supposons maintenant que $F$ soit un \textit{corps valuŽ admissible}, \cad qu'il soit muni d'une valeur absolue $\nu$ telle que:
\begin{itemize}
\item $(F,\nu)$ soit hensŽlien (\textit{i.e.} si $F'$ est une extension finie de $F$, il existe une \textit{unique} valeur absolue $\nu'$ sur $F'$ qui prolonge $\nu$);
\item le complŽtŽ $\wh{F}$ de $F$ pour la valeur absolue $\nu$ soit une extension sŽparable de $F$.
\end{itemize}

\begin{remark}
\textup{
Tout corps local est \textit{a fortiori} un corps valuŽ admissible. D'autre part si 
$F$ est un corps global et si $v$ est une place de $F$, le corps valuŽ $(F,v)$ est admissible (\textit{cf.} \ref{sŽparabilitŽ du complŽtŽ}) et le complŽtŽ $F_v$ de $F$ en $v$ 
est un corps local.}
\end{remark}

 D'aprs \cite[3.6.2]{GMB}, on a la

\begin{proposition}\label{GMB corps valuŽ admissible}
($(F,\nu)$ est un corps valuŽ admissible.)
\begin{enumerate}
\item[(i)] Si $X$ est une $F$-variŽtŽ, alors $X(F)$ est \TopwhF-dense dans $X(\wh{F})$.
\item[(ii)] Si $\pi: X\rightarrow Y$ est un morphisme propre de $F$-variŽtŽs, alors $\pi_F(X(F))$ est \TopwhF-fermŽ dans $Y(F)$.
\end{enumerate}
\end{proposition}

\begin{corollary}
Pour $v\in \NF$, $\bsfrX_{F,v}$ est \TopwhF-fermŽ dans $V(F)$.
\end{corollary}

\begin{proof}
On reprend en l'adaptant celle de \cite[ch.~IV, 11.9]{B}. 
Soit $v\in \NF$. Posons $P={_FP_v}$ et $X= {_F\scrX_v}$. ConsidŽrons les $F$-morphismes $$\alpha: G\times V \rightarrow G\times V\quad\hbox{et}\quad \beta : G\times V \rightarrow G/P \times V$$ dŽfinis par 
$\alpha(g,v)=(g,g\cdot v)$ et $\beta(g,v)= (\dot{g},v)$ avec $\dot{g}=gP$. Observons que $\alpha$ est un isomorphisme et que $\beta$ est un morphisme ouvert. Posons 
$W= \beta \circ \alpha (G\times X)$. Puisque $X$ est une $F$-sous-variŽtŽ fermŽe $P$-invariante de $V$, $\beta^{-1}(W)= \alpha(G\times X)$ est une $F$-sous-variŽtŽ fermŽe 
de $G\times V$ et (puisque $\beta$ est ouverte) $W$ est une $F$-sous-variŽtŽ fermŽe de $G/P\times V$. L'immersion fermŽe $\iota:W\hookrightarrow G/P \times V$ et 
la projection naturelle $q: G/P\times V\rightarrow V$ sont des $F$-morphismes propres; donc $\pi= q\circ \iota : W \rightarrow V$ est un $F$-morphisme propre. On 
lui applique \ref{GMB corps valuŽ admissible}\,(ii): $\pi_F(W(F))$ est \TopwhF-fermŽ dans $V(F)$. Il nous reste ˆ identifier $\pi_F(W(F))$. 
Observons que $$W=\{(\dot{g},v)\in G/P\times V\,\vert g^{-1}\cdot v \in X\}$$ 
(la condition $g^{-1}\cdot X$ ne dŽpend que de $\dot{g}$ car $X$ est $P$-invariant). Par consŽquent 
$W(F)$ est l'ensemble des $(\dot{g}\times v)\in G(F)/P(F)\times V(F)$ tels que $g^{-1} \cdot v \in X(F)= \scrX_{F,v}$; et 
$\pi_F(W(F))= G(F)\cdot \scrX_{F,v}= \bsfrX_{F,v}$.  
\end{proof}

Comme dans le cas o $F$ est un corps local, on en dŽduit que:
\begin{itemize}
\item $\NF$ est \TopwhF-fermŽ dans $V(F)$;
\item pour $s\in \mbb{Q}_+$, $\ES{N}_{F,<s}$ est \TopwhF-fermŽ dans $V(F)$;
\item pour $v\in \NF$, la $F$-lame $\scrY_{F,v}$ est \TopwhF-ouverte dans $\scrX_{F,v}$ et 
la $F$-strate $\bsfrY_{F,v}$ est \TopwhF-ouverte dans $\bsfrX_{F,v}$.
\item les $F$-lames, resp. $F$-strates, de $\NF$ sont des sous-ensembles localement fermŽs de $V(F)$ pour \TopwhF.
\end{itemize}

\subsection{Le thŽorme de Seshadri et l'hypothse \ref{hypo bonnes F-strates faible}}\label{Seshadri et hypo Fstratesfaibles} 
On donne dans cette sous-section des conditions suffisantes pour que l'hypothse \ref{hypo bonnes F-strates faible} soit vŽrifiŽe.

\vskip2mm
\textit{\S1.\hskip1mm Le thŽorme de Seshadri. ---} Soit $\scr{G}$ un $\mbb{Z}$-groupe rŽductif de Chevalley-Demazure (\textit{cf.} \cite{Dem}), \textit{i.e.} un $\mbb{Z}$-schŽma en groupes affine et lisse dont toutes les fibres gŽomŽtriques sont des groupes algŽbriques linŽaires rŽductifs connexes au sens de Borel \cite{B}, et soit $\scr{V}$ un $\mbb{Z}$-schŽma affine rŽduit. On dit que 
$\scr{G}$ opre sur $\scr{V}$ si pour toute $\mbb{Z}$-algbre $A$, il existe une application $$\rho_A:\scr{G}(A)\times \scr{V}(A)\rightarrow \scr{V}(A),\,(g,v)\mapsto \rho_A(g,v)=g\cdot_A v$$ 
donnŽe par des polyn™mes en $A$, qui soit fonctorielle en $A$. Si $\scr{V}$ est isomorphe ˆ l'espace affine $\mbb{A}^n_{\mbb{Z}}$ (vu comme un $\mbb{Z}$-schŽma), on dit que cette action est linŽaire ou plus simplement que $\scr{V}$ est un \textit{$\scr{G}$-module}\footnote{Observons que cette notion est plus restrictive que celle de Jantzen \cite[I,2.7]{J1}: dans 
\textit{loc.~cit.}, un \guill{$\scr{G}$-module} est simplement la donnŽe d'un $\mathbb{Z}$-module $M$ et d'une opŽration de $\scr{G}$ sur le $\mathbb{Z}$-foncteur en groupes $M_{\mathrm{a}}$ dŽfini par $M_{\mathrm{a}}(A)= M\otimes_{\mbb{Z}}A$. On se limite ici au cas o $M$ est projectif et de type fini (sur $\mathbb{Z}$), ce qui assure que $M_{\mathrm{a}}$ est un $\mbb{Z}$-schŽma en groupes affine isomorphe ˆ $ \mbb{A}_{\mbb{Z}}^n$ (via le choix d'une $\mbb{Z}$-base de $M$).} si pour 
toute $\mbb{Z}$-algbre $A$ et tout $g\in \scr{G}(A)$, l'application $A^n \rightarrow A^n,\, v \mapsto g \cdot_A v$ est $A$-linŽaire. Dans ce cas le changement de base 
$\mbb{Z}\rightarrow F$ fournit une action $F$-linŽaire du $F$-groupe rŽductif $\scr{G}_F=\scr{G}\times_{\mbb{Z}}F$ sur le $F$-espace affine 
$\scr{V}_F\simeq \mbb{A}_{F}^n$ (vu comme un $F$-schŽma); \textit{i.e.} le $\overline{F}$-espace vectoriel $\scr{V}(\overline{F})= \scr{V}(F)\otimes_F \overline{F}\;(\simeq \smash{\overline{F}}^n)$ est muni d'une structure de $\scr{G}_F$-module dŽfini sur $F$ (o $\scr{G}_F$ est vu comme un groupe algŽbrique linŽaire dŽfini sur $F$, \textit{cf.} \ref{notations et rappels}). 

Le rŽsultat suivant est dž ˆ Seshadri \cite{Se}. Il permet en particulier de transfŽrer la semi-stabilitŽ gŽomŽtrique de la caractŽristique $>0$ ˆ la caractŽristique nulle (et rŽciproquement). 

\begin{theorem}\label{seshadri}
Soit $\scr{G}$ un $\mbb{Z}$-groupe rŽductif de Chevalley-Demazure opŽrant linŽ\-airement sur un espace affine $\scr{V}\simeq \mbb{A}_{\mbb{Z}}^n$ (avec $n\in \mbb{N}^*$) et soit 
$\scr{W}=\mathrm{Spec}(B)$ un sous-schŽma fermŽ $\scr{G}$-invariant de $\scr{V}$. 
\begin{enumerate}
\item[(i)] Soit $w\in \scr{W}(\overline{F})$ un point $(\overline{F},\scr{G}_{\overline{F}})$-semi-stable, \textit{i.e.} tel que la fermeture de 
Zariski de la $\scr{G}(\overline{F})$-orbite de $w$ dans $\scr{W}(\overline{F})$ ne contienne pas $0$. Il existe 
un polyn™me $f\in \mbb{Z}[\scr{V}]^{\scr{G}}$ homogne de degrŽ strictement positif tel que $f(w)\neq 0$.
\item[(ii)] Il existe un sous-schŽma ouvert $\scr{W}^{\mathrm{ss}}$ de $\scr{W}$ tel que pour tout corps (commutatif) $\bs{k}$ algŽbriquement clos, 
$\scr{W}^{\mathrm{ss}}(\bs{k})$ soit l'ensemble des points $(\bs{k},\scr{G}_{\bs{k}})$-semi-stables de $\scr{W}(\bs{k})$. PrŽcisŽment, 
$\scr{W}\smallsetminus \scr{W}^{\mathrm{ss}}$ est le sous-schŽma fermŽ de $\scr{W}$ dŽfini par l'idŽal de $B$ 
engendrŽ par les images des polyn™mes homognes de degrŽ strictement positifs $f\in \mbb{Z}[\scr{V}]^{\scr{G}}$ par le morphisme naturel 
$\mbb{Z}[\scr{V}]\rightarrow B$. 
\end{enumerate}
\end{theorem}

\textit{\S2.\hskip1mm $F$-semi-stabilitŽ et semi-stabilitŽ gŽomŽtrique: le cas o $F$ est local. ---}
Soit $\scr{G}$ un $\mbb{Z}$-groupe rŽductif de Chevalley-Demazure opŽrant linŽairement sur un espace affine $\scr{V}\simeq  \mbb{A}_{\mbb{Z}}^n$. 
En posant $G=\scr{G}_F$ et $V=\scr{V}_F$ (identifiŽ ˆ la $F$-variŽtŽ $\scr{V}(\overline{F})\simeq \smash{\overline{F}}^n$), 
on a le sous-ensemble ${_F\ES{N}}^{\,G}(V,0)\subset V $ des ŽlŽments $(F,G)$-instables de 
$V$; et on a aussi le sous-ensemble
$$\NF^{\,G}(V,0)= \scr{V}(F)\cap {_F\ES{N}}^{G}(V,0)\subset \scr{V}(F)\simeq  F^n\ptf$$ 
Soit $\scr{T}$ un tore $\mbb{Z}$-dŽployŽ maximal de $\scr{G}$, \cad un sous-schŽma en groupes fermŽ de $\scr{G}$ isomorphe ˆ $(\mbb{G}_{\mathrm{m}/\mbb{Z}})^n$ (\textit{cf.} \cite{Dem}). Posons 
$T= \scr{T}_F$; c'est un tore $F$-dŽployŽ maximal de $G$. Le groupe $\check{X}(T)$ s'identifie ˆ $\check{X}(\scr{T})\bydef \mathrm{Hom}_{\mbb{Z}}(\mbb{G}_{\mathrm{m}/\mbb{Z}}, \scr{T})$, \cad que tout co-caractre $\lambda \in \check{X}(\scr{T})$ provient par le changement de 
base $\mbb{Z}\rightarrow F$ d'un ŽlŽment de $\check{X}(\scr{T})$, encore notŽ $\lambda$. De la mme manire, 
le groupe $X(T)$ s'identifie ˆ $X(\scr{T})= \mathrm{Hom}_{\mbb{Z}}(\scr{T}, \mathbb{G}_{\mathrm{m}/\mbb{Z}})$; d'o les identifications  
$\check{X}(T)_{\mbb{Q}}= \check{X}(\scr{T})_{\mbb{Q}}$ et $X(T)_{\mbb{Q}}= X(\scr{T})_{\mbb{Q}}$. Pour $\lambda\in \check{X}(\scr{T})$ et $k\in \mbb{N}^*$, la $F$-sous-variŽtŽ fermŽe $V_\lambda(k)$ de $V$ provient par le changement de base $\mbb{Z}\rightarrow F$ d'un sous-schŽma fermŽ $\scr{V}_\lambda(k)$ de $\scr{V}$ qui est un espace affine ($\simeq \mbb{A}_{\mbb{Z}}^{n'}$). En effet on a la dŽcomposition en somme directe de sous-espace propres $\scr{V}= \bigoplus_{\chi \in \check{X}(T)} \scr{V}_\chi$ pour l'opŽration linŽaire de $\scr{T}$ sur $\scr{V}$; chaque $\scr{V}_\chi$ est un sous-schŽma fermŽ de $\scr{V}$ et un espace affine. En posant 
$$\scr{V}_\lambda(k) = \bigoplus_{\{\chi \in X(\scr{A}_0)\,\vert \, \langle \chi,\lambda\rangle =k \}} \scr{V}_\chi\vg$$ on a bien $V_\lambda(k) = \scr{V}_\lambda(k)\times_{\mbb{Z}} F$. 
Observons que le rang du $\mbb{Z}$-module libre $\scr{V}_{\lambda}(k)(\mbb{Z})$ est Žgal ˆ la dimension du $\overline{F}$-espace vectoriel 
$V_\lambda(k)=V_\lambda(k)(\overline{F})$.  
De mme la $F$-sous-variŽtŽ fermŽe $V_{\lambda,k}= \bigoplus_{i\geq k} V_\lambda(i)$ de $V$ 
provient par le changement de base $\mbb{Z}\rightarrow F$ du sous-schŽma fermŽ 
$$\scr{V}_{\lambda,k} = \bigoplus_{\{\chi \in X(\scr{A}_0)\,\vert \, \langle \chi,\lambda\rangle \geq k \}} \scr{V}_\chi\subset \scr{V}$$ qui est lui aussi un espace affine. Pour $\mu\in \check{X}(\scr{T})_{\mbb{Q}}$, on dŽfinit de la mme manire le sous-schŽma fermŽ $\scr{V}_{\mu,1}$ de $\scr{V}$; on a $V_{\mu,1}= V_{k\mu,k}$ pour tout $k\in \mbb{N}^*$ tel que $\lambda = k\mu$ 
appartienne ˆ $X(\scr{T})$. 

\begin{proposition}\label{F-ss implique gŽo-ss}
Soit $\scr{G}$ un $\mbb{Z}$-groupe rŽductif de Chevalley-Demazure opŽrant linŽairement sur un espace affine $\scr{V}\simeq \mbb{A}_{\mbb{Z}}^n$ (avec $n\in \mbb{N}^*$) et  
soit $F$ un corps local. Posons $G=\scr{G}_F$ et $V= \scr{V}_{F}$. On a 
$$\scr{V}(F) \smallsetminus \NF^{\,G}(V,0) \neq \emptyset  \Rightarrow \scr{V}^{\mathrm{ss}}(\overline{F})\neq \emptyset\ptf$$ 
Autrement dit s'il existe un ŽlŽment de $\scr{V}(F)$ qui soit $(F,G)$-semi-stable, alors il existe un ŽlŽment de $V$ qui soit 
$(\overline{F},G)$-semi-stable.
\end{proposition}

\begin{proof} Posons $G=\scr{G}_F$ et $V= \scr{V}(\overline{F})$. 
Observons tout d'abord que si $p=1$, il n'y a rien ˆ dŽmontrer: si $v\in \scr{V}(F)\cap \ES{N}^{\,G}(V,0)$, 
puisque d'aprs Kempf \cite{K1} (\textit{cf.} \ref{Kempf2}) $\check{X}_F(G)\cap \bs{\Lambda}_v \neq \emptyset$, $v$ est forcŽment $(F,G)$-instable; par consŽquent 
$$v\in \scr{V}(F)\smallsetminus \NF^{\,G}(V,0)\Rightarrow v\in V \smallsetminus \ES{N}^{\,G}(V,0)= \scr{V}^{\mathrm{ss}}(\overline{F}) \ptf$$ 
On va se ramener ˆ ce rŽsultat par la mŽthode des corps proches.

On suppose que $F$ est de caractŽristique $p>1$, \cad que $F\simeq \mbb{F}_q((\varpi))$ avec $q=p^f$. Notons $\mathfrak{o}_F$ l'anneau des entiers de $F$ et $\mathfrak{p}_F$ l'idŽal maximal de $\mathfrak{o}$. 
Fixons un sous-groupe de Borel $\scr{P}_0$ de $\scr{G}$ et un tore $\mbb{Z}$-dŽployŽ maximal $\scr{A}_0$ de $\scr{P}_0$ (\textit{cf.} \cite{Dem}). 
Posons $(P_0,A_0)= (\scr{P}_0\times_{\mbb{Z}}F,\scr{A}_0\times_{\mbb{Z}}F)$; c'est une paire de Borel de $G$ dŽfinie sur $F$. 
Posons aussi $\FN= {_F\ES{N}^G(V,0)}$ et $\NF = V(F)\cap \FN$. 
Le groupe $K= \scr{G}(\mathfrak{o}_F)$ est un 
sous-groupe ouvert compact maximal de $G(F)$ pour {\TopF} et on a la dŽcomposition d'Iwasawa 
$$G(F)=KP_0(F)=P_0(F)K\ptf$$ La paire $(P_0,A_0)$ dŽfinit un sous-ensemble (fini) 
$\bs{\Lambda}_{F,\mathrm{st}}\subset \check{X}(A_0)_{\mbb{Q}}$ de co-caractres virtuels $F$-optimaux en position standard; on a aussi le sous-ensemble 
$\bs{\Lambda}_{F,\mathrm{st}}^*\subset \bs{\Lambda}_{F,\mathrm{st}}$ 
correspondant aux $F$-strates de $\FN$ qui possdent un point $F$-rationnel (\textit{cf.} \ref{la stratification de Hesselink}).  
Pour $v\in \NF$, il existe un $k\in K$ tel que $k\cdot v$ appartienne ˆ une $F$-lame standard de $\NF$. On en dŽduit qu'un ŽlŽment $v\in V(F)$ est dans $\NF$ si et seulement s'il existe un 
$\mu\in\bs{\Lambda}_{F,\mathrm{st}}^*$ tel que $K\cdot v\cap  V_{\mu,1}\neq \emptyset $. 

Soit $v\in V(F)\smallsetminus \NF$. Quitte ˆ multiplier $v$ par un ŽlŽment de $F^\times$, on peut supposer que $v$ appartient ˆ $\scr{V}(\mathfrak{o}_F)\smallsetminus \mathfrak{p}_F \scr{V}(\mathfrak{o}_F)$. 
Puisque $V(F)\smallsetminus \NF$ est \TopF-ouvert dans $V(F)$, il existe un plus petit entier $m\geq 1$ tel que 
$v + \mathfrak{p}_F^{m} \scr{V}(\mathfrak{o}_F)$ soit contenu dans $V(F)\smallsetminus \NF$. Alors $K\cdot (v + \mathfrak{p}_F^{m} \scr{V}(\mathfrak{o}_F))= K\cdot v + \mathfrak{p}_F^{m} \scr{V}(\mathfrak{o}_F)$ 
est contenu dans $V(F)\smallsetminus \NF$. On a donc 
$$\left(K\cdot v + \mathfrak{p}_F^{m} \scr{V}(\mathfrak{o}_F)\right)\cap V_{\mu,1}(F) = \emptyset \quad\hbox{pour tout}\quad \mu \in \check{X}(A_0)_{\mbb{Q}}\ptf$$
Observons que $K\cdot v + \mathfrak{p}_F^{m} \scr{V}(\mathfrak{o}_F)$ est contenu dans $\scr{V}(\mathfrak{o}_F)$ et que pour tout $\mu\in \check{X}(A_0)_{\mbb{Q}}$, on a 
$\scr{V}(\mathfrak{o}_F)\cap V_{\mu,1}(F) = \scr{V}_{\mu,1}(\mathfrak{o}_F)$ avec (rappel) $\scr{V}_{\mu,1} = \bigoplus_{\{\chi \in X(\scr{A}_0)\,\vert \, \langle \chi,\mu\rangle \geq 1 \}} \scr{V}_\chi \ptf$
Notons $\mathfrak{o}_{F,m}$ l'anneau tronquŽ $\mathfrak{o}_F/\mathfrak{p}_F^m$ et 
$$\pi_{F,m}: \scr{V}(\mathfrak{o}_F) \rightarrow \scr{V}(\mathfrak{o}_{F,m})=\scr{V}(\mathfrak{o}_F)/\mathfrak{p}_F^m\scr{V}(\mathfrak{o}_F)$$ la projection naturelle. 
Elle est $K$-Žquivariante pour l'action de $K$ sur $\scr{V}(\mathfrak{o}_{F,m})$ donnŽe par la projection naturelle $K \rightarrow \scr{G}(\mathfrak{o}_{F,m})$. 
On obtient les ŽgalitŽs dans $\scr{V}(\mathfrak{o}_{F,m})$: 
$$\scr{G}(\mathfrak{o}_{F,m})\cdot \pi_{F,m}(v) \cap \scr{V}_{\mu,1}(\mathfrak{o}_{F,m})=\emptyset \quad \hbox{pour tout} \quad \mu\in \check{X}(A_0)_{\mbb{Q}}\ptf $$ 

Soit $F'/\mbb{Q}_p$ une extension finie telle que les anneaux tronquŽs $\mathfrak{o}_{F'\!,m}$ et $\mathfrak{o}_{F,m}$ soient isomorphes (\textit{i.e.} telle que 
$\mathfrak{o}_{F'}/\mathfrak{p}_{F'}= \mbb{F}_q$ et l'indice de ramification de $F'/\mbb{Q}_p$ soit $\geq m$). Fixons un tel isomorphisme d'anneaux $\phi: \mathfrak{o}_{F,m}\buildrel\simeq\over{\longrightarrow} \mathfrak{o}_{F'\!,m}$. 
Pour tout $\mbb{Z}$-schŽma $\scr{S}$, $\phi$ induit par transport de structure une application bijective $\phi_{\scr{S}}: \scr{S}(\mathfrak{o}_{F,m})\rightarrow \scr{S}(\mathfrak{o}_{F'\!,m})$ 
qui est un isomorphisme de groupes si $\scr{S}$ est un $\mbb{Z}$-schŽma en groupes. L'application bijective 
$\phi_{\scr{V}}: \scr{V}(\mathfrak{o}_{F,m})\rightarrow \scr{V}(\mathfrak{o}_{F'\!,m})$ vŽrifie les propriŽtŽs: 
\begin{itemize}
\item $\phi_{\scr{V}}$ est un isomorphisme de $\mathfrak{o}_{F,m}$-modules pour l'action de $\mathfrak{o}_{F,m}$ sur $\scr{V}(\mathfrak{o}_{F'\!,m})$ donnŽe par l'isomorphisme d'anneaux $\phi$;
\item $\phi_{\scr{V}}$ est $\scr{G}(\mathfrak{o}_{F,m})$-Žquivariante pour l'action de $\scr{G}(\mathfrak{o}_{F,m})$ sur $\scr{V}(\mathfrak{o}_{F'\!,m})$ donnŽe par 
l'ismorphisme de groupes $\phi_{\scr{G}}$; 
\item pour tout co-caractre virtuel $\mu\in \check{X}(A_0)_{\mbb{Q}}$, la restriction 
de $\phi_{\scr{V}}$ ˆ $\scr{V}_{\mu,1}$ co\"{\i}ncide avec $\phi_{\scr{V}_{\mu,1}}: \scr{V}_{\mu,1}(\mathfrak{o}_{F,m})\rightarrow \scr{V}_{\mu,1}(\mathfrak{o}_{F'\!,m})$. 
\end{itemize}
Soit $v'\in \scr{V}(\mathfrak{o}_{F'})$ tel que 
$$\pi_{F'\!,m}(v') = \phi_{\scr{V}}\circ \pi_{F,m}(v)\ptf$$ D'aprs ce qui prŽcde, on a 
$$\scr{G}(\mathfrak{o}_{F'\!,m})\cdot \pi_{F'\!,m}(v') \cap \scr{V}_{\mu,1}(\mathfrak{o}_{F'\!,m})=\emptyset \quad \hbox{pour tout} \quad \mu\in\check{X}(A_0)_{\mbb{Q}}\ptf $$ 
On en dŽduit que $\scr{G}(\mathfrak{o}_{F'})\cdot v' \cap \scr{V}_{\mu,1}(F')=\emptyset$ pour tout $\mu\in \check{X}(A_0)_{\mbb{Q}}$, donc \textit{a fortiori} pour tout 
$\mu\in \bs{\Lambda}_{F'\!,\mathrm{st}}^*$, ce qui assure que $v'$ appartient ˆ $\scr{V}(F')\smallsetminus \ES{N}_{F'}$. Par consŽquent 
(d'aprs le premier paragraphe de la dŽmonstration) $v'\notin \ES{N}^{\,\scr{G}_{F'}}(\scr{V}(\overline{\mathbb{Q}}_p),0)$, \textit{i.e.} $v'\in \scr{V}^{\mathrm{ss}}(\overline{\mbb{Q}}_p)$. On a donc prouvŽ que $\scr{V}^{\mathrm{ss}}(\overline{\mbb{Q}}_p)\neq \emptyset$. 
Gr‰ce ˆ \ref{seshadri}, cela entra"ne que $\scr{V}^{\mathrm{ss}}(\overline{F})\neq \emptyset$ (et aussi que $\scr{V}^{\mathrm{ss}}(\overline{\mbb{F}}_p)\neq \emptyset$).
\end{proof}

\vskip2mm
\textit{\S3.\hskip1mm Application ˆ l'hypothse \ref{hypo bonnes F-strates faible}. ---} 
Soit $\scr{G}$ un $\mbb{Z}$-groupe rŽductif de Chevalley-Demazure opŽrant linŽairement sur un espace affine $\scr{V}\simeq \mbb{A}_{\mbb{Z}}^n$. Posons $G=\scr{G}_F$ et 
$V= \scr{V}_F$. Pour $\lambda\in \check{X}_F(G)$, on a le $F$-sous-groupe parabolique $P_\lambda= M_\lambda \ltimes U_\lambda$ de $G$. Il provient par le changement de base 
$\mbb{Z}\rightarrow F$ d'un sous-schŽma en groupes parabolique $\scr{P}_\lambda = \scr{M}_\lambda \rtimes \scr{U}_\lambda $ de $\scr{G}$ (\textit{cf.} \cite{Dem}). Ici $\scr{M}_\lambda$ est le centralisateur schŽmatique de $\lambda$ dans $\scr{G}$; c'est un sous-schŽma en groupes fermŽ de $\scr{G}$ et un $\mbb{Z}$-groupe rŽductif de Chevalley-Demazure,   
qui vŽrifie $\scr{M}_\lambda \times_{\mbb{Z}} F = M_\lambda$. D'aprs \cite[4.3]{CP}, le $F$-sous-groupe fermŽ $M_\lambda^\perp$ de $M_\lambda$  
provient lui aussi par le changement de base $\mbb{Z}\rightarrow F$ d'un sous-schŽma en groupes fermŽ $\scr{M}_\lambda^\perp$ de $\scr{M}_\lambda$; c'est encore un $\mbb{Z}$-groupe rŽductif de Chevalley-Demazure. 
Pour $k\in \mbb{N}^*$, la $F$-sous-variŽtŽ fermŽe $V_\lambda(k)\subset V_{\lambda,k}$ est munie d'une structure de $M_\lambda$-module dŽfini sur $F$. Cette structure provient via le changement de base $\mathbb{Z}\rightarrow F$ d'une opŽration linŽaire de $\scr{M}_\lambda$ sur l'espace affine $\scr{V}_\lambda(k)$. Pour voir cela, choisissons un tore $\mbb{Z}$-dŽployŽ maximal $\scr{T}$ de $\scr{M}_\lambda$. Alors $T=\scr{T}_F$ est un tore $F$-dŽployŽ maximal de $M_\lambda$. Le co-caractre $\lambda$ appartient ˆ $\check{X}(T)=\check{X}(\scr{T})$ et on a 
la dŽcomposition $$\scr{V}_\lambda(k) = \bigoplus_{\{\chi \in X(\scr{T}))\,\vert\, \langle \xi, \lambda \rangle =k\}} \scr{V}_\chi\ptf$$ 
Le sous-schŽma fermŽ $\scr{V}_\lambda(k)\;(\simeq \mbb{A}_{\mbb{Z}}^{n'})$ de $\scr{V}$ est stable sous l'action de $\scr{M}_\lambda$, ce qui le munit d'une opŽration linŽaire de $\scr{M}_\lambda$. Cette structure de $\scr{M}_\lambda$-module sur $\scr{V}_\lambda(k)$ ne dŽpend pas du choix de $\scr{T}$ et elle redonne bien via le changement de base $\mbb{Z}\rightarrow F$ la structure de $M_\lambda$-module sur $V_\lambda(k)=\scr{V}_\lambda(k)\times_{\mbb{Z}}F$. 

\begin{corollary}\label{cor 1 bonnes F-strates}
(Sous les hypothses de \ref{F-ss implique gŽo-ss}.) L'hypothse \ref{hypo bonnes F-strates faible} est vŽrifiŽe: pour tout $v\in \NF^{\, G}(V,0)$, on a $\check{X}_F(G)_{\mbb{Q}}\cap \bs{\Lambda}_v \neq \emptyset$. 
\end{corollary}

\begin{proof}
Soient $v\in \NF^{\, G}(V,0)$. Rappelons que le $F$-saturŽ ${_F\scrX_v}$ de la $F$-lame ${_F\scrY_v}$ de $\FN^G(V,0)$ est muni d'une structure de 
${_FP_v}$-module dŽfini sur $F$. Soient $\lambda\in \Lambda_{F,v}^{\mathrm{opt}}$ et $k=m_v(\lambda)$. On a donc 
${_FP_v}=P_\lambda$ et ${_F\scrX_v}=V_{\lambda,k}$. Posons $\scr{G}'= \scr{M}_\lambda^\perp$ et $\scr{V}'= \scr{V}_\lambda(k)$; posons aussi  
$G'= \scr{G}'_F$ et $V'= \scr{V}'(\overline{F})$. 
D'aprs le critre de Kirwan-Ness rationnel \ref{thm de KN rationnel}, l'ŽlŽment $v_\lambda(k)$ appartient ˆ 
$\scr{V}'(F) \smallsetminus \ES{N}_F^{\,G'}(V',0)$. D'aprs la proposition \ref{F-ss implique gŽo-ss} appliquŽe au couple $(\scr{G}',\scr{V}')$, on a $\scr{V}'^{\mathrm{ss}}(\overline{F})\neq \emptyset$. D'aprs le critre de 
Kirwan-Ness gŽomŽtrique (\textit{cf.} \ref{KNG Tsuji}), il existe un $v'\in V_{\lambda,k}(\overline{F})$ tel que $\lambda \in \Lambda_{v'}^{\mathrm{opt}}$. Posons 
$\mu= \frac{1}{k}\lambda\in \check{X}_F(G)_{\mathbb{Q}}$. On a ${_F\scrX_{v'}}=\scrX_{v'}= V_{\mu,1}= {_F\scrX_v}$. Comme $\bs{q}_F(v') = \bs{q}(v')= \| \mu \| = \bs{q}_F(v)$, on a aussi ${_F\scrY_{v'}}=\scrY_{v'}= {_F\scrY_v}$. 
La $\overline{F}$-lame $\scrY_{v'}$ de $\ES{N}$ est ouverte dans son $\overline{F}$-saturŽ $\scrX_{v'}$. Par suite la $F$-lame ${_F\scrY_v}$ de $\FN$ est ouverte dans son $F$-saturŽ ${_F\scrX_v}$, ce qui 
assure (d'aprs \ref{deux critres bonnes F-lames} et \ref{deux critres bonnes F-lames bis}) que $\mu$ appartient ˆ $\bs{\Lambda}_v$. \end{proof}


\vskip2mm
\textit{\S4.\hskip1mm Variante pour un corps global. ---} 
On peut affaiblir les hypothses et remplacer le corps de base (supposŽ local) par un corps global:

\begin{lemme}\label{cor 3 bonnes F-strates}
La proposition \ref{F-ss implique gŽo-ss} et son corollaire \ref{cor 1 bonnes F-strates} restent vrais si l'on remplace 
le corps de base (supposŽ local) par un corps valuŽ admissible $(F,\nu)$ tel que le complŽtŽ 
$\wh{F}$ de $F$ en $\nu$ soit un corps local vŽrifiant $(\wh{F}^{\mathrm{s\acute{e}p}})^{\mathrm{Aut}_F(\wh{F}^{\mathrm{s\acute{e}p}})}=F$; \textit{e.g.} un corps global $F$ muni d'une place finie $\nu$ 
(\textit{cf.} \ref{sŽparabilitŽ du complŽtŽ}).
\end{lemme}
 
\begin{proof}
Les hypothses sont celles de \ref{F-ss implique gŽo-ss} (ˆ l'exception bien sžr de celle sur le corps de base). 
Soit $\bs{k}$ une une cl™ture algŽbrique de $\wh{F}$. 
Posons $\NF= \NF^{\,G}(V,0)$ et $\ES{N}_{\smash{\wh{F}}}= \ES{N}_{\wh{F}}^{\,\scr{G}_{\wh{F}}}(\scr{V}(\bs{k}),0)$. 
D'aprs \ref{descente sŽparable c™ne F-nilpotent}, on a $\NF = \scr{V}(F)\cap \ES{N}_{\smash{\wh{F}}}$. Par consŽquent si $\scr{V}(\wh{F})= \ES{N}_{\smash{\wh{F}}}$, alors 
$\scr{V}(F)=\scr{V}(F)\cap \ES{N}_{\smash{\wh{F}}}= \NF$. Puisque $\wh{F}$ est un corps local, on a (d'aprs \ref{F-ss implique gŽo-ss} pour $\wh{F}$) 
$$\NF \neq \scr{V}(F)\Rightarrow \ES{N}_{\smash{\wh{F}}} \neq \scr{V}(\wh{F}) \Rightarrow \scr{V}^{\mathrm{ss}}(\bs{k})\neq\emptyset \Rightarrow \scr{V}^{\mathrm{ss}}(\overline{F})\neq \emptyset \ptf$$ 
Cela prouve \ref{F-ss implique gŽo-ss} pour $F$. 

Quant ˆ \ref{cor 1 bonnes F-strates}, pour tout $v\in \NF$, 
on a $\bs{\Lambda}_{F,v}= \check{X}_F(G)_{\mbb{Q}}\cap \bs{\Lambda}_{\smash{\wh{F}},v}$ (d'aprs \ref{descente sŽparable lames}). On en dŽduit que 
si $\check{X}_{\smash{\wh{F}}}(G)_{\mbb{Q}}\cap \bs{\Lambda}_v \neq \emptyset$, alors puisque $\bs{\Lambda}_{\smash{\wh{F}},v}= \check{X}_{\smash{\wh{F}}}(G)_{\mbb{Q}}\cap \bs{\Lambda}_v$, on a 
$\bs{\Lambda}_{F,v}= \check{X}_F(G)_{\mbb{Q}}\cap \bs{\Lambda}_v$. Cela prouve \ref{cor 1 bonnes F-strates} pour $F$ (d'aprs \ref{cor 1 bonnes F-strates} pour $\wh{F}$). 
\end{proof}

\subsection{Stratifications de Hesselink sur 
$\overline{F}$ et sur $\mathbb{C}$.}\label{comparaison stratif Hesselink sur F et sur C} On suppose dans cette sous-section que $p>1$ et $F= \overline{F}$. 
Soit $\scr{G}$ un $\mbb{Z}$-groupe rŽductif de Chevalley-Demazure. On note $\mathrm{Dist}_{\mbb{Z}}(\scr{G})$ la $\mbb{Z}$-algbre des distributions sur $\scr{G}$. 
On introduit ci-aprs la notion de $\scr{G}_F$-module admissible, plus flexible que celle de $\scr{G}$-module utilisŽe en \ref{Seshadri et hypo Fstratesfaibles}. 

\vskip1mm
\textit{\S1.\hskip1mm Modules admissibles. ---} 
Posons $G= \scr{G}_F$ et $G'= \scr{G}_{\mathbb{C}}$. Tout $\scr{G}$-module $\scr{V}\;(\simeq \mbb{A}_{\mbb{Z}}^n)$ est muni d'une structure de $\mathrm{Dist}_{\mbb{Z}}(\scr{G})$-module 
(\textit{cf.} \cite[I, 7.1]{J2}): 
pour toute $\mbb{Z}$-algbre $A$, on a une opŽration de $\mathrm{Dist}_{\mbb{Z}}(\scr{G})\otimes_{\mbb{Z}}A$ sur $\scr{V}(A)$. 
Observons que la $\mbb{Z}$-algbre $\mathrm{Dist}_{\mbb{Z}}(\scr{G})$ est naturellement une sous-algbre de $\mathrm{Dist}_{\mathbb{C}}(G')=\mathrm{Dist}_{\mathbb{Z}}(\scr{G})\otimes_{\mbb{Z}}\mbb{C}$. 
Un $\mathbb{Z}$-rŽseau $M'$ dans un $G'$-module 
$V'$ est dit \textit{admissible} s'il est $\mathrm{Dist}_{\mathbb{Z}}(\scr{G})$-stable et si, pour un tore $\mbb{Z}$-dŽployŽ maximal $\scr{T}$ de $\scr{G}$, le $\mathbb{Z}$-foncteur en groupes $$M'_{\mathrm{a}}:A \mapsto M'\otimes_{\mathbb{Z}} A$$ est muni d'une structure de  
$\scr{T}$-module induisant une structure de $\mathrm{Dist}_{\mbb{Z}}(\scr{T})$-module qui co\"{\i}ncide avec celle dŽduite (par restriction) de l'opŽration de 
$\mathrm{Dist}_{\mbb{Z}}(\scr{G})$\footnote{La $\mbb{Z}$-algbre $\mathrm{Dist}_{\mbb{Z}}(\scr{T})$ est naturellement une sous-algbre de $\mathrm{Dist}_{\mbb{Z}}(\scr{G})$.}; 
$M'_{\mathrm{a}}$ est alors un $(\mathrm{Dist}_{\mbb{Z}}(\scr{G}),\scr{T})$-module au sens de Jantzen \cite[II, 1.20]{J2}. Observons que la seconde condition ne dŽpend pas de $\scr{T}$: si elle est vŽrifiŽe pour un tore $\mbb{Z}$-dŽployŽ maximal $\scr{T}$ de $\scr{G}$, elle l'est pour tout tore $\mbb{Z}$-dŽployŽ maximal de $\scr{G}$ (par conjugaison dans $\scr{G}(\mbb{Z})$). Si $M'$ est un $\mbb{Z}$-rŽseau admissible dans $V'$, alors le $\mathbb{Z}$-foncteur en groupes $M'_{\mathrm{a}}$ est muni d'une structure de 
$\scr{G}$-module qui ne dŽpend que de sa structure de $(\mathrm{Dist}_{\mbb{Z}}(\scr{G}),\scr{T})$-module (\textit{cf.} \textit{loc.~cit.}). 
RŽciproquement si $\scr{G}$ opre linŽairement sur un espace affine $\scr{V}\;(\simeq \mathbb{A}_{\mathbb{Z}}^n)$, alors 
$M' = \scr{V}(\mathbb{Z})\simeq \mathbb{Z}^n$ est un $\mathbb{Z}$-rŽseau admissible dans le $G'$-module $V'= \scr{V}(\mathbb{C}) \simeq  \mathbb{C}^n$. 

Un $G$-module $V$ est dit \textit{$(\scr{G},F)$-admissible}, ou plus simplement \textit{admissible}, s'il existe un 
$G'$-module $V'$ et un $\mathbb{Z}$-rŽseau \textit{admissible} $M'$ dans $V'$ tel que $$V\simeq  M'\otimes_{\mathbb{Z}}F\;(= M'/pM'\otimes_{\mathbb{F}_p} F)$$ en tant que $G$-module. Par dŽfinition, tout $G$-module admissible $V$ provient par le changement de base $\mbb{Z}\rightarrow F$ d'un $\scr{G}$-module $\scr{V}$ (en l'occurrence $\scr{V}\simeq M'_{\mathrm{a}}$). 
\begin{remark}
\textup{Soit $V'$ un $G'$-module. Deux $\mbb{Z}$-rŽseaux admissibles $M'_1$ et $M'_2$ dans $V'$ donnent lieu ˆ des $G$-modules (admissibles) 
$V_1= M'_1\otimes_{\mbb{Z}}F$ et $V_2= M'_2\otimes_{\mbb{Z}}F$ qui, en gŽnŽral, ne sont pas isomorphes. 
}
\end{remark}

\textit{\S2.\hskip1mm Le thŽorme de comparaison de Clarke-Premet \cite{CP}. ---}
Fixons une paire de Borel $(P_0,A_0)$ de $G$ de la forme 
$(\scr{P}_0\times_{\mathbb{Z}}F,\scr{A}_0 \times_{\mathbb{Z}} F)$ pour une paire de Borel $(\scr{P}_0,\scr{A}_0)$ de $\scr{G}$ (\textit{cf.} \cite{Dem}). Posons $(P'_0,A'_0)= (\scr{P}_0\times_{\mathbb{Z}}\mathbb{C},\scr{A}_0 \times_{\mathbb{Z}} \mathbb{C})$; c'est une paire de Borel de $G'$. 
On a les identifications $\check{X}(A_0)= \check{X}(\scr{A}_0)=\check{X}(A'_0)$; d'o les identifications 
$\check{X}(A_0)_{\mathbb{Q}}=\check{X}(\scr{A}_0)_{\mathbb{Q}}= \check{X}(A'_0)_{\mathbb{Q}}$. Notons $W$ le groupe de Weil $W^\scr{G}(\scr{A}_0)=N^{\scr{G}}(\scr{A}_0)/\scr{A}_0$. On a les identification $W^G(A_0)=W= W^{G'}(A'_0)$. Fixons aussi une norme $W$-invariante $\| \, \|:\check{X}(\scr{A}_0)\rightarrow \mbb{R} _+$ sur 
$\check{X}(\scr{A}_0)$. Elle dŽfinit une norme $G$-invariante sur $\check{X}(G)$ et une norme $G'$-invariante sur $\check{X}(G')$, encore notŽes $\| \, \|$.

Soit $M'$ un $\mathbb{Z}$-rŽseau admissible dans un $G'$-module $V'$. Il dŽfinit un $G$-module 
$V= M'\otimes_{\mbb{Z}}F$ et un $\scr{G}$-module $\scr{V}=(M')_{\mathrm{a}}$. Pour $\lambda\in \check{X}(\scr{A}_0)$ et $k\in \mbb{N}^*$, on a les sous-schŽmas fermŽs $\scr{V}_\lambda(k)\subset \scr{V}_{\lambda,k}\subset \scr{V}$ qui redonnent via les changements de base $\mbb{Z}\rightarrow F$ et $\mbb{Z}\rightarrow \mbb{C}$ les sous-variŽtŽs fermŽes $V_{\lambda}(k)\subset V_{\lambda,k}\subset V$ et $V'_{\lambda}(k)\subset V'_{\lambda,k}\subset V'$. Observons que 
$$\dim_F(V_{\lambda,k})= \dim_{\mbb{C}}(V'_{\lambda,k}) \quad \hbox{et}\quad \dim_F(V_{\lambda}(k))= \dim_{\mbb{C}}(V'_{\lambda}(k)) \ptf$$
Posons $$\ESN={_F\ESN^G}(V,0)\quad \hbox{et} \quad \ESN'= {_{\mbb{C}}\ESN^{G'}}(V'\!,0) \vg$$ 
$$\bs{\Lambda}_{\mathrm{st}}= \bs{\Lambda}_{F,\mathrm{st}}^G(V,0) \quad \hbox{et}\quad  \bs{\Lambda}'_{\mathrm{st}}= \bs{\Lambda}_{\mbb{C},\mathrm{st}}^{G'}(V'\!,0)\ptf$$ 
L'indice \guill{st} pour \guill{standard} se rŽfre bien sžr aux paires paraboliques $(P_0,A_0)$ et $(P'_0,A'_0)$ fixŽes plus haut. Observons qu'un co-caractre $\lambda \in \check{X}(\scr{A}_0)$ 
est en position standard pour $F$ si et seulement s'il l'est pour $\mbb{C}$: on a $P_0 \subset P_\lambda$ si et seulement si $P'_0\subset P'_\lambda$. 
Pour $\mu\in \bs{\Lambda}_{\mathrm{st}}$, notons $$\bsfrY(\mu)= \{v\in \ESN\,\vert\, \mu \in G\bullet \bs{\Lambda}_{F,v}  \}$$ 
la $F$-strate de $\ESN$ associŽe ˆ la $G$-orbite de $\mu$ dans $\check{X}(G)_{\mbb{Q}}$. Pour $\mu\in \bs{\Lambda}'_{\mathrm{st}}$, on dŽfinit 
$\bsfrY'(\mu)\subset \ESN'$ de la mme manire. 

\vskip1mm
Le thŽorme suivant est dž ˆ Clarke-Premet \cite[6.1, theorem]{CP}\footnote{Dans \textit{loc.~cit.}, le rŽseau $M'$ dans le $G'$-module $V'$ est seulement supposŽ $\mathrm{Dist}_{\mbb{Z}}(\scr{G})$-stable. Cela assure que pour tout $\mu \in \check{X}(\scr{A}_0)_{\mbb{Q}}$ le sous-schŽma fermŽ $\scr{V}_\mu(1)$ est $\mathrm{Dist}_{\mbb{Z}}(\scr{M}_\mu)$-stable mais, me semble-t-il, est insuffisant pour pouvoir appliquer le thŽorme de Seshadri (\ref{seshadri}). }. 

\begin {theorem}\label{F versus C}
\begin{enumerate}
\item[(i)] On a l'ŽgalitŽ $\bs{\Lambda}_{\mathrm{st}}= \bs{\Lambda}'_{\mathrm{st}}$.
\item[(ii)] Les strates de $\ESN$ sont paramŽtrisŽes par celles de $\ESN'$: pour $\mu \in \bs{\Lambda}_{\mathrm{st}}= \bs{\Lambda}'_{\mathrm{st}}$, la $F$-strate 
$\bsfrY(\mu)$ de $\ESN$ correspond ˆ la $\mbb{C}$-strate $\bsfrY'(\mu)$ de $\ESN'$. 
\item[(iii)] La paramŽtrisation (ii) prŽserve la dimension des strates: pour $\mu \in \bs{\Lambda}_{\mathrm{st}}= \bs{\Lambda}'_{\mathrm{st}}$, on a 
$\dim_F(\bsfrY(\mu))= \dim_{\mbb{C}}(\bsfrY'(\mu))$. 
\end{enumerate}
\end{theorem}

\begin{proof}
On la reprend brivement car le point (i) de \cite[6.1, theorem]{CP} n'est pas correct (prŽcisŽment, ce n'est pas l'ŽnoncŽ que les auteurs prouvent): un sous-ensemble saturŽ de $V$, resp. $V'$, n'est pas forcŽment la fermeture d'une lame\footnote{Mme pour l'action adjointe de $G$ sur $V= \mathrm{Lie}(G)$; \textit{e.g.} l'algbre de Lie du radical unipotent d'un sous-groupe parabolique propre de $\mathrm{GL}_3$.}.

Soit $\mu \in \check{X}(\scr{A}_0)_{\mbb{Q}}$ en bonne position par rapport ˆ $\ES{P}_0$, \textit{i.e.} tel que $\ES{P}_0 \subset \ES{P}_\mu$. Soit $k$ le plus petit entier $\geq 1$ tel que $\lambda = k\mu$ appartienne ˆ $\check{X}(\scr{A}_0)$. 
Le $\mbb{Z}$-groupe rŽductif de Chevalley-Demazure $\scr{M}_\lambda$ opre linŽairement sur l'espace affine $\scr{V}_\mu(1)=\scr{V}_\lambda(k)$. Notons $\scr{V}_\mu(1)^{\mathrm{ss}}$ le sous-schŽma ouvert de $\scr{V}_\mu(1)$ dŽfini par le thŽorme \ref{seshadri}\,(ii) appliquŽ au couple $(\scr{M}_\lambda^\perp,\scr{V}_\mu(1))$. Puisque $\scr{V}_\mu(1)^{\mathrm{ss}}(F)\neq \emptyset$ si et seulement si $\scr{V}_{\mu}(1)^{\mathrm{ss}}(\mbb{C})\neq \emptyset$, d'aprs le critre de Kirwann-Ness gŽomŽtrique (\textit{cf.} \ref{KNG Tsuji}), on a $\mu \in \bs{\Lambda}_{\mathrm{st}}\Leftrightarrow \mu \in   \bs{\Lambda}'_{\mathrm{st}}$. Cela prouve (i) et (ii). Pour (iii), on renvoie au point (iv) de \cite[6.1, theorem]{CP}.
\end{proof}

\begin{remark}
\textup{
\begin{enumerate}
\item[(i)] 
Posons $\mathfrak{S}_{\mathrm{st}}= \{\scr{V}_{\mu,1}(\mbb{Z})\,\vert \, \mu \in \bs{\Lambda}_{\mathrm{st}}=\bs{\Lambda}'_{\mathrm{st}}\}$; c'est un ensemble fini de $\mbb{Z}$-modules libres de rang fini. Notons $\mathcal{S}_{\mathrm{st}}$ l'ensemble (fini) des sous-ensembles saturŽs de $\ES{N}$ de la forme $\scr{X}(\scr{Y})$ pour une lame standard $\scr{Y}$ de $\ES{N}$. En remplaant $F$ par $\mbb{C}$, on dŽfinit de la mme manire l'ensemble $\mathcal{S}'_{\mathrm{st}}$. D'aprs \ref{F versus C}, on a 
$$\mathcal{S}_{\mathrm{st}}= \{S\otimes_{\mbb{Z}}F\,\vert \, S \in \mathfrak{S}_{\mathrm{st}}\} \quad \hbox{et} \quad \mathcal{S}'_{\mathrm{st}}= \{S\otimes_{\mbb{Z}}\mbb{C}\,\vert \, S \in \mathfrak{S}_{\mathrm{st}}\}\,;$$ et pour $S\in \mathfrak{S}_{\mathrm{st}}$, on a 
$$\bs{\Lambda}_{F,S\otimes_{\mbb{Z}}F} \cap \check{X}(\scr{A}_0)_{\mbb{Q}}=\{\mu\} = \bs{\Lambda}_{\mbb{C},S\otimes_{\mbb{Z}}\mbb{C}} \cap \check{X}(\scr{A}_0)_{\mbb{Q}}\ptf$$ 
C'est la variante des points (i) et (ii) de \cite[6.1, theorem]{CP} pour les fermetures des lames standard. On retrouve l'ŽnoncŽ de \textit{loc.~cit.} en posant 
$\bs{\Lambda}_{\mathrm{s-st}}= W\bullet \bs{\Lambda}_{\mathrm{st}}$ et $\bs{\Lambda}'_{\mathrm{s-st}}= W\bullet \bs{\Lambda}'_{\mathrm{st}}$: 
on a $\bs{\Lambda}_{\mathrm{s-st}}=\bs{\Lambda}'_{\mathrm{s-st}}$ et l'ensemble $\mathfrak{S}_{\mathrm{s-st}}= W\bullet \mathfrak{S}_{\mathrm{st}} \;(=\{\scr{V}_{\mu,1}(\mbb{Z})\,\vert \, \mu \in \bs{\Lambda}_{\mathrm{s-st}}=\bs{\Lambda}'_{\mathrm{s-st}}\})$ paramŽtrise comme ci-dessus les saturŽs des lames semi-standard de $\ES{N}$, resp. $\ES{N}'$.  
\item[(ii)] Soient $M'_1$ et $M'_2$ deux rŽseaux admissibles dans $V'$. Mme si les $G$-modules $V_1 = M'_1\otimes_{\mbb{Z}}F$ et $V_2= M'_2\otimes_{\mbb{Z}}\mbb{C}$ ne sont pas isomorphes, on a une bijection canonique entre les $F$-strates de ${_F\ESN}^G(V_1,0)$ et celles de 
${_F\ESN}^G(V_2,0)$.
\item[(iii)] On peut se demander si la correspondance bijective de \ref{F versus C}\,(ii) prŽserve la relation d'ordre partiel sur les strates 
donnŽe par $\bsfrY_1 \preccurlyeq  \bsfrY_2 \Leftrightarrow \bsfrY_1 \subset \bsfrX(\bsfrY_2)$. La question est: pour $\mu_1,\, \mu_2 \in \bs{\Lambda}_{\mathrm{st}}=\bs{\Lambda}'_{\mathrm{st}}$, a-t-on $$\bsfrY(\mu_1) \subset \bsfrX(\mu_2)\Leftrightarrow \bsfrY'(\mu_1) \subset \bsfrX'(\mu_2)?$$ Une question similaire mais diffŽrente est: 
a-t-on $$\bsfrY(\mu_1) \cap \bsfrX(\mu_2)\neq \emptyset \Leftrightarrow \bsfrY'(\mu_1) \cap \bsfrX'(\mu_2)\neq \emptyset ?$$ 
Observons que pour que le 
$G$-module $V=M'\otimes_{\mbb{Z}}F$ vŽrifie l'hypothse sur le bord \ref{hypo sur le bord} (pour $F=\overline{F}$), il suffit que le $G'$-module $V'$ la vŽrifie (pour $\mbb{C}$) et que les deux implications suivantes soient vŽrifiŽes (pour tous $\mu_1,\, \mu_2 \in \bs{\Lambda}_{\mathrm{st}}=\bs{\Lambda}'_{\mathrm{st}}$): 
$$\bsfrY(\mu_1) \cap \bsfrX(\mu_2)\neq \emptyset \Rightarrow \bsfrY'(\mu_1) \cap \bsfrX'(\mu_2)\neq \emptyset \vg$$
$$\bsfrY'(\mu_1) \subset \bsfrX'(\mu_2)\Rightarrow \bsfrY(\mu_1) \subset \bsfrX(\mu_2)\ptf$$
\end{enumerate}
}\end{remark}

\subsection{RŽcapitulatif des hypothses dans la section \ref{la theorie de KRH rat}}\label{sur les hypothses} Pour une $G$-variŽtŽ pointŽe $(V,e_V)$ dŽfinie sur $F$ avec $e_V\in V(F)$, on a 
introduit plusieurs hypothses qui, si elles sont vŽrifiŽes, permettent d'obtenir des propriŽtŽs intŽressantes. On rappelle ici ces hypothses et les propriŽtŽs qui en dŽcoulent. 

Tout d'abord une hypothse de rŽgularitŽ:
\begin{enumerate}[leftmargin=1.3cm]\item[(\ref{hyp reg})]le point-base $e_V$ est \textit{rŽgulier} (\textit{i.e.} non singulier) dans $V$.\end{enumerate}
L'hypothse \ref{hyp reg} assure que pour tout $v\in \FN$, le $F$-saturŽ ${_F\scrX_v}$ de la $F$-lame ${_F\scrY_v}$ 
est $F$-isomorphe ˆ son espace tangent $T_{e_V}({_F\scrX_v})$; en particulier c'est un $F$-espace affine (\textit{i.e.} ${_F\scrX_v}\simeq_F \mbb{A}_F^n$) et donc 
une $F$-variŽtŽ irrŽductible. 

Ensuite une hypothse simplificatrice:
\begin{enumerate}[leftmargin=1.3cm]
\item[(\ref{hypo bonnes F-strates})] pour tout $v\in \FN$, on a $\check{X}_{F}(G)_{\mbb{Q}}\cap \bs{\Lambda}_v \neq \emptyset$. 
\end{enumerate}
Et aussi celle (moins forte):
\begin{enumerate}[leftmargin=1.15cm]
\item[(\ref{hypo bonnes F-strates faible})] pour tout $v\in \NF$, on a $\check{X}_{F}(G)_{\mbb{Q}}\cap \bs{\Lambda}_v \neq \emptyset$. 
\end{enumerate}
L'hypothse \ref{hypo bonnes F-strates}, resp. \ref{hypo bonnes F-strates faible}, assure que pour tout $v\in \FN$, resp. tout $v\in \NF$, la $F$-lame ${_F\scrY_v}$ est ouverte dans son $F$-saturŽ ${_F\scrX_v}$; si de plus l'hypothse 
\ref{hyp reg} est vŽrifiŽe, cela entra"ne que la $F$-lame ${_F\scrY_v}$ est dense dans la $F$-variŽtŽ irrŽductible ${_F\scrX_v}$. Si $F$ est infini, 
les hypothses \ref{hyp reg} et \ref{hypo bonnes F-strates} assurent que toute $F$-lame, ou ce qui revient au mme toute $F$-strate, 
de $\FN$ possde un point $F$-rationnel. D'autre part les hypothses \ref{hyp reg} et \ref{hypo bonnes F-strates faible} assurent que l'application $[\ES{N}]_F\rightarrow [\ES{N}_F],\, \bsfrY \mapsto \ES{N}_F \cap \bsfrY$ est bijective. 
Dans le cas o $V$ est un $G$-module dŽfini sur $F$, l'hypothse \ref{hypo bonnes F-strates}, resp. \ref{hypo bonnes F-strates faible}, assure que pour tout $v\in \FN$, resp. tout $v\in \NF$, si $\lambda\in \Lambda_{F,v}^{\mathrm{opt}}$ et $k=m_v(\lambda)$, 
la sous-variŽtŽ (ouverte) $V_\lambda(k) \smallsetminus \ES{N}^{M_\lambda^\perp}(V_\lambda(k),0)\subset V_\lambda(k)$ est non vide. 

Enfin l'hypothse sur le bord:
\begin{enumerate}[leftmargin=1.15cm]
\item[(\ref{hypo sur le bord})] pour tout $v\in \FN$, le bord ${_F\bsfrX_v}\smallsetminus {_F\bsfrY_v}$ est rŽunion (finie) de $F$-strates de $\FN$; en d'autres termes pour tout $v'\in {_F\bsfrX_v}\smallsetminus {_F\bsfrY_v}$, 
on a ${_F\bsfrY_{v'}}\subset {_F\bsfrX_v}$. 
\end{enumerate}
Et aussi celle (moins forte):
\begin{enumerate}[leftmargin=1.3cm]
\item[(\ref{hypo sur le bord faible})] pour tout $v\in \NF$, le bord $\bsfrX_{F,v}\smallsetminus \bsfrY_{F,v}$ est rŽunion (finie) de $F$-strates de $\NF$; en d'autres termes pour tout $v'\in \bsfrX_{F,v}\smallsetminus \bsfrY_{F,v}$, 
on a $\bsfrY_{F,v'}\subset \bsfrX_{F,v}$. 
\end{enumerate}
Si $F$ est infini et si les hypothses \ref{hyp reg} et \ref {hypo bonnes F-strates} sont vŽrifiŽes, alors les hypothses \ref{hypo sur le bord} et \ref{hypo sur le bord faible} sont Žquivalentes. D'autre part si les hypothses \ref{hyp reg} et \ref{hypo bonnes F-strates faible} sont vŽrifiŽes, et si l'hypothse \ref{hypo sur le bord} pour $F=\overline{F}$ est vŽrifiŽe, alors la variante 
suivante de \ref{hypo sur le bord faible} l'est aussi (o pour $v\in \NF$, on a posŽ $\bsfrX^{\scriptscriptstyle +}_{F,v}=\NF \cap \bsfrX_v$): 
\begin{enumerate}[leftmargin=1.3cm]
\item[(\ref{hypo sur le bord faible +})] pour tout $v\in \NF$, l'ensemble $\bsfrX^{\scriptscriptstyle +}_{F,v}\smallsetminus \bsfrY_{F,v}$ est rŽunion (finie) de $F$-strates de $\NF$; en d'autres termes pour tout $v'\in \bsfrX_{F,v}^{\scriptscriptstyle +}\smallsetminus \bsfrY_{F,v}$, 
on a $\bsfrY_{F,v'}\subset \bsfrX_{F,v}^{\scriptscriptstyle +}$. 
\end{enumerate}

\section{Le cas de la variŽtŽ unipotente}\label{le cas de la variŽtŽ U}
Les hypothses sont celles de la section \ref{la theorie de KRH rat}. On se limite dŽsormais ˆ l'action de $G$ sur lui-mme par conjugaison (\cad que $V=G$ et $e_V=e_G=1$) 
donnŽe par $$g\bullet x = gxg^{-1}\ptf$$

\subsection{Les ensembles $\UU$, $\FU$ et $\UF$} On a diffŽrentes notions d'ŽlŽments unipotents de $G$, resp. $G(F)$.

\vskip2mm
\textit{\S1.\hskip1mm Les unipotents gŽomŽtriques. ---} 
Le groupe $G$ est une $G$-variŽtŽ pointŽe pour l'action par conjugaison, lisse et dŽfinie sur $F$. On note\index{U@$\UU$} $\UU=\UU^G$ la variŽtŽ prŽcŽdemment notŽe $\ES{N}^G(G,1)$. 
C'est la sous-variŽtŽ fermŽe de $G$, dŽfinie sur $F$\footnote{Puisque $\UU= G \cdot \wt{U}_0$ o $\wt{U}_0$ est le radical unipotent d'un sous-groupe de 
Borel de $G$ dŽfini sur $F^{\mathrm{s\acute{e}p}}$, la variŽtŽ $\UU$ est dŽfinie sur $F^{\mathrm{s\acute{e}p}}$; donc sur $F$ d'aprs le critre galoisien (\ref{critre galoisien}).} et $G$-invariante, formŽe des ŽlŽments $x\in G$ 
tels que $\ES{F}_x=\{1\}$; o $\ES{F}_x$ est l'unique $G$-orbite fermŽe contenue dans la fermeture de Zariski de la $G$-orbite $\ES{O}_x=G\bullet x$. 
D'aprs le critre d'instabilitŽ de Hilbert-Mumford, un ŽlŽment de $ G$ est dans $\UU$ si et seulement s'il est dans le radical unipotent 
d'un sous-groupe de Borel de $G$. 

\vskip2mm
\textit{\S2.\hskip1mm Les \guill{vrais} unipotents rationnels. ---}
On note $\FU = {_F\UU^G}$ l'ensemble prŽcŽdemment notŽ $\FN^{G}(G,1)$, \cad le  
sous-ensemble de $\UU$ formŽ des ŽlŽments qui sont $(F,G)$-instables. Un ŽlŽment $u\in G$ est dans $\FU$ si et seulement s'il est dans le radical unipotent d'un $F$-sous-groupe parabolique 
de $G$. Enfin on note $\UF=\UF^G$\index{UF@$\FU$, $\UF$} l'ensemble des (vrais) ŽlŽments unipotents de $G(F)$: $$\UF= G(F)\cap \FU\ptf$$

Pour $u\in \FU$, on a dŽfini les sous-ensembles 
$${_F\scrY_u}\subset {_F\scrX_u}= G_{\mu,1}\subset U_\mu \subset \FU\quad \hbox{avec} \quad \mu \in \bs{\Lambda}_{F,u}\vg$$
$${_F\bsfrY_u}= G(F)\bullet {_F\scrY_u} \subset {_F\bsfrX_u} = G(F) \bullet {_F\scrX_u}\subset \FU\ptf$$ 
Rappelons que ces ensembles sont dŽfinis via le choix d'une $F$-norme $G$-invariante $\|\,\|$ sur $\check{X}(G)$ mais 
qu'ils ne dŽpendent pas de ce choix (\textit{cf.} \ref{indep norme}). En revanche la fonction $\bs{q}_F: {_F\UU} \rightarrow \mbb{R}_+$ 
dŽpend bien sžr du choix de $\|\,\|$. Les ensembles ${_F\scrY_u}$ et ${_F\bsfrY_v}$ sont respectivement les $F$-lames et les $F$-strates de $\FU$; 
on les appellera aussi $F$-lames unipotentes et $F$-strates unipotentes de $G$.

L'hypothses \ref{hyp reg} est vŽrifiŽe ($e_G=1$ est rŽgulier). 

Pour $\mu\in \check{X}_F(G)$ et $r\in \mbb{Q}_{>0}$, la variŽtŽ $G_{\mu,r}$ est un 
$F$-sous-groupe unipotent connexe de $G$. Cela s'applique en particulier au $F$-saturŽ ${_F\scrX_u}$ d'une $F$-lame ${_F\scrY_u}$ de $\FU$ (avec $u\in \FU$): on a ${_F\scrX_u} = G_{\mu,1}$ avec $\mu\in \bs{\Lambda}_{F,u}$. 

Pour $u\in \UF$, on a notŽ $\scrY_{F,u}$, resp. $\scrX_{F,u}$, $\bsfrY_{F,u}$, $\bsfrX_{F,u}$ l'intersection de 
${_F\scrY_u}$, resp. ${_F\scrX_u}$, ${_F\bsfrY_u}$, ${_F\bsfrX_u}$ avec $G(F)$, \cad avec $\UF$. 
Les ensembles $\scrY_{F,u}$ et $\bsfrY_{F,u}$ sont respectivement 
les $F$-lames et les $F$-strates de $\UF$; on les appellera aussi $F$-lames unipotentes et $F$-strates unipotentes de $G(F)$. 

\vskip2mm
\textit{\S3.\hskip1mm Passage au groupe dŽrivŽ. ---} On note avec un exposant \guill{der} les objets dŽfinis comme ci-dessus en rempla\c{c}ant $G$ par 
son groupe dŽrivŽ $G_\mathrm{der}$. La $F$-sous-variŽtŽ fermŽe $\UU^\mathrm{der}= \UU^{G_\mathrm{der}}$ de $G_{\mathrm{der}}$ co\"{\i}ncide avec $\UU$ et 
le sous-ensemble ${_F\UU}^{\mathrm{der}} ={_F \UU}^{G_{\mathrm{der}}}$ de $\UU$ 
co\"{\i}ncide avec $\FU$. La $F$-norme $G$-invariante $\|\, \|$ sur $\check{X}(G)$ 
dŽfinit par restriction une $F$-norme $G_{\mathrm{der}}$-invariante sur $\check{X}_F(G_{\mathrm{der}})$ et donc une 
fonction $\bs{q}_F^\mathrm{der}:\FU\rightarrow \mbb{R}_+$ qui co\"{\i}ncide avec $\bs{q}_F$. 

\begin{lemme}\label{le cas der}
Soit $u\in \FU$. 
\begin{enumerate}
\item[(i)] $\bs{\Lambda}_{F,u}^{\mathrm{der}} = \bs{\Lambda}_{F,u}$ et $\smash{_FP_u}^{\!\!\rm der}= G_\mathrm{der} \cap {_FP_u}$.
\item[(ii)] $\smash{_F\scrX_u}^{\!\!\mathrm{der}}= {_F\scrX_u}$ et $\smash{_F\scrY_u}^{\!\!\mathrm{der}} ={_F\scrY_u}$. 
\item[(iii)] ${_F\bsfrX_u}^{\!\!\!\mathrm{der} }= {_F\bsfrX_u}$ et $\smash{_F\bsfrY_u}^{\!\!\!\mathrm{der}} = {_F\bsfrY_u}$.
\end{enumerate}
\end{lemme}

\begin{proof}
D'aprs \cite[7.3]{H1}, l'ensemble $\bs{\Lambda}_{F,u}$ est contenu dans $\check{X}(G^\mathrm{der})_{\mbb{Q}}$. Cela prouve (i) et (ii). 
Si $S$ est un tore $F$-dŽployŽ maximal de ${_FP_u}$, on a l'ŽgalitŽ $$G(F)= G_\mathrm{der}(F)G^S(F)$$ o $G^S$ est 
le centralisateur de $S$ dans $G$. Comme $G^S(F)\subset P_{F,u}$, le point (iii) est une consŽquence de (ii). 
\end{proof}

\subsection{Morphismes et optimalitŽ}\label{morphismes et optimalitŽ} 
Soit $H$ un groupe rŽductif connexe dŽfini sur $F$. On note avec un exposant $H$ les objets dŽfinis comme prŽcŽdemment en remplaant $G$ par $H$. 
Ainsi $\UU_F^H= H(F)\cap \FU^H$ est l'ensemble des (vrais) ŽlŽments unipotents de $H(F)$. 
Pour un $F$-morphisme de groupes $ f: G \rightarrow H$, resp. $f: H \rightarrow G$, on peut se demander ce que devient la thŽorie de l'optimalitŽ ˆ travers $f$. 
Si $H=G^\mathrm{der}$ et $f: H \rightarrow G$ est l'immersion fermŽe naturelle (donnŽe par l'inclusion), on a vu que la rŽponse est trs simple (\ref{le cas der}). 
On se limite ici aux deux cas particuliers suivants: le cas d'un sous-groupe \textit{critique} et le cas d'un morphisme \textit{spŽcial}. Les $F$-morphismes 
naturels $G\rightarrow G_{\rm ad}$ et $G_{\rm sc}\rightarrow G$ sont spŽciaux.

\vskip2mm
\textit{\S1.\hskip1mm Le cas d'un sous-groupe critique \cite[9.4]{H1}. ---} Soit $H=G^S$ le centralisateur dans $G$ d'un tore $S\subset G$. 
C'est un sous-groupe fermŽ rŽductif connexe de $G$. Un tel sous-groupe $H\subset G$ est dit \textit{critique}\footnote{Typiquement, un facteur de Levi de 
$G$ est un sous-groupe critique.}. Si de plus $S$ est dŽfini sur $F$, ce que l'on suppose, alors $H$ l'est aussi. 
La norme $\|\,\|$ dŽfinie par restriction une $F$-norme $H$-invariante 
sur $\check{X}_F(H)$ et donc une fonction $\bs{q}_F^H: {_F\UU^H} \rightarrow \mbb{R}_+$. On a clairement les inclusions $${_F\UU^H \subset \FU} \quad \hbox{et}\quad \UU^H_F\subset \UU_F\ptf$$
D'aprs \cite[9.4]{H1}, on a le

\begin{lemme}\label{le cas critique} 
Soient $S\subset G$ un $F$-tore, $H=G^S$ et $u\in {_F\UU^H}$.
\begin{enumerate}
\item[(i)] $\bs{\Lambda}_{F,u}^{H} = \check{X}_F(H)_{\mbb{Q}} \cap \bs{\Lambda}_{F,u}$, $\bs{q}_F^H(u)= \bs{q}_F(u)$ et 
$\smash{_FP_u}^{\!\!\!H} = H \cap {_FP_u}$;
\item[(ii)] $\smash{_F\scrX_u}^{\!\!\!H} = H \cap {_F\scrX_u}$ et $\smash{_F\scrY_u}^{\!\!H} = H \cap {_F\scrY_u}$; 
\item[(iii)] $\smash{_F\bsfrX_u}^{\!\!\!\!H} \subset H \cap \bsfrX_{F,u}$ et $\smash{_F\bsfrY_u}^{\!\!\!\!H} \subset  H \cap {_F\bsfrY_u}$.
\end{enumerate}
\end{lemme}

\begin{proof}
Puisque $H=G^S$, le $F$-tore $S$ est contenu dans le centralisateur $G^u$ de $u$ dans $G$. 
Or d'aprs \ref{cor de HMrat1}\,(iii) et \ref{BMRT(4.7)}, on a 
$$G^u(F^\mathrm{s\acute{e}p}) \subset P_{F^\mathrm{s\acute{e}p},u}={_FP_u}(F^\mathrm{s\acute{e}p})\ptf$$ 
Comme $S(F^\mathrm{s\acute{e}p})$ est Zariski-dense dans $S$, on a $S\subset {_FP_u}$. 
On peut donc choisir un tore $F$-dŽployŽ maximal $S'$ dans ${_FP_u}$ tel que $S\subset S'$. Ce tore $S'$ est $(F,u)$-optimal relativement ˆ l'action  
de $G$ par conjugaison et comme il est contenu dans $H$, il est aussi $(F,u)$-optimal relativement ˆ l'action de $H$ par conjugaison. 
On a donc $$\bs{\Lambda}_{F,u} \cap \check{X}(S')_{\mbb{Q}}= \bs{\Lambda}_{F,u}^{H}\cap \check{X}(S')_{\mbb{Q}} = \{\mu\}\ptf$$ On en dŽduit (i). 
D'autre part, on a $$\smash{_F\scrX_u}^{\!\!\!H}  = H_{\mu,1}= H\cap G_{\mu,1} = H\cap {_F\scrX_u}\ptf$$ 
Un ŽlŽment $x\in \smash{_F\scrX_u}^{\!\!\!H} $ appartient ˆ $\smash{_F\scrY_u}^{\!\!H} $ si et seulement si $\bs{q}^H_F(x) = \| \mu \|$. 
D'aprs le point (i) et l'ŽgalitŽ $\smash{_F\scrX_u}^{\!\!\!H} =H \cap {_F\scrX_u}$, 
on obtient $$\smash{_F\scrY_u}^{\!\!H}  = H\cap {_F\scrY_u}\ptf$$ Cela prouve (ii). Quant aux inclusions du point (iii), 
elles se dŽduisent des ŽgalitŽs du point (ii) par conjugaison dans $H(F)\subset G(F)$.  
\end{proof}

\textit{\S2.\hskip1mm Le cas d'un morphisme spŽcial. ---} Soit $f: G \rightarrow H$ un morphisme surjectif de groupes algŽbriques rŽductifs connexes. 
Si $T$ un tore maximal de $G$, alors $T_H=f(T)$ est un tore maximal de $H$ et $f$ induit des morphismes $\mbb{Z}$-linŽaires 
$$\varphi_T= X(T_H)\rightarrow X(T)\quad \hbox{et} \quad \phi_T: \check{X}(T)\rightarrow \check{X}(T_H)\ptf$$ 
Le morphisme $f$ est dit \textit{spŽcial} si $\varphi_T$ envoie le systme de racines de $H$ dans celui de $G$. 
Cette dŽfinition ne dŽpend pas du choix de $T$. Observons que $f$ est spŽcial si et seulement le morphisme 
$f_\mathrm{der}: G_\mathrm{der} \rightarrow H_\mathrm{der}$ est (surjectif et) spŽcial. 

\begin{remark}
\textup{Si le morphisme surjectif $f: G\rightarrow H$ est sŽparable, ou s'il est central\footnote{Rappelons qu'un morphisme surjectif $f$ 
est central si et seulement si $\ker(f)$ est central 
dans $G$ et $\ker (\mathrm{Lie}(f))$ est central dans $\mathrm{Lie}(G)$. Si $p>1$, le morphisme naturel $\mathrm{SL}_p \rightarrow \mathrm{PGL}_p$ 
est une isogŽnie centrale 
mais il n'est pas sŽparable.}, alors il est spŽcial \cite[22.4]{B}. }
\end{remark}

Le morphisme (surjectif) $f:G\rightarrow H$ induit des applications 
$$\phi: \check{X}(G) \rightarrow \check{X}(H)\quad \hbox{et}\quad \phi_{\mbb{Q}}:  \check{X}(G)_{\mbb{Q}} \rightarrow \check{X}(H)_{\mbb{Q}}\ptf$$ 
Si de plus $f$ est dŽfini sur $F$, ce que l'on suppose, alors les applications $\phi$ et $\phi_{\mbb{Q}}$ le sont aussi, \cad qu'elles sont 
$\Gamma_F$-Žquivariantes. Elles induisent des applications 
$$\phi_F: \check{X}_F(G)\rightarrow \check{X}_F(H)\quad \hbox{et} \quad \phi_{F,\mbb{Q}}:
 \check{X}_F(G)_{\mbb{Q}} \rightarrow \check{X}_F(H)_{\mbb{Q}}\ptf$$ 
On a les inclusions $$f(\FU)\subset {_F\UU^H}\quad \hbox{et} \quad f(\UF) \subset \UU_F^H\ptf$$
Pour $\mu \in \check{X}(H)_{\mbb{Q}}$, notons $\iota_{\mbb{Q}}(\mu)$ l'unique ŽlŽment de $\phi^{-1}_{\mbb{Q}}(\mu)$ de norme minimale. 
Cela dŽfinit une application $\iota_{\mbb{Q}}: \check{X}(H)_{\mbb{Q}} \rightarrow \check{X}(G)_{\mbb{Q}}$ qui vŽrifie les propriŽtŽs \cite[10.3]{H1}: 
\begin{itemize}
\item $\iota_{\mbb{Q}}$ est une \guill{section} de $\phi_{\mbb{Q}}$, \textit{i.e.} $\phi_{\mbb{Q}} \circ \iota_{\mbb{Q}} = \mathrm{Id}$; 
\item pour tout tore maximal $T$ de $G$, $\iota_{\mbb{Q}}$ se restreint en une application $\mbb{Q}$-linŽaire 
$\iota_{T,\mbb{Q}}: \check{X}(T)_{\mbb{Q}} \rightarrow \check{X}(T_H)_{\mbb{Q}}$. 
\end{itemize}
Par construction, l'application $\iota_{\mbb{Q}}$ est dŽfinie sur $F$. Pour $\mu \in \check{X}(H)_{\mbb{Q}}$, on pose 
$$\| \mu \| = \inf\{ \| \lambda \| \,\vert \, \lambda \in \phi^{-1}_{\mbb{Q}}(\mu) \}= \| \iota_{\mbb{Q}}(\mu) \|\ptf$$ 
Cela dŽfinit une $F$-norme $H$-invariante sur $\check{X}(H)_{\mbb{Q}}$, laquelle dŽfinit par restriction une $F$-norme $H$-invariante sur 
$\check{X}(H)$. On a donc une fonction $\bs{q}^H_F:{_F\UU^H} \rightarrow \mbb{R}_{\geq 0}$. 
D'aprs  \cite[10.6]{H1}, on a la

\begin{proposition}\label{le cas spŽcial}
Soit $f: G \rightarrow H$ un $F$-morphisme de groupes algŽbriques rŽductifs connexes dŽfinis sur $F$, tel que le morphisme 
$f_\mathrm{der}: G_\mathrm{der}\rightarrow H_\mathrm{der}$ 
soit surjectif et spŽcial. Soit $u\in \FU$.
\begin{enumerate}
\item[(i)] $\phi_{\mbb{Q}}(\bs{\Lambda}_{F,u}) \subset \bs{\Lambda}_{F,f(u)}^H$, $\bs{q}_F(u)\geq \bs{q}_F^H(f(u))$ et 
$f({_FP_u})\subset \smash{_FP}_{\!\!f(u)}^H$. Si de plus $f$ est surjectif, on a l'ŽgalitŽ $f({_FP_u})=\smash{_FP}_{\!\!f(u)}^H$.
\item[(ii)] $f({_F\scrX_u})\subset \smash{_F\scrX}_{\!\!f(u)}^H$ et $f({_F\scrY_u})\subset \smash{_F\scrY}_{\!\!f(u)}^H$.
\item[(iii)] $f({_F\bsfrX_u})\subset \smash{_F\bsfrX}_{\!f(u)}^H$ et $f({_F\bsfrY_u})\subset  \smash{_F\bsfrY}_{\!f(u)}^H$. 
\end{enumerate}
\end{proposition}

\begin{proof}
D'aprs \ref{le cas der}, quitte ˆ remplacer $G$ et $H$ par leurs groupes dŽrivŽs, on peut supposer $f$ surjectif. 
Alors d'aprs \cite[10.6]{H1}, on a $\phi_{\mbb{Q}}(\bs{\Lambda}_{F,u}) \subset \bs{\Lambda}_{F,f(u)}^H$ et 
$f({_FP_u})= \smash{_FP}_{\!\!f(u)}^H$. Pour $\lambda\in \bs{\Lambda}_{F,u}$ et $\mu = \phi_{\mbb{Q}}(\lambda)\in \bs{\Lambda}_{F,f(u)}^H$, on a donc 
$$\bs{q}_F^H(f(u)) = \| \mu \| = \| \iota_{\mbb{Q}}(\mu) \| \leq \| \lambda \| = \bs{q}_F(u)\ptf $$ 
Cela prouve (i). Puisque ${_F\scrX_u}= G_{\lambda,1}$ et $f(G_{\lambda,1})=H_{\mu,1}$, on a 
$$f({_F\scrX_u})= f(G_{\lambda,1})\subset  H_{\mu,1}= \smash{_F\scrX}_{\!\!f(u)}^H\ptf$$ 
Si $u'\in {_F\scrY_u}$, alors $\lambda \in \bs{\Lambda}_{F,u'}$ et (d'aprs (i)) $\mu\in \bs{\Lambda}_{F,f(u')}$, par consŽquent 
$$\bs{q}_F^H( f(u') ) = \| \mu \|  = \bs{q}_F^H(f(u))$$ et 
$f(u')\in \smash{_F\scrY}_{\!\!f(u)}^H$. Cela prouve (ii). On en dŽduit (iii) par conjugaison dans $G(F)$ gr‰ce ˆ l'inclusion $f(G(F))\subset H(F)$. 
\end{proof}

D'aprs \cite[10.7]{H1}, on a le

\begin{corollary}\label{le cas spŽcial, cor}(Sous les hypothses de \ref{le cas spŽcial}.)
\begin{enumerate}
\item[(i)] $\bs{\Lambda}_{F,u} =\phi_{\mbb{Q}}^{-1}( \bs{\Lambda}_{F,f(u)}^H)$, $\bs{q}_F(u)= \bs{q}_F^H(f(u))$ et ${_FP_u}= 
 f^{-1}(\smash{_FP}_{\!\!f(u)}^H)$.
\item[(ii)] Les inclusions \ref{le cas spŽcial}\,(ii) sont des ŽgalitŽs. De plus 
$${_F\scrX_u}=f^{-1}(\smash{_F\scrX}_{\!\!f(u)}^H)\quad \hbox{et}\quad {_F\scrY_u}=f^{-1}(\smash{_F\scrY}_{\!\!f(u)}^H)\ptf$$
\item[(iii)] Les inclusions \ref{le cas spŽcial}\,(iii) sont des ŽgalitŽs. De plus 
$${_F\bsfrX_u}= f^{-1}(\smash{_F\bsfrX}_{\!f(u)}^H)\quad \hbox{et} \quad {_F\bsfrY_u} = f^{-1}(\smash{_F\bsfrY}_{\!f(u)}^H)\ptf$$
\item[(iv)]$f$ induit une application bijective $\FU \rightarrow {_F\UU^H}$.  
\end{enumerate}
\end{corollary}

\begin{proof}
D'aprs \ref{le cas der}, quitte ˆ remplacer $G$ et $H$ par leurs groupes dŽrivŽs, on peut supposer que $f$ est une isogŽnie centrale. Alors d'aprs 
\cite[10.7\,(a)]{H1}, l'application $\phi_{\mbb{Q}}: \check{X}(G)_{\mbb{Q}}\rightarrow \check{X}(H)_{\mbb{Q}}$ est bijective et elle induit une application bijective 
$\phi_{F,\mbb{Q}}: \check{X}_F(G)_{\mbb{Q}} \rightarrow \check{X}_F(H)_{\mbb{Q}}$. D'aprs \cite[10.7\,(b)]{H1}, 
on a $\bs{\Lambda}_{F,u} =\phi_{\mbb{Q}}^{-1}( \bs{\Lambda}_{F,f(u)}^H)$ et 
${_FP_u}=  f^{-1}(\smash{_FP}_{\!\!f(u)}^H)$. De plus, avec les notations de la preuve de \ref{le cas spŽcial}, 
on a $\iota_{\mbb{Q}}(\mu)= \lambda$ par consŽquent $\bs{q}_F^H(f(u))= \bs{q}_F(u)$. 
Cela prouve (i). Le point (ii) est une consŽquence de \cite[22.4]{B}: pour tout 
sous-groupe fermŽ unipotent connexe $U\subset G$, $f$ induit un isomorphisme de $U$ sur $f(U)$. 
En particulier $f$ induit un $F$-isomorphisme de $G_{\lambda,1}$ sur $f(G_{\lambda,1})= H_{\mu,1}$. 
Compte-tenu de (i), cela prouve (ii). Quant au point (iii), soit $S$ un tore $F$-dŽployŽ 
maximal de ${_FP_u}$ et soit $M=G^S$ le centralisateur de $S$ dans $G$. 
Alors $S_H=f(S)$ est un tore $F$-dŽployŽ maximal de $\smash{_FP}_{\!\!f(u)}^H= f({_FP_u})$ 
et $M_H = f(M)$ est le centralisateur $H^{S_H}$ de $S_H$ dans $H$. Soit $G^*\subset G$ le $F$-sous-groupe distinguŽ engendrŽ par les les sous-groupes 
radiciels $U_{(\alpha)}$ pour $\alpha \in \ES{R}_S$; o $\ES{R}_S= \ES{R}_S^G$ est l'ensemble des racines de $S$ dans $G$. On a $$G(F)=M(F) G^*(F)=G^*(F)M(F)\ptf$$ 
On dŽfinit de la mme manire $H^*= \langle U_{(\beta)}^H\,\vert \, \beta \in \ES{R}_{S_H}^H\}\subset H$; on a aussi 
$$H(F) = M_H(F) H^*(F)= H^*(F) M_H(F)\ptf$$ D'aprs \cite[ch.~V, 22.6]{B}, 
le morphisme $\mbb{Z}$-linŽaire $\varphi_{T}: X(S_H)\rightarrow X(S)$ induit par $f$ 
envoie $\ES{R}_{S_H}^H$ bijectivement sur $\ES{R}_S$; par consŽquent si $\beta\in \ES{R}_{S_H}^H$ et $\alpha = \varphi_T(\beta)$, $f$ induit un 
$F$-isomorphisme de $U_{(\alpha)}$ sur $U_{(\beta)}$. On a donc $$H(F)= M_H(F) f(G^*(F))= f(G^*(F))M_H(F)\ptf$$ Puisque 
$${_F\bsfrX_u} = G^*(F)\bullet {_F\scrX_u} \quad \hbox{et} \quad {_F\bsfrY_u} = G^*(F) \bullet {_F\scrY_u}\vg$$ 
le point (iii) dŽcoule de (ii) par conjugaison dans $G^*(F)$. 
Quant au point (iv), c'est une consŽquence de (ii) et (iii).
\end{proof}

\begin{corollary}
(Sous les hypothses de \ref{le cas spŽcial}.) $f$ induit une application bijective $\FU \rightarrow \FU^H$ qui envoie bijectivement les $F$-lames, resp. $F$-strates, de 
$\FU$ sur celles de $\FU^H$.
\end{corollary}

\textit{\S3.\hskip1mm Application: passage au groupe adjoint. ---} On note $G_\mathrm{ad}\subset \mathrm{Aut}(\mathrm{Lie}(G))$ le groupe adjoint de $G$. Le corollaire \ref{le cas spŽcial, cor} s'applique au $F$-morphisme naturel 
$$ f=\mathrm{Ad}: G \rightarrow G_\mathrm{ad} \;(\subset \mathrm{Aut}(\mathrm{Lie}(G))\ptf $$ D'aprs le dŽbut de la preuve de \ref{le cas spŽcial, cor}, le $F$-morphisme 
$f_\mathrm{der}: G_\mathrm{der}\rightarrow G_\mathrm{ad}$ induit une identification canonique 
$$\check{X}_F(G_\mathrm{der})_{\mbb{Q}}= \check{X}_F(G_\mathrm{ad})_{\mbb{Q}}\ptf$$
On remplace l'exposant $H=G_\mathrm{ad}$ par un simple exposant \guill{ad} dans les notations prŽcŽdentes. D'aprs \ref{le cas spŽcial, cor}\,(iv), le morphisme 
$f=\mathrm{Ad}$ induit une application bijective $\FU= \FU^{\mathrm{der}}\rightarrow \FU^{\mathrm{ad}}$ via laquelle on identifie $\FU$ et $\FU^{\mathrm{ad}}$.

\begin{corollary}\label{le cas adjoint} 
Soit $u\in \FU$.
\begin{enumerate}
\item[(i)] $\bs{\Lambda}_{F,u}\;(= \bs{\Lambda}_{F,u}^\mathrm{der}) = \bs{\Lambda}_{F,u}^\mathrm{ad}$, $\bs{q}_F(u)=
 \bs{q}_F^\mathrm{ad}(u)$ et ${_FP_u} =  f^{-1}( \smash{_FP}_{\!\!u}^{\mathrm{ad}})$;
\item[(ii)] ${_F\scrX_u} = \smash{_F\scrX}_{\!\!u}^{\mathrm{ad}}$ et ${_F\scrY_u}= \smash{_F\scrY}_{\!\!u}^{\mathrm{ad}}$; 
\item[(iii)] ${_F\bsfrX_u} = \smash{_F\bsfrX}_{\!u}^{\mathrm{ad}}$ et ${_F\bsfrY_u}= {_F\bsfrY}_{\!u}^{\mathrm{ad}}$.
\end{enumerate}
\end{corollary}

\textit{\S4.\hskip1mm Application: passage au revtement simplement connexe. ---} 
Le corollaire \ref{le cas spŽcial, cor} s'applique aussi au $F$-morphisme naturel $$f:G_{\rm sc} \rightarrow G\vg$$ \cad le revtement simplement connexe de $G_{\rm der}$. Il induit une identification canonique 
$$\check{X}_F(G_\mathrm{sc})_{\mbb{Q}}= \check{X}_F(G_\mathrm{der})_{\mbb{Q}}\ptf$$ On remplace l'exposant $H=G_\mathrm{sc}$ par un simple exposant \guill{sc} dans les notations prŽcŽdentes. 
D'aprs \ref{le cas spŽcial, cor}\,(iv), le morphisme $f$ induit une application bijective ${_F\UU}^{\mathrm{sc}}\rightarrow \FU= \FU^{\mathrm{der}}$ via laquelle on identifie $\FU^{\mathrm{sc}}$ et $\FU$.

\begin{corollary}\label{le cas ss} 
Soit $u\in \FU$.
\begin{enumerate}
\item[(i)] $\bs{\Lambda}_{F,u}^{\mathrm{sc}}= \bs{\Lambda}_{F,u}\;(= \bs{\Lambda}_{F,u}^\mathrm{der}) $, $\bs{q}_F^{\mathrm{sc}}(u)=
 \bs{q}_F(u)$ et $\smash{_FP}_{\!\!u}^{\mathrm{sc}} =  f^{-1}({_FP_u})$;
\item[(ii)] $\smash{_F\scrX}_{\!\!u}^{\mathrm{sc}} = {_F\scrX_u}$ et $\smash{_F\scrY}_{\!\!u}^{\mathrm{sc}}= {_F\scrY_u}$; 
\item[(iii)] $\smash{_F\bsfrX}_{\!u}^{\mathrm{sc}} = {_F\bsfrX_u}$ et $\smash{_F\bsfrY}_{\!u}^{\mathrm{sc}}= {_F\bsfrY_u}$.
\end{enumerate}
\end{corollary}

\subsection{$F$-lames standard de $\FU$}\label{rŽduction standard}
On fixe une $F$-paire parabolique mini\-male $(P_0,A_0)$ de $G$, \cad un $F$-sous-groupe parabolique minimal $P_0$ de $G$ et 
un tore $F$-dŽployŽ maximal $A_0$ de $P_0$. On note $M_0$ le centralisateur de $A_0$ dans $G$ et $U_0$ le radical unipotent de $P_0$. 
Ils sont tous les deux dŽfinis sur $F$ et $P_0=M_0 \ltimes U_0$.\index{P0@$P_0=M_0 \ltimes U_0$} On note $\ES{R}=\ES{R}_{A_0}^G\subset X(A_0)$ l'ensemble des racines de $A_0$ dans $G$ et 
$\ES{R}^+ \subset \ES{R}$ le sous-ensemble formŽ des racines dans $P_0$. 

\vskip2mm
\textit{\S1.\hskip1mm Co-caractres en position standard. ---} 
Rappelons qu'un co-caractre virtuel 
$\mu\in \check{X}_F(G)_{\mbb{Q}}$ est dit 
\textit{en position standard} (relativement ˆ $(P_0,M_0)$) si la paire parabolique $(P_\mu,A_\mu)$ est standard, \textit{i.e.} si $P_0\subset P_\lambda $ et $A_\mu \subset A_0$; auquel 
cas $\mu$ appartient ˆ $\check{X}(A_0)_{\mbb{Q}}$. Inversement, 
un co-caractre virtuel $\mu\in \check{X}(A_0)_{\mbb{Q}}$ est en position standard si et seulement si $\langle \alpha , \mu \rangle \geq 0$ pour toute 
racine $\alpha \in \ES{R}^+$.

\begin{remark}
\textup{
Pour $\lambda\in \check{X}(A_0)$ en position standard et $u\in U_0$, la limite 
$\lim_{t\rightarrow 0} \mathrm{Int}_{t^\lambda}(u)$ existe; et elle vaut $1$ si et seulement si $u$ appartient ˆ $U_\lambda$.
}
\end{remark}

Notons $\Delta_0=\Delta_{P_0}\subset \ES{R}^+$ le sous-ensemble formŽ des racines simples 
et posons\index{d0lambda@$\bs{d}_0(\lambda)$} $$\bs{d}_0(\lambda) \bydef \min \{\langle \alpha ,\lambda \rangle \,\vert \, \alpha \in \Delta_0 \}\ptf$$ 
Observons que $\lambda$ est en position standard si et seulement si $\bs{d}_0(\lambda)\geq 0$. Observons aussi que 
$(P_\lambda,A_\lambda)=(P_0,A_0)$ si et seulement si $\bs{d}_0(\lambda)>0$, auquel cas 
$$\lim_{t\rightarrow 0} \mathrm{Int}_{t^\lambda}(u)=1\quad\hbox{pour tout}\quad  u\in U_0\ptf$$ 
Par consŽquent $U_0\subset \FU$ et $G(F)\bullet U_0 \subset \FU$. RŽciproquement, tout 
ŽlŽment $u\in \FU$ est contenu dans le radical unipotent ${_FU_u}$ de ${_FP_u}$. En choisissant un 
$g\in G(F)$ tel que $g ({_FP_u})g^{-1}= {_FP_{g\bullet u}}$ contienne $P_0$, on obtient que $g\bullet u \in U_0$. 
On en dŽduit que $$\FU = G(F)\bullet U_0 \quad \hbox{et} \quad \UF = G(F)\bullet U_0(F)\ptf$$ 

\vskip2mm
\textit{\S2.\hskip1mm Unipotents en position standard. ---} 
On a dŽfini en \ref{F-lame standard et ŽlŽment en ps} les notions d'ŽlŽment de $\FU$ en position standard et de $F$-lame standard de $\FU$ (relativement ˆ $P_0$). 
Un ŽlŽment $u\in \FU$ est en position standard si et seulement si le tore $A_0$ est $(F,u)$-optimal 
et si l'unique (d'aprs \ref{HMrat1}\,(iii)) ŽlŽment de $\Lambda_{F,u}^\mathrm{opt}\cap \check{X}(A_0)$ 
est en position standard; auquel cas $u$ appartient ˆ $U_\lambda \subset U_0$. 
Tout ŽlŽment de $\FU$ est $G(F)$-conjuguŽ ˆ un ŽlŽment en position standard et 
on a vu (\textit{cf.} \ref{la stratification de Hesselink}) que:
\begin{itemize}
\item deux $F$-lames standard de $\UF$ qui sont conjuguŽes dans $G(F)$ sont Žgales;
\item l'application $v \mapsto {_F\bsfrY_{v}}$ induit une bijection de l'ensemble des $F$-lames standard de $\FU$ sur 
l'ensemble des $F$-strates de $\FU$.
\end{itemize} 

\subsection{Du groupe ˆ l'algbre de Lie}\label{du groupe ˆ l'algbre de Lie}
Posons $\mathfrak{u}_0=\mathrm{Lie}(U_0)$ et fixons un $F$-isomorphisme de variŽtŽs\index{j0@$j_0$} $j_0: \mathfrak{u}_0 \rightarrow U_0$ qui soit compatible ˆ l'action de $A_0$, 
\cad tel que $$j_0\circ \mathrm{Ad}_a = \mathrm{Int}_a \circ j_0\quad \hbox{pour tout}\quad  a\in A_0$$ 
(\textit{cf.} la preuve de \cite[9.1.1]{LL}; voir aussi plus loin, \S2). Puisque tout ŽlŽment de $\FU$ est $G(F)$-conjuguŽ ˆ un ŽlŽment de $\FU$ en position standard, via $j_0$  on ramne 
l'Žtude de l'action par conjugaison de $\check{X}_F(G)$ sur $\UU_F$ ˆ celle de l'action adjointe de $\check{X}(A_0)$ sur $\mathfrak{u}_0(F)$. 

\vskip2mm
\textit{\S1.\hskip1mm L'algbre de Lie $\mathfrak{g}$ comme $G$-module. ---} Posons\index{g@$\mathfrak{g}$} $$\mathfrak{g}= \mathrm{Lie}(G)\ptf$$ 
On considre $\mathfrak{g}$ comme un $G$-module pour l'action adjointe donnŽe par $$g\cdot X = \mathrm{Ad}_g(X)\ptf$$ 
C'est en particulier une $G$-variŽtŽ pointŽe (avec $e_{\mathfrak{g}}=0$). 
On note $\mathfrak{N}=\mathfrak{N}^G$ la variŽtŽ prŽcedemment notŽe 
$\ES{N}^G(\mathfrak{g},0)$. C'est une sous-variŽtŽ fermŽe de $\mathfrak{g}$ dŽfinie sur $F$ et 
$G$-invariante. On note\index{NF@${_F\mathfrak{N}}$, $\mathfrak{N}_{F}$} ${_F\mathfrak{N}}$ le sous-ensemble de $\mathfrak{N}$ prŽcŽdemment notŽ ${_F\ES{N}}^G(\mathfrak{g},0)$ et on pose $\mathfrak{N}_F = \mathfrak{g}(F)\cap {_F\mathfrak{N}}$.
\index{NfrakF@${_F\mathfrak{N}}$, $\mathfrak{N}_F$} On a donc 
$${_F\mathfrak{N}}= G(F)\cdot \mathfrak{u}_0 \quad \hbox{et}\quad \mathfrak{N}_F = G(F)\cdot \mathfrak{u}_0(F)\ptf$$
On dŽfinit comme en \ref{rŽduction standard} les notions d'ŽlŽment de ${_F\mathfrak{N}}$ en position standard et de $F$-lames standard de $\mathfrak{N}_F$. 

Soit $S$ un tore $F$-dŽployŽ maximal de $G$. On note $G^S$ le centralisateur de $S$ dans $G$ et on pose 
$\mathfrak{g}^S= \mathrm{Lie}(G^S)$. Observons que $\mathfrak{g}^S$ co\"{\i}ncide avec l'intersection des noyaux 
$\mathfrak{g}^t=\ker (\mathrm{Ad}_t - \mathrm{Id}\,\vert \, \mathfrak{g} )$ pour $t\in S$.  
On note $\ES{R}_S=\ES{R}_S^G\subset X(S)$ l'ensemble des racines de $S$ dans $G$ et, 
pour chaque $\alpha\in \ES{R}_S$, on note $\mathfrak{u}_\alpha\subset \mathfrak{g}$ 
le sous-espace radiciel associŽ ˆ $\alpha$; \cad le sous-espace notŽ $V_\alpha$ en \ref{la variante de Hesselink} (\S5) pour $V=\mathfrak{g}$. 
On a la dŽcomposition en somme directe, dŽfinie sur $F$,
$$\mathfrak{g}=  \mathfrak{g}^S \oplus \left(\bigoplus_{\alpha \in \ES{R}_S} \mathfrak{u}_\alpha\right)\ptf $$
Posons $\ES{R}'_S = \ES{R}_S \cup\{0\}$ et pour $X\in \mathfrak{g}$, Žcrivons 
$X = \sum_{\alpha \in \ES{R}'_S} X_\alpha$ avec $X_\alpha \in \mathfrak{u}_\alpha$ et 
$X_0\in \mathfrak{g}^S$. Pour $\lambda\in \check{X}(S)$, on a
$$\mathrm{Ad}_{t^\lambda}(X) = \sum_{\alpha \in \ES{R}'_S} t^{\langle \alpha, \lambda \rangle} X_\alpha  \ptf$$
La somme porte en fait sur le sous-ensemble $\ES{R}'_S(X)\subset \ES{R}'_S$ formŽ des $\alpha\in \ES{R}'_S$ tels que $X_\alpha\neq 0$.
On dŽfinit l'invariant $m'_{X}(\lambda)\in \mbb{Z}  \cup \{+ \infty \}$ comme en  \ref{la variante de Hesselink}: 
si $X=0$, on pose $m'_{X}(\lambda)= +\infty$ et sinon, on pose 
$$m'_{X}(\lambda)= \min \{\langle \alpha ,\lambda \rangle \,\vert \, \alpha \in \ES{R}'_S(X)\}\in \mbb{Z}  \ptf$$ 
Si la limite $\lim_{t\rightarrow 0} t^\lambda\cdot X$ existe (\textit{i.e.} si $m'_X(\lambda)\geq 0$), on a (\ref{m'=m}) 
$$m'_X(\lambda)= m_{X}(\lambda)\;(=m_{0,X}(\lambda)) \ptf$$ 

Pour $\lambda \in \check{X}(S)$ et $i\in \mbb{Z} $, on a dŽfini en \ref{la variante de Hesselink} des sous-espaces $\mathfrak{g}_\lambda(i)$ 
et $\mathfrak{g}_{\lambda ,i}$ de $\mathfrak{g}$. Notons $\ES{R}_{S,\lambda}(i)$ l'ensemble des racines $\alpha \in \ES{R}_S$ 
telles que $ \langle \alpha ,\lambda \rangle = i$. On a 
$$\mathfrak{g}_{\lambda}(i) = 
\left\{\begin{array}{ll}\mathfrak{g}^S \oplus \left(\bigoplus_{\alpha \in \ES{R}_{S,\lambda}(0)} \mathfrak{u}_\alpha\right)& \hbox{si $i=0$}\\ 
\bigoplus_{\alpha \in \ES{R}_{S,\lambda}(i)} \mathfrak{u}_\alpha & \hbox{sinon}\end{array}\right.\quad \hbox{et} \quad \mathfrak{g}_{\lambda ,i} =
 \bigoplus_{j\geq i} \mathfrak{g}_\lambda(j)\ptf$$ 
Observons que 
$$\mathfrak{g}_{\lambda ,0}= \mathrm{Lie}(P_\lambda)\vgq \mathfrak{g}_{\lambda ,1} =
 \mathrm{Lie}(U_\lambda) \quad \hbox{et} \quad \mathfrak{g}_\lambda(0) = \mathrm{Lie}(M_\lambda)\ptf$$ 
Puisque $\lambda$ est dŽfini sur $F$, la filtration $\{\mathfrak{g}_{\lambda,i}\}_{i\in \mbb{Z} }$ et la graduation 
$\mathfrak{g}_\lambda(i)$ de $\mathfrak{g}$ dŽfinies ci-dessus le sont aussi. 
Rappelons que d'aprs \ref{m'=m}, la filtration $\{\mathfrak{g}_{\lambda,i}\}_{i\in \mbb{N}}$ co\"{\i}ncide avec celle dŽfinie en \ref{la variante de Hesselink} 
dans le cas d'un $G$-variŽtŽ pointŽe quelconque.

\vskip2mm
\textit{\S2.\hskip1mm Passage du groupe $G$ ˆ l'algbre de Lie $\mathfrak{g}$. ---} 
PrŽcisons la construction du $F$-isomorphisme $j_0:\mathfrak{u}_0 \rightarrow U_0$. 
Fixons une extension galoisienne finie $\wt{F}\subset F^\mathrm{s\acute{e}p}$ de $F$ dŽployant $G$ et un tore $\wt{F}$-dŽployŽ maximal 
$\wt{A}_0$ de $G$, dŽfini sur $F$ et contenant $A_0$. Soit $\wt{\ES{R}}= \ES{R}_{\smash{\wt{A}_0}}$ l'ensemble des racines de $\wt{A}_0$ dans $G$.  
Pour $\wt{\alpha}\in \wt{\ES{R}}$, fixons un $\wt{F}$-isomorphisme $e_{\wt{\alpha}}: \mbb{G}_{\mathrm{a}/F} \rightarrow U_{\wt{\alpha}}$. 
Posons $$\epsilon_{\wt{\alpha}}= e_{\wt{\alpha}}(1)\quad \hbox{et} \quad E_{\wt{\alpha}}= \mathrm{Lie}(e_{\wt{\alpha}})(1)\ptf $$ 
Alors pour $t\in \overline{F}$ et $X\in \overline{F}[t]\otimes_{\overline{F}}\mathfrak{g}$, on a 
$$\mathrm{Ad}_{e_{\wt{\alpha}}(t)}(X)\equiv X + t [E_{\wt{\alpha}},X] 
\quad (\mathrm{mod}\; t^2 \overline{F}[t]\otimes_{\overline{F}}
\mathfrak{g})\ptf$$
Puisque l'ensemble $\wt{\ES{R}}$ est $\Gamma_F$-invariant, on peut imposer  
$\gamma(e_{\wt{\alpha}})= e_{\gamma(\wt{\alpha})}$ pour tout $\gamma\in \Gamma_F$ et tout $\wt{\alpha}\in \wt{\ES{R}}$; 
alors $\gamma(\epsilon_{\wt{\alpha}})= \epsilon_{\gamma(\wt{\alpha})}$ et $\gamma(E_{\wt{\alpha}})= E_{\gamma(\wt{\alpha})}$\footnote{On ne fait ici aucune hypothse supplŽmentaire sur la paire $(e_{\wt{\alpha}},e_{-\wt{\alpha}})$; en particulier si $p=1$, 
on ne suppose pas que la famille $(E_{\wt{\alpha}})_{\smash{\wt{\alpha}\in \wt{\ES{R}}}}$ est un systme de Chevalley (\text{cf.} \cite[ch.~VIII, \S2, n\textsuperscript{o}4]{Bou}). Voir aussi plus loin 
(\S4).}. 
La restriction ˆ $A_0$ donne une application $\wt{\ES{R}} \rightarrow \ES{R}'= \ES{R}\cup \{0\}$ 
d'image $\ES{R}$ ou $\ES{R}'$ (l'image est $\ES{R}$ si et seulement si $G$ est dŽployŽ sur $F$). 
Pour $\alpha \in \ES{R}$, notons $\wt{\ES{R}}(\alpha)\subset 
\wt{\ES{R}}$ la fibre au-dessus de $\alpha$. 
Comme en \cite[21.7]{B}, on pose $$(\alpha)= \left\{\begin{array}{ll}
\{\alpha,2\alpha\} & \hbox{si $2\alpha\in\ES{R}$}\\
\{\alpha \} & \hbox{sinon}\end{array}\right.
$$ 
et on note $U_{(\alpha)}$ le sous-groupe unipotent de $G$ correspondant ˆ $(\alpha)$. 
On pose $U_{2\alpha}=\{1\}$ si $2\alpha \notin \ES{R}$. 
Alors le choix d'un ordre 
$\wt{\alpha}_1, \dots , \wt{\alpha}_{r(\alpha)}$ sur $\wt{\ES{R}}(\alpha)$ permet de dŽfinir un $F$-isomorphisme de variŽtŽs 
$$ j_\alpha : \mathfrak{u}_\alpha \rightarrow U_\alpha / U_{2\alpha}\vgq 
\sum_{i=1}^{r(\alpha)} t_i E_{\wt{\alpha}_i} \mapsto e_{\wt{\alpha}_1}(t_1)\cdots e_{\wt{\alpha}_{r(\alpha)}}(t_{r(\alpha)}) U_{2\alpha} 
\quad \hbox{avec} \quad t_i\in \overline{F}\ptf$$ 
Par construction $j_\alpha$ est compatible ˆ l'action de $A_0$ et mme ˆ celle de $\wt{A}_0$. 
Si $2\alpha\in \ES{R}$, on dŽfinit de mme un $F$-isomorphisme de variŽtŽs $j_{2\alpha}: \mathfrak{u}_{2\alpha} \rightarrow U_{2\alpha}$ 
compatible ˆ l'action de $A_0$ 
et on note $$j_{(\alpha)} : \mathfrak{u}_{(\alpha)} \bydef \mathfrak{u}_\alpha \oplus \mathfrak{u}_{2\alpha} \rightarrow U_{(\alpha)}$$ 
le $F$-isomorphisme de variŽtŽs obtenu en composant $j_\alpha \times j_{2\alpha}$ avec un $F$-isomorphisme de variŽtŽs  
$U_\alpha/U_{2\alpha} \times U_{2\alpha} \rightarrow U_{(\alpha)}$ compatible ˆ l'action de $A_0$ 
(\textit{cf.} \cite[21.19, 21.20]{B}). Le choix d'un ordre $\alpha_1, \ldots ,\alpha_s$ sur les racines non-divisibles dans $\ES{R}^+$ permet alors de dŽfinir $j_0$: on pose 
$$j_0 = j_{(\alpha_1)}\times \cdots \times j_{(\alpha_s)} : \mathfrak{u}_0= 
\bigoplus_{i=1}^s \mathfrak{u}_{(\alpha_i)} \rightarrow U_{(\alpha_1)}U_{(\alpha_2)} \cdots U_{(\alpha_s)}=U_0\ptf$$  

\begin{remark}\label{j_Q}
\textup{Pour tout $F$-sous-groupe parabolique standard $Q$ de $G$, $j_0$ induit par restriction un $F$-isomorphisme de variŽtŽs $A_0$-Žquivariant 
$$j_Q: \mathfrak{u}_Q=\mathrm{Lie}(U_Q)\rightarrow U_Q\ptf$$}
\end{remark}

Pour $\mu \in \check{X}(G)_\mbb{Q} $, on a dŽfini en \ref{la variante de Hesselink} une filtration $(G_{\lambda,r})_{r\in  \mbb{Q} _+}$. Les 
$G_{\mu,r}$ sont des sous-groupes fermŽs de $G$, qui sont dŽfinis sur $F$ si $\mu \in \check{X}_F(G)_\mbb{Q} $. 
D'aprs \ref{m'=m}, pour $\lambda\in \check{X}(A_0)$ en position standard, on a
$$m_u(\lambda)= m'_{j^{-1}_0(u)}(\lambda) \quad \hbox{pour tout} \quad u\in U_\lambda\ptf$$ 
On en dŽduit que pour $\lambda\in \check{X}_F(G)$ et $i\in \mbb{N}^*$, $G_{\lambda ,i}$ est le sous-groupe de $U_\lambda$ engendrŽ par les sous-groupes 
radiciels $U_{(\alpha)}$ pour $\alpha\in \ES{R}_S$ avec $\langle \alpha , \lambda \rangle \geq i$; o $S$ est un tore $F$-dŽployŽ maximal 
de $G$ contenant $\mathrm{Im}(\lambda)$. De plus on a $$\mathfrak{g}_{\lambda , i} = \mathrm{Lie}(G_{\lambda,i})\ptf$$ 
Pour $\lambda \in \check{X}_F(G)$ et $i\in \mbb{N}$, on pose $$G_{\lambda}(i)= G_{\lambda,i}/ G_{\lambda ,i+1}\ptf$$ 
Puisque le groupe $G_{\lambda ,i+1}$ est distinguŽ dans $G_{\lambda ,0}=P_\lambda$, il l'est \textit{a fortiori} dans $G_{\lambda,i}$. 
Par consŽquent $G_{\lambda}(i)$ est un groupe algŽbrique, dŽfini sur $F$ puisque $\lambda$ l'est. Observons que 
$$G_\lambda(0)=P_\lambda/U_\lambda \ptf$$
Observons aussi que pour $i\in \mbb{N}$, l'action par conjugaison de 
$P_\lambda$ sur $G_\lambda(i)$ se factorise en une action de $G_\lambda(0)\;(\simeq M_\lambda)$. 

Pour $\lambda\in \check{X}(A_0)$ en position standard et $k\in \mbb{N}^*$, $j_0$ induit un $F$-isomorphisme de variŽtŽs  
$\mathfrak{g}_{\lambda ,k}\rightarrow G_{\lambda ,k}$ 
qui se factorise un en $F$-isomorphisme de variŽtŽs  
$$j_\lambda(k):\mathfrak{g}_{\lambda}(k)= \mathfrak{g}_{\lambda ,k}/ \mathfrak{g}_{\lambda,k+1} \rightarrow G_\lambda(k)\ptf$$ 
Par construction $j_\lambda(k)$ est $A_0$-Žquivariant.

\begin{lemme}\label{M-module}
Pour $\lambda\in \check{X}(A_0)$ en position standard et $k\in \mbb{N}^*$, $j_\lambda(k)$ est $M_\lambda$-Žquivariant et c'est un 
$F$-isomorphisme de groupes. En particulier, $j_\lambda(k)$ munit $G_\lambda(k)$ d'une structure de $M_\lambda$-module dŽfini sur $\wt{F}$.
\end{lemme}

\begin{proof}
Notons $\ES{R}_\lambda(k)$ l'ensemble des $\alpha \in \ES{R}$ tels que $\langle \alpha ,\lambda \rangle =k$. On dŽfinit de la mme 
manire $\wt{\ES{R}}_\lambda(k)$. Si $\alpha \in \ES{R}_\lambda(k)$ alors $2\alpha\notin \ES{R}_\lambda(k)$ 
et l'inclusion $U_{(\alpha)}\subset G_{\lambda,k}$ induit 
par passage aux quotients une application injective $U_{(\alpha)}/U_{2\alpha} \rightarrow G_\lambda(k)$. En la composant avec 
$j_\alpha$, on obtient un $F$-morphisme injectif $\mathfrak{u}_\alpha \rightarrow G_\lambda(k)$ qui est $\wt{A}_0$-Žquivariant. 
Pour $X\in \mathfrak{g}_\lambda(k)$, on Žcrit $X= \sum_{\wt{\alpha}\in \wt{\ES{R}}_\lambda(k)} x_{\wt{\alpha}} E_{\wt{\alpha}}$ 
avec $x_{\wt{\alpha}}\in \overline{F}$. 
D'aprs les relations de commutateurs de Chevalley, $j_\lambda(k)(X)$ 
est l'image de $\prod_{\wt{\alpha}\in \wt{\ES{R}}_\lambda(k)} e_{\wt{\alpha}}(x_{\wt{\alpha}})$ 
dans $G_\lambda(k)$ pour n'importe quel ordre sur le produit. En particulier $j_\lambda(k)$ est un $F$-isomorphisme de groupes. Par construction 
$j_\lambda(k)$ est $\wt{A}_0$-Žquivariant et munit 
$G_\lambda(k)$ d'une opŽration linŽaire de $\wt{A}_0$. \`A nouveau d'aprs les relations de commutateurs de Chevalley, pour $\wt{\alpha}\in \wt{\ES{R}}$ tel que 
$\langle \wt{\alpha},\lambda \rangle =0$, $j_\lambda(k)$ est $U_{\wt{\alpha}}$-Žquivariant et munit $G_\lambda(k)$ d'une opŽration linŽaire de 
$U_{\wt{\alpha}}$. Par consŽquent $j_\lambda(k)$ est $M_\lambda$-Žquivariant et munit $G_\lambda(k)$ d'une structure de $M_\lambda$-module. Puisque 
les $E_{\wt{\alpha}}$ sont $\wt{F}$-rationnels, cette structure est dŽfinie sur $\wt{F}$.
\end{proof}

\begin{proposition}\label{thŽorme de KN rationnel sur G}
Soient $\lambda \in \check{X}_F(G)$, $k\in \mbb{N}^*$ et $u\in G_{\lambda,k}(F)\smallsetminus G_{\lambda, k+1}(F)$. Le co-caractre 
$\lambda$ appartient ˆ $\Lambda^\mathrm{opt}_{F,u}$ si et seulement si l'image $u_\lambda(k)$ de $u$ dans $G_\lambda(k)$ est 
$(\wt{F},M_{\lambda}^\perp)$-semi-stable.
\end{proposition}

\begin{proof}
Quitte ˆ remplacer $\lambda$ (et $u$) par un conjugu\'e dans $G(F)$, on peut supposer $\lambda$ en position standard. 
Alors d'apr\`es \ref{M-module} et le critre de Kirwan-Ness rationnel (\ref{thm de KN rationnel}), 
$\lambda\in \Lambda_{\wt{F},u}^\mathrm{opt}$ si et seulement si $u_\lambda(k)$ est $(\wt{F},M_{\lambda}^\perp)$-semi-stable. D'o le lemme, puisque 
$\Lambda_{F,u}^{\mathrm{opt}}= \check{X}_F(G) \cap \Lambda_{\wt{F},u}^{\mathrm{opt}}$ (\ref{descente sŽparable lames}\,(i)).
\end{proof}

\begin{corollary}\label{lames et translations}
Soient $u\in \UF \smallsetminus \{1\}$, $\lambda \in \Lambda^\mathrm{opt}_{F,u}$ et $k= m_u(\lambda)$. On a 
$${_F\scrY_u}\,G_{\lambda,k+1}= G_{\lambda ,k+1}\,{_F\scrY_u} = {_F\scrY_u}\ptf$$
\end{corollary}

\begin{proof}
Puisque $\lambda$ appartient ˆ $\Lambda_{\wt{F},u}^{\mathrm{opt}}$, on a (\ref{Y et translations})
$${_{\smash{\wt{F}}}\scrY_u}\,G_{\lambda,k+1}= G_{\lambda ,k+1}\,{_{\smash{\wt{F}}}\scrY_u}= {_{\smash{\wt{F}}}\scrY_u}\ptf$$
Or ${_F\scrY_u}= {_{\smash{\wt{F}}}\scrY_u}$ (\ref{descente sŽparable lames}\,(ii)). D'o le corollaire.
\end{proof}

\begin{proposition}\label{optimalitŽ et j} Soit $X\in \mathfrak{u}_0(F)$.
\begin{enumerate}
\item[(i)]Si $\lambda \in \check{X}(A_0)$ est en position standard, alors 
$$\lambda \in \Lambda^\mathrm{opt}_{F,X}\quad \hbox{si et seulement si}\quad \lambda\in \Lambda^\mathrm{opt}_{F,j_0(X)}\ptf$$
\item[(ii)]Si $X$ est en position standard, alors $j_0(X)$ l'est aussi et 
$${_F\scrY_{j_0(X)}}= \{u \in \FU\,\vert \, \bs{\Lambda}_{F,u} = \bs{\Lambda}_{F,X} \} \ptf$$
\item[(iii)]L'application $X\mapsto j_0(X)$ induit une bijection entre l'ensemble des $F$-lames standard de ${_F\mathfrak{N}}$ qui possdent un point $F$-rationnel et l'ensemble des 
$F$-lames standard de ${_F\mathfrak{U}}$ qui possdent un point $F$-rationnel. 
\end{enumerate}
\end{proposition}

\begin{proof}
Puisque $j_0(0)=1$, on peut supposer $X\neq 0$. Notons $u$ l'ŽlŽment $j_0(X)\in U_0(F)\smallsetminus \{1\}$. 

Prouvons (i). Posons $k= m_X(\lambda)$. On a 
$m_u(\lambda)=k$. Si $k=0$ il n'y a rien ˆ dŽmontrer. On peut donc supposer $k>0$. D'aprs \ref{M-module}, la composante $X_\lambda(k)$ de $X$ sur 
$\mathfrak{g}_\lambda(k)$ est $(\wt{F},M_{\lambda}^\perp)$-semi-stable si et seulement l'image 
$u_\lambda(k)$ de $u$ dans $G_\lambda(k)= G_{\lambda,k}/G_{\lambda,k+1}$ est $(\wt{F},M_{\lambda}^\perp)$-semi-stable. 
D'aprs \ref{thm de KN rationnel} et \ref{thŽorme de KN rationnel sur G}, on a donc $$\lambda \in \Lambda_{\wt{F},X}^\mathrm{opt}\quad \hbox{si et seulement si}\quad
\lambda\in \Lambda_{\wt{F},u}^\mathrm{opt}\ptf$$ 
D'o le point (i) gr‰ce ˆ \ref{descente sŽparable lames}\,(i). 

Prouvons (ii). On suppose $X$ en position standard. Si $\Lambda_{F,X}^\mathrm{opt}\cap \check{X}(A_0) = \{\lambda\} $, alors d'aprs (i) on a 
$\Lambda_{F,u}^\mathrm{opt}\cap \check{X}(A_0)= \{\lambda\}$, d'o l'on dŽduit l'ŽgalitŽ 
$\Lambda_{F,u}^\mathrm{opt}= \Lambda_{F,X}^\mathrm{opt}$. 
Comme $m_{u}(\lambda)= k=m_X(\lambda)$, on a aussi l'ŽgalitŽ $\bs{\Lambda}_{F,u}= \bs{\Lambda}_{F,X}$. Cela prouve (ii). 

Le point (iii) dŽcoule de (ii).
\end{proof}

 \begin{corollary}\label{lames et passage ˆ l'inverse}
 Soit $u\in \UF$. Pour $u'\in {_F\scrY_u}$, $u'^{-1}$ appartient ˆ ${_F\scrY_u}$.
 \end{corollary}
 
 \begin{proof}
 On peut supposer $u\neq 1$. On peut aussi supposer $u$ en position standard. Soient $\lambda\in \Lambda^{\rm opt}_{F,u}$ et 
 $k=m_u(\lambda)$. Soit $X=j_0^{-1}(u)$. Soient aussi $u'\in {_F\scrY_u}$ et $X'=j^{-1}_0(u')\in {_F\scrY_X}$.  Alors $j^{-1}_0(u'^{-1})=X'' \in  -X' + \mathfrak{g}_{\lambda , k+1}$. Comme $-X'\in {_F\scrY_{X}}$ et 
 ${_F\scrY_X}+ \mathfrak{g}_{\lambda , k+1} = {_F\scrY_X}$ (d'aprs \ref{Y et translations}), $X''$ appartient ˆ ${_F\scrY_X}$. 
 On en dŽduit que $$\bs{\Lambda}_{F,u^{-1}}= \bs{\Lambda}_{F,X''}= \bs{\Lambda}_{F,X}= \bs{\Lambda}_{F,u}\vg$$ ce qui prouve le corollaire.
 \end{proof}
 
 \begin{corollary}\label{bijection lames unip-lames nilp}
 \begin{enumerate}
 \item[(i)] Pour $X\in \mathfrak{N}_F$, l'ensemble $${_F\scrY_X^G} \bydef \{u\in \FU \,\vert \,\bs{\Lambda}_{F,u}= \bs{\Lambda}_{F,X}\}$$ 
est une $F$-lame de $\FU$.
 \item[(ii)] L'application $X \mapsto {_F\scrY_X^G}$ induit une bijection entre 
l'ensemble des $F$-lames de ${_F\mathfrak{N}}$ qui possdent un point $F$-rationnel et l'ensemble $F$-lames de $\FU$ qui possdent un point $F$-rationnel.
 \end{enumerate}
 \end{corollary}
 
 \begin{proof}
 Pour $X\in \mathfrak{N}_F$, on choisit un $g\in G(F)$ tel que $X'={\rm Ad}_g(X)$ soit en position standard. Alors 
 $${_F\scrY_{X}^G} = g^{-1}\bullet  {_F\scrY_{X'}^G} = g^{-1} \bullet {_F\scrY_{j_0(X')}}\vg$$ ce qui prouve (i). 
 Si $X,\,X'\in \mathfrak{N}_F$ sont tels que ${_F\scrY_{X'}^G}= {_F\scrY_{X}^G}$, alors $\bs{\Lambda}_{F,X'}= \bs{\Lambda}_{F,X}$ et par consŽquent ${_F\scrY_{X'}}= {_F\scrY_{X}}$; d'o l'injectivitŽ dans (ii). Si $u\in \UF$, on 
 choisit un $g\in G(F)$ tel que $u'=g\bullet u$ soit en position standard et on pose $X= {\rm Ad}_{g^{-1}}(j_0^{-1}(u'))$. Alors  
 $${_F\scrY_{u}}= g^{-1} \bullet {_F\scrY_{u'}}= g^{-1}\bullet {_F\scrY_{j_0^{-1}(u')}^G} = {_F\scrY_{X}^G};$$ d'o la surjectivitŽ dans (ii). 
 \end{proof}
 
La bijection \ref{bijection lames unip-lames nilp}\,(ii) s'Žtend naturellement aux $F$-strates: 

\begin{corollary}\label{bijection strates unip-strates nilp}
\begin{enumerate}
\item[(i)] Pour $X\in \mathfrak{N}_F$, l'ensemble $${_F\bsfrY_{X}^G} \bydef \{u\in \FU \,\vert \, \hbox{$\exists g\in G(F)$ tel que $\bs{\Lambda}_{F,u}=g\bullet  \bs{\Lambda}_{F,X}$}\}$$ 
est une $F$-strate de $\FU$.
\item[(ii)] L'application $X \mapsto {_F\bsfrY_{X}^G}$ induit une bijection entre 
l'ensemble des $F$-strates de ${_F\mathfrak{N}}$ qui possdent un point $F$-rationnel et l'ensemble des $F$-strates de $\FU$ qui possdent un point $F$-rationnel.
\end{enumerate}
\end{corollary}

\begin{proof} 
Pour $X\in \mathfrak{N}_F$, on a ${_F\bsfrY_{X}^G} = G(F)\bullet {_F\scrY_{X}^G}$ ce qui prouve (i) (d'aprs \ref{bijection lames unip-lames nilp}\,(i)). 
Si $X,\,X'\in \mathfrak{N}_F$ sont tels que ${_F\bsfrY_{X'}^G}= {_F\bsfrY_{X}^G}$, alors 
$\bs{\Lambda}_{F,X'}= g \bullet \bs{\Lambda}_{F,X}$ pour un $g\in G(F)$, ce qui entra"ne 
l'ŽgalitŽ ${_F\bsfrY_{X'}}= {_F\bsfrY_{X}}$; d'o l'injectivitŽ dans (ii). D'autre part 
toute $F$-strate de $\FU$ qui possde un point $F$-rationnel est de la forme ${_F\bsfrY_{u}}$ avec $u\in \UF$ 
en position standard, et on a ${_F\bsfrY_{u}}= {_F\bsfrY_{j_0^{-1}(u)}^G}$; d'o la surjectivitŽ dans (ii). 
 \end{proof}
 
 \begin{remark}\label{bijection strates-strates canonique}
 \textup{La bijection de \ref{bijection lames unip-lames nilp}\,(ii) est entirement caractŽrisŽe par l'ŽgalitŽ de 
 \ref{bijection lames unip-lames nilp}\,(i); et celle de \ref{bijection strates unip-strates nilp}\,(ii) est entirement caractŽrisŽe par l'ŽgalitŽ de 
 \ref{bijection strates unip-strates nilp}\,(i). Ces bijections ne dŽpendent ni du choix du $F$-isomorphisme $A_0$-Žquivariant $j_0: \mathfrak{u}_0 \rightarrow U_0$, ni de celui de 
la $F$-paire parabolique minimale $(P_0,M_0)$ de $G$. 
 }
\end{remark} 

\vskip2mm
\textit{\S3.\hskip1mm Lien avec le thŽorie des $\mathfrak{sl}_2$-triplets. ---} Dans ce paragraphe, on suppose que 
$p=1$ ou $p\gg 1$. PrŽcisŽment si $p>1$, on suppose que $p \geq 4m +3$ o $m\geq 1$ est un entier dŽfini comme suit. Soit 
$\{\wt{\alpha}_1,\ldots , \wt{\alpha}_n\}$ une base du systme de racines (absolu) $\wt{\ES{R}}$ de $G$ et soit $\wt{\ES{R}}^+$ l'ensemble des racines postives 
relativement ˆ cette base. Pour $\wt{\alpha}\in  \wt{\ES{R}}$, on Žcrit $\wt{\alpha}= \sum_{i=1}^n m_i(\wt{\alpha}) \wt{\alpha}_i$ avec $m_i(\wt{\alpha})\in \mbb{N}$ et on pose 
$m(\wt{\alpha})= \sum_{i=1}^n m_i(\wt{\alpha})$. Alors $m= \max\{m(\wt{\alpha})\,\vert \, \wt{\alpha}\in \wt{\ES{R}}^+\}$. Observons que 
cette hypothse sur $p$ assure que la forme de Killing sur $\mathfrak{g}_{\mathrm{der}}= \mathrm{Lie}(G_{\mathrm{der}})$ est non dŽgŽnŽrŽe. 

Soit $(V,\rho: \mathrm{SL}_{2/F} \rightarrow \mathrm{Aut}(V))$ un $\mathrm{SL}_{2/F}$-module (dŽfini sur $F$). On dit que $(V,\rho)$ est \textit{infinitŽsimalement compltement rŽductible sur $F$} 
si le $\mathfrak{sl}_2$-module $(V,\mathrm{Lie}(\rho))$ est somme directe (dŽfinie sur $F$) de $\mathfrak{sl}_{2/F}$-modules irrŽductibles, 
\textit{cf.} \cite[11.1]{H1}. 

Soit $u\in \UU_F$. D'aprs un rŽsultat non publiŽ de H.~Kraft \cite[11.5]{H1},  
il existe un $F$-morphisme de groupes $\epsilon: \mathrm{SL}_{2/F}\rightarrow G_{\mathrm{der}}$ avec  
$\epsilon\left(\begin{array}{cc}1& 1\\ 0 & 1\end{array}\right)= u$, tel que le $\mathrm{SL}_{2/F}$-module $(\mathfrak{g}, \mathrm{Ad}_{\epsilon(\cdot)})$ soit 
infinitŽsimalement complŽtement rŽductible sur $F$; de plus (\textit{loc.~cit.}) le co-caractre $\lambda_\epsilon$ de $G$ dŽfini par $t^{\lambda_\epsilon} = \epsilon \left( \begin{array}{cc}t& 0\\ 0 & t^{-1}\end{array}\right)$ vŽrifie 
$ \frac{1}{2} \lambda_\epsilon \in \bs{\Lambda}_{F,u}$. On en dŽduit que:
\begin{itemize}
\item si $\lambda_\epsilon$ est primitif, alors $\lambda_\epsilon \in \Lambda_{F,u}^{\mathrm{opt}}$ et 
$m_{F,u}\;(= m_u(\lambda_\epsilon))= 2$;
\item si $\lambda_\epsilon$ n'est pas primitif, alors $\frac{1}{2}\lambda_\epsilon \in \Lambda_{F,u}^{\mathrm{opt}}$ et 
$m_{F,u} \;(= m_u(\frac{1}{2}\lambda_\epsilon)) = 1$.  
\end{itemize}
Sous la mme hypothse sur $p$, on a un rŽsultat analogue pour les ŽlŽments de $\NN_F$ (\textit{cf.} \cite[12]{H1}). 

\begin{remark}
\textup{Si $p>1$, l'hypothse $p\geq 4m+3$ n'est pas optimale. Par exemple pour $G=\mathrm{GL}_n$ ou $G= \mathrm{SL}_n$, aucune hypothse sur $p$ n'est nŽcessaire. 
Pour $G=\mathrm{SO}_n$ ou $G= \mathrm{Sp}_{2n}$, l'hypothse $p\neq 2$ est suffisante (\textit{cf.} \cite[remark~11.6]{H1}). Plus gŽnŽralement, pour des rŽsultats analogues sous des hypothses 
moins restrictives que $p\geq 4m+3$, on renvoie aux travaux de McNinch \cite{MN1,MN2}. }
\end{remark}

\vskip2mm
\textit{\S4.\hskip1mm Sur l'admissibilitŽ du $G$-module $\mathfrak{g}$.} Dans ce paragraphe, on suppose que le groupe 
$G$ est dŽployŽ sur $F$; \textit{i.e.} on peut prendre $\wt{F}=F$. 
On suppose aussi que $G$ provient par le changement de base $\mbb{Z} \rightarrow F$ d'un 
$\mbb{Z}$-groupe rŽductif de Chevalley-Demazure $\scr{G}$ (\textit{cf.} \ref{le cas d'un corps top}), \textit{i.e.} $G = \scr{G}\times_\mbb{Z} F$\footnote{L'existence de 
$\scr{G}$ est due ˆ Chevalley (\textit{cf.} \cite{Dem}).}. L'algbre de Lie 
de $\scr{G}$, notŽe $\mathrm{Lie}(\scr{G})=\mathrm{Dist}_{\mbb{Z},1}^+(\scr{G})$ est une sous-algbre de $\mathrm{Dist}_{\mbb{Z}}(\scr{G})$ et 
un $\mbb{Z}$-module libre de rang fini (\textit{cf.} \cite[I, 7.7]{J2}). 
L'action adjointe de $\scr{G}$ sur le $\mbb{Z}$-foncteur $\mathfrak{g}_{\mbb{Z}}=\mathrm{Lie}(\scr{G})_{\mathrm{a}}\;(\simeq \mbb{A}_{\mbb{Z}}^n)$ le munit d'une structure de 
$\scr{G}$-module, ce qui fait de $\mathfrak{g}=\mathrm{Lie}(\scr{G})\otimes_F \overline{F}$ un $G$-module admissible. Observons que $\mathfrak{g}'=\mathrm{Lie}(\scr{G})\otimes_{\mbb{Z}}\mbb{C}$ est l'algbre de Lie de $G'= \scr{G}\times_{\mbb{Z}}\mbb{C}$ et que 
$\mathrm{Lie}(\scr{G})\;(\subset \mathfrak{g}')$ est un rŽseau admissible dans le $G'$-module $\mathfrak{g}'$ (pour l'action adjointe). 

\begin{remark}
\textup{On suppose aussi 
que le tore $F$-dŽployŽ maximal $A_0$ de $G$ provient par le changement de base $\mbb{Z}\rightarrow F$ d'un tore $\mbb{Z}$-dŽployŽ maximal $\scr{A}_0$ de $\scr{G}$. 
Alors on a les identifications $\check{X}(A_0)=\check{X}(\scr{A}_0)$ et 
$X(A_0)=X(\scr{A}_0)$; en 
particulier chaque racine $\alpha \in \ES{R}$ provient par le changement de base $\mbb{Z} \rightarrow F$ 
d'un morphisme de $\mbb{Z}$-schŽmas en groupes $\scr{A}_0 \rightarrow \mbb{G}_{m/\mbb{Z}}$, encore notŽ $\alpha$. 
Pour $\alpha \in \ES{R}$, on note $\scr{U}_{\alpha}$ le sous-schŽma en groupes fermŽ lisse  
de $\scr{G}$ associŽ ˆ $\alpha$, de sorte que $U_\alpha = \scr{U}_\alpha \times_{\mbb{Z}}F$. On a la dŽcomposition 
$$\mathrm{Lie}(\scr{G})= \mathrm{Lie}(\scr{A}_0) \oplus \left(\bigoplus_{\alpha \in \ES{R}} \mathrm{Lie}(\scr{U}_\alpha)\right)$$ 
avec $\mathrm{Lie}(\scr{A}_0)\otimes_{\mbb{Z}} \overline{F}= \mathrm{Lie}(A_0)$ et $\mathrm{Lie}(\scr{U}_\alpha)\otimes_{\mbb{Z}}\overline{F} = \mathrm{Lie}(U_\alpha)$. 
Pour chaque racine $\alpha \in \ES{R}$, choisissons un $\mbb{Z}$-isomorphisme $\mathfrak{e}_\alpha : \mbb{G}_{\mathrm{a}/\mbb{Z}} \rightarrow \scr{U}_\alpha$; observons que le $\mbb{Z}$-isomorphisme $\mathrm{Lie}(\mathfrak{e}_\alpha): \mathrm{Lie}(\mbb{G}_{\mathrm{a}/\mbb{Z}})\rightarrow \mathrm{Lie}(\scr{U}_\alpha)$ est uniquement dŽterminŽ au signe prs.  
Supposons que le $F$-isomorphisme $e_\alpha: \mbb{G}_{\mathrm{a}/F} \rightarrow U_\alpha$ fixŽ plus haut (\S2) 
prolonge $\mathfrak{e}_\alpha$, \textit{i.e.} que $e_\alpha = \mathfrak{e}_\alpha \times_{\mbb{Z}} F$. 
L'ŽlŽment $E_\alpha = \mathrm{Lie}(e_\alpha)(1)$ appartient ˆ $\mathrm{Lie}(\scr{U}_\alpha)$ et on a $\mathrm{Lie}(\scr{U}_\alpha) = \mbb{Z}E_\alpha$. 
Cela entra"ne aussi que les relations de commutateurs $(e_\alpha(x),e_\beta(y))$ pour $\alpha,\, \beta \in \ES{R}$, $\beta \neq -\alpha$, sont ˆ coefficients dans $\mbb{Z}$.  }
\end{remark}

On peut aussi construire $\scr{G}$ ˆ partir d'un systme de Chevalley dans l'algbre de Lie du groupe algŽbrique sur $\mbb{C}$ associŽ ˆ $G$ (c'est en gros la mŽthode de Chevalley\footnote{La notion mme de rŽseau admissible dans un $G'$-module $V'$ est liŽe ˆ la 
construction de $\scr{G}$.}). \`A la donnŽe radicielle $\Phi= (X(A_0),\ES{R},\check{X}(A_0),\check{\ES{R}})$ est associŽe 
un $\mbb{C}$-groupe algŽbrique rŽductif connexe $G'$ muni d'une donnŽe radicielle $\Phi'=(X(A'_0),\ES{R},\check{X}(A'_0),\ES{R})$ et d'un isomorphisme 
$\Phi \simeq  \Phi'$, \textit{i.e.} un isomorphisme de $\mbb{Z}$-modules $\iota: X(A_0)\rightarrow X(A'_0)$ tel que $\iota(\ES{R})=\ES{R}'$ et $\check{\iota}(\check{\ES{R}}')= \check{\ES{R}}$. 
Posons $\mathfrak{g}'= \mathrm{Lie}(G)$ et $\mathfrak{a}'_0= \mathrm{Lie}(A'_0)$. Pour $\alpha \in \ES{R}$ identifiŽ ˆ $\ES{R}'$ via $\iota$, notons $U'_\alpha$ le sous-groupe radiciel de $G'_\alpha$ associŽ ˆ $\alpha$ et posons $\mathfrak{u}'_\alpha = \mathrm{Lie}(U'_\alpha)$. 
Notons $H'_\alpha$ l'unique ŽlŽment de $\mathfrak{a}'_\alpha = [\mathfrak{u}'_\alpha,\mathfrak{u}'_{-\alpha}] \;(\subset \mathfrak{a}'_0)$ tel que $\alpha(H'_\alpha)=2$. 
Pour chaque $\alpha\in \ES{R}$, fixons un ŽlŽment non nul $E'_\alpha \in \mathfrak{u}_\alpha(\mbb{C})$. On suppose, ce qui est loisible d'aprs \cite[ch.~VIII, \S4, n\textsuperscript{o}4, cor.~de la prop.~5]{Bou}, que la famille $(E'_\alpha)_{\alpha\in \ES{R}}$ est un systme de Chevalley de 
$\mathfrak{g}'$ au sens de \cite[ch.~VIII, \S2, n\textsuperscript{o}4]{Bou}: 
\begin{itemize}
\item pour tout $\alpha \in \ES{R}$, on a $[E'_\alpha,E'_{-\alpha}]= -H'_\alpha$;
\item l'application linŽaire de $\mathfrak{g}'$ dans $\mathfrak{g}'$ qui vaut $-1$ sur $\mathfrak{a}'_0$ et envoie $E'_\alpha$ sur $E'_{-\alpha}$ pour tout $\alpha\in \ES{R}$, est un automorphisme de $\mathfrak{g}'$.
\end{itemize} 
Observons que $\mathfrak{a}'_0= \check{X}(A'_0)\otimes_{\mbb{Z}} \mbb{C}$ et que pour $\alpha \in \ES{R}$, l'ŽlŽment $H'_\alpha$ appartient ˆ $\check{X}(A'_0)$. 
D'aprs \cite[ch.~VIII, \S2, n\textsuperscript{o}2, prop.~8]{Bou}, le rŽseau $M'$ de $\mathfrak{g}'$ dŽfini par 
$$M'= \check{X}(A'_0) \oplus \left(\bigoplus_{\alpha \in \mbb{R}} \mbb{Z} E'_\alpha \right)$$  
est une sous-$\mbb{Z}$-algbre de Lie de $\mathfrak{g}'$ et l'application canonique de $M'\otimes_{\mbb{Z}}\mbb{C} \rightarrow \mathfrak{g}'$ est un isomorphisme. En particulier pour $\alpha,\, \beta \in \ES{R}$ tels que $\alpha + \beta \in \ES{R}$, l'ŽlŽment  
$[E'_\alpha,E'_\beta]$ appartient ˆ $\mbb{Z}E'_{\alpha + \beta} =M' \cap \mathfrak{u}'_{\alpha + \beta}$. D'aprs \cite[ch.~VIII, \S12, n\textsuperscript{o}7, theorem~2]{Bou},  pour $\alpha \in \ES{R}$, on a 
$[\mathfrak{u}'_\alpha, \mathfrak{u}'_{-\alpha}]= \mbb{Z}H_\alpha$; et pour $X\in \mathfrak{u}'_\alpha$ et $n\in \mbb{N}$, 
l'endormorphisme $\frac{1}{n\!\!\!!}(\mathrm{ad}\,X)^n$ de $\mathfrak{g}'$ stabilise $M'$. Cette dernire propriŽtŽ est essentielle pour la construction d'un 
$\mbb{Z}$-groupe rŽductif de Chevalley-Demazure $\scr{G}$ tel que $M'= \mathrm{Lie}(\scr{G})$.  

\vskip2mm
\textit{\S4.\hskip1mm L'hypothse \ref{hypo bonnes F-strates faible} pour $G$. ---} 
Le passage du groupe ˆ l'algbre de Lie permet aussi de prouver l'hypothse \ref{hypo bonnes F-strates faible} pour la $G$-variŽtŽ $G$, \textit{i.e.} pour les Žlements 
de $\UF$, ˆ partir de l'hypothse \ref{hypo bonnes F-strates faible} pour le $G$-module $\mathfrak{g}$, \textit{i.e.} pour les ŽlŽments de $\NN_F$.

\begin{lemme}\label{hyp bonnes F-strates et passage ˆ Lie(G)}
L'hypothse \ref{hypo bonnes F-strates faible} est vŽrifiŽe pour la $G$-variŽtŽ $G$ (munie de l'action par conjugaison) si et seulement si elle l'est pour 
le $G$-module $\mathfrak{g}$ (muni de l'action adjointe).
\end{lemme}

\begin{proof}
Par descente sŽparable, on peut supposer que $G$ est $F$-dŽployŽ. 
Pour $u\in \UU_F$, on choisit un $g\in G(F)$ tel que $u'=gug^{-1}$ soit en position standard et on pose 
$X'=j_0^{-1}(u')$. D'aprs la preuve de \ref{optimalitŽ et j}, on a 
$\bs{\Lambda}_{F,u'}= \bs{\Lambda}_{F,X'}$; et en prenant $F=\overline{F}$, on a aussi $\bs{\Lambda}_{u'}= \bs{\Lambda}_{X'}$. On en dŽduit que $\check{X}_F(G)_{\mbb{Q}} \cap \bs{\Lambda}_{u'}\neq \emptyset$ si et seulement si 
$\check{X}_F(G)_{\mbb{Q}}\cap \bs{\Lambda}_{X'}\neq \emptyset$. Puisque $\bs{\Lambda}_{u}= g^{-1}\bullet \bs{\Lambda}_{u'}$ et $\bs{\Lambda}_X= g^{-1}\bullet \bs{\Lambda}_{X'}$, cela prouve que si l'hypothse 
\ref{hypo bonnes F-strates faible} est vŽrifiŽe pour le $G$-module $\mathfrak{g}$, alors l'est pour la $G$-variŽtŽ $G$. La rŽciproque s'obtient de la mme manire.
\end{proof}

Commenons par le cas particulier suivant:

\begin{lemme}\label{hyp bonnes F-strates pour G dans le cas trs bon}
Si pour tout $X\in \NN_F$, la $G$-orbite de $X$ est sŽparable (\textit{e.g.} si $p\geq 1$ est trs bon pour $G$, \textit{cf.} \ref{sŽparabilitŽ des orbites}), alors 
l'hypothse \ref{hypo bonnes F-strates faible} est vŽrifiŽe pour le $G$-module $\mathfrak{g}$ et pour la $G$-variŽtŽ $G$.
\end{lemme}

\begin{proof}
Supposons que tous les ŽlŽments de $\NN_F$ soient sŽparables. 
D'aprs \ref{projseprat}\,(ii), l'hypothse \ref{hypo bonnes F-strates faible} est vŽrifiŽe pour le $G$-module $\mathfrak{g}$: pour tout $X\in \NN_F$, on a $\check{X}_F(G)_{\mbb{Q}}\cap \bs{\Lambda}_X \neq \emptyset$. 
D'o le rŽsultat d'aprs \ref{hyp bonnes F-strates et passage ˆ Lie(G)}. 
\end{proof}

\begin{exemple}\label{hypo BFSF pour GLn}
\textup{Si $G=\mathrm{GL}_{n/F}$, l'application $\mathfrak{g} \rightarrow G,\, X\mapsto 1+X$ induit une bijection entre les $G$-orbites de $\mathfrak{g}$, resp. $\mathfrak{N}$, et les $G$-orbites de $G$, resp. $\UU$. De 
plus pour tout $X\in \mathfrak{N}$, le centralisateur schŽmatique $G^X=G^{1+X}$ est rŽduit; en particulier les $G$-orbites $\ES{O}_X$ et $\ES{O}_{1+X}= 1 + \ES{O}_X$ sont sŽparables. Le lemme  \ref{hyp bonnes F-strates pour G dans le cas trs bon} 
assure donc que l'hypothse \ref{hypo bonnes F-strates faible} est vŽrifiŽe pour $G$. }
\end{exemple}

Le cas suivant est celui qui nous servira par la suite:

\begin{proposition}\label{hypo bonnes F-strates faible pour Lie(G) et G}
Soit $(F,\nu)$ un corps valuŽ admissible (\textit{cf.} \ref{le cas d'un corps top}) tel que le complŽtŽ 
$\wh{F}$ de $F$ en $\nu$ soit un corps local vŽrifiant la propriŽtŽ $(\wh{F}^{\mathrm{s\acute{e}p}})^{\mathrm{Aut}_F(\wh{F}^{\mathrm{s\acute{e}p}})}=F$; 
\textit{e.g.} un corps local $F=\wh{F}$ ou un corps global $F$ muni d'une place finie $\nu$ (\textit{cf.} \ref{sŽparabilitŽ du complŽtŽ}). L'hypothse \ref{hypo bonnes F-strates faible} est vŽrifiŽe pour 
le $G$-module $\mathfrak{g}$ et pour la $G$-variŽtŽ $G$.
\end{proposition}

\begin{proof}
Pour $X\in \NN_F$, on a $\bs{\Lambda}_{F,X}= \check{X}_F(G)_{\mbb{Q}} \cap \bs{\Lambda}_{\smash{\wt{F},X}}$; par consŽquent si $\check{X}_{\smash{\wt{F}}}(G)_{\mbb{Q}}\cap \bs{\Lambda}_X\neq \emptyset$, auquel cas 
$\bs{\Lambda}_{\smash{\wt{F},X}}= \check{X}_{\smash{\wt{F}}}(G)_{\mbb{Q}}\cap \bs{\Lambda}_X$, alors $\check{X}_F(G)\cap \bs{\Lambda}_X \neq \emptyset$. 
Quitte ˆ remplacer $(F,\nu)$ par $(\wt{F},\wt{\nu})$ o $\wt{\nu}$ est l'unique valuation de $\wt{F}$ prolongeant $\nu$, on peut donc supposer que $G$ est dŽployŽ sur $F$ (\textit{i.e.} on peut supposer que $\wt{F}=F$). On suppose aussi que $G$ provient par le changement de base $\mbb{Z} \rightarrow F$ d'un 
$\mbb{Z}$-groupe rŽductif de Chevalley-Demazure $\scr{G}$ (\textit{cf.} \ref{le cas d'un corps top}), \textit{i.e.} $G = \scr{G}\times_{\mbb{Z}} F$. 
On a vu (\S4) que le $G$-module $\mathfrak{g}$ est admissible: l'action adjointe de $\scr{G}$ sur le $\mbb{Z}$-foncteur $\mathfrak{g}_{\mbb{Z}}= \mathrm{Lie}(\scr{G})_{\mathrm{a}}$ 
le munit d'une structure de $\scr{G}$-module. On peut donc appliquer \ref{cor 3 bonnes F-strates}: l'hypothse 
\ref{hypo bonnes F-strates faible} est vŽrifiŽe pour le $G$-module $\mathfrak{g}$. Elle l'est donc aussi pour la $G$-variŽtŽ $G$ (d'aprs \ref{hyp bonnes F-strates et passage ˆ Lie(G)}).
\end{proof}

\begin{remark}\label{description explicite d'une F-lame standard} 
\textup{(Avec les hypothses de la dŽmonstration de \ref{hypo bonnes F-strates faible pour Lie(G) et G}.) Pour $\lambda\in \check{X}(A_0)$ et 
$k\in \mbb{N}^*$, notons $\mathfrak{g}_{\lambda}(k)_{\mbb{Z}}=M'_{\mathrm{a}}$ le $\mbb{Z}$-espace affine dŽfini par le $\mbb{Z}$-module libre de rang fini $M'=\bigoplus_{\alpha\in \ES{R}_\lambda(k)}\mbb{Z} E_\alpha$ 
avec (rappel) $\ES{R}_\lambda(k)= \{\alpha \in \ES{R}\,\vert \, \langle \alpha , \lambda \rangle = k \}$. C'est un sous-schŽma fermŽ de $\mathfrak{g}_{\mbb{Z}}$ tel que  
$\mathfrak{g}_\lambda(k)_{\mbb{Z}}\times_{\mbb{Z}}F = \mathfrak{g}_\lambda(k)$. De plus $\mathfrak{g}_\lambda(k)_{\mbb{Z}}$ est $\scr{M}_\lambda$-stable et la structure de $M_\lambda$-module (dŽfini sur $F$) sur $\mathfrak{g}_\lambda(k)$ provient par le changement de base $F\rightarrow \mbb{Z}$ de l'opŽration linŽaire de 
$\scr{M}_\lambda$ sur $\mathfrak{g}_\lambda(k)_{\mbb{Z}}$. Cela s'applique en particulier au cas o $\lambda\in \Lambda_{F,X}^{\mathrm{opt}}$ et $k=m_X(\lambda)$ 
pour un ŽlŽment $X\in \mathfrak{N}_F$ tel que $\check{X}(A_0)\cap \Lambda_{F,X}^{\mathrm{opt}}\neq \emptyset$. Puisque $\lambda$ appartient ˆ $\Lambda_X^{\mathrm{opt}}$, en notant $\mathfrak{g}_{\lambda}(k)_{\mbb{Z}}^{\mathrm{ss}}$ le sous-schŽma ouvert de $\mathfrak{g}_\lambda(k)_{\mbb{Z}}$ 
tel que pour tout corps $\bs{k}$ algŽbriquement clos, $\mathfrak{g}_{\lambda}(k)_{\mbb{Z}}^{\mathrm{ss}}(\bs{k})$ soit l'ensemble des ŽlŽments $(\bs{k},\scr{M}_\lambda^\perp\times_{\mbb{Z}}\bs{k})$-semi-stables de $\mathfrak{g}_{\lambda}(k)_{\mbb{Z}}(\bs{k})$, on a 
$$\scrY_{F,X}= \mathfrak{g}_{\lambda}(k)_{\mbb{Z}}^{\mathrm{ss}}(F)\oplus \mathfrak{g}_{\lambda, k+1}(F)\ptf$$}
\end{remark}

Pour un corps commutatif $F$ quelconque, on a le 

\begin{lemme}\label{k=1 ou 2}
Si l'hypothse \ref{hypo bonnes F-strates faible} est vŽrifiŽe pour le $G$-module $\mathfrak{g}$, alors pour tout 
$X\in \NN_F$, on a $m_{F,X}\in \{1,2\}$; et pour tout $u\in \UU_F$, on a $m_{F,u}\in \{1,2\}$. 
\end{lemme}

\begin{proof}
Par descente sŽparable, on peut supposer que $G$ est $F$-dŽployŽ. Il suffit de traiter les ŽlŽments $X\in \NN_F$ tels que $X(A_0)\cap \Lambda_{F,X}^{\mathrm{opt}}\neq \emptyset $ (on suppose toujours que $A_0= \scr{A}_0\times_{\mbb{Z}}F$). Soit $X\in \NN_F$ 
un tel ŽlŽment. Posons $\check{X}(A_0)\cap \Lambda_{F,X}^{\mathrm{opt}}=\{\lambda\}$ et $k= m_X(\lambda)$. Reprenons les notations de la remarque 
\ref{description explicite d'une F-lame standard}. Puisque par hypothse $\lambda$ appartient ˆ $\Lambda_X^{\mathrm{opt}}$, 
d'aprs le critre de Kirwan-Ness gŽomŽtrique, l'ŽlŽment $X$ est $(\overline{F},M_\lambda^\perp)$-semi-stable, \textit{i.e.} 
$X$ appartient ˆ l'ouvert $ \mathfrak{g}_\lambda(k)_{\mathbb{Z}}^{\mathrm{ss}}(\overline{F})$ de $\mathfrak{g}_\lambda(k)(\overline{F})$. 
D'aprs \ref{seshadri}, il existe un ŽlŽment $X'\in \mathfrak{g}_\lambda(k)_{\mbb{Z}}^{\mathrm{ss}}(\mbb{C})$. \`A nouveau d'aprs le critre de Kirwan-Ness gŽomŽtrique, 
le co-caractre $\lambda$ appartient ˆ $ \Lambda_{\mbb{C},X'}^{\mathrm{opt}}$. 
D'aprs la thŽorie des $\mathfrak{sl}_2$-triplets de Jacobson-Morosov, 
il existe un co-caractre $\lambda'=\lambda_{\epsilon'} \in \check{X}_{\mbb{C}}(\scr{G}\times_{\mbb{Z}}{\mbb{C}})$ tel que $\frac{1}{2}\lambda'\in \bs{\Lambda}_{\mbb{C},X'}$ 
(\textit{cf.} \S3). Les deux co-caractres virtuels $\frac{1}{k}\lambda$ et $\frac{1}{2}\lambda'$ appartiennent ˆ $\bs{\Lambda}_{\mbb{C},X'}$ et l'on sait 
que $U_\lambda(\mbb{C})=U_{\lambda'}(\mbb{C})$ opre simplement transitivement sur $\bs{\Lambda}_{\mbb{C},X'}$. Par consŽquent il existe un (unique) $u' \in U_{\lambda}(\mbb{C})$ tel que 
$u'\bullet \frac{1}{k}\lambda = \frac{1}{2}\lambda'$. Donc $\frac{2}{k}\lambda= u'^{-1}\bullet \lambda'$ appartient ˆ $\check{X}(A_0)$, ce qui n'est possible 
que si $k\in \{1,2\}$. D'o le lemme.
\end{proof}

\begin{remark}\label{k=1 ou 2 pour F local ou global}
\textup{Supposons que le corps $F$ vŽrifie les hypothses de \ref{hypo bonnes F-strates faible pour Lie(G) et G}. 
Alors on peut prouver \ref{k=1 ou 2} sans utiliser le thŽorme de Seshadri (\ref{seshadri}). En effet, soient $X$, $\lambda$ et $k$ 
comme dans la dŽmonstration de \ref{k=1 ou 2}. D'aprs la preuve de \ref{F-ss implique gŽo-ss}, il existe un corps $F'$ de caractŽristique nulle et un ŽlŽment nilpotent $X'\in \mathfrak{g}_\lambda(k)_{\mbb{Z}}(F')$ tels que $\lambda$ appartienne 
ˆ $\Lambda_{F'\!,X'}^{\mathrm{opt}}$ avec $k= m_{X'}(\lambda)$. D'aprs \cite[11.6, 12]{H1},  
il existe un co-caractre $\lambda'\in \check{X}_{F'}(G)$ tel que $\frac{1}{2}\lambda'\in \bs{\Lambda}_{F'\!,X'}$. 
Puisque les deux co-caractres virtuels $\frac{1}{k}\lambda$ et $\frac{1}{2}\lambda'$ appartienne ˆ $\bs{\Lambda}_{F'\!,X'}$ et que $U_{\lambda}(F')=U_{\lambda'}(F')$ 
opre simplement transitivement sur $\bs{\Lambda}_{F'\!,X'}$, il existe un (unique) $u' \in U_{\lambda}(F')$ tel que 
$u'\bullet \frac{1}{k}\lambda = \frac{1}{2}\lambda'$. On conclut comme dans la preuve de \ref{k=1 ou 2}. 
} 
\end{remark}

\subsection{$\overline{F}$-strates et orbites gŽomŽtriques}\label{strates et orbites} 
Dans cette sous-section, on s'intŽresse ˆ la thŽorie gŽomŽtrique, \text{i.e.} sur $\overline{F}$; on suppose donc que $F=\overline{F}$. 

\vskip2mm
\textit{\S1.\hskip1mm La $G$-orbite ouverte $\ES{O}(\bsfrY)\subset \bsfrY$.} 
Puisqu'il n'y a qu'un nombre fini de $G$-orbites unipotentes \cite{L1} et qu'elles sont localement fermŽes dans $G$ (\textit{i.e.} dans $\UU$), 
pour toute ($\overline{F}$-)strate $\bsfrY$ de $\UU$, 
il existe une unique $G$-orbite de $\UU$ qui soit (contenue et) dense dans la variŽtŽ irrŽductible 
$\bsfrY$; on la note $\ES{O}(\bsfrY)$. La $G$-orbite $\ES{O}(\bsfrY)$ est ouverte dans $\bsfrY$; elle l'est donc aussi dans sa fermeture $\bsfrX(\bsfrY)=\overline{\bsfrY}$. 
Pour toute ($\overline{F}$-)lame $\scrY$ de $\UU$ contenue dans une ($\overline{F}$-)strate $\bsfrY$ de $\UU$, $\ES{O}(\bsfrY)$ est l'unique $G$-orbite de $\UU$ telle que l'intersection $\bsfrY \cap \scrX(\scrY)$ soit dense dans la variŽtŽ irrŽductible $\scrX(\scrY)=\overline{\scrY}$; et cette intersection est ouverte dans $\scrX(\scrY)$. 

Pour toute strate $\bsfrY$ de $\UU$, on a la double inclusion 
$$\ES{O}(\bsfrY) \subset \bsfrY \subset \bsfrX(\bsfrY)= \overline{\bsfrY}$$ et 
la fermeture $\overline{\ES{O}(\bsfrY)}$ de $\ES{O}(\bsfrY)$ dans $\UU$ co\"{\i}ncide avec $\bsfrX(\bsfrY)$. 
En particulier l'ensemble $\bsfrX(\bsfrY) \smallsetminus \ES{O}(\bsfrY)$ est rŽunion (finie) de $G$-orbites de dimension 
strictement infŽrieure ˆ celle de $\ES{O}(\bsfrY)$. Pour tout strate $\bsfrY'$ de $\UU$ telle que $\bsfrY'\cap \bsfrX(\bsfrY)\neq \emptyset$, on a 
$$ \ES{O}(\bsfrY') \subset \bs{X}(\bsfrY)\Leftrightarrow \bsfrY'\subset \bsfrX(\bsfrY)  \Leftrightarrow \bsfrX(\bsfrY')\subset \bsfrX(\bsfrY) \ptf$$ 
Ainsi, pour la $G$-variŽtŽ $G$ munie de l'action par conjugaison\footnote{Et plus gŽnŽralement pour toute $G$-variŽtŽ affine pointŽe $(V,e_V)$ vŽrifiant la propriŽtŽ de rŽgularitŽ 
\ref{hyp reg} et telle que toute strate de $\bsfrY$ de $\ESN^G(V,e_V)$ contienne une (unique) $G$-orbite ouverte.}, l'hypothse sur le bord \ref{hypo sur le bord} 
pour $F=\overline{F}$ est Žquivalente ˆ: 
pour toutes strates $\bsfrY$ et $\bsfrY'$ de $\UU$ telles que $\bsfrY' \cap \bsfrX(\bsfrY)\neq \emptyset$, 
on a l'inclusion $\ES{O}(\bsfrY')\subset \bsfrX(\bsfrY)$. 

De la mme manire, puisqu'il n'y a qu'un nombre fini de $G$-orbites nilpotentes (supŽrieur ou Žgal au nombre des $G$-orbites unipotentes) et qu'elles sont localement fermŽes dans $\mathfrak{g}$ (\textit{cf.} \cite{LiSe}), si $\bsfrY$ est une ($\overline{F}$-)strate de $\NN$, on note $\ES{O}(\bsfrY)$ l'unique $G$-orbite qui soit (contenue et) dense dans $\bsfrY$. 
Les propriŽtŽs ci-dessus sont vraies si l'on remplace la $G$-variŽtŽ $G$ munie de l'action par conjugaison par le $G$-module $\mathfrak{g}$ muni de l'action adjointe. 
En particulier l'hypothse sur le bord \ref{hypo sur le bord} 
pour $F=\overline{F}$ est Žquivalente ˆ: 
pour toutes $\overline{F}$-strates $\bsfrY$ et $\bsfrY'$ de $\NN$ telles que $\bsfrY' \cap \bsfrX(\bsfrY)\neq \emptyset$, 
on a l'inclusion $\ES{O}(\bsfrY')\subset \bsfrX(\bsfrY)$. 

\vskip2mm
\textit{\S2.\hskip1mm ParamŽtrisation $\UU/G \leftrightarrow D_G/G$ pour $F= \mbb{C}$.} 
Supposons momentanŽment que $F=\mbb{C}$. Soit\index{DG@$D_G$} $D_G$ l'ensemble des co-caractres $\lambda \in \check{X}(G_{\mathrm{der}})$ tels qu'il existe un $\mbb{C}$-morphisme de groupes $\epsilon: \mathrm{SL}_{2/\mbb{C}} \rightarrow G_{\mathrm{der}}$ vŽrifiant $\lambda = \lambda_\epsilon$ avec $t^{\lambda_\epsilon} = \epsilon\left(\begin{array}{cc} t & 0\\ 0 & t^{-1} \end{array}\right)$. L'ensemble $D_G$ est $G$-stable pour l'action de $G$ sur $\check{X}(G_{\mathrm{der}})$ donnŽe par $(g,\lambda) \mapsto g\bullet \lambda =  \mathrm{Int}_{g}\circ \lambda$. On sait (\textit{cf.} \cite[1.1]{CP}) que l'application qui ˆ $\lambda_\epsilon \in D_G$ associe l'ŽlŽment $u_\epsilon = \epsilon \left( \begin{array}{cc} 1& 1\\ 0 & 1 \end{array}\right) $ de $\UU$ induit une bijection entre:
\begin{itemize}
\item l'ensemble $D_G/G$ des $G$-orbites dans $D_G$;
\item l'ensemble $\UU/G$ des $G$-orbites unipotentes dans $G$.
\end{itemize}
D'autre part en posant $\wt{\lambda}= \frac{1}{2}\lambda \in \check{X}(G)_{\mbb{Q}}$, on a vu (\ref{du groupe ˆ l'algbre de Lie}, \textit{\S.\hskip1mm Lien avec le thŽorie des $\mathfrak{sl}_2$-triplets}) que $\wt{\lambda}$ appartient ˆ $\bs{\Lambda}_{\mbb{C},u_\epsilon}$. 
On en dŽduit que l'application qui ˆ $\lambda \in D_G$ associe le sous-ensemble 
$$\bsfrY^G(\wt{\lambda})= \{u\in \UU \,\vert \, \bs{\Lambda}_{\mbb{C},u} \subset  G \bullet \wt{\lambda}\}\subset \UU$$ 
induit une bijection entre: 
\begin{itemize}
\item l'ensemble $D_G/G$ des $G$-orbites dans $D_G$;
\item l'ensemble des $\mbb{C}$-strates de $\UU$. 
\end{itemize}
Par consŽquent les $\mbb{C}$-strates de $\UU$ co\"{\i}ncident avec les $G$-orbites unipotentes. 
D'aprs \ref{bijection strates unip-strates nilp} (pour $F= \mbb{C}$), on en dŽduit que l'application qui 
ˆ $\lambda \in D_G$ associe le sous-ensemble 
$$\bsfrY^{\mathfrak{g}}(\wt{\lambda})= \{X\in \NN \,\vert \, \bs{\Lambda}_{\mbb{C},X} \subset  G \bullet \wt{\lambda}\}\subset \NN$$ 
induit une bijection entre: 
\begin{itemize}
\item l'ensemble $D_G/G$ des $G$-orbites dans $D_G$;
\item l'ensemble des $\mbb{C}$-strates de $\NN$. 
\end{itemize}
Comme d'autre part on a une bijection canonique entre les $G$-orbites dans $\UU$ et les $G$-orbites dans $\NN$ 
(donnŽe par un isomorphisme de Springer), on en dŽduit que les $\mbb{C}$-strates de $\NN$ co\"{\i}ncident avec les $G$-orbites nilpotentes.

\vskip2mm
\textit{\S3.\hskip1mm ParamŽtrisation $\bs{\Lambda}_{\mathrm{st}} \leftrightarrow \UU^{G'}\!/G'$ pour $F=\overline{F}$.} 
Revenons ˆ $F= \overline{F}$ quelconque. On peut supposer que $G= \scr{G}_{\overline{F}}$ pour un $\mbb{Z}$-groupe rŽductif de Chevalley-Demazure $\scr{G}$. 
Posons $G'= \scr{G}_{\mbb{C}}$ et $\mathfrak{g}'= \mathrm{Lie}(G')$. On a $\mathfrak{g}= \mathrm{Lie}(\scr{G})\otimes_{\mbb{Z}} \overline{F}$ et 
$\mathfrak{g}'= \mathrm{Lie}(\scr{G})\otimes_{\mbb{Z}}\mbb{C}$. On a vu en \ref{du groupe ˆ l'algbre de Lie} (\S4) que $\mathrm{Lie}(\scr{G})$ est rŽseau admissible dans le $G'$-module 
$\mathfrak{g}'$. D'aprs \ref{F versus C}, on a une bijection canoniques entre les strates de $\NN = \ES{N}^G(\mathfrak{g},0)$ et les strates de $\NN'= \ES{N}^{G'\!}(\mathfrak{g}'\!,0)$. On en dŽduit une bijection canonique entre les strates de $\UU= \ES{N}^G(G,1) $ et les strates de $\UU'= \ES{N}^{G'\!}(G'\!,1)$. D'autre part on a vu que les strates de $\UU'$, resp. $\NN'$, co\"{\i}ncident avec les $G'$-orbites unipotentes, resp. nilpotentes. 
Concrtement, fixons une paire de Borel $(\scr{P}_0,\scr{A}_0)$ de $\scr{G}$. Posons $(P_0,A_0)=(\scr{P}_0\times_{\mbb{Z}} \overline{F}, \scr{A}_0\times_{\mbb{Z}}\overline{F})$ et  
$(P'_0,A'_0)=(\scr{P}_0\times_{\mbb{Z}} \mbb{C}, \scr{A}_0\times_{\mbb{Z}}\mbb{C})$. Fixons aussi une norme $W^{\scr{G}}(\scr{A}_0)$-invariante sur $\check{X}(\scr{A}_0)$. Ces donnŽes permettent de dŽfinir les sous-ensembles 
$$\bs{\Lambda}_{\mathrm{st}}= \bs{\Lambda}_{\overline{F},\mathrm{st}}^G(\mathfrak{g},0)\subset \check{X}(A_0)_{\mbb{Q}}\quad\hbox{et}\quad \bs{\Lambda}'_{\mathrm{st}}= \bs{\Lambda}_{\smash{\mbb{C}},\mathrm{st}}^{G'}(\mathfrak{g}'\!,0)\subset X(A'_0)_{\mbb{Q}}\ptf$$ 
D'aprs \textit{loc.~cit.}, on a $\bs{\Lambda}_{\mathrm{st}}= \bs{\Lambda}'_{\mathrm{st}} \; (= \wt{D}_{G'}/G'$ avec $\wt{D}_{G'}=\{\frac{1}{2}\lambda 
\,\vert \, \lambda \in D_{G'}\})$ et pour $\mu \in \bs{\Lambda}_{\mathrm{st}}= \bs{\Lambda}'_{\mathrm{st}}$, 
la strate $\bsfrY^{\mathfrak{g}}(\mu)$ de $\NN$ correspond ˆ la $G'$-orbite $\bsfrY^{\mathfrak{g}'}(\mu)$ de $\NN'$; et 
la strate $\bsfrY^{G}(\mu)$ de $\UU$ correspond ˆ la $G'$-orbite $\bsfrY^{G'}(\mu)$ de $\UU'$.
   
 \vskip2mm
\textit{\S3.\hskip1mm Le cas o $p$ est \guill{bon} pour $G$. ---} 
Si $p=1$ ou $p\gg 1$, la thŽorie sur $\overline{F}$ est pratiquement la mme que sur $\mbb{C}$: les strates de $\UU$, resp. $\NN$, co\"{\i}ncident avec les 
$G$-orbites unipotentes, resp. nilpotentes. Pour prŽciser cette affirmation, il convient de rappeler la  

\begin{definition}[rappel]\label{p bon et p trs bon} 
\textup{
\begin{enumerate}
\item[(i)]Pour $G$ (absolument) 
quasi-simple, on dit que $p\geq 1$ est \guill{bon} pour $G$ si $p=1$ ou si $p>1$ ne divise aucun c\oe fficient de la plus grande racine (exprimŽe comme combinaison linŽaire de 
racines simples) du systme de racines de $G$. Ainsi les \guill{mauvais} $p>1$ pour $G$, \cad ceux qui ne sont pas bons, sont:
\begin{itemize}
\item aucun si $G$ est de type ${\bf A}_n$; 
\item $p=2$ si $G$ n'est pas de type ${\bf A}_n$; 
\item $p=3$ si $G$ est de type exceptionnel (${\bf G}_2$, ${\bf F}_4$ ou ${\bf E}_*$);
\item $p=5$ si $G$ est de type ${\bf E}_8$.
\end{itemize}
\item[(ii)]Pour $G$ rŽductif connexe quelconque, on dit que $p$ est \guill{bon} pour $G$ 
si pour tout sous-groupe distinguŽ quasi-simple $H$ de $G$, $p$ est bon pour $H$. 
\item[(iii)]Pour $G$ comme en (ii), on dit que $p$ est \guill{trs bon} pour $G$ si $p=1$ ou si $p>1$ est bon pour $G$ et pour tout sous-groupe distinguŽ quasi-simple 
de $G$ de type ${\bf A}_n$, $p$ ne divise pas $n+1$. Observons que cette dŽfinition (aujourd'hui standard) n'est pas celle de Morris dans \cite[3.13]{M}. 
\end{enumerate}
}
\end{definition}

\begin{lemme}[\cite{CP,L2}]\label{p bon, strate-orbite}
Supposons que $p$ soit bon pour $G$. Les $\overline{F}$-strates de $\UU$ co\"{\i}ncident avec les $G$-orbites unipotentes et les $\overline{F}$-strates 
de $\NN$ co\"{\i}ncident avec les $G$-orbites nilpotentes. 
\end{lemme}

\begin{remark}\label{p=1 et p bon}
\textup{\begin{enumerate}
\item[(i)]Supposons que $p\geq 1$ soit bon pour $G$. Pour un ŽlŽment unipotent $u$ de $G$, la $\overline{F}$-strate unipotente $\bsfrY_u$ co\"{\i}ncide avec l'orbite gŽomŽtrique $\ES{O}_u= \{gug^{-1}\,\vert g\in G\}$, 
caractŽrisŽe par la propriŽtŽ d'tre l'unique $G$-orbite ouverte dans $\bsfrX_u$; et $\scrY_u$ est l'unique $P_u$-orbite ouverte dans $\scrX_u$. 
Observons que le bord $\bsfrX_u\smallsetminus \bsfrY_u = \overline{\ES{O}_u} \smallsetminus \ES{O}_u$ est rŽunion (finie) de $\overline{F}$-strates, \cad de $G$-orbites unipotentes, et 
l'hypothse \ref{hypo sur le bord} est vŽrifiŽe (pour $F=\overline{F}$). Si de plus $\scrX_u=U_{P_u}$, alors $\scrY_u$ est 
la $P_u$-orbite de Richardson dans $U_{P_u}$. 
\item[(ii)]Si $p>1$ n'est pas bon pour $G$, une $\overline{F}$-strate de $\UU$ peut contenir plusieurs $G$-orbites unipotentes (\textit{cf.} l'exemple de $\mathrm{Sp}_4$ en \ref{exemples en basse dimension}); idem pour les $\overline{F}$-strates de $\NN$ et les $G$-orbites nilpotentes. De plus les $G$-orbites unipotentes ne sont en gŽnŽral pas en bijection avec les $G$-orbites nilpotentes (le cardinal (fini) de  $\NN/G$ peut tre strictement supŽrieur ˆ celui de $\UU/G$).
\end{enumerate}}\end{remark}

\vskip2mm
\textit{\S4.\hskip1mm Comparaison avec les \guill{morceaux unipotents} de Lusztig-Xue. ---} La thŽorie des \textit{morceaux unipotents} de $G$ a ŽtŽ initiŽe par Lusztig pour traiter les mauvais $p$, \cad ceux qui ne sont pas \guill{bons} pour $G$. D'aprs \ref{p bon et p trs bon}, ces 
mauvais $p$ appartiennent ˆ $\{2,3,5\}$ et si $G$ est un groupe classique, seul $p=2$ est mauvais pour $G$. La dŽfinition des morceaux unipotents ayant ŽvoluŽ au cours du temps, il convient de la rappeler prŽcisŽment. Continuons avec les notations prŽcŽdentes: $G= \scr{G}_{\overline{F}}$ et $G'= \scr{G}_{\mbb{C}}$ pour un $\mbb{Z}$-groupe rŽductif de Chevalley-Demazure $\scr{G}$, etc. L'ensemble $\check{X}(G)/G$ des $G$-orbites de co-caractres de $G$ est naturellement en bijection avec l'ensemble $\check{X}(\scr{A}_0)/W^{\scr{G}}(\scr{A}_0)$ et donc aussi avec $\check{X}(G')/G'$. Notons\index{DG@$D_G$} $D_G\subset \check{X}(G)$ le sous-ensemble ($G$-invariant) formŽ des ŽlŽments dont la $G$-orbite correspond ˆ un ŽlŽment de $D_{G'}/G' \;(\subset \check{X}(G')/G') $ par cette bijection. En posant $\wt{D}_G= \{\frac{1}{2}\lambda \,\vert \, \lambda \in D_G\}$, on a les identifications $$\wt{D}_G/G = \bs{\Lambda}_{\mathrm{st}}= \bs{\Lambda}'_{\mathrm{st}} = \wt{D}_{G'}/G'\ptf$$ Pour $\lambda\in \check{X}(G)$, on a le $\overline{F}$-sous-groupes fermŽ $G_{\lambda,2}$ de $G$. Pour une $G$-orbite $\blacktriangle$ dans $D_G$, posons $$G_{\blacktriangle ,2}= \bigcup_{\lambda \in \blacktriangle} G_{\lambda,2}\quad \hbox{et}\quad  
G^{\blacktriangle} = G_{\blacktriangle,2} \smallsetminus \bigcup_{\blacktriangle'} G_{\blacktriangle'\!,2}$$ o $\blacktriangle'$ parcourt les ŽlŽments de $D_G/G$ tels que $G_{\blacktriangle'\!,2} \subsetneq G_{\blacktriangle,2}$. Les ensembles $G^{\blacktriangle}$ ont ŽtŽ introduits par Lusztig dans \cite{L4}; ils les appellent \guill{morceaux} unipotents. 
Ces morceaux unipotents $G^\blacktriangle$ ainsi que leur variante pour l'algbre de Lie, les morceaux nilpotents $\mathfrak{g}^{\blacktriangle}$ (mme dŽfinition, mutatis mutandis),  sont ceux utilisŽs dans \cite{CP}. On les appellera dŽsormais \guill{CP-morceaux} unipotents, resp. nilpotents. Observons que la dŽfinition des CP-morceaux est rŽminiscente de l'ordre partiel naturel sur les orbites gŽomŽtriques ($\ES{O}_1 \preccurlyeq \ES{O}_2 \Leftrightarrow \ES{O}_1 \subset \overline{\ES{O}_2}$). Dans \cite{L4}, Lusztig traite le cas des groupes symplectiques; il montre en particulier que pour ces groupes, les morceaux unipotents forment une partitions de $\UU$. D'autre part le thŽorme de \cite[5.2]{CP} dit que pour n'importe quel groupe $G$, les morceaux unipotents  
co\"{\i}ncident avec les strates de $\UU$: pour $\blacktriangle \in D_G/G$, on a $G^{\blacktriangle} = \bsfrY^{G}(\wt{\blacktriangle})$ avec $\wt{\blacktriangle}=\frac{1}{2}\blacktriangle$. Malheureusement la dŽmonstration de \textit{loc.~cit.} utilise la proposition de \cite[2.7]{CP} qui, pour une $G$-variŽtŽ pointŽe quelconque, est en gŽnŽral fausse (\textit{cf.} \ref{hypo sur le bord en dŽfaut}). \textit{A priori} on sait seulement que pour $\blacktriangle \in D_G/G$, on a $$G_{\blacktriangle,2}= \bsfrX^G(\wt{\blacktriangle})\bydef \overline{\bsfrY^G(\wt{\blacktriangle})} \quad \hbox{et}\quad \bsfrY^G(\wt{\blacktriangle})\subset G^{\blacktriangle}\ptf$$ On en dŽduit que pour que le thŽorme de \cite[5.2]{CP} soit vrai, il faut et il suffit que les conditions Žquivalentes suivantes soient vŽrifiŽes:
\begin{itemize}
\item les CP-morceaux unipotents $G^{\blacktriangle}$ forment une partition de $\UU$;
\item les strates de $\UU$ vŽrifient 
l'hypothse sur le bord \ref{hypo sur le bord} (pour $F= \overline{F})$. 
\end{itemize}
La mme observation vaut pour les CP-morceaux nilpotents et les strates nilpotentes. Nous y reviendrons plus loin.

Pour $\lambda \in D_G$, posons $$\mathfrak{g}_\lambda(2)^{\string!} = \{X\in \mathfrak{g}_\lambda(2)\,\vert \, G^X \subset P_\lambda\}\quad \hbox{et} \quad \mathfrak{g}_{\lambda,2}^{\,\string!}=\mathfrak{g}_{\lambda}(2)^{\string!} \oplus \mathfrak{g}_{\lambda,3}\ptf$$ Pour $\blacktriangle$, posons 
$\mathfrak{g}^{\blacktriangle, \mathrm{LX}}= \bigcup_{\lambda \in \blacktriangle} \mathfrak{g}_{\lambda,2}^{\,\string!}$. Les ensembles $\mathfrak{g}^{\blacktriangle, \mathrm{LX}}$ ont ŽtŽ introduits par Lusztig et Xue dans \cite{L5}; on les appelle ici \guill{LX-morceaux} nilpotents. La variante unipotente de ces LX-morceaux est dŽfinie via les $\overline{F}$-isomorphismes $M_\lambda$-Žquivariant $j_\lambda(2): \mathfrak{g}_\lambda(2) \buildrel\simeq\over{\longrightarrow} G_\lambda(2)= G_{\lambda,2}/G_{\lambda,3}$ (\textit{cf.} \ref{du groupe ˆ l'algbre de Lie}).   
Pour $\lambda \in D_G$, on note $G_{\lambda}(2)^{\string!}$ l'image de $\mathfrak{g}_\lambda(2)^{\string!}$ par $j_\lambda(2)$ et $G_{\lambda,2}^{\,\string!}$ l'image rŽciproque de $G_\lambda(2)^{\string!}$ par la projection naturelle $G_{\lambda,2} \rightarrow G_\lambda(2)$; et pour $\blacktriangle \in D_G/G$, on pose $G^{\blacktriangle, \mathrm{LX}} = \bigcup_{\lambda\in \blacktriangle} G_{\lambda,2}^{\,\string!}$. Rappelons que pour $u\in \UU$ et $\mu \in \bs{\Lambda}_u$, on a l'inclusion $G^u \subset P_\mu$. 
Par suite pour $\blacktriangle \in D_G/G$, on a l'inclusion $\bsfrY^{G}(\wt{\blacktriangle}) \subset G^{\blacktriangle, \mathrm{LX}}$. On en dŽduit que pour que les LX-morceaux unipotents co\"{\i}ncident avec les strates unipotentes, il faut et il suffit 
que les LX-morceaux unipotents forment une partition de $\UU$. La mme observation vaut pour les LX-morceaux nilpotents et les strates nilpotentes.

Si $G$ est un groupe classique (\textit{i.e.} de type ${\bf B}_n$, ${\bf C}_n$ ou ${\bf D}_n$), Xue \cite{X} a montrŽ que les LX-morceaux unipotents de $G$, resp. nilpotents de $\mathfrak{g}$, forment une partition de $\UU$, resp. $\NN$, et qu'ils co\"{\i}ncident avec les CP-morceaux\footnote{Pour les groupes symplectiques, cela rŽsulte de \cite{L2,L3}.}: pour $\blacktriangle \in D_G/G$, on a $$G^{\blacktriangle, \mathrm{LX}}= G^{\blacktriangle}\quad\hbox{et}\quad\mathfrak{g}^{\blacktriangle}= \mathfrak{g}^{\blacktriangle}\,;$$ 
on a donc aussi 
$$  G^{\blacktriangle} =\bsfrY^G(\wt{\blacktriangle})\quad\hbox{et}\quad \mathfrak{g}^{\blacktriangle} =\bsfrY^{\mathfrak{g}}(\wt{\blacktriangle}) \ptf$$ Dans ce cas ($G$ est de type ${\bf B}_n$, ${\bf C}_n$ ou ${\bf D}_n$), les strates de $\UU$, resp. de $\NN$, vŽrifient l'hypothse sur le bord \ref{hypo sur le bord} (pour $F=\overline{F}$). 

Pour tre exhaustif, traitons brivement le cas o $G$ est de type ${\bf A}_n$. Dans ce cas, on sait (\ref{p bon, strate-orbite}) que les strates unipotentes, resp. nilpotentes, co\"{\i}ncident avec les $G$-orbites dans $\UU$, resp. $\NN$; elles co\"{\i}ncident donc avec les CP-morceaux unipotents, resp. nilpotents; et d'aprs \cite{L5}, les LX-morceaux unipotents, resp. nilpotents, co\"{\i}ncident eux aussi avec les $G$-orbites dans $\UU$, resp. $\NN$. 

Il ne reste donc plus qu'ˆ traiter le cas des groupes exceptionnels. 
Si $G$ est de type ${\bf E}_6$, ${\bf F}_4$ ou ${\bf G}_2$, Voggesberger \cite{V} a montrŽ que les LX-morceaux nilpotents forment une partition de $\NN$. Dans ce cas ils co\"{\i}ncident donc avec les strates de $\NN$: pour $\blacktriangle \in D_G/G$, on a $\bsfrY^{\mathfrak{g}}(\wt{\blacktriangle}) = \mathfrak{g}^{\blacktriangle,\mathrm{LX}}$. Observons que les calculs ont ŽtŽ effectuŽs en utilisant \textsc{Magma}. RŽcemment, Premet \cite{P} a Žtendu ce rŽsultat aux groupes de type ${\bf E}_7$ ou ${\bf E}_8$; d'ailleurs il donne aussi une autre dŽmonstration du rŽsultat pour les groupes de type ${\bf E}_6$, ${\bf F}_4$ ou ${\bf G}_2$, sans utiliser \textsc{Magma}. Enfin il en dŽduit le mme rŽsultat pour les LX-morceaux unipotents: ils forment une partition de $\UU$  \cite[cor.~2.9]{P}. En conclusion, si $G$ est de type ${\bf G}_2$, ${\bf F}_4$ ou ${\bf E}_*$, on a 
$$  G^{\blacktriangle,\mathrm{LX}} =\bsfrY^G(\wt{\blacktriangle})\quad\hbox{et}\quad \mathfrak{g}^{\blacktriangle,\mathrm{LX}} =\bsfrY^{\mathfrak{g}}(\wt{\blacktriangle}) \ptf$$

Supposons que $G$ soit de type exceptionnel (${\bf E}_*$, ${\bf F}_4$ ou ${\bf G}_2$). Dans \cite{P}, Premet dŽcrit explicitement les $G$-orbites contenues dans une strate de $\NN$. 
Il utilise pour cela les tables de Liebeck-Seitz \cite{LS}. D'aprs \cite[remark~2.5]{P}, une strate $\bsfrY$ de $\NN$ est constituŽe d'au plus deux $G$-orbites. Dans le cas o $\bsfrY$ est constituŽe de deux orbites, on pose $\ES{O}_{\mathrm{min}}(\bsfrY)= \bsfrY \smallsetminus \ES{O}(\bsfrY)$\footnote{Cette convention est celle des travaux de Mizuno et Spaltenstein. C'est l'inverse de celle adoptŽe par Liebeck-Seitz dans \cite{LS} (et reprise dans \cite{P}) o c'est la $G$-orbite de dimension minimale (parmi les $G$-orbites 
dans $\bsfrY$) qui porte le mme label que la $G'$-orbite correspondante; voir aussi plus loin.}. Les cas o une strate de $\NN$ se dŽdouble en deux $G$-orbites 
sont: $({\bf G}_2, p=3)$, $({\bf F}_4,p=2)$, $({\bf E}_7, p=2)$, $({\bf E}_8,p=2)$, $({\bf E}_8,p=3)$. Dans les cas $({\bf G}_2,p=3)$ et $({\bf E}_8,p=3)$, seule une strate de $\NN$ se dŽdouble en deux $G$-orbites. 

\begin{remark}\label{sur l'hypo sur le bord pour les strates nil}
\textup{Supposons que $G$ soit de type ${\bf E}_6$, ${\bf E}_7$, ${\bf E}_8$, ${\bf F}_4$ ou ${\bf G}_2$.
\begin{enumerate}
\item[(i)] Pour que l'hypothse sur le bord \ref{hypo sur le bord} (pour $F=\overline{F}$) soit vŽrifiŽe, il suffit que la propriŽtŽ suivante le soit: pour toutes strates $\bsfrY_1$ et $\bsfrY_2$ de $\NN$, si $\ES{O}_1$ est une $G$-orbite contenue dans $\bsfrY_1 \cap \bsfrX(\bsfrY_2)$, alors $\ES{O}(\bsfrY_1) \subset \bsfrX(\bsfrY_2)$. Observons que si $\ES{O}_1= \ES{O}(\bsfrY_1)$, il n'y a rien ˆ dŽmontrer car $\bsfrY_1 \subset \overline{\ES{O}_1}\;(= \bsfrX(\bsfrY_1)) \subset \bsfrX(\bsfrY_2)$. On est donc ramenŽ ˆ traiter le cas o la strate $\bsfrY_1$ est constituŽe de deux orbites et o $\ES{O}_1= \ES{O}_{\mathrm{min}}(\bsfrY_1)$. Il s'agit d'une vŽrification sur les diagrammes de Hasse des orbites de $\NN$: dans le cas o $\bsfrY_1= \ES{O}_{\mathrm{min}}(\bsfrY_1) \coprod \ES{O}(\bsfrY_1)$, il faut vŽrifier que si $\bsfrY_2$ est une strate de $\NN$ telle que 
$\ES{O}_{\mathrm{min}}(\bsfrY_1) \prec \ES{O}(\bsfrY_2)$, alors $\ES{O}(\bsfrY_1)\preccurlyeq \ES{O}(\bsfrY_2)$. Par exemple pour $({\bf G}_2, p=3)$, la vŽrification est immŽdiate (il y seulement $6$ orbites nilpotentes). Les diagrammes de Hasse des orbites nilpotentes des groupes exceptionnels en mauvaise caractŽristique ne sont malheureusement pas tous Žcrits dans la littŽrature. 
\item[(ii)] Une propriŽtŽ plus faible que l'hypothse sur le bord \ref{hypo sur le bord} (pour $F=\overline{F}$) est: pour tout $\blacktriangle_1,\, \blacktriangle_2 \in D_G/G$, on a $\bsfrY^{\mathfrak{g}}(\wt{\blacktriangle}_1)\subset \overline{\bsfrY^{\mathfrak{g}}(\wt{\blacktriangle}_2)}\Leftrightarrow \ES{O}^{\mathfrak{g}'\!}(\blacktriangle_1)\preccurlyeq \ES{O}^{\mathfrak{g}'\!}(\blacktriangle_2)$. En termes de diagrammes de Hasse, cette propriŽtŽ Žquivaut ˆ: le diagramme de Hasse des strates de $\NN$ est le mme que celui des $G'$-orbites de $\NN'$.
\end{enumerate}
}\end{remark}

On suppose toujours que $G$ est de type exceptionnel. La description des $G$-orbites ˆ l'intŽrieur d'une strate de $\UU$ est similaire ˆ celle du cas nilpotent. D'ailleurs dans \cite{LS}, la description des $G$-orbites unipotentes se dŽduit de celle des $G$-orbites nilpotentes. Concrtement, dans \cite{LS} est fournie une liste de reprŽsentants $e\in \NN$ des $G$-orbites dans $\NN$.  \`A presque tout reprŽsentant $e$ est associŽ un reprŽsentant $u$ d'une $G$-orbite dans $\UU$; seuls certains reprŽsentants \guill{exceptionnels} $e$ ne correspondent ˆ aucun $u$ (\textit{cf.} \cite[ch.~18]{LS}). 
En gros, chaque ŽlŽment $e$ s'exprime comme une somme d'ŽlŽments radiciels dans $\NN$ et l'Žlement $u$, s'il existe, s'obtient en remplaant cette somme par le produit des ŽlŽments radiciels correspondant dans $\UU$. Commenons par rappeler comment les reprŽsentants $e$ s'organisent dans les strates de $\NN$ (\textit{cf.} 
\cite[remark~2.5]{P}). \`A chaque $e$ est associŽ un co-caractre $\lambda \in \blacktriangle $ tel que $e \in \mathfrak{g}_{\lambda,2}$ et $G^e \subset P_\lambda$; en particulier $e$ appartient ˆ $\mathfrak{g}^{\blacktriangle,\mathrm{LX}}$. Dans les cas o deux $G$-orbites de $\NN$ sont associŽes au mme label $\blacktriangle$ (ce qui n'est possible que si $p=2$ ou $3$), on note $e$ et $\wt{e}$ les deux reprŽsentants. S'il n'y a pas de $\wt{e}$, alors $e\in \mathfrak{g}_\lambda(2)$ et $\overline{\mathrm{Ad}_{M_\lambda}(e)}= \mathfrak{g}_\lambda(2)$, ce qui assure que $e$ est $M_\lambda^\perp$-semi-stable et $$\bsfrY^{\mathfrak{g}}(\wt{\blacktriangle})= \ES{O}^{\mathfrak{g}}(\blacktriangle)\ptf$$ 
S'il y a un $\wt{e}$, on peut sans perte de gŽnŽralitŽ supposer que $\wt{e}= e + e_\beta$ avec $e\in \mathfrak{g}_\lambda(2)$, $e_\beta\in \mathfrak{g}_{\lambda,2}$ et $\overline{\mathrm{Ad}_{P_\lambda}(\wt{e})}= \mathfrak{g}_{\lambda,2}$. Notons $\wt{e}(2)$ la projection naturelle de $\wt{e}$ sur $\mathfrak{g}_\lambda(2)$. Comme $\overline{\mathrm{Ad}_{M_\lambda}(\wt{e}(2))} = \mathfrak{g}_\lambda(2)$, l'ŽlŽment $\wt{e}(2)$ est $M_\lambda^\perp$-semi-stable. D'autre part on a 
$\overline{\mathrm{Ad}_{M_\lambda}(e)}= \mathfrak{g}_\lambda(2)$ si et seulement si $e_\beta \in \mathfrak{g}_{\lambda,3}$\footnote{Il y a seulement quatre cas o le reprŽsentant $\wt{e}$ n'est pas homogne de degrŽ $2$, \cad o $e_\beta$ appartient ˆ $ \mathfrak{g}_{\lambda,3}$. Dans \cite{LS}, ces quatre cas correspondent aux labels  
$({\bf E}_8, (A_7)_3)$, $({\bf F}_4,C_3(a_2)_2)$, $({\bf F}_4,(\wt{A}_2A_1)_2)$ et $({\bf G}_2, (A_1)_3)$.}, auquel cas on sait d'emblŽe que $e$ est $M_\lambda^\perp$-semi-stable. Premet \cite{P} prouve que cette propriŽtŽ reste vraie dans 
le cas (difficile) o $e_\beta \in \mathfrak{g}_\lambda(2)$. Dans tous les cas, en notant $\ES{O}^{\mathfrak{g}}(\blacktriangle)= \ES{O}(\wt{e})$ la $G$-orbite de $\wt{e}$ et $\ES{O}^{\mathfrak{g}}(\blacktriangle)_p=\ES{O}(e)$ celle de $e$\footnote{Insistons sur le fait que nos conventions d'Žcritures sont celles de Mizuno et Spaltenstein, \cad l'inverse de celles de Liebeck-Seitz et Premet. Dans \cite{P}, la $G$-orbite de $e$ est notŽe $(\blacktriangle)$ et celle de $\wt{e}$ est notŽe $(\blacktriangle)_p$.}, on a $$\bsfrY^{\mathfrak{g}}(\wt{\blacktriangle})= \ES{O}^{\mathfrak{g}}(\blacktriangle)\coprod \ES{O}^{\mathfrak{g}}(\blacktriangle)_p\ptf$$ 
Observons que $\ES{O}^{\mathfrak{g}}(\blacktriangle)= \ES{O}(\bsfrY^{\mathfrak{g}}(\wt{\blacktriangle}))$ et 
$\ES{O}^{\mathfrak{g}}(\blacktriangle)_p= \ES{O}_{\mathrm{min}}(\bsfrY^{\mathfrak{g}}(\wt{\blacktriangle}))$. Rappelons maintenant les rŽsultats de \cite[ch.~18]{LS}. Ë l'ŽlŽment $e\in \mathfrak{g}_\lambda(2)$ est toujours associŽ un ŽlŽment $u\in G_{\lambda,2}$. S'il y a un $\wt{e}$, ou bien il lui est associŽ un Žl'Žment $\wt{u}\in G_{\lambda,2}$, ou bien aucun ŽlŽment de $\UU$ n'est associŽ ˆ $\wt{e}$. S'il n'y a pas de $\wt{u}$ (\textit{i.e.} s'il n'y pas de $\wt{e}$ ou bien si aucun $\wt{u}$ nest associŽ ˆ $\wt{e}$), on note 
$\ES{O}^G(\blacktriangle)$ la $G$-orbite de $u$. S'il y a un $\wt{u}$, on note $\ES{O}^G(\blacktriangle)=\ES{O}(u)$ la $G$-orbite de $\wt{u}$ et $\ES{O}^G(\blacktriangle)_p=\ES{O}(\wt{u})$ celle de $u$. 
S'il y a un $\wt{u}$, on a toujours 
$\ES{O}^G(\blacktriangle)_p \subset \overline{\ES{O}^G(\blacktriangle)}$. Par construction, ces applications $e\rightarrow u$ et $\wt{e}\rightarrow \wt{u}$ sont compatibles ˆ la bijection qui ˆ une strate de $\NN$ associe une strate de $\UU$: s'il n'y a pas de $\wt{u}$, on a $$\bsfrY^G(\wt{\blacktriangle})= \ES{O}^G(\blacktriangle)\,;$$ 
et s'il y a un $\wt{u}$, on a $$\bsfrY^G(\wt{\blacktriangle})= \ES{O}^G(\blacktriangle)\coprod \ES{O}^G(\blacktriangle)_p\ptf$$ 
Il y a moins de $G$-orbites unipotentes que de $G$-orbites nilpotentes: certaines $G$-orbites exceptionnelles $\ES{O}^{\mathfrak{g}}(\blacktriangle)= \ES{O}(\wt{e})$ ne correspondent ˆ aucune $G$-orbite de $\UU$.

\begin{lemme}\label{preuve de l'hypo sur le bord pour la variŽtŽ unipotente}
\begin{enumerate}
\item[(i)] Pour le groupe $G$ opŽrant sur lui-mme par conjugaison, l'hypothse sur le bord \ref{hypo sur le bord} (pour $F=\overline{F}$) 
est vŽrifiŽe: pour toute strate $\bsfrY$ de $\UU$, le bord $\bsfrX(\bsfrY)\smallsetminus \bsfrY$ est une rŽunion de strates de $\UU$. 
\item[(ii)] Le diagramme de Hasse des strates de $\UU$ est le mme que le diagramme de Hasse des $G$-orbites $\ES{O}({\bsfrY})$ pour $\bsfrY$ parcourant les strates de 
$\UU$; c'est aussi le mme que le diagramme de Hasse des $G'\;(= \scr{G}_{\mbb{C}})$-orbites de $\UU'= \ES{N}^{G'}(G'\!,1)$.
\end{enumerate}
\end{lemme}

\begin{proof} On peut supposer que $G$ est quasi-simple. 

Prouvons (i). Pour $G$ de type ${\bf A}_n$, ${\bf B}_n$, ${\bf C}_n$ ou ${\bf D}_n$, on a vu que l'hypothse sur le bord \ref{hypo sur le bord} est vŽrifiŽe. Reste ˆ traiter le cas o $G$ est de type ${\bf E}_*$, ${\bf F}_4$ ou ${\bf G}_2$. Pour allŽger l'Žcriture, on omet l'exposant \guill{$G$} dans cette dŽmonstration: $\bsfrY(\wt{\blacktriangle})= \bsfrY^G(\wt{\blacktriangle})$, $\ES{O}(\blacktriangle)= \ES{O}^G(\blacktriangle)$, etc. Comme pour les strates nilpotentes (\textit{cf.} \ref{sur l'hypo sur le bord pour les strates nil}), il s'agit de vŽrifier que pour tout $\blacktriangle\in D_G/G$ tel que $\bsfrY(\wt{\blacktriangle})=\ES{O}(\blacktriangle)\coprod \ES{O}(\blacktriangle)_p$, si $\ES{O}(\blacktriangle)_p \prec \ES{O}(\bsfrY(\wt{\blacktriangle}_1))$ pour un $\blacktriangle_1\in D_G/G$, alors $\ES{O}(\blacktriangle)\prec \ES{O}(\bsfrY(\wt{\blacktriangle}_1))$. L'inspection des diagrammes de Hasse dans \cite[IV.2]{Spa} montre que dans tous les cas sauf $({\bf F}_4,p=2)$, la seule $G$-orbite de $\UU$ qui couvre $\ES{O}(\blacktriangle)_p$ est $\ES{O}(\blacktriangle)$; alors on a forcŽment $\blacktriangle_1=\blacktriangle$. Dans le cas $({\bf F}_4,p=2)$, cette propriŽtŽ est vŽrifiŽe pour tous les labels $\blacktriangle \in D_G/G$ tels que $\bsfrY(\wt{\blacktriangle})= \ES{O}(\blacktriangle)\coprod \ES{O}(\blacktriangle)_2$ 
sauf deux \cite[IV.2, page 250]{Spa}: le label $\blacktriangle$ associŽ ˆ la paire de $G$-orbites de $\UU$ notŽe $((\wt{A}_2 + A_1)_2 ,\wt{A}_2 + A_1)$ et celui associŽ 
ˆ la paire de $G$-orbites notŽe $((B_2)_2,B_2)$\footnote{Ces paires de $G$-orbites de $\UU$ sont notŽes $(\wt{A}_2A_1) \prec (\wt{A}_2A_1)_2$ et $(B_2)\prec (B_2)_2$ dans \cite{LS}.}. Ci-dessous la partie du diagramme de Hasse pour $({\bf F}_4,p=2)$ concernŽe: 
$$
\xymatrix{&C_3(a_1) \ar@{-}[dl] \ar@{-}[d] \ar@{-}[dr] & \\ \wt{A}_2 + A_1 \ar@{-}[d] \ar@{-}[drr]&C_3(a_1)_2\ar@{-}[d]|(0.52)\hole \ar@{-}[dl]|(0.35)\hole&B_2 \ar@{-}[d] \ar@{-}[dl]|(0.68)\hole\\ 
(\wt{A}_2+ A_1)_2 \ar@{-}[d] & (B_2)_2 \ar@{-}[dd]& A_2+ \wt{A}_1 \ar@{-}[d]\\
\wt{A}_2 \ar@{-}[dr]&& A_2 \ar@{-}[dl] \\ &A_1 + \wt{A}_1  &}$$
En notant $\blacktriangle_1$ le label associŽe ˆ la paire de $G$-orbites de $\UU$ notŽe $(C_3(a_1)_2, C_3(a_1))$ dans \textit{loc.~cit.}, on voit que la $G$-orbite $\ES{O}(\blacktriangle)_2$ a deux $G$-orbites couvrantes: $\ES{O}(\blacktriangle)$ et $\ES{O}(\blacktriangle_1)_2$. Mais comme d'aprs \textit{loc.~cit.}, on a $\ES{O}(\blacktriangle)\prec \ES{O}(\blacktriangle_1)$, on en dŽduit que $\bsfrY(\wt{\blacktriangle})\subset \bsfrX(\bsfrY(\wt{\blacktriangle}_1))$. Dans tous les cas, l'hypothse sur le bord \ref{hypo sur le bord} (pour $F=\overline{F}$) est vŽrifiŽe. Cela prouve (i).

Le point (ii) est une consŽquence de (i).
\end{proof}

\begin{remark}
\textup{
\begin{enumerate}
\item[(i)] 
Pour toute strate $\bsfrY$ de $\UU$, la fermeture $\bsfrX(\bsfrY)$ de $\bsfrY$ dans $G$ (\textit{i.e.} dans $\UU$) est la rŽunion 
de $\ES{O}(\bsfrY)$ et des $G$-orbites $\ES{O}'$ dans $\UU$ telles que $\ES{O}' \prec \ES{O}(\bsfrY)$; et comme l'hypothse sur le bord \ref{hypo sur le bord} pour $F=\overline{F}$ est vŽrifiŽe 
(\ref{preuve de l'hypo sur le bord pour la variŽtŽ unipotente}\,(i)), pour toute $G$-orbite $\ES{O}'$ dans 
une strate $\bsfrY'$ de $\UU$, on a:
$$\ES{O}'\subset \bsfrX(\bsfrY) \Leftrightarrow  \bsfrY' \subset \bsfrX(\bsfrY) \Leftrightarrow \ES{O}(\bsfrY') \subset \bsfrX(\bsfrY)\ptf$$ 
\item[(ii)] Si l'on disposait de tous les diagrammes de Hasse pour les orbites nilpotentes des groupes exceptionnels en mauvaise caractŽristique $p$, 
la vŽrification de l'hypothse sur le bord \ref{hypo sur le bord} (pour $F=\overline{F}$) pour l'action adjointe de $G$ sur $\mathfrak{g}$ serait la mme que celle 
de l'hypothse sur le bord \ref{hypo sur le bord} (pour $F=\overline{F}$) pour l'action de $G$ sur lui-mme par conjugaison 
(\ref{preuve de l'hypo sur le bord pour la variŽtŽ unipotente}). 
\end{enumerate}}
\end{remark}

\subsection{$F$-strates et orbites gŽomŽtriques}\label{strates et orbites} 
Dans cette sous-section, 
on reprend la construction de \ref{comparaison avec les strates gŽo} (\S2) en remplaant les $\overline{F}$-strates par les $G$-orbites (de $\UU$). 

\vskip2mm
\textit{\S1.\hskip1mm La $G$-orbite $\ES{O}(F,u)$ associŽe ˆ $u\in \FU$. ---}  
Puisqu'il 
n'y a qu'un nombre fini d'orbites gŽomŽtriques unipotentes \cite{L1} et qu'elles sont localement fermŽes dans $G$ (\textit{i.e.} dans $\UU$), pour $u\in \FU$, il existe une unique $G$-orbite $\ES{O}_{u_1}$ dans $\UU$ (avec $u_1\in \UU$ que l'on peut choisir dans ${_F\scrX_u}$) 
telle que l'intersection $\ES{O}_{u_1}\cap {_F\scrX_u}$ soit dense dans la variŽtŽ irrŽductible ${_F\scrX_u}$; on la note\index{OgeoFu@$\ES{O}(F,u)$} 
$$\ES{O}(F,u)=\ES{O}_{u_1}\ptf$$  
Observons que l'intersection $\ES{O}(F,u) \cap {_F\scrX_u}$ est ouverte dans ${_F\scrX_u}$. D'autre part 
puisque $\ES{O}_{u_1}\subset \bsfrY_{u_1}$, l'intersection $\bsfrY_{u_1}\cap {_F\scrX_u} $ est \textit{a fortiori} dense dans ${_F\scrX_u}$. On en dŽduit que 
$$\ES{O}(F,u)\subset \bsfrY(F,u)= \bsfrY_{u_1} \ptf$$ 
Observons qu'\textit{a priori}, cela n'implique pas l'ŽgalitŽ $\ES{O}(F,u)= \ES{O}(\bsfrY(F,u))$: mme si 
$\ES{O}(F,u)$ est l'unique $G$-orbite dans la $\overline{F}$-strate $\bsfrY(F,u)$ qui intersecte ${_F\scrX_u}$ de manire dense, cela n'implique 
pas qu'elle soit dense dans $\overline{F}$-strate $\bsfrY(F,u)$. 
En revanche on a toujours: la $G$-orbite $\ES{O}(F,u)=G\bullet (\ES{O}(F,u)\cap {_F\scrX_u})$ est ouverte dense  
dans la sous-variŽtŽ fermŽe irrŽductible $G\bullet {_F\scrX_u}$ de $\UU$. Enfin comme pour la $\overline{F}$-strate $\bsfrY(F,u)$, la $G$-orbite 
$\ES{O}(F,u)$ ne dŽpend que de la $F$-strate ${_F\bsfrY_u}$: on a 
$$\ES{O}(F,u')= \ES{O}(F,u)\quad \hbox{pour tout} \quad u'\in {_F\bsfrY_u}\ptf$$

\begin{lemme}\label{orbite gŽo et densitŽ}
Soit $u\in \FU$ tel que $\check{X}_F(G)_{\mbb{Q}}\cap \bs{\Lambda}_u\neq \emptyset$. 
\begin{enumerate}
\item[(i)] On a $\ES{O}(F,u)= \ES{O}(\bsfrY(F,u))$ et l'intersection $\ES{O}(F,u)\cap {_F\scrY_u}$ est ouverte dense dans ${_F\scrX_u}$; en particulier elle est non vide. 
\item[(ii)] Si le corps $F$ est infini, l'intersection $\ES{O}(F,u)\cap \scrY_{F,u}$ est dense dans ${_F\scrX_u}$; en particulier elle est non vide.
\end{enumerate}
\end{lemme}

\begin{proof}
D'aprs \ref{strate gŽo associŽe ˆ F-strate}\,(i), on a $\bsfrY(F,u)= \bsfrY_u$ et  
$\bsfrY_u \cap {_F\scrX_u}= {_F\scrY_u}$ avec (\ref{le cas o F-lame = lame gŽo}) ${_F\scrX_u}= \scrX_u$ et ${_F\scrY_u}=\scrY_u$. 
Comme la $G$-orbite $\ES{O}(F,u)$ est ouverte dense dans $G\bullet {_F\scrX_u}= \bsfrX_u = \overline{\bsfrY(F,u)}$ et qu'elle est contenue dans $\bsfrY_u$, on a 
forcŽment $\ES{O}(F,u)= \ES{O}(\bsfrY_u)$. 
Puisque $\scrY_u$ et $\ES{O}(F,u)\cap \scrX_u$ sont deux ouverts denses 
de $\scrX_u$, leur intersection $\ES{O}(F,u)\cap \scrY_u $ est encore un ouvert dense de $\scrX_u$. Cela prouve (i). 

Si de plus le corps $F$ est infini, alors $\scrY_{F,u}= G(F)\cap \scrY_u$ est dense dans ${_F\scrX_u}$ (\textit{cf.} \ref{densitŽ bonnes F-strates}\,(ii)). Cela implique (ii). 
\end{proof}

\begin{remark}
\textup{Observons que mme si $\bsfrY(F,u)= \bsfrY_u$ et $\ES{O}(F,u)= \ES{O}(\bsfrY_u)$, on gŽnŽral on a $\bsfrY(F,u)\neq \ES{O}_u$.}
\end{remark}

\textit{\S2.\hskip1mm Version $F$-rationnelle de \ref{p bon, strate-orbite}. ---} 
D'aprs \ref{descente sŽparable strates} (en prenant $E=F^\mathrm{s\acute{e}p}$), pour $p\geq 1$ et 
$u,\,u'\in \UF$, les trois conditions Žquivalentes suivantes sont Žquivalentes: 
\begin{itemize}
\item $u$ et $u'$ sont dans la mme $F$-strate de $\UF$;
\item $u$ et $u'$ sont dans la mme $F^\mathrm{s\acute{e}p}$-strate de $\UU_{\smash{F^{\mathrm{s\acute{e}p}}}}$;
\item il existe un $g\in G(F^{\mathrm{s\acute{e}p}})$ tel que $g\bullet \bs{\Lambda}_{F^\mathrm{s\acute{e}p},u}=\bs{\Lambda}_{F^\mathrm{s\acute{e}p},u'}$.  
\end{itemize}
On en dŽduit que si $p=1$, pour $u,\, u'\in \UF$, les trois conditions suivantes sont Žquivalentes (d'aprs \ref{p bon, strate-orbite}): 
\begin{itemize}
\item $u$ et $u'$ sont dans la mme $F$-strate de $\FU$;
\item $u$ et $u'$ sont conjuguŽs dans $G\;(=G(\overline{F})=G(F^{\rm s\acute{e}p}))$;
\item il existe un $g\in G$ tel que $g\bullet \bs{\Lambda}_u = \bs{\Lambda}_{u'}$.
\end{itemize}
Nous allons vŽrifier que cela reste vrai si $p>1$ est \textit{trs bon} pour $G$.

\vskip1mm 
D'aprs Richardson-Springer-Steinberg \cite[ch.~I, \S5]{SS} (\textit{cf.} \cite[2.3]{Ti}), on a le 

\begin{lemme}\label{sŽparabilitŽ des orbites}
Si $p\geq 1$ est trs bon pour $G$, alors tous les ŽlŽments  de $G$ sont sŽparables (au sens de \ref{orbites sŽparables}) et tous les ŽlŽments de 
$\mathfrak{g}$ sont sŽparables. 
\end{lemme}

\begin{remark}\label{forme bilinŽaire G-inv sym nd}
\textup{Rappelons que la sŽparabilitŽ est une notion gŽomŽtrique (on peut donc supposer $F=\overline{F}$). 
Notons $C_G$ le tore central maximal de $G$.  
Le morphisme produit $$\wt{\pi}:\wt{G}= C_G \times G_{\rm sc} \rightarrow G= C_G \cdot G_{\rm der}$$ est une isogŽnie centrale (pas forcŽment sŽparable); 
o $G_{\rm sc} \rightarrow G_{\rm der}$ est le revtement simplement connexe 
du groupe dŽrivŽ de $G$. Supposons de plus que  $p$ soit trs bon pour $G$. Alors il existe une forme bilinŽaire $\wt{G}$-invariante non dŽgŽnŽrŽe \textit{symŽtrique} $B(\cdot,\cdot)$ sur $\wt{\mathfrak{g}}={\rm Lie}(\wt{G})$. 
De plus $\wt{\pi}$ est sŽparable, \textit{i.e.} $\dd (\wt{\pi})_1 : \wt{\mathfrak{g}} \rightarrow \mathfrak{g}$ est un isomorphisme, et $B$ est une forme bilinŽaire $G$-invariante non dŽgŽnŽrŽe symŽtrique sur $\mathfrak{g}$. Cette propriŽtŽ 
assure que $G$ vŽrifie la condition (*) de \cite[ch.~I, 5.1]{SS}. 
}
\end{remark}

Pour les ŽlŽments de $\UU_F$ sŽparables, les points $F$-rationnels de l'orbite gŽomŽtrique sont contenus dans la $F$-strate: 
%
\begin{lemme}\label{inclusion F-strate-orbite rat}
\begin{enumerate}
\item[(i)]Pour $u\in \UF$ sŽparable, on a l'inclusion $\ES{O}_u(F)\subset \bsfrY_{F,u}$.
\item[(ii)]Pour $X\in \mathfrak{N}_F$ sŽparable, on a l'inclusion $\ES{O}_X(F)\subset \bsfrY_{F,X}$.
\end{enumerate} 
\end{lemme} 

\begin{proof}
Soit $u\in \UF$ sŽparable. Le centralisateur schŽmatique $G^u$ lisse, \textit{i.e.} gŽomŽtriquement rŽduit (d'aprs \cite[ch.~II, 6.7]{B}) et le $F$-morphisme de variŽtŽs $G \rightarrow \ES{O}_u,\, g \mapsto gug^{-1}$ se factorise en un $F$-isomorphisme 
$G/G^u \buildrel\simeq\over{\longrightarrow} \ES{O}_u$. De plus on a la suite exacte courte (\textit{cf.} \cite[3.1]{M})
$$1 \rightarrow G^u(F^{\rm s\acute{e}p})\rightarrow G(F^{\rm s\acute{e}p}) \rightarrow  \ES{O}_u(F^{\rm s\acute{e}p}) \rightarrow 1\ptf$$ 
On en dŽduit que si $u'\in \ES{O}_u(F)$, alors $u'$ est conjuguŽ ˆ $u$ dans $G(F^{\rm s\acute{e}p})$ et par consŽquent $u'$ appartient ˆ 
$\bsfrY_{F,u}$ (d'aprs \ref{descente sŽparable strates}). D'o l'inclusion 
$$\ES{O}_u(F)\subset \bsfrY_{F,u}\ptf$$ L'inclusion de (ii) s'obtient de la mme manire.
\end{proof}

\begin{proposition}\label{comparaison F-strate-orbite rat}
On suppose que $p$ est trs bon pour $G$.
\begin{enumerate}
\item[(i)]Pour $u\in \UF$, on a $ \ES{O}_u(F)=\bsfrY_{F,u}$.
\item[(ii)] Pour $X\in \mathfrak{N}_F$, on a $ \ES{O}_X(F)= \bsfrY_{F,X}$.
\end{enumerate}
\end{proposition}

\begin{proof}
Soit $u\in \UF$. Quitte ˆ remplacer $u$ par un conjuguŽ dans $G(F)$, on peut supposer $u$ en position standard. Soit $X\in \mathfrak{u}_0(F)$ tel que $j_0(X)=u$. 
Soient $\lambda \in \Lambda_X^{\rm opt}$ et $k= m_u(\lambda)\in \mbb{N}^*$. 

D'aprs \ref{forme bilinŽaire G-inv sym nd}, il existe une forme bilinŽaire $G$-invariante symŽtrique 
non dŽgŽnŽrŽe $B$ sur $\mathfrak{g}$. On en dŽduit comme dans la preuve de \cite[lemma 5.7]{J2} que pour tout $n\in \mbb{Z}$, on a 
$$[X,\mathfrak{g}_\lambda(n-k)]= \mathfrak{g}_\lambda(n) \Leftrightarrow \mathfrak{g}^X \cap \mathfrak{g}_\lambda(-n)=\{0\}$$ o l'on a posŽ (rappel) $\mathfrak{g}^X= \{H\in \mathfrak{g}\,\vert [X,H]=0\}$. 
Puisque $G^X\subset P_\lambda$ et que $X$ est sŽparable (d'aprs \ref{sŽparabilitŽ des orbites}), on a $\mathfrak{g}^X \subset \mathfrak{p}_\lambda = \mathfrak{g}_{\lambda,0}$. On en dŽduit que 
$$[X,\mathfrak{g}_\lambda(n-k)]= \mathfrak{g}_\lambda(n)\quad \hbox{pour tout} \quad n \in \mbb{N}^*\ptf$$ D'o l'ŽgalitŽ (\textit{cf.} la dŽmonstration de \cite[prop. 3.1]{Ts})
$${\rm Ad}_{U_\lambda}(X) = X + \mathfrak{g}_{\lambda,k+1}\ptf$$ Par consŽquent quitte ˆ remplacer $\lambda$ par ${\rm Int}_{u'}\circ\lambda$ pour un $u' \in {_FU_X}=U_\lambda$, on peut supposer que 
$X$ appartient ˆ $\mathfrak{g}_\lambda(k)$. On peut donc appliquer \ref{projseprat}\,(iii): $$\bsfrY_{F,X}= \mathfrak{g}(F) \cap \bsfrY_X\ptf$$ Or $\bsfrY_X= \ES{O}_X$ (\ref{p bon, strate-orbite}), d'o  
le point (ii). 

Quant au point (i), d'aprs \ref{descente sŽparable strates}, on peut supposer $F=F^{\rm s\acute{e}p}$. En ce cas $P_0$ est un sous-groupe de Borel de $G$ (dŽfini sur $F$). 
Puisque $\bsfrY_{F,u}=\bsfrY_{F,u}^{\rm sc}$ (\ref{le cas ss}\,(iii)),  
on peut aussi supposer $G$ quasi-simple et simplement connexe. Il existe un $F$-isomorphisme $G$-Žquivariant (appelŽ $F$-isomorphisme de Springer) \cite[ch.~III, 3.12]{SS}
$\zeta:\mathfrak{N}\rightarrow \mathfrak{U}$. Il se restreint en un $F$-isomorphisme $P_0$-Žquivariant $\zeta_0: \mathfrak{u}_0 \rightarrow U_0$. 
En particulier $\zeta$ induit une bijection entre 
les orbites gŽomŽtriques nilpotentes qui rencontrent $\mathfrak{u}_0(F)$ et les orbites gŽomŽtriques unipotentes qui rencontrent $U_0(F)$. D'autre part, d'aprs \ref{bijection strates unip-strates nilp}, 
on a une bijection canonique $$\bsfrY_{F,X} \mapsto \bsfrY_{F,X}^G= G(F)\cap {_F\bsfrY_X^G}$$ entre les $F$-strates de $\mathfrak{N}_F$ et les $F$-strates de $\UF$ qui, d'aprs \ref{bijection strates-strates canonique}, co\"{\i}ncide avec celle construite ˆ partir 
de $\zeta_0$: pour $X\in \mathfrak{N}_F$, on choisit $X_0\in \bsfrY_{F,X}$ en position standard; alors on a $\bsfrY_{F,X}^G = \bsfrY_{F,\zeta_0(X)}$.  
D'aprs \ref{p bon, strate-orbite} et le point (ii) dŽjˆ prouvŽ, les $F$-strates de $\mathfrak{N}_F$ sont 
en bijection avec les orbites gŽomŽtriques nilpotentes de $\mathfrak{g}$ qui rencontrent $\mathfrak{u}_0(F)$. On en dŽduit qu'une $F$-strate de $\UF$ rencontre au plus une orbite gŽomŽtrique unipotente. 
Cela prouve que l'inclusion $\ES{O}_u(F)\subset \bsfrY_{F,u}$ (\ref{comparaison F-strate-orbite rat}\,(i)) est une ŽgalitŽ, \cad le point (i). 
\end{proof}

\begin{remark}\label{rem sur iso springer et G de type An}
\textup{Supposons $G$ quasi-simple. 
\begin{enumerate}
\item[(i)] Si $p\geq 1$ est \textit{trs bon} pour $G$, on a utilisŽ dans la preuve \ref{comparaison F-strate-orbite rat} l'existence d'un $F$-isomorphisme 
de Springer $\zeta: \mathfrak{N}\rightarrow \mathfrak{U}$; observons qu'un tel $F$-isomorphisme existe mme si $G$ n'est pas simplement connexe (pourvu que $p$ soit trs 
bon pour $G$). Il induit une bijection entre les $G$-orbites de $\mathfrak{N}$ qui rencontrent $\mathfrak{N}_F$ 
et les $G$-orbites de $\mathfrak{U}$ qui rencontrent $\UF$. Compte-tenu des ŽgalitŽs de \ref{comparaison F-strate-orbite rat}, cette bijection co\"{\i}ncide avec 
celle dŽduite de \ref{bijection strates unip-strates nilp}\,(ii). En particulier elle ne dŽpend pas du $F$-isomorphisme de Springer choisi pour la dŽfinir. Pour $F=\overline{F}$, cela redonne un rŽsultat connu. 
\item[(ii)] Si $G$ est de type ${\bf A}_n$, d'aprs \ref{p bon, strate-orbite}, les ($\overline{F}$-)strates de $\UU=\UU^{G}$ 
co\"{\i}ncident avec les $G$-orbites de $\UU$; et les strates de $\NN=\NN^G$ co\"{\i}ncident avec les $G$-orbites de $\NN$. Si de plus $G=\mathrm{GL}_{n+1/F}$, 
pour une $G$-orbite $\ES{O}$ de $\UU$, l'intersection $G(F)\cap \ESO $ est non vide et c'est une $G(F)$-orbite de $\UF = G(F) \cap \UU$ (car deux ŽlŽments de $G(F)$ sont $G(\overline{F})$-conjuguŽs si et seulement s'ils sont $G(F)$-conjuguŽs); c'est aussi une $F$-strate de $\UF$. Toujours si $G=G_{n+1/F}$, on a le mme rŽsultat pour les $G$-orbites de $\NN$: si $\ES{O}$ est une $G$-orbite de $\NN$, alors $\mathfrak{g}(F)\cap \ES{O}$ est une $G(F)$-orbite de $\NN_F$; c'est aussi une $F$-strate de $\NN_F$. 
D'aprs \ref{le cas adjoint}, on en dŽduit que pour $G$ de type ${\bf A}_n$, l'application $\ES{O} \mapsto \UF \cap \ES{O}$ est une bijection 
entre les $G$-orbites de $\UU$ et les $F$-strates de $\UF$; et l'application $\ES{O}\mapsto \NN_F \cap \ES{O}$ est une bijection entre les $G$-orbites de $\NN$ et les $F$-strates de $\NN_F$. \end{enumerate}
}\end{remark}

D'aprs \ref{rem sur iso springer et G de type An}\,(ii), on a la variante suivante \ref{comparaison F-strate-orbite rat}: 

\begin{corollary}\label{comparaison F-strate-orbite rat p bon}
On suppose que $p$ est bon pour $G$.
\begin{enumerate}
\item[(i)]Pour $u\in \UF$, on a $ \UF \cap \ES{O}_u(\overline{F})=\bsfrY_{F,u}$.
\item[(ii)] Pour $X\in \mathfrak{N}_F$, on a $ \NN_F \cap \ES{O}_X(\overline{F})= \bsfrY_{F,X}$.
\end{enumerate}
\end{corollary}


\textit{\S3.\hskip1mm Le cas o l'hypothse \ref{hypo bonnes F-strates faible} est vŽrifiŽe.} 
On suppose dans ce paragraphe que l'hypothse \ref{hypo bonnes F-strates faible} est vŽrifiŽe. 
Rappelons les notations de \ref{bracket et bracketF}\index{Ubracket@$[\UU]$, $[\UU]_F$, $[\UF]$}: 
\begin{itemize}
\item $[\UU]$ est l'ensemble des ($\overline{F}$-)strates de $\UU$;  
\item $[\UU]_F \subset [\UU]$ est le sous-ensemble formŽ des strates qui intersectent non trivialement $\UF$; 
\item $[\UF]$ est l'ensemble des 
$F$-strates de $\UF$. 
\end{itemize}
D'aprs \ref{descente 3 sous HBFS}\,(iii)) l'application $\bsfrY \mapsto \UF \cap \bsfrY$ est une bijection entre $[\UU]_F$ et $[\UF]$.  

\begin{remark}\label{(U/G)_F et (Y/G)_F}
\textup{(L'hypothse \ref{hypo bonnes F-strates faible} est vŽrifiŽe.)
Notons \index{UGF@$(\UU/G)_F$} 
$(\UU/G)_F\;(\subset \UU/G)$ l'ensemble des $G$-orbites dans $\UU$ qui intersectent non trivialement $\UF$; et pour une strate 
$\bsfrY \in [\UU]_F$, notons\index{YbsgeoGF@$(\bsfrY/G)_F$} $(\bsfrY/G)_F\subset (\UU/G)_F$ le sous-ensemble formŽ des $G$-orbites qui intersectent non trivialement 
$\UF \cap \bsfrY$. Ainsi pour toute strate $\bsfrY \in [\UU]_F$, la $F$-strate $\UF \cap \bsfrY$ de $\UF$ est l'union disjointe des ensembles $\UF \cap \ES{O}$ o $\ES{O}$ parcourt les ŽlŽments de $(\bsfrY/G)_F$. 
}\end{remark}

Soit $\scr{G}$ un $\mbb{Z}$-groupe rŽductif de Chevalley-Demazure tel que 
$G\times_F\overline{F}\simeq \scr{G}\times_{\mbb{Z}} \overline{F}$. Posons $G'= \scr{G}\times_{\mbb{Z}}\mbb{C}$ et $\UU'= \ES{N}^{G'}(G'\!,1)$. On a vu en \ref{strates et orbites} (\S3) qu'on a 
une bijection canoni\-que $\bsfrY \mapsto \ES{O}'(\bsfrY)$ entre $[\UU]$ et l'ensemble $\UU'/G'$ des $G'$-orbites dans $\UU'$. 
Notons $$(\UU'/G')_F= \{\ES{O}'(\bsfrY)\,\vert \, \bsfrY \in [\UU]_F \} \subset \UU'$$ le sous-ensemble formŽ des $G'$-orbites correspondant ˆ un 
ŽlŽment de $[\UU]_F$; ou, ce qui revient au mme via la bijection $[\UU]_F \rightarrow [\UF],\, \bsfrY \mapsto \UF \cap \bsfrY$, correspondant ˆ un ŽlŽment de 
$[\UF]$. Comme l'hypothse sur le bord \ref{hypo sur le bord} (pour $F= \overline{F}$) est vŽrifiŽe 
(\ref{preuve de l'hypo sur le bord pour la variŽtŽ unipotente}\,(i)), pour $\bsfrY \in [\UU]_F$, l'ensemble 
$$\bsfrX(\UF \cap \bsfrY)^+ = \UF \cap \bsfrX(\bsfrY)$$ est une rŽunion de $F$-strates de $\UF$ (\ref{lemme sur X^+_{F,v}}\,(iii)); et  
d'aprs \ref{preuve de l'hypo sur le bord pour la variŽtŽ unipotente}\,(ii), les $F$-strates de $\bsfrX(\UF \cap \bsfrY)^+\smallsetminus (\UF \cap \bsfrY)$ correspondent bijectivement aux $G'$-orbites $\ES{O}'\in (\UU'/G')_F$ telles que $\ES{O}' \prec \ES{O}'(\bsfrY)$. 

Pour $\ES{O}'\in \UU'/G'$, notons $\ES{O}(\ES{O}')$ la $G$-orbite ouverte dense dans la strate de $\UU$ correspond ˆ $\ES{O}$, \textit{i.e.} $\ES{O}(\ES{O}')= \ES{O}(\bsfrY)$ si $\ES{O}'= \ES{O}'(\bsfrY)$. L'application $\UU'/G' \rightarrow \UU/G$ est injective. On peut se demander si pour $\ES{O}'\in (\UU'/G')_F$, la $G$-orbite $\ES{O}(\ES{O}')$ intersecte $\UF$ non trivialement. 
La rŽponse est toujours positive:

\begin{lemme}\label{Y rat versus O(Y) rat}
(L'hypothse \ref{hypo bonnes F-strates faible} est vŽrifiŽe.) 
Pour tout $\bsfrY\in [\UU]$, on a: 
$$\UF \cap \bsfrY \neq \emptyset \Leftrightarrow \UF \cap \ES{O}(\bsfrY) \neq \emptyset\ptf$$
\end{lemme}

\begin{proof}
L'implication $\Leftarrow$ est claire. Quant ˆ l'autre l'implication, soit $\bsfrY$ une strate de $\UU$ telle que $\UF \cap \bsfrY \neq \emptyset$. 
Posons $\ES{O}=\ES{O}(\bsfrY)$.
On distingue le cas o $F$ est fini du cas o $F$ est infini. 

Si $F$ est fini, puisque la ($\overline{F}$-)strate $\bsfrY$ est $\Gamma_F$-stable, la $G$-orbite $\ES{O}$ l'est aussi. Elle est donc dŽfinie sur $F$ et l'ensemble $\ES{O}(F)= G(F)\cap \ES{O}$ est non vide 
(thŽorme de Lang-Steinberg); d'o le rŽsultat dans ce cas, puisque $\UF = G(F)\cap \UU$ (\ref{descente sŽparable c™ne F-nilpotent}). 

Supposons maintenant que $F$ soit infini. Soit $u\in \UF \cap \bsfrY$. Puisque (par hypothse) $\check{X}_F(G)_{\mbb{Q}} \cap \bs{\Lambda}_u\neq \emptyset$, on a 
${_F\scrX_u}= \scrX_u$ et ${_F\scrY_u}= \scrY_u$; on a aussi $\bsfrY(F,u)= \bsfrY\;(=\bsfrY_u)$ et $\bsfrY_{F,u} = \UF \cap \bsfrY$. D'aprs \ref{orbite gŽo et densitŽ}\,(ii), on a $\scrY_{F,u}\cap \ES{O}(F,u)\neq \emptyset$ avec (d'aprs \ref{orbite gŽo et densitŽ}\,(i)) $\ES{O}(F,u) = \ES{O}$. Le lemme est dŽmontrŽ. 
\end{proof}

\begin{lemme}\label{orbite dense dans F-lame TopF-ouverte et TopF-dense dans F-saturŽ}
(L'hypothse \ref{hypo bonnes F-strates faible} est vŽrifiŽe.) On suppose que $F$ est un corps commutatif topologique sŽparŽ et non discret (\textit{cf.} \ref{le cas d'un corps top}). 
Pour tout $u\in \UF$, l'ensemble $\scrY_{F,u} \cap \ES{O}(\bsfrY_u)$ est $\mathrm{Top}_F$-ouvert et $\mathrm{Top}_F$-dense dans $\scrX_{F,u}$.
\end{lemme}

\begin{proof}
Puisque $\check{X}_F(G)_{\mbb{Q}}\cap \bs{\Lambda}_u\neq \emptyset$, on a ${_F\scrY_u}= \scrY_u$ et ${_F\scrX_u}= \scrX_u$ (\ref{le cas o F-lame = lame gŽo}). 
Comme $\ES{O}(\bsfrY_u)\subset \bsfrY_u$, on a $\scrY_{F,u}\cap \ES{O}(\bsfrY_u)= \scrX_{F,u}\cap \ES{O}(\bsfrY_u)$. 
D'aprs 
\ref{Y rat versus O(Y) rat}, l'intersection $\scrX_{F,u} \cap \ES{O}(\bsfrY_u)$ est non vide; et c'est l'ensemble des points 
$F$-rationnels de $\scrX_u\cap \ES{O}(\bsfrY_u)$, qui est un ouvert dŽfini sur $F$ de $\scrX_u \;(\simeq_F \mbb{A}_F^n)$. On peut donc appliquer \ref{dense implique TopF-dense}: 
l'ensemble des points $F$-rationnels $(\scrX_u \cap \ES{O}(\bsfrY_u))(F)= \scrY_{F,u}\cap \ES{O}(\bsfrY_u)$ est $\mathrm{Top}_F$-ouvert et $\mathrm{Top}_F$-dense dans $\scrX_u(F)= \scrX_{F,u}$. 
\end{proof}

\begin{remark}
\textup{
(L'hypothse \ref{hypo bonnes F-strates faible} est vŽrifiŽe.) Soit $u\in \UF$. D'aprs \ref{Y rat versus O(Y) rat}, 
l'ensemble $\UF \cap \ES{O}(\bsfrY_u)$ est non vide. Soit $u'\in \UF \cap \ES{O}(\bsfrY_u)$. Puisque l'ensemble $\UF \cap \ES{O}(\bsfrY_u)$ est contenu dans $\UF \cap \bsfrY_u = G(F)\bullet \scrY_{F,u}$, quitte ˆ conjuguer $u'$ par un ŽlŽment de $G(F)$, on peut supposer que 
$u'\in \scrY_{F,u} \cap \ES{O}(\bsfrY_u)$. Rappelons que ${_FP_u}= P_u$ et ${_F\scrY_u}= \scrY_u$. Comme l'application 
$\bs{\pi}'_u: G \times^{P_u} \scrY_u \rightarrow \bsfrY_u$ est bijective, l'intersection $\scrY_u \cap \ES{O}(\bsfrY_u)$ est une 
$P_u$-orbite; d'o l'ŽgalitŽ 
$$\scrY_{F,u} \cap \ES{O}(\bsfrY_u)= \UF \cap (P_u\bullet u')\;(=G(F) \cap (P_u\bullet u'))\ptf$$ 
}
\end{remark}

\subsection{Inflation parabolique des sous-ensembles $F$-saturŽs}\label{IP des ensembles FS} 
Pour un $F$-sous-groupe parabolique $P$ de $G$ et une composante de Levi $M$ dŽfini sur $F$, on 
s'intŽresse dans cette sous-section ˆ l'application qui ˆ un sous-ensemble uniformŽment $(F,M)$-instable $Z$ de $\FU^M$ associe le sous-ensemble 
$ZU_P \subset \FU\cap P$. 

\vskip2mm
\textit{\S1.\hskip1mm Le co-caractre (virtuel) $\mu_{M,P}\in \check{X}(A_M)_{\mbb{Q}}$.} 
On appelle \textit{$F$-facteur de Levi} de $G$ une composante de Levi dŽfinie 
sur $F$ d'un $F$-sous-groupe parabolique de $G$. Si $M$ est un $F$-facteur de Levi de $G$, on note 
$A_M$ le sous-tore $F$-dŽployŽ maximal du centre (schŽmatique) $Z(M)$ de $M$.

Soit $P$ un $F$-sous-groupe parabolique de $G$. Supposons pour commencer que $P$ soit standard (\textit{i.e.} $P_0\subset P$). 
La composante de Levi semi-standard $M_P$ de $P$ (\cad celle contenant $M_0$) est dŽfinie sur $F$ et on pose $A_P=A_{M_P}$. Puisque $M_0\subset M_P$, on a l'inclusion $A_P \subset A_0=A_{P_0}$. 
D'autre part l'inclusion $A_P \subset P$ dŽfinit un morphisme injectif ˆ conoyau fini 
$X_F(P) \rightarrow X(A_P)$ qui, dualement, donne un isomorphisme de $\mbb{Q}$-espaces vectoriels
$$\check{X}(A_P)_{\mbb{Q}}= \mathrm{Hom}(X(A_P),\mbb{Q}) \buildrel\simeq\over{\longrightarrow} \mathfrak{a}_{P,\mbb{Q}} = \mathrm{Hom}(X_F(P),\mbb{Q})\ptf$$ 
L'inclusion $P_0\subset P$ induit une inclusion $X_F(P)\subset X_F(P_0)$ et donc un morphisme surjectif $\ag_{0,\mbb{Q}}=\ag_{P_0,\mbb{Q}} \rightarrow \ag_{P,\mbb{Q}}$. D'autre part l'inclusion $\check{X}(A_P)\subset \check{X}(A_0)$ induit un morphisme injectif $\check{X}(A_P)_{\mbb{Q}}\rightarrow \check{X}(A_0)_{\mbb{Q}}$ qui, compte-tenu 
des isomorphismes $\check{X}(A_P)_{\mbb{Q}}\simeq \ag_{P,\mbb{Q}}$ et $\check{X}(A_0)_{\mbb{Q}} \simeq \ag_{0,\mbb{Q}}$, fournit une section de la surjection 
$\ag_{0,\mbb{Q}}\rightarrow \ag_{P,\mbb{Q}}$. 
On a donc la dŽcomposition $$\mathfrak{a}_{0,\mbb{Q}}= \mathfrak{a}_{P,\mbb{Q}} \oplus \mathfrak{a}_{0,\mbb{Q}}^P \quad \hbox{avec}\quad 
\mathfrak{a}_{0,\mbb{Q}}^P = \ker [\mathfrak{a}_{0,\mbb{Q}} \rightarrow \mathfrak{a}_{P,\mbb{Q}}]\ptf$$ 
On identifie l'ensemble $\ES{R}=\ES{R}_{A_0}$ des racines de $A_0$ dans $G$ ˆ un sous-ensemble de $\ag_{0,\mbb{Q}}$ via l'isomorphisme $\check{X}(A_0)_{\mbb{Q}} \simeq \ag_{0,\mbb{Q}}$. 
L'ensemble $\Delta_{0}^P$ des racines simples de $A_0$ dans $M_P\cap U_0$ est une base du dual $(\mathfrak{a}_{0,\mbb{Q}}^P)^*=\mathrm{Hom}(\mathfrak{a}_{0,\mbb{Q}}^P,\mbb{Q})$. 
On dispose aussi de l'ensemble $\Delta_P= \Delta_P^G$ 
des restrictions non nulles des ŽlŽments de $\Delta_0=\Delta_{P_0}^G$ au sous-espace 
$\mathfrak{a}_{P,\mbb{Q}}$ de $\mathfrak{a}_{0,\mbb{Q}}$; c'est une base du dual 
$(\mathfrak{a}_{P,\mbb{Q}}^G)^*$ de $\mathfrak{a}_{P,\mbb{Q}}^G=\ker[ \mathfrak{a}_{P,\mbb{Q}}\rightarrow \mathfrak{a}_{G,\mbb{Q}}]$. 
Notons $\mathfrak{a}_{P,\mbb{Q},1}$ le c™ne dans $\mathfrak{a}_{P,\mbb{Q}}$ formŽ des ŽlŽments $\mu$ tels que $\langle \alpha , \mu \rangle \geq 1$ 
pour tout $\alpha \in \Delta_P$. 
On note $\mu_P$\index{muP@$\mu_P$} l'ŽlŽment de norme minimale dans $\mathfrak{a}_{P,\mbb{Q},1}$\footnote{Rappelons que la $F$-norme $G$-invariante sur $\check{X}(G)$ est construite ˆ partir d'une forme bilinŽaire 
symŽtrique dŽfinie positive $W^G(T)$-invariante et $\Gamma_F$-invariante $(\cdot ,\cdot )$ sur $\check{X}(T)$, o $T$ est un tore maximal de $G$ dŽfini sur $F$. Si $T$ contient $A_0$, ce que l'on peut supposer, cette forme induit une 
forme $\mbb{Q}$-bilinŽaire symŽtrique dŽfinie positive $W^G(A_0)$-invariante sur $\check{X}(A_0)_{\mbb{Q}}\simeq \ag_{P_0,\mbb{Q}}$. \`A chaque composante irrŽductible $\ES{R}_i$ du systme de racine $\ES{R}= \ES{R}_{A_0}$ correspond un 
sous-espace vectoriel $V_i$ de $\ag_{P_0,\mbb{Q}}^G$ et $(\cdot , \cdot)$ se restreint en une forme bilinŽaire symŽtrique dŽfinie positive sur $V_i$ qui est dŽterminŽe de manire unique ˆ une constante positive prs. 
En particulier la structure euclidienne sur $\ag_{P_0,\mbb{Q}}$ dŽfinie par $(\cdot ,\cdot)$ est compatible ˆ la dŽcomposition $\ag_{P_0,\mbb{Q}}= \ag_{P_0,\mbb{Q}}^P\oplus \ag_{P,\mbb{Q}}^G \oplus \ag_{G,\mbb{Q}}$. }. 
Cet ŽlŽment existe par convexitŽ et il appartient au c™ne 
$\mathfrak{a}_{P,\mbb{Q},1} \cap \mathfrak{a}_{P,\mbb{Q}}^G$ 
de $\mathfrak{a}_{P,\mbb{Q}}^G$. 
On l'identifie ˆ un ŽlŽment de $\check{X}(A_P)_{\mbb{Q}}$ via l'isomorphisme $\check{X}(A_P)_{\mbb{Q}} \simeq \ag_{P,\mbb{Q}}$. 
Notons $\{\mu_{P,\alpha}\,\vert\, \alpha \in \Delta_P\}$ la base de $\mathfrak{a}_{P,\mbb{Q}}^G$ duale de $\Delta_P$. Observons que l'application 
$$\Delta_0\smallsetminus \Delta_0^P\rightarrow \Delta_P\vgq \alpha \mapsto \alpha_P= \alpha\vert_{\ag_{P,\mbb{Q}}}$$ est bijective et que pour tout 
$\alpha\in \Delta_0\smallsetminus \Delta_0^P$, on a $\mu_{P_0,\alpha}= \mu_{P,\alpha_P}$.

\begin{lemme}\label{rŽcriture de mu_P}
$\mu_P = \sum_{\alpha \in \Delta_P} \mu_{P,\alpha}$. 
\end{lemme}

\begin{proof}
\'Ecrivons $\mu_P= \sum_{\alpha\in \Delta_P}a_\alpha \mu_{P,\alpha}$ avec $a_\alpha \in \mbb{Q}$. Pour $\alpha \in \Delta_P$, on a 
$\langle \alpha, \mu_P \rangle = a_\alpha  \geq 1$. D'autre part on a 
$$\| \mu_P \|^2 = \sum_{\alpha\in \Delta_P} \sum_{\beta\in \Delta_P} a_\alpha a_\beta (\mu_{P,\alpha},\mu_{P,\beta})\ptf$$ 
Or on sait (\textit{cf.} \cite[1.2.6]{LW}) que $\{\mu_{P,\alpha}\,\vert\, \alpha \in \Delta_P\}$ est une base aig\"{u}e de $\ag_{P,\mbb{Q}}^G$: pour tous $\alpha,\,\beta \in \Delta_P$, on a 
$(\mu_{P,\alpha},\mu_{P,\beta})\geq 0$. Cela entra"ne le lemme.
\end{proof}

\begin{remark}\label{tmu_P=mu_P}
\textup{
Rappelons que pour construire le $F$-isomorphisme $A_0$-Žquivariant $j_0: \mathfrak{u}_0 \rightarrow U_0$, 
on a introduit une sous-extension galoisienne finie $\wt{F}/F$ de $F^{\rm s\acute{e}p}/F$ dŽployant $G$ (\textit{cf.} \ref{du groupe ˆ l'algbre de Lie}). 
On a choisi un tore maximal $\wt{A}_0$ de $G$ dŽfini sur $F$ et contenant $A_0$. 
Soit $\wt{P}_0^{M_0}$ un sous-groupe de Borel de $M_0$ contenant $\wt{A}_0$; on note $\wt{U}_0^{M_0}$ son radical unipotent. Alors $\wt{P}_0= \wt{P}_0^{M_0}\ltimes U_0$ est un sous-groupe de Borel de $G$ contenant $U_0$, 
de radical unipotent $\wt{U}_0= \wt{U}_0^{M_0}\ltimes U_0$. Soit $\wt{\Delta}_0$ l'ensemble des racines simples de $\wt{A}_0$ dans $\wt{U}_0$. 
On note $\wt{\mathfrak{a}}_{0,\mbb{Q}}$, $\wt{\mathfrak{a}}_{P,\mbb{Q}}$, $\wt{\mathfrak{a}}_{0,\mbb{Q}}^P$ (etc.), les objets dŽfinis comme plus haut en remplaant $F$ par $\wt{F}$. 
Soit $\{\wt{\mu}_{\wt{P}_0,\wt{\alpha}}\,\vert\, \wt{\alpha}\in \wt{\Delta}_0\}$ la base de $\wt{\ag}_{0,\mbb{Q}}^G$ duale de $\wt{\Delta}_0$ et soit 
$\{\wt{\mu}_{P,\wt{\alpha}}\,\vert \, \wt{\alpha}\in \wt{\Delta}_P\}$ la base de $\wt{\mathfrak{a}}_{P,\mbb{Q}}^G$ duale de $\wt{\Delta}_P$. Posons  
$$\wt{\mu}_P = \sum_{\wt{\alpha}\in \wt{\Delta}_P} \wt{\mu}_{P,\wt{\alpha}}= \sum_{\wt{\alpha} \in \wt{\Delta}_0\smallsetminus \wt{\Delta}_0^P} \wt{\mu}_{\wt{P}_0,\wt{\alpha}}\ptf$$ 
Soit $\wt{A}_P= \wt{A}_{M_P}$ le sous-tore $\wt{F}$-dŽployŽ maximal du centre $Z(M_P)$ de $M_P$. 
Le morphisme de restriction $X(\wt{A}_P)\rightarrow X(A_P)$ induit une application surjective $\wt{\Delta}_P \rightarrow \Delta_P$ et 
pour $\alpha \in \Delta_P$, on a $\mu_{P,\alpha}= \sum_{\wt{\alpha}\in \wt{\Delta}_P(\alpha)}\wt{\mu}_{P,\wt{\alpha}}$ o $\wt{\Delta}_P(\alpha)\subset \wt{\Delta}_P$ est la fibre au-dessus de $\alpha$. 
D'aprs \ref{rŽcriture de mu_P}, on a donc $$\wt{\mu}_P=\mu_P\ptf$$ Puisque $U_P$ est une $F$-sous-variŽtŽ de $V$ uniformŽment $F$-instable, l'ŽgalitŽ ci-dessus 
se dŽduit aussi de \ref{BMRT(4.7)}.
}
\end{remark}

La base $\Delta_P$ de $(\mathfrak{a}_{P,\mbb{Q}}^G)^*$ est indŽpendante du choix du 
$F$-sous-groupe parabolique minimal $P_0 \subset P$ (elle ne dŽpend pas non plus du choix de la composante de Levi $M_0$ de $P_0$ dŽfinie sur $F$). 
On peut donc dŽfinir  $\mu_P\in \mathfrak{a}_{P,\mbb{Q},1}\cap  \ag_{P,\mbb{Q}}^G$ sans supposer $P$ standard. 
Si $M$ est une composante de Levi de $P$ dŽfinie sur $F$, on note $\mu_{M,P}\in \check{X}(A_M)_{\mbb{Q}}$ l'ŽlŽment correspondant ˆ $\mu_P$ 
via l'isomorphisme naturel $\check{X}(A_M)_{\mbb{Q}} \simeq \ag_{P,\mbb{Q}}$ (si $P$ est standard, on a donc l'identification $\mu_P=\mu_{M_P,P}$). 
Si $M'$ est une autre composante de Levi de $P$ dŽfinie sur $F$, il existe un unique $u\in U_P(F)$ 
tel que $M' = u M u^{-1}$ et on a $\mu_{M' \!,P}= u \bullet \mu_{M,P}$. En particulier le sous-ensemble $$\{\mu_{M,P}\,\vert\, \hbox{$M$ composante de Levi de $P$ dŽfinie sur $F$}\}\subset \check{X}_F(P)_{\mbb{Q}}$$ 
est un espace principal homogne sous $U_P(F)$.

Le sous-ensemble $ U_P\subset \FU$ est uniformŽment $F$-instable: pour $M$ une composante de Levi de $P$ dŽfinie sur $F$ et 
$n\in \mbb{N}^*$ tel que $\lambda= n\mu_{M,P}\in \check{X}(A_M)$, on a 
$$\lim_{t\rightarrow 0} t^{\lambda}\bullet u = 1\quad \hbox{pour tout}\quad  u\in U_P\ptf$$ Il lui est donc associŽ 
un sous-ensemble $\bs{\Lambda}_{F,U_P}$ de $\check{X}_F(G)_{\mbb{Q}}$. Rappelons que le $F$-saturŽ ${_F\scrX_{U_P}}$ de $U_P$ est le sous-ensemble 
uniformŽment $F$-instable de $\FU$ dŽfini par 
$${_F\scrX_{U_P}}= G_{\mu,1}\quad \hbox{pour un (\textit{i.e.} pour tout)}\quad \mu \in \bs{\Lambda}_{F,U_P}\ptf$$  

\begin{lemme}\label{radical unipotent = lame}
Soit $P$ un $F$-sous-groupe parabolique de $G$.
\begin{enumerate}
\item[(i)] $\bs{\Lambda}_{F,U_P}= \{ \mu_{M,P}\,\vert \, \hbox{$M$ composante de Levi de $P$ dŽfinie sur $F$}\}$ et ${_FP_{U_P}}= P$.
\item[(ii)] ${_F\scrX_{U_P}}= U_P$ (\textit{i.e.} $U_P$ est $F$-saturŽ).
\item[(iii)] $U_P$ est le $F$-saturŽ d'une $F$-lame de $\FU$ si et seulement s'il existe un ŽlŽment $u\in U_P$ tel que $\bs{q}_F(u) =\bs{q}_F(U_P)$, \textit{i.e.} si et seulement si l'ensemble\index{YFUP@${_F\scrY_{U_P}}$} 
$${_F\scrY_{U_P}}\bydef \{u\in U_P\,\vert \, \bs{\Lambda}_{F,u}= \bs{\Lambda}_{F,U_P}\}$$ est non vide; auquel ${_F\scrY_{U_P}}$ est la $F$-lame en question.
\item[(iv)] Si $P$ est un $F$-sous-groupe parabolique minimal de $G$, alors ${_F\scrY_{U_P}}\neq \emptyset$. PrŽcisŽment (pour $P=P_0$), 
${_F\scrY_{U_0}}$ est l'ensemble des $u=j_0(X)$ avec $X\in \mathfrak{u}_0$ tels que pour toute racine $\alpha\in \Delta_0$, la composante $X_\alpha$ de $X$ sur $\mathfrak{u}_\alpha$ 
soit non nulle. 
\end{enumerate}
\end{lemme}

\begin{proof}
On peut supposer $P$ standard. Posons $M=M_P$ et identifions $\mu_P$ ˆ $\mu_{M,P}$ comme plus haut.

Si $\eta \in \bs{\Lambda}_{F,U_P}$, tout ŽlŽment de $G(F)$ qui normalise $U_P$ normalise aussi 
$P_{\eta} = {_FP_{U_P}}$. Puisque le normalisateur de $U_P$ dans $G(F)$ est Žgal ˆ $P(F)$ et que $P_{\eta}$ est son propre normalisateur, on obtient l'inclusion 
$P(F) \subset P_{\eta}(F)$. D'o l'inclusion $P\subset P_{\eta} $. Le tore $A_0$ est donc $(F,U_P)$-optimal (relativement ˆ l'action de $G$ par conjugaison). 
L'unique ŽlŽment de $\bs{\Lambda}_{F,U_P} \cap \check{X}(A_0)_{\mbb{Q}}$ est le co-caractre virtuel $\eta\in \check{X}(A_0)_{\mbb{Q}}$ 
tel que $\langle \alpha , \eta \rangle \geq 1$ pour toute racine 
$\alpha$ de $A_0$ dans $U_P$ et $\| \eta \| $ soit minimal pour cette propriŽtŽ. 
Identifions $\eta$ ˆ un ŽlŽment de $\ag_{P_0,\mbb{Q}}$ via l'isomorphisme naturel $\check{X}(A_0)_{\mbb{Q}} \simeq \ag_{P_0,\mbb{Q}}$. 
La projection orthogonale $\eta_P^G$ de $\eta$ sur $\mathfrak{a}_{P,\mbb{Q}}^G$ pour la dŽcomposition 
$$\mathfrak{a}_{P_0,\mbb{Q}} = \mathfrak{a}_{P_0,\mbb{Q}}^P \oplus \mathfrak{a}_{P,\mbb{Q}}^G \oplus \mathfrak{a}_{G,\mbb{Q}}$$ 
appartient au c™ne $\mathfrak{a}_{P,\mbb{Q},1} \cap \mathfrak{a}_{P,\mbb{Q}}^G$ de $\mathfrak{a}_{P,\mbb{Q}}^G$ et 
la condition de minimalitŽ sur $\eta$ entra"ne que $\eta$ appartient dŽjˆ ˆ ce c™ne; \textit{i.e.} $\eta= \eta_P^G$. Ë nouveau par minimalitŽ, on obtient $\eta= \mu_P$. Par consŽquent ${_FP_{U_P}}=P_{\eta}=P$ et 
${_F\scrX_{U_P}}=G_{\eta,1}=U_P$. Puisque $$\bs{\Lambda}_{F,U_P}= \{u\bullet \eta \,\vert \, u\in U_P(F)\}\quad \hbox{et}\quad u\bullet \eta= \mu_{uMu^{-1},P}\vg$$ 
les points (i) et (ii) sont dŽmontrŽs. 

Quant au point (iii), pour $u\in U_P$, on a $\bs{q}_F(u)\leq \| \mu_P \| \;(= \bs{q}_F(U_P))$ 
avec ŽgalitŽ si et seulement si $\bs{\Lambda}_{F,u}=\bs{\Lambda}_{F,U_P}$ (\ref{lemma 2.8 de [H2]}\,(iii)). 

Supposons $P=P_0$ et prouvons (iv). Soit un ŽlŽment $u=j_0(X)\in U_0$. \'Ecrivons $X= \sum_{\alpha \in \ES{R}^+}X_\alpha$ avec $X_\alpha \in \mathfrak{u}_\alpha$. 
On a $\bs{q}_F(u)\leq \| \mu_{P_0} \| \;(=\bs{q}_F(U_0))$ et 
on veut prouver que 
$$u\in {_F\scrY_{U_0}} \Leftrightarrow \hbox{$X_\alpha\neq 0$ pour toute racine $\alpha \in \Delta_0$}\ptf$$ 
Supposons que $X_\alpha\neq 0$ pour tout $\alpha \in \Delta_0$. 
Soit $S$ un tore $F$-dŽployŽ maximal de $G$ tel que 
$S\subset P_0 \cap {_FP_X}$. Puisque $M_0=Z_G(A_0)$, il existe un $u'\in U_0(F)$ tel que $u'Su'^{-1}= A_0$. Posons $X'= \mathrm{Ad}_{u'}(X)$. Comme 
$A_0\subset u'({_FP_X})u'^{-1}={_FP_{X'}}$, le tore $A_0$ est $(F,X')$-optimal. 
Observons que puisque $u'\in U_0(F)$, on a $X'_\alpha = X_\alpha \neq 0$ pour tout $\alpha\in \Delta_0$. Soit $\eta$ l'unique ŽlŽment de $\check{X}(A_0)_{\mbb{Q}}\cap \bs{\Lambda}_{F,X'}$. 
\'Ecrivons $$\eta= \eta_G + \sum_{\alpha \in \Delta_0}a_\alpha \mu_{P_0,\alpha}$$ avec $\eta_G \in \ag_G$ et $a_\alpha \in \mbb{Q}$. La condition de minimalitŽ sur 
$\| \eta \|$ assure que $\eta_G=0$. Pour $\alpha\in \Delta_0$, comme $X'_\alpha \neq 0$, on a $a_\alpha\geq 1$; et 
puisque $\{\mu_{P,\alpha}\,\vert\, \alpha\in \Delta_0\}$ est une base aig\"ue de $\ag_{0,\mbb{Q}}^G$ (\textit{cf.} \cite[1.2.6]{LW}), on a $\| \eta \| \geq \| \mu_{P_0}\|$ avec ŽgalitŽ si et seulement si $a_\alpha=1$ 
pour tout $\alpha \in \Delta_0$, \textit{i.e.} si et seulement si $\eta = \mu_{P_0}$. Comme d'autre part $\bs{q}_F(X')= \| \eta \| \leq \| \mu_{P_0}\|$, cela prouve que $\eta = \mu_{P_0}$. 
On en dŽduit que $\bs{\Lambda}_{F,X'}= \bs{\Lambda}_{F,U_0}$ et donc que $\bs{\Lambda}_{F,X}= u'^{-1}\bullet \bs{\Lambda}_{F,U_0}= \bs{\Lambda}_{F,U_0}$. 
En particulier $X$ est en position standard ce qui entra"ne que $u$ l'est aussi et que l'on a $\bs{\Lambda}_{F,u}= \bs{\Lambda}_{F,X}= \bs{\Lambda}_{F,U_0}$; \textit{i.e.} $u\in {_F\scrY_{U_0}}$. 
RŽciproquement, supposons 
que $X_\alpha=0$ pour une racine $\alpha\in \Delta_0$. Soit $k$ le plus petit entier $\geq 1$ tel que $\lambda = k\mu_{P_0}$ appartienne ˆ 
$\check{X}(A_0)$ et soit $\eta\in \check{X}(A_0)_{\mbb{Q}}$ le co-caractre virtuel dŽfini par 
$$\eta = \frac{k}{k+1}\mu_{P_0,\alpha} + \sum_{\beta\in \Delta_0\smallsetminus\{\alpha\}} \mu_{P_0,\beta}\ptf$$
Puisque $\langle \gamma,\eta \rangle \geq 1$ pour toute racine 
$\gamma\in \ES{R}^+\smallsetminus \{\alpha\}$, on a $m_u(\eta)=m_X(\eta)\geq 1$. Comme $\{\mu_{P_0,\alpha}\,\vert \, \alpha \in \Delta_0\}$ est une base aig\"{u}e de $\ag_{0,\mbb{Q}}^G$ 
(\textit{cf.} \textit{loc.~cit.}), on a $\| \eta\| < \| \mu_{P_0}\|$. Par consŽquent $\bs{q}_F(u)< \|\mu_{P_0}\|$ et $u\notin {_F\scrY_{U_0}}$. 
Cela achve la preuve du point (iv).
\end{proof}

\textit{\S2.\hskip1mm Le sous-ensemble $\check{X}(A_M,P)\subset \check{X}(A_M)$. ---} Soit $M$ une composante de Levi 
de $P$ dŽfinie sur $F$.  
On note avec un exposant $M$ les objets dŽfinis comme 
prŽcŽdemment en remplaant $G$ par $M$ (\textit{cf.} \ref{morphismes et optimalitŽ})\footnote{Pour $\lambda \in \check{X}(G)$, 
on a notŽ $M_\lambda=G_\lambda(0)$ le facteur de Levi \textit{de $G$} associŽ ˆ $\lambda$. Pour $\lambda\in \check{X}(M)$, le conflit de notations empche de 
noter $M_\lambda^M = M_\lambda(0)$ le facteur de Levi \textit{de $M$} associŽ ˆ $\lambda$; pour Žviter toute ambigu\"{\i}tŽ, on le notera donc $G_\lambda(0)\cap M$.}. 

\begin{lemme}\label{intersection FUG et P}
$\FU \cap P= \FU^M U_P$.\end{lemme}

\begin{proof}
L'inclusion $\FU^MU_P\subset \FU \cap P$ est claire. Pour l'autre inclusion, soit $u\in \FU\cap P$. Soit $T$ un tore $F$-dŽployŽ maximal de 
$G$ contenu dans ${_FP_u}\cap P$ et soit $M'$ une $F$-composante de Levi de $P$ contenant $T$. 
Alors $\check{X}(T)\cap \Lambda_{F,x}^{\rm opt}= \{\lambda'\}$ et $$u\in U_{\lambda'} \cap P\subset U_{\lambda'}^{M'}U_P\subset \FU^{M'}U_P\ptf$$ 
Soit $u'\in U_P(F)$ tel que $M'= u' M u'^{-1}$. On a $\FU^{M'}= u'\FU^M u'^{-1} \subset \FU^M U_P$. Par consŽquent $ \FU^{M'}U_P = \FU^MU_P$ et le 
lemme est dŽmontrŽ. 
\end{proof}

Posons\index{XcheckAMP@$\check{X}(A_M,P)$} $$\check{X}(A_M,P)=\{\mu \in \check{X}(A_M)\,\vert \, P_\mu =P \}\ptf$$ 
Un co-caractre $\mu\in \check{X}(A_M)$ est dans $\check{X}(A_M,P)$ si et seulement si $\langle \alpha, \mu \rangle >0$ pour tout $\alpha \in \Delta_P$. 
Pour $\lambda \in \check{X}(M)$ et $\mu\in \check{X}(A_M)$, puisque $\lambda$ et $\mu$ commutent entre eux, le co-caractre $\lambda + \mu$ est bien dŽfini. 

\begin{lemme}\label{co-caractre et Levi}
Soient  $\lambda \in \check{X}_F(M)$ et $\mu \in \check{X}(A_M,P)$. Il existe un entier $m_0>0$ tel que pour 
tout entier $m\geq m_0$, on ait 
$$P_{\lambda +m\mu} = P_\lambda^M U_P \subset P \quad \hbox{avec}\quad P_\lambda^M = M\cap P_\lambda\ptf$$
\end{lemme}

\begin{proof}
Choisissons un tore $F$-dŽployŽ maximal $T$ dans $M$ contenant $\mathrm{Im}(\lambda)$. Pour $m_0\in \mbb{N}^*$ suffisamment grand, la propriŽtŽ suivante 
est vŽrifiŽe: pour toute racine $\alpha\in \ES{R}_T$ telle que $\langle  \alpha,\mu \rangle \neq 0$, 
on a $\langle \alpha , \mu \rangle \langle \alpha, \lambda + m_0\mu \rangle >0$;  
autrement dit $\langle \alpha , \lambda + m_0\mu \rangle$ est non nul et de mme signe que $\langle \alpha,\mu \rangle$. 
Alors pour tout entier $m\geq m_0$, on a $P_{\lambda + m\mu} \subset P$ et $P_{\lambda +m\mu} = (M\cap P_\lambda) U_P $. 
\end{proof}

\begin{lemme}\label{limites unipotents M vs P}
Soient $w \in \FU^M$, $\lambda\in \check{X}_F(M)$ et $\mu\in \check{X}(A_M)$. 
\begin{enumerate}
\item[(i)] Pour $v\in U_P$, on a $\lambda + \mu \in \Lambda_{F,wv}$ si et seulement si $\lambda \in \Lambda_{F,w}^M$ et $\lambda + \mu\in \Lambda_{F,v}$.
\item[(ii)] Supposons que $\lambda\in \Lambda_{F,w}^M$. Les conditions suivantes sont Žquivalentes:
\begin{enumerate}
\item[(1)] $\lambda + \mu \in \Lambda_{F,U_P}$ et $\lim_{t\rightarrow 0} t^{\lambda + \mu}\bullet wv = \lim_{t\rightarrow 0} t^\lambda \bullet w$ pour tout $v\in U_P$; 
\item[(2)] $\lambda\vert_{A_M} + \mu \in \check{X}(A_M,P)$
\end{enumerate} 
En particulier si $\lim_{t\rightarrow 0}t^\lambda\bullet w = 1$ et $\lambda\vert_{A_M}=1$ (\textit{e.g.} si $\lambda\in \Lambda_{F,w}^{M,{\mathrm{opt}}}$), alors on a 
$\lim_{t\rightarrow 0}t^{\lambda+\mu}\bullet u=1$ pour tout $\mu \in \check{X}(A_M,P)$ et tout $u\in wU_P$.  
\end{enumerate}
\end{lemme}

\begin{proof}
On peut supposer $P$ standard et $M=M_P$. 

Pour $v\in U_P$, on a $t^{\lambda+\mu}\bullet wv = (t^\lambda\bullet w)(t^{\lambda + \mu}\bullet v)$ 
et la limite $\lim_{t\rightarrow 0} t^{\lambda+\mu}\bullet wv$ existe si et seulement si les limites $\lim_{t\rightarrow 0}t^\lambda\bullet w$ et $\lim_{t\rightarrow 0}t^{\lambda+\mu}\bullet v$ existent, auquel cas 
$$\lim_{t\rightarrow 0} t^{\lambda+\mu}\bullet wv = \left(\lim_{t\rightarrow 0} t^\lambda\bullet w\right)\left(\lim_{t\rightarrow 0} t^{\lambda + \mu}\bullet v\right)\ptf$$ Cela prouve (i). 

Supposons de plus que $\lambda\in \Lambda_{F,w}^M$. D'aprs (i), on a $\lambda + \mu \in \Lambda_{F,wU_P}$ si et seulement si $\lambda + \mu \in \Lambda_{F,U_P}$, auquel cas 
pour $v\in U_P$, on a $\lim_{t\rightarrow 0} t^{\lambda+\mu}\bullet wv = \left(\lim_{t\rightarrow 0} t^\lambda\bullet w\right)$ si et seulement si 
$\lim_{t\rightarrow 0} t^{\lambda + \mu}\bullet v = 1$. Soit un ŽlŽment $m\in M(F)$ tel que le co-caractre $\lambda'= m\bullet \lambda$ de $M$ soit en position standard relativement ˆ l'action de $M$ 
par conjugaison, \textit{i.e.} tel que $\mathrm{Im}(\lambda')\subset A_0$ et $P_{\lambda'}^{M}\supset P_0^M\;(=P_0 \cap M)$. 
Alors $\lambda'\in \Lambda_{F,w'}$ avec $w'= m\bullet w\in \FU^M$; et puisque pour $v\in U_P$, on a 
$$t^{\lambda + \mu}\bullet wv= m^{-1}(t^{\lambda'+ \mu}\bullet w'v')m\quad \hbox{avec}\quad v'= mvm^{-1}\vg$$ on est ramenŽ ˆ montrer que $\lim_{t\rightarrow 0} t^{\lambda'+\mu}\bullet v=1$ pour tout $v\in U_P$ 
si et seulement si $\lambda'\vert_{A_M}+\mu\in \check{X}(A_M,P)$. Or $\lim_{t\rightarrow 0} t^{\lambda' + \mu}\bullet v =1$ pour tout $v\in U_P$ si et seulement si $\langle \gamma , \lambda' + \mu \rangle >0$ pour 
toute racine $\gamma$ de $A_0$ dans $U_P$. Une telle racine $\gamma$ s'Žcrit $\gamma = \sum_{\alpha \in \Delta_0} a_\alpha(\gamma) \alpha$ 
avec $a_\alpha(\gamma)\geq 0$ et $a_\alpha(\gamma)\neq 0$ pour au moins un $\alpha\in \Delta_0 \smallsetminus \Delta_0^P$. Identifions $\lambda'$ ˆ un ŽlŽment de $\ag_{0,\mbb{Q}}$ et 
\'ecrivons $\lambda' = \lambda'_G+ \sum_{\alpha\in \Delta_0} b_\alpha(\lambda') \mu_{P_0,\alpha}$ 
avec $\lambda'_G\in \ag_{G,\mbb{Q}}$. Puisque $P^M_{\lambda'}\supset P_0^M$, on a $b_\alpha(\lambda') \geq 0$ pour tout $\alpha \in \Delta_0^P$. On en dŽduit que 
$\langle \gamma, \lambda' + \mu\rangle >0$ pour toute racine $\gamma$ de $A_0$ dans $U_P$ si et seulement si $\langle \alpha , \lambda' + \mu \rangle >0$ pour tout $\alpha \in \Delta_0 \smallsetminus \Delta_0^P$.  
Or pour $\alpha \in \Delta_0\smallsetminus \Delta_0^P$, on a $$\langle \alpha , \lambda' + \mu \rangle  = b_\alpha(\lambda') + \langle \alpha, \mu \rangle 
= \langle \alpha_P , \lambda'\vert_{A_M}+ \mu\rangle\ptf$$ Par consŽquent  
$\langle \gamma , \lambda' + \mu\rangle >0$ pour toute racine $\gamma$ de $A_0$ dans $U_P$ si et seulement si $\lambda'\vert_{A_M} + \mu \in \check{X}(A_M,P)$. Puisque 
$\lambda'\vert_{A_M}+ \mu = \lambda\vert_{A_M}+\mu$, cela prouve (ii). 

Observons que si $\lambda\in \Lambda_{F,w}^{M,\mathrm{opt}}$, alors $\lambda'\in \Lambda_{F,w'}^{M,\mathrm{opt}}$ et  
puisque $\{\mu_{P,\alpha}\,\vert\, \alpha\in \Delta_0\}$ est une base aig\"ue de $\ag_{0,\mbb{Q}}^G$ (\textit{cf.} \cite[1.2.6]{LW}), par minimalitŽ on a $\lambda' = \sum_{\beta\in \Delta_0^P}b_\alpha(\lambda')\mu_{P_0,\alpha}$, 
\textit{i.e.} $\lambda' \in \ag_{0,\mbb{Q}}^P$; ce qui Žquivaut ˆ $\lambda'\vert_{A_M}=0= \lambda\vert_{A_M}$. 
\end{proof}

\vskip2mm
\textit{\S3.\hskip1mm L'application $Z\mapsto {_Fi^P}(Z)$. ---} Si $Z$ est un sous-ensemble uniformŽment $(F,M)$-instable de $\FU^M$, \cad uniformŽment $F$-instable pour l'action de 
$M$ par conjugaison, alors (d'aprs \ref{intersection FUG et P}) $ZU_P\subset \FU\cap P$ est un sous-ensemble uniformŽment $(F,G)$-instable. Il lui est donc associŽ 
un sous-ensemble $\bs{\Lambda}_{F,ZU_P}$ de $\check{X}_F(G)_{\mbb{Q}}$.

\begin{proposition}\label{induction des F-lames}
Soit $Z\subset \FU^M$ un sous-ensemble uniformŽment $(F,M)$-instable.
\begin{enumerate}
\item[(i)] $\bs{\Lambda}_{F,ZU_P}= U_P(F)\bullet (\bs{\Lambda}^M_{F,Z} + \mu_{M,P})$ et ${_FP_{ZU_P}}={_FP_Z^M}\ltimes U_P$. 
\item[(ii)] ${_F\scrX_{ZU_P}}= {_F\scrX_Z^{\!M}}U_P$; en particulier $ZU_P$ est $(F,G)$-saturŽ si et seulement si $Z$ est 
$(F,M)$-saturŽ. 
\item[(iii)] ${_F\scrX_{ZU_P}}$ est le $F$-saturŽ d'une $F$-lame de $\FU$ si et seulement s'il existe un ŽlŽment $u\in ZU_P$ tel que 
$\bs{q}_F(u)=\bs{q}_F(ZU_P)$, \textit{i.e.} si et seulement si  l'ensemble\index{YFZUP@${_F\scrY_{ZU_P}}$} 
$${_F\scrY_{ZU_P}}\bydef \{u\in \FU\,\vert\, \bs{\Lambda}_{F,u}=\bs{\Lambda}_{F,ZU_P}\}$$ est non vide; auquel cas ${_F\scrY_{ZU_P}}$ est la $F$-lame en question. 
\item[(iv)] Si $Z$ est le radical unipotent d'un $F$-sous-groupe parabolique minimal de $M$, alors ${_F\scrY_{ZU_P}}\neq \emptyset$. 
\end{enumerate}
\end{proposition}

\begin{proof} 
On peut supposer $P$ standard et $M=M_P$. 

Soit $S$ un tore $F$-dŽployŽ maximal de $G$ tel que $S\subset P \cap {_FP_{ZU_P}}$ et soit $p\in P(F)$ tel que $pSp^{-1}=A_0$. \'Ecrivons $p=mu$ avec $m\in M(F)$ et $u\in U_P(F)$. Le tore $A_0$ est $F$-optimal pour 
$p ZU_P p^{-1} = Z'U_P$ avec $Z'=mZm^{-1}$. 
Quitte ˆ remplacer $Z$ par $Z'$, on peut supposer que $A_0$ est $(F,ZU_P)$-optimal. Soit $\eta \in \check{X}(A_0)_{\mbb{Q}}$ le co-caractre 
virtuel dŽfini par $\check{X}(A_0)_{\mbb{Q}} \cap \bs{\Lambda}_{F,ZU_P}= \{\eta\}$. 
Identifions $\eta$ ˆ un ŽlŽment de $\ag_{0,\mbb{Q}}$ et Žcrivons $\eta= \eta^P + \eta_P$ avec $\eta^P\in \ag_{0,\mbb{Q}}^P$ et $\eta_P\in \ag_{P,\mbb{Q}}$. Quitte 
ˆ remplacer $Z$ par $nZn^{-1}$ pour un ŽlŽment $n\in N^M(A_0)$, on peut supposer de plus que $\xi = \eta^P$ vŽrifie $P_{\xi}^M \supset P_0^M$. D'aprs \ref{limites unipotents M vs P} et par minimalitŽ, 
on a $\xi\in \bs{\Lambda}_{F,Z}^M$ et $\eta_P= \mu_P\;(=\mu_{M,P})$. 

Puisque $U_{\eta}(F)= U_\xi^M(F) \ltimes U_P(F)$, on a$$U_\eta(F)\bullet \eta = U_\eta (F)\bullet (\xi + \mu_P) = 
U_P(F) \bullet((U_\xi^M(F)\bullet \xi) + \mu_P)\ptf$$ Or on sait que $\bs{\Lambda}_{F,ZU_P}= U_\eta(F)\bullet \eta$ et $\bs{\Lambda}_{F,Z}^M= U_\xi^M(F)\bullet \xi$. 
D'autre part puisque $P_\xi^M\;(={_FP_Z^M})\supset P_0^M$, on a $\langle \alpha , \xi \rangle \geq 0$ pour toute racine 
$\alpha \in \Delta_0^P$. D'aprs la preuve de \ref{co-caractre et Levi}, 
cela assure que $P_{\eta}\;(={_FP_{ZU_P}}) = P_\xi^M \ltimes U_P$. Le point (i) est dŽmontrŽ.

D'aprs (i), on a ${_F\scrX_{ZU_P}}= G_{\eta,1}$. Pour toute racine $\alpha\in \ES{R}$, en Žcrivant $\alpha=\alpha^P + \alpha_P$ 
avec $\alpha^P\in \ag_{0,\mbb{Q}}^P$ et $\alpha_P\;(=\alpha\vert_{\ag_{P,\mbb{Q}}})\in \ag_{P,\mbb{Q}}$, on a 
$$\langle \alpha, \eta\rangle = \langle \alpha^P\!,\xi \rangle + \langle \alpha_P,\mu_P\rangle\ptf$$ 
On en dŽduit que $G_{\eta,1}= M_{\xi,1} G_{\mu_P,1}$ avec $M_{\xi,1}\;(= G_{\xi,1} \cap M)= {_F\scrX_Z^{\!M}}$ et $G_{\mu_P,1}=U_P$. 
D'o le point (ii). 

Quant au point (iii), pour $u\in ZU_P$, on a $\bs{q}_F(u)\leq \| \eta \| \;(= \bs{q}_F(ZU_P))$ 
avec ŽgalitŽ si et seulement si $\bs{\Lambda}_{F,u}=\bs{\Lambda}_{F,ZU_P}$ (\ref{lemma 2.8 de [H2]}\,(iii)). 

Le point (iv) est une consŽquence de \ref{radical unipotent = lame}\,(iv): si $Z=U_0^M$, on a $ZU_P=U_0$. 
\end{proof} 

\begin{lemme}
L'application $Z\mapsto ZU_P$ induit une bijection entre:
\begin{itemize}
\item les sous-ensembles uniformŽment $(F,M)$-instables de $\FU^M$;
\item les sous-ensembles uniformŽment $F$-instables $\Omega$ de $\FU\cap P$ vŽrifiant $\Omega U_P= \Omega$. 
\end{itemize}
La bijection rŽciproque est donnŽe par $\Omega \mapsto \Omega\cap M$.
\end{lemme}

\begin{proof}
On peut supposer $P$ standard et $M=M_P$. Si $\Omega$ est un sous-ensemble uniformŽnent $F$-instable de $\FU\cap P$ vŽrifiant $\Omega U_P= \Omega $, on a l'ŽgalitŽ 
$$\Omega =(\Omega \cap M)U_P\ptf$$ Il suffit donc de prouver que pour un tel $\Omega $, l'intersection $\Omega \cap M$ est uniformŽment $(F,M)$-instable. 
Soit $S$ un tore $F$-dŽployŽ maximal de $G$ tel que $S\subset {_FP_{Z}}\cap P$ et soit $p\in P(F)$ tel que $pSp^{-1}= A_0$. \'Ecrivons $p=mu$ avec 
$m\in M(F)$ et $u\in U_P$ et posons $\Omega' = p\Omega p$. On a $$\Omega' U_P = \Omega'= (\Omega'\cap M)U_P\quad\hbox{avec}\quad \Omega'\cap M = m(\Omega \cap M)m^{-1}\ptf$$
Par construction le tore $A_0$ est $(F,\Omega')$-optimal. Soit $\eta\in \check{X}(A_0)_{\mbb{Q}}$ le co-caractre virtuel dŽfini par 
$\check{X}(A_0)_{\mbb{Q}} \cap \bs{\Lambda}_{F,\Omega' }=\{\eta\}$. \'Ecrivons $\eta = \eta^P + \eta_P$ avec $\eta^P\in \ag_{0,\mbb{Q}}^P$ et $\eta_P\in \ag_{P,\mbb{Q}}$. 
Soit un entier $k\geq 1$ tel que $k\eta^P$ et $k\eta_P$ appartiennent ˆ 
$\check{X}(A_0)$. D'aprs \ref{limites unipotents M vs P}, $k\eta^P \in \Lambda_{F, \Omega'\cap M}^M$ et $k\eta_P\in \check{X}(A_M,P)$. En particulier 
$\Omega'\cap M$ est uniformŽment $(F,M)$-instable; d'ailleurs par minimalitŽ on a $\eta^P\in \bs{\Lambda}_{F,\Omega'\cap M}^M$ et $\eta^P=\mu_P$. 
Cela prouve que $\Omega \cap M= m^{-1}(\Omega' \cap M) m$ est unifomŽment $(F,M)$-instable. 
\end{proof}

Pour $Z\subset \FU^M$ un sous-ensemble uniformŽment $(F,M)$-instable, on note ${_Fi^P}(Z)$ le sous-ensemble $F$-saturŽ de $\FU$ dŽfini par\index{iFPZ@${_Fi^P}(Z)$} 
$${_Fi^P}(Z)= {_F\scrX_Z^{\!M}}U_P={_F\scrX_{ZU_P}}\ptf$$ 
Il ne dŽpend que du $(F,M)$-saturŽ ${_F\scrX_Z^{\!M}}$ de $Z$: on a $${_Fi^P}(Z)={_Fi^P}({_F\scrX_Z^{\!M}})\ptf$$ 
On pose aussi\index{iFPZrat@$i_F^P(Z)$} $$i_F^P(Z)= {_Fi^P}(Z)(F)= \scrX_{F,Z}^{\!M}U_P(F)\quad \hbox{avec} \quad \scrX^{\!M}_{F,Z}=M(F)\cap {_F\scrX^{\!M}_Z}\ptf$$ 
Observons que si $F$ est infini, puisque ${_Fi^P}(Z)$ est un $F$-espace affine, l'ensemble $i_F^P(Z)$ est dense dans ${_Fi^P}(Z)$; en particulier il est non vide. 

\subsection{Induction parabolique des $F$-strates}\label{induction parabolique des F-strates} 
On a une notion d'induction parabolique pour les $G$-orbites de $\UU$ due ˆ 
Lusztig-Spaltenstein \cite{LS} (\textit{cf.} \ref{lien avec l'IP de LS}, \S1). Sous certaines hypothses, 
on dŽfinit dans cette sous-section une variante de cette construction pour les $F$-strates de $\UF$. 

\vskip2mm
\textit{\S1.\hskip1mm L'application $Z\mapsto {_FI}_{\!P}^G(Z)$. ---} 
Soient $P$ un $F$-sous-groupe parabolique de $G$ et $M$ une composante de Levi de $P$ dŽfinie sur $F$. Pour tout sous-ensemble uniformŽment $(F,M)$-instable $Z$ de $\FU^M$, 
le sous-ensemble $F$-saturŽ ${_Fi^P}(Z)={_F\scrX_{ZU_P}}$ de $\FU$ est un $F$-sous-groupe unipotent fermŽ connexe de $P$: on a 
$${_Fi^P}(Z)= G_{\mu,1}\quad\hbox{pour un (\textit{i.e.} pour tout)}\quad \mu\in \bs{\Lambda}_{F,ZU_P}\ptf$$
En particulier ${_Fi^P}(Z)$ est une $F$-variŽtŽ irrŽductible. Puisque $\UU= {_{\smash{\overline{F}}}\UU}$ est rŽunion finie de ($\overline{F}$-)strates et que ces strates sont 
localement fermŽes dans $G$, il en existe une et une seule $\bsfrY_{u}$ (avec $u\in \UU$ que l'on peut choisir dans ${_Fi^P}(Z)$) telle que l'intersection $\bsfrY_{u}\cap {_Fi^P}(Z)$ soit dense dans ${_Fi^P}(Z)$\footnote{Il s'agit lˆ 
d'une simple variante de la construction de \ref{comparaison avec les strates gŽo}: pour $P=G$ et $Z=\{w\}$, on a ${_Fi^G}(w)={_F\scrX_{w}}$ et 
$\bsfrY(F,Z\,\vert\,G)$ 
est la $\overline{F}$-strate de $\UU$ notŽe $\bsfrY(F,w)$ dans \textit{loc.~cit.}} ; 
on la note \index{YbsFZP@$\bsfrY(F,Z\,\vert\,P)$}
$$\bsfrY(F,Z\,\vert\, P) = \bsfrY_{u}\ptf$$ 
Observons que l'intersection 
$\bsfrY(F,Z\,\vert\,P) \cap {_Fi^P}(Z)$ est ouverte dans ${_Fi^P}(Z)$. 

On aimerait associer ˆ cette strate $\bsfrY(F,Z\,\vert\,P)$ de $\UU$ une $F$-strate de $\FU$. 
On dispose pour cela du lemme \ref{bonnes F-strates}: si pour un (\textit{i.e.} pour tout) $u\in \bsfrY(F,Z\,\vert\,P)$, l'intersection $\check{X}_F(G)_{\mbb{Q}}\cap G\bullet \bs{\Lambda}_{u}$ est non vide, alors il existe une unique $F$-strate ${_F\bsfrY_{u'}}$ de $\FU$ (avec $u'\in \FU$) telle que $\bsfrY(F,Z\,\vert\,P)= \bsfrY(F,u') = \bsfrY_{u'}$. Si de plus le corps $F$ est infini, 
on a le rŽsultat plus prŽcis suivant:

\begin{lemme}\label{strate induite et points rationnels} 
On suppose que le corps $F$ est infini. Pour tout sous-ensemble 
uniformŽment $(F,M)$-instable $Z$ de $\FU^M$, l'intersection $\bsfrY(F,Z\,\vert\,P)\cap i_{F,P}^G(Z)$ est dense dans ${_Fi^P}(Z)$; en particulier elle est non vide. Si de plus $\check{X}_F(G)_{\mbb{Q}}\cap \bs{\Lambda}_u\neq \emptyset $ pour un $u\in \UF\cap  \bsfrY(F,Z\,\vert\,P)$, on a $\bsfrY(F,Z\,\vert\,P) = \bsfrY_u= \bsfrY(F,u)$. 
\end{lemme}

\begin{proof}
La $F$-variŽtŽ ${_Fi^P}(Z)= {_F\scrX_{ZU_P}}$ est un $F$-espace affine, par suite $i_F^P(Z)\;(={_Fi^P}(Z)(F))$ est dense dans 
${_Fi^P}(Z)$. On en dŽduit que  l'intersection $\bsfrY(F,Z\,\vert\,P)\cap {i_F^P}(Z)$ est dense dans ${_Fi^P}(Z)$; en particulier elle est non vide. 
La dernire assertion rŽsulte de 
\ref{strate gŽo associŽe ˆ F-strate}\,(i).
\end{proof}

DŽsormais et jusqu'ˆ la fin de \ref{induction parabolique des F-strates}, on suppose que l'hypothse \ref{hypo bonnes F-strates faible} 
est vŽrifiŽe pour $G$. Cela assure (\ref{descente 3 sous HBFS}\,(iii)) que l'application $Y\mapsto \UF \cap Y$ 
induit une bijection entre les $\overline{F}$-strates de $\UU$ qui possdent un point $F$-rationnel et les $F$-strates de $\UF$. 
D'autre part l'application $Y\mapsto G(F)\cap Y$ induit une bijection entre les $F$-strates de $\FU$ qui possdent un point $F$-rationnel et les $F$-strates de $\UF$. 

On suppose aussi que le corps $F$ est infini. Le cas des corps fini sera traitŽ dans le dernier paragraphe (\S6) de cette sous-section \ref{induction parabolique des F-strates}.

Pour tout sous-ensemble uniformŽment $(F,M)$-instable $Z$ de $\FU^M$, on note ${_FI_P^G}(Z)$ la $F$-strate 
de $\FU$ dŽfinie par (\ref{strate induite et points rationnels}) \index{IFPGZ@${_FI_P^G}(Z)$} 
$${_FI_P^G}(Z)= {_F\UU}\cap \bsfrY(F,Z\,\vert\,P)\ptf$$
Par construction, ${_FI_P^G}(Z)$ possde un point $F$-rationnel, que l'on peut choisir dans 
${i_F^P}(Z)= {_Fi^P}(Z)(F)$. On note $I_{F,P}^G(Z)$ la $F$-strate de $\UF$ dŽfinie par\index{IFPGZrat@$I^G_{F,P}(Z)$} 
$$I^G_{F,P}(Z)= G(F)\cap {_FI_P^G}(Z)= \UU_F \cap \bsfrY(F,Z \,\vert\, P)\ptf$$ 
On a donc 
$$I_{F,P}^G(Z)= \UF \cap \bsfrY_u \quad \hbox{pour tout} \quad u\in \bsfrY(F,Z\,\vert\, P) \cap i_{F,P}^G(Z) \ptf $$

\begin{remark}
\textup{Par dŽfinition, une $F$-strate de $\FU$ de la forme ${_FI_P^G}(Z)$ pour un sous-ensemble uniformŽment $(F,M)$-instable 
$Z$ de ${_F\UU^M}$ possde un point $F$-rationnel. Par consŽquent la donnŽe de ${_FI_P^G}(Z)$ Žquivaut ˆ celle de $I_{F,P}^G(Z)$.}
\end{remark}

\begin{lemme}\label{FIPGZ} ($F$ est infini et l'hypothse \ref{hypo bonnes F-strates faible} est vŽrifiŽe pour $G$.) 
Soit $Z$ un sous-ensemble uniformŽment $(F,M)$-instable de $\FU^M$. 
\begin{enumerate}
\item[(i)] ${I_{F,P}^G}(Z)$ est l'unique $F$-strate de $\UF$ telle que l'intersection ${I_{F,P}^G}(Z)\cap {i_F^P}(Z)$ soit dense dans 
${_Fi^P}(Z)$.
\item[(ii)]Pour $m\in M(F)$, on a ${_FI_P^G}(m\bullet Z)= {_FI_P^G}(Z)$ et donc $I_{F,P}^G(m\bullet Z)= I_{F,P}^G(Z)$.
\end{enumerate}
\end{lemme}

\begin{proof}
Prouvons (i). On a vu (\ref{strate induite et points rationnels}) que $\bsfrY(F,Z\,\vert\,P)\cap {i_F^P}(Z)$ est dense dans ${_Fi^P}(Z)$; 
or $\bsfrY(F,Z\,\vert\,P)\cap {i_F^P}(Z)$ co\"{\i}ncide avec 
$$(\UF \cap \bsfrY(F,Z\,\vert\,P)\cap {i_F^P}(Z)={I_{F,P}^G}(Z)\cap {i_F^P}(Z)\ptf$$ Quant ˆ l'unicitŽ, soit $\bsfrY_{F,u}$ une $F$-strate de $\UF$ (avec $u\in \UF$) telle que l'intersection 
${\bsfrY_{F,u}}\cap {i_F^P}(Z)$ soit dense 
dans ${_Fi^P}(Z)$. Puisque $\bsfrY_{F,u} = \UF \cap \bsfrY_{u}$, on a l'ŽgalitŽ $\bsfrY_u \cap i_F^P(Z) = \bsfrY_{F,u} \cap i_F^P(Z)$ et 
la strate $\bsfrY_u$ intersecte ${_Fi^P}(Z)$ de manire dense; donc $\bsfrY_u = \bsfrY(F,Z\,\vert\, P)$ et $\bsfrY_{F,u}=I_{F,P}^G(Z) $. Cela prouve (i). 
 
Prouvons (ii). Pour $m\in M(F)$, l'ensemble $m\bullet Z=mZm^{-1}$ est uniformŽment $(F,M)$-instable et on a $${_Fi^P}(m\bullet Z)= (m\bullet {_F\scrX_Z})U_P = m \bullet {_Fi^P}(Z)\ptf$$ 
Puisque les $\overline{F}$-strates de $\UU$ sont $G$-invariantes, on a $\bsfrY(F,m\bullet Z\,\vert\,P)= \bsfrY(F,Z\,\vert\,P)$.  
D'o le point (ii). 
\end{proof}

\begin{remark}\label{caractŽrisation de l'intersection}
\textup{
Supposons $P$ standard et $M=M_P$. Soit $Z\subset \FU^M$ un sous-ensemble uniformŽment $(F,M)$-instable et soit $\bsfrY_{F,u'}$ une $F$-strate de $\FU$ avec $u'\in \UF$ en position standard. Pour dŽterminer si l'intersection 
$\bsfrY_{F,u'} \cap {i_F^P}(Z)$ est dense dans ${i_F^P}(Z)$, quitte ˆ remplacer $Z$ par $m\bullet Z$ pour un $m\in M(F)$, 
on peut d'aprs \ref{FIPGZ}\,(iii) supposer que le $F$-sous-groupe parabolique ${_FP_Z^M}$ de $M$ associŽ ˆ $Z$ est standard, 
\cad qu'il contient $P_0^M=M\cap P_0$. Puisque la $F$-lame $\scrY_{F,u'}$ de $\UF$ est $P_0(F)$-invariante, 
d'aprs la dŽcomposition de Bruhat on a $$\bsfrY_{F,u'} = U_0(F)W_0\bullet \scrY_{F,u'}\quad \hbox{avec}\quad  
W_0= N^G(A_0)/M_0\ptf$$ Puisque $i_F^P(Z)= \scrX_{F,Z}^{\!M}U_P(F)$ et que $\scrX_{F,Z}^{\!M}$ est $P_0^M(F)$-invariant, l'ensemble 
$i_F^P(Z)$ est $P_0(F)$-invariant et donc \textit{a fortiori} $U_0(F)$-invariant. 
On en dŽduit que 
$$\bsfrY_{F,u'} \cap {i_F^P}(Z)= U_0(F) \bullet \left((W_0\bullet \scrY_{F,u'} )\cap {i_F^P}(Z)\right)\ptf$$ 
}
\end{remark}

\begin{lemme}\label{cas particulier d'induite} ($F$ est infini et l'hypothse \ref{hypo bonnes F-strates faible} est vŽrifiŽe pour $G$.)
Soit $Z$ un sous-ensemble uniformŽment $(F,M)$-instable de $\FU^M$. Si l'ensemble ${_Fi^P}(Z)$ est le $F$-saturŽ d'une $F$-lame de $\FU$ qui possde un point $F$-rationnel, \textit{i.e.} (compte-tenu de \ref{hypo bonnes F-strates faible}) 
si l'ensemble ${_F\scrY_{ZU_P}}$ dŽfini en \ref{induction des F-lames}\,(iii) est un ouvert dense de ${_Fi^P}(Z)$, on a 
$${_FI_P^G}(Z)\cap {_Fi^P(Z)}= {_F\scrY_{ZU_P}} \quad \hbox{et}\quad {_FI_P^G}(Z)= G(F)\cdot {_F\scrY_{ZU_P}}\ptf$$
\end{lemme}

\begin{proof}
Si $u\in {_F\scrY_{ZU_P}}$, alors ${_F\scrY_{ZU_P}}={_F\scrY_u}$ et ${_Fi^P}(Z)= {_F\scrX_u}$. Si de plus on peut choisir $u$ dans $G(F)$, \textit{i.e.} dans $\UF$, alors (d'aprs \ref{hypo bonnes F-strates faible}) ${_F\scrY_u}= \scrY_u$ est ouvert dense dans ${_F\scrX_u}= \scrX_u$. Puisque ${_F\bsfrY_u} \cap {_F\scrX_u}= {_F\scrY_u}$, on a ${_FI_P^G}(Z)= {_F\bsfrY_u}$. D'o le lemme.
\end{proof}

\begin{exemple}\label{F-strtae rŽgulire}
\textup{On a vu (\ref{radical unipotent = lame}\,(iv)) que $U_0$ est le $F$-saturŽ d'une $F$-lame ${_F\scrY_{U_0}}$ de $\FU$. 
On en dŽduit (d'aprs \ref{cas particulier d'induite}) que la $F$-strate $G(F)\cdot {_F\scrY_{U_0}}$ de $\FU$ est l'induite parabolique 
${_FI_{P_0}^G}(1)$ de la $F$-strate triviale $\{1\}$ de $\FU^{M_0}$. On appelle cette induite la \textit{$F$-strate rŽgulire} de $\FU$. } 
\end{exemple}

\vskip2mm
\textit{\S2.\hskip1mm L'application $w\;(\in \FU^M)\mapsto {_FI_P^G}(w)$. ---}
D'aprs \ref{FIPGZ}\,(ii), l'application qui ˆ $w\in \FU^M$ associe la $F$-strate ${_FI_P^G}(w)$ de $\FU$ est constante sur 
les $F$-strates de $\FU^M$: on a $${_FI_P^G}(w')={_FI_P^G}(w)\quad \hbox{pour tout}\quad w'\in {_F\bsfrY_w^M}\ptf$$ 

\begin{remark}
\textup{Pour $P=G$ et $u\in \UF$, on a ${_FI_G^G}(u)={_F\bsfrY_u}$. En revanche pour $u\in \FU$ tel que la $F$-strate ${_F\bsfrY_u}$ 
de $\FU$ ne possde aucun point $F$-rationnel, on a 
${_FI_G^G}(u)\neq {_F\bsfrY_u}$. }
\end{remark}

Pour $u\in \FU\cap P$, on Žcrit $u=wv$ avec $w\in \FU^M$ et $v\in U_P$ (\ref{intersection FUG et P}), et on pose 
$${_FI_P^G}(u)= {_FI_P^G}(w)\quad \hbox{et}\quad I_{F,P}^G(u)=G(F)\cap {_FI_P^G}(u)\;(=I_{F,P}^G(w)) \ptf$$

\begin{remark}
\textup{
Pour $u\in \FU\cap P$, la $F$-strate ${_FI_P^G}(u)$ de $\FU$ ne dŽpend pas de $M$ (raison pour laquelle $M$ n'appara"t pas dans la notation).
En effet si $M'$ est une autre composante de Levi de $P$ dŽfinie sur $F$, alors $M'= u'Mu'^{-1}$ pour un (unique) ŽlŽment $u'\in U_P(F)$. \'Ecrivons 
$u= wv$ avec $w\in \FU^M$ et $v\in U_P$.  On a $u=w'v'$ avec $w'= u'wu'^{-1}\in \FU^{M'}$ et $v'=u'w^{-1}u'^{-1}wv\in U_P$. Donc 
$${_F\scrX_{w'}^{\!M'}}U_P = (u'\bullet {_F\scrX_w^{\!M}})U_P = u' \bullet ({_F\scrX_w^{\!M}}U_P)\ptf$$ Pour une $\overline{F}$-strate $\bsfrY_{u_1}$ de $\UU$, 
l'intersection $\bsfrY_{u_1} \cap {_F\scrX_w^{\!M}}U_P$ est dense dans ${_F\scrX_w^{\!M}}U_P$ si et seulement si l'intersection 
$\bsfrY_{u_1} \cap {_F\scrX_{w'}^{\!M'}}U_P$ est dense dans ${_F\scrX_{w'}^{\!M'}}U_P$.
}\end{remark}

Le lemme \ref{FIPGZ} caractŽrise les $F$-strates de $\FU$ qui sont induites ˆ partir des $F$-strates de $\FU^M$. Pour celles qui sont induites ˆ partir des 
$F$-strates de $\FU^M$ qui possdent un point $F$-rationnel, en supposant l'hypothse \ref{hypo bonnes F-strates faible} vŽrifiŽe pour 
$M$, on ˆ la caractŽrisation alternative suivante: 

\begin{lemme}\label{FIPGwrat}
($F$ est infini et l'hypothse \ref{hypo bonnes F-strates faible} est vŽrifiŽe pour $G$). On suppose que l'hypothse \ref{hypo bonnes F-strates faible} est vŽrifiŽe pour $M$.  
Soit $w\in \UF^M\;(=M(F)\cap {_F\UU^M})$ et soit $\bsfrY$ une $F$-strate de $\FU$. Les conditions suivantes sont Žquivalentes:
\begin{enumerate}
\item[(i)] $ \bsfrY= {_FI_P^G}(w)$, auquel cas $\bsfrY={_F\bsfrY_u}$ pour un $u\in i_F^P(w)$.
\item[(ii)] $ \bsfrY \cap {_F\scrY_{w}^{\!M}}U_P$ est ouvert dense dans ${_Fi^P}(w)$; 
\item[(iii)] $ \bsfrY \cap \scrY_{F,w}^{\!M}U_P(F)$ est dense dans ${_Fi^P}(w)$;
\item[(iv)] $ \bsfrY \cap {_F\bsfrY_{w}^{M}}U_P$ est dense dans ${_F\bsfrX_w^{M}}U_P$;
\item[(v)] $ \bsfrY \cap \bsfrY_{F,w}^{M}U_P(F)$ est dense dans ${_F\bsfrX_w^M}U_P$;
\end{enumerate} 
\end{lemme}

\begin{proof}
L'hypothse \ref{hypo bonnes F-strates faible} pour $M$ assure que la $F$-lame ${_F\scrY_w^{\!M}}$ de $\FU^M$ est ouverte dense 
dans son $(F,M)$-saturŽ ${_F\scrX_w^{\!M}}$ (\ref{le cas o F-lame = lame gŽo}). 
Par consŽquent l'ensemble ${_F\scrY_w^{\!M}}U_P$ est ouvert dense dans ${_Fi^P}(w)$ et $\bsfrY(F,w\,\vert\,P)$ est l'unique 
$\overline{F}$-strate de $\UU$ qui intersecte ${_F\scrY_w^{\!M}}U_P$ de manire ouverte dense, avec $$\bsfrY(F,w\,\vert\,P)\cap {_F\scrY_w^{\!M}}U_P=\bsfrY \cap 
{_F\scrY_w^{\!M}}U_P\ptf$$ D'o l'Žquivalence $(i)\Leftrightarrow(ii)$. 

Observons 
que l'intersection $ {_FI_P^G}(w) \cap {_F\scrY_{w}^{\!M}}U_P$ est ouverte dans ${_Fi^P}(w)$. 
Puisque ${_F\scrX_w^{\!M}}\simeq_F \mbb{A}_F^n$, la $F$-lame $\scrY_{F,w}^{\!M}$ de $\UF^M$ est dense dans ${_F\scrX_w^{\!M}}$ (\textit{cf.} \ref{densitŽ bonnes F-strates}\,(ii)). 
Par consŽquent $\scrY_{F,w}^{\!M} U_P$ est dense dans ${_Fi^P}(w)$. On en dŽduit que  
$ \bsfrY \cap {_F\scrY_w^{\!M}}U_P$ est ouvert dense dans ${_Fi^P}(w)$ si et seulement si $ \bsfrY \cap \scrY_{F,w}^{\!M}U_P(F)$ est dense dans ${_Fi^P}(w)$. D'o 
l'Žquivalence $(ii)\Leftrightarrow (iii)$. 

Comme 
$$\bsfrY \cap {_F\bsfrY_w^M}U_P = \coprod_{m\in M(F)/P_{F,w}^M} m \bullet (\bsfrY \cap {_F\scrY_w^{\!M}}U_P)$$ et $${_F\bsfrX_w^M}U_P = M(F)\bullet  {_Fi^P}(w)\vg$$ 
on a aussi l'Žquivalence $(ii)\Leftrightarrow (iv)$. On obtient l'Žquivalence $(iii)\Leftrightarrow (v)$ de la mme manire.
\end{proof}

\begin{remark}\label{FiP=iP}
\textup{(Sous les hypothses de \ref{FIPGwrat}.) \begin{enumerate}
\item[(i)]Pour $w\in \UF^M$, d'aprs l'hypothse \ref{hypo bonnes F-strates faible} pour $M$, non seulement la $F$-lame ${_F\scrY_w^{\!M}}$ est ouverte dans son $F$-saturŽ ${_F\scrX_w^{\!M}}$, mais on a aussi 
${_F\scrY_w^{\!M}}= \scrY_w^{\!M}$ et ${_F\scrX_w^{\!M}}= \scrX_w^{\!M}$ (\ref{le cas o F-lame = lame gŽo}). Par consŽquent $${_Fi^P}(w)= i^P(w)= \scrX_w U_P\ptf$$ 
\item[(ii)] Si $F=\overline{F}$, alors $F$ est infini et l'hypothse \ref{hypo bonnes F-strates faible} est vide. Dans ce cas pour toute $(\overline{F}$-)strate $\bsfrY^M$ de $\UU^M$, $I_P^G(\bsfrY^M)= {_{\overline{F}}I_P^G}(\bsfrY^M)$ est l'unique strate de $\UU$ qui intersecte $\bsfrY^MU_P$ de manire dense. 
\end{enumerate}  }
\end{remark}

\vskip2mm
\textit{\S3.\hskip1mm TransitivitŽ des applications $w\;(\in \FU^M)\mapsto {_FI_P^G}(w)$. ---} 
Soit $P'$ un $F$-sous-groupe paraboli\-que de $G$ contenant $P$ et soit $M'$ l'unique composante de Levi de $P'$ contenant $M$. 
Puisque $M$ est dŽfini sur $F$, $M'$ l'est aussi; et l'intersection $P\cap M'$ est un $F$-sous-groupe 
parabolique de $M'$ de composante de Levi $M$. 

Supposons de plus que l'hypothse \ref{hypo bonnes F-strates faible} soit vŽrifiŽe pour $M'$. Pour 
$w\in \FU^M$, on peut dŽfinit comme plus haut la $F$-strate  
${_FI_{P\cap M'}^{M'}}(w)$ de $\FU^{M'}$. 

\begin{lemme}\label{transitivitŽ de l'induction} ($F$ est infini et l'hypothse \ref{hypo bonnes F-strates faible} est vŽrifiŽe pour $G$ et pour $M'$.) 
Pour $w\in \FU^M$ et $w'\in {_FI_{P\cap M'}^{M'}}(w)$, on a 
$${_FI_P^G}(w)= {_FI_{P'}^G}(w')\ptf$$
\end{lemme}

\begin{proof}
Posons $$\bsfrY'= {_FI_{P\cap M'}^{M'}}(w)\quad \hbox{et}\quad\bsfrY= {_FI_{P'}^G}(w')\ptf$$ 
Par construction $\bsfrY'={_F\bsfrY_{u'}}$ avec $u'\in \UU_F^{M'}\cap \bsfrY^{M'}(F,w\,\vert \,P\cap M')$ et 
$\bsfrY={_F\bsfrY_u}$ avec $u\in \UF \cap \bsfrY(F,w'\,\vert\,P')$. 
L'ensemble $(\bsfrY^{M'}_{u'}\cap {_Fi^{P\cap M'}}(w))U_{P'}$ est ouvert dense dans $${_Fi^{P\cap M'}}(w)U_{P'}=({_F\scrX_{w}^{\!M}}U_{P\cap M'})U_{P'}={_F\scrX_w^{\!M}}U_P= {_Fi^P}(w)\ptf$$ 
Puisque d'autre part $\bsfrY_u \cap {_Fi^{P'}}(w')$ est par dŽfinition ouvert dense dans ${_Fi^{P'}}(w')U_{P'}$, on en dŽduit que l'ensemble 
$\bsfrY_u\cap (\bsfrY^{M'}_{u'}\cap {_F\scrX_w^{\!M}}U_{P\cap M'})U_{P'}$ est lui aussi ouvert dense dans ${_Fi^P}(w)$; et donc \textit{a fortiori} que l'ensemble 
$\bsfrY_u\cap  {_F\scrX_w^{\!M}}U_{P\cap M'}U_{P'}= \bsfrY_u\cap {_Fi^P}(w)$ est dense 
dans ${_Fi^P}(w)$. Cela prouve que $\bsfrY_u = \bsfrY(F,w\,\vert\, P)$ et donc (puisque $u\in \UF$) que ${_F\bsfrY_u}= {_FI_P^G}(w)$. 
\end{proof}

\vskip2mm
\textit{\S4.\hskip1mm Application $w\;(\in \UF^M) \mapsto I_{F,P}^G(w)$ et descente sŽparable. ---} Soit $E/F$ une extension sŽparable (algŽbrique ou non) telle que $(E^{\mathrm{s\acute{e}p}})^{\mathrm{Aut}_F(E^{\mathrm{s\acute{e}p}})}=F$. 
D'aprs \ref{descente sŽparable c™ne F-nilpotent}, on a 
$$\UF = G(F) \cap \UU_E\ptf$$ Rappelons aussi  (\ref{notation allŽgŽe}\,(ii)) que si $\bsfrY$ est une 
$F$-strate de $\UF$, on a notŽ $\bsfrY_E$ la $E$-strate de $\UU_E$ dŽfinie par $\bsfrY_E= \bsfrY_{E,u}$ pour un (\textit{i.e.} pour tout) $u\in \bsfrY$. D'aprs \ref{descente sŽparable strates}, on a 
$$\bsfrY= G(F)\cap \bsfrY_E\ptf$$

Rappelons que l'on a supposŽ le corps $F$ infini et l'hypothse \ref{hypo bonnes F-strates faible} vŽrifiŽe pour $G$ (relativement ˆ $F$). 
Pour $u\in \UF$, puisque (d'aprs \ref{hypo bonnes F-strates faible}) $\check{X}_F(G)_{\mbb{Q}}\cap \bs{\Lambda}_u \neq \emptyset$, 
on a \textit{a fortiori} $\check{X}_E(G)_{\mbb{Q}} \cap \bs{\Lambda}_u\neq \emptyset$ et $$\bsfrY(F,u) = \bsfrY_u = \bsfrY(E,u)\ptf$$ 
D'autre part pour $w\in \UF^M$, puisque (d'aprs \ref{descente sŽparable lames}\,(ii)) ${_E\scrX_w^{\!M}}= {_F\scrX_w^{\!M}}$, on a 
$$\bsfrY(E,w\,\vert\, P) = \bsfrY(F,w\,\vert \, P)\ptf$$ 
Pour $w\in \UF^M$, on peut donc dŽfinir la $E$-strate ${_EI_P^G}(w)$ de ${_E\UU^M}$ et la $E$-strate $I_{E,P}^G(w)$ comme on la fait pour $F$, 
en posant $${_EI_P^G}(w) = {_E\UU} \cap \bsfrY(E,w\,\vert \, P) \quad \hbox{et} \quad I_{E,P}^G(w) = \UU_E \cap \bsfrY(E,w\,\vert \, P) \ptf $$ 
L'application qui ˆ $w\in \UF^M$ associe la $F$-strate $I_{F,P}^G(w)=G(F)\cap {_FI_P^G}(w)$ de $\UF$ commute ˆ la descente sŽparable: 

\begin{lemme}\label{induction et descente} ($F$ est infini et l'hypothse \ref{hypo bonnes F-strates faible} est vŽrifiŽe pour $G$\footnote{Relativement ˆ $F$; il n'est pas nŽcessaire de supposer qu'elle le soit relativement ˆ $E$.}.)
Soit $w\in \UF^M$.
\begin{enumerate}
\item[(i)] $I_{F,P}^G(w)_E = I_{E,P}^G(w)$.
\item[(ii)] $I_{F,P}^G(w)=G(F) \cap I_{E,P}^G(w)$.
\end{enumerate}
\end{lemme}

\begin{proof}
Pour $u\in \UF \cap \bsfrY(F,w\,\vert\,P)$, on a ${_FI_P^G}(w)= {_F\bsfrY_u}$; et puisque $u$ appartient \textit{a fortiori} ˆ 
$\UU_E \cap \bsfrY(E,w\,\vert\,P)$, on a aussi ${_EI_P^G}(w)= {_E\bsfrY_u}$. Cela prouve (i). Le point (ii) est une consŽquence de (i), d'aprs \ref{descente sŽparable strates}.
\end{proof}

\vskip2mm
\textit{\S5.\hskip1mm Comparaison avec l'induction parabolique des $\overline{F}$-strates. ---} Si $F= \overline{F}$, alors $F$ est infini et l'hypothse \ref{hypo bonnes F-strates faible} pour $G$ est vide. Dans ce cas, 
pour une ($\overline{F}$-)strate $\bsfrY^M$ de $\UU^M$, l'induite parabolique 
$I_P^G(\mathfrak{Y}^M)$ est l'unique strate de $\UU$ qui intersecte 
$\bsfrY^MU_P$ de manire dense (\textit{cf.} \ref{FiP=iP}\,(ii)). Observons que l'intersection $I_P^G(\bsfrY^M)\cap \bsfrY^MU_P $ est ouverte dense dans $\bsfrX(\bsfrY^M)U_P$.

\begin{remark}
\textup{On suppose que $F=\overline{F}$. Soit $\bsfrY^M$ une ($\overline{F}$-)strate de $\UU^M$. Soit $\scrY^M$ une lame de $\bsfrY^M$ et soit $P'^M$ le sous-groupe parabolique de $M$ associŽ ˆ $\scrY^M$. Posons 
$P'= P'^M \ltimes U_P$; c'est un sous-groupe parabolique de $G$ contenu dans $P$. Rappelons que 
$i^P(\scrY^M)= \scrX(\scrY^M)U_P$ est un sous-ensemble saturŽ de $\UU$ et que $P'$ est le sous-groupe parabolique de 
$G$ associŽ ˆ $i^P(\scrY^M)$ (\ref{induction des F-lames}\,(i)). 
Supposons que l'ensemble $\bs{\Lambda}_{\scrY^MU_P}=\{u\in \scrY^MU_P\,\vert\, \bs{\Lambda}_u = \bs{\Lambda}_{\scrY^MU_P}\}$ soit non vide. 
Alors (\ref{induction des F-lames}\,(iii)) $i^P(\scrY^M)$ est le saturŽ d'une lame $\scrY$ de $\UU$ 
et $I_P^G(\bsfrY^M)$ est la strate $\bsfrY= G\bullet \scrY$ de $\UU$. Comme $P_u=P'$ pour tout $u\in \scrY$, on a (d'aprs \cite[prop.~4.5\,(c)]{H2})
$$\dim(\bsfrY)= \dim(G)+ \dim(\scrY) - \dim(P')\ptf$$ Comme $\dim(\scrY)= \dim(\scrY^M) + \dim(U_P)$ et $\dim(P')= \dim(P'^M) + \dim(U_P)$, on en dŽduit que 
$$\dim(\bsfrY) = \dim(G) + \dim(\scrY^M) - \dim(P'^M)\ptf$$ Or (ˆ nouveau d'aprs \cite[prop.~4.5\,(c)]{H2}), on a 
$$\dim(\scrY^M)- \dim(P'^M) = \dim(\bsfrY^M) - \dim(M)\ptf$$ 
Comme $\dim(G)-\dim(M)= 2 \dim(U_P)$, on obtient l'ŽgalitŽ 
$$\dim (I_P^G(\bsfrY^M))= \dim (\bsfrY^M) + 2 \dim (U_P)\ptf$$
On verra plus loin (\ref{compat strates-orbites ouvertes}\,(i)) que cette ŽgalitŽ reste vraie mme si $\bs{\Lambda}_{\scrY^MU_P}=\emptyset$.
}
\end{remark}

Revenons ˆ nos hypothses prŽcŽdentes: le corps $F$ est infini et l'hypothse 
\ref{hypo bonnes F-strates faible} est vŽrifiŽe pour $G$. Pour $w\in {_F\UU^M}$, on a la $F$-strate ${_F\bsfrY^M_w}$ de ${_F\UU^M}$ 
et on a aussi la ($\overline{F}$-)strate $\bsfrY^M(F,w)$ de 
$\UU^M$ (\ref{comparaison avec les strates gŽo}, \S2). Par dŽfinition,  $\bsfrY^M(F,w)$ est l'unique strate de $\UU^M$ qui intersecte ${_F\scrX^{\!M}_w}$ de manire dense; rappelons qu'elle ne dŽpend que de ${_F\bsfrY^M_w}$. Si de plus 
$\check{X}_F(M)_{\mathbb{Q}}\cap \bs{\Lambda}_w \neq \emptyset$ (\textit{e.g.} si $w\in \UU^M_F$ et si l'hypothse \ref{hypo bonnes F-strates faible} est vŽrifiŽe pour $M$), 
alors $\bsfrY^M(F,w)= \bsfrY^M_w$ (\ref{strate gŽo associŽe ˆ F-strate}\,(i)). Dans tous les cas, on dispose de deux strates de $\UU$: 
$$I_P^G(\bsfrY^M(F,w))\quad \hbox{et} \quad \bsfrY(F,{_FI_P^G}(w))\ptf$$ 

Supposons de plus que l'hypothse \ref{hypo bonnes F-strates faible} soit vŽrifiŽe pour $M$. Dans ce cas on a le

\begin{lemme}\label{IP via descente ˆ partir de IPgŽo}
($F$ est infini et l'hypothse \ref{hypo bonnes F-strates faible} est vŽrifiŽe pour $G$ et pour $M$.) Soit $w\in \UU^M_F$. 
\begin{enumerate}
\item[(i)] $I_P^G(\bsfrY^M(F,w))= \bsfrY(F,{_FI_P^G}(w))\;(= \bsfrY(F,w\,\vert \, P))$
\item[(ii)] $I_{F,P}^G(w) = \UF \cap I_P^G(w)$. 
\end{enumerate}
\end{lemme}

\begin{proof}Posons $\bsfrY^M = \bsfrY^M(F,w)$ et $\bsfrY= I_P^G(\bsfrY^M)$. Rappelons que $\bsfrY$ est l'unique strate de $\UU$ qui intersecte $\bsfrY^MU_P$ de manire dense. 
Comme (d'aprs (\ref{strate gŽo associŽe ˆ F-strate}\,(i))) $\bsfrY^M= \bsfrY^M_w = M\bullet \scrY^{\!M}_w$, on en dŽduit que $\bsfrY$ est aussi l'unique strate de $\UU$ qui intersecte 
$\scrY^{\!M}_wU_P$ de manire dense. Or ${_F\scrX^{\!M}_w}= \scrX^{\!M}_w$ et ${_F\scrY^{\!M}_w}= \scrY^{\!M}_w$ (car par hypothse $\check{X}_F(M)_{\mbb{Q}} \cap \bs{\Lambda}_w^M \neq \emptyset$). 
Comme $\scrY_w^{\!M}$ est (ouvert) dense dans $\scrX^{\!M}_w$, on en dŽduit que 
$\bsfrY$ est l'unique strate de $\UU$ qui intersecte ${_F\scrX^{\!M}_w}U_P$ de manire dense; \textit{i.e.} $\bsfrY = \bsfrY(F,w\,\vert \, P)$. 

Quant au point (ii), l'hypothse \ref{hypo bonnes F-strates faible} pour $M$ entra"ne l'ŽgalitŽ $\bsfrY^M_{F,w}= \UF^M \cap \bsfrY^M_w$ 
(d'aprs \ref{descente 3 sous HBFS}\,(iii)). D'autre part, par dŽfinition la $F$-strate ${_FI_P^G}(w)$ de $\FU$ possde un point $F$-rationnel, disons $u\;(\in I_{F,P}^G(w))$; 
et l'hypothse \ref{hypo bonnes F-strates faible} pour $G$ entra"ne l'ŽgalitŽ $\bsfrY_{F,u} = \UF \cap \bsfrY_u$ (d'aprs \ref{descente 3 sous HBFS}\,(iii)). 
Comme $\bsfrY^M(F,w)= \bsfrY^M_w$ et $\bsfrY(F,u) = \bsfrY_u$ (d'aprs (\ref{strate gŽo associŽe ˆ F-strate}\,(i))), le point (i) nous dit que 
$I_P^G(w)= \bsfrY_u$. Cela entra"ne (ii). 
\end{proof}

\begin{remark}\label{dŽfinition simplifiŽe}
\textup{En supposant $F$ infini et l'hypothse \ref{hypo bonnes F-strates faible} vŽrifiŽe pour $G$ et pour $M$,  
l'induction parabolique des $F$-strates de $\UF^M$ s'obtient par descente ˆ partir de l'induction parabolique des $\overline{F}$-strates de $\UU^M$ 
(\ref{IP via descente ˆ partir de IPgŽo}\,(ii) ): pour $w\in \UF^M$, on a
$I_{F,P}^G(w) = \UF \cap I_P^G(w)$. Si $F$ est infini et si l'hypothse \ref{hypo bonnes F-strates faible} est vŽrifiŽe pour tous les $F$-facteurs de Levi de $G$,  
\textit{e.g.} si $F$ est un corps local ou un corps global, on peut donc tout aussi bien utiliser cette ŽgalitŽ comme dŽfinition\footnote{Observons que cette dŽfinition ne s'applique qu'aux $F$-strates 
de $\UF^M$, \cad aux $F$-strates de $\FU^M$ qui possdent un point $F$-rationnel (qui sont celles qui nous intŽressent vraiment).}.}
\end{remark}

\textit{\S5. Le cas des corps finis. ---}\label{le cas o F est fini} 
On suppose dans ce paragraphe que le corps $F$ est fini. Dans ce cas la thŽorie des $F$-lames et des $F$-strates de $\UF$ se ramne 
par descente galoisienne ˆ celle sur $\overline{F}=F^{\mathrm{s\acute{e}p}}$. En particulier l'hypothse 
 \ref{hypo bonnes F-strates faible} est vŽrifiŽe pour tous les $F$-facteurs de Levi de $G$. 
 
 \begin{lemme}\label{ind par pour F fini}
 ($F$ est fini.) Pour $w\in \UF^M$, on a $\UF \cap I_P^G(w)\neq \emptyset$. 
 \end{lemme}

\begin{proof}
Soit $w\in \UF^M$. Puisque $M$ et $P$ sont dŽfinis sur $F$, la ($\overline{F}$-)strate $I_P^G(w)\;(\subset \UU)$ est $\Gamma_F$-stable donc 
dŽfinie sur $F$. La $G$-orbite $\ES{O}= \ES{O}(I_P^G(w))$ est elle aussi $\Gamma_F$-stable donc dŽfinie sur $F$. Par consŽquent l'ensemble 
$\ES{O}(F)=G(F)\cap \ES{O}$ est non vide (thŽorme de Lang-Steinberg). Puisque $\UF = G(F) \cap \UU$ (\ref{descente sŽparable c™ne F-nilpotent}), 
cela prouve le lemme. 
\end{proof}

Pour $w\in \UF^M$, on note $I_{F,P}^G(w)$ la $F$-strate de $\UF$ dŽfinie par $$I_{F,P}^G(w)= \UF \cap I_P^G(w)\ptf$$ 
On a 
$$I_{F,P}^G(w')= I_{F,P}^G(w)\quad \hbox{pour tout} \quad w'\in \bsfrY_{F,w}^M\ptf$$ 
On a les analogues suivants des propriŽtŽs de transitivitŽ \ref{transitivitŽ de l'induction} et de descente sŽparable \ref{induction et descente}. 
Si $P'$ est un $F$-sous-groupe parabolique de $G$ contenant $P$ et si $M'$ est l'unique composante de Levi de $G$ contenant $M$, 
pour $w\in \UF^M$, on a $$I_{F,P}^G(w) = I_{F,P'}^G \circ I_{F,P\cap M'}^{M'}(w)\ptf$$ 
Si maintenant $E/F$ est une sous-extension finie de $\overline{F}/F$, pour $w\in \UF^M$, on a $$I_{F,P}^G(w) = \UF \cap I_{E,P}^G(w)\ptf$$ 

\subsection{Comparaison avec l'application de Lusztig-Spaltenstein}\label{lien avec l'IP de LS} 
Dans cette sous-section, on relie l'opŽration d'induction  parabolique dŽfinie en \ref{induction parabolique des F-strates} (pour les $F$-strates unipotentes) 
ˆ celle de Lusztig-Spaltenstein \cite{LS}. 

\vskip2mm
\textit{\S1.\hskip1mm L'application $w \mapsto \ES{J}_P^{G}(w)$ de Lusztig-Spaltenstein \cite{LS}. ---}
Dans ce paragraphe, on suppose que $F=\overline{F}$. 
Soit $P$ un sous-groupe parabolique de $G$ et soit $M$ une composante de Levi de $M$. Pour un ŽlŽment $w\in \UU^M= {_{\smash{\overline{F}}}\UU^M}$, notons\index{IPGLS(w)@$\ES{J}_P^{G}(w)$} 
$$\ES{J}_P^{G}(w)= \ES{J}_P^{G}(\ES{O}^M_w)\quad \hbox{avec}\quad \ES{O}_w^M=M\bullet w$$ l'induite parabolique de $\ES{O}_w^M$ au sens de Lusztig-Spaltenstein \cite{LS}: 
c'est l'unique $G$-orbite $\ES{O}$ de $\UU$ qui intersecte $\ES{O}^M_w U_P$ de manire dense; ou, ce qui revient au mme, telle que 
l'intersection $\ES{O} \cap \ES{O}^M_w U_P$ soit dense dans $\overline{\ES{O}^M_w} U_P$. 
L'intersection $\ES{J}_P^{G}(w) \cap \ESO^M_w U_P$ est alors ouverte dense dans $\overline{\ES{O}_w^M}U_P$; et 
d'aprs \cite[theorem~1.3\,(c)]{LS}, elle forme une unique $P$-orbite. 

\begin{remark}\label{propriŽtŽs de IP-LS}\textup{
\begin{enumerate}
\item[(i)] Pour $w\in \UU^M$, on a la caractŽrisation Žquivalente de $\ES{J}_P^{G}(w)$ suivante: c'est l'unique $G$-orbite $\ES{O}$ de $G$ telle que 
l'intersection $\ES{O} \cap wU_P$ soit dense dans la variŽtŽ irrŽductible $wU_P$. Cette intersection est alors ouverte dans $wU_P$.
\item[(ii)] Pour $w\in \UU^M$ et $w'\in \overline{\ES{O}_w^M}$, on a l'inclusion $$\ES{J}_P^{G}(w')\subset \overline{\ES{J}_P^{G}(w)}\ptf$$ En effet puisque l'intersection 
$\ES{J}_P^{G}(w) \cap \ES{O}_w^MU_P$ est dense dans $\ES{O}_w^MU_P$, on a les inclusions 
$$\overline{\ES{J}_P^{G}(w)}\supset \overline{\ES{O}_w^MU_P}\supset \overline{\ES{O}_w^M}U_P \supset \ES{O}_{w'}^MU_P\ptf$$
\item[(iii)] Pour $w\in \UU^M$ et $u\in \ES{J}_P^{G}(w)$, on a $\dim (G^u) = \dim (M^w)$ \cite[theorem 1.3~(a)]{LS}. Puisque $\dim(G)= \dim(M) + 2\dim(U_P)$, 
on en dŽduit que 
$$\dim(\ES{J}_P^{G}(w))= \dim(\ES{O}_w^M) + 2 \dim(U_P)\ptf$$
\item[(iv)] L'application $w\;(\in \UU^M) \mapsto \ES{J}_P^{G}(w)$ ne dŽpend pas du sous-groupe parabolique $P$ de $G$ de composante 
de Levi $M$ \cite[theorem~2.2]{LS}: si $P'$ est un autre sous-groupe parabolique de $G$ de composante de Levi $M$, alors pour tout $w\in \UU^M$, 
on a $\ES{J}_{P'}^{G}(w)= \ES{J}_{P'}^{G}(w)$; on peut donc poser $\ES{J}_M^{G}(w)= \ES{J}_P^{G}(w)$.
\end{enumerate}}
\end{remark}

Rappelons que si $\bsfrY^M$ est une strate de $\UU^M$, on a notŽ $\ES{O}(\bsfrY^M)$ l'unique $M$-orbite de $\UU^M$ qui soit (contenue et) dense dans $\bsfrY^M$. 
Puisque la $G$-orbite $\ES{J}_P^{G}(\ES{O}(\bsfrY^M)$ de $\UU$ intersecte $\ES{O}(\bsfrY^M)U_P$ de manire dense et que $\overline{\ES{O}(\bsfrY^M)}= \overline{\bsfrY^M}=\bsfrX(\bsfrY^M)$, on a toujours l'inclusion $$\ES{J}_P^{G}(\ES{O}(\bsfrY^M)) \subset I_P^G(\bsfrY^M)\ptf$$ 

\begin{lemme}\label{l'induite de la strate gŽo ne dŽpend pas de P}
Pour une strate $\bsfrY^M$ de $\UU^M$, la strate induite $I_P^G(\bsfrY^M)$ ne dŽpend pas de $P$: si $P'$ est un autre sous-groupe parabolique de $G$ de composante de Levi $M$, on 
a $I_{P'}^G(\bsfrY^M)= I_P^G(\bsfrY^M)$. On peut donc poser $I_M^G(\bsfrY^M)= I_P^G(\bsfrY^M)$.
\end{lemme}

\begin{proof}
D'aprs \ref{propriŽtŽs de IP-LS}\,(iv)), la $G$-orbite $\ES{J}_{P'}^{G}(\ES{O}(\bsfrY^M)) = \ES{J}_P^{G}(\ES{O}(\bsfrY^M))$ est contenue dans 
l'intersection $I_{P'}^G(\bsfrY^M)\cap I_P^G(\bsfrY^M)$. D'o le rŽsultat. 
\end{proof}

\begin{remark}
\textup{Si $p$ est bon pour $G$, les $\overline{F}$-strates de $\UU$ co\"{\i}ncident avec les orbites gŽomŽtriques unipotentes 
(\ref{p bon, strate-orbite}); si de plus $p$ est bon pour $M$, alors il en est de mme pour les $\overline{F}$-strates de $\UU^M$ et dans ce cas, l'induction parabolique 
des $\overline{F}$-strates de $\UU^M$ donnŽe par l'application $I_M^G= {_{\smash{\overline{F}}}I_M^G}$ co\"{\i}ncide avec l'induction parabolique des orbites gŽomŽtriques unipotentes de Lusztig-Spaltenstein. }
\end{remark}

\textit{\S2.\hskip1mm CompatibilitŽ entre les applications $I_M^G$ et $\ES{J}_M^{G}$. ---} Pour une $\bsfrY^M$-strate de $\UU^M$, deux questions naturelles se posent alors: 
\begin{enumerate}
\item[(Q1)] a-t-on $\ES{J}_M^{G}(\ES{O}(\bsfrY^M)) = \ES{O}(I_M^G(\bsfrY^M))$? 
\item[(Q2)] pour n'importe quelle $M$-orbite $\ES{O}^M\subset \bsfrY^M$, a-t-on $\ES{J}_M^{G}(\ES{O}^M) \subset I_M^G(\bsfrY^M)$? En d'autres termes, l'induction parabolique des strates de $\UU^M$ est-elle compatible ˆ l'induction parabolique des $M$-orbites unipotentes de Lusztig-Spaltenstein? 
\end{enumerate}

Soit $G'$ \guill{le} $\mbb{C}$-groupe algŽbrique rŽductif connexe correspondant ˆ $G$, \textit{i.e.} $G= \scr{G}_\mbb{C}$ o 
$\scr{G}$ est un $\mbb{Z}$-groupe rŽductif de Chevalley-Demazure tel que $G\simeq \scr{G}_{\,\overline{F}}$. 
De mme soit $M'$ le $\mbb{C}$-groupe algŽbrique rŽductif connexe correspondant ˆ $M$. 
On a vu 
(\textit{cf.} \ref{comparaison stratif Hesselink sur F et sur C}, \ref{du groupe ˆ l'algbre de Lie} et \ref{strates et orbites}) qu'on a une bijection canonique $\bsfrY \mapsto \ES{O}'(\bsfrY)$ entre les $G$-strates de $\UU$ et les $G'$-orbites unipotentes de $\UU'=\UU^{G'}$; de mme on a une bijection canonique $\bsfrY^M \mapsto \ES{O}'(\bsfrY^M)$ entre les $M$-strates de $\UU^M$ et les $M'$-orbites unipotentes de $M'$. Notons $\UU/G$ l'ensemble des $G$-orbites de $\UU$. On dŽfinit de la mme manire $\UU'/G'$, $\UU^M/M$ et $\UU^{M'}\!/M'$. 
Les bijections prŽcŽdentes induisent des applications surjectives 
$$\pi_G: \UU/G \rightarrow \UU^{G'}\!/G'\quad \hbox{et} \quad \pi_M: \UU^M/M \rightarrow \UU^{M'}\!/M'\,;$$ 
les fibres de $\pi_G$ sont les strates de $\UU$ et les fibres de $\pi_M$ sont les strates de $\UU^M$. On note 
$$\iota_G: \UU^{G'}\!/G' \rightarrow \UU/G \quad \hbox{et} \quad \iota_M: \UU^{M'}\!/M' \rightarrow \UU^M/M$$ 
les sections de $\pi_G$ et $\pi_M$ dŽfinies respectivement par $\iota_G(\ES{O}')= \ES{O}(\bsfrY)$ si $\pi_G(\bsfrY)= \ES{O}'$ et 
$\iota_M (\ES{O}^{M'})= \ES{O}(\bsfrY^{M'})$ si $\pi_M(\bsfrY^{M'})= \ES{O}^{M'}$. On a donc deux diagrammes: 
$$
 \xymatrix{
\UU'/G' \ar[r]^{\iota_G} & \UU/G \\
\UU^{M'}/M' \ar[u]^{\ES{J}_{M'}^{G'}}\ar[r]_{\iota_M} & \UU^M/M \ar[u]_{\ES{J}_{M}^{G}}
}
 \quad \hbox{et} \quad
\xymatrix{
\UU/G \ar[r]^{\pi_G} & \UU'/G' \\
\UU^M/M \ar[u]^{\ES{J}_M^{G}}\ar[r]_{\pi_M} & \UU^{M'}\!/M' \ar[u]_{\ES{J}_{M'}^{G'}}
\ptf}
$$
La premire question (Q1) se ramne ˆ la commutativitŽ du diagramme de gauche; la seconde (Q2) ˆ la commutativitŽ du diagramme de droite. 

Commenons par la question Q1, pour laquelle la rŽponse est positive: 

\begin{proposition}\label{compatibilitŽ I et J pour l'orbite ouverte}
Pour toute strate $\bsfrY^M$ de $\UU^M$, on a l'ŽgalitŽ 
$$\ES{J}_M^{G}(\ES{O}(\bsfrY^M))= \ES{O}(I_M^G(\bsfrY^M))\ptf$$
\end{proposition} 

\begin{proof}
Pour une $G$-orbite $\ES{O}\in \UU/G$, notons $\gamma_G(\ES{O})$ la reprŽsentation de Springer\footnote{On considre ici la reprŽsentation dŽfinie par 
Lusztig (sans restriction sur la caractŽristique) dans \cite{L2} plut™t que celle dŽfinie originellement par Springer; l'une se dŽduit de l'autre par tensorisation par la reprŽsentation \guill{signe} de $W$.}  de $(\ES{O},1)$ o $1$ dŽnote la reprŽsentation triviale du groupe des composantes 
connexes $A_G(u)= Z_G(u)/Z_G(u)^\circ$ du centralisateur schŽmatique $G^u$ d'un ŽlŽment $u\in \ES{O}$ (\textit{cf.} \ref{notations et rappels}, \S3); c'est un ŽlŽment de l'ensemble 
$\mathrm{Irr}(W^G)$ des classes d'isomorphisme de reprŽsentations irrŽductibles du groupe de Weyl $W^G$ de $G$. Le point important est que l'application $\gamma_G: \UU/G \rightarrow \mathrm{Irr}(W^G)$ est injective. 
Pour une $G'$-orbite $\ES{O}'\in \UU'/G'$, on a (en identifiant canoniquement $W^G$ ˆ $W^{G'}$) $$\gamma_G(\iota_G(\ES{O}'))= \gamma_{G'}(\ES{O}')\ptf$$ 
D'autre part, la compatibilitŽ entre la correspondance de Springer et l'induction parabolique de Lusztig-Spaltenstein \cite[theorem~3.5]{LS}\footnote{Le rŽsultat de \textit{loc.~cit.} est valable sous certaines  hypothses restrictives qui ont ŽtŽ supprimŽes ultŽrieurement; voir aussi \cite[S7]{L3}.} 
donne, pour une $M$-orbite $\ES{O}^M\in \UU^M/M$: 
$$\gamma_G\left(\ES{J}_M^{G}(\ES{O}^M)\right) = j_{W^{M}}^{W^{G}}\left(\gamma_M(\ES{O}^M)\right)$$ o $j_{W^M}^{W^G}: \mathrm{Irr}(W^M)\rightarrow \mathrm{Irr}(W^G)$ dŽsigne l'opŽration de j-induction (induction tronquŽe). 
Pour une $M'$-orbite $\ES{O}^{M'} \in \UU^{M'}/M'$, on obtient que 
\begin{align*}
 \gamma_G\left(\ES{J}_M^{G}(\iota_M(\ES{O}^{M'}))\right) & = j_{W^M}^{W^{G}}\left(\gamma_M(\iota_M(\ES{O}^{M'}))\right)\\
 &= j_{W^{M'}}^{W^{G'}}\left(\gamma_{M'}(\ES{O}^{M'})\right)\\
 &= \gamma_{G'}\left(\ES{J}_{M'}^{G'}(\ES{O}^{M'})\right) = \gamma_G\left(\iota_G(\ES{J}_{M'}^{G'}(\ES{O}^{M'})\right)\ptf
 \end{align*}
Par injectivitŽ de $\gamma_G$, on en dŽduit l'ŽgalitŽ cherchŽe.
\end{proof}

\begin{corollary}\label{compat strates-orbites ouvertes}
Soit $\bsfrY^M$ une strate de $\UU^M$. 
\begin{enumerate}
\item[(i)] $\dim(I_M^G(\bsfrY^M)) = \dim(\bsfrY^M) + \dim(G)-\dim(M)$. 
\item[(ii)] Pour toute strate $\bsfrY_1^M$ de $\UU^M$ telle que $\bsfrY_1^M \subset \bsfrX(\bsfrY^M)\;(= \overline{\bsfrY^M})$, on a l'inclusion 
$I_M^G(\bsfrY_1^M)\subset \bsfrX(I_M^G(\bsfrY^M))$.
\end{enumerate}
\end{corollary}

\begin{proof}
Posons $\bsfrY= I_M^G(\bsfrY^M)$, $\ES{O}^M = \ES{O}(\bsfrY^M)$ et $\ES{O}= \ES{O}(\bsfrY)$. 

Comme $\dim(\bsfrY) = \dim(\ES{O})$ et $\dim(\bsfrY^M)=\dim(\ES{O}^M)$, le point (i) rŽsulte de \ref{compatibilitŽ I et J pour l'orbite ouverte} et 
\ref{propriŽtŽs de IP-LS}\,(iii). 

Prouvons (ii). Si $\bsfrY_1^M$ est une strate de $\UU^M$ contenue dans $\bsfrX(\bsfrY^M)$ alors (\textit{a fortiori}) $\ES{O}(\bsfrY^M_1)  \subset \overline{\ES{O}^M}$ et (d'aprs 
\ref{propriŽtŽs de IP-LS}\,(ii)) $\ES{J}_M^{G}(\ES{O}(\bsfrY^M_1))\subset \overline{\ES{J}_M^{G}(\ES{O}^M)}$. On a donc 
(d'aprs \ref{compatibilitŽ I et J pour l'orbite ouverte}) $\ES{O}(I_M^G(\bsfrY^M_1))\subset \overline{\ES{O}}= \bsfrX(\bsfrY)$, ce qui n'est possible que si la strate 
$I_M^G(\bsfrY^M_1)$ est contenue dans $\bsfrX(\bsfrY)$. 
\end{proof}

Quant ˆ la question Q2, pour toute $M$-orbite $\ES{O}^M\subset \bsfrY^M$, et plus gŽnŽralement $\ES{O}^M \subset \overline{\bsfrY^M}= \bsfrX(\bsfrY^M)$, on a 
(d'aprs \ref{compat strates-orbites ouvertes}\,(ii)) 
$$\ES{J}_M^{G}(\ES{O}^M) \subset \overline{I_M^G(\bsfrY^M)}= \bsfrX(I_M^G(\bsfrY^M))\ptf$$ 
Mais l'inclusion ci-dessus n'implique pas que la $G$-orbite $\ES{J}_M^{G}(\ES{O}^M)$ soit contenue dans la strate $I_G^M(\bsfrY^M)$. 
La rŽponse ˆ la question Q2 est nŽgative en gŽnŽral, comme le montre l'exemple suivant: 

\begin{exemple}\label{exemple pour la non compatibilitŽ des inductions}
\textup{Soit $G$ un groupe quasi-simple de type $\mathbf{F}_4$ sur un corps $F=\overline{F}$ de caractŽristique $p=2$. 
Prenons pour $M$ un facteur de Levi de type $\mathbf{C}_2$ et pour $\bsfrY^M$ la strate sous-rŽgulire de $\UU^M$: elle contient deux $M$-orbites, l'orbite ouverte 
$\ES{O}^M_1=\ES{O}(\bsfrY^M)$ et l'autre $\ES{O}^M_2 = \bsfrY^M \smallsetminus \ES{O}^M_1$; voir plus loin \ref{exemples en basse dimension}, \S3. Or avec les notations de \cite{Spa}, 
l'induite parabolique $\ES{J}_M^{G}(\ES{O}^M_1)$ est la $G$-orbite $F_4(a_1)$ et l'induite parabolique $\ES{J}_M^{G}(\ES{O}^M_2)$ est la $G$-orbite $F_4(a_2)$. Ci-dessous la partie du diagramme de 
Hasse pour $({\bf F}_4,p=2)$ concernŽe \cite[IV.2, p.~249]{Spa}: 
$$\xymatrix{ &F_4  \ar@{-}[d] &\\  &F_4(a_1)  \ar@{-}[d] &\\ &F_4(a_2)\ar@{-}[dl]  \ar@{-}[dr] \\ B_3 && C_3 }$$ Or pour $i=1,\,2$, 
la $G$-orbite $F_4(a_i)$ est une strate; et bien sžr on a $F_4(a_2) \prec F_4(a_1)$. Observons que l'on peut 
vŽrifier cela uniquement en calculant la dimension des orbites induites (\ref{propriŽtŽs de IP-LS}\,(iii)). En effet on a $\dim(G)= 52$. D'autre part on a $\dim(Z(M))=2$ et $\dim(\mathrm{Sp}_4)=10$, par consŽquent $\dim(M)=12$ et $\dim(G)-\dim(M) = 40$. Comme $\dim(\ES{O}^M_1)= 6$ et $\dim(\ES{O}^M_2)=4$ (\textit{cf.} \ref{exemples en basse dimension}, \S3), on obtient que 
$\dim(\ES{J}_M^{G}(\ES{O}^M_1))= 46$ et $\dim(\ES{J}_M^{G}(\ES{O}^M_1))=44$. Or $F_4(a_1)$ est l'unique $G$-orbite de $\UU$ de dimension 46 et 
$F_4(a_2)$ est l'unique $G$-orbite de $\UU$ de dimension 44.   } 
\end{exemple} 

\begin{remark}\label{Q2 pour les groupes classiques}
\textup{Pour $G$ quasi-simple de type $\mathbf{A}_n$, $\mathbf{B}_n$, $\mathbf{C}_n$ ou $\mathbf{D}_n$, la rŽponse ˆ la question Q2 est positive. En effet si $p\neq 2$ ou si $G$ est de type $\mathbf{A}_n$, c'est Žvident puisque les strates et les $G$-orbites 
de $\UU$ co\"{\i}ncident. Supposons $p=2$ et $G$ de type $\mathbf{B}_n$, $\mathbf{C}_n$ ou $\mathbf{D}_n$. Notons $G'$ le $\mbb{C}$-groupe rŽductif connexe correspondant ˆ $G$. 
Dans \cite{Spa}, les $G$-orbites de $\UU$ sont paramŽtrisŽes et l'application $\pi_G: \UU/G \rightarrow \UU^{G'}\!/G'$ est dŽcrite explicitement en termes de cette paramŽtrisation. 
Dans \textit{loc.~cit.}, il est aussi prouvŽ que cette application est compatible ˆ l'ordre partiel sur les orbites (donnŽ par $\ES{O}_1 \preccurlyeq \ES{O}_1 \Leftrightarrow \ES{O}_1 \subset \overline{\ES{O}_2}$) ainsi qu'ˆ l'induction parabolique de Lusztig-Spaltenstein: pour toute $M$-orbite 
$\ES{O}^M \in \UU^M/M$, on a l'ŽgalitŽ $\pi_G(\ES{J}_M^G(\ES{O}^M)) = \ES{J}_{M'}^{G'} (\pi_M(\ES{O}^M))$. 
}\end{remark}

\textit{\S3.\hskip1mm Sur la rigiditŽ des orbites, resp. des strates. ---} 
Une $G$-orbite $\ES{O}$ de $\UU$ est dite \textit{rigide} si elle est non-induite, \textit{i.e.} si pour tout facteur de Levi propre $M$ de $G$ et 
toute $M$-orbite $\ES{O}^M$ de $\UU^M$, on a $\ES{J}_M^{G}(\ES{O}^M) \neq \ES{O}$. De la mme manire, une strate $\bsfrY$ de 
$\UU$ est dite \textit{rigide} si elle est non-induite, \textit{i.e.} si pour tout facteur de Levi propre $M$ de $G$ et 
toute strate $\bsfrY^M$ de $\bsfrY^M$, on a $I_M^G(\bsfrY^M)\neq \bsfrY$. D'aprs \ref{compatibilitŽ I et J pour l'orbite ouverte}, une 
strate $\bsfrY$ de $\UU$ est rigide si et seulement si les conditions Žquivalentes suivantes sont vŽrifiŽes: 
\begin{itemize}
\item pour tout facteur de Levi propre $M$ de $G$ et toute strate $\bsfrY^M$ de $\UU^{M}$, on a $\ES{J}_M^{G}(\ES{O}(\bsfrY^M))\neq \ES{O}(\bsfrY)$; 
\item la $G'$-orbite $\ES{O}'(\bsfrY)$ de $\UU'$ est rigide. 
\end{itemize}
En particulier si la $G$-orbite $\ES{O}(\bsfrY)$ est rigide, alors la strate $\bsfrY$ l'est aussi. 

Observons que si la strate $\bsfrY$ est rigide, resp. induite, il peut exister une $G$-orbite $\ES{O}\subset \bsfrY$ qui soit 
induite, resp. rigide: 

\begin{exemple}
\textup{Reprenons le groupe $G$ de l'exemple \ref{exemple pour la non compatibilitŽ des inductions} (type $\mathbf{F}_4$, $p=2$). 
Il y a exactement quatre strates $\bsfrY$ de $\UU$ qui se dŽdoublent, \cad qui sont formŽes de deux $G$-orbites $\ES{O}=\ES{O}(\bsfrY)$ et $\ES{O}_2 = \bsfrY \smallsetminus \ES{O}$ (\textit{cf.} \cite[IV.2, page 250]{Spa}). 
Les paires $(\ES{O}_2,\ES{O})$ de $G$-orbites en question sont $(C_3(a_1)_2,C_3(a_1))$, $((\wt{A}_2+A_1)_2, \wt{A}_2+ A_1)$, $((B_2)_2,B_2)$ et $(\wt{A}_1)_2, \wt{A}_1)$. 
L'orbite $\ES{O}=C_3(a_1)$ est induite et l'autre $\ES{O}_2=(C_3(a_1))_2$ est rigide; pour les trois autres paires d'orbites, c'est l'inverse, 
l'orbite $\ES{O}$ est rigide et l'autre $\ES{O}_2$ est induite. 
}\end{exemple} 

\vskip2mm
\textit{\S2.\hskip1mm Version $F$-rationnelle de \ref{compatibilitŽ I et J pour l'orbite ouverte}. ---} 
Revenons ˆ $F$ quelconque. Soit $P$ un $F$-sous-groupe parabolique de $G$ et soit $M$ une composante de Levi de $P$ dŽfinie sur $F$. 
Pour $Z\subset \FU^M$ uniformŽment $F$-instable, on a dŽfini la $\overline{F}$-strate $\bsfrY(F,Z\,\vert\,P)$ de $\UU={_{\smash{\overline{F}}}\UU}$: c'est l'unique 
$\overline{F}$-strate de $\UU$ qui intersecte la variŽtŽ irrŽductible ${_Fi^P}(Z)$ de manire dense. On peut faire la mme construction avec les $G$-orbites de $\UU$, 
compte-tenu du fait qu'il n'y en a qu'un nombre fini \cite{L1} et qu'elles sont localement fermŽes dans $G$: il existe une unique $G$-orbite $\ES{O}_u$ de $\UU$ (avec $u\in \UU$ que l'on peut choisir dans ${_Fi^P(Z)}$) telle que l'intersection 
$\ES{O}_u \cap {_Fi^P}(Z)$ soit dense dans ${_Fi^P}(Z)$. On la note\index{OFZP@$\ES{O}(F,Z\,\vert\,P)$} $$\ES{O}(F,Z\,\vert\, P) = \ES{O}_{u}\ptf$$ 
L'intersection $\ES{O}(F,Z\,\vert\,P)\cap {_Fi^P}(Z)$ est ouverte dans 
${_Fi^P}(Z)$. Comme $\bsfrY(F,Z\,\vert\,P)$ est rŽunion (finie) de $G$-orbites, on a l'inclusion 
$$\ES{O}(F,Z\,\vert\, P)\subset \bsfrY(F,Z\,\vert\,P)\ptf$$ 
Observons que: 
\begin{itemize}
\item $\FU \cap \ES{O}(F,Z\,\vert\, P)$ est une partie $G(F)$-invariante non vide de $\FU \cap \bsfrY(F,Z\,\vert \, P)$; 
\item sans hypothse supplŽmentaire (\textit{cf.} \ref{F infini et points rationnels de la G-orbite O(F,Z,P)}), l'intersection 
$\UF\cap \ES{O}(F,Z\,\vert \, P)$, tout comme l'intersection $\UF \cap \bsfrY(F,Z\,\vert \, P)$, peut tre vide.  
\end{itemize}

\begin{remark}\label{F infini et points rationnels de la G-orbite O(F,Z,P)}
\textup{Si le corps $F$ est infini, 
puisque $i_{F}^P(Z)={_Fi^P}(Z)(F)$ est dense dans ${_Fi^P}(Z)$, l'intersection $\ES{O}(F,Z\,\vert \,P)\cap i_F^P(Z)$ est dense dans 
${_Fi^P}(Z)$; en particulier elle est non vide et $\ES{O}(F,Z\,\vert\, P)$ contient un point $F$-rationnel (que l'on peut choisir dans $\iota_F^P(Z)$). 
}
\end{remark}

On se limite maintenant au cas o $Z=\{w\}$. L'application qui ˆ $w\in \FU^M$ associe 
la $G$-orbite $\ES{O}(F,w\,\vert\,P)$ ne dŽpend que de la $F$-strate ${_F\bsfrY_w^M}$ de $\FU^M$: on a 
$$\ES{O}(F,w'\,\vert\,P)= \ES{O}(F,w\,\vert\,P)\quad \hbox{pour tout}\quad w'\in {_F\bsfrY_w^M}\ptf$$ 
Observons que pour $P=G$ et $u\in \FU$, on retrouve la construction de \ref{strates et orbites}: 
on a $$\ES{O}(F,u\,\vert\,G)= \ES{O}(F,u)\ptf$$ 
Pour $u\in \UF$, d'aprs \ref{orbite gŽo et densitŽ}\,(i) on a $$\ES{O}(F,u)= \ES{O}(\bsfrY(F,u))\ptf$$ 

DŽsormais et jusqu'ˆ la fin de \ref{lien avec l'IP de LS}, on suppose que l'hypothse \ref{hypo bonnes F-strates faible} 
est vŽrifiŽe pour $G$. La version $F$-rationnelle suivante de la proposition \ref{compatibilitŽ I et J pour l'orbite ouverte} est une gŽnŽralisation de l'ŽgalitŽ ci-dessus.

\begin{proposition}\label{inclusion IPLS dans IP} 
(L'hypothse \ref{hypo bonnes F-strates faible} est vŽrifiŽe pour $G$.) 
\begin{enumerate}
\item[(i)] Pour tout $w\in \FU^M$, on a $\ES{J}_M^{G}(\ES{O}^M(F,w)) = \ES{O}(F,w\,\vert\,P)$.
\item[(ii)] Supposons que l'hypothse \ref{hypo bonnes F-strates faible} soit vŽrifiŽe pour $M$. Pour tout 
$w\in \UF^M$, on a $\ES{J}_M^{G}(\ES{O}^M(F,w))  = \ES{O}(\bsfrY(F,w\,\vert\, P))$. 
\end{enumerate}
\end{proposition}

\begin{proof}
Prouvons (i). Soit $w\in \FU^M$ et soit $w_1\in \ES{O}^M(F,w)\cap {_F\scrX_w^{\!M}}$ (on rappelle que cette intersection est ouverte dense dans ${_F\scrX_w^{\!M}}$). 
Posons $\ES{O}=\ES{J}_P^{G}(w_1)$. Il s'agit de prouver que $\ES{O}= \ES{O}(F,w\,\vert\, P)$. Puisque  $\ES{O}_{w_1}^M \cap {_F\scrX_w^{\!M}}$ est ouvert dense dans ${_F\scrX_w^{\!M}}$, l'ensemble 
$$(\ES{O}_{w_1}^M\cap {_F\scrX_w^{\!M}})U_P= \ES{O}_{w_1}^MU_P \cap {_Fi^P}(w)$$ est ouvert dense dans ${_Fi^P}(w)={_F\scrX_w^{\!M}}U_P$. Comme d'autre part 
$\ES{O}\cap \ES{O}_{w_1}^MU_P$ est ouvert dense dans $\ES{O}_{w_1}^MU_P$, 
on obtient que $\ES{O}\cap \ES{O}_{w_1}^MU_P\cap {_Fi^P}(w)$ est ouvert dense dans ${_Fi^P}(w)$. 
Par suite $\ES{O}\cap {_Fi^P}(w)$ est dense dans ${_Fi^P}(w)$. Cela prouve (i): on a $\ES{O}= \ES{O}(F,w\,\vert \, P)$. 

Supposons de plus que l'hypothse \ref{hypo bonnes F-strates faible} soit vŽrifiŽe pour $M$ et prouvons (ii). Soit $w\in \UF^M$. 
D'aprs \ref{orbite gŽo et densitŽ}\,(i), on a $\ES{O}^M(F,w)= \ES{O}^M(\bsfrY^M(F,w))$. On peut donc appliquer \ref{compatibilitŽ I et J pour l'orbite ouverte}: on a 
$$\ES{J}_M^{G}(\ES{O}(F,w))= \ES{O}(I_{M}^G(\bsfrY^M(F,w))\ptf$$ 
Par dŽfinition, $I_{M}^G(\bsfrY^M(F,w))$ est l'unique ($\overline{F}$-)strate de $\UU$ qui intersecte l'ensemble 
$\iota^P(w)= \scrX_wU_P$ de manire dense. Or on a ${_F\scrX_w}= \scrX_w$ (d'aprs \ref{le cas o F-lame = lame gŽo}), par consŽquent 
${_F\iota^P}(w)= \iota^P(w)$ et $\bsfrY(F,w\,\vert \,P)= I_{M}^G(\bsfrY^M(F,w))$. Cela prouve (ii). 
\end{proof}

\begin{corollary}\label{indep P sous hypo}
(L'hypothse \ref{hypo bonnes F-strates faible} est vŽrifiŽe pour $G$.) Pour $w\in \FU^M$, la strate 
$\bsfrY(F,w\,\vert\,P)$ de $\UU$ ne dŽpend pas du $F$-sous-groupe parabolique $P$ de $G$ de composante de Levi $M$. 
Si de plus le corps $F$ est infini, on peut donc poser 
$${_FI_M^G}(w)= {_FI_P^G}(w)\quad \hbox{et} \quad I_{F,M}^G(w)= I_{F,P}^G(w)\ptf$$ 
\end{corollary}

\begin{proof}
Soit $w\in \FU^M$. La strate $\bsfrY(F,w\,\vert \, P)$ de $\UU$ contient la $G$-orbite $\ES{O}(F,w\,\vert \,P)$ et (d'aprs \ref{inclusion IPLS dans IP}\,(i)) cette dernire co\"{\i}ncide avec 
$\ES{J}_M^{G}(\ES{O}^M(F,w)) $. Par consŽquent si $P'$ est un autre $F$-sous-groupe parabolique de $G$ de composante de Levi $M$, l'intersection 
$\bsfrY(F,w\,\vert \, P)\cap \bsfrY(F,w\,\vert \, P')$ contient $\ES{O}(F,w\,\vert \, P) = \ES{O}(F,w\,\vert \, P')$. D'o le lemme. 
\end{proof}

\begin{remark}\label{comparaison IPLS et IP}
\textup{\begin{enumerate}
\item[(i)] Supposons que $F$ soit un corps fini. Pour $w\in \UF^M$, puisque la ($\overline{F}$-)strate $I_P^G(w)$ ne dŽpend pas de $P$, il en est de 
mme de la $F$-strate $I_{F,P}^G(w) = G(F)\cap I_P^G(w)$; on peut donc poser $I_{F,M}^G(w)= I_{F,P}^G(w)$. 
\item[(ii)] Supposons que $p$ soit bon pour $G$. D'aprs \ref{comparaison F-strate-orbite rat p bon}\,(i), l'application $\ES{O}\mapsto \UF \cap \ES{O}$ est bijection entres les $G$-orbites de $\UU$ qui intersectent non trivialement $\UF$ et les $F$-strates de $\UF$ 
(\ref{comparaison F-strate-orbite rat p bon}). Comme $p$ est \textit{a fortiori} bon pour $M$, si de plus le corps 
$F$ est infini, pour $w\in \UF^M$, la $G$-orbite $\ES{J}_M^{G}(w)$ intersecte non trivialement $\UF$ et on a l'ŽgalitŽ $\UF \cap \ES{J}_M^{G}(w)= I_{F,M}^G(w)$.
\item[(iii)] Pour $G=\mathrm{GL}_{n/F}$, quel que soit $F$, l'application $\ES{O} \mapsto G(F) \cap \ES{O}$ est une bijection entre 
les $G$-orbites de $\UU$ et les $F$-strates de $\UF$ (\ref{rem sur iso springer et G de type An}\,(ii)). Puisque 
$M$ est $F$-isomorphe ˆ un produit ${\rm GL}_{n_1}\times \cdots \times {\rm GL}_{n_r}$ avec $n_1+\cdots + n_r =n$, il en est de mme pour $M$: l'application 
$\ES{O}^M \mapsto M(F)\cap \ES{O}^M$ est une bijection entre les $M$-orbites de $\UU^M$ et les $F$-strates de $\UF^M$. 
D'autre part l'hypothse \ref{hypo bonnes F-strates faible} est 
toujours vŽrifiŽe pour $G$ et pour $M$ (\textit{cf.} \ref{hypo BFSF pour GLn}). 
Pour tout $w\in \UU^M$, on a $\ES{J}_M^{G}(w) = I_{M}^G(w)$; de plus 
l'intersection $G(F)\cap \ES{J}_M^{G}(w)$ est non vide et c'est une $F$-strate, \textit{i.e.} une $G(F)$-orbite, de $\UF$. 
\item[(iv)] Supposons que  l'hypothse \ref{hypo bonnes F-strates faible} soit vŽrifiŽe pour $G$ et pour $M$ (\textit{e.g.} si $F$ est un corps fini, un corps local ou un corps global). 
Dans ce cas l'induction parabolique des $F$-strates de $\UU^M_F$ peut se dŽfinir comme suit ˆ partir de l'induction parabolique de Lusztig-Spaltenstein des $M$-orbites de $\UU^M$. 
Toute $F$-strate $\bsfrY^M_F$ de $\UF^M$ est de la forme $\bsfrY^M_F= \UF^M \cap \bsfrY^M$ pour une unique strate $\bsfrY^M$ de $\UU^M$. 
On considre l'induite parabolique de Lusztig-Spaltenstein $\ES{O}= \ES{J}_M^{G}(\ES{O}(\bsfrY^M))$ de l'unique $M$-orbite $\ES{O}(\bsfrY^M)$ qui soit (contenue et) dense dans 
$\bsfrY^M$. L'intersection $\UU_F \cap \ES{O}$ est non vide et en notant $\bsfrY$ la strate de $\UU$ contenant $\ES{O}$, on a (d'aprs \ref{inclusion IPLS dans IP}\,(ii))
$$\ES{O}= \ES{O}(\bsfrY)\quad \hbox{et}\quad I_{F,M}^G(\bsfrY_F^M)= \UF \cap \bsfrY\ptf$$ 
\end{enumerate}
}
\end{remark}

On suppose jusqu'ˆ la fin de ce paragraphe (\S2) que l'hypothse \ref{hypo bonnes F-strates faible} est vŽrifiŽe pour $G$ et pour $M$. 
On peut alors remplacer l'application $\ES{J}_M^{G}$ dans 
\ref{comparaison IPLS et IP}\,(iv) 
par son analogue complexe $\ES{J}_{M'}^{G'}$ (\textit{cf.} \S1). Rappelons qu'on a notŽ 
$[\UU]_F$ l'ensemble des strates de $\UU$ qui intersectent non trivialement $\UF$ et $(\UU'/G')_F$ l'ensemble des $G'$-orbite de $\UU'$ de la forme 
$\ES{O}'(\bsfrY)$ pour une strate $\bsfrY\in [\UU]_F$ (\textit{cf.} \ref{strates et orbites}, \S2). L'application $$\ES{O}'(\cdot): [\UU]_F \rightarrow (\UU'/G')_F$$ est par dŽfinition bijective.

\begin{lemme}\label{F-compatibilitŽ I et J}
(L'hypothse \ref{hypo bonnes F-strates faible} est vŽrifiŽe pour $G$ et pour $M$.) L'application 
$I_M^G: [\UU^M] \rightarrow [\UU]$ envoie $[\UU^M]_F$ dans $[\UU]_F$ et rend commutatif le diagramme suivant: 
$$
\xymatrix{ [\UU]_F \ar[r]^<(.3){\ES{O}'(\cdot)} & (\UU'/G')_F  & \\
[\UU^M]_F \ar[r]_<(.2){\ES{O}'(\cdot)} \ar[u]^{I_{M}^G}& (\UU^{M'}\!/M')_F \ar[u]_{\ES{J}_{M'}^{G'}} 
}
$$
\end{lemme}

\begin{proof}
Pour $\bsfrY^M\in [\UU^M]_F$, on a $I_{F,M}^G(\UF^M \cap \bsfrY^M) = \UF \cap I_M^G(\bsfrY^M)$; en particulier l'opŽration d'induction parabolique 
(gŽomŽtrique) $I_M^G: [\UU^M] \rightarrow [\UU]$ envoie $[\UU^M]_F$ dans $[\UU]_F$. D'autre part $\ES{O}'(\bsfrY^M)$ est l'unique $M'$-orbite de $\UU^{M'}$ telle que 
$\iota_M (\ES{O}'(\bsfrY^M))= \ES{O}(\bsfrY^M)$; et on a (d'aprs \ref{compatibilitŽ I et J pour l'orbite ouverte}) 
$$\iota_G(\ES{J}_{M'}^{G'}(\ES{O}'(\bsfrY^M)) 
= \ES{J}_M^{G}(\ES{O}(\bsfrY^M))= \ES{O}(I_M^G(\bsfrY^M))\ptf$$ Cela entra"ne l'ŽgalitŽ 
$\ES{J}_{M'}^{G'}(\ES{O}'(\bsfrY^M))= \ES{O}'(I_M^G(\bsfrY^M))$.
\end{proof}

\textit{\S3.\hskip1mm Sur la $F$-rigiditŽ des orbites, resp. des strates. ---} On suppose dans ce paragraphe que l'hypothse \ref{hypo bonnes F-strates faible} est vŽrifiŽe pour tous les 
$F$-facteurs de Levi de $G$ (\textit{e.g.} si $F$ est un corps fini, un corps local ou un corps global). Dans ce cas, pour tout $F$-facteur de Levi $M$ de $G$, 
l'application $\bsfrY^M \mapsto \UF^M \cap \bsfrY^M$ est une bijection entre $[\UU^M]_F$ et les $F$-strates de $\UF^M$. De plus on a vu en 
\ref{induction parabolique des F-strates} 
que pour toute strate 
$\bsfrY^M \in [\UU^M]_F$, la strate $I_M^G(\bsfrY^M)$ de $\UU$ appartient ˆ $[\UU]_F$ et 
$$I_{F,M}^G(\UF^M \cap \bsfrY^M)= \UF \cap I_M^G(\bsfrY^M)\ptf$$ 
Reprenons les notations de \ref{(U/G)_F et (Y/G)_F}: notons $(\UU^M/M)_F\subset \UU^M/M$ le sous-ensemble formŽ des $M$-orbites qui intersectent non trivialement $\UF^M$; et pour une 
strate $\bsfrY^M\in [\UU^M]_F$, notons $(\bsfrY^M/M)_F\subset (\UU^M/M)_F$ le sous-ensemble formŽ des $M$-orbites qui intersectent non trivialement $\UF^M \cap \bsfrY^M$. 

\begin{definition}\label{orbites et strates F-rigides}
\textup{
(L'hypothse \ref{hypo bonnes F-strates faible} est vŽrifiŽe pour tous les $F$-facteurs de Levi de $G$.) 
\begin{enumerate}
\item[(i)] Une $G$-orbite $\ES{O}\in (\UU/G)_F$ est dite \textit{$F$-rigide} si pour tout $F$-facteur de Levi propre $M$ de $G$ et toute 
$M$-orbite $\ES{O}^M\in (\UU^M/M)_F$, on a $\ES{J}_M^G(\ES{O}^M) \neq \ES{O}$. 
\item[(ii)] Une strate $\bsfrY \in [\UU]_F$ est dite \textit{$F$-rigide} si 
pour tout $F$-facteur de Levi propre $M$ de $G$ et toute 
strate $\bsfrY^M \in [\UU^M]_F$, on a $I_M^G(\bsfrY^M)\neq \bsfrY$. 
\item[iii)] Une $G$-orbite $\ES{O}\in (\UU/G)_F$, resp. une strate $\bsfrY \in [\UU]_F$, est dite \textit{$F$-induite} si elle n'est pas $F$-rigide. 
\end{enumerate}
}
\end{definition}

D'aprs \ref{Y rat versus O(Y) rat}, pour toute strate $\bsfrY \in [\UU]$, on a 
$$\bsfrY \in [\UU]_F \Leftrightarrow \ES{O}(\bsfrY) \in (\UU/G)_F\ptf$$

\begin{lemme}
(L'hypothse \ref{hypo bonnes F-strates faible} est vŽrifiŽe pour tous les $F$-facteurs de Levi de $G$.) 
Une strate $\bsfrY \in [\UU]_F$ est $F$-rigide si et seulement si pour tout $F$-facteur de Levi propre $M$ de $G$ et toute   
strate $\bsfrY^M \in [\UU^M]_F$, les conditions Žquivalentes suivantes sont vŽrifiŽes: 
\begin{itemize}
\item $\ES{J}_M^{G}(\ES{O}(\bsfrY^M)) \neq \ES{O}(\bsfrY)$; 
\item $\ES{J}_{M'}^{G'}(\ES{O}'(\bsfrY^M)) \neq \ES{O}'(\bsfrY)$. 
\end{itemize}
\end{lemme}

\begin{proof}
C'est une consŽquence de \ref{compatibilitŽ I et J pour l'orbite ouverte} et \ref{F-compatibilitŽ I et J}.
\end{proof}

\begin{remark}
\textup{
(L'hypothse \ref{hypo bonnes F-strates faible} est vŽrifiŽe pour tous les $F$-facteurs de Levi de $G$.) Soit $\bsfrY\in [\UU]_F$. Comme pour $F= \overline{F}$, on a que si la $G$-orbite $\ES{O}(\bsfrY)$ est $F$-rigide, alors la strate $\bsfrY$ l'est aussi; mais la rŽciproque n'est pas vraie en gŽnŽral. D'autre part la $F$-rigiditŽ de la $G$-orbite $\ES{O}(\bsfrY)$ n'implique pas que les $G$-orbites $\ES{O}\in (\bsfrY/G)_F \smallsetminus \ES{O}(\bsfrY)$ soient $F$-rigides; et $\ES{O}(\bsfrY)$ peut tre $F$-induite sans que les $G$-orbites $\ES{O} \in (\bsfrY/G)_F\smallsetminus \ES{O}(\bsfrY)$ le soit. 
}
\end{remark}

\vskip2mm
\textit{\S3.\hskip1mm Sur l'automorphisme d'opposition. ---} 
L'une des principales propriŽtŽs de l'induction parabolique de Lusztig-Spaltenstein est que pour tout $w\in \UU^M$, la $G$-orbite $\ES{J}_P^{G}(w)$ ne dŽpend pas du sous-groupe parabolique $P$ de $G$ de composante de Levi 
$M$ \cite[theorem~2.2]{LS} (\textit{cf.} \ref{propriŽtŽs de IP-LS}\,(iv)). Lusztig et Spaltenstein donnent deux dŽmonstrations de \textit{loc.~cit.} 
La premire, basŽe sur la thŽorie des reprŽsentations d'un groupe rŽductif connexe sur un corps fini, 
n'est valable que si $p>1$ (\textit{cf.} \cite[2.7]{LS}). La seconde, valable en toute caractŽristique, utilise l'automorphisme d'opposition: si 
$\phi$ est un ($\overline{F}$-)automorphisme d'opposition de $G$, \cad un automorphisme de $G$ qui stabilise un tore maximal $T$ de $G$ et opre 
par $-1$ sur le systme de racines $\ES{R}_T$, alors pour toute $G$-orbite $\ES{O}$ de $\UU$, on a $\phi(\ES{O})= \ES{O}$ \cite[2.10]{LS}. Ce rŽsultat implique 
le thŽorme~2.2 de \cite{LS} (\textit{cf.} \cite[2.11]{LS}).

On suppose jusqu'ˆ la fin de ce paragraphe (\S3) que $G$ est $F$-dŽployŽ. 
Soit $\phi$ un $F$-automorphisme de $G$ tel que $\phi(A_0)=A_0$ et $\phi$ opre par $-1$ sur $\ES{R}$. 
Un tel $\phi$ existe et il est appelŽ $F$-automorphisme d'opposition (par rapport ˆ $A_0$) de $G$. Si $H$ est un $F$-sous-groupe fermŽ de $ G$ rŽductif connexe et contenant 
$A_0$, alors $H$ est $F$-dŽployŽ et la restriction de $\phi$ ˆ $H$ est un $F$-automorphisme d'opposition de $H$. 
On choisit $\phi$ de telle manire que $\phi^2$ soit l'identitŽ de $G$ et la restriction de $\phi$ ˆ $A_0$ soit le passage ˆ l'inverse. 
L'automorphisme de $W=W^G(A_0)$ induit par $\phi$ est l'identitŽ et si $w_0\in W$ est l'ŽlŽment de plus grande longueur, 
alors $w_0 \circ \phi$ est l'unique automorphisme involutif de $\ES{R}$ 
qui laisse invariants $\ES{R}^+$ et $\Delta$. 

Pour $u\in \FU$, l'ŽlŽment $\phi(u)$ appartient ˆ $\FU$ et (par transport de structure) on a $\phi({_F\scrY_u})= {_F\scrY_{\phi(u)}}$, $\phi({_F\bsfrY_u})= {_F\bsfrY_{\phi(u)}}$, (etc.). 

\begin{lemme}\label{strates et involution d'opp sous 2.6.9}
($G$ est $F$-dŽployŽ). On suppose que l'hypothse \ref{hypo bonnes F-strates faible} est vŽrifiŽe. Pour tout $u\in \UF$, on a 
$\phi(\bsfrY_{F,u})=\bsfrY_{F,u}$.
\end{lemme}

\begin{proof}
Soit $u\in \UF$. On a $\bsfrY_{F,u}= \UF \cap \bsfrY_u$. La ($F$-)strate $\bsfrY_u$ contient la $G$-orbite $\ES{O}= G\bullet u$; de mme   
la strate $\phi(\bsfrY_u)= \bsfrY_{\phi(u)}$ contient la $G$-orbite $G\bullet \phi(u) = \phi(\ES{O})$. Comme $\phi(\ES{O})= \ES{O}$ (d'aprs \cite[2.10]{LS}) et que cette $G$-orbite $\ES{O}$ 
est contenue dans 
$\phi(\bsfrY_u)\cap \bsfrY_u$, on a forcŽment $\phi(\bsfrY_u)= \bsfrY_u$. Puisque $\phi(\UF)= \UF$, on en dŽduit que $\phi(\bsfrY_{F,u})= \bsfrY_{F,u}$. 
\end{proof}

\begin{remark}
\textup{($G$ est $F$-dŽployŽ). 
Sans supposer \ref{hypo bonnes F-strates faible}, on a le rŽsultat partiel suivant. Pour $u\in \FU$, 
Si l'une des deux conditions suivantes est vŽrifiŽe: 
\begin{enumerate}
\item[(i)] ${_FP_u}$ est un sous-groupe de Borel de $G$; 
\item[(ii)] $G$ ne contient aucun facteur (absolument) quasi-simple de type ${\bf D}_{2n+1}$ ou ${\bf E}_6$; 
\end{enumerate} on a $\phi({_F\bsfrY_u})={_F\bsfrY_u}$. En effet, quitte ˆ remplacer $u$ par $gug^{-1}$ pour un $g\in G(F)$, on peut supposer que $u$ est en position standard. 
On peut aussi supposer que $G$ est (absolument) quasi-simple. 
Soit $\lambda$ l'unique ŽlŽment de $\check{X}(A_0)\cap \Lambda_{F,u}^\mathrm{opt}$ et soit $k= m_u(\lambda)$. Alors 
$\phi \circ \lambda = -\lambda $ est l'unique ŽlŽment de $\check{X}(A_0)\cap \Lambda_{F,\phi(u)}^\mathrm{opt}$ et $m_{\phi(u)}(-\lambda)=m_u(\lambda)=k$. 
On en dŽduit que s'il existe un $g\in G(F)$ tel que $g\bullet \lambda = -\lambda$, alors $\phi(u)\in {_F\bsfrY_u}$ et $\phi({_F\bsfrY_u})= {_F\bsfrY_u}$. C'est immŽdiat si $-1\in W$ ou si 
$P_\lambda = P_0$ (car $w_0\circ \phi(P_0)=P_0$). Rappelons que $-1\in W$ si et seulement si $G$ n'est pas de type 
${\bf A}_n$ avec $n\geq 2$, ${\bf D}_{2n+1}$ ou ${\bf E}_6$. 
Si $G$ est de type ${\bf A}_n$ avec $n\geq 2$, on a $\bsfrY_u = G\bullet u$ et ${_F\bsfrY_u}= \FU \cap \bsfrY_u$ (\textit{cf.} \ref{rem sur iso springer et G de type An}\,(ii)). 
Puisque $\phi(G\bullet u) = G\bullet u$ (d'aprs \cite[2.10]{LS}) et $\phi(\FU)= \FU$, on a bien $\phi({_F\bsfrY_u})= {_F\bsfrY_u}$.
}\end{remark}

\subsection{RŽcapitulatif sur l'induction parabolique des $F$-strates}\label{rŽcapitulatif sur l'IA} 
Soit $M$ un $F$-facteur de Levi de $G$ et soit $P$ un $F$-sous-groupe parabolique de $G$ de composante de Levi $M$. Pour $w\in \FU^M$, on note $\bsfrY(F,w\,\vert \, P)$ l'unique ($\overline{F}$-)strate de $\UU$ qui intersecte l'ensemble (uniformŽment $F$-instable) ${_F\iota^P}(w)= {_F\scrX_w^{\!M}}U_P$ de manire dense. Si l'hypothse \ref{hypo bonnes F-strates faible} est vŽrifiŽe pour $G$ et si l'une des deux conditions suivantes est vŽrifiŽe: 
\begin{itemize}
\item $F$ est infini; 
\item $w\in \UF^M\;(= M(F) \cap \FU^M)$; 
\end{itemize}
alors l'intersection $\UF\cap \bsfrY(F,w\,\vert \,P)$ est non vide et c'est une $F$-strate de $\UF$, que l'on note $I_{F,P}^G(w)$. On suppose jusqu'ˆ la fin de cette sous-section 
que l'hypothse \ref{hypo bonnes F-strates faible} est vŽrifiŽe pour $G$; et on se limite au cas o $w\in \UF^M$ (qui est celui qui nous intŽresse vraiment). 
Pour tout $w\in \UF^M$, on a: 
\begin{itemize}
\item $I_{F,P}^G(w')= I_{F,P}^G(w)$ pour tout $w'\in \bsfrY_{F,w}^M$; 
\item $I_{F,P'}^G(w)= I_{F,P}^G(w)$ pour tout $F$-sous-groupe parabolique $P'$ de composante de Levi $M$; on peut donc poser $I_{F,M}^G(w)= I_{F,P}^G(w)$.
\end{itemize}
Si de plus l'hypothse \ref{hypo bonnes F-strates faible} est vŽrifiŽe pour $M$, alors pour tout $w\in \UF^M$, on a 
$$\bsfrY_{F,w}^M= \UF^M \cap \bsfrY_{w}^M \quad \hbox{et} \quad I_{F,M}^G(w) = \UF \cap I_{M}^G(w)\ptf$$ 
Si l'hypothse \ref{hypo bonnes F-strates faible} est vŽrifiŽe pour tous les $F$-facteurs de Levi de $G$ --- \textit{e.g} si $F$ est un corps local ou global, ou un corps fini ---, on peut tout aussi bien utiliser les ŽgalitŽs ci-dessus comme dŽfinition pour l'application $I_{F,M}^G$. 

Pour $F=\overline{F}$, l'application d'induction parabolique $I_M^G \;(= I_{\overline{F},M}^G)$ pour les strates est reliŽe de la manire suivante 
ˆ celle de Lusztig-Spaltenstein $\ES{J}_M^{G}: \UU^M/M \rightarrow \UU/G$ pour les orbites gŽomŽtriques. Les deux applications ne sont pas compatibles, 
au sens o pour $w\in \UU^M$ et $w'\in \bsfrY^M_w$, l'induite parabolique $\ES{J}_M^{G}(w)$ de la $M$-orbite $\ES{O}_w^M$ n'est en gŽnŽral pas contenue dans 
la strate induite $I_M^G(w)$. En revanche elle l'est toujours si $w'$ appartient ˆ la $M$-orbite dense $\ES{O}(\bsfrY_w^M)$ dans $\bsfrY_w^M$. Plus prŽcisŽment, pour 
tout $w\in \UU^M$, on a $$\ES{J}_M^{G}(\ES{O}(\bsfrY_w^M)) = \ES{O}(I_M^G(\bsfrY_w^M))\ptf$$ 

\subsection{Trois exemples en petit $F$-rang}\label{exemples en basse dimension} 
On dŽcrit dans cette sous-section les $F$-strates unipotentes pour les groupes suivants: 
$\mathrm{SL}_2$, $\mathrm{SU}_{2,1}$ et $\mathrm{Sp}_4$. Dans les trois cas, $p=2$ est le \textit{mauvais} nombre premier. 

\vskip2mm
\textit{\S1.\hskip1mm Les $F$-strates de $\UF=\UF^G$ pour $G=\mathrm{SL}_2$. ---} Pour $x\in \overline{F}$, 
notons $u(x)$ l'ŽlŽment unipotent $\left(\begin{array}{cc} 1& x \\0 & 1\end{array} \right)$ de $\mathrm{SL}_2$. 
Notons $A_0$ le tore diagonal de $\mathrm{SL}_2$ et $\lambda \in \check{X}(A_0)$ le co-caractre dŽfini par $t^\lambda= \mathrm{diag}(t,t^{-1})$. 
Tout ŽlŽment de $\UU_F$ est conjuguŽ dans $\mathrm{SL}_2(F)$ ˆ un $u(x)$ avec $x\in F$. 
Pour $x\neq 0$, on a $$\Lambda_{E, u(x)}^\mathrm{opt} \cap \check{X}(A_0)= \{\lambda\}\quad \hbox{et} \quad m_{u(x)}(\lambda)=2 \ptf $$ 
Ainsi $\UU_F$ est constituŽ de deux $F$-strates: $\bsfrY_{F,e_G}=\{e_G\}$ et $\bsfrY_{F,u(x)}= \UU_F\smallsetminus \{e_G\}$ 
pour tout $x\in F^\times$. D'autre part, pour tous $x,\, y \in F^\times $, les ŽlŽments $u(x)$ et $u(y)$ sont conjuguŽs 
dans $\mathrm{SL}_2(F)$ si et seulement si $x^{-1}y \in (F^\times)^2$. La $F$-strate non triviale $\UU_F\smallsetminus \{e_G\}$ est constituŽe 
des $\mathrm{SL}_2(F)$-orbites unipotentes non triviales de $\mathrm{SL}_2(F)$, qui sont paramŽtrisŽes par $F^\times/(F^\times)^2$. 
Si $p\neq 2$ ou $F$ est parfait, l'ensemble $F^\times / (F^\times)^2$ est fini. Si $p=2$ et $F$ n'est pas parfait, 
l'ensemble $F^\times/(F^\times)^2$ est infini.

\begin{remark}
\textup{
Soit $u=u(x)$ pour un $x\in \overline{F}\smallsetminus \{0\}$. Le centralisateur $G^{u}$ de $u$ dans $G$ 
est l'ensemble des $u(y)$ avec $y\in \overline{F}$. Si $p=2$, le centralisateur $\mathfrak{g}^u= \ker(\mathrm{Ad}_u - \mathrm{Id})$ 
de $u$ dans $\mathfrak{g}$ 
est la sous-algbre de Borel $\mathfrak{p}_\lambda= \mathfrak{g}_{\lambda,0}$ de $\mathfrak{g}$ formŽe des 
$\left(\begin{array}{cc} t &y\\0 & t \end{array}\right)$ avec $t,\,y\in \overline{F}$ et 
l'inclusion $\mathrm{Lie}(G^u)\subset  \mathfrak{g}^u$ est stricte, \textit{i.e.} le morphisme $G \rightarrow \ES{O}_u ,\, g \mapsto gug^{-1}$ 
n'est pas sŽparable (\textit{cf.} \ref{notations et rappels}). 
}\end{remark}

\vskip2mm
\textit{\S2.\hskip1mm Les $F$-strates de $\UF=\UF^G$ pour $G=\mathrm{SU}_{2,1}$. ---} 
On considre maintenant le groupe spŽcial unitaire $G=\mathrm{SU}_{2,1}$ (dŽfini et quasi-dŽployŽ sur $F$) relativement ˆ une extension quadratique sŽparable 
$E/F$. Concrtement, $G$ est le groupe spŽcial unitaire de la forme hermitienne 
$(x_{-1},x_0,x_1) \mapsto x_{-1}^\sigma x_1 + x_0^\sigma x_0 + x_1^\sigma x_{-1}$ sur $E^3$, o 
$\sigma$ est le gŽnŽrateur de $\mathrm{Gal}(E/F)$. Soit $H(E/F)$ l'ensemble des $(u,v)\in E\times E$ tels que $v+ v^\sigma = u^\sigma u$ et soit 
$U_0(F)$ est le sous-groupe de $\mathrm{SL}_3(E)$ formŽ des matrices triangulaires supŽrieures de la forme 
$$\eta(u,v)=\left(\begin{array}{ccc}
1 & -u^\sigma & -v \\
0 & 1 & u\\0 & 0 & 1\end{array}\right) \quad \hbox{avec}\quad (u,v)\in H(E/F)\ptf
$$
La loi de composition sur $U_0(F)$, qui munit $H(E/F)$ d'une structure de groupe, est donnŽe par $$\eta(u,v)\eta(u',v')= \eta(u+u', v+ v' + u^\sigma u')\ptf$$ 
En particulier $\eta(u,v)^{-1}=\eta(-u,v^\sigma)$ et $$\eta(u,v)\eta(u',v')\eta(u,v)^{-1}= \eta(u', v' + u^\sigma u' - u'^\sigma u)\ptf$$ 
Il peut tre commode d'utiliser une autre description 
du groupe $H(E/F)$. Soit $E^0$ le sous-$F$-espace vectoriel de $E$ formŽ des ŽlŽments de trace nulle et soit $e\in E$ un ŽlŽment tel que $e+ e^\sigma =1$. 
L'application $\phi_e: (u,v) \mapsto (u, v - e u^\sigma u)$ est un isomorphisme de $H(E/F)$ sur $E\times E^0$ muni de la loi de groupe 
$$(x,y)\cdot (x',y')= (x+ x', y + y' -e x x'^\sigma - e^\sigma x^\sigma x')\ptf$$ Pour $(u,v^0)\in E\times E^0$, 
en notant $\eta_e(u,v^0)$ l'ŽlŽment $\eta(u, v^0 + eu^\sigma u)$ de $U(F)$. 
on a $$\eta_e(u,v^0)= \eta_e(u,0)\eta_e(0,v^0)\ptf$$ 

\begin{remark}{\rm\'Ecrivons $E= F[z]$ avec $z^2- \alpha z + \beta=0$ et $\alpha, \, \beta\in F$. Si $p\neq 2$, on peut remplacer 
$z$ par $z-\frac{1}{2}$ ce qui a pour effet d'annuler le terme linŽaire de l'Žquation $z^2-\alpha z + \beta =0$; autrement dit on peut supposer $\alpha=0$. Dans ce cas on a 
$E^0= zF$ et $E^1 = \frac{1}{2} + z F$; o l'on a notŽ $E^1$ l'ensemble des ŽlŽments de $E$ de trace $1$. 
En revanche si $p=2$, ce qui entra"ne $\alpha \neq 0$ car $E/F$ est sŽparable, alors $E^0= F$ et $E^1 = \alpha^{-1} z F$. }\end{remark}

\vskip2mm
Le sous-groupe de $\mathrm{SL}_3(E)$ formŽ des matrices diagonales de la forme 
$$\delta(t)=\left(\begin{array}{ccc}
t & 0 & 0 \\
0 & t^{-1}t^\sigma & 0\\0 & 0 & (t^\sigma)^{-1}\end{array}\right) \quad \hbox{avec} \quad t\in E^\times
$$
est le groupe des points $F$-rationnels d'un $F$-tore maximal $M_0$ de $G$. Le groupe $P_0= M_0 \ltimes U_0$ est 
un sous-groupe de Borel de $G$ dŽfini sur $F$. Pour $t\in E^\times$ et $(u,v)\in H(E/F)$, on a $$\delta(t) \bullet \eta(u,v)= 
\eta(t^{-1}(t^\sigma)^2u,t t^\sigma v)= \eta(\psi(t)u,N_{E/F}(t)v)$$ 
avec $$\psi(t) = (t^{-1}t^\sigma) t^\sigma =(t^{-1} t^\sigma)^2 t  \quad \hbox{et} \quad N_{E/F}(t)= N_{E/F}(\psi(t))\ptf$$ 
Observons que le morphisme de groupes $\psi:E^\times \rightarrow E^\times $ est l'identitŽ sur $F^\times$. 
De manire Žquivalente, on a $$\delta(t)\bullet \eta_e(u,v^0) = \eta_e(\psi(t)u , N_{E/F}(t) v^0)\ptf$$

Soit $\lambda \in \check{X}_F(M_0)$ le co-caractre dŽfini par 
$t^\lambda=  \mathrm{diag}(t,1,t^{-1})$. Son image ${\rm Im}(\lambda)$ est le tore $F$-dŽployŽ maximal $A_0=A_{M_0}$ de $M_0$ et l'on a $M_\lambda=M_0$. Pour $t\in F^\times$ et 
$(u,v)\in H(E/F)$, on a $t^\lambda\bullet \eta(u,v)= \eta(tu,t^2v)$. 
Par consŽquent $$G_{\lambda, 1}(F)= U_0(F)\quad \hbox{et} \quad G_{\lambda,2}(F)= \{\eta(0,v)\,\vert \, v\in E,\, v + v^\sigma = 0\}\ptf$$ 
On note $\overline{\eta}(u,v)$ l'image de $\eta(u,v)$ dans $G_\lambda(1\!\!;F)= G_{\lambda,1}(F)/G_{\lambda,2}(F)$. 
Soit $\xi\in \check{X}(M_0)$ le co-caractre de $M_0$ dŽfini par $t^\xi \bullet \mathrm{diag}(a,b,c)= \mathrm{diag}(ta,t^{-2}b,tc)$. Il est dŽfini sur 
$E$ et son image $\mathrm{Im}(\xi)$ est le sous-tore $M_\lambda^\perp=(M_0)^\lambda$ de $M_0$ orthogonal ˆ $\lambda$: on a $\langle \xi \rangle = \check{X}(M_\lambda^\perp)$. 
Pour $(u,v)\in H(E/F)$, si $u\neq 0$ alors $\langle \xi \rangle \cap \Lambda_{E,\overline{\eta}(u,v)} = \{0\}$ (\textit{i.e.} 
pour $\xi'\in \langle \xi \rangle \smallsetminus \{0\}$, la limite $\lim_{t\rightarrow 0} t^\xi \bullet \overline{\eta}(u,v)$ n'existe pas); et si $v\neq 0$, alors $t^{\xi'} \bullet \eta(0,v) = \eta(0,v)$ pour tout $\xi'\in \langle \xi\rangle$. Par consŽquent si $u\neq 0$, $\overline{\eta}(u,v)$ 
est $(E,M_\lambda^\perp)$-semi-stable (dans $G_\lambda(1)$); et si $v\neq 0$, $\eta(0,v)$ est $(E,M_\lambda^\perp)$-semi-stable (dans $G_\lambda(2)$). On peut donc dans les deux cas 
appliquer \ref{thŽorme de KN rationnel sur G}: on a 
$$\Lambda_{F,\eta(u,v)}^\mathrm{opt}\cap \check{X}(A_0) = \{\lambda\} \quad \hbox{avec}\quad m_{\eta(u,v)}(\lambda)= \left\{\begin{array}{ll}
1 & \hbox{si $u\neq 0$}\\
2 & \hbox{si $u=0$ et $v\neq 0$} 
\end{array}\right..$$
En d'autres termes, $\lambda \in \bs{\Lambda}_{F,\eta(u,v)}$ si $u\neq 0$ et $\frac{1}{2}\lambda \in \bs{\Lambda}_{F,\eta(0,v)}$ si $v\neq 0$. Les deux $F$-strates unipotentes non triviales de $\UF$ sont donc: 
$$\bsfrY_1 = \{g \eta(u,v) g^{-1}\,\vert \, u\neq 0,\, g\in G(F)\}\vg$$
$$\bsfrY_2 = \{g \eta(0,v)g^{-1} \,\vert \, v\neq 0 ,\,g \in G(F)\}\ptf$$ 
Observons que $\bsfrY_1$ est la $F$-strate de $\UF$ induite de l'unique $F$-strate de $\UF^{M_0}=\{1\}$, ˆ savoir la $F$-strate triviale; \cad que 
$\bsfrY_1= I_{F,P_0}^G(1)$ et $\lambda=\mu_{P_0}$ o $\mu_{P_0}$ est le co-caractre (virtuel) de $A_0$ dŽfini en \ref{IP des ensembles FS} . 
D'autre part $\bsfrY_2$ est une $F$-strate non-induite.

\vskip2mm
\S3.\hskip1mm{\it Les $F$-strates de $\UF=\UF^G$ pour $G={\rm Sp}_4$. ---} Soit $V$ un $F$-espace vectoriel de dimension $4$ muni d'une base 
$(e_1,\ldots , e_4)$. Soit $G$ le groupe symplectique (dŽfini sur $F$) qui laisse invariante la forme bilinŽaire alternŽe non dŽgŽnŽrŽe $(\cdot , \cdot )$ sur $F^{(4)}$ 
dŽfinie par la matrice $\left(\begin{array}{cc} 0& J\\ -J&0\end{array}\right)$ avec $J= \left(\begin{array}{cc}0&1\\1&0\end{array}\right)$. Soit $P_0=A_0\ltimes U_0$ le $F$-sous-groupe 
de Borel de $G$ dŽfini comme suit: $U_0$ est formŽ des matrices 
triangulaires supŽrieures de la forme 
$$\eta(u,v,w,x)=\left(\begin{array}{cccc}
1 & u& w+uv& x+uw\\
0 & 1 & v &w  \\
0 & 0 & 1 & -u\\
0 & 0 &0& 1\end{array}\right) \quad \hbox{avec}\quad (u,v,w,x)\in \smash{\overline{F}}^{(4)}
$$ et $A_0$ est formŽ des matrices diagonales de la forme $$\mathrm{diag}(z,t,t^{-1},z^{-1})
\quad \hbox{avec} \quad (z,t)\in \smash{\overline{F}}^\times \times \smash{\overline{F}}^\times \ptf$$
On a 
$$\delta(z,t)\bullet \eta(u,v,w,x)= \eta(zt^{-1} u, t^2v, zt w, z^2 x)\ptf$$ Soient $\alpha,\,\beta \in X(A_0)$ les caractres dŽfinis par 
$$\delta(z,t)^\alpha = zt^{-1}\quad \hbox{et} \quad \delta(z,t)^\beta = t^2\ptf$$ Alors $\Delta_0=\{\alpha, \beta\}$ est la base de l'ensemble $\ES{R}^+ =  \{\alpha, \beta, \alpha + \beta , 2\alpha + \beta\}$ 
des racines de $A_0$ dans $U_0$. Pour $(u,v,w,x)\in F^{(4)}$, on pose $e_\alpha(u)= \mu(u,0,0,0)$, $e_\beta(v)= \mu(0,v,0,0)$, $e_{\alpha + \beta}(w)= \mu(0,0,w,0)$ et $e_{2\alpha + \beta}(x)= \mu(0,0,0,x)$. 
On a $$\eta(u,v,w,x) = e_\alpha(u)e_\beta(v) e_{\alpha + \beta}(w)e_{2\alpha + \beta}(x)\ptf$$ Puisque pour 
$\gamma\in \ES{R}^+$ et $X\in \overline{F}$, on a $$\delta(z,t)\bullet e_\gamma(X)= e_X(\delta(z,t)^\gamma X)\vg$$ 
$\eta$ dŽfini un $F$-isomorphisme $A_0$-Žquivariant $j_0: \mathfrak{u}_0 \rightarrow U_0$. 

Soit $\lambda,\, \mu \in \check{X}(A_0)$ les co-caractres dŽfini par $t^{\lambda} = \delta(t,1)$ et $t^{\mu}= \delta(t,t)$. On a donc 
$$\delta(z,t)= z^\lambda t^{\mu-\lambda}= (zt^{-1})^\lambda t^\mu\ptf$$ Observons que la base $\{\mu_\alpha, \mu_\beta\}$ de $\check{X}(A_0)_{\mbb{Q}}$ duale de $\Delta_0$ est donnŽe par 
$$\mu_\alpha= \lambda \quad \hbox{et} \quad \mu_\beta = \frac{1}{2}\mu\ptf$$ 
Les groupes $P_\lambda$ et $P_\mu$ sont les deux $F$-sous-groupes paraboliques standard maximaux (propres) de $G$. Le premier est donnŽ par $P_\lambda= M_\lambda \ltimes U_\lambda$ o 
$M_\lambda \simeq \mbb{G}_{\rm m} \times \mathrm{SL}_2$ est le groupe des matrices de la forme $\left(\begin{array}{ccc} a& 0&0\\ 0&A &0 \\ 0& 0 & a^{-1}\end{array}\right)$ avec $a\in \smash{\overline{F}}^\times$ et 
$A\in \mathrm{SL}_2$, et $U_\lambda= \eta(\overline{F},0,\overline{F},\overline{F})$. 
Le second (dit de \textit{Siegel}) est donnŽ par $P_\mu= M_\mu\ltimes U_\mu$ o $M_\mu $ o $M_\mu\simeq \textrm{GL}_2$ est le groupe des matrices de la forme $\left(\begin{array}{cc} A& 0\\ 0&\wt{A}\end{array}\right)$ 
avec $A\in \textrm{GL}_2$ et $\wt{A}= J ({^{\mathrm{t}}A^{-1}})J$, et $U_\mu = \eta(0,\overline{F},\overline{F},\overline{F})$. 
Observons que:
\begin{itemize}
\item pour $u,\,w\in F^\times$ et $x\in F$, l'ŽlŽment $\eta(u,0,w,x)$ est conjuguŽ ˆ $\eta(0,0,1,x)$ par un ŽlŽment de $M_\lambda(F)$;
\item pour $u\in F^\times$ et $x\in F$, l'ŽlŽment $\eta(u,0,0,x)$ est conjuguŽ ˆ $\eta(0,0,u,x)$ par un ŽlŽment de $M_\lambda(F)$ (en notant $r_\beta$ la reflexion 
associŽe ˆ la racine $\beta$, on a $r_\beta(\alpha)= \alpha+\beta$ et $r_\beta(2\alpha+\beta)= 2\alpha+\beta$); 
\item pour $v\in F^\times$, l'ŽlŽment $\eta(0,v,0,0)$ est conjuguŽ ˆ $\eta(0,0,0,v)$ par un ŽlŽment de $M_\mu$ (en notant $r_\alpha$ la reflexion associŽe ˆ la racine $\alpha$, on a 
$r_\alpha(\beta)=2\alpha + \beta$). 
\end{itemize}

Posons 
\begin{eqnarray*}
\scrY_1&=& \{\eta(u,v,w,x)\,\vert\, uv\neq 0\}\vgq \mu_1 = \mu_\alpha + \mu_\beta \vg\\
\scrY_2 &= & \{\eta(0,v,w,x)\,\vert \, \hbox{$w^2-vx\neq 0$}\}\vgq \mu_2 = \mu_\beta\vg\\
\scrY_3 &= & \{\eta(0,0,0,x)\,\vert\, x\neq 0\} \vgq \mu_3= \frac{1}{2}\mu_\alpha \ptf
\end{eqnarray*}
Pour $i=1,\,2,\,3$, $\scrY_i$ est une $F$-lame standard (non triviale) de $\UU_F$ et $\mu_i$ est l'unique ŽlŽment de 
$\check{X}(A_0)_{\mbb{Q}}\cap \bs{\Lambda}_{F,u_i}$ pour un (\textit{i.e.} pour tout) $u_i\in \scrY_i$. La vŽrification est laissŽe au lecteur (on peut dans les trois cas 
passer ˆ l'algbre de Lie et utiliser le critre de Kirwan-Ness rationnel). Pour $i=1,\,2,\,3$, posons $\bsfrY_i=G(F)\bullet \scrY_i$. Puisque les co-caractres virtuels 
(en position standard) $\mu_1,\,\mu_2,\,\mu_3$ sont deux-ˆ-deux distincts, les $F$-strates $\bsfrY_1,\,\bsfrY_2,\,\bsfrY_3$ de $\UU_F$ sont deux-ˆ-deux disjointes; et comme 
tout ŽlŽment de $U_0(F)\smallsetminus\{1\}$ est $G(F)$-conjuguŽ ˆ un ŽlŽment de l'une des $F$-lames standard $\scrY_1,\,\scrY_2,\,\scrY_3$, on a la dŽcomposition en union disjointe 
$$\UF= \{1\} \coprod \bsfrY_1\coprod \bsfrY_2\coprod \bsfrY_3\ptf$$ 
Les $F$-strates $\bsfrY_1$ et $\bsfrY_2$ sont induites ˆ partir de la $F$-strate triviale $\{1\}$: on a $$\bsfrY_1= I_{F,P_0}^G(\{1\})\quad\hbox{et}\quad \bsfrY_2= I_{F,P_\mu}^G(\{1\})= I_{F,P_\lambda}^G(\{1\})\ptf$$ 
La $F$-strate $\bsfrY_3$ est non-induite. 
Observons que $i_F^{P_\mu}(\{1\})= U_\mu(F)$ est le $F$-saturŽ de la $F$-lame standard $\scrY_2$ de $\UU_F$; en revanche $i_F^{P_\lambda}(\{1\})=U_\lambda(F)$ \textit{n'est pas} le $F$-saturŽ d'une $F$-lame de $\UU_F$. 

Pour $F=\overline{F}$, les strates (gŽomŽtriques) $\bsfrY_1$ et $\bsfrY_3$ sont des $G$-orbites, et la strate $\bsfrY_2$ est une $G$-orbite si et seulement si 
$p$ est bon pour $G$ (\textit{i.e.} $p\neq 2$). Si $p=2$, ce que l'on suppose jusqu'ˆ la fin de cette sous-section, la strate $\bsfrY_2$ est rŽunion de deux $G$-orbites distinctes: celle de 
$u_{2,1}=\eta(0,0,1,1)$ et celle de $u_{2,2}= \eta(0,0,1,0)$. Pour $i=1,\,2$, on a donc $\scrY_2=\scrY_{u_{2,i}}$ et le saturŽ $\scrX_2= \scrX_{u_{2,i}}$ de $\scrY_2$ est le radical unipotent $U_\mu$ de $P_\mu$. 
Puisque le co-caractre $\mu$ est $u_{2,i}$-optimal, le centralisateur rŽduit $Z_G(u_{2,i})$ de $u_{2,i}$ est contenu dans $P_\mu$. Or pour  
$p= \left( \begin{array}{cc}A& B\\ 0& \wt{A}\end{array}\right)\in P_\mu=P_\mu(\overline{F})$, on a  
$$pu_{2,1} p^{-1}= \left(\begin{array}{cc} 1 & A\left(\begin{array}{cc}1&1\\0& 1 \end{array}\right) \wt{A}^{-1}\\ 0 & 1 \end{array}\right)\quad \hbox{et}\quad 
p u_{2,2} p^{-1}= \left(\begin{array}{cc} 1 & A \wt{A}^{-1}\\ 0 & 1 \end{array}\right)\ptf$$ 
On en dŽduit que $G^{u_{2,1}}=U_\mu$ et $G^{u_{2,2}}= \left\{\left( \begin{array}{cc}A& B\\ 0& \wt{A}\end{array}\right)\in P_\mu\,\vert \, A \in \mathrm{SL}_2\right\}$. 
Comme $\dim(G^{u_{2,1}})=4$ et $\dim(G^{u_{2,2}})=6$, on obtient que $\ES{O}_{u_{2,1}}= \ES{O}(\bsfrY_2)$ est l'unique $G$-orbite de $\bsfrY_2$ qui soit dense dans $\bsfrY_2$; 
c'est aussi l'unique $G$-orbite de $\UU$ qui intersecte $U_\mu$, resp. $U_\lambda$, de manire dense. En d'autres termes, $\ES{O}_{u_{2,1}}$ est l'induite parabolique de Lusztig-Spaltenstein de l'orbite triviale de $M_\mu$, resp. $M_\lambda$ (\textit{i.e.} l'orbite de Richardson associŽe ˆ  
$P_\mu$, resp. $P_\lambda$): 
$$\ES{O}_{u_{2,1}}= \ES{J}_{P_\mu}^{G}(\{1\})=\ES{J}_{P_\lambda}^{G}(\{1\})\ptf$$ 
D'autre part, en notant $\ES{O}_{\mathrm{r\acute{e}g}}=\ES{O}_{\mathrm{r\acute{e}g}}^G\;(=\bsfrY_1)$ l'orbite unipotente rŽgulire de $G$, on a 
$$\ES{O}_{\mathrm{r\acute{e}g}}= \ES{J}_{P_\mu}^{G}(\ES{O}_{\mathrm{r\acute{e}g}}^{M_\mu})= \ES{J}_{P_\lambda}^{G}(\ES{O}_{\mathrm{r\acute{e}g}}^{M_\lambda})\ptf$$  
Comme $\ES{O}_{\mathrm{r\acute{e}g}}^{M_\mu}$ est l'unique $M_\mu$-orbite non triviale de $\UU^{M_\mu}$ et $\ES{O}_{\mathrm{r\acute{e}g}}^{M_\lambda}$ est l'unique $M_\lambda$-orbite non triviale $\UU^{M_\lambda}$, 
on obtient que $\ES{O}_{u_{2,1}}\;(= \ES{O}(\bsfrY_2))$ et $\ES{O}_{\mathrm{r\acute{e}g}}= \bsfrY_1$ sont les seules $G$-orbites de $\UU$ qui sont induites 
au sens de Lusztig-Spaltenstein; en particulier l'orbite $\ES{O}_{u_{2,2}}$ est \textit{rigide}, \textit{i.e.} non-induite.

\part*{Partie II: prŽ-dŽveloppement fin de la contribution unipotente ˆ la formule des traces}
\section{La distribution $\Jres^T_\u$}\label{la distribution J}
Dans cette section, on reprend sans les rŽintroduire les notations de \cite{LL} dans 
le cas non tordu ($\wt{G}=G$ et $\omega=1$). En particulier, $F$ est un corps global de caractŽristique $p>1$.

\subsection{Troncature}\label{troncature}
Rappelons que l'on a notŽ $\UF=\UF^G$ l'ensemble des (vrais) ŽlŽments unipotents de $G(F)$: 
$$
\UF= \{gug^{-1}\,\vert\, g\in G(F),\, u\in U_0(F)\}\ptf
$$

Pour une fonction $f\in C^\infty_\mathrm{c}(G(\mbb{A}))$ et un paramtre $T\in \ag_0$, on rappelle la dŽfinition 
de la fonction $k_\u^T(x)= k_\u^T(f\!\!;x)$:\index{kunipTx@$k_\u^T(x)$}
$$
k_\u^T(x)= \sum_{P\in \ESP_\mathrm{st}}(-1)^{a_P-a_G} \sum_{\xi \in P(F)\backslash G(F)} 
k_{P,\u}^T(\xi x)
$$
avec
$$
k_{P,\u}^T(x)= \wh{\tau}_P({\bf H}_0(x)-T)k_{P,\u}(x)\vg
$$
$$
k_{P,\u}(x)=\sum_{\delta \in \UF^{M_P}}\int_{U_P(\mbb{A})}f(x^{-1} \delta u x) \dd u\ptf
$$
D'aprs \cite[3.7.1]{LW}, la somme sur $\xi$ porte sur un ensemble fini. 
Observons que pour $P=G$, la fonction $ k_{G,\u}^T$ sur $G(\mbb{A})$ ne dŽpend pas de $T$ et elle co\"{\i}ncide avec la fonction $k_{\u}=k_{G,\u}$ donnŽe par 
$$
k_{\u} (x)= \sum_{\delta \in \UF}f(x^{-1}\delta x)\ptf
$$

Pour $Q\in \ESP_\mathrm{st}$, on rappelle la dŽfinition de l'opŽrateur de troncature $\bs{\Lambda}^{T,Q}$ appliquŽ ˆ une fonction 
$\varphi \in L^1_\mathrm{loc}(Q(F)\backslash G(\mbb{A}))$:
$$
\bs{\Lambda}^{T,Q}\varphi (x) = \sum_{P\in \ESP_\mathrm{st}^Q} (-1)^{a_P-a_Q}
\sum_{\xi \in P(F)\backslash Q(F)} \wh{\tau}_P^{\hspace{0.7pt}Q}({\bf H}_0(\xi x)-T)\varphi_P(\xi x)
$$
avec $$\varphi_P(x)= \int_{U_P(F)\backslash U_P(\mbb{A})}\varphi(xu) \dd u;$$ o 
$\dd u$ est la mesure de Tamagawa sur $U_P(\mbb{A})$. Pour $Q=G$, on pose $\bs{\Lambda}^T=\bs{\Lambda}^{T,G}$. 

Pour $P\in \ESP_\mathrm{st}$ et $\phi\in L^1_\mathrm{loc}(P(F)\backslash G(\mbb{A}))$, on a l'identitŽ \cite[8.2.1]{LW} 
$$
\wh{\tau}_P({\bf H}_0(x)-T) \phi_P(x)=\sum_{\substack{Q,R\in \ESP_\mathrm{st}\\ Q\subset P \subset R}}\sum_{\xi \in Q(F)\backslash P(F)}
\sigma_Q^R({\bf H}_0(\xi x)-T)\bs{\Lambda}^{T,Q}\phi(\xi x) \ptf
\leqno{(1)}$$
D'aprs \cite[1.7.1, 3.7.1]{LW}, la somme sur $\xi$ porte sur un ensemble fini.  
En appliquant (1) ˆ la fonction $\phi = k_{P,\u}$, puisque $(k_{P,\u})_P=k_{P,\u}$, 
on obtient 
$$k_{P,\u}^T(x)= \sum_{\substack{Q,R\in \ESP_\mathrm{st}\\ Q\subset P \subset R}}
\sum_{\xi \in Q(F)\backslash P(F)}\sigma_Q^R({\bf H}_0(\xi x)-T)\bs{\Lambda}^{T,Q} k_{P,\u}(\xi x)\ptf \leqno{(2)}$$ 

\subsection{Convergence et rŽcriture de l'intŽgrale tronquŽe}\label{CV et rŽcriture de l'intŽgrale tronquŽe} 
Soit $\bs{K}'$ un sous-groupe ouvert compact de $G(\mbb{A})$ et soit $Q\in \ESP_\mathrm{st}$. Observons que pour 
toute fonction $\phi$ sur $Q(F)\backslash G(\mbb{A})$ invariante ˆ droite par $\bs{K}'$, la fonction $\phi_Q$ sur $M_Q(F)\backslash M_Q(\mbb{A})$ est invariante 
ˆ droite par $\bs{K}'\cap M_Q(\mbb{A})$. On sait d'aprs \cite[4.2.2]{LL} que si $T\in \ag_0$ est assez rŽgulier, prŽcisŽment si $\bs{d}_0(T)\geq c$ 
pour une constante $c >0$ ne dŽpendant que de $\bs{K}'\cap M_Q(\mbb{A})$, alors pour toute fonction 
$\phi$ sur $Q(F)\backslash G(\mbb{A})$ invariante ˆ droite par $\bs{K}'$, on a l'ŽgalitŽ
$$\bs{\Lambda}^{T,Q}\phi(x) = F_{P_0}^Q(x,T)\phi(x);$$ 
o $F_{P_0}^Q(\cdot ,T)$ est la fonction caractŽristique d'un ensemble $Q(F)\bs{\mathfrak{S}}_{P_0}^Q(T_G,T)$ 
dŽfini en \cite[3.4]{LL}. L'ensemble $\bs{\mathfrak{S}}_{P_0}^Q(T_G,T)$ dŽpend d'un sous-ensemble compact $C_Q$ de $G(\mbb{A})$ 
qui n'appara"t pas dans la notation; en pratique, on prend 
$C_Q$ assez gros et $-T_G$ et $T$ suffisamment rŽgulier --- prŽcisement, tels que $\bs{d}_0(-T_G) > c'$ et $\bs{d}_0(T)>c $ 
pour des constantes $c'>0$ et $c>0$ ne dŽpendant que de $G$ --- 
de telle manire que la proposition \cite[3.6.3]{LW} soit vŽrifiŽe. 
La fonction $x\mapsto F_{P_0}^Q(x,T)$ est supposŽe invariante ˆ gauche par l'image $\mathfrak{B}_Q$ d'une section du morphisme 
${\bf H}_Q: A_Q(\mbb{A})\rightarrow \ES{B}_Q$ et invariante ˆ droite par $\bs{K}$. 
Quitte ˆ grossir $C_Q$, ce qui est loisible puisque l'ensemble $A_Q(F)\mathfrak{B}_Q\backslash A_Q(\mbb{A})= A_Q(F)\backslash A_Q(\mbb{A})^1$ 
est compact, on peut mme supposer qu'elle est invariante ˆ gauche par $A_Q(\mbb{A})$. 

Soit $$C_\mathrm{c}(G(\mbb{A})/\!\!/ \bs{K}') \subset C^\infty_\mathrm{c}(G(\mbb{A}))$$ le sous-espace vectoriel formŽ des fonctions 
invariantes ˆ droite et ˆ gauche par $\bs{K}'$\footnote{Pour $\bs{K}''\subset \bs{K}'$, on a l'inclusion 
$ C_\mathrm{c}(G(\mbb{A})/\!\!/ \bs{K}')\subset C_\mathrm{c}(G(\mbb{A})/\!\!/ \bs{K}'')$. Par consŽquent, si l'on a fixŽ un sous-groupe ouvert compact maximal 
$\bs{K}$ de $G$ (\textit{e.g.} tel que $G(\mbb{A})=P_0(\mbb{A})\bs{K}$), 
quitte ˆ remplacer $\bs{K}'$ par un groupe plus petit, on peut toujours 
supposer que c'est un sous-groupe distinguŽ de $\bs{K}$.}. D'aprs la discussion prŽcŽdente, il existe une constante $c_{\bs{K}'}>0$ telle que 
pour tout $T\in \ag_0$ tel que $\bs{d}_0(T)>c_{\bs{K}'}$ et toute fonction $f\in C_\mathrm{c}(G(\mbb{A})/\!\!/ \bs{K}')$, l'expression 
$k_\u^T(x)$ soit Žgale ˆ $$\sum_{\substack{Q,R\in \ESP_\mathrm{st}\\ Q\subset R}}\sum_{\xi \in Q(F)\backslash G(F)}F_{P_0}^Q(\xi x,T)
\sigma_Q^R({\bf H}_0(\xi x)-T)\sum_{\substack{P\in \ESP_\mathrm{st}\\Q \subset P \subset R}}(-1)^{a_P-a_G} 
k_{P,\u}(\xi x)\ptf$$
Pour prouver la convergence absolue de l'intŽgrale\index{JunipTf@$\Jres^T_\u(f)$}
$$\Jres^T_\u(f)\bydef \int_{\overline{\bs{X}}_G} k^T_\u (x)\dd x \vg$$
il suffit de prouver, pour chaque paire $Q\subset R$ de sous-groupes paraboliques standard, celle de l'intŽgrale
$$\int_{\bs{Y}_Q}F^Q_{P_0}(x,T)\sigma_Q^R({\bf H}_0(x)-T) 
\left|\sum_{\substack{P\in \ESP_\mathrm{st}\\Q \subset P \subset R}}(-1)^{a_P-a_G} k_{P,\u}(x) \right|\dd x
\leqno{(1)}$$
o l'on a posŽ \index{YbsQ@$\bs{Y}_Q$}
$$\bs{Y}_Q = A_G(\mbb{A})Q(F)\backslash G(\mbb{A})\ptf$$
Cette convergence est prouvŽe dans \cite[9.1.1]{LL}\footnote{Cette preuve est une adaptation 
de celle de \cite[9.1.1]{LW}, laquelle 
reprend dans le cas tordu celle du thŽorme 7.1 de \cite{A1}.}. De plus la fonction $T\mapsto \mathfrak \Jres^T_\u(f)$ dŽfinit un ŽlŽment de  \textrm{PolExp}. 
PrŽcisŽment, pour $T$ et $X$ dans $\ag_{0,\mbb{Q} }$ assez rŽguliers, on a
$$\Jres^{T+X}_\u(f)=\sum_{Q\in\ESP_\mathrm{st}}\sum_{Z\in \bsbbc_Q}
\eta_{Q,F}^{G,T}(Z\mathpvg X)\;\Jres^{M_Q,X}_\u(Z\mathpvg f_Q)$$
avec $$\eta_{Q,F}^{G,T}(Z\mathpvg X)=\sum_{H\in\ESB_Q^G(Z)}\Gamma_Q^G(H-X,T)$$ et 
$$\Jres^{M_Q,X}_\u(Z\mathpvg f_Q)=\int_{\overline{\bs{X}}_{M_Q}(Z)}k^{M_Q,X}_\u(f_Q\mathpvg m)\dd m$$ o 
$$f_Q(m)=\int_{U_Q(\mbb{A})\times\bs{K}}f(k^{-1}m uk)\dd u\dd k\ptf$$
Rappelons que $\ESB_Q^G(Z)\subset \ESC_Q^G = \ESB_G\backslash \ESA_Q$ 
est la fibre au-dessus de $Z$ pour la suite exacte courte
$$ 0 \rightarrow \ESB_Q^G \rightarrow \ESC_Q^G \rightarrow \bsbbc_Q= \ES{B}_Q\backslash \ES{A}_Q  \rightarrow 0$$ et que $\overline{\bs{X}}_{\!M_Q}(Z)$ 
est l'image dans $\overline{\bs{X}}_{\!M_Q} = A_Q(\mbb{A})M_Q(F)\backslash M_Q(\mbb{A})$ de l'ensemble des $m\in M_Q(\mbb{A})$ 
tels que ${\bf H}_Q(m) + \ESB_Q= Z$.

\begin{remark}\label{propriŽtŽ cruciale}
\textup{La preuve de \cite[9.1.1]{LL} utilise la propriŽtŽ cruciale suivante: 
pour tout $Q\in \ESP$, on a la dŽcomposition \cite[3.3.2\,(ii)]{LL}
$$\UF^G \cap Q(F) = (\UF^G\cap M_Q(F))U_Q(F)= \UF^{M_Q}U_Q(F)\ptf$$ 
Observons que cette dŽcomposition est aussi une consŽquence de \ref{intersection FUG et P}.}
\end{remark}

Revenons ˆ l'expression pour $\Jres^{T}_\u(f)$. C'est la somme sur les paires de sous-groupes paraboliques standard $Q\subset R$ 
de l'intŽgrale sur $\bs{Y}_Q$ obtenue en otant les valeurs absolues dans l'expression (1). Le terme principal correspondand ˆ $Q=R=G$ est donnŽ par 
l'intŽgrale (pour $T\in \ag_0$ assez rŽgulier)
$$\int_{\overline{\bs{X}}_G} F^G_{P_0}(x,T) k_{\u}(x) \dd x = \int_{\overline{\bs{X}}_G} \bs\Lambda^T k_{\u}(x) \dd x \ptf$$ 
La proposition suivante est la version corps de fonctions du thŽorme 3.1 de \cite{A2}. 

\begin{proposition}\label{raffinement}Il existe une constante $c >0$ (qui dŽpend de $f$) telle que 
 pour $\bs{d}_0(T)\ge c$, on ait 
$$  \mathfrak{J}^T_\mathrm{unip}(f) = 
 \int_{\overline{\bsX}_G} \bs\Lambda^Tk_{\u}(x) \dd x \ptf$$
\end{proposition}

\begin{remark}\label{formule exacte, support et niveau}
\textup{
\begin{enumerate}
\item[(i)] Dans le cas des corps de nombres, Arthur (\textit{loc.~cit.}) donne une borne pour l'expression
$$\left| \mathfrak{J}^T_\mathrm{unip}(f) - \int_{\overline{\bsX}_G} \bs\Lambda^Tk_{\u}(x) \dd x \right| \ptf$$ 
Ici la formule est exacte: si $T$ est assez rŽgulier, seul le terme correspondant ˆ $Q=R=G$ dans l'expression \ref{troncature}\,(2) pour $k^T_\u(x)$ peut donner une 
contribution non triviale ˆ l'intŽgrale $ \mathfrak{J}^T_\mathrm{unip}(f)$.
\item[(ii)] La constante $c>0$ dŽpend du support de $f$ mais aussi d'un sous-groupe ouvert compact $\bs{K}'$ de $G(\mbb{A})$ tel que 
$f$ soit $\bs{K}'$-biinvariante. Si l'on a fixŽ $\bs{K}'$ et un sous-ensemble ouvert compact $\Omega$ de $G(\mbb{A})$ tel que $\bs{K}'\Omega \bs{K}'=\Omega$, 
alors en notant $$C(\Omega/\!\!/\bs{K}')\subset C_\mathrm{c}(G(\mbb{A})/\!\!/\bs{K}')$$ le sous-espace vectoriel de dimension finie 
(sur $\mbb{C}$) formŽ des fonctions ˆ support dans $\Omega$, 
on peut choisir la constante $c>0$ telle que pour $\bs{d}_0(T)\geq c$, l'egalitŽ de la proposition \ref{raffinement} 
soit vraie pour toute fonction $f\in C(\Omega/\!\!/\bs{K}')$. 
\end{enumerate}
}\end{remark}

\begin{proof}
Pour $Q=R \subsetneq G$, on a $\sigma_Q^R =0$. On veut prouver que les termes (1) 
associŽs aux paires de sous-groupes paraboliques standard 
$Q\subsetneq R \subset G$ sont tous nuls. Il suffit pour cela de reprendre en la prŽcisant la preuve de \cite[9.1.1]{LL}\footnote{Le rŽsultat 
est en fait dŽjˆ contenu dans \textit{loc.~cit.}: la fonction 
$g(x,\Lambda, \eta)$ est ˆ support compact en $\Lambda$ comme transformŽe de Fourier d'une fonction lisse et ˆ support compact sur 
$\mathfrak{u}_Q^*(\mbb{A})$. 
Comme $x$ varie dans un compact et que $\mathrm{Ad}_a$ dilate $\Lambda$, il en rŽsulte que si $T$ est suffisamment rŽgulier, 
$\mathrm{Ad}_a(\Lambda)$ sort du support de $g(x,\cdot,\eta)$ pour tout $(x,\eta)$. 
Cette dŽmonstration est le prototype de celles qui vont suivre dans la section \ref{dŽcomposition suivant les strates}, 
c'est pourquoi nous la reprenons en dŽtail ici.}.
Pour une telle paire $(Q,R)$, si $T$ est assez rŽgulier, prŽcisŽment si 
$\bs{d}_0(T)\geq c_f$ pour une constante $c_f >0$ ne dŽpendant que du support de $f$\footnote{ Dans l'intŽgrale (1), on a une somme alternŽe 
sur $P$ d'expressions $K_{P,\u}(x)$ faisant intervenir une somme sur $\UF^{M_P}$. La condition sur le support de $f$ permet 
d'appliquer \cite[3.6.7]{LW} et de remplacer cette somme sur $\UF^{M_P}$ par une somme sur $\UF^{M_P}\cap Q(F)= \UF^{M_Q}U_Q^P(F)$ avec $U_Q^P=M_P\cap U_Q$.}, l'intŽgrale (1) se rŽcrit
$$\int_{\bsY_Q} F_{P_0}^Q(x,T) \Xi_Q^R(x) \dd x\leqno{(2)}$$
avec $$\Xi_Q^R(x)=\sigma_Q^R({\bf H}_0(x)-T)\sum_{\eta \in \UF^{M_Q}}
\left| \sum_{\Lambda\in\nn^\vee(Q,R)}g(x,\Lambda,\eta) \right|$$
et $$g(x,\Lambda,\eta)=\int_{\mathfrak{u}_Q(\mbb{A})}\psi(\langle \Lambda,X\rangle) f(x^{-1}\eta j(X)x)\dd X; $$ 
o $j=j_Q: \mathfrak{u}_Q \rightarrow U_Q$ (\textit{cf.} \ref{j_Q}) est un $F$-isomorphisme de variŽtŽs   
tel que $$j\circ \mathrm{Ad}_a = \mathrm{Int}_a\circ j\quad\hbox{pour tout}\quad a\in A_Q$$ et $\nn^\vee(Q,R)$ 
est un sous-ensemble de $\mathfrak{u}_Q^*(\mbb{A})$ dŽfini comme suit. Fixons 
un caractre non trivial $\psi$ de $F\backslash \mbb{A}$ et notons $\nn^\vee$ l'orthogonal de $\nn=\mathfrak{u}_Q(F)$ pour ce caractre: 
$$\nn^\vee = \{\Lambda \in \mathfrak{u}_Q^*(\mbb{A})\,\vert\, \psi(\langle \Lambda, X\rangle)=1,\, \forall X\in \mathfrak{u}_Q(F)\}\ptf$$
Pour $P\in \ES{P}_\mathrm{st}$ tel que $Q\subset P \subset R$, notons $\mathfrak{n}^\vee(P)$ le sous-ensemble de $\nn^\vee$ formŽ des $\Lambda\in \nn^\vee$ 
tels que $\Lambda\vert_{\mathfrak{u}_P(\mbb{A})}= 0$. Alors\footnote{Pour $P\in \ES{P}_\mathrm{st}$ tel que $Q\subset P \subset R$, la formule 
de Poisson donne l'ŽgalitŽ 
$$\int_{U_P(F)\backslash U_P(\mbb{A})}\sum_{X\in \mathfrak{u}_Q(F)}f(x^{-1}\eta j(X)ux)\dd u= \sum_{\Lambda \in \nn^\vee(P)}g(x,\Lambda,\eta)\ptf$$ 
On fait ensuite la somme alternŽe sur les $P$ de l'expression ci-dessus. On a $\nn^\vee(P) \subset \nn^\vee(R)$ et comme $\nn^\vee(P)\cap \nn^\vee(P') = \nn^\vee(P\cap P')$, 
seuls les $\Lambda$ qui sont dans $\nn^\vee(Q,R)$ peuvent donner 
une contribution non triviale ˆ cette somme altern\'ee (d'aprs \cite[1.2.3]{LW}).}
 $$\nn^\vee(Q,R)\bydef \nn^\vee(R) \smallsetminus \bigcup_{Q \subset P \subsetneq R} \nn^\vee(P)\ptf$$

Posons\index{ZbsQ@$\bs{Z}_Q$} $$\bs{Z}_Q = A_G(\mbb{A})A_Q(F)\backslash A_Q(\mbb{A})\subset \bsY_Q$$
et
$$\Theta_Q^R(u,x)= \int_{\bs{Z}_Q}\Xi_Q^R(uax)\bs{\delta}_Q(a)^{-1}\dd a\ptf$$
En utilisant la dŽcomposition 
d'Iwasawa on obtient que l'intŽgrale (2) est Žgale ˆ
$$\int_{\bs{K}}\int_{\overline{\bs{X}}_{\!M_Q}} \int_{U_Q(F)\backslash U_Q(\mbb{A})} F_{P_0}^Q(m,T)
\Theta_Q^R(u,mk) \bs{\delta}_Q(m)^{-1}\dd u \dd m \dd k\ptf\leqno{(3)}$$
D'aprs \cite[1.8.3]{LW} et \cite[3.4]{LL}, l'intŽgrale en $m$ porte sur un ensemble compact 
de $\overline{\bs{X}}_{\!M_Q}=A_Q(\mbb{A})M_Q(F)\backslash M_Q(\mbb{A})$. On peut donc la remplacer par une intŽgrale 
sur un sous-ensemble ouvert compact $\Gamma_{M_Q}$ d'un ensemble de Siegel (dans $M_Q(\mbb{A})$) 
pour le quotient $\overline{\bs{X}}_{\!M_Q}$. On peut aussi remplacer l'intŽgrale en $u$ par une intŽgrale sur un 
ouvert compact $\bs{\mathfrak{S}}_{U_Q}$ de $U_Q(\mbb{A})$ qui soit un \textit{domaine de Siegel} 
pour le quotient $U_Q(F)\backslash U_Q(\mbb{A})$, \cad tel que l'application naturelle $\Siegel_{U_Q}\rightarrow U_Q(F)\backslash U_Q(\mbb{A})$ soit bijective. 
Posons $U_Q^R = M_R \cap U_Q$. On a $U_Q= U_R \rtimes U_Q^R$. 
On peut supposer\footnote{
Puisque $U_Q^R$ normalise $U_R$, on a $$U_Q(\mbb{A})= U_R(\mbb{A})U_Q^R(F)\Siegel_{U_Q^R}= U_Q^R(F)U_R(\mbb{A})\Siegel_{U_Q^R}= 
U_Q^R(F)U_R(F)\Siegel_{U_R}\Siegel_{U_Q^R}$$ avec $$U_Q^R(F)U_R(F)=U_R(F)U_Q^R(F)=U_Q(F)\ptf$$ On prouve de la mme manire que l'application naturelle
$\Siegel_{U_R}\Siegel_{U_Q^R} \rightarrow U_Q(F)\backslash U_Q(\mbb{A})$ est bijective.} que $\bs{\mathfrak{S}}_{U_Q}= \bs{\mathfrak{S}}_{U_R}\bs{\mathfrak{S}}_{U_Q^R}$ 
avec $\Siegel_R\subset U_R(\mbb{A})$ et $\bs{\mathfrak{S}}_{U_Q^R}\subset U_Q^R(\mbb{A})$. On remplace l'intŽgrale en 
$u\in \bs{\mathfrak{S}}_{U_Q}$ par une intŽgrale en $u'\in \bs{\mathfrak{S}}_{U_R}$ suivie d'une intŽgrale en $u''\in \bs{\mathfrak{S}}_{U_Q^R}$. 
Puisque $\Lambda \vert_{\mathfrak{u}_R(\mbb{A})}=0$, l'intŽgrale en $u'$ est absorbŽe par l'intŽgrale en $X$ dans la dŽfinition de $g(u'u''amk,\Lambda,\eta)$. 
D'autre part on a
$$\bs{\delta}_Q(a)^{-1}g(u''amk,\Lambda,\eta)= g(a^{-1}u''a mk,\mathrm{Ad}_{a^{-1}}^*(\Lambda),\eta)\leqno{(4)}$$ 
avec $\mathrm{Ad}_{a^{-1}}^*(\Lambda)= \Lambda \circ \mathrm{Ad}_{a}$. 
On considre les ŽlŽments $x=mk\in \Gamma_{M_Q}\bs{K}$ et $a\in \bs{Z}_Q$ qui vŽrifient la condition
$$\sigma_Q^R({\bf H}_0(ax)-T)\neq 0\ptf\leqno{(5)}$$ 
D'aprs \cite[2.11.6]{LW}, il existe une constante $C>0$ telle que pour tout $x\in \Gamma_{M_Q}\bs{K}$ 
et tout $a\in \bs{Z}_Q$ vŽrifiant (5), en posant 
$H={\bf H}_Q(a)^R \in \ag_Q^R$, on ait
$$\alpha(H) > \alpha(T)-C \quad\hbox{pour tout $\alpha \in \Delta_Q^R$}\ptf$$ 
On en dŽduit que pour de tels ŽlŽments $x$ et $a$, les ŽlŽments $a^{-1}u'' a$ pour $u''\in \bs{\mathfrak{S}}_{U_Q^R}$ restent dans un compact fixŽ de 
$U_Q^R(\mbb{A})$. Par consŽquent $y= a^{-1}u'' a mk$ varie dans un compact fixŽ de $G(\mbb{A})$, disons 
$\Gamma_{G}$\footnote{Observons que $\Gamma_G$ ne dŽpend pas de la fonction $f$; il dŽpend 
bien sžr de $G$, $Q$ et $R$ mais aussi de $\bs{K}$, $\Gamma_{M_Q}$ et $\Siegel_{U_Q^R}$.}. 
Cela a pour consŽquence que la somme 
sur $\eta$ dans la dŽfinition de $\Xi_Q^R(ay)$ porte sur un sous-ensemble fini de $\UF^{M_Q}$ 
qui peut tre choisi indŽpendemment de $y\in \Gamma_{G}$, 
disons $\mathfrak{E}$. 
Puisque $F$ est un corps de fonctions, pour tout $(y,\eta)\in \Gamma_{G}\times \mathfrak{E}$, 
la fontion $\Lambda \mapsto g(y,\Lambda, \eta)$ sur $\mathfrak{u}_Q^*(\mbb{A})$ 
est ˆ support compact comme transformŽe de Fourier d'une fonction lisse et ˆ support compact 
sur $\mathfrak{u}_Q(\mbb{A})$.
D'aprs la fin de la preuve ce \cite[9.1.1]{LW}, si $\bs{d}_0(T)$ est assez grand, 
pour $\Lambda\in  \nn^\vee(Q,R)$ l'ŽlŽment 
$\mathrm{Ad}_{a^{-1}}^*(\Lambda)$ sort du support de $g(y,\cdot,\eta)$ pour tout $(y,\eta)\in \Gamma_{G}\times
 \mathfrak{E}$\footnote{C'est ici qu'intervient le groupe $\bs{K}'$ dans la remarque \ref{formule exacte, support et niveau}\,(ii): fixŽ $\bs{K}'$, il existe un sous-groupe 
 ouvert compact $\Gamma_{\mathfrak{u}_Q}$ de $\mathfrak{u}_Q(\mbb{A})$ telle que pour toute fonction 
 $f\in C_\mathrm{c}(G(\mbb{A})/\!\!/\bs{K}')$ et tout $(y,\eta)\in \Gamma_G\times \mathfrak{E}$, 
 la fonction $X \mapsto f(y^{-1} \eta j(X) y)$ sur $\mathfrak{u}_Q(\mbb{A})$ soit invariante par $\Gamma_{\mathfrak{u}_Q}$, ce qui a pour consŽquence que 
 le support de sa transformŽe de Fourier $\Lambda \mapsto g(y,\Lambda ,\eta)$ est contenu dans le sous-ensemble ouvert compact 
 $\{\Lambda \in \mathfrak{u}_Q^*(\mbb{A})\,\vert \, \psi(\langle \Lambda, X \rangle)=1,\, \forall X \in \Gamma_{\mathfrak{u}_Q}\}$ de $\mathfrak{u}_Q^*(\mbb{A})$.}. 
 Le terme (1) associŽ ˆ la paire $(Q, R)$ avec $P_0\subset Q\subsetneq R \subset G$ est donc nul, 
 ce qui achve la preuve de la proposition. 
\end{proof}

\section{DŽcomposition suivant les $F$-strates unipotentes}\label{dŽcomposition suivant les strates}

Les hypothses sont celles de la section \ref{la distribution J}, ˆ l'exception des sous-sections \ref{la stratŽgie d'Arthur} 
et \ref{la variante de Hoffmann} o $F$ est un corps de nombres. Les principaux rŽsultats de cette section sont regroupŽs en \ref{les rŽsultats} 
et leurs dŽmonstrations sont donnŽes dans les sous-sections suivantes.  

\subsection{La stratŽgie d'Arthur}\label{la stratŽgie d'Arthur}
On suppose dans cette sous-section que $F$ est un corps de nombres. 
On sait que la variŽtŽ unipotente $\UU$ est rŽunion d'un nombre fini de $G$-orbites $\ESO$\footnote{D'aprs Lusztig \cite{L1}, 
cela reste vrai sur un corps de fonctions mais nous n'utiliserons pas ce rŽsultat ici.}. 
Pour une $G$-orbite $\ESO\subset \UU$, supposŽe dŽfinie sur $F$\footnote{PrŽcisŽment, Arthur considre la fermeture de Zariski 
$\overline{\UU(F)}$ de $\UU(F)$ dans $G$, qui est une variŽtŽ  
en gŽnŽral plus petite que $\UU$ (\textit{e.g.} si $G$ est $F$-anisotrope). De plus, il regroupe les orbites unipotentes gŽomŽtriques en 
$\mathrm{Gal}(\overline{F}/F)$-orbites et traite 
toutes ces $\mathrm{Gal}(\overline{F}/F)$-orbites. Parmi ces dernires, on a les orbites unipotentes gŽomŽtriques dŽfinies sur $F$, 
\cad celles qui sont stables sous l'action de 
$\mathrm{Gal}(\overline{F}/F)$. Les orbites gŽomŽtrique $\ESO$ qui nous intŽressent sont celles qui possdent un point $F$-rationnel, 
\cad telles que $\ESO\cap G(F)\neq \emptyset$; elles sont toutes dŽfinies sur $F$.}, 
notons $\overline{\ESO}$ la fermeture de Zariski de $\ESO$ dans $G$ (\textit{i.e.} dans $\UU$). La stratŽgie d'Arthur \cite{A2} pour estimer l'intŽgrale 
$$\int_{\overline{\bsX}_G} \bs\Lambda^Tk_{\ESO} (x) \dd x \quad \hbox{avec} \quad k_{\ESO}(x) = \sum_{\eta\in \ESO(F)}f(x^{-1}\eta y) $$ 
consiste ˆ remplacer la fonction $f\in C^\infty_{\rm c}(G(\mbb{A}))$ par une famille de fonctions dont le support est de plus en plus proche du sous-ensemble 
$\overline{\ESO}(\mbb{A})$ de $G(\mbb{A})$, puis 
par passage ˆ la limite d'en dŽduire un polyn™me en $T$, disons $J^T_{\overline{\ESO}}(f)$, qui approxime l'expression 
$$\int_{\overline{\bsX}_G} \bs\Lambda^Tk_{\overline{\ESO}}(x) \dd x \quad \hbox{avec}\quad k_{\overline{\ESO}}(x)= 
\sum_{\{\ESO' \,\vert \, \ESO' \subset \overline{\ESO}\}}k_{\ESO'}(x)\ptf $$ 
Par rŽcurrence sur le nombre de $G$-orbites contenues dans $\overline{\ESO}$, il en dŽduit un polyn™me en $T$, disons $J^T_{\ESO}(f)$, 
qui approxime l'expression 
$$\int_{\overline{\bsX}_G} \bs\Lambda^Tk_{\ESO}(x) \dd x $$ et (par construction) vŽrifie la dŽcomposition 
$$J^T_{\overline{\ESO}}(f)= \sum_{\{\ESO' \,\vert \, \ESO' \subset \overline{\ESO}\}} J^T_{\ESO}(f)\ptf$$

\subsection{La variante de Hoffmann}\label{la variante de Hoffmann} Continuons avec les hypothses et les notations de \ref{la stratŽgie d'Arthur}. 
La variante de Hoffmann consiste ˆ dŽfinir, pour chaque $P\in \ES{P}_\mathrm{st}$, 
une fonction 
$$k_{P,\ESO}(x)= \sum_{\{\ESO'\,\vert \, I_{P}^{G}(\ESO')=\ESO\}}\sum_{\eta'\in \ESO'(F)} \int_{U_P(\mbb{A})} f(x^{-1}\eta' u x) \dd u $$ 
o $\ESO'$ parcourt les orbites gŽomŽtriques de  
$\UU^{M_P}$ telles que $\ESO$ soit l'induite parabolique $I_{P}^{G}(\ESO')$ de $\ESO'$ au sens de 
de Lusztig-Spaltenstein \cite{LS} (voir \ref{induction parabolique des F-strates}).  
Cette fonction $k_{P,\ES{O}}(x)$ n'est autre que l'analogue de $k_{P,\u}(x)$ o la somme sur $\UU^{M_P}_F$ 
a ŽtŽ remplacŽe par une somme sur les points $F$-rationnels 
des orbites gŽomŽtriques $\ESO'$ de $\UU^{M_P}$ telles que $I_{P}^{G}(\ESO')=\ESO$. 
On peut dans l'expression pour $K_{P,\ESO}(x,y)$ remplacer la double 
somme par une simple somme $\sum_{\eta \in (\ES{O}\cap P)(F)U_P(F)/U_P(F)}$. 
On pose ensuite, pour $T\in \ag_0$, $$k_{\ESO}^T(x)=\sum_{P\in \ESP_\mathrm{st}}(-1)^{a_P-a_G} \sum_{\xi \in P(F)\backslash G(F)} 
\wh{\tau}_P({\bf H}_0(\xi x)-T)k_{P,\ESO}(\xi x)\ptf$$ On a donc l'ŽgalitŽ 
$$k^T_\u(x) = \sum_{\ESO} k^T_{\ESO}(x)$$ o $\ESO$ parcourt les orbites gŽomŽtriques de $\UU^G$ qui possdent un point $F$-rationnel. 
Pour $T$ assez rŽgulier, on en dŽduit au moins formellement l'ŽgalitŽ 
$$\mathfrak{J}^T_\u(f)= \sum_{\ESO} \int_{\overline{\bs{X}}_G}k^T_{\ESO}(x)\dd x\ptf$$ 
Hoffmann \cite{Ho} a conjecturŽ que cette expression est absolument convergente: 
$$\sum_{\ESO} \int_{\overline{\bs{X}}_G}\vert k^T_{\ESO}(x)\vert \dd x<+\infty \ptf$$ 
Cette conjecture a ŽtŽ dŽmontrŽe par Finis et Lapid \cite{FL} pour des fonctions test bien plus gŽnŽrales que les fonctions localement constantes ˆ support compact 
(voir aussi \ref{les Žtapes du dŽveloppement}).   
PrŽcisŽment, si $T\in \ag_0$ est assez rŽgulier, 
pour chaque $G$-orbite $\ESO\subset \UU=\UU^G$ qui possde un point $F$-rationnel, on a (\textit{loc.~cit.}): 
\begin{itemize}
\item l'intŽgrale $\int_{\overline{\bs{X}}_G}k^T_{\ESO}(x)\dd x$ est absolument convergente; 
\item l'application $T \mapsto  \int_{\overline{\bs{X}}_G}k^T_{\ESO}(x)\dd x$ est un polyn™me en $T$ qui approxime l'intŽgrale 
$\int_{\overline{\bsX}_G} \bs\Lambda^Tk_{\ESO} (x) \dd x$, autrement dit le polyn™me $J^T_{\ESO}(f)$ 
associŽ par Arthur ˆ l'orbite $\ESO$ est donnŽ par l'intŽgrale 
$$J^T_{\ESO}(f)= \int_{\overline{\bs{X}}_G}k^T_{\ESO}(x)\dd x\ptf$$ 
\end{itemize}
C'est l'analogue de ces rŽsultats pour les corps de fonctions que nous allons Žtablir ici, en remplaant les $G$-orbites $\ESO\subset \UU$ qui possŽdent un 
point $F$-rationnel par les strates de $\UU$ qui intersectent non trivialement $\UF$. 

\subsection{\'EnoncŽ des rŽsultats}\label{les rŽsultats} Ë partir de maintenant, les hypothses sont celles de la section \ref{la distribution J}; 
en particulier $F$ est un corps global de caractŽristique $p>1$. 

On commence par reprendre en les adaptant les dŽfinitions de \ref{la variante de Hoffmann}. 
Rappelons qu'on a notŽ $[\UU]_F$ l'ensemble des $F$-strates de $\UU=\UU^G$ qui intersectent non trivialement $\UU_F= \UF^G$, et $[\UU_F]$ l'ensemble des 
$F$-strates de $\UF$. 
Puisque l'hypothse \ref{hypo bonnes F-strates faible} est vŽrifiŽe pour $G$ (\ref{hypo bonnes F-strates faible pour Lie(G) et G}), d'aprs \ref{descente 3 sous HBFS}\,(iii), 
l'application\index{YbsF@ $\bsfrY_F$}\footnote{La notation est cohŽrente avec celle de \ref{notation allŽgŽe}\,(ii): si 
$E/F$ est une extension sŽparable telle que $(E^{\mathrm{s\acute{e}p}})^{\mathrm{Aut}_F(E^{\mathrm{s\acute{e}p}})}=F$, pour toute strate $\bsfrY\in [\UU]_F\;(\subset [\UU]_E)$, on a 
$\bsfrY_E = (\bsfrY_F)_E$.} $$[\UU]_F \rightarrow [\UF],\, \bsfrY \mapsto \bsfrY_F\bydef \UF \cap \UU$$ est bijective. Rappelons aussi qu'on a notŽ $(\UU/G)_F$ l'ensemble 
des $G$-orbites $\ES{O}\subset \UU$ qui intersectent non trivialement $\UF$; et que pour $\bsfrY\in [\UU]_F$, on a notŽ $(\bsfrY/G)_F\subset (\UU/G)_F$ le sous-ensemble formŽ des $G$-orbites $\ES{O}\subset \bsfrY$ qui intersectent non trivialement $\bsfrY_F$. Pour 
$\ES{O}\in (\UU/G)_F$, on pose 
$$\ES{O}_F \bydef \UF \cap \ES{O}\ptf$$

Pour $\ES{O}\in (\UU/G)_F$ 
et $P\in \ES{P}_\mathrm{st}$, on dŽfinit une fonction\index{KPOx@$K_{P,\ES{O}}(x)$} 
$$k_{P,\ES{O}}(x)= \sum_{\eta \in (\ES{O}_F \cap P(F))U_P(F)/U_P(F)}  \int_{U_P(\mbb{A})} f(x^{-1}\eta u x) \dd u \ptf $$ 
En notant $\xi_{P,\ES{O}}$ la fonction sur $\UF\cap P(F)=\UF^{M_P}U_P(F)$ dŽfinie par\index{xiPOeta@$\xi_{P,\ES{O}}(\eta)$} $$\xi_{P,\ES{O}}(\eta)= \left\{\begin{array}{ll}
1 & \hbox{si $\ES{J}_P^{G}(\eta)\subset \ES{O}$}\\
0 & \hbox{sinon}
\end{array}
\right.,$$ on a donc 
$$k_{P,\ES{O}}(x) = \sum_{\eta \in \UF^{M_P}}\xi_{P,\ES{O}}(\eta)\int_{U_P(\mbb{A})} f(x^{-1}\eta u x) \dd u \ptf$$ 
Observons que si la $G$-orbite $\ES{O}$ est \textit{$F$-rigide} (\ref{orbites et strates F-rigides}), alors $k_{P,\ES{O}}=0$ pour tout $P\neq G$. 
Puisque $$\sum_{\ES{O} \in (\UU/G)_F} \xi_{P,\ES{O}}(\eta)= 1\quad \hbox{pour tout} \quad \eta \in \UF^{M_P}\vg $$ 
on a l'ŽgalitŽ
$$k_{P,\u}(x)= \sum_{\ES{O} \in (\UU/G)_F}k_{P,\ES{O}}(x)\ptf$$ Pour $P=G$, la fonction $k_{\ES{O}}= k_{G,\ES{O}}$ sur $G(\mbb{A})$ est donnŽe par
$$k_{\ES{O}}(x) = \sum_{\eta \in \ES{O}_F}f(x^{-1}\eta x)\ptf$$ 
Pour $T\in \ag_0$, on pose\index{kOTx@$k_{\ES{O}}^T(x)$} 
$$k_{\ES{O}}^T(x)=\sum_{P\in \ESP_\mathrm{st}}(-1)^{a_P-a_G} \sum_{\xi \in P(F)\backslash G(F)} 
\wh{\tau}_P({\bf H}_0(\xi x)-T)k_{P,\ES{O}}(\xi x)\ptf$$ On a l'ŽgalitŽ $$k^T_\u(x) = \sum_{\ES{O} \in (\UU/G)_F} k^T_{\ES{O}}(x)\ptf$$

\begin{theorem}\label{CV}
Si $T\in \ag_0$ est assez rŽgulier, \cad si $\bs{d}_0(T)\geq c_f$ pour une constante $c_f>0$ dŽpendant de $f$, on a 
$$\sum_{\ES{O} \in (\UU/G)_F} \int_{\overline{\bs{X}}_{\!G}}\vert k^T_{\ES{O}}(x) \vert \dd x < + \infty\ptf$$
\end{theorem}

Puisque l'ensemble $(\UU/G)_F$ est fini, le thŽorme \ref{CV} est une consŽquence de la 

\begin{proposition}\label{CV orbite}
Soit $\ES{O}\in (\UU/G)_F$. Il existe une constante $c_0>0$ telle que 
pour tout $\epsilon >0$, il existe des constantes $c >0$ et $\epsilon' >0$ telles que pour 
tout $T\in \ag_0$ tel que $\bs{d}_0(T) \geq c_0$ et $\bs{d}_0(T) \geq \epsilon \| T \|$, on ait  
$$\int_{\overline{\bs{X}}_{\!G}} \left| F_{P_0}^G(x,T) k_{\ES{O}}(x) - k_{\ES{O}}^T(x) \right| \dd x \leq c q^{- \epsilon' \| T \|}\ptf$$ 
\end{proposition}

\begin{remark}\label{dŽpendance support et niveau}
\textup{
Les constantes $c_0>0$ et $c>0$ dŽpendent de $f$ mais on peut prŽciser cette dŽpendance comme dans la remarque \ref{formule exacte, support et niveau}\,(ii). 
Fixons un sous-groupe ouvert compact $\bs{K}'$ de $G(\mbb{A})$ et un sous-ensemble ouvert compact $\Omega$ de $G(\mbb{A})$ tel 
que $\bs{K}'\Omega \bs{K}' = \Omega$. Alors en remplaant dans l'ŽnoncŽ \guill{il existe une constante $c>0$} par \guill{il existe une constante 
$c'>0$}  et en posant $c = c' \| f \|$ avec $\| f \| = \sup_{x\in G(\mbb{A})} \vert f(x)\vert $, l'inŽgalitŽ de la proposition \ref{CV orbite} 
est vraie pour toute fonction $f\in C(\Omega/\!\!/ \bs{K}')$.}
\end{remark}

On a aussi: 
\begin{proposition}\label{PolExp}Soit $\ES{O}\in (\UU/G)_F$.
Il existe une fonction\index{JObsTf@$\mathfrak{J}^T_{\ES{O}}(f)$} $$T \mapsto \mathfrak{J}^T_{\ES{O}}(f)$$ dans \textup{PolExp} telle que 
pour $T\in \ag_{0,\mbb{Q}}$ assez rŽgulier (prŽcisŽment, tel que $\bs{d}_0(T) \geq c_f$ pour une constante $c_f>0$ dŽpendant de $f$), on ait 
$$\mathfrak{J}^T_{\ES{O}}(f)= \int_{\overline{\bs{X}}_{\!G}} k^T_{\ES{O}}(x) \dd x \ptf$$ 
\end{proposition}

\begin{remark}\label{annulation sur l'orbite adŽlique}
\textup{Puisque la fonction $T\mapsto \mathfrak{J}^T_{\ES{O}}(f)$ est dans PolExp, l'intŽgrale $\int_{\overline{\bs{X}}_{\!G}} k^T_{\ES{O}}(x) \dd x$ 
(qui est absolument convergente si $T$ est assez rŽgulier, 
d'aprs \ref{CV}) la caractŽrise de manire unique. D'aprs \ref{CV orbite}, la distribution $f \mapsto \mathfrak{J}^T_{\ES{O}}(f)$ sur $G(\mbb{A})$ annule toute 
fonction $f\in C^\infty_\mathrm{c}(G(\mbb{A}))$ qui s'annulle sur l'ensemble\footnote{On renvoie ˆ \ref{F_S-strates et A-strates} pour la dŽfinition de $\UU_{\mbb{A}}$.}
$$G(\mbb{A})\bullet \ES{O}_F=\{gug^{-1} \,\vert \, g\in G(\mbb{A}),\,u\in \ES{O}_F\}\;(\subset \ES{O}_{\mbb{A}}= \UU_{\mbb{A}}\cap \ES{O}(\mbb{A}))\ptf$$
}\end{remark}
%

\begin{corollary}\label{raffinement suivant les G-orbites}
On a l'ŽgalitŽ dans \textup{PolExp} 
$$\mathfrak{J}_\mathrm{unip}^T(f) = \sum_{\ES{O} \in (\UU/G)_F}\mathfrak{J}^T_{\ES{O}}(f)\ptf$$
\end{corollary}

D'aprs \ref{CV orbite}, la fonction $T \mapsto \mathfrak{J}^T_{\ES{O}}(f)$ de \ref{PolExp} est asymptotique ˆ l'intŽgrale 
$$\int_{\overline{\bs{X}}_{\!G}}F_{P_0}^G(x,T) k_{\ES{O}}(x)\dd x = \int_{\overline{\bs{X}}_{\!G}}\bs{\Lambda}^T k_{\ES{O}}(x)\dd x\ptf$$ 
Si la $G$-orbite $\ES{O}\in (\ES{O}/G)_F$ est $F$-rigide (\ref{orbites et strates F-rigides}), 
la fonction $T\mapsto \mathfrak{J}^T_{\ES{O}}(f)$ 
ne dŽpend pas de $T$ et on peut noter $J_{\bsfrY}(f)$ sa valeur constante; on a  
$$J_{\bsfrY}(f) = \int_{\overline{\bs{X}}_{\!G}}\left(\sum_{\eta \in \bsfrY} f(x^{-1}\eta x) \right)\!\dd x $$ et 
l'intŽgrale est absolument convergente. 

\begin{remark}
\textup{
ConsidŽrons la $G$-orbite triviale strate $\bsfrY= \{1\}$. Elle est rigide, donc \textit{a fortiori} $F$-rigide,  
et la constante $J_{\{1\}}(f)$ est donnŽe par 
$$J_{\{1\}}(f)= \mathrm{vol}(\overline{\bs{X}}_{\!G})f(1)\ptf$$ Rappelons que $F_{P_0}^G(\cdot ,T)$ est la fonction caractŽristique de 
$G(F)\bs{\mathfrak{S}}_{P_0}^G(T_G,T)$ et que l'on a choisi $\bs{\mathfrak{S}}_{P_0}^G(T_G,T)$ 
de telle manire que $F_{P_0}^G(\cdot,T)$ soit invariante par $A_G(\mbb{A})$. Notons 
$\smash{\overline{\bs{\mathfrak{S}}}}_{P_0}^G(T_G,T)$ l'image de $\bs{\mathfrak{S}}_{P_0}^G(T_G,T)$ dans $\overline{\bs{X}}_{\!G}$. Alors 
$$ \int_{\overline{\bs{X}}_G} \bs{\Lambda}^T k_{\{1\}}(x) \dd x = \int_{\overline{\bs{X}}_G} F^G_{P_0}(x,T) f(1)\dd x = 
\mathrm{vol}(\smash{\overline{\bs{\mathfrak{S}}}}_{P_0}^G(T_G,T)) f(1) \ptf$$
}\end{remark}

La dŽmonstration de la proposition \ref{CV orbite} occupe les sous-sections \ref{qqs lemmes utiles} ˆ \ref{preuve de CV strate}; une esquisse des diffŽrentes Žtapes de 
cette dŽmonstration est donnŽe en \ref{les Žtapes}. La proposition \ref{PolExp} est dŽmontrŽe en \ref{preuve de PolExp}. 

\subsection{DŽcomposition suivant les $F$-strates de $\UF$}\label{decomp suivant les F-strates} On verra dans un prochain travail que le 
raffinement \ref{raffinement suivant les G-orbites} sui\-vant les $G$-orbites $\ES{O}\in (\UU/G)_F$ se prte mal ˆ une dŽcomposition en somme de produits d'objets locaux. Il convient pour cela de regrouper les orbites $\ES{O}\in (\UU/G)_F$ suivant les strates $\bsfrY\in [\UU]_F$. 

Pour $\bsfrY\in [\UU]_F$, on note $\mathfrak{J}^T_{\bsfrY}(f)$ l'ŽlŽment de PolExp dŽfini par 
$$\mathfrak{J}_{\bsfrY}^T(f)= \sum_{\ES{O}\in (\bsfrY/G)_F} \mathfrak{J}_{\ES{O}}^T(f)\ptf$$ 
Pour $T\in \mathfrak{a}_0$ assez rŽgulier (prŽcisŽment, tel que $\bs{d}_0(T)\geq c_f$ pour un constante $c_f >0$ dŽpendant de $f$), on a 
$$\mathfrak{J}_{\bsfrY}^T(f)=  \int_{\overline{\bs{X}}_{\!G}} k^T_{\bsfrY}(x) \dd x \quad \hbox{avec} \quad k^T_{\bsfrY}(x)= \sum_{\ES{O}\in (\bsfrY/G)_F} k^T_{\ES{O}}(x)\ptf$$ 
D'aprs \ref{CV orbite}, l'expression $\mathfrak{J}^T_{\bsfrY}(f)$ est asymptotique ˆ l'intŽgrale 
$$\int_{\overline{\bs{X}}_{\!G}}  F_{P_0}^G(x,T) k_{\bsfrY}(x)\dd x \quad \hbox{avec} \quad k_{\bsfrY}(x)= \sum_{\eta \in \bsfrY} f(x^{-1}\eta x)\ptf$$
D'aprs \ref{raffinement suivant les G-orbites}, on a l'ŽgalitŽ dans PolExp
$$\mathfrak{J}^T_\u(f) = \sum_{\bsfrY \in [\UU]_F} \mathfrak{J}^T_{\bsfrY}(f)\ptf$$

\begin{remark}\label{annulation sur la strate adŽlique}
\textup{Soit $\bsfrY \in [\UU]_F$. Comme en \ref{annulation sur l'orbite adŽlique}, la distribution $f \mapsto \mathfrak{J}^T_{\bsfrY}(f)$ sur $G(\mbb{A})$ annule toute 
fonction $f\in C^\infty_\mathrm{c}(G(\mbb{A}))$ qui s'annulle sur l'ensemble 
$$G(\mbb{A})\bullet \bsfrY_F=\{gug^{-1} \,\vert \, g\in G(\mbb{A}),\,u\in \bsfrY_F\}\;(\subset \bsfrY_{\mbb{A}}= \UU_{\mbb{A}} \cap \bsfrY(\mbb{A}))\ptf$$ 
Puisque $\UU_{\mbb{A}}\cap \ES{O}(\bsfrY)(\mbb{A})$ est dense dans $\bsfrY_{\mbb{A}}$ pour la topologie adŽlique, on en dŽduit que 
la distribution $f \mapsto \mathfrak{J}^T_{\bsfrY}(f)$ sur $G(\mbb{A})$ annule toute fonction $f\in C^\infty_\mathrm{c}(G(\mbb{A}))$ qui s'annule 
sur l'ensemble $\UU_{\mbb{A}}\cap \ES{O}(\bsfrY)(\mbb{A})$. }
\end{remark}

\subsection{Les Žtapes de la dŽmonstration de \ref{CV orbite}}\label{les Žtapes} 
La preuve de \ref{CV orbite} est longue et laborieuse. Les premires rŽductions sont les mmes 
que celles du dŽbut de la preuve de \ref{raffinement}, raffinŽes comme dans celle de \cite[theorem 3.1]{A2}: 
il suffit en effet de remplacer l'ensemble $\UF$ de tous les ŽlŽments unipotents de $G(F)$ par $\ES{O}_F$. 
La diffŽrence avec la preuve de \ref{raffinement} est qu'ici, on ne passe pas via la formule de Poisson 
ˆ l'annulateur $\nn^\vee$ de $\mathfrak{u}_Q(F)$ dans $\mathfrak{u}_Q^*(\mbb{A})$: les somme rationnelles sont des intŽgrales de fonctions ˆ 
support compact mais qui en gŽnŽral ne sont pas lisses (voir la remarque \ref{no poisson}). L'ingrŽdient essentiel sera ici encore la formule de Poisson mais on l'utilisera plus loin 
(prŽcisŽment, dans le lemme \ref{la majoration clŽ}), 
aprs une sŽrie de manipulations et d'estimŽes assez peu intuitives, qui reprennent les arguments de \cite{CL} et \cite{C}. 

On dŽcrit dans cette sous-section \ref{les Žtapes} les principales Žtapes de preuve de \ref{CV orbite}. On procde par rŽductions successives, comme dans la preuve 
de \ref{raffinement}. 

La $G$-orbite $\ES{O}\in (\UU/G)_F$ Žtant fixŽe, l'intŽgrale $$\int_{\overline{\bs{X}}_{\!G}}\vert k^T_{\ES{O}}(x) \vert \dd x$$ est bornŽe par une somme, 
indexŽe par les paires $Q\subset R$ de sous-groupes paraboliques standard, d'expressions  obtenues en remplaant dans \ref{CV et rŽcriture de l'intŽgrale tronquŽe}\,(1) le terme $k_{P,\u}(x)$ par le terme $k_{P,\ES{O}}(x)$. Pour $Q=R=G$, l'expression est Žgale ˆ l'intŽgrale (absolument convergente) 
$$\int_{\overline{\bs{X}}_{\!G}}  F_{P_0}^G(x,T)\left| k_{\ES{O}}(x) \right| \dd x;$$ 
et pour $Q=R \neq G$, l'expression est nulle. On est donc ramenŽ ˆ majorer, pour les paires $Q\subsetneq R$, l'intŽgrale  
$$\int_{\bs{Y}_Q}F_{P_0}^Q(x,T)\sigma_Q^R({\bf H}_0(x)-T) \left| k_{Q,\ES{O}}^R(x)  \right|\dd x\leqno{(1)}$$ 
avec $$k_{Q,\ES{O}}^R(x)= \sum_{\substack{P\in \ESP_\mathrm{st}\\Q \subset P \subset R}}(-1)^{a_P-a_G} 
k_{P,\ES{O}}(x)\ptf$$

\begin{remark}\label{no poisson}
\textup{Comme dans la preuve de \ref{raffinement}, l'expression $k_{Q,\ES{O}}^R(x)$ est Žgale ˆ
$$\sum_{\substack{P\in \ESP_\mathrm{st}\\Q \subset P \subset R}}(-1)^{a_P-a_G}\sum_{\eta\in \UF^{M_Q}}
\left(\sum_{\nu \in U_Q(F)} \xi_{P,\ES{O}}(\eta\nu) \int_{\Siegel_{U_P}} f(x^{-1}\eta\nu u x) \dd u\right)\ptf$$ 
Pour $P\in \ESP_{\mathrm{st}}$ et $v\in \vert \ES{V}\vert$, posons $\ES{O}_v = \ES{O}\times_F F_v\;(\in (\UU/G)_{F_v})$. On dŽfinit comme plus haut une fonction 
$\xi_{P,\ES{O}_v}$ sur $\UU_{F_v}\cap P(F_v)= \UU^{M_P}_{F_v}U(P_v)$. Soit $\xi_{P,\ES{O}\times_F\mbb{A}}$ 
la fonction sur $\UU_{\mbb{A}} \cap P(\mbb{A})=\UU_{\mbb{A}}^{M_P}U_P(\mbb{A})$ dŽfinie par $$\xi_{P,\ES{O}\times_F{\mbb{A}}}(\gamma) =
 \prod_{v\in \vert \ES{V}\vert} \xi_{P,\ES{O}_{v}}(\gamma_v)\quad \hbox{pour} \quad \gamma = \prod_v \gamma_v\ptf$$ 
La fonction $\xi_{P,\ES{O}\times_F\mbb{A}}$ ainsi dŽfinie prolonge la fonction $\xi_{P,\ES{O}}$ sur $\UF \cap P(F)$. Elle est invariante par translations ˆ droite et ˆ gauche  
par $U_P(\mbb{A})$, et aussi par conjugaison dans $P(\mbb{A})$.  
La fonction $$X \mapsto \xi_{P,\ES{O}\times_F\mbb{A}}(\eta j(X)) \int_{\Siegel_{U_P}}f(x^{-1}\eta j(X)u x)\dd u$$ sur $\mathfrak{u}_Q(\mbb{A})$ est ˆ support compact 
mais elle n'est en gŽnŽral pas lisse 
(rappelons que $j: \mathfrak{u}_Q \rightarrow U_Q$ est un $F$-isomorphisme de variŽtŽs compatible ˆ l'action de $A_Q$). 
On peut bien sžr dŽfinir sa transformŽe de Fourier, pour $\Lambda \in \mathfrak{u}_Q^*(\mbb{A})$: 
$$g_{P,\ES{O}}(x,\Lambda,\eta)= \int_{\mathfrak{u}_Q(\mbb{A})} \psi(\langle \Lambda , X \rangle )\xi_{P,\ES{O}\times_F\mbb{A}}(\eta j(X))
 \left( \int_{\Siegel_{U_P}}f(x^{-1}\eta j(X)ux )\dd u\right) \dd X\ptf $$ 
Pour $\Lambda \in \mathfrak{n}^\vee$, l'intŽgrale en $u\in \Siegel_{U_P}$ est absorbŽe par l'intŽgrale en $X$: on a 
$$g_{P,\ES{O}}(x,\Lambda,\eta)=
\int_{\mathfrak{u}_Q(\mbb{A})}\psi(\langle \Lambda,X\rangle) \xi_{P,\ES{O}\times_F\mbb{A}}(\eta j(X))f(x^{-1}\eta j(X)x)\dd X$$ avec 
$$g_{P,\ES{O}}(x,\Lambda,\eta)= 0 \quad \hbox{si} \quad \Lambda\vert_{\mathfrak{u}_P(\mbb{A})}\neq 0\ptf$$ 
La fonction $\Lambda \mapsto g_{P,\ES{O}}(x, \Lambda, \eta)$ sur $\mathfrak{u}_Q^*(\mbb{A})$ est lisse mais contrairement ˆ la fonction $g$ 
de la preuve de \ref{raffinement}, elle n'est en gŽnŽral pas ˆ support compact. On ne peut donc pas lui appliquer la formule de Poisson. 
De plus la prŽsence de la fonction $\xi_{P,\ES{O}\times_F\mbb{A}}$ perturbe la somme alternŽe sur 
les $P$.
}
\end{remark}

Pour obtenir les estimŽes nŽcessaires ˆ la suite de la dŽmonstration, on commence par raffiner l'intŽgrale ˆ majorer (1): 
gr‰ce ˆ la partition \cite[3.6.3]{LW} appliquŽe au sous-groupe parabolique $Q$, on se ramne ˆ majorer, pour $Q_1\in \ESP^Q(P_0)$ et $T_1\in \ag_0$ assez rŽgulier, l'intŽgrale 
$$ \int_{\bs{Y}_{Q_1}} F_{P_0}^{Q_1}(x,T_1) \tau_{Q_1}^Q({\bf H}_0(x)-T_1)F_{P_0}^Q({\bf H}_0(x)-T) \sigma_Q^R({\bf H}_0(x)-T)\vert k_{Q,\ES{O}}^R(x)\vert \dd x \ptf\leqno{(2)}$$ 
Rappelons que $$\bs{Y}_{Q_1}= A_G(\mbb{A})Q_1(F) \backslash G(\mbb{A})\quad\hbox{et}\quad \bs{Z}_{Q_1}= A_G(\mbb{A})A_{Q_1}(F)\backslash A_{Q_1}(\mbb{A})\ptf$$ 
En utilisant la dŽcomposition d'Iwasawa $G(\mbb{A})= U_{Q_1}(\mbb{A})M_{Q_1}(\mbb{A})\bs{K}$ 
et en remplaant l'intŽgrale sur $A_G(\mbb{A})M_{Q_1}(F)\backslash M_{Q_1}(\mbb{A})$ par une intŽgrale sur $\bs{Z}_{Q_1} \times  \Siegel_{M_{Q_1}}^*$ o 
$\Siegel_{M_{Q_1}}^*$ est un ensemble de Siegel pour le quotient $\overline{\bs{X}}_{\!M_{Q_1}}$, on obtient que l'intŽgrale (2) 
est essentiellement majorŽe\footnote{C'est-ˆ-dire qu'il existe 
une constante $c>0$ telle que $(2) \leq c\,(3) $.} 
par une intŽgrale du type $$\int_{\bs{Z}_{Q_1}(T,T_1)} \bs{\delta}_{Q_1}(a)^{-1} \vert k_{Q,\ES{O}}^R(ax) \vert \dd a\quad \hbox{pour}\quad x\in \Gamma_G\vg \leqno{(3)}$$ 
o $\bs{Z}_Q(T,T_1)$ est un sous-ensemble (dŽfini explicitement) 
de $$\bs{Z}_{Q_1}^{R,+} = \{a\in \bs{Z}_{Q_1}\,\vert \, \langle \alpha , {\bf H}_{Q_1}(a) \rangle >0,\, \forall \alpha \in \Delta_{Q_1}^R \}$$ et $\Gamma_G$ est 
un sous-ensemble compact de $G(\mbb{A})$ indŽpendant de $a$. Ainsi pour estimer l'intŽgrale (3), il faut commencer par majorer l'expression 
$\bs{\delta}_{Q_1}(a)^{-1} \vert k_{Q,\ES{O}}^R(ax) \vert$. 

On prouve --- c'est la majoration fondamentale de la dŽmonstration de \ref{CV orbite} --- que pour chaque racine $\alpha \in \Delta_{Q_1}^R\smallsetminus \Delta_Q^R$, 
il existe une constante $c >0$ 
telle que pour tout $a\in \bs{Z}_{Q_1}^{R,+}$ et tout $x\in \Gamma_G$, on ait 
$$\bs{\delta}_{Q_1}(a)^{-1}\vert k_{Q,\ES{O}}^R(ax) \vert \leq c q^{- \langle \alpha , {\bf H}_{Q_1}(a) \rangle}\ptf\leqno{(4)}$$ 
Cela entra"ne que l'intŽgrale (3) est essentiellement majorŽe par 
$$ \int_{\bs{Z}_{Q_1}(T,T_1)} \prod_{\alpha \in \Delta_{Q_1}^R \smallsetminus \Delta_{Q_1}^Q} q^{-r \langle \alpha, {\bf H}_{Q_1}(a) \rangle} \dd a\quad\hbox{avec} \quad r^{-1} = 
\vert \Delta_{Q_1}^R \smallsetminus \Delta_{Q_1}^Q \vert\ptf
\leqno{(5)}$$ On en dŽduit comme dans \cite{CL} et \cite{C} qu'il existe une constante $c_0>0$ telle que pour tout $\epsilon >0$, il existe des constantes $c_1 >0$ et $\epsilon'_1 >0$ telles que pour 
tout $T\in \ag_0$ tel que $\bs{d}_0(T)>c_0$ et $\bs{d}_0(T) > \epsilon \| T \|$, l'intŽgrale (5) soit majorŽe par $c_1 q^{-\epsilon'_1 \| T \|}$. Compte-tenu des rŽductions successives 
effectuŽes prŽcŽdemment, cela prouve \ref{CV orbite}. 

Les diffŽrentes Žtapes de la dŽmonstration de \ref{CV orbite} dŽcrite ci-dessus sont organisŽes comme suit: 
\begin{itemize}
\item en \ref{qqs lemmes utiles}, on prouve des rŽsultats ŽlŽmentaires basŽs sur la formule de Poisson;
\item en \ref{majoration de sommes rationnelles}, le triple $Q_1\subset Q\subsetneq R$ et un sous-ensemble compact 
$\Gamma_G$ de $G(\mbb{A})$ Žtant fixŽs, on prouve la majoration (4);
\item en \ref{estimŽe d'une intŽgrale}, on dŽfinit l'ensemble $Z_{Q_1}(T,T_1)$ et on majore l'intŽgrale (3); 
\item en \ref{preuve de CV strate}, on achve la dŽmonstration de \ref{CV orbite}.
\end{itemize}

\subsection{Quelques lemmes utiles}\label{qqs lemmes utiles} 
Les rŽsultats contenus dans cette sous-section seront utilisŽs dans la preuve de \ref{la majoration clŽ}. 

Commenons par un rŽsultat ŽlŽmentaire:

\begin{lemme}\label{poisson utile}
Soit $V$ un espace vectoriel sur $F$ de dimension finie $n$. 
Posons $V_{\mbb{A}}= V\otimes_F \mbb{A}$ et notons $(a,v) \mapsto a\cdot v$ l'action de $\mbb{A}^\times $ sur $V_{\mbb{A}}$ par homothŽties. 
Soit $\phi\in C^\infty_\mathrm{c}(V_\mbb{A})$ avec $\phi \geq 0$. 
\begin{enumerate}
\item[(i)] Il existe une constante $c>0$ telle que pour tout $a\in \mbb{A}^\times$ avec $\vert a  \vert_{\mbb{A}}>1$, on ait 
$$\sum_{v\in V} \phi (a\cdot v)  \leq c \quad \hbox{et} \quad \sum_{v \in  V} \vert a \vert_{\mbb{A}}^{-n} \phi(a^{-1}\cdot v) \leq c\ptf$$
\item[(ii)] Pour tout entier $k\geq 1$, il existe une constante $c>0$ telle que pour tout $a\in \mbb{A}^\times$ avec $\vert a  \vert_{\mbb{A}}>1$, 
on ait $$\sum_{v\in V\smallsetminus \{0\}} \phi(a\cdot v) \leq c \vert a \vert_{\mbb{A}}^{-k}\ptf$$
\end{enumerate}
\end{lemme}

\begin{proof} 
Fixons une base $e_1\ldots , e_n$ de $V$ sur $F$ et pour $v\in V_{\mbb{A}}$, Žcrivons $v= \sum_{i=1}^n e_i\otimes v_i$ avec $v_i\in \mbb{A}$. 
On note $\wh{\phi}\in C^\infty_\mathrm{c}(V_{\mbb{A}})$ la transformŽe de Fourier de $\phi$ dŽfinie par 
$$\wh{\phi}(v')= \int_{V_{\mbb{A}}} \psi(\textstyle{\sum_{i=1}^n} v_i v'_i) \phi(v)\dd v$$ o $\psi$ est un caractre additif non trivial de $\mbb{A}$ trivial sur $F$ 
et $dv$ est la mesure de Haar sur $V_{\mbb{A}}$ telle que $\mathrm{vol}(V\backslash V_{\mbb{A}})=1$.

La premire assertion du point (i) est claire et la seconde se dŽduit de la premire par la formule de Poisson: 
pour tout $a\in \mbb{A}^\times$, on a $$\sum_{v\in V} \vert a \vert_{\mbb{A}}^{-n} \phi(a^{-1} \cdot v) = \sum_{v \in V}\wh{\phi}(a\cdot v)\ptf$$ 
Quant au point (ii), il suffit d'appliquer la premire assertion du point (i) ˆ la fonction $\phi_k$ dŽfinie par 
$\phi_k(v) = \| v \|^k \phi(v)$ avec $\| v \| = \textstyle{\sum_{i=1}^n }\vert v_i \vert_{\mbb{A}}$. 
\end{proof}

Le lemme suivant est la version \guill{corps de fonctions} de \cite[3.2]{A2} (pour $P=P_0$), qui est lui-mme une version ŽdulcorŽe de 
\cite[12.2.1]{LW}\footnote{Ce rŽsultat est citŽ un peu rapidement pour les corps de fonctions en \cite[12.2.1]{LL}, raison pour laquelle nous dŽtaillons 
la preuve de \ref{Arthur [A2,3.2]}. Observons que l'ingrŽdient principal de cette preuve n'est autre que la seconde inŽgalitŽ de \ref{poisson utile}\,(i).}.

\begin{lemme}\label{Arthur [A2,3.2]} 
Soient $f\in C^\infty_\mathrm{c}(G(\mbb{A}))$, $P\in \ESP_\mathrm{st}$ et $X \in \ag_P$. Il existe une constante 
$c>0$ telle que pour tout $a\in A_P(\mbb{A})$ avec $\tau_P({\bf H}_P(a) - X) =1$, on ait 
$$\bs{\delta}_{P}(a)^{-1}\sum_{\gamma \in G(F)} \vert f(a^{-1}\gamma a) \vert \leq c\ptf$$   
\end{lemme}

\begin{remark}
\textup{Observons que seule la restriction $f^1$ de $f$ ˆ $G(\mbb{A})^1$ joue un r™le pertinent dans ce rŽsultat. 
La dŽpendance ˆ l'Žgard de $f^1$ est la mme que pour \ref{CV orbite} (\textit{cf.} \ref{dŽpendance support et niveau}): fixŽs un sous-groupe ouvert compact 
$\bs{K}'$ de $G(\mbb{A})$ et un sous-ensemble ouvert compact $\Omega$ de $G(\mbb{A})^1$ tel que $\bs{K}'\Omega \bs{K}'= \Omega$,  
il existe une constante $c'>0$ telle que, en posant $c= c' \| f^1\|$, l'inŽgalitŽ \ref{Arthur [A2,3.2]} est vraie pour toute fonction $f$ telle que 
$f^1\in C(\Omega/\!\!/ \bs{K}')$.}
\end{remark}

\begin{proof} 
Pour $Y\in \ag_0^P$ tel que $\langle \alpha , Y \rangle <0$ pour tout $\alpha \in \Delta_0^P$, on a 
$$\tau_{P_0}({\bf H}_0(a) -(X+Y)) =   \tau_P({\bf H}_P(a)-X) \quad \hbox{pour tout} \quad 
a\in A_P(\mbb{A})\ptf$$ Puisque d'autre part $\bs{\delta}_{P_0}(a)= \bs{\delta}_P(a)$ pour tout $a\in A_P(\mbb{A})$, 
il suffit de dŽmontrer le lemme pour $P=P_0$. 

On suppose donc $P=P_0$ et on procde comme dans la preuve de \cite[12.2.1]{LW}. Notons $\Omega$ le support de la restriction de $f$ ˆ $G(\mbb{A})^1$. 
Soient $a\in A_0(\mbb{A})$ et $\gamma \in G(F)$ tels que $$\tau_{P_0}({\bf H}_0(a)-X)=1\quad \hbox{et}\quad a^{-1}\gamma a \in \Omega \ptf$$
D'aprs la partition \cite[3.6.3]{LW}, si $T_1\in \ag_0$ est assez rŽgulier (ne dŽpendant que de $G$), 
il existe un unique $Q\in \ESP_\mathrm{st}$ tel que $$F_{P_0}^Q(a,T_1) \tau_Q({\bf H}_0(a)-T_1)=1\ptf$$ 
Soit $\mathfrak{B}_Q^G\subset A_Q(\mbb{A})$ l'image d'une section du morphisme composŽ 
$$A_Q(\mbb{A})\rightarrow A_Q(\mbb{A})/A_G(\mbb{A})\rightarrow \ES{B}_Q^G\ptf$$ la condition $F_{P_0}^Q(a,T)=1$ 
implique que la projection $H^Q$ de $H={\bf H}_0(a)$ sur $\ag_0^Q$ reste dans 
un compact de $\ag_0^Q$. On en dŽduit qu'il existe un compact $\omega$ de $A_0(\mbb{A})$ tel que $a$ appartienne ˆ 
$\mathfrak{B}_Q^G  \omega A_G(\mbb{A})$. \'Ecrivons $a=b a'$ avec $b\in \mathfrak{B}_Q^0$ et $a'\in \omega A_G(\mbb{A})$. 
La condition $a^{-1}\gamma a\in \Omega$ Žquivaut ˆ $b^{-1} \gamma b\in \Omega'= \bigcup_{a' \in \omega} a' \Omega a'^{-1}$. 
D'autre part la condition $b^{-1}\gamma b \in \Omega'$ Žquivant ˆ 
$\gamma b\in b\Omega'$. On a donc $\tau_Q({\bf H}_0(\gamma b)-T_2)=1$ pour un $T_2\in T_1 + {\bf H}_0(\Omega')$. Pour $T_1$ assez rŽgulier, 
le lemme \cite[3.6.1]{LW} assure que $\gamma$ appartient ˆ $Q(F)$. On Žcrit $\gamma = \delta \eta $ avec $\delta\in M_Q(F)$ et 
$\eta\in U_Q(F)$. Pour $b\in \mathfrak{B}_Q^G$, on a $b^{-1} \gamma b = \delta b^{-1}\eta b$ et la condition $b^{-1}\gamma b\in \Omega'$ assure que $\delta$ 
reste dans un compact de $M_Q(F)$. On est donc ramenŽ ˆ Žvaluer le nombre de $\eta\in U_Q(F)$ tels que $b^{-1}\eta b$ appartienne ˆ un sous-ensemble 
ouvert compact, disons $C$, de $U_Q(\mbb{A})$. Fixons un $F$-isomorphisme de variŽtŽs $j: \mathfrak{u}_Q \rightarrow U_Q$ compatible ˆ l'action de 
$A_Q$. Notons $\phi\in C^\infty_\mathrm{c}(\mathfrak{u}_Q(\mbb{A}))$ la fonction caractŽristique de $j^{-1}(C)$. 
On veut Žvaluer l'expression $\sum_{Y \in \mathfrak{u}_Q(F)} \phi(b^{-1}Y b)$. 
ConsidŽrons la transformŽe de Fourier $\wh{\phi} \in C^\infty_\mathrm{c}(\mathfrak{u}_Q^*(\mbb{A}))$ de $\phi$, dŽfinie par 
$$\wh{\phi}(\Lambda) = \int_{\mathfrak{u}_Q(\mbb{A})} \psi(\langle \Lambda , Y \rangle) \phi( Y) \dd Y\quad \hbox{pour} \quad \Lambda\in \mathfrak{u}_Q^*(\mbb{A}) \ptf$$ 
En notant $\nn^\vee$ l'annulateur de $\mathfrak{u}_Q(F)$ dans $\mathfrak{u}_Q^*(\mbb{A})$, la formule de Poisson donne 
$$\sum_{Y\in \mathfrak{u}_Q(F)}\phi(\mathrm{Ad}_{b^{-1}}(Y)) = \bs{\delta}_{Q}(b) \sum_{\Lambda \in \nn^\vee}\wh{\phi}(\mathrm{Ad}_{b^{-1}}^*(\Lambda))\quad \hbox{avec} 
\quad \mathrm{Ad}_{b^{-1}}^*(\Lambda)= \Lambda \circ \mathrm{Ad}_b\ptf$$ Rappelons que $b= aa'^{-1}$ avec $a'\in \omega A_G(\mbb{A})$ et $\tau_{P_0}({\bf H}_0(a)-X)=1$. 
Ces conditions impliquent que: 
\begin{itemize}
\item en dehors d'un nombre fini de $b$ (dŽpendant de $X$), l'automorphisme $\mathrm{Ad}_{b^{-1}}^*$ de 
$\mathfrak{u}_Q^*(\mbb{A})$ dilate $\mathfrak{u}_Q^*(\mbb{A})$, par consŽquent 
la somme sur $\Lambda \in \nn^\vee$ est bornŽe par une constante indŽpendante de $b$; 
\item $\bs{\delta}_Q(b)$ est essentiellement majorŽ par $\bs{\delta}_{P_0}(a)$, \cad qu'il existe une constante $c_1$ (dŽpendant de $\omega$) telle que 
$\bs{\delta}_P(b) \leq c_1 \bs{\delta}_{P_0}(a)$.
\end{itemize}
Cela achve la dŽmonstration du lemme. 
\end{proof}

\subsection{Majoration de sommes rationnelles}\label{majoration de sommes rationnelles} 
Dans cette sous-section, on fixe une $G$-orbite $\ES{O}\in (\UU/G)_F$. On fixe aussi deux ŽlŽments $Q\subset R$ de $\ESP_\mathrm{st}$ tels que $Q\neq  R$ et un ŽlŽment 
$Q_1\in \ESF^{Q}(P_0)$. On a donc $$P_0\subset Q_1\subset Q \subsetneq R \subset G\ptf$$ 
On pose $\bs{Z}_{Q_1}= A_G(\mbb{A})A_{Q_1}(F)\backslash A_{Q_1}(\mbb{A})$ et 
$$\bs{Z}_{Q_1}^{R,+}= \{a \in \bs{Z}_{Q_1},\, \vert \, \langle \alpha , {\bf H}_{Q_1}(a) \rangle > 0 \quad 
\hbox{pour toute racine} \quad \alpha \in \Delta_{Q_1}^R\}\ptf$$ 

Pour $P\in \ESP_\mathrm{st}$ avec $Q\subset P \subset R$, on pose $$A_{P,\ES{O}}(g)=
\sum_{\gamma \in \UF^{M_P}\cap Q(F)}\xi_{P,\ES{O}}(\gamma)
\int_{U_P(\mbb{A})} f(g^{-1}\gamma ug) \dd u\ptf$$ 
Soit\footnote{L'expression $\wt{k}_{Q,\ES{O}}^R(g)$ est une variante de $k_{Q,\ES{O}}^R(g)$: pour chaque $P\in \ES{P}_{\mathrm{st}}$ tel que $Q\subset P \subset R$, la somme sur 
$\gamma \in \mathfrak{U}_F^{M_P}$ a ŽtŽ remplacŽe par une somme sur $\gamma \in \mathfrak{U}_F^{M_P} \cap Q(F)$. }  $$\wt{k}_{Q,\ES{O}}^R(g) = \sum_{\substack{P\in \ESP_\mathrm{st}\\ Q\subset P \subset R}}\epsilon_P A_{P,\ES{O}}(g)\quad \hbox{avec}\quad \epsilon_P= (-1)^{a_P-a_G}\ptf\leqno{(1)}$$  

L'objectif est ici d'Žtablir, pour chaque racine $\alpha \in \Delta_{Q_1}^R \smallsetminus \Delta_{Q_1}^Q$, une majoration uniforme 
de l'expression $\bs{\delta}_{Q_1}(a)^{-1}\vert \wt{k}_{Q,\ES{O}}^R(ax) \vert$, valable pour tout $a\in \bs{Z}_{Q_1}^{R,+}$ et tout $x$ 
dans un sous-ensemble compact $\Gamma_G$ de $G(\mbb{A})$ fixŽ 
(\textit{cf.} \ref{la majoration clŽ bis} pour un rŽsultat prŽcis). On adapte pour cela le travail de Chaudouard \cite[3.7-3.12]{C}. 

Soit $\alpha \in \Delta_{Q_1}^R \smallsetminus \Delta_{Q_1}^Q$ et soit 
$S\in \ESP_\mathrm{st}$ le $F$-sous-groupe parabolique propre maximal de $G$ contenant $Q_1$ dŽfini par la condition 
$$\Delta_{Q_1}^S = \Delta_{Q_1}\smallsetminus \{\alpha\}\ptf$$ 
Puisque $\alpha$ n'appartient pas ˆ l'ensemble des racines de $A_{Q_1}$ dans $M_Q$, on a aussi l'inclusion $Q\subset S$. L'application 
$$P \mapsto P \cap M_S\leqno{(2)}$$ induit une application surjective de $\ESF^R(Q)$ sur l'ensemble $\ESF^{R\cap M_S}(Q\cap M_S)$ des $F$-sous-groupes paraboliques 
de $M_S$ contenus entre $Q\cap M_S$ et $R\cap M_S$. Soit $\wt{P}$ un $F$-sous-groupe parabolique de $M_S$. Il a exactement deux antŽcŽdents par l'application (2): 
l'un, notŽ $P$, est contenu dans $S$ tandis que l'autre, notŽ $P'$, ne l'est pas. De plus on a 
$$P= S \cap P'\ptf$$ Pour un tel $\wt{P}$ correspondant au couple $(P,P')$, on pose 
$$A_{\wt{P},\ES{O}}(g)= \epsilon_P A_{P,\ES{O}}(g)+ \epsilon_{P'} A_{P'\!,\ES{O}}(g)\ptf $$ 
D'aprs (1), on a $$\wt{k}_{Q,\ES{O}}^R(g) = \sum_{\wt{P}}A_{\wt{P},\ES{O}}(g)\leqno{(3)}$$ o $\wt{P}$ parcourt les ŽlŽments de $ \ESF^{R\cap M_S}(Q\cap M_S)$. 

Fixons un ŽlŽment $\wt{P} \in  \ESF^{R\cap M_S}(Q\cap M_S)$. On se propose dans un premier temps de rŽcrire l'expression 
$A_{\wt{P},\ES{O}}(g)$. Rappelons que $j_0: \mathfrak{u}_0 \rightarrow U_0$ est un $F$-isomorphisme de variŽtŽs compatible ˆ l'action de $A_0$. Il induit par restriction un 
$F$-isomorphisme de variŽtŽs $$j=j_P^{P'}: \mathfrak{u}_P^{P'}=\mathrm{Lie}(U_P^{P'})\rightarrow U_P^{P'}$$ lui aussi compatible ˆ l'action de 
$A_0$; avec (rappel) $U_P^{P'}= U_P \cap M_{P'}$. 
L'expression $A_{P'\!,\ES{O}}(g)$ est Žgale ˆ $$\sum_{\gamma \in \UF^{M_P}\cap Q(F)}\sum_{X\in \mathfrak{u}_P^{P'}(F)} \xi_{P'\!,\ES{O}}(\gamma j(X))
\int_{U_{P'}(\mbb{A})} f(g^{-1}\gamma j(X)u'g) \dd u'\ptf \leqno{(4)}$$ 
Notons $\xi_{\wt{P},\ES{O}}$ la fonction sur $P(F)$ dŽfinie par 
$$\xi_{\wt{P},\ES{O}}= \epsilon_P \xi_{P,\ES{O}}+ \epsilon_{P'}\xi_{P'\!,\ES{O}}\vert_{P(F)} $$ 
et $A_{P'\!,\ES{O}}^1(g)$, resp $B_{\wt{P},\ES{O}}(g)$, l'expression obtenue en remplaant la fonction 
$\xi_{P'\!,\ES{O}}$ par $\xi_{P,\ES{O}}$, resp. $\xi_{\wt{P},\ES{O}}$, dans 
l'expression (4). Puisque $\epsilon_P+\epsilon_{P'}=0$ et $\epsilon_{P'}\epsilon_{P'}=1$, on a $$A_{P'\!,\ES{O}}(g)= 
A_{P'\!,\ES{O}}^1(g) + \epsilon_{P'} B_{\wt{P},\ES{O}}(g)\ptf\leqno{(5)}$$ 
On note  $\nn^\vee$ l'annulateur de $\mathfrak{u}_P^{P'}(F)$ dans $\mathfrak{u}_P^{P'\!,*}(\mbb{A})$. 
La formule de Poisson relative ˆ la somme sur $\mathfrak{u}_P^{P'}(F)$ donne 
$$\sum_{X\in \mathfrak{u}_P^{P'}(F)} \int_{U_{P'}(\mbb{A})} f(g^{-1}\gamma j(X)u'g) \dd u' = 
\sum_{\Lambda \in \nn^\vee}\int_{U_P(\mbb{A})} \psi(\Lambda ,u ) f(g^{-1}\gamma u g)\dd u$$ 
o l'on a posŽ $$\Psi(\Lambda , u) = \Psi (\langle \Lambda , X_P^{P'}\rangle) \quad \hbox{si} \quad u=j(X)\ptf$$ 
Puisque la fonction $\xi_{P,\ES{O}}$ sur $\UF^G\cap P(F)=\UF^{M_P}U_P(F)$ est invariante ˆ droite par $U_P(F)$, on obtient 
$$A_{P'\!,\ES{O}}^1(g)= A_{P,\ES{O}}(g) + B_{P,\ES{O}}(g)\leqno{(6)}$$ o l'on a posŽ  
$$B_{P,\ES{O}}(g)=\sum_{\gamma \in \UF^{M_P}\cap Q(F)}\xi_{P,\ES{O}}(\gamma)
\sum_{\Lambda \in \nn^\vee \smallsetminus \{0\}}\int_{U_P(\mbb{A})} \psi(\Lambda , u) f(g^{-1}\gamma ug) \dd u\ptf$$ 
Les dŽcompositions (5) et (6) donnent 
$$A_{\wt{P},\ES{O}}(g)=  \epsilon_{P'}B_{P,\ES{O}}(g) + B_{\wt{P},\ES{O}}(g)= -\epsilon_PB_{P,\ES{O}}(g) + B_{\wt{P},\ES{O}}(g)\ptf\leqno{(7)}$$ 
Traitons maintenant l'expression $B_{\wt{B},\ES{O}}(g)$. 
Pour $\gamma \in \UU^{M_P}$, posons $$\ES{J}_{P}^{P'}(\gamma)= \ES{J}_{P\cap M_{P'}}^{M_{P'}}(\gamma)\ptf$$

\begin{lemme}\label{lemme nŽgation strate induite}
Soit $\gamma\in \UF^{M_P}$. 
\begin{enumerate}
\item[(i)] Pour $\nu \in U_P^{P'}(F)$ tel que $\gamma\nu \in \ES{J}_{P}^{P'}(\gamma)$, on a $\xi_{\wt{P},\ES{O}}(\gamma \nu)=0$. 
\item[(ii)] Il existe un $\nu\in U_P^{P'}(F)$ tel que $\gamma \nu \in \ES{J}_{P}^{P'}(\gamma)$. 
\end{enumerate}
\end{lemme}

\begin{proof}
Soit $\nu \in U_P^{P'}(F)$. Si $\gamma \nu \in \ES{J}_{P}^{P'}(\gamma)$, alors d'aprs la propriŽtŽ de transitivitŽ \cite[1.7]{LS}, on a 
$\ES{J}_{P'}^{G}(\gamma \nu )= \ES{J}_{P}^{G}(\gamma)$ et donc $\xi_{P,\ES{O}}(\gamma) = \xi_{P'\!,\ES{O}}(\gamma\nu)$. D'o le point (i) puisque $\epsilon_P + \epsilon_{P'}=0$. 

Quant au point (ii), puisque $\ES{J}_P^{P'}(\gamma)\cap \gamma U_P^{P'}$ est ouvert dense dans 
$ \gamma U_P^{P'}$ et que $U_P^{P'}(F)$ est (Zariski-)dense dans 
$U_P^{P'}=U_P^{P'}(\overline{F})$, l'intersection $\ES{J}_P^{P'}(\gamma) \cap \gamma U_P^{P'}(F)$ est dense dans $\gamma U_P^{P'}$; en particulier elle est non vide. 
\end{proof}

Le lemme \ref{lemme nŽgation strate induite}\,(i) implique que l'expression $B_{\wt{P},\bsfrY}(a)$ est Žgale ˆ 
$$\sum_{\gamma \in \UF^{M_{P}}\cap Q(F)}\sum_{\substack{X\in \mathfrak{u}_P^{P'}(F)\\ \gamma j(X)\notin \ES{J}_{P}^{P'}(\gamma)}} \xi_{\wt{P},\bsfrY}(\gamma j(X))
\int_{U_{P'}(\mbb{A})} f(a^{-1}\gamma j(X)u'a) \dd u'\ptf \leqno{(8)}$$ 

\begin{lemme}\label{la majoration clŽ} 
Soit $\wt{P} \in  \ESF^{R\cap M_S}(Q\cap M_S)$. 
\begin{enumerate}
\item[(i)] Pour tout entier $k\geq 1$, il existe une constante $c=c_k>0$ telle que pour tout $a\in \bs{Z}_{Q_1}^{R,+}$, on ait 
$$\bs{\delta}_{Q_1}(a)^{-1}\vert B_{P,\ES{O}}(a) \vert \leq c q^{- k\langle \alpha , {\bf H}_{Q_1}(a) \rangle}\ptf$$
\item[(ii)] Il existe une constante $\tilde{c} >0$ telle que pour tout $a\in \bs{Z}_{Q_1}^{R,+}$, on ait 
$$\bs{\delta}_{Q_1}(a)^{-1}\vert B_{\wt{P},\ES{O}}(a) \vert \leq \tilde{c}  q^{- \langle \alpha , {\bf H}_{Q_1}(a) \rangle}\ptf$$ 
\end{enumerate}
\end{lemme}

\begin{proof} Prouvons (i). Posons $$f_{P'}(g) = \int_{U_{P'}(\mbb{A})}f(gu')\dd u'\ptf$$ 
L'expression $B_{P,\ES{O}}(a)$ est Žgale ˆ 
$$ \sum_{\gamma \in \UF^{M_P} \cap Q(F)}\xi_{P,\ES{O}}(\gamma) 
\sum_{\Lambda \in \nn^\vee\smallsetminus \{0\}} \int_{\mathfrak{u}_P^{P'}(\mbb{A})}\psi(\langle \Lambda , X \rangle ) \bs{\delta}_{P'}(a) f_{P'}(a^{-1}\gamma j(X) a ) \dd X$$ 
soit encore ˆ 
$$ \sum_{\gamma \in \UF^{M_P} \cap Q(F)}\!\!\!\xi_{P,\ES{O}}(\gamma)\!\!\! 
\sum_{\Lambda \in \nn^\vee\smallsetminus \{0\}} \int_{\mathfrak{u}_P^{P'}(\mbb{A})}\psi(\langle \Lambda, \mathrm{Ad}_{a}(X) \rangle ) 
\bs{\delta}_{P}(a) f_{P'}(a^{-1}\gamma a j(X) ) \dd X\ptf$$ 
Pour $\Lambda\in \mathfrak{u}_{P}^{P'\!,*}(\mbb{A})$ et $m\in M_P(\mbb{A})$, posons 
$$g(\Lambda, m)= \int_{\mathfrak{u}_P^{P'}(\mbb{A})}\psi(\langle \Lambda , X \rangle ) f_{P'}(m j(X) ) \dd X$$ 
La fonction $\Lambda \mapsto g(\cdot, m)$, comme transformŽe de Fourier d'une fonction lisse ˆ support compact, est elle aussi lisse ˆ support compact. 
On en dŽduit qu'il existe des fonctions positives $\phi\in C^\infty_\mathrm{c}(\mathfrak{u}_P^{P'\!,*}(\mbb{A}))$ et $h\in C^\infty_\mathrm{c}(M_P(\mbb{A}))$ telles que 
$$\vert g(\Lambda , m) \vert \leq \phi(\Lambda) h(m) \quad \hbox{pour tout} \quad (\Lambda,m)\in \mathfrak{u}_P^{P'\!,*}(\mbb{A})\times M_P(\mbb{A})\ptf$$ 
L'expression $\bs{\delta}_{Q_1}(a)^{-1}\vert B_{P,\ES{O}}(a)\vert $ est donc majorŽe par le produit des expressions 
$$\bs{\delta}_{Q_1}(a)^{-1}\bs{\delta}_P(a) \sum_{\gamma \in \UF^{M_P}} h(a^{-1}\gamma a)\leqno{(9)}$$ 
et $$\sum_{\Lambda \in \nn^\vee \smallsetminus \{0\}} \phi(\mathrm{Ad}_{a^{-1}}^*(\Lambda))\ptf \leqno{(10)}$$ 
Pour $a\in \bs{Z}_{Q_1}$, on a $$\bs{\delta}_{Q_1}(a)^{-1}\bs{\delta}_P(a) = \bs{\delta}_{P_0}(a)^{-1} \bs{\delta}_P(a)= \bs{\delta}_{P_0 \cap M_P}(a)^{-1}\vg$$ et si de plus $a$ appartient 
ˆ $\bs{Z}_{Q_1}^{R,+}$, alors $\langle \beta ,{\bf H}_{Q_1}(a)) >0$ pour toute racine $\beta\in \Delta_{Q_1}^P$; en d'autres termes 
$\tau_{Q_1}^P({\bf H}_{Q_1}(a))>0$. On en dŽduit (\ref{Arthur [A2,3.2]}) que l'expression (9) est majorŽe par une constante 
$c >0$ indŽpendante de $a\in \bs{Z}_{Q_1}^{R,+}$. 

Quant ˆ l'expression (10), notons $\Phi=\ES{R}_{Q_1}^{U_P^{P'}}\;(\subset \ES{R}_{Q_1}^+)$ l'ensemble des racines de $A_{Q_1}$ dans $U_{P}^{P'}= U_P \cap M_{P'}$. 
Pour $\theta \in \Phi$, notons $\mathfrak{u}_\theta$ 
le sous-espace de $\mathfrak{u}_{P}^{P'}= \mathrm{Lie}(U_P^{P'})$ associŽ ˆ $\theta$ et $\mathfrak{n}_\theta^\vee$ l'annulateur de $\mathfrak{u}_\theta(F)$ 
dans $\mathfrak{u}_\theta^*(\mbb{A})$. Pour $\Lambda \in \mathfrak{u}_P^{P'\!,*}(\mbb{A})$, Žcrivons 
$\Lambda = \sum_{\theta \in \Phi}\Lambda_\theta$ avec $\Lambda_\theta = \Lambda\vert_{\mathfrak{u}_\theta(\mbb{A})}\in \mathfrak{u}_\theta^*(\mbb{A})$. On a donc 
$$\mathrm{Ad}_{a^{-1}}^*(\Lambda)= \sum_{\theta \in \Phi} \theta(a)  \Lambda_\theta \quad \hbox{pour tout} \quad a\in \bs{Z}_{Q_1}\ptf$$ De plus il existe 
des fonctions $\phi_\theta \in C^\infty_\mathrm{c}(\mathfrak{u}_\theta^*(\mbb{A}))$ telles que $$\phi(\Lambda) \leq \prod_{\theta \in \Phi}\phi_\theta(\Lambda_\theta) \quad 
\hbox{pour tout} \quad \Lambda\in \mathfrak{u}_P^{P'\!,*}(\mbb{A})\ptf$$
On en dŽduit que l'expression (10) est majorŽe par le produit sur les $\theta\in \Phi$ des expressions 
$$\left( \sum_{\Lambda \in \nn_\theta^\vee\smallsetminus \{0\}}\phi_\theta\left(\theta(a)\Lambda_\theta\right)\right)\cdot \left( \prod_{\mu \in \Phi\smallsetminus \{\theta\}} 
\sum_{\Lambda_\mu \in \nn_\mu^\vee} \phi_\mu\left(\mu(a)\Lambda_\mu\right)\right)\!\ptf\leqno{(11)}$$ 
Puisque pour $\mu \in \Phi$ et $a \in \bs{Z}_{Q_1}^{R,+}$, on a $\vert \mu(a) \vert_{\mbb{A}} = q^{\langle \theta, {\bf H}_{Q_1}(a) \rangle} >1$, d'aprs 
\ref{poisson utile}\,(ii), le terme de droite dans l'expression (11) est majorŽ par une constante indŽpendante de $a\in \bs{Z}_{Q_1}^{R,+}$; et d'aprs \ref{poisson utile}\,(ii), 
pour tout entier $k\geq 1$, le terme de gauche dans l'expression (11) est majorŽ par $c q^{-k\langle \theta,{\bf H}_{Q_1}(a) \rangle}$ pour une constante $c>0$ elle aussi 
indŽpendante de $a\in \bs{Z}_{Q_1}^{R,+}$. Comme $\theta$ s'Žcrit $\theta = \sum_{\beta \in \Delta_{Q_1}^{P'}} n_\beta \beta$ avec $n_\beta \geq 0$ et 
$n_\alpha >0$, on a $$q^{- k \langle \theta, {\bf H}_{Q_1}(a)\rangle} \leq q^{-k\langle \alpha , {\bf H}_{Q_1}(a) \rangle }\quad \hbox{pour tout}\quad a \in \bs{Z}_{Q_1}^{R,+}\ptf$$ 
D'o le point (i) puisque $$q^{- \vert  \Phi \vert k \langle \alpha, {\bf H}_{Q_1}(a) \rangle} \leq q^{-k\langle \alpha,{\bf H}_{Q_1}(a) \rangle}\ptf $$

\vskip1mm 
Prouvons (ii). L'expression $B_{\wt{P},\ES{O}}(a)$ est Žgale ˆ 
$$\sum_{\gamma \in \UF^{M_P}\cap Q(F)}\sum_{\substack{X\in \mathfrak{u}_P^{P'}(F)\\ \gamma j(X)\notin \ES{J}_{P}^{P'}(\gamma)}} \xi_{\wt{P},\ES{O}}(\gamma j(X))
\bs{\delta}_{P'}(a) f_{P'}(a^{-1}\gamma j(X)a)\ptf \leqno{(12)}$$ 
Il existe des fonctions positives $h\in C^\infty_\mathrm{c}(M_P(\mbb{A}))$ et  $\phi\in C^\infty_\mathrm{c}(\mathfrak{u}_P^{P'\!}(\mbb{A}))$ telles que 
$$\vert f_{P'}(mj(X)) \vert \leq h(m) \phi(X) \quad \hbox{pour tout} \quad (m,X)\in M_P(\mbb{A})\times \mathfrak{u}_P^{P'}(\mbb{A})\ptf$$ 
L'expression $\bs{\delta}_{Q_1}(a)^{-1}\vert B_{\wt{P},\ES{O}}(a)\vert$ est donc majorŽe par l'expression 
$$\bs{\delta}_{Q_1}(a)^{-1}\bs{\delta}_P(a) \sum_{\gamma \in \UF^{M_P}} h(a^{-1}\gamma a)\Phi(a,\gamma)\leqno{(13)}$$ avec  
$$\Phi(a,\gamma) = \bs{\delta}_P(a)^{-1}\bs{\delta}_{P'}(a)\sum_{\substack{X\in \mathfrak{u}_P^{P'}(F)\\ \gamma j(X)\notin \ES{J}_{P}^{P'}(\gamma)}} \phi(\mathrm{Ad}_{a^{-1}}(X)) \ptf \leqno{(14)}$$ 
On a vu plus haut que l'expression (9) --- obtenue en prenant $\Phi \equiv 1$ dans l'expression (13) ---  
est majorŽe par une constante indŽpendante de $a\in \bs{Z}_{Q_1}^{R,+}$. Il suffit donc pour prouver (ii) de montrer que l'expression (14) est majorŽe par 
$c q^{-\langle \alpha ,{\bf H}_{Q_1}(a) \rangle} $ pour une constante $c>0$ indŽpendante de 
$a\in \bs{Z}_{Q_1}^{R,+}$. Posons $M=M_P$ et $M'=M_{P'}$. La somme $\sum_{\gamma \in \UU_F^{M}}$ dans (13) est Žgale ˆ la double somme 
$\sum_{\ES{O}^M\in (\UF^M/M)_F} \sum_{\gamma \in \UF^M \cap \ES{O}^M}$ 
o on a notŽ $(\UU^M/M)_F$ l'ensemble (fini) des $M$-orbites dans $\UU^M$ qui intersectent non trivialement $\UU_F^M$. On peut donc fixer $\ES{O}^M\in (\UU^M/M)_F$ et faire varier $\gamma$ dans $\UF^M\cap \ES{O}^M$. Posons 
$\ES{O}^{M'}=\ES{J}_{P}^{P'}(\ES{O}^M)$. Puisque $\ES{O}^{M'}\cap \ES{O}^M U_P^{P'}$ est un ouvert dense dŽfini sur $F$ de la sous-variŽtŽ (localement fermŽe) 
$\ES{O}^MU_P^{P'}$ de $P\cap M'= MU_P^{P'}$, il existe un nombre fini de polyn™mes $f_1,\ldots , f_n$ dans $F[M \times \mathfrak{u}_P^{P'}]$ tels que 
pour tout $(\gamma,X)\in \ES{O}^M \times \mathfrak{u}_P^{P'}$, on ait $$\gamma j(X) \notin \ES{O}^{M'} \Leftrightarrow (\gamma,X)\in \ES{V}=\ES{V}(f_1,\ldots ,f_n)\,;$$ 
o $\ES{V}$ est le fermŽ de $M\times \mathfrak{u}_P^{P'}$ formŽ des ŽlŽments qui annulent tous les $f_i$. 

\begin{lemme}\label{non-nullitŽ des fi}
Pour tout $\gamma\in \ES{O}^M$, il existe un $i\in \{1,\ldots ,n\}$ tel que le polyn™me $f_i(\gamma,\cdot)\in F[\mathfrak{u}_P^{P'}]$ soit non nul.
\end{lemme}

\begin{proof}
Soit $\gamma \in \ES{O}^M$. 
D'aprs \ref{lemme nŽgation strate induite}\,(ii), il existe un $X\in \mathfrak{u}_P^{P'}(F)$ tel que $u=\gamma j(X)$ appartienne ˆ $\ES{O}^{M'}$, ce qui n'est possible que s'il existe un 
$i\in \{1,\ldots ,n\}$ tel que $f_i(\gamma,X)\neq 0$. 
\end{proof}

Pour $\gamma\in \ES{O}^M$, notons $\ES{V}_\gamma=\ES{V}(f_1(\gamma,\cdot),\ldots , f_n(\gamma,\cdot)) $ le fermŽ de $\mathfrak{u}_P^{P'}$ 
formŽ des ŽlŽments qui annulent tous les $f_i(\gamma,\cdot)$. Pour $X\in \mathfrak{u}_P^{P'}$, 
on a donc $(\gamma,X)\in \ES{V} \Leftrightarrow X\in \ES{V}_\gamma$. 
L'expression (14) se rŽcrit alors 
$$\Phi(a,\gamma) = \bs{\delta}_P(a)^{-1}\bs{\delta}_{P'}(a)\sum_{X\in \ES{V}_{\gamma}(F) } \phi(\mathrm{Ad}_{a^{-1}}(X)) \ptf \leqno{(15)}$$ 
Les polyn™mes $f_i(\gamma,\cdot)\in F[\mathfrak{u}_P^{P'}]$ sont non tous nuls (d'aprs \ref{non-nullitŽ des fi}) et leurs degrŽs partiels sont bornŽs par ceux des $f_i$, 
donc indŽpendemment de $\gamma$. On peut appliquer la proposition \cite[5.3.1]{CL}: l'expression (15) est 
majorŽe par $c q^{-\langle \alpha ,{\bf H}_{Q_1}(a) \rangle} $ pour une constante $c>0$ indŽpendante de 
$a\in \bs{Z}_{Q_1}^{R,+}$. Cela achve la preuve du point (ii) de \ref{la majoration clŽ}. 
\end{proof}

Compte-tenu de (7), le lemme suivant est une simple consŽquence de \ref{la majoration clŽ}:

\begin{lemme} Soit $\wt{P} \in  \ESF^{R\cap M_S}(Q\cap M_S)$. 
Il existe une constante $c>0$ telle que pour tout $a\in \bs{Z}_{Q_1}^{R,+}$, on ait 
$$\bs{\delta}_{Q_1}(a)^{-1}\vert A_{\wt{P},\ES{O}}(a) \vert \leq c q^{- \langle \alpha , {\bf H}_{Q_1}(a) \rangle}\ptf$$
\end{lemme}

On en dŽduit la majoration cherchŽe: 

\begin{proposition}\label{la majoration clŽ bis}
Soit $\Gamma_G$ un sous-ensemble compact de $G(\mbb{A})$. Il existe une constante $c>0$ 
telle que pour tout $a\in \bs{Z}_{Q_1}^{R,+}$ et tout $x\in \Gamma_G$, on ait 
$$\bs{\delta}_{Q_1}(a)^{-1}\vert \wt{k}_{Q,\ES{O}}^R(ax) \vert \leq c \prod_{\alpha \in \Delta_{Q_1}^R \smallsetminus \Delta_{Q_1}^Q} q^{- r \langle \alpha , {\bf H}_{Q_1}(a) \rangle}$$ 
avec $r^{-1}  = \vert \Delta_{Q_1}^R \smallsetminus \Delta_{Q_1}^Q \vert = \vert \Delta_Q^R \vert$. 
\end{proposition}

\begin{proof}
La fonction $f$ est invariante ˆ gauche et ˆ droite par un sous-groupe ouvert 
$\bs{K}'$ de $G(\mbb{A})$; en particulier son support $\Omega(f)$ vŽrifie $\bs{K}'\Omega(f)\bs{K}'= \Omega (f)$. 
Pour $x\in G(\mbb{A})$, posons ${^x\!f}=f \circ \mathrm{Int}_{x^{-1}}$. Puisque l'ensemble $\Gamma_{G}\bs{K}'/\bs{K}'$ est fini, l'ensemble $\{{^x\!f}\,\vert \, x \in \Gamma_G\}$ 
est lui aussi fini. Il suffit donc de prouver l'inŽgalitŽ de l'ŽnoncŽ pour $x=1$. Puisque (d'aprs (3)) 
$$\vert \wt{k}_{Q,\ES{O}}^R(g)\vert  \leq \sum_{\wt{P}} \vert A_{\wt{P},\ES{O}}(g)\vert$$ o $\wt{P}$ parcourt les ŽlŽments de $ \ESF^{R\cap M_S}(Q\cap M_S)$, 
le lemme \ref{la majoration clŽ bis} assure que pour chaque racine $\alpha \in \Delta_{Q_1}^R \smallsetminus \Delta_{Q_1}^Q$, il existe 
une constante $c_\alpha>0$ telle que pour tout $a\in \bs{Z}_{Q_1}^{R,+}$, on ait la majoration 
$$\bs{\delta}_{Q_1}(a)^{-1} \vert \wt{k}_{Q,\ES{O}}^R(a)\vert \leq c_\alpha q^{-\langle \alpha , {\bf H}_{Q_1}(a) \rangle }\ptf$$ 
Posons $c = \max \{c_\alpha\,\vert \, \alpha \in \Delta_{Q_1}^R \smallsetminus \Delta_{Q_1}^Q\}$. 
Pour $a\in \bs{Z}_{Q_1}^{R,+}$ et $\beta \in \Delta_{Q_1}^R \smallsetminus \Delta_{Q_1}^R$ tels que 
$\langle \beta , {\bf H}_{Q_1}(a) \rangle = \max \{\langle \alpha , {\bf H}_{Q_1}(a) \rangle\,\vert \, \alpha \in \Delta_{Q_1}^R\smallsetminus \Delta_{Q_1}^Q \}$, on a 
$$\bs{\delta}_{Q_1}(a)^{-1} \vert \wt{k}_{Q,\ES{O}}^R(a)\vert \leq c q^{-\langle \beta , {\bf H}_{Q_1}(a) \rangle }\ptf$$
Puisque pour $H\in \ag_{Q_1}$, 
$$ \max_{\alpha \in \Delta_{Q_1}^R\smallsetminus \Delta_{Q_1}^Q} \langle \alpha , H \rangle \geq 
\sum_{\alpha \in \Delta_{Q_1}^R \smallsetminus \Delta_{Q_1}^Q}r\langle \alpha, H \rangle \quad \hbox{avec} \quad 
r^{-1}=\vert \Delta_{Q_1}^R \smallsetminus \Delta_{Q_1}^Q \vert = \vert \Delta_Q^R \vert \vg$$ 
la proposition est dŽmontrŽe.
\end{proof}

\subsection{EstimŽe d'une intŽgrale}\label{estimŽe d'une intŽgrale} Continuons avec les hypothses de \ref{majoration de sommes rationnelles}, ˆ savoir qu'on a fixŽ la 
$F$-strate $\ES{O}$ de $\UF$ et le triple $Q_1\subset Q \subsetneq R$ d'ŽlŽments de $\ESP_\mathrm{st}$. 

Commenons par quelques dŽfinitions standard. Pour la dŽfinition des bases $\check{\Delta}_{Q_1}^R$ de 
$\ag_{Q_1}^R$ et $\Delta_{Q_1}^R$ de $\ag_{Q_1}^{R,*}$, on renvoie ˆ \cite[1.4]{LL}. On note $\hat{\Delta}_{Q_1}^R= \{\varpi_\alpha^R\,\vert \, \alpha \in \Delta_{Q_1}^R\}$ la base 
de $\ag_{Q_1}^{R,*}$ duale de $\check{\Delta}_{Q_1}^R$ dŽfinie par $\langle \varpi_\alpha^R , \check{\beta} \rangle = \delta_{\alpha,\beta}$ et $\hat{\Delta}_{Q_1}^{R,\vee}= 
\{\varpi_\alpha^\vee\,\vert \, \alpha\in  \Delta_{Q_1}^R\}$ la base 
de $\ag_{Q_1}^R$ duale de $\Delta_{Q_1}^R$ dŽfinie par $\langle \varpi_\alpha^\vee , \beta \rangle = \delta_{\alpha,\beta }$. Observons que pour $R_1\in \ESF^G(R)$, 
on a les inclusions  $$\hat{\Delta}_Q^R \subset \hat{\Delta}_{Q_1}^R\quad \hbox{et} \quad \hat{\Delta}_Q^{R,\vee} \subset \hat{\Delta}_Q^{R_1,\vee}\ptf$$ 
Observons aussi que pour $\alpha \in \Delta_Q^R \subset \Delta_Q^{R_1}$, $\varpi_\alpha^R$ est la projection de $\varpi_\alpha^{R_1}$ sur $\ag_Q^{R,*}$ (ce qui rend les 
notations cohŽrentes). Pour $R=G$, on supprimera comme il est d'usage l'exposant 
$R$ dans les notations. 

Fixons un ensemble de Siegel $\Siegel_{M_{Q_1}}^*$ pour le quotient $$\overline{\bs{X}}_{\!M_{Q_1}}= A_{Q_1}(\mbb{A})M_{Q_1}(F) \backslash M_{Q_1}(\mbb{A})\ptf$$
Rappelons  (\textit{cf.} \cite[3.4]{LL}) que $\Siegel_{M_{Q_1}}^*$ est de la forme 
$$\Siegel_{M_{Q_1}}^* = \mathfrak{C}_{M_{Q_1}}\! \Siegel_{M_{Q_1}}^1$$ o $\mathfrak{C}_{M_{Q_1}}$ est un sous-ensemble fini de $M_{Q_1}(\mbb{A})$ 
tel que $$M_{Q_1}(\mbb{A})= 
A_{Q_1}(\mbb{A})M_{Q_1}(\mbb{A})^1\mathfrak{C}_{M_{Q_1}}=A_{Q_1}(\mbb{A})\mathfrak{C}_{M_{Q_1}}M_{Q_1}(\mbb{A})^1$$ et $\Siegel_{M_{Q_1}}^1$ est 
un ensemble de Siegel pour le quotient $M_{Q_1}(F)\backslash M_{Q_1}(\mbb{A})^1$.  

\begin{lemme}\label{une condition commode}
On peut choisir $\mathfrak{C}_{M_{Q_1}}$ de telle manire que $\langle \alpha , X \rangle > 0$ pour tout $\alpha \in \Delta_{Q_1}$ et tout 
$X \in {\bf H}_{Q_1}(\Siegel_{M_{Q_1}}^*)= {\bf H}_{Q_1}(\mathfrak{C}_{M_{Q_1}})$.
\end{lemme}

\begin{proof}
L'application ${\bf H}_{Q_1}\!: M_{Q_1}(\mbb{A})\rightarrow \ag_{Q_1}$ se factorise en une application surjective 
$\overline{\bs{X}}_{\!M_{Q_1}} \rightarrow \bsbbc_{Q_1}= \ES{B}_{Q_1}\backslash \ESA_{Q_1}$. 
On peut prendre pour $\mathfrak{C}_{M_{Q_1}}$ un ensemble de relvements des ŽlŽments de $\bsbbc_{Q_1}$ dans $M_{Q_1}(\mbb{A})$. Quitte ˆ 
remplacer $\mathfrak{C}_{M_{Q_1}}$ par un translatŽ $a\mathfrak{C}_{M_{Q_1}}$ pour un ŽlŽment $a\in A_{Q_1}(\mbb{A})$ tel que 
$\inf \{\langle \alpha , {\bf H}_{Q_1}(a) \rangle \,\vert \, \alpha \in \Delta_{Q_1} \}\gg 1$, la condition du 
lemme est vŽrifiŽe.
\end{proof}

Une suppose dŽsormais que l'ensemble $\mathfrak{C}_{M_{Q_1}}$ est choisi comme en \ref{une condition commode}. 
Posons $$C= \max \{\langle \alpha , X \rangle \,\vert X\in {\bf H}_{Q_1}(\mathfrak{C}_{M_{Q_1}}),\, 
\alpha \in \Delta_{Q_1}  \}\vg$$
$$D= \max \{\langle \varpi , X \rangle \,\vert X\in {\bf H}_{Q_1}(\mathfrak{C}_{M_{Q_1}}),\,
\varpi \in \hat{\Delta}_{Q_1}  \}\ptf$$
On a donc $C>0$; et aussi $D>0$ (d'aprs \cite[1.2.8]{LW}). Rappelons qu'on a posŽ 
$$U_{Q_1}^R = U_{Q_1}\cap M_R\ptf$$

Le lemme suivant est une simple adaptation de \cite[3.6.2]{C}.

\begin{lemme}\label{des bornes explicites}
Soient $u\in U_{Q_1}^R(\mbb{A})$, $a\in A_{Q_1}(\mbb{A})$ et $m\in \Siegel_{M_{Q_1}}^*$ tels que 
$$\tau_{Q_1}^Q({\bf H}_0(am)-T_1)F_{P_0}^Q(uam,T)\sigma_Q^R({\bf H}_0(am)-T) =1 \ptf$$ 
Posons $H= {\bf H}_0(a)\in \ESB_{Q_1}$. On a:
\begin{enumerate}
\item[(i)] $\langle \alpha , H \rangle \geq \langle \alpha , T\rangle - C$ pour tout $\alpha \in \Delta_{Q_1}^R\smallsetminus \Delta_{Q_1}^Q$.
\item[(ii)] $\langle \alpha , H \rangle \geq \langle \alpha , T_1 \rangle - C$ et $\langle \varpi_\alpha^Q , H\rangle < \langle \varpi_\alpha^Q, T \rangle $ pour tout $\alpha \in \Delta_{Q_1}^Q$. 
\item[(iii)] Si $T$ et $T_1$ sont assez rŽguliers, 
prŽcisŽment si $\bs{d}_0(T)>C$ et $\bs{d}_0(T_1)>C$, 
alors $0 < \langle \alpha , H\rangle $ pour tout $\alpha \in \Delta_{Q_1}^R$ et 
$$ 0 < \langle \alpha, H \rangle <  \frac{\langle \varpi^Q_\alpha, T\rangle }{ \langle \varpi^Q_\alpha, \varpi_\alpha^\vee \rangle} 
\quad \hbox{pour tout} \quad \alpha \in \Delta_{Q_1}^Q\ptf$$ 
\item[(iv)] $\langle \alpha , H_R \rangle \leq \langle \alpha, T \rangle - \langle \alpha , H^R\rangle $ 
pour tout $\alpha \in \Delta_Q\smallsetminus \Delta_Q^R$ et $\langle \alpha , H \rangle > \langle \alpha ,T \rangle - D$ pour tout 
$\alpha \in \Delta_R$. 
\end{enumerate}
\end{lemme}

\begin{proof}

Posons $X= {\bf H}_{Q_1}(m)$. 

Prouvons (i). La condition $F_{P_0}^Q(uam,T)$ implique que $\langle \varpi , {\bf H}_0(uam) -T \rangle \leq 0$ pour tout $\varpi \in \hat{\Delta}_0^Q$. 
On obtient que $\langle \varpi ,H + X-T\rangle \leq 0$ pour tout $\varpi \in \hat{\Delta}^Q_{Q_1} \subset \hat{\Delta}^Q_0$. 
La condition $\sigma_Q^R(H+X-T)=1$ entra"ne alors comme dans la preuve de \cite[3.6.2\,(1)]{C} que $\langle \alpha, H+X-T \rangle \geq 0$ pour tout 
$\alpha \in \Delta_{Q_1}^R \smallsetminus \Delta_{Q_1}^Q$. Puisque $-\langle \alpha , X\rangle \geq -C$ pour tout 
$\alpha \in \Delta_{Q_1}$, cela prouve (i). 

Prouvons (ii). La premire inŽgalitŽ est claire, par dŽfinition de $\tau_{Q_1}^Q$ et puisque $-\langle\alpha , X \rangle \geq -C$ pour tout $\alpha \in \Delta_{Q_1}^Q$. 
La seconde inŽgalitŽ rŽsulte des deux inŽgalitŽs $\langle \varpi_\alpha^Q, H + X-T \rangle \leq 0$ (Žtablie au point (i)) et 
$- \langle \varpi_\alpha^Q , X \rangle < 0$ pour tout $\alpha \in \Delta_Q^{Q_1}$ 
(d'aprs \cite[1.2.8]{LW}, puisque $\langle \alpha , X\rangle >0$ pour tout $\alpha \in \Delta_{Q_1}^Q$). 

Prouvons (iii). Pour $\alpha \in \Delta_{Q_1}$, on a (d'aprs \cite[1.2.9]{LW})
$$\langle \alpha , T \rangle \geq \bs{d}_0(T) >0 \quad \hbox{et} \quad \langle \alpha , T_1 \rangle \geq \bs{d}_0(T_1) >0\ptf$$ 
Si $\bs{d}_0(T)>C$ et $\bs{d}_0(T_1)>C$, ce que l'on suppose, alors (d'aprs (i) et (ii)) 
$$\langle \alpha , H \rangle  > 0 \quad \hbox{pour tout} \quad \alpha \in \Delta_{Q_1}^R\ptf$$ 
\'Ecrivons $H^G = \sum_{\beta \Delta_{Q_1} }\langle \beta, H \rangle \varpi_\beta^\vee$. 
Pour $\alpha \in \Delta_{Q_1}^Q$, on a 
$$\langle \varpi_\alpha^Q,H \rangle = \langle \varpi_\alpha , H^Q \rangle = \sum_{\beta \in \Delta_{Q_1}^Q} \langle \beta, H \rangle \langle \varpi_\alpha, \varpi_\beta^\vee \rangle  \ptf$$ 
Comme $\langle \varpi , \varpi_\beta^\vee \rangle \geq 0$ pour tout $\varpi\in \hat{\Delta}_Q^{Q_1}$ et tout $\beta \in \Delta_Q^{Q_1}$ (d'aprs \cite[1.2.8]{LW}), on a 
$$\langle \varpi_\alpha , \varpi_\beta^\vee \rangle = \langle \varpi_\alpha^Q,\varpi_\beta^\vee \rangle \geq 0 \quad \hbox{pour tout}\quad \beta \in \Delta_Q^{Q_1}\ptf$$ 
Comme pour $\beta=\alpha$, on a $\langle \varpi_\alpha^Q , \varpi_\alpha^\vee \rangle >0$, on en dŽduit que 
$$\langle \alpha,H \rangle  \leq \frac{\langle \varpi_\alpha^Q, H \rangle}{\langle \varpi_\alpha^Q, \varpi_\alpha^\vee \rangle}
\quad \hbox{pour tout} \quad \alpha \in \Delta_{Q_1}^{Q}\ptf$$ D'o le point (iii) d'aprs 
la seconde inŽgalitŽ du point (ii). 

Prouvons (iv). \'Ecrivons $H = H^R + H_R$ avec $H^R\in \ag_{Q_1}^R$ et $H_R \in \ag_R$. 
Pour $\alpha \in \Delta_{Q}\smallsetminus \Delta_{Q}^R$, on a 
$$\langle \alpha, H_R \rangle = \langle \alpha , H+ X -T \rangle + \langle \alpha , T\rangle - \langle \alpha , X\rangle  - \langle \alpha , H^R \rangle$$ 
avec $\langle \alpha, H+X-T \rangle$ car $\sigma_Q^R(H+X-T)=1$ et $- \langle \alpha , X \rangle <0$ car si $\alpha'$ est l'unique ŽlŽment de 
$\Delta_{Q_1}$ tel que $\alpha'_Q = \alpha$, on a 
$\langle \alpha, X \rangle \geq \langle \alpha' , X \rangle >0$ (d'aprs \cite[1.2.9]{LW}). D'o la premire inŽgalitŽ du point (iv). Pour $\varpi\in \hat{\Delta}_R$, on a 
$\langle \varpi, H+X-T \rangle >0 $ car $\sigma_Q^R(H+X-T)=1$, par consŽquent $$\langle \varpi , H \rangle > - \langle \varpi, X  \rangle + \langle \varpi,T \rangle$$ avec 
$-\langle \varpi , X \rangle \geq - D$ (rappelons que $\hat{\Delta}_R\subset \hat{\Delta}_{Q_1}$).
\end{proof}

L'ensemble de Siegel $\Siegel_{M_{Q_1}}^*= \mathfrak{C}_{M_{Q_1}}\!\Siegel_{M_{Q_1}}^1$ pour le quotient $\overline{\bs{X}}_{\!M_{Q_1}}$ 
Žtant fixŽ, les constantes $C >0$ et $D> 0$ le sont aussi. 
On suppose dŽsormais que $$\bs{d}_0(T)>C \quad \hbox{et}\quad  \bs{d}_0(T_1)>C\ptf\leqno{(1)}$$ Soit $\bs{Z}_{Q_1}(T,T_1)$ le sous-ensemble de $\bs{Z}_{Q_1}$ formŽ 
des $a$ tel que l'ŽlŽment $H= {\bf H}_{Q_1}(a)$ de $\ag_{Q_1}$ vŽrifie les conditions (i), (ii), (iii) et (iv) de \ref{des bornes explicites}. D'aprs (1) et \ref{des bornes explicites}\,(iii), on a 
l'inclusion $$\bs{Z}_{Q_1}(T,T_1)\subset \bs{Z}_{Q_1}^{R,+}\ptf\leqno{(2)}$$ Notons $\ESB_{Q_1}^G(T,T_1)$ le sous-ensemble de $\ESB_{Q_1}^G= \ESB_G \backslash \ESB_{Q_1}$ dŽfini par 
$$\ESB_{Q_1}^G(T,T_1)= \{{\bf H}_{Q_1}(a)\,\vert \, a\in \bs{Z}_{Q_1}(T,T_1)\}\ptf $$ 

\begin {lemme}\label{estimŽe d'une somme}
Pour tout $\epsilon >0$, il existe des constantes $c_1 >0$ et $\epsilon'_1 >0$ telles que pour 
tout $T\in \ag_0$ tel que $\bs{d}_0(T)>C$ et $\bs{d}_0(T) > \epsilon \| T \|$, on ait 
$$\sum_{H \in \ESB_{Q_1}^G(T,T_1)} \prod_{\alpha \in \Delta_{Q_1}^Q} q^{-r \langle \alpha, {\bf H}_{Q_1}(a) \rangle} \leq c_1 q^{-\epsilon'_1 \| T \|}\ptf$$ 
\end{lemme}

\begin{proof} 
Il suffit de reprendre celle de \cite[3.6.3]{C} (voir aussi \cite[6.8, page 152]{CL}).
\end{proof}

On en dŽduit la 

\begin{proposition}\label{l'estimŽe fondamentale}
Soit $\Gamma_G$ un sous-ensemble compact de $G(\mbb{A})$. Pour tout $\epsilon >0$, il existe des constantes $c >0$ et $\epsilon' >0$ telles que pour 
tout $T\in \ag_0$ tel que $\bs{d}_0(T)>C$ et $\bs{d}_0(T) > \epsilon \| T \|$, on ait 
$$\int_{\bs{Z}_{Q_1}(T,T_1)} \bs{\delta}_{Q_1}(a)^{-1} \sup_{x\in \Gamma_G}\vert \wt{k}_{Q,\ES{O}}^R(ax)\vert \dd a\leq c q^{-\epsilon' \| T \|}\ptf$$ 
\end{proposition}

\begin{proof} 
D'aprs \ref{la majoration clŽ bis}, l'intŽgrale de l'ŽnoncŽ est essentiellement majorŽe par 
$$\int_{\bs{Z}_{Q_1}(T,T_1)} \left(\prod_{\alpha \in \Delta_{Q_1}^Q} q^{-r \langle \alpha, {\bf H}_{Q_1}(a) \rangle}\right)\dd a \ptf \leqno{(3)}$$ 
Posons $\bs{Z}_{Q_1}^1 = \ker ({\bf H}_{Q_1}: \bs{Z}_{Q_1}\rightarrow \ESB_{Q_1}^G)$. Puisque $\bs{Z}_{Q_1}^1\bs{Z}_{Q_1}(T,T_1)= \bs{Z}_{Q_1}(T,T_1) $ et que 
$\bs{Z}_{Q_1}^1$ est compact, ˆ un volume fini prs, l'intŽgrale (3) est Žgale ˆ 
$$\sum_{H \in \ESB_{Q_1}^G(T,T_1)} \prod_{\alpha \in \Delta_{Q_1}^Q} q^{-r \langle \alpha, {\bf H}_{Q_1}(a) \rangle}\ptf$$
On conclut gr‰ce ˆ \ref{estimŽe d'une somme}. 
\end{proof}

\subsection{DŽmonstration de la proposition \ref{CV orbite}}\label{preuve de CV strate} La $G$-orbite $\ES{O} \in (\UU/G)_F$  
est fixŽe. Pour $T\in \ag_0$, notons $k^T_{\ES{O}}(x)$ l'expression obtenue en remplaant $\UF$ par $\ES{O}_F$ dans l'expression \ref{troncature}\,(2) pour $k^T_\u(x)$: 
$$k_{\ES{O}}^T(x) = \sum_{P\in \ESP_\mathrm{st}}(-1)^{a_P-a_G} \sum_{\substack{Q,R\in \ESP_\mathrm{st}\\ Q\subset P \subset R}}
\sum_{\xi \in Q(F)\backslash G(F)}\sigma_Q^R({\bf H}_0(\xi x)-T)\bs{\Lambda}^{T,Q}k_{P,\ES{O}}(\xi x) \ptf$$ 
Pour $T\in \ag_0$ assez rŽgulier (ne dŽpendant que du support de $f$), 
on obtient comme en \ref{CV et rŽcriture de l'intŽgrale tronquŽe} que l'expression 
$k^T_{\ES{O}}(x)$ est Žgale ˆ 
$$\sum_{\substack{Q,R\in \ESP_\mathrm{st}\\ Q\subset R}}\sum_{\xi \in Q(F)\backslash G(F)}F_{P_0}^Q(\xi x,T)
\sigma_Q^R({\bf H}_0(\xi x)-T)\sum_{\substack{P\in \ESP_\mathrm{st}\\Q \subset P \subset R}}(-1)^{a_P-a_G} 
k_{P,\ES{O}}( \xi x)\ptf$$ L'intŽgrale $$\int_{\overline{\bs{X}}_{\!G}}\vert k^T_{\ES{O}}(x) \vert \dd x$$ est alors 
bornŽe par $$\sum_{\substack{Q,R\in \ESP_\mathrm{st}\\ Q\subset R}}\int_{\bs{Y}_Q}F_{P_0}^Q(x,T)
\sigma_Q^R({\bf H}_0(x)-T) \left| \sum_{\substack{P\in \ESP_\mathrm{st}\\Q \subset P \subset R}}(-1)^{a_P-a_G} 
k_{P,\ES{O}}(x)  \right|\dd x\ptf $$ Pour $Q,\,R\in \ES{P}_\mathrm{st}$ tels que $Q\subset R$, rappelons que l'on a posŽ
$$k_{Q,\ES{O}}^R(x)= \sum_{\substack{P\in \ESP_\mathrm{st}\\Q \subset P \subset R}}(-1)^{a_P-a_G} 
k_{P,\ES{O}}(x)\ptf$$ La fonction $x\mapsto k_{Q,\ES{O}}^R(x)$ sur $G(\mbb{A})$ se factorise par $\bs{Y}_Q = A_G(\mbb{A}) Q(F)\backslash G(\mbb{A})$. 
Il s'agit en particulier de prouver, pour chaque paire $Q\subset R$ de sous-groupes paraboliques standard, la convergence de l'intŽgrale
$$\int_{\bs{Y}_Q}F_{P_0}^Q(x,T)
\sigma_Q^R({\bf H}_0(x)-T) \left| k_{Q,\ES{O}}^R(x)\right|\dd x\ptf\leqno{(1)}$$
Pour $Q=R=G$, la contribution est l'intŽgrale sur $\bs{Y}_G= \overline{\bs{X}}_{\!G}$ obtenue en otant les valeurs absolues dans l'expression (1), \cad 
$$\int_{\overline{\bs{X}}_{\!G}} F_{P_0}^G(x,T) k_{\ES{O}}(x)\dd x =   \int_{\overline{\bsX}_{\!G}} \bs\Lambda^Tk_{\ES{O}}(x) \dd x \ptf$$ 
Puisque la fonction $F_{P_0}^G(\cdot ,T)$ est ˆ support compact sur $\overline{\bs{X}}_{\!G}$, on a 
$$\int_{\overline{\bs{X}}_{\!G}} F_{P_0}^G(x,T) \left| k_{\ES{O}}(x) \right|\dd x\leq \int_{\overline{\bs{X}}_{\!G}} F_{P_0}^G(x,T) 
\left(\sum_{\eta \in \ES{O}_F}\vert f(x^{-1}\eta x)\vert \right)\dd x < +\infty\ptf$$ 
Pour $Q=R \subsetneq G$, la contribution est nulle car $\sigma_Q^R =0$. Reste ˆ Žvaluer la contribution pour $Q\subsetneq R \subset G$. 

On fixe jusqu'ˆ la fin de \ref{preuve de CV strate} (\cad la fin de la dŽmonstration de \ref{CV orbite}) une paire $Q\subsetneq R$ d'ŽlŽments 
de $\ESP_\mathrm{st}$. On a $$k_{Q,\ES{O}}^R(x)= \sum_{\substack{P\in \ESP_\mathrm{st}\\Q \subset P \subset R}}(-1)^{a_P-a_G} 
\sum_{\eta \in \UF^{M_P}} \xi_{P,\ES{O}}(\eta) \int_{U_P(\mbb{A})} f(x^{-1}\eta u x) \dd u\ptf$$ 
Comme au dŽbut de la preuve de \cite[9.1.1]{LL}, on peut invoquer \cite[3.6.7]{LW}: pour $T$ assez rŽgulier (ne dŽpendant que du support de $f$), 
la condition $$F_{P_0}^Q(x,T)\sigma_Q^R({\bf H}_0(x)-T)\neq 0$$ assure que l'expression $k_{Q,\ES{O}}^R(x)$ est Žgale 
ˆ $$\wt{k}_{Q,\ES{O}}^R(x)= \sum_{\substack{P\in \ESP_\mathrm{st}\\Q \subset P \subset R}}(-1)^{a_P-a_G} \sum_{\eta \in \UF^{M_Q}} \Phi_{P,\eta,\ES{O}}(x)\ptf\leqno{(2)}$$ avec  
$$\Phi_{P,\eta,\ES{O}}(x)= \sum_{\nu \in U_Q^{P}(F)} \xi_{P,\ES{O}}(\eta\nu) \int_{U_P(\mbb{A})} f(x^{-1}\eta\nu u x) \dd u\ptf$$ 
Rappelons que la fonction $\xi_{P,\ES{O}}$ sur $\UF^{M_P}U_P(F)$ est invariante par translations ˆ droite (et ˆ gauche) 
par $U_P(F)$. On obtient 
$$ \Phi_{P,\eta,\ES{O}}(x)=\int_{U_P(F)\backslash U_P(\mbb{A})} \left(\sum_{\nu \in U_Q(F)} \xi_{P,\ES{O}}(\eta\nu)f(x^{-1}\eta\nu u x)\right)\! \dd u\ptf$$ 
On peut alors remplacer l'intŽgrale en $u\in U_P(F)\backslash U_P(\mbb{A})$ par une intŽgrale sur un domaine de Siegel $\Siegel_{U_P}$ pour le quotient 
$U_P(F)\backslash U_P(\mbb{A})$: 
$$\Phi_{P,\eta,\ES{O}}(x)= \sum_{\nu \in U_Q(F)} \xi_{P,\ES{O}}(\eta\nu) \int_{\Siegel_{U_P}} f(x^{-1}\eta\nu u x) \dd u\ptf\leqno{(3)}$$ 

Pour $\eta \in \UF^{M_Q}$, on pose 
$$\Phi_{Q,\eta,\ES{O}}^R(x)= \sum_{\substack{P\in \ESP_\mathrm{st}\\Q \subset P \subset R}}(-1)^{a_P-a_G} \Phi_{P,\eta,\ES{O}}(x)\ptf$$ 
Puisque pour $P\in \ES{F}^R(Q)$, la fonction $\xi_{P,\ES{O}}$ sur $P(F)$ est invariante par translations ˆ droite par $U_R(F)$, l'expression 
$\Phi_{Q,\eta,\ES{O}}^R(x)$ est Žgale ˆ 
$$\sum_{\substack{P\in \ESP_\mathrm{st}\\Q \subset P \subset R}}(-1)^{a_P-a_G}  \!\!
\sum_{\nu\in U_Q^R(F)}\xi_{P,\ES{O}}(\eta\nu)
\int_{\Siegel_{U_P^R}\times U_R(\mbb{A})} f(x^{-1}\eta\nu v u' x) \dd v \dd u'\leqno{(4)}$$ o $\Siegel_{U_P^R}$ est un sous-ensemble ouvert compact de $U_P^R(\mbb{A})$ 
qui soit un domaine de Siegel pour le quotient $U_P^R(F)\backslash U_P^R(\mbb{A})$. Comme (par dŽfinition)  
$$\wt{k}_{Q,\ES{O}}^R(x)= \sum_{\eta \in \UU^{M_Q}(F)} \Phi_{Q,\eta,\ES{O}}^R(x)$$ et que $\UF^{M_Q}U_Q^R(F)= \UF^{M_R} \cap Q(F)$, l'expression $\wt{k}_{Q,\ES{O}}^R(x)$ est Žgale ˆ
$$ \sum_{\substack{P\in \ESP_\mathrm{st}\\Q \subset P \subset R}}(-1)^{a_P-a_G}  \!\!\!\!\!
\sum_{\gamma \in \UF^{M_R}\cap Q(F)}\xi_{P,\ES{O}}(\gamma)
\int_{\Siegel_{U_P^R}\times U_R(\mbb{A})} f(x^{-1}\gamma v u' x) \dd v \dd u'\ptf\leqno{(5)}$$

D'aprs la proposition 3.6.3 de \cite{LW} (appliquŽe au parabolique $P=Q$), il existe une constante $c_1>0$ telle que pour tout $T_1\in \ag_0$ avec $\bs{d}_0(T_1)\geq c_1$, 
on ait $$ \sum_{\substack{Q_1\in \ESP_\mathrm{st}\\P_0 \subset Q_1 \subset Q}}\sum_{\xi \in Q_1(F)\backslash Q(F)} F^{Q_1}_{P_0}(\xi x , T_1) \tau_{Q_1}^Q({\bf H}_0(\xi x)-T_1)=1\ptf$$ 
Fixons un ŽlŽment $T_1 \in \ag_0$ avec $d_0(T_1)\geq c_1$. Pour $T$ assez rŽgulier (ne dŽpendant que du support de $f$), l'intŽgrale (1) est Žgale ˆ 
$$\sum_{Q_1\in \ES{F}^{Q}(P_0)} \int_{\bs{Y}_{Q_1}}\bs{\chi}_{Q,Q_1}^R(x,T,T_1) \vert \wt{k}_{Q,\ES{O}}^R(x)\vert \dd x  \leqno{(6)}$$ 
avec $$\bs{\chi}_{Q,Q_1}^R(x,T,T_1)\bydef F_{P_0}^{Q_1}(x,T_1) \tau_{Q_1}^Q({\bf H}_0(x)-T_1)F_{P_0}^Q({\bf H}_0(x)-T) \sigma_Q^R({\bf H}_0(x)-T)\ptf$$ 
Il s'agit donc, pour chaque $Q_1\in \ESF^{Q}(P_0)$, d'estimer l'intŽgrale 
$$\int_{\bs{Y}_{Q_1}}\bs{\chi}_{Q,Q_1}^R(x,T,T_1) \vert \wt{k}_{Q,\ES{O}}^R(x)\vert \dd x  \leqno{(7)}$$ 

Fixons $Q_1\in \ESF^Q(P_0)$ et rŽcrivons l'intŽgrale (7) en utilisant de la dŽcomposition d'Iwasawa $G(\mbb{A})= U_{Q_1}(\mbb{A})M_{Q_1}(\mbb{A})\bs{K}$. 
On obtient que l'intŽgrale (7) est Žgale ˆ 
\begin{eqnarray*}
\lefteqn{\int_{\bs{K}}\int_{\overline{\bs{X}}_{\!M_{Q_1}}}\int_{\bs{Z}_{Q_1}} 
\int_{U_{Q_1}(F)\backslash U_{Q_1}(\mbb{A})}\bs{\delta}_{Q_1}(am)^{-1} F_{P_0}^{Q_1}(m,T_1)\tau_{Q_1}^Q({\bf H}_0(am)-T_1)}\\ 
&&\hspace{1.5cm}\times F_{P_0}^Q(uam,T)\sigma_Q^R({\bf H}_0(am)-T) 
\wt{k}_{Q,\ES{O}}^R(uamk)\dd u \dd a \dd m \dd k \ptf
\end{eqnarray*}
L'expression ci-dessus a un sens car les fonctions $x \mapsto F_{P_0}^{Q}(x,T)$ et $x\mapsto \wt{k}_{Q,\ES{O}}^R(x)$ sont invariantes ˆ gauche par $Q(F)\;(\supset U_{Q_1}(F))$. 
On remplace l'intŽgrale en $u$ par une intŽgrale sur un domaine de Siegel $\Siegel_{U_{Q_1}}$ pour le quotient $U_{Q_1}(F)\backslash U_{Q_1}(\mbb{A})$. Comme dans la preuve de \ref{raffinement}, 
on prend $\Siegel_{U_{Q_1}}$ de la forme $\Siegel_{U_R}\Siegel_{U_{Q_1}^R}$. On remplace l'intŽgrale en 
$u\in \bs{\mathfrak{S}}_{U_{Q_1}}$ par une intŽgrale en $u'\in \bs{\mathfrak{S}}_{U_R}$ suivie d'une intŽgrale en $u''\in \bs{\mathfrak{S}}_{U_{Q_1}^R}$. 
Pour $x= u'u''amk$, $P\in \ES{F}^R(Q)$ et $\gamma \in \UU^{M_R}(F)$, on a
\begin{eqnarray*}
\lefteqn{
\int_{\Siegel_{U_P^R}\times U_R(\mbb{A})} f(x^{-1}\gamma v u'_1 x) \dd v\dd u'_1 }\\
&&=\int_{\Siegel_{U_P^R}\times U_R(\mbb{A})} f(k^{-1}m^{-1}a^{-1}u''^{-1} \gamma v \mathrm{Int}_{(\gamma v)^{-1}}(u'^{-1}) u'_1 u'u''amk) \dd u_1\dd u_2\\
&& = \int_{\Siegel_{U_P^R}\times U_R(\mbb{A})} f(k^{-1}m^{-1}a^{-1}u''^{-1}\gamma v u'_1 u''amk) \dd v\dd u'_1\ptf
\end{eqnarray*}
En d'autres termes, l'intŽgrale en $u'\in \Siegel_{U_R}$ est absorbŽe (pour chaque $P\in \ESF^R(Q)$) par l'intŽgrale en $u'\in U_R(\mbb{A})$ dans $(5)$. Comme d'autre part 
la fonction $x\mapsto F_{P_0}^Q(x,T)$ est invariante ˆ gauche par $U_Q(\mbb{A})\;(\supset U_R(\mbb{A}))$, on obtient que l'intŽgrale (7) est Žgale 
ˆ 
\begin{eqnarray*}
\lefteqn{\int_{\bs{K}}\int_{\overline{\bs{X}}_{\!M_{Q_1}}}\int_{\bs{Z}_{Q_1}} \int_{\Siegel_{U_{Q_1}^R}}\bs{\delta}_{Q_1}(am)^{-1} F_{P_0}^{Q_1}(m,T_1)\tau_{Q_1}^Q({\bf H}_0(am)-T_1)}\\ 
&&\hspace{0.5cm}\times F_{P_0}^Q(uam,T)\sigma_Q^R({\bf H}_0(am)-T) 
\vert \wt{k}_{Q,\ES{O}}^R(uamk)\vert \dd u \dd a \dd m \dd k \ptf
\end{eqnarray*}
L'expression ci-dessus est majorŽe par celle obtenue en remplaant l'intŽgrale en $m$ par une intŽgrale sur 
un ensemble de Siegel $\Siegel_{M_{Q_1}}^*$ pour le quotient $\overline{\bs{X}}_{\!M_{Q_1}}$; et d'aprs \cite[1.8.3]{LW} et \cite[3.4]{LL}, 
la condition $F^{Q_1}_{P_0}(m,T_1)\neq 0$ implique 
que cette intŽgrale en $m\in \Siegel_{M_{Q_1}}^*$ porte sur un sous-ensemble ouvert compact $\Gamma_{M_{Q_1}}$ de $\Siegel_{M_{Q_1}}^*$. 
L'intŽgrale (7) est donc majorŽe par l'expression 
\begin{eqnarray*}
\lefteqn{\int_{\bs{K}}\int_{\Gamma_{M_{Q_1}}}\int_{\bs{Z}_{Q_1}} \int_{\Siegel_{U_{Q_1}^R}}\bs{\delta}_{Q_1}(am)^{-1} \tau_{Q_1}^Q({\bf H}_0(am)-T_1)}\\ 
&&\hspace{0.5cm}\times F_{P_0}^Q(uam,T)\sigma_Q^R({\bf H}_0(am)-T) 
\vert \wt{k}_{Q,\ES{O}}^R(uamk)\vert \dd u \dd a \dd m \dd k 
\end{eqnarray*}
Le terme crucial dans l'expression ci-dessus est l'intŽgrale en $a$, les trois autres intŽgrales portant sur des ensembles compacts. On suppose 
que $\Siegel_{M_{Q_1}}^*$ est choisi comme en \ref{estimŽe d'une intŽgrale}: $\Siegel_{M_{Q_1}}^* = \mathfrak{C}_{M_{Q_1}}\! \Siegel_{M_{Q_1}}^1$ 
avec $\mathfrak{C}_{M_{Q_1}}$ comme en 
\ref{une condition commode}. Compte-tenu de \ref{des bornes explicites} et de la dŽfinition de $\bs{Z}_{Q_1}(T,T_1)$, 
ˆ des volumes finis prs qui ne dŽpendent ni de $f$, $T$ ou $T_1$, il suffit pour majorer l'intŽgrale (7) de majorer l'intŽgrale
$$\int_{\bs{Z}_{Q_1}(T,T_1)} \bs{\delta}_{Q_1}(am)^{-1} \vert \wt{k}_{Q,\ES{O}}^R(uamx)\vert \dd a\leqno{(8)}$$ 
pour $u\in \Siegel_{U_Q^R}$, $m\in \Gamma_{M_{Q_1}}$ et $k\in \bs{K}$. Pour $m\in M_{Q_1}(\mbb{A})$, on a 
$$\bs{\delta}_{Q_1}(m)^{-1}= e^{- \langle 2\rho_{Q_1},{\bf H}_{Q_1}(m)\rangle}$$ o 
$\rho_{Q_1}$ dŽsigne la demi-somme des racines de $A_{Q_1}$ dans $U_{Q_1}$. En particulier pour $m\in \Siegel_{M_{Q_1}}^*$, 
et donc \textit{a fortiori} pour $m\in \Gamma_{M_{Q_1}}$, 
le terme $\bs{\delta}_{Q_1}(m)^{-1}$ est majorŽ par une constante ne dŽpendant que de $\mathfrak{C}_{M_{Q_1}}$.  
D'autre part pour $a\in \bs{Z}_{Q_1}^{R,+}$, et donc \textit{a fortiori} pour $a\in \bs{Z}_{Q_1}(T,T_1)$, l'automorphisme $\mathrm{Int}_{a^{-1}}$ de $G(\mbb{A})$ 
contracte $U_{Q_1}^R(\mbb{A})$.  
On peut donc fixer un sous-ensemble ouvert compact $\Gamma_{U_{Q_1}^R}$ de $U_{Q_1}^R(\mbb{A})$ 
tel que $a^{-1}u a \subset \Gamma_{U_{Q_1}^R}$ pour tout $a\in \bs{Z}_{Q_1}^{R,+}$ et tout 
$u\in \Siegel_{U_{Q_1}^R}$. Posons 
$$\Gamma_G = \Gamma_{U_{Q_1}^R} \Gamma_{M_{Q_1}} \bs{K}\ptf$$ Pour 
$u\in \Siegel_{U_{Q_1}^R}$, $a\in \bs{Z}_{Q_1}^{R,+}$, $m\in \Gamma_{M_{Q_1}}$ et $k\in \bs{K}$, on a 
$$x= a^{-1}ua mk \in \Gamma_G\ptf$$ L'intŽgrale (8) est essentiellement majorŽe par 
$$\int_{\bs{Z}_{Q_1}(T,T_1)} \bs{\delta}_{Q_1}(a)^{-1} \sup_{x\in \Gamma_G}\vert \wt{k}_{Q,\ES{O}}^R(ax)\vert \dd a\ptf \leqno{(9)}$$ 
L'intŽgrale (7) est elle aussi essentiellement majorŽe par l'intŽgrale (9). On peut donc lui appliquer \ref{l'estimŽe fondamentale}: 
il existe une constante $c_0>0$ telle que 
pour tout $\epsilon >0$, il existe des constantes $c >0$ et $\epsilon' >0$ telles que pour 
tout $T\in \ag_0$ tel que $\bs{d}_0(T) \geq c_0$ et $\bs{d}_0(T)\geq \epsilon \| T \|$, on ait 
$$ \int_{\bs{Y}_{Q_1}}\bs{\chi}_{Q,Q_1}^R(x,T,T_1) \vert \wt{k}_{Q,\ES{O}}^R(x)\vert \dd x  \leq c q^{- \epsilon' \| T \|}\ptf \leqno{(10)} $$ 
Puisque l'intŽgrale (1) est Žgale ˆ la somme sur $Q_1\in \ESF^Q(P_0)$ des intŽgrales (7), cela achve la dŽmonstration de 
\ref{CV orbite}. 

\subsection{DŽmonstration de \ref{PolExp}}\label{preuve de PolExp} Il suffit de reprendre celle du thŽorme 11.1.1 de \cite{LL}. Avant cela, rappelons quelques notations 
de \cite{LL}. Pour $Q\in \ESP_\mathrm{st}$, le $\mbb{Z}$-module de type fini $\ESC_Q^G = \ESB_G \backslash \ESA_Q$ s'insre 
dans la suite exacte courte $$0 \rightarrow \ESB_Q^G= \ESB_G \backslash \ESB_Q  \rightarrow \ESC_Q^G \rightarrow \bsbbc_Q = \ESB_Q\backslash \ESA_Q\rightarrow 0\ptf$$ 
Pour $Z\in \bsbbc_Q$, on note $\ESB_Q^G(Z)\subset \ESC_Q^G$ la fibre au-dessus de $Z$ (c'est un espace principal homogne sous le groupe $\ESB_Q^G$)
et pour $T,\, X\in \ag_0$, on pose $$\eta_{Q,F}^{G,T}(Z\mathpvg X) = \sum_{H\in \ESB_Q^G(Z)} \Gamma_Q(H-X,T)\ptf\leqno{(1)}$$ 
Pour une $G$-orbite $\ES{O}\in (\UU/G)_F$, l'intersection $\ES{O}\cap M_Q$ est une rŽunion, Žventuellement vide, de $M_Q$-orbite dans $\UU^{M_Q}$; on note 
$((\ES{O} \cap M_Q)/M_Q)_F\subset (\UU^{M_Q}/M_Q)_F$ le sous-ensemble, Žventuellement vide, formŽ des $M_Q$-orbites dans $\ES{O}\cap M_Q$ qui intersectent non trivialement 
$\UF^{M_Q}$. On note 
$$k^{M_Q,T}_{\ES{O}}(f_Q\mathpvg \cdot)= \sum_{\ES{O}^M \in ((\ES{O}\cap M_Q)/M_Q)_F} k^{M_Q,T}_{\ES{O}^M}(f_Q\mathpvg \cdot)\leqno{(2)}$$ la fonction sur $M_Q(\mbb{A})$ obtenue en remplaant $G$, $f$ et $\ES{O}_F$ respectivement par $M_Q$, $f_Q$ et 
$\UF^{M_Q} \cap \ES{O}$ dans la dŽfinition de $k_{\bsfrY}^T= k_{\bsfrY}^{G,T}(f\mathpvg \cdot)$; avec la convention que (2) est la fonction identiquement nulle si l'intersection $\UF^{M_Q}\cap \ES{O}$ est vide. 
D'aprs \ref{CV}, pour $T\in \ag_0$ assez rŽgulier, l'intŽgrale 
$$\int_{\overline{\bs{X}}_{\!M_Q}} k^{M_Q,T}_{\ES{O}}(f_Q\mathpvg m) \dd m\leqno{(3)}$$ est absolument convergente.

 La $G$-orbite $\ES{O}\in (\UU/G)_F$ Žtant fixŽe, on peut commencer la dŽmonstration de \ref{PolExp}.
Comme dans celle de \cite[11.1.1]{LW}, pour $T,\, X \in \ag_{0,\mbb{Q}}$ assez rŽguliers, on obtient 
$$\int_{\overline{\bs{X}}_{\!G}}k_{\ES{O}}^{T+X} (x) \dd x = \sum_{Q\in \ESP_\mathrm{st}} \int_{\bs{Y}_Q}\Gamma_Q^G({\bf H}_0(x)-X,T) k_{Q,\ES{O}}^X(x)\dd x$$ avec 
$$k_{Q,\ES{O}}^X(x)= \sum_{P\in \ESF^Q(P_0)}\sum_{\xi\in P(F)\backslash Q(F)} (-1)^{a_P-a_G}\hat{\tau}_P^Q({\bf H}_0(\xi x) -X) k_{P,\ES{O}}(\xi x)\ptf$$ 
Soit $\mathfrak{B}_Q$ l'image d'une section du morphisme $A_Q(\mbb{A})\rightarrow \ESB_Q$. 
Comme dans la preuve de \textit{loc.~cit.}, on suppose que $\mathfrak{B}_Q = \mathfrak{B}_G \mathfrak{B}_Q^G$. 
Posons $$\bs{Y}'_Q=\mathfrak{B}_G Q(F)\backslash G(\mbb{A})\quad \hbox{et}\quad \bs{Y}''_Q =\mathfrak{B}_Q Q(F)\backslash G(\mbb{A})\ptf$$ 
Puisque $\mathrm{vol}(A_G(F)\backslash A_G(\mbb{A})^1)=1$, on peut remplacer l'intŽgrale sur $\bs{Y}_Q $ par une intŽgrale sur $\bs{Y}'_Q$. 
Pour $Z\in \bsbbc_Q = \ESB_Q\backslash \ESA_Q$, notons 
$\bs{Y}''_Q(Z)$ l'image dans $\bs{Y}''_Q$ de l'ensemble des $g\in G(\mbb{A})$ tels que ${\bf H}_Q(g) =Z'$ o $Z'$ est un relvement de $Z$ dans $\ESA_Q$. 
On dŽfinit de la mme manire 
le sous-ensemble $\smash{\overline{\bs{X}}}'_{\!M_Q}(Z)$ de $$\smash{\overline{\bs{X}}}'_{\!M_Q}= \mathfrak{B}_Q M_Q(F) \backslash M_Q(\mbb{A})\ptf$$
On obtient $$\int_{\bs{Y}_Q}\Gamma_Q^G({\bf H}_0(x)-X,T) k_{Q,\ES{O}}^X(x)\dd x= \sum_{Z\in \bsbbc_Q} \eta_{Q,F}^{G,T}(Z\mathpvg X) \int_{\bs{Y}''_Q(Z)} k^X_{Q,\ES{O}}(x)\dd x$$ 
avec $$\int_{\bs{Y}''_Q(Z)} k^X_{Q,\ES{O}}(x)\dd x = \int_{\smash{\overline{\bs{X}}}'_{\!M_Q}(Z)}k^{M_Q,X}_{\ES{O}}(f_Q\mathpvg m)\dd m\ptf $$ 
Comme plus haut, on peut remplacer l'intŽgrale sur $\smash{\overline{\bs{X}}}'_{\!M_Q}(Z)$ par une intŽgrale sur $\overline{\bs{X}}_{\!M_Q}(Z)\subset \overline{\bs{X}}_{\!M_Q}$. 
En posant $$\mathfrak{J}^{G,T}_{\ES{O}}(f)= \int_{\overline{\bs{X}}_G}k_{\ES{O}}^{G,T}(x)\dd x\vg$$ on a donc 
$$\mathfrak{J}^{G,T+X}_{\ES{O}}(f)= \sum_{Q\in \ESP_\mathrm{st}} \sum_{Z\in \bsbbc_Q} \eta_{Q,F}^{G,T}(Z\mathpvg X)  \mathfrak{J}^{M_Q,X}_{\ES{O}}(Z\mathpvg f_Q)\leqno{(4)}$$ avec 
$$ \mathfrak{J}^{M_Q,X}_{\ES{O}}(Z\mathpvg  h)= \int_{\overline{\bs{X}}_{\!M_Q}(Z)} k^{M_Q,X}_{\ES{O}}(h\mathpvg m)\dd m\ptf$$ 
Puisque d'aprs \cite[2.1.3]{LL}, pour $X\in \ag_{0,\mbb{Q}}$, la fonction 
$T \mapsto  \eta_{Q,F}^{G,T}(Z\mathpvg X) $ sur $\ag_{0,\mbb{Q}}$ est dans PolExp, la proposition \ref{PolExp} est dŽmontrŽe. 

\subsection{Action de la conjugaison}\label{action de la conjugaison} 
Si $\bsfrY$ est une $F$-strate de $\UF$, on s'intŽresse ˆ l'action de la conjugaison sur l'ŽlŽment $\mathfrak{J}_{\bsfrY}^{G,T}(f)$ de PolExp 
donnŽ par $$\mathfrak{J}^{G,T}_{\bsfrY}(f)= \sum_{\ES{O}\in (\bsfrY/G)_F}\mathfrak{J}_{\ES{O}}^{G,T}(f)\ptf$$

Pour $Q\in \ES{P}_{\mathrm{st}}$, puisque $M_Q$ est un $F$-sous-groupe fermŽ \textit{critique} de $G$, d'aprs \ref{le cas critique}\,(iii) l'intersection 
$\bsfrY \cap M_Q$ est une rŽunion finie, Žventuellement vide, de strates de $\UU^{M_Q}$. On note $[\bsfrY \cap M_Q]_F\subset [\UU^{M_Q}]_F$ le sous-ensemble, 
Žventuellement vide, formŽ des strates dans $\bsfrY \cap M_Q$ qui intersectent non trivialement $\UF^{M_Q}$. Pour 
$h\in C^\infty_\mathrm{c}(M_Q(\mbb{A}))$ et $T\in \ag_0$, on note 
$$k^{M_Q,T}_{\bsfrY}(h\mathpvg \cdot)= \sum_{\bsfrY^M \in [\bsfrY \cap M_Q]_F} k^{M_Q,T}_{\bsfrY^M}(h\mathpvg \cdot)\leqno{(2)}$$ la fonction sur $M_Q(\mbb{A})$ obtenue en remplaant $G$, $f$ et $\bsfrY$ respectivement par $M_Q$, $h$ et 
$\UF^{M_Q} \cap \bsfrY$ dans la dŽfinition de $k_{\bsfrY}^T= k_{\bsfrY}^{G,T}(f\mathpvg \cdot)$; avec la convention que (2) est la fonction identiquement nulle si l'intersection $\UF^{M_Q} \cap \bsfrY$ est vide. 
D'aprs \ref{CV}, pour $T\in \ag_0$ assez rŽgulier, l'intŽgrale 
$$\mathfrak{J}^{M_Q,T}_{\bsfrY}(h)=\int_{\overline{\bs{X}}_{\!M_Q}} k^{M_Q,T}_{\bsfrY}(h\mathpvg m) \dd m\leqno{(3)}$$ est absolument convergente; et d'aprs 
\ref{PolExp}, la fonction $T\mapsto \mathfrak{J}^{M_Q,T}_{\bsfrY}(h)$ sur $\ag_{0,\mbb{Q}}$ dŽfinit un ŽlŽment de PolExp. 

Pour $y\in G(\mbb{A})$, on note $f^y$ la fonction $f\circ \mathrm{Int}_y$. Soient $y\in G(\mbb{A})$ et 
$T\in \ag_0$. Pour $Q\in \ESP_\mathrm{st}$ et $Z\in \bsbbc_Q= \ESB_Q\backslash \ESA_Q$, notons $f_{Q,y}^T(Z\mathpvg  \cdot)$ la fonction dans 
$C^\infty_\mathrm{c}(M_Q(\mbb{A}))$ dŽfinie par 
$$f_{Q,y}^T(Z\mathpvg m)= \int_{\bs{K}}\int_{U_Q(\mbb{A})} f(k^{-1}mu k) \eta_{Q,F}^{G,-{\bf H}_0(ky)}(Z\mathpvg T) \dd u \dd k$$ 
o, pour $X\in \ag_0$, l'expression $\eta_{Q,F}^{G,X}(Z\mathpvg T)$ a ŽtŽ dŽfinie en \ref{preuve de PolExp}\,(1). 

\begin{proposition}
Soient $\bsfrY\in [\UU]_F$ et $y\in G(\mbb{A})$. Pour $T\in \ag_{0,\mbb{Q}}$ assez rŽgulier, on a 
$$\mathfrak{J}^{G,T}_{\bsfrY}({f^y})= \sum_{Q\in \ESP_\mathrm{st}}\sum_{Z\in \bsbbc_Q} \mathfrak{J}^{M_Q,T}_{\bsfrY}(f^T_{Q,y}(Z))\ptf$$
\end{proposition}

\begin{proof}
Elle est identique ˆ celle de \cite[11.2.1]{LL}.
\end{proof}


\setcounter{section}{0}
\renewcommand{\thesection}{\Alph{section}}

\part*{Partie III: annexes}
\section{Sur la sŽparabilitŽ non algŽbrique}\label{annexe A}

Dans cette annexe, $F$ est un corps commutatif. 

\subsection{Algbres sŽparables sur un corps}\label{algbres sŽparables} 
Le corps $F^\mathrm{s\acute{e}p}$ est par dŽfinition une cl™ture \textit{algŽbrique} sŽparable de $F$. 
On a une notion plus gŽnŽrale de sŽparabilitŽ, valable pour les extensions non algŽbriques de $F$ (\textit{cf.} \cite[{\bf 4.6}]{EGA}):

\begin{definition}\label{sŽp(gŽnŽral)}
\textup{Une $F$-algbre (commutative) $A$ est dite \textit{sŽparable} si pour toute extension 
$F'/F$, l'anneau $A\otimes_F F'$ est rŽduit (\cad sans ŽlŽment nilpotent autre que $0$). En d'autres termes, $A$ est sŽparable si et seulement si le $F$-schŽma affine 
$\mathrm{Spec}(A)$ est gŽomŽtriquement rŽduit\footnote{Observons que si de plus $A$ est intgre, le fait que la $F$-algbre $A$ soit sŽparable ne signifie pas que 
le $F$-schŽma $\mathrm{Spec}(A)$ soit gŽomŽtriquement intgre; il l'est s'il est ˆ la fois gŽomŽtriquement rŽduit et gŽomŽtriquement irrŽductible.}.}
\end{definition}

On rappelle quelques faits connus:
\begin{itemize}
\item Une extension algŽbrique de $F$ est sŽparable (au sens usuel) si et seulement si c'est une $F$-algbre 
sŽparable. On peut donc sans ambigu•tŽ parler d'extension sŽparable (algŽbrique ou non) du corps $F$.
\item Les extensions purement transcendantes de $F$ sont des $F$-algbres sŽparables.    
\item Si $A$ et $B$ sont des $F$-algbres sŽparables, alors $A\otimes_FB$ l'est aussi.
\item Si $F$ est de caractŽristique nulle, une $F$-algbre est sŽparable si et seulement si elle est rŽduite. 
\end{itemize}

\vskip1mm
Pour $F$ de caractŽristique quelconque, on a la caractŽrisation suivante:

\begin{proposition}\label{critre sŽparabilitŽ 1}
Une $F$-algbre $A$ est sŽparable si et seulement elle vŽrifie les conditions 
Žquivalentes suivantes:
\begin{enumerate}
\item[(i)] l'anneau $A\otimes_FF^\mathrm{rad}$ est rŽduit;
\item[(ii)] pour toute extension radicielle finie $F'/F$, l'anneau $A\otimes_F F'$ est rŽduit;
\item[(iii)] l'anneau $A$ est rŽduit et pour tout idŽal premier minimal $\mathfrak{p}$ de $A$, le corps rŽsiduel 
$\kappa(\mathfrak{p})$ de l'anneau local $A_\mathfrak{p}$ est une extension sŽparable de $F$. 
\end{enumerate}
\end{proposition}

Pour les extensions de corps, on a aussi:

\begin{proposition}\label{critre sŽparabilitŽ 2}
Soit une extension $\Omega/F$ tel que le corps $\Omega$ soit parfait. Une sous-extension $F'/F$ de 
$\Omega/F$ est sŽparable si et seulement elle vŽrifie les conditions 
Žquivalentes suivantes:
\begin{enumerate}
\item[(i)] $F'$ et $F^\mathrm{rad}$ sont linŽairement disjointes sur $F$;
\item[(ii)] $F'$ est linŽairement disjointe de toute sous-extension radicielle de $\Omega/F$. 
\end{enumerate}
\end{proposition}

\begin{corollary}\label{cor au critre de sŽparabilitŽ 2}
Soit $F'/F$ une extension sŽparable (algŽbrique ou non). La $F'$-algbre $F'\otimes_F F^\mathrm{rad}$ est un corps. Si 
de plus $F'/F$ est une sous-extension de $\overline{F}/F$, alors $F'\otimes_F F^\mathrm{rad}$ est la cl™ture radicielle de $F'$ dans $\overline{F}$.
\end{corollary}

\begin{remark}\label{sŽparabilitŽ du complŽtŽ}
\textup{Soit $F$ un corps global. Il est bien connu que pour toute place $v$ de $F$, le complŽtŽ $F_v$ de $F$ en $v$ est une $F$-algbre sŽparable. 
Si $F$ est un corps de nombres, c'est immŽdiat (\ref{critre sŽparabilitŽ 2}). En gŽnŽral, il suffit de voir que 
pour toute sous-extension finie $F'/F$ de $\overline{F}/F$, 
l'anneau $F'_v= F_v \otimes_F F'$ est isomorphe ˆ un produit de corps $\prod_{w\vert v}F'_{w}$; 
c'est mme un corps si $F'\subset F^\mathrm{rad}$ car dans ce cas il y a une unique 
place $w$ de $F'$ au-dessus de $v$. D'o le rŽsultat d'aprs le critre \ref{critre sŽparabilitŽ 1}\,(ii). 
On en dŽduit aussi que la $F$-algbre 
$F_v \otimes_F F^\mathrm{rad}$ est un corps (\ref{cor au critre de sŽparabilitŽ 2}). Observons que le degrŽ de transcendance de $F_v/F$ 
est infini, et mme non dŽnombrable (sinon $F_v$ lui-mme serait dŽnombrable). 
}\end{remark}

\subsection{Descente sŽparable non algŽbrique}\label{descente sŽparable non algŽbrique}
On s'intŽresse ici ˆ la descente des variŽtŽs relativement ˆ une extension sŽparable non algŽbrique 
$F'/F$. Soit $$ \mathrm{Aut}_F(F')$$ le groupe des automorphismes du corps $F'$ qui fixent $F$. On note $F'^{\mathrm{Aut}_F(F')}$ 
le sous-corps de $F'$ formŽ des ŽlŽments fixŽs par $\mathrm{Aut}_F(F')$. 
On a bien sžr toujours l'inclusion $$F \subset F'^{\mathrm{Aut}_F(F')}\ptf$$ Si $F'=F'^{\mathrm{s\acute{e}p}}$ et $F'^{\mathrm{Aut}_F(F')}= F$, 
le \guill{critre galoisien} \ref{critre galoisien} est encore vrai ici\footnote{Observons 
que dans la terminologie d'Artin, une extension de corps $F'/F$ est 
\textit{galoisienne} si le sous-corps des points fixes $F'^{\mathrm{Aut}_F(F')}\subset F'$ est Žgal ˆ $F$.}: 

\begin{proposition}\label{critre galoisien bis}
Soit $F'/F$ une extension sŽparable. On suppose que $F'=F'^{\mathrm{s\acute{e}p}}$ et $F'^{\mathrm{Aut}_F(F')}= F$.  
Soit $V$ une $F$-variŽtŽ et soit $X$ une sous-variŽtŽ fermŽe de $V$. Les conditions suivantes sont Žquivalentes:
\begin{enumerate}
\item[(i)] $X$ est dŽfinie sur $F$;
\item[(ii)] $X$ est dŽfinie sur $F'$ et $X(F')$ est $\mathrm{Aut}_F(F')$-stable;
\item[(iii)] il existe un sous-ensemble 
$\mathrm{Aut}_F(F')$-stable de $X(F')= X\cap V(F')$ qui soit Zariski-dense dans $X(\smash{\overline{F}}')$. 
\end{enumerate}
\end{proposition}

\begin{proof}
Elle est identique ˆ celle de \ref{critre galoisien} (\textit{cf.} \cite[ch.~AG, 14.1-14.4]{B}), compte-tenu de la densitŽ de $V(F')$ dans $V(\smash{\overline{F}}')$. 
\end{proof}

Sous les hypothses de \ref{critre galoisien bis}, on obtient comme dans le cas galoisien un foncteur pleinement fidle 
$$V \mapsto (\hbox{$V_{F'}$, $\mathrm{Aut}_F(F')$-ensemble $V(F')$})\ptf\leqno{(1)}$$ 

\begin{remark}
\textup{
\begin{enumerate}
\item[(i)] Soit $\mathfrak{B}$ une base de transcendance de $F'/F$. L'extension $F'/F(\mathfrak{B})$ est algŽbrique sŽparable et mme galoisienne 
puisque $F'=F'^{\mathrm{s\acute{e}p}}$; de plus $F'^{\mathrm{Gal}(F'/F(\mathfrak{B}))}= F(\mathfrak{B})$ et 
$$F'^{\mathrm{Aut}_F(F')}=F\quad\hbox{si et seulement si}\quad F(\mathfrak{B})^{\mathrm{Aut}_F(F(\mathfrak{B}))}= F\ptf$$ Cette condition est toujours rŽalisŽe si 
$F'/F$ est de degrŽ de transcendance infini, \cad si la base $\mathfrak{B}$ est infinie. Dans le cas contraire, elle peut ne pas l'tre (voir (ii)). 
\item[(ii)]
Prenons pour $F'$ une cl™ture sŽparable algŽbrique $F(t)^\mathrm{s\acute{e}p}$ de l'extension purement transcendante $F(t)$ de $F$. 
On a $F'^{\mathrm{Aut}_F(F')}= F(t)^{\Omega}$ avec $\Omega= \mathrm{Aut}_F(F(t))$.  
Pour $g = \left(\begin{array}{cc}a & b \\ c & d \end{array}\right)\in GL(2,F)$, notons $\sigma_g$ le 
$F$-automorphisme de $F(t)$ dŽfini par l'homographie $\sigma_g(t) = (a t + b)(ct + d)^{-1}$. 
L'application $g\mapsto \sigma_g$ se quotiente est un isomorphisme de $\mathrm{PGL}_2(F)$ sur $\Omega$. On a $F(t)^\Omega = F$ si et seulement si 
le corps $F$ est infini. Si $F= \mbb{F}_q$ et $\alpha \in F(t)^\Omega \smallsetminus F$, l'automorphisme $x\mapsto \alpha x$ de la droite affine est $\Omega$-Žquivariant mais il 
n'est certainement pas dŽfini sur $F$!
\item[(iii)] Supposons $F=F^\mathrm{s\acute{e}p}$. On prouve facilement qu'un $F$-schŽma intgre de type fini a au plus un $F$-modle. 
Soit $F''= F(\mathfrak{B})$ pour une base de transcendance $\mathfrak{B}$ de $F'/F$. Pour dŽcrire l'image du foncteur \ref{descente sŽparable non algŽbrique}\,(1), il suffit de dŽcrire celle du foncteur 
\ref{notations et rappels}\,(1) pour $F=F''$ (\cad le foncteur qui ˆ une $F''$-variŽtŽ $W$ associŽ la $F'\;(=F''^{\mathrm{s\acute{e}p}})$-variŽtŽ $W_{F'}$) et celle du foncteur $$V \mapsto V_{F''}\ptf$$ 
En reprenant les arguments de \cite{Der}, on prouve que si $\mathfrak{B}$ est infini, pour qu'un $F$-schŽma intgre de type $V''$ admette un $F$-modle, il faut et il suffit que 
pour tout $\sigma\in \mathrm{Aut}_{F}(F'')$, le $F''$-schŽma conjuguŽ ${^\sigma{V}}$ soit isomorphe ˆ $V$. 
\end{enumerate}}
\end{remark}

\section{Sur la c-$F$-topologie}\label{annexe B}

Dans cette annexe, les hypothses sont celles de \ref{le thm de KR rat}: $F$ est un corps commutatif d'exposant caractŽristique $p\geq 1$, 
$G$ est un groupe rŽductif connexe dŽfini sur $F$ et $V$ est une $G$-variŽtŽ affine elle aussi dŽfinie sur $F$. 

Pour une sous-variŽtŽ fermŽe $\ES{Q}$ de $V$ de $G$ dŽfinie sur $F$, on a dŽfini en \ref{le thm de KR rat} le sous-ensemble ${_F\ES{N}^G(V,\ES{Q})}\subset V$ des ŽlŽments qui sont 
$(F,G,\ES{Q})$-instables (\ref{ŽlŽments (F,G,Q)-instables}). Pour $v\in {_F\ES{N}^G(V,\ES{Q})}$, on a une thŽorie de l'optimalitŽ pour les co-caractres $\lambda\in \check{X}_F(G)$ 
tels que la limite $\lim_{t\rightarrow 0}t^\lambda\cdot v$ existe et appartienne ˆ $\ES{Q}$ (\ref{HMrat1}). Cela s'applique en particulier au cas o $\ES{Q}$ est une $G$-orbite fermŽe dans 
$V$ (\textit{e.g.} $\ES{Q}=\{e_V\}$ si $(V,e_V)$ est une $G$-variŽtŽ pointŽe dŽfinie sur $F$ avec $e_V\in V(F)$). Pour $v\in V$ tel que l'unique $G$-orbite fermŽe $\ES{F}_v$ dans $\overline{\ES{O}_v}$ 
ne soit pas dŽfinie sur $F$, on n'a \textit{a priori} pas de thŽorie de l'optimalitŽ. La $\cF$-topologie introduite dans \cite{BHMR} permet cependant de prouver des rŽsultats (\textit{cf.} \ref{topo co-car}). 
On applique en \ref{application au cas tordu} ces rŽsultats au cas o $V$ est un $G$ espace tordu $\wt{G}$ tel que $\wt{G}(F)\neq \emptyset$, muni de l'action de $G$ par conjugaison; \cad dans le cadre de la formule des traces tordue.

\subsection{La topologie donnŽe par les co-caractres rationnels}\label{topo co-car}
Dans cette sous-section, on reprend les notions et les rŽsultats de \cite{BHMR}. 

\begin{definition}\label{dŽfinition c-F-top}
\textup{Soit $X$ un sous-ensemble de $V$ (on ne demande pas que $X$ soit contenu dans $V(F)$, ni que $G(F)\cdot X= X$).
\begin{enumerate}
\item[(i)]On dit que $X$ est \textit{$\cFG$-fermŽ} ou simplement \textit{$\cF$-fermŽ} (dans $V$) si pour tout $v\in X$ et tout $\lambda\in \Lambda_{F,v}$, la limite 
$\phi_{v,\lambda}^+(0)$ est dans $X$.
\item[(ii)]On appelle \textit{$\cFG$-fermeture} ou simplement \textit{$\cF$-fermeture} de $X$ (dans $V$) le plus petit sous-ensemble $X'$ de $V$ tel que 
$X\subset X'$ et $X'$ soit c-$F$-fermŽ; on le note $\smash{\overline{X}}^{(\cF)}= \smash{\overline{X}}^{(\cFG)}$. 
\end{enumerate}}
\end{definition}

Les sous-ensembles c-$F$-fermŽs forment les parties fermŽes d'une topologie sur $V$: 
\begin{itemize}
\item une intersection quelconque de sous-ensembles c-$F$-fermŽs est c-$F$-fermŽe; 
\item une rŽunion quelconque de sous-ensembles c-$F$-fermŽs est c-$F$-fermŽe; 
\item l'ensemble vide et l'ensemble $V$ tout entier sont c-$F$-fermŽs. 
\end{itemize}
Observons que si $G$ opre trivialement sur $V$, alors la c-$F$-topologie sur $V$ est la topologie discrte. 
Un sous-ensemble $G$-invariant (Zariski-)fermŽ de $V$ est c-$F$-fermŽ.

\begin{remark}\label{Zariski/c}
\textup{Soit $v\in V$. 
\begin{enumerate}
\item[(i)] La $G(F)$-orbite $\ES{O}_{F,v}=\{g\cdot v\,\vert g\in G(F)\}$ est c-$F$-fermŽe si et seulement si pour tout 
$\lambda\in \check{X}_F(G)$ tel que la limite $\lim_{t\rightarrow 0} t^\lambda\cdot v$ existe, cette limite est dans $\ES{O}_{F,v}$.
\item[(ii)]Si $F=\overline{F}$, d'aprs le critre de Hilbert-Mumford, la $G$-orbite $\ES{O}_v=\{g\cdot v\,\vert g\in G\}$ 
est fermŽe (dans $V$) si et seulement si elle est c-$\overline{F}$-fermŽe. 
\item[(iii)]Si $F$ est parfait et $v\in V(F)$, on a aussi (d'aprs \ref{Kempf3}): 
la $G(F)$-orbite $\ES{O}_{F,v}$ est c-$F$-fermŽe si et seulement si la $G$-orbite 
$\ES{O}_v$ est fermŽe. Cela n'est plus vrai en gŽnŽral si $F$ n'est pas parfait: la $\mathrm{GL}_p(F)$-orbite de l'ŽlŽment $\gamma$ de l'exemple 
\ref{exemple ŽlŽment primitif sur un corps non parfait} est c-$F$-fermŽe mais sa $\mathrm{GL}_p$-orbite n'est pas fermŽe; ou, ce qui revient au mme, sa $\mathrm{GL}_p(F^\mathrm{rad})$-orbite n'est pas 
c-$F^\mathrm{rad}$-fermŽe.
\end{enumerate}
}
\end{remark}

On peut munir $V(F)$ de la topologie induite par n'importe quelle topologie sur $V$; en particulier la topologie de Zariski ou la c-$F$-topologie. 
Commenons par deux lemmes trs simple: 

\begin{lemme}\label{fermeture et invariance}
Soient $H$ un sous-groupe fermŽ de $G$ dŽfini sur $F$ et $\Omega$ un sous-ensemble $H(F)$-invariant et (Zariski-)fermŽ de $V(F)$. 
La fermeture $\overline{\Omega}$ de $\Omega$ dans $V$ 
est une sous-variŽtŽ fermŽe de $V$, dŽfinie sur $F$ et telle que $\overline{\Omega}(F)= \Omega$. 
Si de plus $H(F)$ est dense dans $H$ (\textit{e.g.} si $H$ est rŽductif connexe et $F$ est infin \cite[ch.~V, 18.3]{B}), 
alors $H\cdot \overline{\Omega}= \overline{\Omega}$.  
\end{lemme}

\begin{proof}La variŽtŽ $\overline{\Omega}$ est fermŽe (par dŽfinition) et dŽfinie sur $F$ d'aprs le critre galoisien \ref{critre galoisien}\,(iii); 
et puisque 
$\Omega$ est fermŽ dans $V(F)$ et contenu dans $V(F)\cap \overline{\Omega}= \overline{\Omega}(F)$, on a l'ŽgalitŽ 
$\Omega = \overline{\Omega}(F)$. 
Notons $X$ la fermeture de $H\cdot \overline{\Omega}$ dans $V$. C'est une sous-variŽtŽ fermŽe de $V$, 
dŽfinie sur $F^{\rm rad}$ (ˆ nouveau d'aprs le critre galoisien) et 
$H$-invariante. Le morphisme produit $\alpha : H \times \overline{\Omega} \rightarrow X$ est dominant et dŽfini sur $F$. Si $H(F)$ 
est dense dans $H$, alors $\alpha (H(F)\times \Omega)$ est dense dans 
$\overline{\mathrm{Im}(\alpha)}=X$. Or $\alpha(H(F)\times \Omega)= \Omega$, par consŽquent $\overline{\Omega}=X$, 
ce qui prouve que $H\cdot \overline{\Omega}= \overline{\Omega}$. 
\end{proof}

\begin{lemme}
On suppose que le corps $F$ est infini. Si $X\subset V(F)$ est un sous-ensemble $G(F)$-invariant et (Zariski-)fermŽ dans $V(F)$, il est c-$F$-fermŽ dans $V$. 
\end{lemme}

\begin{proof}
Soit $\overline{X}$ la fermeture de Zariski de $X$ dans $V$. D'aprs \ref{fermeture et invariance}, c'est une sous-variŽtŽ 
fermŽe de $V$, dŽfinie sur $F$ et $G$-invariante, telle $\overline{X}(F)= X$. 
Si $v\in X$ et $\lambda \in \Lambda_{F,v}$, l'image du $F$-morphisme $\phi_{\lambda,v}: \mbb{G}_{\rm m} \rightarrow V$ est contenue dans $ \overline{X}$, par 
consŽquent il se prolonge (de manire unique) en un $F$-morphisme $\phi_{\lambda,v}^+: \mbb{G}_{\rm a} \rightarrow \overline{X}$. Donc 
$\phi_{\lambda,v}^+(0) \in \overline{X}(F)= X$. 
\end{proof}

Pour $v\in V$, la c-$F$-fermeture $\smash{\overline{X}}^{(\cF)}$ de la $G(F)$-orbite $X=\ES{O}_{F,v}$ est rŽunion de 
$G(F)$-orbites dont les points s'obtiennent tous ˆ partir de $\ES{O}_{F,v}$ par limites successibles ˆ l'aide de co-caractres $F$-rationnels de $G$ 
\cite[lemma~3.3\,(i)]{BHMR}. En particulier si $v\in V(F)$, la $\cF$-fermeture $\smash{\overline{X}}^{(\cF)}$ de $X=\ES{O}_{F,v}$ est contenue dans $V(F)$. 
On a donnŽ en \ref{HMrat1} une version rationnelle du critre de Hilbert-Mumford. 
Le rŽsultat suivant \cite[theo.~4.3]{BHMR} est une autre version rationnelle du mme critre.

\begin{theorem}\label{HMrat2}
Soit $v\in V$. Soit $\smash{\overline{X}}^{(\cF)}$ la c-$F$-fermeture de $X=\ES{O}_{F,v}$.
\begin{enumerate}
\item[(i)]Il existe une unique $G(F)$-orbite c-$F$-fermŽe dans $V$ qui soit contenue dans $\smash{\overline{X}}^{(\cF)}$; on la note $\ES{F}_{F,v}$.
\item[(ii)] Si $v'\in \smash{\overline{X}}^{\cF}$, alors $\ES{F}_{F,v'}= \ES{F}_{F,v}$.
\item[(iii)]Il existe un $\lambda \in \Lambda_{F,v}$ tel que la limite $\phi_{v,\lambda}^+(0)$ soit dans $\ES{F}_{F,v}$.
\end{enumerate}
\end{theorem}

D'aprs \cite[cor.~4.5]{BHMR}, on a le 

\begin{corollary}
Pour $v' \in \smash{\overline{X}}^{\cF}$, il existe un $\lambda' \in \Lambda_{F,v}$ tel que la limite $\phi_{F,v}^+(0)$ soit dans  
la c-$F$-fermeture $\smash{\overline{X}}'^{\cF}$ de la $G(F)$-orbite $X'= \ES{O}_{F,v'}$. 
\end{corollary}

\begin{remark}\label{HMrat3}
\textup{
Soit $v\in V(F)$.
\begin{enumerate}
\item[(i)] La $G(F)$-orbite c-$F$-fermŽe $\f_{F,v}$ est contenue dans $V(F)$. 
Supposons de plus qu'elle rencontre la $G$-orbite fermŽe $\ES{F}_v$, \cad que $\ES{F}_{F,v} \cap \ES{F}_v \neq \emptyset$. 
Alors on a l'inclusion $\ES{F}_{F,v} \subset \ES{F}_v(F)$ et d'aprs \ref{HMrat2}\,(iii), $v$ est $(F,G,\ES{F}_v)$-instable. D'aprs 
\ref{HMrat1}, l'ensemble $\Lambda_{F,\ES{F}_v,v}^\mathrm{opt}$ est non vide et pour tout $\lambda \in \Lambda_{F,v}^\mathrm{opt}$, 
la limite $\phi_{v,\lambda}^+(0)$ est dans $ \ES{F}_{F,v}$.
\item[(ii)] Le thŽorme \ref{HMrat1} est une gŽnŽralisation partielle du thŽorme de Kempf (\ref{Kempf1}).
 Pour avoir une gŽnŽralisation complte, 
il faudrait savoir traiter le cas o la $G(F)$-orbite c-$F$-fermŽe $\ES{F}_{F,v}$ ne rencontre pas la 
$G$-orbite fermŽe $\ES{F}_v$. Dans ce cas, le problme suivant reste ouvert: existe-t-il un co-caractre 
$\lambda \in \Lambda_{F,v}$ tel que $\phi_{v,\lambda}^+(0)$ appartienne ˆ $ \ES{F}_{F,v}$ et $\phi_{v,\lambda}(t)$ tende vers $\ES{F}_{F,v}$ 
(quand $t$ tend vers $0$) \guill{plus vite} que pour tout autre co-caractre $\lambda'\in \Lambda_{F,v}$? 
\item[(iii)]On suppose que $F$ est un corps commutatif topologique sŽparŽ et non discret. Munissons $V(F)$ et $G(F)$ de la topologie dŽfinie par $F$ (notŽe \TopF, \textit{cf.} \ref{le cas d'un corps top}). Si 
la $G(F)$-orbite $\ES{O}_{F,v}$ est \TopF-fermŽe dans $G(F)$, alors elle est c-$F$-fermŽe dans $V$. 
La rŽciproque est-elle vraie en gŽnŽral? Dans le cas o $G=\mathrm{GL}_n$ opre sur $V=G$ par conjugaison, la rŽponse est oui. 
En effet dans ce cas, les $G(F)$-orbites (de $G(F)$) qui sont \TopF-fermŽes dans $G(F)$ sont celles dont le polyn™me minimal est produit de polyn™mes irrŽductibles sur $F$ 
deux-ˆ-deux distincts (ce sont les classes de conjugaison \textit{semi-simples} au sens de Bourbaki, les classes de conjugaison 
\textit{absolument semi-simples} Žtant celles 
dont le polyn™me minimal est produit de polyn™mes irrŽductibles \textit{sŽparables} deux-ˆ-deux distincts). D'aprs \cite[prop.~10.1]{BHMR}, 
les $G(F)$-orbites qui sont \TopF-fermŽes dans $G(F)$ sont exactement celles qui sont c-$F$-fermŽes dans $G$ (voir aussi \ref{application au cas tordu}).
\end{enumerate}
}
\end{remark}

\subsection{Application aux $G(F)$-orbites dans $\wt{G}(F)$}\label{application au cas tordu} On considre ici le cas o $V$ est 
un $G$-espace tordu $\wt{G}$ dŽfini sur $F$ avec $\wt{G}(F)\neq \emptyset$, muni de l'action de $G$ par conjugaison: 
$$g\cdot \gamma = \mathrm{Int}_g(\gamma)\quad \hbox{pour tout} \quad (g,\gamma)\in G\times \wt{G} \ptf$$
Rappelons que pour $\gamma\in \wt{G}$, on a notŽ $\Lambda_\gamma$, resp. $\Lambda_{F,\gamma}$, l'ensemble des $\lambda \in \check{X}(G)$, resp. $\lambda\in \check{X}_F(G)$ tels que la limite 
$\lim_{t\rightarrow 0} \mathrm{Int}_{t^\lambda}(\gamma)$ 
existe; on note alors $\gamma_\lambda $ cette limite\footnote{PrŽcedemment notŽe $\phi_{\gamma,\lambda}^+(0)$.}. 

Pour $\lambda \in \check{X}(G)$, on pose\index{Ptildelambda@$\wt{P}_\lambda = \wt{M}_\lambda\ltimes U_\lambda$} 
$$ \wt{P}_\lambda = \{\gamma \in \wt{G} \,\vert\, \lambda \in \Lambda_\gamma \}\vg$$
$$ \wt{M}_\lambda = \{\gamma \in \wt{G} \,\vert \, \mathrm{Int}_\gamma \circ \lambda = \lambda \}\ptf$$

Rappelons la dŽfinition de \textit{sous-espace parabolique} de $\wt{G}$ \cite[2.7]{LW}. 
Pour un sous-groupe parabolique $P$ de $G$, on note $\wt{P}$ le normalisateur de $P$ dans $\wt{G}$. Si $\wt{P}$ est non vide, on dit que c'est un 
\textit{sous-espace parabolique} de $\wt{G}$ et on appelle \textit{composante de Levi} de $\wt{P}$ le normalisateur $\wt{M}$ dans $\wt{P}$ d'une composante 
de Levi $M$ de $P$. On a alors la dŽcomposition de Levi $$\wt{P} = \wt{M} \ltimes U_P\ptf$$
Si $P$ est dŽfini sur $F$, alors $\wt{P}$ l'est aussi et puisque par hypothse $\wt{G}(F)$ est non vide, 
la condition $\wt{P}\neq \emptyset$ entra"ne $\wt{P}(F)\neq \emptyset$. 

\begin{lemme}\label{description sous-espace-para-lambda}
Soit $\lambda \in \check{X}(G)$. On suppose que l'ensemble $\wt{P}_\lambda$ n'est pas vide. 
\begin{enumerate}
\item[(i)] $\wt{P}_\lambda$ est un sous-espace parabolique de $\wt{G}$; c'est le normalisateur de $P_\lambda$ dans $\wt{G}$. 
\item[(ii)] $\wt{M}_\lambda$ est une composante de Levi de $\wt{P}_\lambda$; c'est le normalisateur de $M_\lambda$ 
dans $\wt{P}_\lambda$.
\item[(iii)] $\lambda$ est ˆ valeurs dans le centre $Z(\wt{M}_\lambda)=Z(M_\lambda)^\theta$ de $\wt{M}_\lambda$, 
\cad le fixateur dans  $Z(M_\lambda)$ de l'automorphisme $\theta= \theta_{\smash{\wt{M}_\lambda}}$ de $Z(M_\lambda)$ dŽfini par $\wt{M}_\lambda$. 
\item[(iv)] Pour $\gamma \in \wt{P}_\lambda$, en Žcrivant $\gamma= \delta u$ avec $\delta\in \wt{M}_\lambda$ et $u\in U_\lambda$, on a $\gamma_\lambda = \delta$.
\end{enumerate}
\end{lemme}

\begin{proof}
Soit un ŽlŽment $\gamma\in \wt{P}_\lambda$. Pour $g\in G$, 
on a $$\lim_{t\rightarrow 0} \mathrm{Int}_{t^\lambda}(g\gamma)= \left(\lim_{t\rightarrow 0}\mathrm{Int}_{t^\lambda}(g) \right) \gamma_\lambda\vg$$
par consŽquent $g\gamma \in \wt{P}_\lambda$ si et seulement si $g\in P_\lambda$. De la mme manire on a  $\gamma g\in \wt{P}_\lambda$ si et seulement 
si $g\in P_\lambda$. Puisque $\gamma g = \mathrm{Int}_\gamma(g) \gamma$, on en dŽduit que $\wt{P}_\lambda$ est le norma\-lisateur de $P_\lambda$ dans $\wt{G}$. 
Cela prouve (i).

Notons $\wt{M}$ le normalisateur de $M_\lambda$ dans $\wt{P}$ (c'est une composante de Levi de $\wt{P}_\lambda$). 
Puisque l'ensemble $\wt{P}_\lambda$ est Zariski-fermŽ, la limite 
$\gamma_\lambda = \lim_{t\rightarrow} \mathrm{Int}_{t^\lambda}(\gamma)$ est dans  $ \wt{P}_\lambda$. Cette limite se calcule comme suit: Žcrivons 
$\gamma = \delta u $ avec $\delta \in \wt{M}$ et $u\in U_P$. Puisque $\lim_{t\rightarrow 0}\mathrm{Int}_{t^\lambda}(u)=1$, on a 
$$\gamma_\lambda = \left( \lim_{t\rightarrow 0}\mathrm{Int}_{t^\lambda}(\delta)\right)
= \left(\lim_{t\rightarrow 0} t^{\lambda'} \right) \delta \quad \hbox{avec} \quad \lambda' = (1-\theta)\circ \lambda$$ o $\theta= \theta_{\smash{\wt{M}}}$ 
est l'automorphisme de $Z_{M_\lambda}$ dŽfini par $\wt{M}$. Or la limite $\lim_{t\rightarrow 0} t^{\lambda'}$ existe si et seulement si $\lambda'=1$. On a donc 
$\theta\circ \lambda=\lambda$. Par consŽquent $\wt{M} = M_\lambda\delta = \delta M_\lambda$ est contenu dans le stabilisateur $\wt{M}_\lambda$ de $\lambda$ 
dans $\wt{G}$. En particulier 
l'ensemble $\wt{M}_\lambda$ est non vide et comme par dŽfinition c'est un $M_\lambda$-espace tordu, l'inclusion $\wt{M}\subset \wt{M}_\lambda$ entra"ne l'ŽgalitŽ 
$\wt{M}= \wt{M}_\lambda$. Cela prouve (ii), (iii) et (iv).
\end{proof}

Si $M$ est un groupe rŽductif connexe dŽfini sur $F$, on a notŽ $A_M$ le sous-tore $F$-dŽployŽ maximal de $Z_M$. Si $\wt{M}$ est 
un $M$-espace tordu dŽfini sur $F$ tel que $\wt{M}(F)\neq \emptyset$, on note $A_{\smash{\wt{M}}}$ le sous-tore $\theta$-invariant maximal 
de $A_M$, \cad la composante neutre du sous-groupe des 
points fixes $A_M^\theta\subset A_M$; o $\theta=\theta_{\smash{\wt{M}}}$ est le $F$-automorphisme de $Z(M)$ dŽfini par $\wt{M}$. 
Pour $\lambda \in \check{X}_F(G)$, 
d'aprs \ref{description sous-espace-para-lambda}\,(iii) on a $$\wt{P}_\lambda \neq \emptyset \Rightarrow \lambda\in \check{X}(A_{\smash{\wt{M}_\lambda}})\ptf$$

\begin{lemme}\label{existence Plambda}
Soit $\wt{P}$ est un sous-espace parabolique de $\wt{G}$ et soit $\wt{M}$ une composante de Levi de $\wt{P}$, tous deux dŽfinis sur $F$. Il existe un 
co-caractre $\lambda \in \check{X}(A_{\smash{\wt{M}}})$ tel que $\wt{P}_\lambda=\wt{P}$ et $\wt{M}_\lambda = \wt{M}$.
\end{lemme}  

\begin{proof}
On fixe une paire parabolique dŽfinie sur $F$ minimale $(P_0,A_0)$ de $G$ qui soit $\theta_0$-invariante, o $\theta_0=\mathrm{Int}_{\delta_0}$ avec 
$\delta_0\in \wt{G}(F)$. Quitte ˆ remplacer $\wt{P}$ et $\wt{M}$ par $\mathrm{Int}_g(\wt{P})$ et $\mathrm{Int}_g(\wt{M})$ 
pour un $g\in G(F)$, on peut supposer que $\wt{P}=P\delta_0$ avec $P\supset P_0$ et $\wt{M}= M\delta_0$ avec $A_M \subset A_0$. 
Notons $A'_0\subset A_0$ la composante neutre de $A_0\cap G_\mathrm{der}$ et $A''_0$ le tore $F$-dŽployŽ de groupe des caractres 
$X(A''_0)= \mbb{Z} [\Delta_0]$. 
L'inclusion $X(A''_0) \subset X(A'_0)$ correspond ˆ un morphisme surjectif de tores $\phi:A'_0 \rightarrow A''_0$. 
L'action de $\theta_0$ sur $\Delta_0$ induit une action sur $A''_0$ qui rend ce morphisme $\theta_0$-Žquivariant. 
Choisissons un co-caractre $\mu\in \check{X}(A''_0)$ tel que 
$$\left\{
\begin{array}{ll}
\langle \alpha , \mu\rangle  = 0 & \hbox{si $\alpha \in \Delta_0^M $}\\
\langle \alpha , \mu \rangle = \langle \theta_0(\alpha),\mu \rangle >0 & \hbox{si $\alpha \in \Delta_0\smallsetminus \Delta_0^M$}
\end{array}
\right..
$$
Le groupe $X(A''_0)$ est d'indice fini dans $X(A'_0)$. Il induit donc dualement un morphisme injectif $\check{X}(A'_0)\rightarrow \check{X}(A''_0)$ qui 
fait de $\check{X}(A'_0)$ un sous-groupe d'indice fini de $\check{X}(A''_0)$. En consŽquence il existe un entier $m>0$ tel que le co-caractre 
$\lambda = m\mu$ appartienne ˆ $\check{X}(A'_0)$. Par construction on a $P_\lambda = P$, $M_\lambda = M$ et $\theta_0 \circ \lambda = \lambda$. 
Cela dŽmontre le lemme. 
\end{proof}

Pour $\gamma\in \wt{G}(F)$, on pose 
$$\Lambda'_{F,\gamma} \bydef  \Lambda_\gamma \cap (\check{X}_F(G)\smallsetminus \check{X}(A_G)) \subset \Lambda_{F,\gamma}\smallsetminus\{0\}\ptf$$

\begin{lemme}
Un ŽlŽment $\gamma \in \wt{G}(F)$ est primitif (au sens o il n'appartient ˆ aucun sous-espace parabolique propre de $\wt{G}(F)$) 
si et seulement si 
$\Lambda'_{F,\gamma}=\emptyset$.
\end{lemme}

\begin{proof}
Si $\lambda \in \Lambda'_{F,\gamma}$, puisque $\gamma$ appartient ˆ $ \wt{P}_\lambda(F)$, $\gamma$ n'est pas primitif. 
RŽciproquement, si $\gamma$ n'est pas primitif alors il existe un sous-espace parabolique propre $\wt{P}$ de $\wt{G}$ dŽfini sur $F$ tel que 
$\gamma\in \wt{P}(F)$. Choisissons une composante de Levi $\wt{M}$ de $\wt{P}$ dŽfinie sur $F$. D'aprs \ref{existence Plambda}, il existe un co-caractre 
$\lambda\in \check{X}(A_{\smash{\wt{M}}})$ tel que $\wt{P}_\lambda = \wt{P}$ et $\wt{M}_\lambda = \wt{M}$. En Žcrivant $\gamma = \delta u$ avec $\delta\in \wt{M}(F)$ 
et $u\in U_P(F)$, on a 
$\gamma_\lambda = \delta$. Par consŽquent $\lambda$ appartient ˆ $\Lambda_{F,\gamma}$ et mme ˆ $\Lambda'_{F,\gamma}$ 
puisque $P_\lambda = P$ est propre.  
\end{proof}

Pour $\gamma\in \wt{G}(F)$, on note $\ES{O}_{F,\gamma}$ la $G(F)$-orbite de $\gamma$: 
$$\ES{O}_{F,\gamma}= \{\mathrm{Int}_g(\gamma)\,\vert \, g\in G(F)\}\ptf$$ D'aprs \ref{HMrat2}, la c-$F$-fermeture de $\ES{O}_{F,\gamma}$ contient 
une unique $G(F)$-orbite c-$F$-fermŽe, notŽe $\ES{F}_{F,\gamma}$, et il existe un co-caractre $\lambda\in \Lambda_{F,\gamma}$ tel que la limite 
$\gamma_\lambda$ soit dans $\ES{F}_{F,\gamma}$. Si de plus $\ES{F}_{F,\gamma}$ rencontre l'unique orbite Zariski-fermŽe 
$\ES{F}_\gamma$ contenue dans la fermeture de Zariski de la $G$-orbite $\ES{O}_\gamma =\{\mathrm{Int}_g(\gamma)\,\vert \, g\in G\}$,  
alors d'aprs \ref{HMrat3}\,(i), on peut choisir $\lambda$ dans $\Lambda_{F,\ES{F}_\gamma,\gamma}^\mathrm{opt}$.

Soit $\mathfrak{o}=[\wt{M},\delta]$ une classe de $G(F)$-conjugaison de paires primitives dans $\wt{G}(F)$. \`A cette classe $\mathfrak{o}$ est associŽ 
un sous-ensemble $G(F)$-invariant $\ES{O}_{\mathfrak{o}}$ de $\wt{G}(F)$ : 
$$\ES{O}_{\mathfrak{o}}= \bigcup_{\wt{Q}\in \ES{P}(\wt{M})}\bigcup_{u\in U_Q(F)}\ES{O}_{F,\delta u} \ptf$$
Rappelons que $\wt{G}(F)$ est l'union disjointe des $\ES{O}_{\mathfrak{o}}$ pour $\mathfrak{o}$ parcourant l'ensemble  
$\mathfrak{O}$ des classes de $G(F)$-conjugaison de paires primitives dans $\wt{G}$. Pour $\gamma \in \wt{G}(F)$, notons 
$\mathfrak{o}(\gamma)$ l'unique ŽlŽment de $\mathfrak{O}$ tel que $\gamma \in \ES{O}_{\mathfrak{o}(\gamma)}$. Il est dŽfini comme suit \cite[3.3.1]{LL}:
\begin{itemize}
\item on choisit un $\wt{P}\in \wt{\ES{P}}$ tel que $\ES{O}_{F,\gamma} \cap \wt{P}(F)\neq \emptyset $ et $\wt{P}$ soit minimal pour cette propriŽtŽ; 
\item on choisit une composante de Levi $\wt{M}$ de $\wt{P}$ dŽfinie sur $F$ et un ŽlŽment $g\in \wt{G}(F)$ tel que $g\gamma g^{-1}\in \wt{P}(F)$; 
\item on Žcrit $g\gamma g^{-1} = \delta u $ avec $\delta \in \wt{M}(F)$ et $u\in U_P(F)$; 
\end{itemize}
alors $\mathfrak{o}(\gamma) = [\wt{M},\delta]$.

\begin{remark}\label{dŽfinition plus simple}
\textup{D'aprs la preuve de \cite[3.3.1]{LL}, on a la caractŽrisation Žquivalente suivante de $\mathfrak{o}(\gamma)$: 
\begin{itemize}
\item on choisit un $\wt{P}_1\in \wt{\ES{P}}$ tel que $\gamma \in \wt{P}(F)$ et $\wt{P}_1$ soit minimal pour cette propriŽtŽ; 
\item on choisit une composante de Levi $\wt{M}_1$ de $\wt{P}$ dŽfinie sur $F$;  
\item on Žcrit $\gamma = \delta_1 u_1$ avec $\delta \in \wt{M}_1(F)$ et $u\in U_{P_1}(F)$; 
\end{itemize}
alors $\mathfrak{o}(\gamma)=[\wt{M}_1,\delta_1]$. 
}\end{remark}

\begin{lemme}\label{caractŽrisation Oo}
Soit $\mathfrak{o}=[\wt{M}, \delta]\in \mathfrak{O}$. 
\begin{enumerate}
\item[(i)]$ \ES{O}_{\mathfrak{o}}= \{\gamma\in \wt{G}(F)\,\vert \, \hbox{il existe un $\lambda \in \Lambda_{F,\gamma}$ 
tel que $\gamma_\lambda \in  \ES{O}_{F,\delta} $}\}$.
\item[(ii)] L'ensemble $\ES{O}_{F,\delta}$ est 
l'unique $G(F)$-orbite c-$F$-fermŽe contenue dans $\ES{O}_{\mathfrak{o}}$.
\item[(iii)]L'ensemble $\ES{O}_{\mathfrak{o}}$ est c-$F$-fermŽ.
\end{enumerate}
\end{lemme}

\begin{proof}
Prouvons (i).  Notons $\ES{O}'_\mathfrak{o}$ l'ensemble des $\gamma \in \wt{G}(F)$ tels qu'il existe un $\lambda \in \Lambda_{F,\gamma}$ 
avec $\gamma_\lambda\in \ES{O}_{F,\delta}$. 

Prouvons l'inclusion $\ES{O}_{\mathfrak{o}}\subset \ES{O}'_{\mathfrak{o}}$. 
Soient $\wt{Q}\in \ESP(\wt{M})$ et $\gamma = \mathrm{Int}_g(\delta u) $ avec $u\in U_Q(F)$ et $g\in G(F)$. 
D'aprs \ref{existence Plambda}, il existe un $\mu\in \check{X}(A_{\smash{\wt{M}}})$ tel que 
$P_\mu = Q$ et $M_\mu = M$. Alors $\mu\in \Lambda_{F,\delta u}$ et $(\delta u)_\mu = \delta$. On en dŽduit que le co-caractre 
$\lambda = \mathrm{Int}_g\circ \mu$ appartient ˆ $\Lambda_{F,\gamma}$ et vŽrifie $\gamma_\lambda = \mathrm{Int}_g(\delta)$. Par consŽquent $\gamma$ appartient 
ˆ $\ES{O}'_{\mathfrak{o}}$. 

Prouvons l'inclusion $\ES{O}'_\mathfrak{o}\subset \ES{O}_\mathfrak{o}$. Soient $\gamma\in \wt{G}(F)$, $\lambda\in \Lambda_{F,\gamma}$ et $g\in G(F)$ tels que 
$\gamma_\lambda = \mathrm{Int}_g(\delta)$. Quitte ˆ remplacer $\gamma$ par $\mathrm{Int}_{g^{-1}}(\gamma)$ et $\lambda$ par $\mathrm{Int}_{g^{-1}}\circ \lambda$, on peut 
supposer que $\gamma_\lambda= \delta$. On a donc $\gamma = \delta u$ avec $\delta\in \wt{M}_\lambda(F)$ et $u\in U_\lambda(F)$. 
Soit $\wt{Q}_1$ un sous-espace parabolique de $\wt{M}_\lambda$ 
dŽfini sur $F$ tel que $\delta\in \wt{Q}_1(F)$ et $\wt{Q}_1$ soit minimal pour cette propriŽtŽ. Alors $\wt{P}_1= \wt{Q}_1U_\lambda$ est un sous-espace parabolique de $\wt{G}$ dŽfini sur $F$ tel que 
$\delta\in \wt{P}_1(F)$ et $\wt{P}_1$ est minimal pour cette propriŽtŽ. Soit $\wt{M}_1$ une composante de Levi de $\wt{Q}_1$ dŽfinie sur $F$; c'est aussi une composante de Levi 
de $\wt{P}_1$. \'Ecrivons $\delta= \delta_1u_1$ avec $\delta \in \wt{M}_1(F)$ et $u_1\in U_{Q_1}(F)\;(\subset U_{P_1}(F))$. 
D'aprs \ref{dŽfinition plus simple}, on a $$[\wt{M}_1,\delta_1]= [\wt{M},\delta]\ptf$$ 
Or $\gamma = \delta_1 u_1 u$ avec $u_1u\in U_{P_1}(F)$, par consŽquent $\gamma$ appartient ˆ $\ES{O}_{[\wt{M}_1,\delta_1]}= \ES{O}_{\mathfrak{o}}$. 

Prouvons (ii). Puisque d'aprs (i), les $G(F)$-orbites contenues dans $\ES{O}_{\mathfrak{o}}\smallsetminus \ES{O}_{F,\delta}$ ne sont pas c-$F$-fermŽes, 
il suffit de prouver que pour $\lambda \in \Lambda_{F,\delta}$, on a $\delta_\lambda \in \ES{O}_{F,\delta}$. Fixons un tel $\lambda$. Comme l'ŽlŽment $\delta$ appartient ˆ $\wt{P}_\lambda(F)$, 
il s'Žcrit $\delta = \delta' u'$ avec $\delta'\in \wt{M}_\lambda(F)$ et $u'\in U_\lambda(F)$; de plus on a $\delta'= \delta_\lambda$. Soit $\wt{Q}'_1$ 
un sous-espace parabolique de $\wt{M}_\lambda$ 
dŽfini sur $F$ tel que $\delta'\in \wt{Q}'_1(F)$ et $\wt{Q}'_1$ soit minimal pour cette propriŽtŽ. Soit $\wt{M}'_1$ une composante de Levi de $\wt{Q}'_1$ 
dŽfinie sur $F$. \'Ecrivons $\delta'= \delta'_1u'_1$ avec $\delta'_1\in \wt{M}'_1(F)$ et $u'_1\in U_{Q'_1}(F)$. D'aprs la preuve du point (i), on a 
$$[\wt{M}'_1,\delta'_1]= [\wt{M},\delta]\ptf$$ En posant $X= \ES{O}_{F,\delta}$ et $X' = \ES{O}_{F,\delta'}$, on a donc 
$$X' \subset \smash{\overline{X}}^{(\cF)} \quad \hbox{et}\quad X\subset \smash{\overline{X'}}^{(\cF)};$$ 
ce qui n'est possible que si $X=X'$.   

Prouvons (iii). 
Soit $\gamma\in \ES{O}_{\mathfrak{o}}$. D'aprs (i) et (ii), on a $\ES{F}_{F,\gamma}= \ES{O}_{F,\delta}$. Soit $\lambda \in \Lambda_{F,\gamma}$ 
et posons $\gamma'=\gamma_\lambda$. Il 
s'agit de vŽrifier que $\gamma'$ appartient ˆ $\ES{O}_{\mathfrak{o}}$. Soit $\mathfrak{o}(\gamma')= [\wt{M}'\!,\delta']$.  D'aprs  (i), 
il existe un $\lambda'\in \Lambda_{F,\gamma'}$ 
tel que $\gamma_{\lambda'}\in \ES{O}_{F,\delta'}$. En particulier, la $G(F)$-orbite $\ES{O}_{F,\delta'}$ 
est $2$-accessible ˆ partir de la 
$G(F)$-orbite $\ES{O}_{F,\gamma}$ au sens de \cite[3.2]{BHMR}, ce qui entra"ne \cite[3.3\,(i)]{BHMR} que la $G(F)$-orbite $\ES{O}_{F,\delta'}$ 
est contenue dans la c-$F$-fermeture de $\ES{O}_{F,\gamma}$. On a donc $$\ES{O}_{F,\delta'}= \ES{F}_{F,\gamma'}= \ES{F}_{F,\gamma}= \ES{O}_{F,\delta}\quad\hbox{et}\quad 
\mathfrak{o}(\gamma')=\mathfrak{o}=\mathfrak{o}(\gamma)\ptf $$ Cela achve la dŽmonstration du lemme. 
\end{proof}

\section{Des $F$-strates aux $F_S$-strates}\label{annexe C}

Dans cette annexe, $F$ est un corps global d'exposant caractŽristique $p\geq 1$, \cad un corps de nombres ($p=1$) ou un corps de fonctions ($p>1$). 
Fixons un ensemble fini non vide $S$ de places de $F$. 
On reprend les hypothses de \ref{la variante de Hesselink} et \ref{la stratification de Hesselink}: $(V,e_V)$ est une $G$-variŽtŽ affine pointŽe dŽfinie sur $F$ et 
$e_V\in V(F)$. Dans un premier temps, on ne suppose pas que $e_V$ soit rŽgulier dans $V$ (hypothse \ref{hyp reg}); on le supposera 
ˆ partir de \ref{AF}. En \ref{le cas de la var unip (global)}, on suppose que $V$ est le groupe $G$ lui-mme muni de l'action par conjugaison (avec $e_V=1$).  

\subsection{Les $F_S$-strates, $\mbb{A}$-strates, etc.}\label{F_S-strates et A-strates} 
On note $\mbb{A}= \mbb{A}_F$ l'anneau des adles de $F$. Pour une place $v$ de $F$, on note $F_v$ le complŽtŽ de $F$ en $v$. 
D'aprs le lemme de Krasner, les opŽrations de complŽtion et de cl™ture sŽparable commutent: 
le complŽtŽ $(F^\mathrm{s\acute{e}p})_w$ de $F^\mathrm{s\acute{e}p}$ en une place $w$ au-dessus de $v$ 
est une cl™ture sŽparable $(F_v)^\mathrm{s\acute{e}p}$ de $F_v$. \`A isomorphisme prs, cette cl™ture sŽparable 
ne dŽpend pas du choix de la place $w$. On peut donc noter $F_v^\mathrm{s\acute{e}p}$ ce corps 
$(F_v)^\mathrm{s\acute{e}p}=(F^\mathrm{s\acute{e}p})_w$. Le complŽtŽ $F_v$ de $F$ en $v$ est une $F$-algbre sŽparable (\ref{sŽparabilitŽ du complŽtŽ}), 
par consŽquent $F_v^\mathrm{s\acute{e}p}$ est encore une $F$-algbre sŽparable. 

D'aprs \ref{descente sŽparable c™ne F-nilpotent}, on a $$\NF= V(F) \cap \ES{N}_{F_v}= V(F) \cap \ES{N}_{\smash{F_v^{\mathrm{s\acute{e}p}}}}\ptf$$ 

Pour Žviter les conflits de notations, dans cette sous-section on notera $x$ (au lieu de $v$) les ŽlŽments de $V$. Tout ŽlŽment $x\in V(\mbb{A})$ s'Žcrit 
$x= \prod_{v}x_v$ avec $v\in V(F_v)$, o $v$ parcourt les places de $F$. On pose\index{NA@$\ESN_{\mbb{A}}$} 
$$\ESN_{\mbb{A}}\bydef \left(\prod_v \ESN_{F_v} \right) \cap V(\mbb{A})\ptf$$

\vskip2mm
\textit{\S1.\hskip1mm Les $F_S$-lames et les $F_S$-strates de $\ES{N}_{F_S}$. ---} On pose 
$$F_S= \prod_{v\in S}F_v \quad \hbox{et} \quad F_S^\mathrm{s\acute{e}p}= \prod_{v\in S}F_v^\mathrm{s\acute{e}p}$$ 
Tout ŽlŽment $x\in \ESN_{\mbb{A}}$ dŽfinit un ŽlŽment $x_S= \prod_{v\in S} x_v$ de\index{NFS@$\ESN_{F_S}$} 
$$\ESN_{F_S}\bydef \prod_{v\in S}\ESN_{F_v} \ptf$$ On dŽfinit 
la $F_S$-lame $\scrY_{F_S,x_S}$ et le sous-ensemble $\scrX_{F_S,x_S}$ de $\ESN_{F_S}$ par 
$$ \scrY_{F_S,x_S}= \prod_{v\in S} \scrY_{F_v,x_v}\quad \hbox{et}\quad \scrX_{F_S,x_S}= \prod_{v\in S} \scrX_{F_v,x_v}  \ptf$$ 
On pose aussi $${_{F_S}P_{x_S}}= \prod_{v\in S} {_{F_v}P_{x_v}} \quad \hbox{et}\quad P_{F_S,x_S}= \prod_{v\in S} P_{F_v,x_v} \ptf$$
On dŽfinit de la mme manire la $F_S$-strate $\bsfrY_{F_S,x_S}$ et le sous-ensemble $\bsfrX_{F_S,x_S}$ de $\ESN_{F_S}$. On a donc 
$$\bsfrY_{F_S,x_S}= G(F_S)\cdot \scrY_{F_S,x_S} \quad \hbox{et}\quad \bsfrX_{F_S,x_S}= G(F_S)\cdot \scrX_{F_S,x_S}\ptf$$ 

On note {\TopFS} la topologie (forte) sur $V(F_S)$ dŽfinie par $F_S$ (\textit{cf.} \ref{le cas d'un corps top}). Puisque $F_S$ est un produit (fini) de corps locaux, 
d'aprs \ref{le cas d'un corps top} on a le 

\begin{lemme}\label{rŽsultats pour topFS}
Soit $x_S \in \ES{N}_{F_S}$. 
\begin{enumerate}
\item[(i)] $\scrX_{F_S,x_S}$ et $\bsfrX_{F_S,x_S}$ sont fermŽs dans $G(F_S)$ pour \TopFS.
\item[(ii)] La $F_S$-lame $\scrY_{F_S,x_S}$ est ouverte dans $\scrX_{F_S,x_S}$ pour \TopFS.
\item[(iii)] La $F_S$-strate $\bsfrY_{F_S,x_S}$ est ouverte dans $\bsfrX_{F_S,x_S}$ pour \TopFS.
\end{enumerate} 
\end{lemme}

Puisqu'il n'y a qu'un nombre fini d'ensembles $\bsfrX_{F_S,x_S}$ avec $x_S\in \ES{N}_{F_S}$, on obtient aussi 
(d'aprs \ref{rŽsultats pour topFS}\,(i)) que $\ES{N}_{F_S}$ 
est fermŽ dans $G(F_S)$ pour \TopFS.

\vskip1mm
L'ensemble $\NF$ est plongŽ diagonalement dans $\ESN_{F_S}$. Pour $x\in \NF$, on peut donc dŽfinir comme ci-dessus la 
$F_S$-lame $\scrY_{F_S,x}$ et la $F_S$-strate $ \bsfrY_{F_S,x}$. Pour allŽger l'Žcriture, 
on reprend ici les notations de \ref{notation allŽgŽe}. Pour une $F$-lame $\scrY$ de $\NF$, on note $\scrY_{F_S}$ la $F_S$-lame de $\ESN_{F_S}$ dŽfinie par 
$\scrY_{F_S}=\prod_{v\in S} \scrY_{F_v}$; et pour une $F$-strate de $\bsfrY$, on note $\bsfrY_{F_S}$ la $F_S$-strate de $\ESN_{F_S}$ dŽfinie 
par $\bsfrY_{F_S}= \prod_{v\in S}\bsfrY_{F_v}$. Pour $\scrX= \scrX(\scrY)$ et $\bsfrX= \bsfrX(\bsfrY)$, on dŽfinit de la mme manire les sous-ensembles $\scrX_{F_S}$ et $\bsfrX_{F_S}$ 
de $\ESN_{F_S}$. D'aprs \ref{descente sŽparable lames} et \ref{variante de BMRT(4.7)}, on a les ŽgalitŽs 
$$\scrY = V(F)\cap \scrY_{F_S} \quad \hbox{et}\quad \bsfrY= V(F) \cap \bsfrY_{F_S};$$ 
et on a aussi 
$$\scrX = V(F)\cap \scrX_{F_S} \quad \hbox{et}\quad \bsfrX= V(F) \subset  \bsfrX_{F_S}$$ 

\vskip2mm
\textit{\S2.\hskip1mm Les $\mbb{A}$-lames $\scrY_{\mbb{A}}$ et les $\mbb{A}$-strates $\bsfrY_{\mbb{A}}$ de $\ES{N}_{\mbb{A}}$\footnote{On peut dŽfinir des $\mbb{A}$-lames et $\mbb{A}$-strates 
plus gŽnŽrales mais nous 
n'en aurons pas besoin ici.}. ---} L'ensemble $\NF$ est plongŽ digonalement dans $\ESN_{\mbb{A}}$. 
\`A toute $F$-lame $\scrY$ de $\NF$, on associe la $\mbb{A}$-lame de $\ES{N}_{\mbb{A}}$ dŽfinie par 
$$\scrY_{\mbb{A}}\bydef  \left(\prod_{v} \scrY_{F_v}\right)\cap V(\mbb{A})$$ o 
$v$ parcourt les places de $F$. \`A toute $F$-strate $\bsfrY$ de $\NF$, on associe de la mme manire une $\mbb{A}$-strate $\bsfrY_{\mbb{A}}$ de $\ES{N}_{\mbb{A}}$. 
Pour $\scrX= \scrX(\scrY)$ et $\bsfrX=\bsfrX(\bsfrY)$, on dŽfinit toujours de la mme manire les sous-ensembles $\scrX_{\mbb{A}}$ et $\bsfrX_{\mbb{A}}$ de $G(\mbb{A})$. 
Comme pour les $F_S$-lames et les $F_S$-strates, on a les ŽgalitŽs 
$$\scrY = V(F) \cap \scrY_{\mbb{A}} \quad \hbox{et} \quad \bsfrY= V(F) \cap \bsfrY_{\mbb{A}};$$ et on a aussi 
$$\scrX= V(F) \cap \scrX_{\mbb{A}} \quad \hbox{et} \quad \bsfrX\subset V(F) \cap \bsfrX_{\mbb{A}}\ptf$$ 
Observons que si $\scrY \subset \bsfrY$, alors on a
$$\bsfrY_{\mbb{A}}=G(\mbb{A})\cdot \scrY_{\mbb{A}} \quad \hbox{et}\quad \bsfrX_{\mbb{A}} = G(\mbb{A})\cdot \scrX_{\mbb{A}}\ptf$$

La dŽcomposition d'Iwasawa $G(\mbb{A})=\bs{K}P_0(\mbb{A})$ assure qu'on a l'analogue de \ref{rŽsultats pour topFS} pour la topologie adŽlique sur $G(\mbb{A})$: 

\begin{lemme}\label{rŽsultats pour topFA}
Soit $\scrY$ une $F$-lame de $\ES{N}_F$. Posons $\bsfrY=G(F)\cdot \scrY$, $\scrX=\scrX(\scrY)$ et $\bsfrX= \bsfrX(\bsfrY)$. 
\begin{enumerate}
\item[(i)] Les ensembles $\scrX_{\mbb{A}}$ et $\bsfrX_{\mbb{A}}$ sont fermŽs dans $G(\mbb{A})$ pour la topologie adŽlique.
\item[(ii)] La $\mbb{A}$-lame $\scrY_{\mbb{A}}$ est ouverte dans $\scrX_{\mbb{A}}$ pour la topologie adŽlique.
\item[(iii)] La $\mbb{A}$-strate $\bsfrY_{\mbb{A}}$ est ouverte dans 
$\bsfrX_{\mbb{A}}$ pour la topologie adŽlique.
\end{enumerate} 
\end{lemme}
 
On note $\mbb{A}^S = \prod'_{v\notin S}F_S$ l'anneau des adles de $F$ en dehors de $S$. 
On dŽfinit $\scrY_{\mbb{A}^S}$, $\bsfrY_{\mbb{A}^S}$, $\scrX_{\mbb{A}^S}$ et $\bsfrX_{\mbb{A}^S}$ comme ci-dessus, en remplaant $\mbb{A}$ par 
$\mbb{A}^S$. On a les dŽcompositions 
$$\scrY_{\mbb{A}} = \scrY_{F_S} \times \scrY_{\mbb{A}^S}\quad \hbox{et} \quad \bsfrY_{\mbb{A}} = \bsfrY_{F_S} \times \bsfrY_{\mbb{A}^S}\vg$$ 
$$\scrX_{\mbb{A}} = \scrX_{F_S} \times \scrX_{\mbb{A}^S}\quad \hbox{et} \quad \bsfrX_{\mbb{A}} = \bsfrX_{F_S} \times \bsfrX_{\mbb{A}^S}\ptf$$ 
Le lemme \ref{rŽsultats pour topFA} reste vrai si l'on remplace $\mbb{A}$ par $\mbb{A}^S$. 

\subsection{Approximation faible}\label{AF} 
On suppose dans cette sous-section que l'hypothse \ref{hyp reg} est vŽrifiŽe: $e_V$ est rŽgulier dans $V$.

On peut se demander si une $F$-lame $\scrY$ de $\NF$ 
plongŽe diagonalement dans $V(F_S)$ est dense dans $\scrY_{F_S}$ pour {\TopFS}; on peut aussi se poser la mme question 
pour une $F$-strate $\bsfrY$ de $\NF$. C'est le principe de l'\textit{approximation faible}.

\begin{lemme}\label{approximation faible}
Soit $\scrY$ une $F$-lame de $\ES{N}_F$. Posons $\bsfrY = G(F)\cdot \scrY$. 
\begin{enumerate}
\item[(i)] La $F$-lame $\scrY$ est dense dans $\scrY_{F_S}$ pour \TopFS. 
\item[(ii)] L'ensemble $G(F_S)\cdot \bsfrY= G(F_S)\cdot \scrY$ 
est dense dans $\bsfrY_{F_S}$ pour \TopFS.
\item[(iii)] Si pour tout $y_S\in \bsfrY_{F_S}$, la $G(F_S)$-orbite $\ES{O}_{F_S,y_S}$ est ouverte dans $\bsfrY_{F_S}$ pour \TopFS, 
alors on a l'ŽgalitŽ $\bsfrY_{F_S} = G(F_S)\cdot \bsfrY\;(= G(F_S)\cdot \scrY)$.
\end{enumerate}
\end{lemme}

\begin{proof}
Posons $\scrX=\scrX(\scrY)$. 
On peut supposer $\scrY\neq \{e_V\}$ (sinon il n'y a rien ˆ dŽmontrer). Fixons un co-caractre virtuel $\mu \in \bs{\Lambda}_{F,x}$ avec $x\in \scrY$. 
D'aprs \ref{variante de BMRT(4.7)}, pour toute place $v\in S$ et tout ŽlŽment $x_v\in \scrY_{F_v}$, le co-caractre (virtuel) 
$\mu$ appartient ˆ $\bs{\Lambda}_{F_v,x_v}$. 
On a donc $\scrX_{F_S}= V_{\mu,1}(F_S)$ et $P_{F_S,x_S}=P_\mu(F_S)$ pour tout $x_S = \prod_{v\in S} x_v \in \scrY_{F_v}$.  
D'aprs \ref{homeo pour TopF}, l'application naturelle 
$$G(F_S)\times^{P_\mu(F_S)} \scrY_{F_S} \rightarrow \bsfrY_{F_S}$$ 
est un homŽomorphisme pour \TopFS.

Prouvons (i). Puisque la variŽtŽ $V_{\mu,1}$ est un $F$-espace affine \cite[3.3]{H2}, 
elle satisfait au principe d'approximation faible: $\scrX=V_{\mu,1}(F)$ est dense dans $\scrX_{F_S}=V_{\mu,1}(F_S)$ pour \TopFS. 
Comme la $F_S$-lame $\scrY_{F_S}$ est ouverte dans $\scrX_{F_S}$ pour \TopFS (d'aprs \ref{rŽsultats pour topFS}\,(ii)), 
on en dŽduit que $\scrY= \scrY_{F_S} \cap \scrX$ est dense dans $\scrY_{F_S}$ pour \TopFS. 

D'aprs l'homŽomorphisme ci-dessus, le point (ii) est une consŽquence de (i). 

Prouvons (iii). On suppose que pour tout $y_S\in \bsfrY_{F_S}$, la $G(F_S)$-orbite $\ES{O}_{F_S,y_S}$ est ouverte dans $\bsfrY_{F_S}$ pour \TopFS. 
Soit $y_S = \prod_{v\in S}y_v$ un ŽlŽment dans $\bsfrY_{F_S}$. On veut montrer que $y_S \in G(F_S)\cdot \scrY$. 
Quitte ˆ remplacer $y_S$ par $g\cdot y_S$ pour un $g\in G(F_S)$, on peut supposer que $y_S$ appartient ˆ la $F_S$-lame $\scrY_{F_S}$. 
D'aprs l'homŽomorphisme ci-dessus, la $P_\mu(F_S)$-orbite de $y_S$ 
est ouverte dans $\scrY_{F_S}$ pour \TopFS. Puisque $\scrX$ est dense dans $\scrX_{F_S}$ pour \TopFS, cela entra"ne que la $P_\mu(F_S)$-orbite de 
$y_S$ intersecte non trivialement $\scrY_{F_S}\cap \scrX= \scrY$. 
\end{proof}

\begin{remark}
\textup{Puisque la variŽtŽ $V_{\mu,1}$ satisfait aussi au principe d'approxi\-mation forte, on a des rŽsultats analogues pour le plongement diagonal 
$$\iota^S : V(F) \rightarrow V(\mbb{A}^S)$$ et la topologie adŽlique, ˆ savoir: l'ensemble $\iota^S(\scrY)$ est dense dans $\scrY_{\mbb{A}^S}$ pour la topologie adŽlique et 
l'ensemble $G(\mbb{A}^S)\cdot \iota^S(\bsfrY)= G(\mbb{A}^S)\cdot \iota^S(\scrY)$ est dense dans $\bsfrY_{\mbb{A}^S}$ pour la topologie adŽlique. De manire Žquivalente, on a: 
l'ensemble $\scrY_{F_S}\times \iota^S(\scrY_F)$ est dense dans $\scrY_{\mbb{A}}$ et l'ensemble 
$$G(\mbb{A})\cdot \left(\scrY_{F_S}\times \iota^S(\scrY)\right)= \bsfrY_{F_S} \times \left(G(\mbb{A}^S)\bullet \iota^S(\scrY)\right)$$ est dense dans 
$\bsfrY_{\mbb{A}}= G(\mbb{A})\cdot \scrY_{\mbb{A}}$.
}
\end{remark}

\subsection{Le cas de la variŽtŽ unipotente}\label{le cas de la var unip (global)}
On suppose dans cette sous-section que la variŽtŽ $V$ est le groupe $G$ lui-mme muni de l'action par conjugaison (avec $e_G=1$). Rappelons qu'on a notŽ: 
\begin{itemize}
\item $\llbracket \UU\rrbracket_F$ l'ensemble des lames de $\UU$ qui intersectent non trivialement $\UF$; 
\item $[\UU]_F$ est l'ensemble des strates de $\UU$ qui intersectent non trivialement $\UF$;
\item $\llbracket \UF\rrbracket$ l'ensemble des $F$-lames de $\UF$;
\item $[\UF]$ l'ensemble des $F$-strates de $\UF$.
\end{itemize} 
Les applications $$\llbracket \UU\rrbracket_F \rightarrow \llbracket \UF \rrbracket,\, \scrY\mapsto \scrY(F)\vgq [\UU]_F \rightarrow [\UF],\, 
\bsfrY\mapsto \bsfrY_F \bydef  \UF \cap \bsfrY $$ sont bijectives.  
Pour $\bsfrY\in [\UU]_F$ et $s\in S$, on pose $\bsfrY_{F_v}= (\bsfrY_F)_{F_v}$; de la mme manire, on pose $\bsfrY_{\mbb{A}}=(\bsfrY_F)_{\mbb{A}}$. 

\begin{lemme}
Soit $\scrY \in \llbracket \UU \rrbracket_F$. Posons $\bsfrY = G \bullet \scrY \in [\UU]_F$. 
\begin{enumerate}
\item[(i)] L'ensemble $\scrY(F)\cap \ES{O}(\bsfrY)$ est dense dans $\scrY(F_S)$ pour $\mathrm{Top}_{F_S}$.
\item[(ii)] L'ensemble $G(F_S)\bullet (\scrY(F)\cap \ES{O}(\bsfrY))$ est dense dans $\bsfrY_{F_S}$ pour $\mathrm{Top}_{F_S}$.
\end{enumerate}
\end{lemme}

\begin{proof}
D'aprs \ref{orbite dense dans F-lame TopF-ouverte et TopF-dense dans F-saturŽ}, l'intersection $$\scrY(F_S) \cap \ES{O}(\bsfrY)= \prod_{v\in S} (\scrY(F_v) \cap \ES{O}(\bsfrY))$$ 
est ouvert dense dans $\scrY(F_S)$ pour $\mathrm{Top}_F$. Puisque $\scrY(F)$ est dense dans $\scrY(F_S)$ pour $\mathrm{Top}_{F_S}$ (\ref{approximation faible}\,(i)), 
on obtient que $\scrY \cap (\scrY(F_S) \cap \ES{O}(\bsfrY))= \scrY(F)\cap \ES{O}(\bsfrY)$ est dense dans $\scrY(F_S)$ pour $\mathrm{Top}_{F_S}$. Cela prouve (i). 
Le point (ii) est une consŽquence de (i), compte-tenu de l'homŽomorphisme $G(F_S)\times^{P_\mu(F_S)}\scrY_{F_S} \rightarrow \bsfrY_{F_S}$ pour \TopFS o $\mu\in \bs{\Lambda}_{F,u}$ pour un ŽlŽment $u\in \scrY$ (\textit{cf.} la preuve de \ref{approximation faible}).   
\end{proof}

\begin{lemme}\label{orbite ouverte pour TopFS}
Si $F$ est un corps de nombres ou un corps global de caractŽristique $p> 1$ avec $p$ trs bon pour $G$, 
l'hypothse de \ref{approximation faible}\,(iii) est toujours vŽrifiŽe: pour tout $u_S\in \UU_{F_S}$, 
la $G(F_S)$-orbite de $u_S$ est ouverte dans $\bsfrY_{F_S,u_S}$ pour \TopFS.
\end{lemme}

\begin{proof}
On peut supposer $S=\{v\}$. Soit $p\geq 1$ l'exposant caractŽristique de $F_v$. Supposons que 
$p$ soit trs bon pour $G$. Pour tout $u_v\in \UU_{F_v}$, d'aprs \ref{comparaison F-strate-orbite rat}\,(i), on a $\bsfrY_{F_v,u_v}= \ES{O}_{u_v}(F_v)$. Comme d'autre part le morphisme 
$G\rightarrow \ES{O}_{u_v},\, g \mapsto g\bullet u_v$ est sŽparable (\ref{sŽparabilitŽ des orbites}),   il induit une application ouverte 
$G(F_v) \rightarrow \ES{O}_{u_v}(F_v)$ pour \TopFv. Par suite l'application $G(F_v)\rightarrow \bsfrY_{F_v},\, g_v \mapsto g_v\bullet u_v$ 
est ouverte pour \TopFv. 
\end{proof}

\begin{remark}
\textup{
\begin{enumerate}
\item[(i)] Supposons vŽrifiŽes les hypothses de \ref{orbite ouverte pour TopFS}. Pour tout $u\in \UF$, toute $G(F_S)$-orbite dans 
$\ESO_u(F_S)= \bsfrY_{F_S,u}$ rencontre $\ESO_u(F)= \bsfrY_{F,u}$ (d'aprs \ref{approximation faible}\,(iii)). Pour les corps de nombres, c'est le lemme~7.1 de \cite{A2}.
\item[(ii)] Si $p>1$ est suffisamment grand, toutes les orbites unipotentes sont sŽparables et l'on s'attend ˆ ce que 
l'hypothse de \ref{approximation faible}\,(iii) soit vŽrifiŽe. 
En revanche si $p>1$ est petit, elle ne l'est en gŽnŽral pas. Prenons le cas du groupe $G=\mathrm{SL}_2$ sur le corp global $F = \mbb{F}_2(t)$ 
et considŽrons le complŽtŽ $F_v= \mbb{F}_2((t))$. D'aprs \ref{exemples en basse dimension}, 
l'ensemble $\UU_{F_v}$ contient une unique $F_v$-strate non triviale et les $G(F_v)$-orbites dans cette $F_v$-strate sont paramŽtrisŽes 
par le groupe $F_v^\times/(F_v^\times)^2$ qui est compact mais de cardinal infini. 
Si l'hypothse de \ref{approximation faible}\,(iii) Žtait vŽrifiŽe, on aurait que l'image de $F^\times $dans $F_v^\times /(F_v^\times)^2$ est surjective; 
ce qui est faux car $F^\times$ est dŽnombrable alors que $F_v^\times/(F_v^\times)^2$ ne l'est pas. 
En effet, considŽrons une unitŽ $x$ de $F_v^\times$, \cad un ŽlŽment inversible dans $\mbb{F}_2[[t]]$: 
$$x= 1 + \epsilon_1 t + \cdots + \epsilon_n t^n + \cdots = 1 + \sum_{i=1}^\infty \epsilon_i t^i\quad \hbox{avec}\quad \epsilon_i \in \{0,1\}\ptf$$ 
Les carrŽs dans $\mbb{F}_2[[t]]$ sont de la forme $y=\sum_{i=0}^\infty \eta_i t^{2i}$ avec $\eta_i \in \{0,1\}$. Par consŽquent $x$ s'Žcrit $x= y_1 + t y_2$ o 
$y_1$ est un carrŽ dans $\mbb{F}_2[[t]]^\times$ et $y_2$ est un carrŽ dans $\mbb{F}_2[[t]]$. Ainsi modulo les carrŽs dans $F_v^\times$, $x$ 
a un unique reprŽsentant de la forme $1+tz$ avec 
$z\;(=y_1^{-1}y_2)$ dans $\mbb{F}_2[[t]]^2\simeq \{0,1\}^{\mbb{N}}$.
\end{enumerate}
}
\end{remark}


\printindex

\end{document}